\documentclass[a4paper,12pt]{article}
\usepackage{amsfonts,amsmath,amssymb}
\newtheorem{theorem}{Theorem}[section]
\newtheorem{lemma}[theorem]{Lemma}
\newtheorem{proposition}[theorem]{Proposition}
\newtheorem{corollary}[theorem]{Corollary}
\newtheorem{remark}[theorem]{Remark}
\newtheorem{definition}[theorem]{Definition}
\numberwithin{equation}{section}
\newenvironment{proof}{{\bf Proof\ }}{\QED\\}
\newcommand{\QED}{\hspace*{\fill}\rule{2.5mm}{2.5mm}}

\newcommand{\be}{\begin{equation}}
\newcommand{\ee}{\end{equation}}
\newcommand{\ba}{\begin{array}}
\newcommand{\ea}{\end{array}}
\newcommand{\bea}{\begin{eqnarray}}
\newcommand{\eea}{\end{eqnarray}}
\newcommand{\bee}{\begin{eqnarray*}}
\newcommand{\eee}{\end{eqnarray*}}

\newcommand{\lab}{\label}
\newcommand{\und}{\underline}

\newcommand{\ds}{\displaystyle}
\newcommand{\nn}{\nonumber}
\providecommand{\norm}[1]{\lVert#1\rVert}
\providecommand{\normm}[1]{\left\lVert#1\right\rVert}
\renewcommand{\c}{\cdot}
\newcommand{\les}{\lesssim}

\newcommand{\la}{\lambda}

\newcommand{\trt}{\textrm{tr}\theta}
\newcommand{\dd}{{\bf D}}

\newcommand{\R}{\mathbb{R}}

\newcommand{\no}{\mathcal{N}_1}
\newcommand{\noo}{\mathcal{N}_2}

\newcommand{\MM}{\mathcal{M}}

\renewcommand{\gg}{{\bf g}}

\renewcommand{\th}{\theta}
\newcommand{\ep}{\varepsilon}

\newcommand{\gn}{{\bf g}(N,N')}
\newcommand{\gl}{{\bf g}(L,L')}

\renewcommand{\l}[2]{L^{#1}_uL^{#2}(\mathcal{H}_u)}
\newcommand{\lprime}[2]{L^{#1}_{u'}L^{#2}(\mathcal{H}_{u'})}
\renewcommand{\ll}[1]{L^{#1}(\mathcal{M})}
\renewcommand{\le}[1]{L^{#1}(\R^3)}

\newcommand{\nabb}{\mbox{$\nabla \mkern-13mu /$\,}}
\newcommand{\ddb}{\mbox{$\nabla \mkern-13mu /$\,}}
\newcommand{\lap}{\mbox{$\Delta \mkern-13mu /$\,}}

\newcommand{\divb}{\mbox{$\textrm{div} \mkern-13mu /$\,}}

\newcommand{\nabn}{\nabla_N}

\renewcommand{\a}{\alpha}
\renewcommand{\b}{\beta}

\newcommand{\kep}{\epsilon}
\newcommand{\Si}{\Sigma}
\newcommand{\Sit}{\Sigma_t}
\newcommand{\pr}{\partial}
\newcommand{\nab}{\nabla}
\newcommand{\lb}{\underline{L}}
\newcommand{\dmt}{d\mu_{t,u}}

\newcommand{\ptu}{P_{t,u}}
\newcommand{\pou}{P_{0,u}}

\newcommand{\half}{\frac{1}{2}}

\renewcommand{\o}{\omega}
\renewcommand{\S}{\mathbb{S}^2}
\newcommand{\po}{\partial_{\omega}}
\newcommand{\xo}{x\cdot\omega}

\renewcommand{\H}{\mathcal{H}}
\newcommand{\li}[2]{L^{#1}_uL^{#2}(\mathcal{H}_u)}
\newcommand{\lh}[1]{L^{#1}(\mathcal{H}_u)}
\newcommand{\tx}[2]{L^{#1}_tL^{#2}_{x'}}
\newcommand{\xt}[2]{L^{#1}_{x'}L^{#2}_t}
\newcommand{\lpt}[1]{L^{#1}(P_{t,u})}

\newcommand{\lsit}[2]{L^{#1}_tL^{#2}(\Sit)}

\renewcommand{\a}{\alpha}
\renewcommand{\b}{\beta}

\newcommand{\ga}{\gamma}
\renewcommand{\d}{\delta}

\newcommand{\s}{\sigma}

\newcommand{\kepb}{\overline{\epsilon}}
\newcommand{\db}{\overline{\delta}}

\newcommand{\z}{\zeta}
\newcommand{\zb}{\underline{\zeta}}
\newcommand{\trc}{\textrm{tr}\chi}
\newcommand{\hch}{\widehat{\chi}}
\newcommand{\chb}{\underline{\chi}}
\newcommand{\trchb}{\textrm{tr}{\underline{\chi}}}
\newcommand{\hchb}{\widehat{\underline{\chi}}}
\newcommand{\xib}{\underline{\xi}}

\begin{document}

\begin{center}
\Large{\bf Parametrix for wave equations on a rough background IV: control of the error term}
\end{center}

\vspace{0.5cm}

\begin{center}
\large{J\'er\'emie Szeftel}
\end{center}

\vspace{0.2cm}

\begin{center}
\large{DMA, Ecole Normale Sup\'erieure,\\
45 rue d'Ulm, 75005 Paris,\\
jeremie.szeftel@ens.fr}
\end{center}

\vspace{0.5cm}

{\bf Abstract.} This is the last of a sequence of four papers  \cite{param1}, \cite{param2}, \cite{param3}, \cite{param4} dedicated to the construction and the control of a parametrix to the homogeneous wave equation $\square_{\bf g} \phi=0$, where ${\bf g}$ is a rough metric satisfying the Einstein vacuum equations. Controlling such a parametrix as well as its error term when one only assumes $L^2$ bounds on the curvature tensor ${\bf R}$ of ${\bf g}$ is a major step of the proof of the bounded $L^2$ curvature conjecture proposed in \cite{Kl:2000}, and solved by S. Klainerman, I. Rodnianski and the author in \cite{boundedl2}. On a more general level, this sequence of papers deals with the control of the eikonal equation on a rough background, and with the  derivation of $L^2$ bounds for Fourier integral operators on manifolds with rough phases and symbols, and as such is also of independent interest.

\vspace{0.2cm}

\section{Introduction}

We consider the Einstein vacuum equations,
\be\lab{eq:I1}
{\bf R}_{\alpha\beta}=0
\end{equation}
where ${\bf R}_{\alpha\beta}$
denotes the  Ricci curvature tensor  of  a four dimensional Lorentzian space time  $(\mathcal{M},\,  {\bf g})$. The Cauchy problem consists in finding a metric ${\bf g}$ satisfying \eqref{eq:I1} such that the metric induced by ${\bf g}$ on a given space-like hypersurface $\Sigma_0$ and the second fundamental form of $\Sigma_0$ are prescribed. The initial data then consists of a Riemannian three dimensional metric $g_{ij}$ and a symmetric tensor $k_{ij}$ on the space-like hypersurface $\Sigma_0=\{t=0\}$. Now, \eqref{eq:I1} is an overdetermined system and the initial data set $(\Sigma_0,g,k)$ must satisfy the constraint equations
\be\lab{const}
\left\{\begin{array}{l}
\nabla^j k_{ij}-\nabla_i \textrm{Tr}k=0,\\
 R-|k|^2+(\textrm{Tr}k)^2=0, 
\end{array}\right.
\ee
where the covariant derivative $\nabla$ is defined with respect to the metric $g$, $R$ is the scalar curvature of $g$, and $\textrm{Tr}k$ is the trace of $k$ with respect to the metric $g$.

The fundamental problem in general relativity is to 
 study the long term regularity and asymptotic 
properties
 of the Cauchy developments of general, asymptotically flat,  
initial data sets $(\Sigma_0, g, k)$. As far as local regularity is concerned it 
is natural to ask what are the minimal regularity properties of
the initial data which guarantee the existence and 
uniqueness of local developments. In \cite{boundedl2}, we obtain the 
following result which solves bounded $L^2$ curvature conjecture proposed 
in \cite{Kl:2000}:

\begin{theorem}[Theorem 1.10 in \cite{boundedl2}]\lab{th:mainbl2}
Let $(\mathcal{M}, {\bf g})$ an asymptotically flat solution to the Einstein vacuum equations \eqref{eq:I1} together with a maximal foliation by space-like hypersurfaces $\Sigma_t$ defined as level hypersurfaces of a time function $t$. Let $r_{vol}(\Sigma_t,1)$ the volume radius on scales $\leq 1$ of $\Sigma_t$\footnote{See Remark \ref{rem:volrad} below for a  definition}. Assume that the initial slice $(\Sigma_0,g,k)$ is such that:
$$\norm{R}_{L^2(\Sigma_0)}\leq \ep,\,\norm{k}_{L^2(\Sigma_0)}+\norm{\nabla k}_{L^2(\Sigma_0)}\leq \ep\textrm{ and }r_{vol}(\Sigma_0,1)\geq \frac{1}{2}.$$
Then, there exists a small universal constant $\ep_0>0$ such that if $0<\ep<\ep_0$, then the following control holds on $0\leq t\leq 1$:
$$\norm{\R}_{L^\infty_{[0,1]}L^2(\Sigma_t)}\lesssim \ep,\,\norm{k}_{L^\infty_{[0,1]}L^2(\Sigma_t)}+\norm{\nabla k}_{L^\infty_{[0,1]}L^2(\Sigma_t)}\lesssim \ep\textrm{ and }\inf_{0\leq t\leq 1}r_{vol}(\Sigma_t,1)\geq \frac{1}{4}.$$
\end{theorem}

\begin{remark}
While  the first   nontrivial improvements  for well posedness for quasilinear  hyperbolic systems (in spacetime dimensions greater than $1+1$), based on Strichartz estimates,  were obtained in   \cite{Ba-Ch1}, \cite{Ba-Ch2}, \cite{Ta1}, \cite{Ta}, \cite{KR:Duke}, \cite{KR:Annals}, \cite{SmTa}, Theorem \ref{th:mainbl2}, is the first  result in which the  full nonlinear structure of the quasilinear system, not just its principal part,  plays  a  crucial  role.  We note that  though  the result is not  optimal with respect to the  standard  scaling  of the Einstein equations, it is  nevertheless critical   with respect to its causal geometry,  i.e. $L^2 $ bounds on the curvature is the minimum requirement necessary to obtain lower bounds on the radius of injectivity of null hypersurfaces. We refer the reader to section 1 in \cite{boundedl2} for more motivations and historical perspectives concerning Theorem \ref{th:mainbl2}. 
\end{remark}

\begin{remark}
The regularity assumptions on $\Sigma_0$ in Theorem \ref{th:mainbl2} - i.e. $R$ and $\nabla k$ bounded in $L^2(\Sigma_0)$ - correspond to an initial data set $(g,\, k )\in H^2_{loc}(\Sigma_0)\times H^1_{loc}(\Sigma_0)$.
\end{remark}

\begin{remark}\lab{rem:reducsmallisok}
In \cite{boundedl2}, our main result is stated for corresponding large data. We then reduce the proof to the  small data statement of Theorem \ref{th:mainbl2} relying on a truncation and rescaling procedure, the control of the harmonic radius of $\Sigma_0$ based on Cheeger-Gromov convergence of Riemannian manifolds together with the assumption on the lower bound of the volume radius of $\Sigma_0$, and the gluing procedure in \cite{Co}, \cite{CoSc}. We refer the reader to section 2.3 in \cite{boundedl2} for the details.
\end{remark}

\begin{remark}\lab{rem:volrad}
We recall for the convenience of the reader the definition of the volume radius of the Riemannian manifold $\Sigma_t$. Let $B_r(p)$ denote the geodesic ball of center $p$ and radius $r$. The volume radius $r_{vol}(p,r)$ at a point $p\in \Sigma_t$ and scales $\leq r$ is defined by
$$r_{vol}(p,r)=\inf_{r'\leq r}\frac{|B_{r'}(p)|}{r^3},$$
with $|B_r|$ the volume of $B_r$ relative to the metric $g_t$ on $\Sigma_t$. The volume radius $r_{vol}(\Sigma_t,r)$ of $\Sigma_t$ on scales $\leq r$ is the infimum of $r_{vol}(p,r)$ over all points $p\in \Sigma_t$.
\end{remark}

The proof of Theorem \ref{th:mainbl2}, obtained in the sequence of papers \cite{boundedl2}, \cite{param1}, \cite{param2}, \cite{param3}, \cite{param4}, \cite{bil2}, relies on the following ingredients\footnote{We also need trilinear estimates and an $L^4(\mathcal{M})$ Strichartz estimate (see the introduction in \cite{boundedl2})}: 
{\em\begin{enumerate}
\item[{\bf A}] Provide  a system of coordinates relative to which \eqref{eq:I1} exhibits a null structure.

\item[{\bf B}] Prove  appropriate bilinear estimates for solutions to $\square_{\bf g} \phi=0$, on
 a fixed Einstein vacuum  background\footnote{Note that the first bilinear estimate of this type was obtained in \cite{BIL}}.

\item[{\bf C}] Construct a parametrix for solutions to the homogeneous wave equations $\square_{\bf g} \phi=0$ on a fixed Einstein vacuum  background, and obtain control of the parametrix and of its error term only using the fact that the curvature tensor is bounded in $L^2$. 
\end{enumerate}
}

Steps {\bf A} and {\bf B} are carried out in \cite{boundedl2}. In particular, the proof of the bilinear estimates rests on a representation formula for the solutions of the wave equation using the following plane wave parametrix\footnote{\eqref{param} actually corresponds to a half-wave parametrix. The full parametrix corresponds to the sum of two half-parametrix. See \cite{param2} for the construction of the full parametrix}:
\be\lab{param}
Sf(t,x)=\int_{\S}\int_{0}^{+\infty}e^{i\lambda u(t,x,\o)}f(\lambda\o)\lambda^2 d\lambda d\o,\,(t,x)\in\mathcal{M} 
\ee
where $u(.,.,\o)$ is a solution to the eikonal equation ${\bf g}^{\alpha\beta}\partial_\alpha u\partial_\beta u=0$ on $\mathcal{M}$ such that $u(0,x,\o)\sim x.\o$ when $|x|\rightarrow +\infty$ on $\Sigma_0$\footnote{The asymptotic behavior for $u(0,x,\o)$ when $|x|\rightarrow +\infty$ is used in \cite{param2} to generate with the parametrix any initial data set for the wave equation}. Therefore, in order to complete the proof of the bounded $L^2$ curvature conjecture, we need to carry out step {\bf C} with the parametrix defined in \eqref{param}. 

\begin{remark}
Note that the parametrix \eqref{param} is invariantly defined\footnote{Our choice is reminiscent of the one used in \cite{SmTa} in the context of $H^{2+\epsilon}$ solutions of quasilinear wave equations. Note however that the construction in that paper is coordinate dependent}, i.e. without reference to any coordinate system. This is crucial since coordinate systems consistent with $L^2$ bounds on the curvature would not be regular enough to control a parametrix. 
\end{remark}

\begin{remark}
In addition to their relevance to the resolution of the bounded $L^2$ curvature conjecture, the methods and results  of step {\bf C} are  also of independent interest. Indeed, they deal on the one hand with the control of the eikonal equation ${\bf g}^{\alpha\beta}\partial_\alpha u\partial_\beta u=0$ at a critical level\footnote{We need at least $L^2$ bounds on the curvature to obtain a lower bound on the radius of injectivity of the null level hypersurfaces of the solution $u$ of the eikonal equation, which in turn is necessary to control the local regularity of $u$ (see \cite{param3})}, and on the other hand with the derivation of $L^2$ bounds for Fourier integral operators with significantly lower differentiability  assumptions both for the corresponding phase and symbol compared to classical methods (see in particular the discussion below \eqref{cestpasbeaucoup}). 
\end{remark}

In view of the energy estimates for the wave equation, it suffices to control the parametrix at $t=0$ (i.e. restricted to $\Sigma_0$)
\be\lab{parami}
Sf(0,x)=\int_{\S}\int_{0}^{+\infty}e^{i\lambda u(0,x,\o)}f(\lambda\o)\lambda^2 d\lambda d\o,\,x\in\Sigma_0 
\ee
 and the error term
\be\lab{err} 
Ef(t,x)=\square_{\bf g}Sf(t,x)=\int_{\S}\int_{0}^{+\infty}e^{i\lambda u(t,x,\o)}\square_{\bf g}u(t,x,\o)f(\lambda\o)\lambda^3 d\lambda d\o,\,(t,x)\in\mathcal{M}. 
\ee
This requires the following ingredients, the two first being related to the control of the parametrix restricted to $\Sigma_0$ \eqref{parami}, and the two others being related to the control of the error term \eqref{err}:
{\em\begin{enumerate}
\item[{\bf C1}] Make an appropriate choice for the equation satisfied by $u(0,x,\o)$ on $\Sigma_0$, and control the geometry of the foliation generated by the level surfaces of $u(0,x,\o)$ on $\Sigma_0$.

\item[{\bf C2}] Prove that the parametrix at $t=0$ given by \eqref{parami} is bounded in $\mathcal{L}(L^2(\mathbb{R}^3), L^2(\Sigma_0))$ using the estimates for $u(0,x,\o)$ obtained in {\bf C1}.

\item[{\bf C3}] Control the geometry of the foliation generated by the level hypersurfaces of $u$ on $\mathcal{M}$.

\item[{\bf C4}] Prove that the error term \eqref{err} satisfies the estimate $\norm{Ef}_{L^2(\mathcal{M})}\leq C\norm{\lambda f}_{L^2(\mathbb{R}^3)}$ using the estimates for $u$ and $\square_{\gg}u$ proved in {\bf C3}.
\end{enumerate}
}

Step {\bf C1} has been carried out in \cite{param1}, step {\bf C2} has been carried out in \cite{param2}, and step {\bf C3} has been carried out in \cite{param3}. In the present paper, we focus on step {\bf C4}. Note that the error term  \eqref{err} is a Fourier integral operator (FIO) with phase $u(t,x,\o)$ and symbol $\square_{\gg}u(t,x,\o)$. Now, we only assume $L^2$ bounds on the curvature tensor ${\bf R}$ in order to be consistent with the statement of Theorem \ref{th:mainbl2}. This severely limits the regularity $(t,x)$ we are able to obtain in step {\bf C3} for the solution $u(t,x,\o)$ of the Eikonal equation ${\bf g}^{\alpha\beta}\partial_\alpha u\partial_\beta u=0$ on $\mathcal{M}$ (see \cite{param3} and section \ref{sec:regassphase}). Although ${\bf R}$ does not depend on the parameter $\o$, the regularity in $\o$ we are able to obtain in step {\bf C3} for $u(t,x,\o)$ is very limited as well\footnote{This is due to the fact that our estimates are sensitive to certain directions tied to the $u$-foliation of $\mathcal{M}$. Now, after differentiation with respect to $\o$, derivatives in "good" directions pick up a nonzero component along "bad" directions (see \cite{param3} for details)}. In particular, we obtain for the symbol of $E$ in \eqref{err}:
\be\lab{cestpasbeaucoup}
\sup_{\o, u}\left(\norm{\square_{\bf g}u}_{L^\infty(\H_u)}+\norm{\dd\square_{\bf g}u}_{L^2(\H_u)}+\norm{\po\square_{\bf g}u}_{L^2(\H_u)}\right)\les\ep,
\ee
where $\H_u$ denotes the level hypersurfaces of the function $u(t,x,\o)$. Let us note that the classical arguments for proving $L^2$ bounds for FIO are based either on a $T T^*$ argument, or a $T^* T$ argument, which requires in our setting taking at least 3 derivatives of the symbol in $L^\infty(\MM\times\S)$ either with respect to $(t,x)$ for $T^*T$, or with respect to $(\la, \o)$ for $TT^*$ (see for example \cite{stein}). Both methods would fail by by a large margin, in particular in view of the regularity \eqref{cestpasbeaucoup} obtained for the symbol of the error term $E$. In order to obtain the control required in step {\bf C4} with the regularity of the symbol of the FIO $E$ given by \eqref{cestpasbeaucoup}, we rely in particular on the following ingredients:
\begin{itemize}
\item geometric integrations by parts taking full advantage of the better regularity properties in certain directions tied  to the level hypersurfaces $\H_u$ of $u$,

\item the standard first and second dyadic decomposition in frequency and angle (see \cite{stein}),

\item an additional decomposition in physical space relying on the geometric Littlewood-Paley projections of \cite{LP}.
\end{itemize}

\vspace{0.3cm}

\noindent{\bf Acknowledgments.} The author wishes to express his deepest gratitude to Sergiu Klainerman and Igor Rodnianski for stimulating discussions and constant encouragements during the long years where this work has matured. He also would like to stress that  the basic strategy of the construction of the parametrix and how it fits  into the whole proof of the bounded $L^2$ curvature conjecture has been done in collaboration with them. Finally, he  would like to mention the influential work \cite{SmTa} providing construction and control of parametrices for $H^{2+\epsilon}$ solutions of quasilinear wave equations. The author is supported by ANR jeunes chercheurs SWAP.\\

\section{Main results}

The error term $E$ in \eqref{err} is a Fourier integral operator on $\mathcal{M}$ with phase $u(t, x,\o)$ and symbol $\square_{{\bf g}}u(t,x,\o)$. The regularity assumptions on $u(t,x,\o)$ will be crucial to complete step {\bf C4}, that is prove the following estimate for the error term:
$$\norm{Ef}_{L^2(\mathcal{M})}\lesssim \norm{\lambda f}_{L^2(\mathbb{R}^3)}.$$ 
In this section, we state our assumptions on $u(t,x,\o)$ before stating our main result.

\subsection{Maximal foliation on $\mathcal{M}$}

We foliate the space-time $\mathcal{M}$ by space-like hypersurfaces $\Sit$ defined as level hypersurfaces of a time function $t$. Denoting by 
$T$ the unit, future oriented, normal to $\Si_t$ and $k$
the second fundamental form 
\be\lab{def:k}
k_{ij}=-<\dd_iT, \pr_j> 
\ee
we find,
$$k_{ij}= -\frac{1}{2}\mathcal{L}_T \gg_{\,ij}$$
with $\mathcal{L}_X$ denoting the Lie derivative with respect to the 
vectorfield $X$. Let Tr$(k)=g^{ij}k_{ij}$ where $g$ is the induced metric on $\Sit$ and Tr is the trace.  In order to be consistent with the statement of Theorem \ref{th:mainbl2}, we impose a maximal foliation 
\be\lab{maxfoliation}
\textrm{Tr}(k)=0.
\ee
We also define the lapse $n$ as 
\be\lab{lapsen}
n^{-1}=T(t).
\ee
We have (see for example \cite{param3}):
\be\lab{3.2}
\dd_TT=n^{-1}\nab n,
\end{equation}
where $\nab$ denotes the gradient with respect to the induced metric 
on $\Si_t$. 

Finally, the lapse $n$ satisfies the following elliptic equation on $\Sit$ (see \cite{ChKl} p. 13):
\be\lab{lapsen1}
\Delta n=|k|^2n,
\ee
where one uses \eqref{def:k}, \eqref{3.2}, Einstein vacuum equations \eqref{eq:I1} and the fact that the foliation generated by $t$ on $\mathcal{M}$ is maximal \eqref{maxfoliation}.

\subsection{Geometry of the foliation generated by $u$ on $\mathcal{M}$}

Remember that $u$ is a solution to the eikonal equation $\gg^{\alpha\beta}\partial_\alpha u\partial_\beta u=0$ on $\mathcal{M}$ depending on a extra parameter $\o\in \S$.  The level hypersufaces $u(t,x,\o)=u$  of the optical function $u$ are denoted by  $\H_u$. Let $L'$ denote the space-time gradient of $u$, i.e.:
\be\lab{def:L'0}
L'=-\gg^{\a\b}\pr_\b u \pr_\a.
\ee
Using the fact that $u$ satisfies the eikonal equation, we obtain:
\be\lab{def:L'1}
\dd_{L'}L'=0,
\ee
which implies that $L'$ is the geodesic null generator of $\H_u$.

We have: 
$$T(u)=\pm |\nab u|$$
where $|\nab u|^2=\sum_{i=1}^3|e_i(u)|^2$ relative to an orthonormal frame $e_i$ on $\Si_t$. Since the sign of $T(u)$ is irrelevant, we choose by convention:
\be\lab{it1'}
T(u)=|\nab u|.
\end{equation}
We denote by $P_{t,u}$  the surfaces of intersection
between $\Si_t$ and  $\H_u$. They play a fundamental role
in our discussion.
\begin{definition}[\textit{Canonical null pair}]
 \be\lab{it2}
L=bL'=T+N, \qquad \lb=2T-L=T-N
\end{equation}
where $L'$ is the space-time gradient of $u$ \eqref{def:L'0}, $b$  is  the  \textit{lapse of the null foliation} (or shortly null lapse)
\be\lab{it3}
b^{-1}=-<L', T>=T(u),
\end{equation} 
and $N$ is a unit normal, along $\Si_t$, to the surfaces $P_{t,u}$. Since $u$ satisfies the 
eikonal equation $\gg^{\alpha\beta}\partial_\alpha u\partial_\beta u=0$ on $\mathcal{M}$, 
this yields $L'(u)=0$ and thus $L(u)=0$. In view of the definition of $L$ and \eqref{it1'}, 
we obtain:
\be\lab{it3bis}
N=-\frac{\nabla u}{|\nabla u|}.
\end{equation} 
\label{def:nulllapse}
\end{definition}

\begin{definition} A  null frame   $e_1,e_2,e_3,e_4$ at a point $p\in P_{t,u}$ 
  consists,
in addition to the null pair     $e_3=\lb,
e_4=L$, of {\sl  arbitrary  orthonormal}  vectors  $e_1,e_2$ tangent
to $P_{t,u}$. 
%All the   estimates in this paper  are in fact local and
% independent of the
% choice of a particular frame. We do not need to worry
% that these  frames cannot be  globally defined.
\end{definition}
\begin{definition}[\textit{Ricci coefficients}]

 Let  $e_1,e_2,e_3,e_4$ be a null frame on
$P_{t,u}$ as above.   The following tensors on  $\ptu$ 
\begin{alignat}{2}
&\chi_{AB}=<\dd_A e_4,e_B>, &\quad 
&\chb_{AB}=<\dd_A e_3,e_B>,\label{chi}\\
&\z_{A}=\half <\dd_{3} e_4,e_A>,&\quad
&\zb_{A}=\half <\dd_{4} e_3,e_A>,\nn\\
&\xib_{A}=\half <\dd_{3} e_3,e_A>.\nn
\end{alignat}
are called the Ricci coefficients associated to our canonical null pair.

We decompose $\chi$ and $\chb$ into
their  trace and traceless components.
\begin{alignat}{2}
&\trc = \gg^{AB}\chi_{AB},&\quad &\trchb = \gg^{AB}\chb_{AB},
\label{trchi}\\
&\hch_{AB}=\chi_{AB}-\half \trc \gg_{AB},&\quad 
&\hchb_{AB}=\chb_{AB}-\half \trchb \gg_{AB},
\label{chih} 
\end{alignat}
\end{definition}
Observe that all tensors defined above are $\ptu$-tangent.

\begin{definition}\lab{def:decompositionk}
We decompose the symmetric traceless 2 tensor $k$ into the scalar $\d$, the $\ptu$-tangent 1-form $\epsilon$, and the $\ptu$-tangent symmetric 2-tensor $\eta$ as follows:
\be\lab{decompositionk}
\left\{\begin{array}{l}
k_{NN}=\d\\
k_{AN}=\kep_A\\
k_{AB}=\eta_{AB}.
\end{array}\right.
\ee
Note that Tr$(k)=$tr$(\eta)+\d$ which together with the maximal foliation assumption \eqref{maxfoliation} 
yields:
\be\lab{eq:traceeta}
\textrm{tr}(\eta)=-\d.
\ee
\end{definition}

The following  \textit{Ricci equations} can be easily derived from the properties 
of $T$ \eqref{def:k} \eqref{3.2}, the fact that $L'$ is geodesic \eqref{def:L'1}, 
and the definition \eqref{chi} of the Ricci coefficients (see \cite{ChKl} p. 171): 
%They express  the covariant 
%derivatives $\dd$ of the null frame $(e_A)_{A=1,2}, e_3, e_4$ relative to itself.
\begin{alignat}{2}
&\dd_A e_4=\chi_{AB} e_B - \kep_{A} e_4, &\quad 
&\dd_A e_3=\chb_{AB} e_B + \kep_{A} e_3,\nn\\
&\dd_{4} e_4 = -\db   e_4, &\quad 
&\dd_{4} e_3= 2\zb_{A} e_A + \db   e_3, 
  \label{ricciform} \\
 &\dd_{3} e_4 = 2\z_{A}e_A +
(\d +n^{-1}\nab_Nn)  e_4,&\quad &\dd_{3} e_3 = 2\xib_{A}e_A -
(\d +n^{-1}\nab_Nn) e_3,\nn\\ &\dd_{4} e_A = \ddb_{4} e_A +
\zb_{A} e_4,&\quad &\dd_{3} e_A = \ddb_{3} e_A+\z_A e_3
+ \xib_A e_4,
\nn\\
&\dd_{B}e_A = \nabb_B e_A +\half \chi_{AB}\, e_3 +\half
 \chb_{AB}\, e_4\nn
\end{alignat}
where, $\ddb_{ 3}$, $\ddb_{ 4}$ denote the 
projection on $P_{t,u}$ of $\dd_3$ and $\dd_4$, $\nabb$
denotes the induced covariant derivative on $P_{t,u}$
and $\db, \kepb$ are defined by: 
\be\lab{newk}
\db=\d-n^{-1}N(n),\,\kepb_A=\kep_A-n^{-1}\nab_A n.
\end{equation}
Also,
\begin{align}
&\chb_{AB}=-\chi_{AB}-2k_{AB},\nn\\
&\zb_A = -\kepb_A,\label{etab}\\
&\xib_A = \kep_{A}+n^{-1} \nabb_A n-\z_A\nn.
\end{align}

Let $\th$ is the second fundamental form of $\ptu$ in $\Sit$. Since $L=T+N$, $\th$ is connected to the second fundamental form $k$ of $\Sit$ and the null second fundamental form $\chi$ of $\ptu$ through the formula:
\be\lab{def:theta}
\theta_{AB}=\chi_{AB}+\eta_{AB}.
\ee
In view of the Ricci equations \eqref{ricciform}, we have:
\be\lab{frame}
\left\{\begin{array}{l}
\nabla_AN=\th_{AB}e_B,\\
\nabn N=-b^{-1}\nabb b.
\end{array}\right.
\ee

Recall that $\trc$ satisfies a transport equation called the Raychaudhuri equation:
\be\lab{raychaudhuri}
L(\trc) + \half (\trc)^2 = - |\hch|^2 -\db   \trc.
\ee
We also recall the transport equation satisfied by the null lapse $b$:
\be\lab{D4a}
L(b) = -\db b.
\ee

The following lemma will allow us to identify the symbol $\square_{\bf g}u$ of the error term \eqref{err}:
\begin{lemma}
For any scalar function $\phi$ on $\mathcal{M}$, we have:
\begin{equation}\label{chudington}
\square_{\bf g}\phi=-\lb(L(\phi))+\lap\phi+2\zeta\cdot\nabb\phi+(\d+n^{-1}\nabla_Nn)L(\phi)+\frac{1}{2}\trc \lb(\phi)+\frac{1}{2}\trchb L(\phi).
\end{equation}
\end{lemma}

\begin{proof}
We have:
\begin{equation}\label{chudington1}
\square_{\bf g}\phi=\dd^\a\dd_\a\phi=-\dd^2\phi(L,\lb)+\dd^A\dd_A\phi.
\end{equation}
Now, using the Ricci equations \eqref{ricciform}, we obtain:
\bee
\dd_A\dd_B\phi&=&e_A(e_B(\phi))-\dd_{\dd_Ae_B}\phi\\
&=& e_A(e_B(\phi))-\nabb_{\nabb_Ae_B}\phi-\half \chi_{AB}\, \lb(\phi) -\half
 \chb_{AB}\, L(\phi)\\
 &=& \nabb_A\nabb_B\phi-\half \chi_{AB}\, \lb(\phi) -\half
 \chb_{AB}\, L(\phi).
 \eee
 Taking the trace, this yields:
 \begin{equation}\label{chudington2}
 \dd^A\dd_A\phi=\lap\phi-\half\trc \lb(\phi)-\half\trchb L(\phi).
 \end{equation}
 Using again the Ricci equations \eqref{ricciform}, we also have:
 \begin{equation}\label{chudington3}
 \dd^2u(L,\lb)\phi=\lb(L(\phi))-\dd_{\dd_{\lb}L}\phi=\lb(L(\phi))-2\zeta\cdot\nabb\phi-(\d+n^{-1}\nabla_Nn)L(\phi).
 \end{equation}
Finally, \eqref{chudington1}, \eqref{chudington2} and \eqref{chudington3} yield the conclusion of the lemma.
\end{proof}

We conclude this section with the identification of the symbol $\square_{\bf g}u$ of the error term \eqref{err}. In view of \eqref{it2}, \eqref{it3}, the fact that $e_A, A=1, 2$ are tangent to $P_{t,u}$, and the fact that $\lap$ is the Laplace-Beltrami on $P_{t,u}$, we have:
$$L(u)=0, e_A(u)=0, A=1,2, \lap(u)=0,\textrm{ and }\lb(u)=2b^{-1}.$$
Together with \eqref{chudington}, this yields:
\begin{equation}\label{symbolE}
\square_{\bf g}u=b^{-1}\trc.
\end{equation} 
Thus, we may rewrite the error term $E$ as:
\be\lab{err1} 
Ef(t,x)=\int_{\S}\int_{0}^{+\infty}e^{i\lambda u(t,x,\o)}b^{-1}(t,x,\o)\trc(t,x,\o)f(\lambda\o)\lambda^3 d\lambda d\o. 
\ee

\subsection{Commutation formulas}

From the Ricci equations \eqref{ricciform}, we immediately deduce the following four useful commutation formulas:
\begin{lemma}
Let $f$ a scalar function on $\mathcal{M}$. Then,
\be\lab{comm1}
\nabb_B \ddb_{4}f - \ddb_{4}\nabb_B f =\chi_{BC} \nabb_C f - n^{-1} \nabb_B n \ddb_{4}f,
\ee
\be\lab{comm2}
\nabb_B \ddb_{3}f - \ddb_{3}\nabb_Bf =\chb_{BC} \nabb_Cf - \xib_B \ddb_{4}f- b^{-1}\nab_Bb \ddb_{3}f
\ee
\be\lab{comm3}
[\lb, L]f  = -\db  \ddb_{3}f+(\d + n^{-1} \nab_Nn)\ddb_{4}f+2(\z_{B}-\zb_B) \nabb_B f. 
\ee
Finally, \eqref{comm1}, \eqref{comm2} together with the fact that $N=\half(L-\lb)$ yield:
\be\lab{comm4}
\nabb_B \nabb_{N} f - \nabb_{N}\nabb_B f =
(\chi_{BC}+k_{BC}) \nabb_C f - b^{-1} \nabb_B b \nabb_{N} f.
\ee
\end{lemma}

For some applications we have in mind, we would like to get rid of the term containing a $\ddb_4$ derivative in the Right-hand side of \eqref{comm1}. This is achieved by considering the commutator $[\nabb,\ddb_{nL}]$ instead of $[\nabb,\ddb_4]$:
\be\lab{comm5}
\nabb_B \ddb_{nL}f - \ddb_{nL}\nabb_B f = n\chi_{BC} \nabb_C \Pi_{\und{A}}.
\ee
Also, we would like to get rid of the term containing a $\nabb_N$ derivative in the right-hand side of \eqref{comm4}. This is achieved by considering the commutator $[\nabb,\nabb_{bN}]$ instead of $[\nabb,\nabb_N]$:
\be\label{comm7}
\nabb_B \nabb_{bN}f - \nabb_{bN}\nabb_B f = b(\chi_{BC}+k_{BC}) \nabb_C f. 
\ee

\subsection{Regularity assumptions on the phase $u(t,x,\o)$}\lab{sec:regassphase}

We define some norms on $\H$. For any $1\leq p\leq +\infty$ and for any tensor $F$ on $\H_u$, we have:
$$\norm{F}_{\lh{p}}=\left(\int_0^1dt\int_{P_{t,u}}|F|^p\dmt\right)^{\frac{1}{p}},$$
where $\dmt$ denotes the area element of $P_{t,u}$.
We also introduce the following norms:
$$\no(F)=\norm{F}_{\lh{2}}+\norm{\nabb F}_{\lh{2}} +\norm{\ddb_LF}_{\lh{2}},$$
$$\noo(F)=\no(F)+\norm{\nabb^2 F}_{\lh{2}} +\norm{\nabb\ddb_LF}_{\lh{2}}.$$
%The structure equations involve transport equations along the null geodesics generated by $L$ (see for example \eqref{D4a} and \eqref{D4trchi}). 
Let $x'$ a coordinate system on $\pou$. By transporting this coordinate system along the null geodesics generated by $L$, we obtain a coordinate system $(t,x')$ of $\H$. We define the following norms:
$$\norm{F}_{\xt{\infty}{2}}=\sup_{x'\in P_{0,u}}\left(\int_0^1 |F(t,x')|^2dt\right)^{\half},$$
$$\norm{F}_{\xt{2}{\infty}}=\normm{\sup_{0\leq t\leq 1}|F(t,x'))|}_{L^2(\pou)}.$$

We now state our assumptions for the phase $u(t,x,\o)$ and the symbol $b^{-1}(t,x,\o)\trc(t,x,\o)$  of the error term $E$ which is given by \eqref{err1}. These assumptions are compatible with the regularity obtained for the functions $u(t,x,\o)$ constructed in \cite{param3} (this construction corresponds to step {\bf C3}). The constant $\ep>0$ below is the one appearing in the statement of Theorem \ref{th:mainbl2}. In particular, it satisfies $0<\ep<1$ and is small. 

\vspace{0.5cm}

{\em \noindent{\bf Assumption 1} (regularity with respect to $(t,x)$):
\be\lab{estn}
\norm{n-1}_{\lh{\infty}}+\norm{\nabla n}_{\lh{\infty}}+\norm{\nabla^2n}_{\tx{\infty}{2}}+\norm{\nabla T(n)}_{\tx{\infty}{2}}\lesssim\ep,
\ee
\be\lab{estk}
\no(k)+\norm{\ddb_{\lb}\kep}_{\lh{2}}+\norm{\lb(\d)}_{\lh{2}}+\norm{\kep}_{\xt{\infty}{2}}+\norm{\d}_{\xt{\infty}{2}}\lesssim\ep,
\ee
\be\lab{estb}
\norm{b-1}_{\lh{\infty}}+\noo(b)+\norm{L(b)}_{\xt{2}{\infty}}+\norm{\nabb(b)}_{\xt{2}{\infty}}+\norm{\lb(b)}_{\xt{4}{\infty}}\lesssim\ep,
\ee
\be\lab{esttrc}
\norm{\trc}_{\lh{\infty}}+\norm{\nabb\trc}_{\xt{2}{\infty}}+\norm{\lb\trc}_{\xt{2}{\infty}}
\lesssim\ep,
\ee
\be\lab{esthch}
\norm{\hch}_{\xt{2}{\infty}}+\no(\hch)+\norm{\ddb_{\lb}\hch}_{\lh{2}}
\lesssim\ep,
\ee
\be\lab{estzeta}
\norm{\z}_{\xt{2}{\infty}}+\no(\z)\lesssim\ep.
\ee

\noindent {\bf Assumption 2} (regularity with respect to $\o$): 

\be\lab{estNomega}
\norm{\po N}_{\lh{\infty}}\lesssim 1,
\ee
\begin{equation}\label{threomega1ter}
||N(x,\o)-N(x,\o')|-|\o-\o'||\lesssim \ep|\o-\o'|,\,\forall x\in\s, \o,\o'\in\S,
\end{equation}
\be\lab{estricciomega}
\norm{\po b}_{\lh{\infty}}+\norm{\po\chi}_{\xt{2}{\infty}}\lesssim \ep.
\ee
Furthermore, we have the following decomposition for $\hch$:
\be\lab{dechch}
\hch=\chi_1+\chi_2,
\ee
where $\chi_1$ and $\chi_2$ are two symmetric traceless $\ptu$-tangent 2-tensors satisfying: 
\be\lab{dechch1}
\no(\chi_1)+\norm{\po\chi_1}_{\tx{\infty}{2}}+\no(\chi_2)+\norm{\chi_2}_{\xt{\infty}{2}}+\norm{\po\chi_2}_{\tx{\infty}{2}}\lesssim \ep
\ee
and for any $2\leq p<+\infty$, we have:
\be\lab{dechch2}
\norm{\chi_1}_{\tx{p}{\infty}}+\norm{\po\chi_2}_{\lh{6_-}}\lesssim \ep.
\ee

\noindent {\bf Assumption 3} (additional regularity with respect to $x$): 

\noindent We introduce the family of intrinsic Littlewood-Paley projections
$P_j$ which have been constructed in \cite{LP} using the heat flow 
on the 2-surfaces $P_{t,u}$ (see section \ref{sec:LPproperties}). There exists a function $\mu$ in $L^2(\R)$ satisfying:
$$\norm{\mu}_{L^2(\R)}\leq 1$$
such that for all $j\geq 0$, we have:
\be\lab{estlblbtrc}
\norm{\nabn P_j\nabn\trc}_{\lh{2}}\lesssim 2^j\ep+2^{\frac{j}{2}}\ep\mu(u).
\ee}

\begin{remark}
In {\bf Assumptions 1-3}, all inequalities hold for any $\o\in\S$ with the constant in the right-hand side being independent of $\o$. Thus, one may take the supremum in $\o$ everywhere. To ease the notations, we do not explicitly write down this supremum. 
\end{remark}

\begin{remark}\label{smallreduc}
The fact that we may take a small constant $\ep>0$ in  {\bf Assumptions 1-3} is directly related to the conclusions of  Theorem \ref{th:mainbl2}.
\end{remark}

\begin{remark}
In the flat case, we have $\mathcal{M}=(\R^{1+3},{\bf m})$,  where ${\bf m}$ is the Minkowski metric, $u(t,x,\o)=t+\xo$, $b= 1$, $N=-\o$, $L=\partial_t-\o\cdot\partial_x$, $\lb=\partial_t+\o\cdot\partial_x$, and $\chi=\chb=\zeta=\zb=\xib=k=0$. Thus, {\bf Assumptions 1-3} are clearly satisfied with $\ep=0$.
\end{remark}

\subsection{Estimates on $\ptu$ and $\MM$}\lab{sec:embeddingass}

In this section, we state the embeddings on $\ptu$ and $\MM$ that will be needed for the proof of the main theorem. We refer to section 3 in \cite{param3}, as well as \cite{LP} for \eqref{eq:GNirenberg}, for their proof within the regularity assumptions of section \ref{sec:regassphase}.\\

We have the following Gagliardo-Nirenberg inequality on $\ptu$ (see \cite{LP}). For an arbitrary tensorfield  $F$ on $\ptu$ and any $2\leq p<\infty$, we have:
\be\lab{eq:GNirenberg}
\|F\|_{L^p(\ptu)}\lesssim \|\nabb F\|_{L^2(\ptu)}^{1-\frac{2}{p}}\|F\|_{L^2(\ptu)}^{\frac 2 p}+\|F\|_{L^2(\ptu)}.
\ee

We have the classical Sobolev inequality on $\H$ (see \cite{param3}):
\begin{lemma}
For any tensor $F$ on $\H_u$, we have:
\be\lab{sobineq}
\norm{F}_{\lh{6}}\lesssim \no(F),
\ee
and
\be\lab{sobineq1}
\norm{F}_{\tx{\infty}{4}}\lesssim \no(F).
\ee
\end{lemma}

On $\H_u$, we also have the following estimate of the $\tx{\infty}{2}$ norm for any tensor $F$ on $\H_u$:
\be\lab{murray}
\norm{F}_{\xt{2}{\infty}}^2\les \int_{\H_u}|F||\dd_LF| dt \dmt+\norm{F}^2_{L^2(\H_u)}. 
\ee

The following lemma will be useful to estimate transport equations (see \cite{param3} for a proof).
\begin{lemma}
Let $W$ and $F$ two $\ptu$-tangent tensors such that $\ddb_LW=F$. Then, for any $p\geq 1$, we have:
\be\lab{estimtransport1}
\norm{W}_{\xt{p}{\infty}}\lesssim \norm{W(0)}_{L^p(\pou)}+\norm{F}_{\xt{p}{1}}.
\ee
\end{lemma}

%We have the Sobolev embedding on $\Si_t$. Given an arbitrary tensorfield $F$ on $\Sit$, we have
%\be\lab{sobineqsit}
%\norm{F}_{L^{6}(\Sit)}\lesssim \norm{\nabla F}_{L^{2}(\Sit)}.
%\ee 
%See \cite{param3} for a proof.

Finally, we have the Sobolev embedding on $\MM$. Given an arbitrary tensorfield $F$ on $\MM$, we have (see \cite{param3})
\be\lab{sobineqm}
\norm{F}_{L^4(\MM)}\lesssim \norm{\dd F}_{L^{2}(\MM)}.
\ee 

\subsection{Geometric Littlewood-Paley projections on $P_{t,u}$}\lab{sec:LPproperties}

In \cite{LP}, Littlewood-Paley projections have been constructed relying on the heat flow on 2-surfaces, within our low regularity assumptions. They recover the basic properties of the standard Littlewood-Paley projections. We denote by $P_j$ such a Littlewood-Paley projection on the 2-surface $\ptu$. In particular, we have from \cite{LP}
\be\lab{eq:partition}
\sum_jP_j=I.
\end{equation}

Also, the following properties of the LP-projections $P_j$ have been proved in \cite{LP}:
\begin{theorem}\label{thm:LP}
 The LP-projections $P_j$ verify the following
 properties:

i)\quad {\sl $L^p$-boundedness} \quad For any $1\leq 
p\leq \infty$, and any interval $I\subset \Bbb Z$,
\be\lab{eq:pdf1}
\|P_IF\|_{\lpt{p}}\lesssim \|F\|_{\lpt{p}}
\end{equation}

ii) \quad  {\sl Bessel inequality} 
$$\sum_j\|P_j F\|_{\lpt{2}}^2\lesssim \|F\|_{\lpt{2}}^2$$

iii)\quad {\sl Finite band property}\quad For any $1\leq p\leq \infty$.
\begin{equation}
\begin{array}{lll}
\|\lap P_j F\|_{\lpt{p}}&\lesssim & 2^{2j} \|F\|_{\lpt{p}}\\
\|P_jF\|_{\lpt{p}} &\lesssim & 2^{-2j} \|\lap F \|_{\lpt{p}}.
\end{array}
\end{equation}

In addition, the $L^2$ estimates
\begin{equation}
\begin{array}{lll}
\|\nabb P_j F\|_{\lpt{2}}&\lesssim & 2^{j} \|F\|_{\lpt{2}}\\
\|P_jF\|_{\lpt{2}} &\lesssim & 2^{-j} \|\nabb F  \|_{\lpt{2}}
\end{array}
\end{equation}
hold together with the dual estimate
$$\| P_j \nabb F\|_{\lpt{2}}\lesssim 2^j \|F\|_{\lpt{2}}$$

iv) \quad{\sl Weak Bernstein inequality}\quad For any $2\leq p<\infty$
\begin{align*}
&\|P_j F\|_{\lpt{p}}\lesssim (2^{(1-\frac 2p)j}+1) \|F\|_{\lpt{2}},\\
&\|P_{<0} F\|_{\lpt{p}}\lesssim \|F\|_{\lpt{2}}
\end{align*}
together with the dual estimates 
\begin{align*}
&\|P_j F\|_{\lpt{2}}\lesssim (2^{(1-\frac 2p)j}+1) \|F\|_{\lpt{p'}},\\
&\|P_{<0} F\|_{\lpt{2}}\lesssim \|F\|_{\lpt{p'}}
\end{align*}
\end{theorem}

We also have the following sharp Bernstein inequality for a scalar function $f$ on $\ptu$:
\bea
\|P_j f\|_{\lpt{\infty}}&\lesssim & 2^j\|f\|_{\lpt{2}},\label{eq:strongbernscalarbis}\\
\|P_{<0} f\|_{\lpt{\infty}}&\lesssim &  
 \|f\|_{\lpt{2}},\label{eq:strong-Bern-0bis}
\eea
and the following Bochner inequality:
\begin{equation}\label{eq:Bochconseqbis}
\int_{\ptu} |\nabb^2 f|^2\lesssim \int_{\ptu} |\lap f|^2  + \ep\int_{\ptu} |\nabb f|^2.
\end{equation}
There is an equivalent of \eqref{eq:Bochconseqbis} for tensors, which yields 
the following consequence. There exists a function $\mu$ in $L^2(\mathbb{R})$ satisfying:
$$\norm{\mu}_{L^2(\mathbb{R})}\les 1$$
such that for any scalar function $f$ on $\ptu$, we have:
\be\lab{eq:Bochconseqter}
\norm{\nabb^3f}_{L^2(\ptu)}\les \norm{\nabb\lap f}_{L^2(\ptu)}+\mu(t)\ep\norm{\lap f}_{L^2(\ptu)}+\mu^2(t)\ep\norm{\nabb f}_{L^2(\ptu)}.
\ee

Finally, we have the following lemma (see Lemma 5.10 in \cite{param3}):
\begin{lemma}\lab{lemma:lbz5}
For any 1-form $F$ on $\ptu$, for any $1<p\leq 2$ and for all $j\geq 0$, we have:
\be\lab{lbz14bis}
\norm{P_j\divb(F)}_{\lpt{2}}\les 2^{\frac{2}{p}j}\norm{F}_{\lpt{p}}.
\ee
\end{lemma}

\subsection{Commutator estimates}

In this section, we state the commutator estimates, as well as two additional estimates for $\trc$, that will be needed for the proof of the main theorem. We refer to section 9 in \cite{param3} for their proof within the regularity assumptions of section \ref{sec:regassphase}.\\

Let $f$ a scalar function on $\MM$. Then, we have the following commutator estimates:
\be\lab{commlp1}
\norm{[bN,P_j]f}_{L^2(\H_u)}+2^{-j}\norm{\nabb[bN,P_j]f}_{L^2(\H_u)}\les \ep\no(f).
\ee
and
\be\lab{commlp2}
\norm{[nL,P_j]f}_{L^2(\H_u)}+2^{-j}\norm{\nabb[nL,P_j]f}_{L^2(\H_u)}\les \ep\no(f).
\ee

\vspace{0.3cm}

We also have the following commutator estimates acting on $\trc$. 
\be\lab{commlp3}
2^j\norm{[nL,P_j]\trc}_{\tx{1}{2}}+\norm{\nabb[nL,P_j]\trc}_{\tx{1}{2}}\les\ep,
\ee
\be\lab{commlp3ter}
2^{\frac{j}{2}}\norm{[nL,P_j]\trc}_{\li{\infty}{2}}+2^{-\frac{j}{2}}\norm{\nabb[nL,P_j]\trc}_{\li{\infty}{2}}\les\ep,
\ee
and
\be\lab{commlp3bis}
2^j\norm{[bN,P_j]\trc}_{\li{\infty}{2}}+\norm{\nabb[bN,P_j]\trc}_{\li{\infty}{2}}\les 2^{\frac{j}{2}}\ep.
\ee

%For any tensors $F$ and $H$ on $\ptu$, we have:
%\be\lab{commlp4}
%\norm{[H,P_j]F}_{L^2(\ptu)}\les 2^{-j}\norm{\nabb H}_{L^2(\ptu)}\norm{\nabb F}_{L^2(\ptu)},
%\ee
%and:
%\be\lab{commlp5}
%\norm{[H,P_j]F}_{L^2(\ptu)}\les \norm{\nabb H}_{L^2(\ptu)}\norm{F}_{L^2(\ptu)}.
%\ee
%Also, we have:
%\be\lab{commlp6}
%\norm{[\nabb,P_j]\trc}_{\tx{2}{4}}\les \ep.
%\ee

\vspace{0.3cm}

Finally, we have the following estimate for $P_m\trc$:
\be\lab{lievremont1}
\norm{P_m\trc}_{\xt{2}{\infty}}+\norm{P_m(nL\trc)}_{\xt{2}{1}}\les 2^{-m}\ep,
\ee
and the following estimate for $P_{\leq m}\trc$:
\be\lab{lievremont2}
\norm{\nabb P_{\leq m}\trc}_{\xt{2}{\infty}}+\norm{\nabb (P_{\leq m}(nL\trc))}_{\xt{2}{1}}\les\ep.
\ee

\subsection{Dependance of the norm $L^\infty_u\lh{2}$ on $\o\in\S$}\label{sec:depomega}

Let $\o$ and $\nu$ in $\S$ such that 
$$|\o-\nu|\les 2^{\frac{j}{2}}.$$ 
Let $u=u(.,\o)$ and $u_\nu=u(.,\nu)$. In this section, we morally evaluate the norm in $\li{\infty}{2}$ of the difference between various scalars and tensors evaluated at $\o$ and their corresponding evaluation at $\nu$. Consider a FIO where the integration in $\o$ is localized in a patch of size $2^{-\frac{j}{2}}$ and of center $\nu$. This will be used to morally replace the symbol of this FIO depending on $\o$ by its value at the middle of the patch $\nu$, so that one may take the symbol outside of the integral in $\o$. The following decompositions are proved in section 8 of \cite{param3} within the regularity assumptions of section \ref{sec:regassphase}.\\

\begin{itemize}
\item We have the following decomposition for $N(.,\o)-N(.,\nu)$:
\be\lab{decNom}
2^{\frac{j}{2}}(N(.,\o)-N(.,\nu))=F^j_1+F^j_2
\ee
where the tensor $F^j_1$ only depends on $\nu$ and satisfies:
$$\norm{F^j_1}_{L^\infty}\les 1,$$
and where the tensor $F^j_2$ satisfies:
$$\norm{F^j_2}_{L^\infty_u\lh{2}}\les 2^{-\frac{j}{2}}.$$

\item We have following decomposition for $\trc$:
\be\lab{dectrcom}
\trc(.,\o)=f^j_1+f^j_2
\ee
where the scalar $f^j_1$ only depends on $\nu$ and satisfies:
$$\norm{f^j_1}_{L^\infty}\les \ep,$$
and where the scalar $f^j_2$ satisfies:
$$\norm{f^j_2}_{L^\infty_u\lh{2}}\les \ep 2^{-\frac{j}{2}}.$$

\item Let $p\in\mathbb{Z}$. We have following estimate for $b^p$:
\be\lab{decbom}
\norm{b^p(.,\o)-b^p(.,\nu)}_{L^\infty_u\lh{2}}\les \ep|\o-\nu|.
\ee

\item We have following decomposition for $\hch$:
\be\lab{dechchom}
\hch(.,\o)=F^j_1+F^j_2
\ee
where the tensor $F^j_1$ only depends on $\nu$ and satisfies:
$$\norm{F^j_1}_{L^\infty_{u_\nu}, x'_{\nu}L^2_t}\les \ep,$$
and where the tensor $F^j_2$ satisfies:
$$\norm{F^j_2}_{L^\infty_u\lh{2}}\les \ep 2^{-\frac{j}{2}}.$$

\item We have following decomposition for $\chi_2$:
\be\lab{decchi2om}
\norm{\chi_2(.,\o)-{\chi_2}(.,\nu)}_{L^\infty_u\lh{4_-}}\les \ep|\o-\nu|.
\ee

\item We have following decomposition for $\chi$:
\be\lab{decchiom}
\chi(.,\o)={\chi_2}(.,\nu)+F^j_1+F^j_2
\ee
where the tensor $F^j_1$ only depends on $\nu$ and satisfies for any $2\leq p<+\infty$:
$$\norm{F^j_1}_{L^\infty_{u_\nu}L^p_t L^\infty_{x'_{\nu}}}\les \ep,$$
and where the tensor $F^j_2$ satisfies:
$$\norm{F^j_2}_{L^\infty_u\lh{2}}\les \ep 2^{-\frac{j}{2}}.$$

\item We have following decomposition for $|\hch|^2$:
\be\lab{dechch2om}
|\hch|^2(.,\o)=|{\chi_2}(.,\nu)|^2+{\chi_2}(.,\nu)\c F^j_1+{\chi_2}(.,\nu)\c F^j_2+f^j_3+f^j_4+f^j_5,
\ee
where the tensor $F^j_1$ and the scalar $f^j_3$ only depends on $\nu$ and satisfy:
$$\norm{F^j_1}_{L^\infty_{u_\nu}L^2_t L^\infty(P_{t,u_\nu})}+\norm{f^j_3}_{L^\infty_{u_\nu}L^2_t L^\infty(P_{t,u_\nu})}\les \ep,$$
where the tensor $F^j_2$, and the scalar $f^j_4$ satisfy:
$$\norm{F^j_2}_{L^\infty_u\lh{2}}+\norm{f^j_4}_{L^\infty_u\lh{2}}\les \ep 2^{-\frac{j}{2}},$$
and where the scalar $f^j_5$ satisfies:
$$\norm{f^j_5}_{L^2(\MM)}\les \ep 2^{-j}.$$

\item We have the following decomposition for $\hch(.,\o)^3$:
\bea\lab{dechch3om}
\hch(.,\o)^3&=&\chi_2(.,\nu)^3+\chi_2(.,\nu)^2F^j_1+\chi_2(.,\nu)^2F^j_2+\chi_2(.,\nu)F^j_3\\
\nn&&+\chi_2(.,\nu)F^j_4+\chi_2(.,\nu)F^j_5+F^j_6+F^j_7+F^j_8+F^j_9
\eea
where $F^j_1$, $F^j_3$ and $F^j_6$ do not depend on $\o$ and satisfy:
$$\norm{F^j_1}_{L^\infty_{u_\nu}L^2_tL^\infty(P_{t, u_\nu})}+\norm{F^j_3}_{L^\infty_{u_\nu}L^2_tL^\infty(P_{t, u_\nu})}+\norm{F^j_6}_{L^\infty_{u_\nu}L^2_tL^\infty(P_{t, u_\nu})}\les \ep,$$
where $F^j_2$, $F^j_4$ and $F^j_7$ satisfy:
$$\norm{F^j_2}_{L^\infty_u\lh{2}}+\norm{F^j_4}_{L^\infty_u\lh{2}}+\norm{F^j_7}_{L^\infty_u\lh{2}}\les 2^{-\frac{j}{2}}\ep,$$
where $F^j_5$ and $F^j_8$ satisfy
$$\norm{F^j_5}_{L^2(\mathcal{M})}+\norm{F^j_8}_{L^2(\mathcal{M})}\les \ep 2^{-j}.$$
and where $F^j_9$ satisfies
$$\norm{F^j_9}_{L^{2_-}(\mathcal{M})}\les \ep 2^{-\frac{3j}{2}}.$$

\item We have the following decomposition for $b(.,\o)-b(.,\nu)$:
\be\lab{decbpom}
2^{\frac{j}{2}}(b(.,\o)-b(.,\nu))=f^j_1+f^j_2
\ee
where the scalar $f^j_1$ only depends on $\nu$ and satisfies:
$$\norm{f^j_1}_{L^\infty}\les \ep,$$
and where the scalar $f^j_2$ satisfies:
$$\norm{f^j_2}_{L^\infty_u\lh{2}}\les 2^{-\frac{j}{4}}\ep.$$

\item We have following decomposition for $\z$ and $\nabb(b)$:
\be\lab{deczetaom}
\z(.,\o), \nabb(b)(.,\o)=F^j_1+F^j_2
\ee
where the tensor $F^j_1$ only depends on $\nu$ and satisfies for any $2\leq p<+\infty$:
$$\norm{F^j_1}_{L^\infty_{u_\nu}L^2_t,L^p_{x'_{\nu}}}\les \ep,$$
and where the tensor $F^j_2$ satisfies:
$$\norm{F^j_2}_{L^\infty_u\lh{2}}\les \ep 2^{-\frac{j}{4}}.$$
\end{itemize}

\begin{remark}
Let us give some insight on these decompositions by considering the particular example of the decomposition for $\trc$ \eqref{dectrcom}. A naive approach consists in writing the following decomposition 
$$\trc(t,x,\o)=\trc(t,x,\nu)+(\trc(t,x,\o)-\trc(t,x,\nu))=f_1^j+f_2^j.$$
$f^j_1$ does not depend on $\o$ and satisfies, in view of the estimate \eqref{esttrc}
$$\norm{f^j_1}_{L^\infty}\les \norm{\trc(.,\nu)}_{L^\infty}\les\ep.$$
Also, we have
$$f^j_2=(\o-\nu)\int_0^1\po\trc(t,x,\o_\sigma)d\sigma,$$
which together with the fact that $|\o-\nu|\les 2^{-\frac{j}{2}}$ yields
$$\norm{f^j_2}_{\li{\infty}{2}}\les 2^{-\frac{j}{2}}\normm{\int_0^1\po\trc(t,x,\o_\sigma)d\sigma}_{\li{\infty}{2}}.$$
Unfortunately, we can not obtain the desired estimate for $f^j_2$ since we have $\po\trc(.,\o_\sigma)\in L^\infty_{u_\sigma}L^2(\H_{u_\sigma})$, and $\li{\infty}{2}$ and $L^\infty_{u_\sigma}L^2(\H_{u_\sigma})$ are not directly comparable. Nevertheless, in section 8 of \cite{param3}, we are able to improve on this naive approach, in order to obtain the above decompositions within the regularity assumptions of section \ref{sec:regassphase}.
\end{remark}

\subsection{The boundedness of the error term}

The main result of this paper is the following $L^2$ bound on the error term \eqref{err1}. It achieves step {\bf C4} and therefore, together with the results in \cite{param1}, \cite{param2}, \cite{param3} completes step {\bf C}.
\begin{theorem}\label{th1}
Let $u$ be a function on $\mathcal{M}\times\S$ satisfying {\bf Assumption 1}, {\bf Assumption 2} and {\bf Assumption 3}, as well as the assumptions of sections \ref{sec:embeddingass}-\ref{sec:depomega}. Let $E$ the Fourier integral operator with phase $u(t,x,\o)$ and symbol $b^{-1}(t,x,\o)\trc(t,x,\o)$:
\be\lab{err2} 
Ef(t,x)=\int_{\S}\int_{0}^{+\infty}e^{i\lambda u(t,x,\o)}b^{-1}(t,x,\o)\trc(t,x,\o)f(\lambda\o)\lambda^3 d\lambda d\o. 
\ee
Then, $E$ satisfies the estimate:
\begin{equation}\label{l2}
\norm{Ef}_{L^2(\mathcal{M})}\lesssim \ep\norm{\lambda f}_{L^2(\mathbb{R}^3)}.
\end{equation}
\end{theorem}

\section{Proof of Theorem \ref{th1} (control of the error term)}\label{sec:th1}

\subsection{The basic computation}

We start the proof of Theorem \ref{th1} with the following instructive computation:
\bea\label{bisb5}
\norm{Ef}_{\ll{2}}&\leq &\ds\int_{\S}\normm{b(t,x,\o)^{-1}\trc(t,x,\o)\int_{0}^{+\infty}e^{i\lambda u}f(\lambda\o)\lambda^2 d\lambda}_{\ll{2}}d\o\\
\nn&&\ds\leq\int_{\S}\norm{b(t,x,\o)^{-1}\trc(t,x,\o)}_{\li{\infty}{2}}\normm{\int_{0}^{+\infty}e^{i\lambda u}f(\lambda\o)\lambda^2 d\lambda}_{L^2_{u}}d\o\\
\nn&&\ds\leq \ep\norm{\lambda^2 f}_{L^2(\R^3)},
\eea
where we have used Plancherel with respect to $\lambda$, Cauchy-Schwarz with respect to $\o$, 
and the estimates \eqref{estb} for $b$ and \eqref{esttrc} for $\trc$. \eqref{bisb5} misses the conclusion \eqref{l2} of Theorem \ref{th1} by a power of $\lambda$. Now, assume for a moment that we may replace a power of $\lambda$ by a derivative on $b(t,x,\o)^{-1}\trc(t,x,\o)$. Then, the same computation yields:
\begin{equation}\label{bisb6}
\begin{array}{ll}
&\ds\normm{\int_{\S}\int_{0}^{+\infty}\dd(b(t,x,\o)^{-1}\trc(t,x,\o)) e^{i\lambda u}f(\lambda\o)\lambda d\lambda d\o}_{\ll{2}}\\
\ds\leq &\ds \int_{\S}\norm{\dd(b(t,x,\o)^{-1}\trc(t,x,\o))}_{\li{\infty}{2}}\normm{\int_{0}^{+\infty}e^{i\lambda u}f(\lambda\o)\lambda^2 d\lambda}_{L^2_{ u}}d\o\\
\ds\leq &\ds \ep\norm{\lambda f}_{L^2(\R^3)},
\end{array}
\end{equation}
which is \eqref{l2}. This suggests a strategy which consists in making integrations by parts to trade powers of $\lambda$ against derivatives of the symbol $b(t,x,\o)^{-1}\trc(t,x,\o)$. 

\subsection{Structure of the proof of Theorem \ref{th1}}

The proof of Theorem \ref{th1} proceeds in three steps. We first localize in frequencies of size $\la\sim 2^j$. We then localize the angle $\o$ in patches on the sphere $\S$ of diameter $2^{-j/2}$. Finally, we estimate the diagonal terms.

\subsubsection{Step 1: decomposition in frequency}
 
For the first step, we introduce $\varphi$  and $\psi$ two smooth compactly supported functions on $\R$ such that: 
\begin{equation}\label{bisb7}
\varphi(\lambda)+\sum_{j\geq 0}\psi(2^{-j}\lambda)=1\textrm{ for all }\lambda\in\R.
\end{equation}
We use \eqref{bisb7} to decompose $Ef$ as follows:
\begin{equation}\label{bisb8}
Ef(t,x)=\sum_{j\geq -1}E_jf(t,x),
\end{equation}
where for $j\geq 0$:
\begin{equation}\label{bisb9}
E_jf(t,x)=\int_{\S}\int_{0}^{+\infty}e^{i\lambda u}b(t,x,\o)^{-1}\trc(t,x,\o)\psi(2^{-j}\lambda)f(\lambda\o)\lambda^2 d\lambda d\o,
\end{equation}
and 
\begin{equation}\label{bisb10}
E_{-1}f(t,x)=\int_{\S}\int_{0}^{+\infty}e^{i\lambda u}b(t,x,\o)^{-1}\trc(t,x,\o)\varphi(\lambda)f(\lambda\o)\lambda^2 d\lambda d\o.
\end{equation}
This decomposition is classical and is known as the first dyadic decomposition (see \cite{stein}). The goal of this first step is to prove the following proposition:
\begin{proposition}\label{bisorthofreq}
The decomposition \eqref{bisb8} satisfies an almost orthogonality property:
\begin{equation}\label{bisorthofreq1}
\norm{Ef}_{\ll{2}}^2\lesssim\sum_{j\geq -1}\norm{E_jf}_{\ll{2}}^2+\ep^2\norm{f}^2_{L^2(\R^3)}.
\end{equation}
\end{proposition}
The proof of Proposition \ref{bisorthofreq} is postponed to section \ref{bissec:orthofreq}. 

\subsubsection{Step 2: decomposition in angle}

Proposition \ref{bisorthofreq} allows us to  estimate $\norm{E_jf}_{\ll{2}}$ instead of 
$\norm{Ef}_{\ll{2}}$. The analog of computation \eqref{bisb5} for $\norm{E_jf}_{\ll{2}}$ yields:
\begin{equation}\label{bisb5bis}
\begin{array}{l}
\ds\norm{E_jf}_{\ll{2}}\leq \ep\norm{\lambda\psi(2^{-j}\la) f}_{L^2(\s)}\lesssim \ep2^j\norm{\psi(2^{-j}\la) f}_{\le{2}},
\end{array}
\end{equation}
which misses the wanted estimate by a power of $2^j$. We thus need to perform a second dyadic decomposition (see \cite{stein}). We introduce a smooth partition of unity on the sphere $\S$:
\begin{equation}\label{bisb14}
\sum_{\nu\in\Gamma}\eta^\nu_j(\o)=1\textrm{ for all }\o\in\S, 
\end{equation}
where the support of $\eta^\nu_j$ is a patch on $\S$ of diameter $\sim 2^{-j/2}$. We use \eqref{bisb14} to decompose $E_jf$ as follows:
\begin{equation}\label{bisb15}
E_jf(t,x)=\sum_{\nu\in\Gamma}E^\nu_jf(t,x),
\end{equation}
where:
\begin{equation}\label{bisb16}
E^\nu_jf(t,x)=\int_{\S}\int_{0}^{+\infty}e^{i\lambda u}b(t,x,\o)^{-1}\trc(t,x,\o)\psi(2^{-j}\lambda)\eta^\nu_j(\o)f(\lambda\o)\lambda^2 d\lambda d\o.
\end{equation}
We also define:
\begin{equation}\label{bisdecf}
\begin{array}{l}
\ds\ga_{-1}=\norm{\varphi(\la)f}_{\le{2}},\, \ga_j=\norm{\psi(2^{-j}\la)f}_{\le{2}},\,j\geq 0, \\
\ds\ga^\nu_j=\norm{\psi(2^{-j}\la)\eta^\nu_j(\o)f}_{\le{2}},\,j\geq 0,\,\nu\in\Gamma, 
\end{array}
\end{equation}
which satisfy:
\begin{equation}\label{bisdecf1}
\norm{f}_{\le{2}}^2=\sum_{j\geq -1}\ga_j^2=\sum_{j\geq -1}\sum_{\nu\in\Gamma}(\ga^\nu_j)^2.
\end{equation}
The goal of this second step is to prove the following proposition:
\begin{proposition}\label{bisorthoangle}
The decomposition \eqref{bisb15} satisfies an almost orthogonality property:
\begin{equation}\label{bisorthoangle1}
\norm{E_jf}_{\ll{2}}^2\lesssim\sum_{\nu\in\Gamma}\norm{E^\nu_jf}_{\ll{2}}^2+\ep^2\ga_j^2.
\end{equation}
\end{proposition}
The proof of Proposition \ref{bisorthoangle} is postponed to section \ref{bissec:orthoangle}. 

\subsubsection{Step 3: control of the diagonal term}

Proposition \ref{bisorthoangle} allows us to  estimate $\norm{E^\nu_jf}_{\ll{2}}$ instead of $\norm{E_jf}_{\ll{2}}$. The analog of computation \eqref{bisb5} for $\norm{E^\nu_jf}_{\ll{2}}$ yields:
\bea\label{bisb5ter}
&&\norm{E_j^\nu f}_{\ll{2}}\\
\nn&\leq &\ds\int_{\S}\norm{b(t,x,\o)^{-1}\trc(t,x,\o)}_{\l{\infty}{2}}\normm{\int_{0}^{+\infty}e^{i\lambda  u}\psi(2^{-j}\la)\eta^\nu_j(\o)f(\lambda\o)\lambda^2 d\lambda}_{L^2_{ u}}d\o\\
\nn&\leq &\ds \ep\sqrt{\textrm{vol}(\textrm{supp}(\eta^\nu_j))}\norm{\lambda\psi(2^{-j}\la)\eta^\nu_j(\o)f}_{\le{2}}\\
\nn&\lesssim &\ds \ep2^{j/2}\gamma^\nu_j,
\eea
where the term $\sqrt{\textrm{vol}(\textrm{supp}(\eta^\nu_j))}$ comes from the fact that we apply Cauchy-Schwarz in $\o$. Note that we have used in \eqref{bisb5ter} the fact that the support of $\eta^\nu_j$ is 2 dimensional and has diameter $2^{-j/2}$ so that:
\begin{equation}\label{bissuppangle}
\sqrt{\textrm{vol}(\textrm{supp}(\eta^\nu_j))}\lesssim 2^{-j/2}.
\end{equation}
Now, \eqref{bisb5ter} still misses the wanted estimate by a power of $2^{j/2}$. Nevertheless, we are able to estimate the diagonal term:
\begin{proposition}\label{bisdiagonal}
The diagonal term $E^\nu_jf$ satisfies the following estimate:
\begin{equation}\label{bisdiagonal1}
\norm{E^\nu_jf}_{\ll{2}}\lesssim \ep\ga^\nu_j.
\end{equation}
\end{proposition}
The proof of Proposition \ref{bisdiagonal} is postponed to section \ref{bissec:diagonal}. 

\subsubsection{Proof of Theorem \ref{th1}}

Proposition \ref{bisorthofreq}, \ref{bisorthoangle} and \ref{bisdiagonal} 
immediately yield the proof of Theorem \ref{th1}. Indeed, \eqref{bisorthofreq1}, \eqref{bisdecf1}, \eqref{bisorthoangle1} and \eqref{bisdiagonal1} imply:
\begin{equation}\label{biscclth3}
\begin{array}{ll}
\ds\norm{Ef}_{\ll{2}}^2 & \ds\lesssim\sum_{j\geq -1}\norm{E_jf}_{\ll{2}}^2+\ep^2\norm{f}^2_{\le{2}}\\
& \ds\lesssim\sum_{j\geq -1}\sum_{\nu\in\Gamma}\norm{E^\nu_jf}_{\ll{2}}^2+\ep^2\sum_{j\geq -1}\gamma_j^2+\ep^2\norm{f}^2_{\le{2}}\\
& \ds\lesssim \ep^2\sum_{j\geq -1}\sum_{\nu\in\Gamma}(\gamma_j^\nu)^2+\ep^2\sum_{j\geq -1}\gamma_j^2+\ep^2\norm{f}^2_{\le{2}}\\
& \ds\lesssim \ep^2\norm{f}^2_{\le{2}},
\end{array}
\end{equation}
which is the conclusion of Theorem \ref{th1}. \QED\\

The rest of the paper is dedicated to the proof of Proposition \ref{bisorthofreq}, \ref{bisorthoangle} and \ref{bisdiagonal}. In section \ref{bissec:orthofreq}, we prove Proposition \ref{bisorthofreq}. In section \ref{bissec:diagonal}, we prove Proposition \ref{bisdiagonal}. Finally, we turn to the proof of Proposition \ref{bisorthoangle} which constitutes the most technical part of this paper and occupies sections \ref{bissec:orthoangle} to \ref{sec:tsonga3}.

\section{Proof of Proposition \ref{bisorthofreq} (almost orthogonality in frequency)}\label{bissec:orthofreq}

We have to prove \eqref{bisorthofreq1}:
\begin{equation}\label{bisof1}
\norm{Ef}_{\ll{2}}^2\lesssim\sum_{j\geq -1}\norm{E_jf}_{\ll{2}}^2+\ep^2\norm{f}^2_{\le{2}}.
\end{equation}
This will result from the following inequality using Shur's Lemma:
\begin{equation}\label{bisof2}
\left|\int_{\MM}E_jf(t,x)\overline{E_kf(t,x)}d\MM \right| \lesssim \ep^22^{-\frac{|j-k|}{4}}\gamma_j\ga_k\textrm{ for }|j-k|> 2.
\end{equation}

Before we proceed with the proof of \eqref{bisof2}, let us recall that 
the volume element on $\MM$ expressed in the coordinate system $(u,t,x')$ is given by:
$$d\MM=nb\, du\, dt\, \dmt,$$
where $\dmt$ denotes the volume element on $\ptu$. Since we have:
$$\norm{n-1}_{L^\infty}\les\ep,$$
from the estimates for $n$ \eqref{estn}, the $L^2(\MM)$ norm of a tensor $F$ on $\MM$ is equivalent to:
\be\lab{defL2norm}
\int_{\MM}|F(t,x)|^2b\, du\, dt\, \dmt,
\ee
where we have removed the lapse $n$ in the definition of the volume element. To avoid  unnecessary terms containing derivatives of $n$ in the numerous integrations by parts of this paper, we will estimate the equivalent \eqref{defL2norm} of the $L^2(\MM)$ norm for $Ef$, $E_jf$ and $E_j^\nu f$. By a slight abuse of notation which we shall do throughout the paper, this is equivalent to modifying the volume element by removing the time lapse $n$:  
\be\lab{coarea}
d\MM=b\, du\, dt\, \dmt.
\ee

\begin{remark}
One may want to further simplify the expression of the volume element \eqref{coarea} by removing the null lapse $b$. However, this is not possible since the decomposition \eqref{coarea} depends of the angle $\o\in\S$ under consideration. Indeed, the coordinate system is $(u,t,x')$ where $u=u(t,x,\o)$ and thus $b=b(t,x,\o)$ in \eqref{coarea}, while the time lapse $n$ is independent of $\o$.
\end{remark}

Also, recall that we have for all integrable scalar functions $f$:
\be\lab{dum}
\frac{d}{du}\left(\int_{\ptu}f\dmt \right)=\int_{\ptu}b(\nabla_Nf+\textrm{tr}\theta)\dmt
\ee
where $\theta$ is the second fundamental form of $\ptu$ in $\Sit$, i.e. $\theta_{ij}=\nabla_iN_j$. Note that from the definition of $k$, $\chi$ and $\theta$, we have:
\be\lab{hehehehe1}
\chi_{AB}=<\dd_AL,e_B>=<\nabla_AT,e_B>+<\nabla_AN,e_B>=-k_{AB}+\theta_{AB}.
\ee
Together with the estimate \eqref{esttrc} \eqref{esthch} for $\chi$ and \eqref{estk} for $k$, we obtain:
\be\lab{estth}
\no(\th)\les\ep.
\ee
Finally, we recall that we have:
\be\lab{estN}
\left\{\begin{array}{l}
\nabla_AN=\th_{AB}e_B,\\
\nabn N=-b^{-1}\nabb b.
\end{array}\right.
\ee

\subsection{A first integration by parts} 

From now on, we focus on proving \eqref{bisof2}. We may assume $j\geq k+3$. We have:
\bea\label{bisof3}
&&\ds\int_{\MM}E_jf(t,x)\overline{E_kf(t,x)}d\MM \\
\nn&= & \ds\int_{\S}\int_{0}^{+\infty}\int_{\S}\int_{0}^{+\infty}\left(\int_{\MM}e^{i\lambda u-i\la' u'}b(t,x,\o)^{-1}\trc(t,x,\o)\overline{b(x,\o')^{-1}\trc(t,x,\o')}d\MM\right)\\
\nn&& \ds\times\psi(2^{-j}\lambda)f(\lambda\o)\lambda^2 \psi(2^{-k}\lambda')\overline{f(\lambda'\o')}(\lambda')^2 d\lambda d\o d\lambda' d\o'.
\eea

In view of the expression of the volume element \eqref{coarea} on $\MM$, we integrate by parts with respect to $\partial_{ u}$ in 
\bee
&&\int_{\MM}e^{i\lambda u-i\la' u'}b(t,x,\o)^{-1}\trc(t,x,\o)\overline{b(x,\o')^{-1}\trc(t,x,\o')}d\MM\\
&=& \int_{\MM}e^{i\lambda u-i\la' u'}b(t,x,\o)^{-1}\trc(t,x,\o)\overline{b(x,\o')^{-1}\trc(t,x,\o')}b\, du\, dt\, \dmt
\eee
using the fact that:
\begin{equation}\label{bisof4}
e^{i\lambda u-i\la' u'}=-\frac{i}{\la-\la'\frac{b}{b'}g(N,N')}\partial_{ u}(e^{i\lambda u-i\la' u'}),
\end{equation}
where we use the notation $u$ for $u(t,x,\o)$, $b$ for $b(t,x,\o)$, $N$ for $N(t,x,\o)$, $u'$ for $u(t,x,\o')$, $b'$ for $b(t,x,\o')$ and $N'$ for $N(x,\o')$. We will also use the notation $\trc$ for $\trc(t,x,\o)$, $\trc'$ for $\trc(t,x,\o')$, $\trt$ for $\trt(t,x,\o)$, and $\trt'$ for $\trt(t,x,\o')$. Using \eqref{bisof4} and the expression for the volume element \eqref{coarea}, we obtain:
\bea\label{bisof5}
&&\ds\int_{\MM}e^{i\lambda u-i\la' u'}b\overline{b'}d\MM\\
\nn & = & \ds i\int_{\MM}e^{i\lambda u-i\la' u'}\frac{b^{-1}\partial_{ u}\trc\overline{{b'}^{-1}\trc'}}{\la-\la'\frac{b}{b'}g(N,N')}d\MM+i\int_{\MM}e^{i\lambda u-i\la' u'}\frac{b^{-1}\trc\partial_{ u}(\overline{{b'}^{-1}\trc'})}{\la-\la'\frac{b}{b'}g(N,N')}d\MM \\
\nn&&\ds +i\int_{\MM}e^{i\lambda u-i\la' u'}\frac{b^{-1}\trc\overline{{b'}^{-1}\trc'}\trt}{\la-\la'\frac{b}{b'}g(N,N')}d\MM\\
\nn&&\ds +i\la'\int_{\MM}e^{i\lambda u-i\la' u'}\frac{b^{-1}\trc\overline{{b'}^{-1}\trc'}(\frac{\nabn b}{b'}g(N,N')-\frac{b\nabn b'}{b'^2}g(N,N'))}{(\la-\la'\frac{b}{b'}g(N,N'))^2}d\MM\\
\nn&&\ds +i\la'\int_{\MM}e^{i\lambda u-i\la' u'}\frac{b^{-1}\trc\overline{{b'}^{-1}\trc'}\frac{b}{b'}(g(\nabn N,N')+g(N,\nabn N'))}{(\la-\la'\frac{b}{b'}g(N,N'))^2}d\MM,
\eea
where we have used \eqref{dum} to obtain the third term in the right-hand side of \eqref{bisof5}. Since $|\la'\frac{b}{b'}g(N,N')|<\la$, we may expand the fractions in \eqref{bisof5}:
\begin{equation}\label{bisof6}
\frac{1}{\la-\la'\frac{b}{b'}g(N,N')}=\sum_{p\geq 0}\frac{1}{\la}\left(\frac{\la'\frac{b}{b'}g(N,N')}{\la}\right)^p,
\end{equation}
and
\begin{equation}\label{bisof6bis}
\frac{1}{(\la-\la'\frac{b}{b'}g(N,N'))^2}=\sum_{p\geq 0}\frac{p+1}{\la^2}\left(\frac{\la'\frac{b}{b'}g(N,N')}{\la}\right)^p.
\end{equation}
For $p\in\mathbb{Z}$, We introduce the notation $F_{j,p}( u)$:
\begin{equation}\label{bisof7}
F_{j,p}( u)=\int_{0}^{+\infty}e^{i\lambda u}\psi(2^{-j}\lambda)f(\lambda\o)(2^{-j}\la)^{p}\lambda^2 d\lambda.
\end{equation}
Together with \eqref{bisof3}, \eqref{bisof5} and \eqref{bisof6}, this implies:
\begin{equation}\label{bisof8}
\int_{\MM}E_jf(t,x)\overline{E_kf(t,x)}d\MM =\sum_{p\geq 0}A^1_p+\sum_{p\geq 0}A^2_p+\sum_{p\geq 0}A^3_p+\sum_{p\geq 0}A^4_p,
\end{equation}
where $A^1_p$, $A^2_p$, $A^3_p$ and $A^4_p$ are given by:
\begin{equation}\label{bisof9}
\begin{array}{ll}
\ds A^1_p= & \ds 2^{-j-p(j-k)}\int_{\MM}\left(\int_{\S}(b^{-1}\nabn \trc+ b^{-1}\trc\trt)b^{p+1}N^pF_{j,-p-1}( u)d\o\right)\\
& \ds\cdot\overline{\left(\int_{\S}{b'}^{-1}\trc'{b'}^{-p}{N'}^pF_{k,p}( u')d\o'\right)}d\MM ,
\end{array}
\end{equation}
\begin{equation}\label{bisof10}
\begin{array}{ll}
\ds A^2_p= & \ds 2^{-j-p(j-k)}\int_{\MM}\left(\int_{\S}b^{-1}\trc b^{p+1}N^{p+1}F_{j,-p-1}( u)d\o\right)\\
& \ds\cdot\overline{\left(\int_{\S}\nabla ({b'}^{-1}\trc'){b'}^{-p}{N'}^pF_{k,p}( u')d\o'\right)}d\MM .
\end{array}
\end{equation}
\bea\label{bisof9b}
\nn A^3_p&= & \ds (p+1)2^{-j-(p+1)(j-k)}\int_{\MM}\left(\int_{\S} b^{-1}\trc(\nabn b N+b\nabn N)b^{p}N^pF_{j,-p-2}( u)d\o\right)\\
&& \ds\cdot\overline{\left(\int_{\S}{b'}^{-1}\trc'{b'}^{-p-1}{N'}^{p+1}F_{k,p+1}( u')d\o'\right)}d\MM ,
\eea
and
\begin{equation}\label{bisof10b}
\begin{array}{ll}
\ds A^4_p= & \ds (p+1)2^{-j-(p+1)(j-k)}\int_{\MM}\left(\int_{\S}b^{-1}\trc b^{p+1}N^{p+2}F_{j,-p-2}( u)d\o\right)\\
& \ds\cdot\overline{\left(\int_{\S}{b'}^{-1}\trc'(\nabla\log(b')N'+\nabla N'){b'}^{-p-1}{N'}^pF_{k,p+1}( u')d\o'\right)}d\MM .
\end{array}
\end{equation}

\begin{remark}\label{bisrmksep}
The expansion \eqref{bisof6} allows us to rewrite $\int_{\MM}E_jf(t,x)\overline{E_kf(t,x)}d\MM $ in the form \eqref{bisof8}, i.e. as a sum of terms $A^1_p$, $A^2_p$, $A^3_p$, $A^4_p$. The key point is that in each of these terms - according to \eqref{bisof9}-\eqref{bisof10b} - one may separate the terms depending of $(\la,\o)$ from the terms depending on $(\la',\o')$. 
\end{remark}

\subsection{Estimates for $A^1_p$ and $A^2_p$} 

Let $H(t,x,\o)$ a tensor on $\MM$ such that $\norm{H}_{\l{\infty}{2}}\lesssim \ep$. Then proceeding as in the basic computation \eqref{bisb5}, we have for any $p\in\mathbb{Z}$:
\begin{equation}\label{bisof11}
\begin{array}{ll}
\ds\normm{\int_{\S}H(x,\o)F_{j,p}( u)d\o}_{\ll{2}} &\ds\leq\int_{\S}\norm{H}_{\l{\infty}{2}}\norm{F_{j,p}( u)}_{L^2_{ u}}d\o\\
&\ds\leq\norm{H}_{\l{\infty}{2}}\norm{\psi(2^{-j}\lambda)f(\lambda\o)(2^{-j}\la)^{p}\la}_{\le{2}}\\
&\ds\lesssim \ep2^{|p|+j}\ga_j.
\end{array}
\end{equation}
where we have used the fact that $1/2\leq 2^{-j}\la\leq 2$ on the support of $\psi(2^{-j}\la)$. Now, 
the estimates \eqref{estb} on $b$, \eqref{estth} on $\th$, and the equation for $\nabla N$ \eqref{estN} yield: 
\bea\label{bisof12}
\norm{(b^{-1}\nabn \trc+ b^{-1}\trc\trt)b^{p+1}N^p}_{\l{\infty}{2}}+\norm{\nabla ({b'}^{-1}\trc'){b'}^{-p}{N'}^p}_{\lprime{\infty}{2}}&&\\
\nn +\norm{b^{-1}\trc(\nabn b N+b\nabn N)b^{p}N^p}_{\l{\infty}{2}}&&\\
\nn +\norm{{b'}^{-1}\trc'(\nabla\log(b')N'+\nabla N'){b'}^{-p-1}{N'}^p}_{\lprime{\infty}{2}} & \ds\lesssim& \ep,
\eea
which together with \eqref{bisof11} implies:
\begin{equation}\label{bisof13b}
\begin{array}{ll}
&\ds\normm{\int_{\S}(b^{-1}\nabn \trc+ b^{-1}\trc\trt)b^{p+1}N^pF_{j,-p-1}( u)d\o}_{\ll{2}}\\
&\ds +\normm{\int_{\S}\nabla ({b'}^{-1}\trc'){b'}^{-p}{N'}^pF_{k,p}( u')d\o'}_{\ll{2}}\\
&\ds+\normm{\int_{\S} b^{-1}\trc(\nabn b N+b\nabn N)b^{p}N^pF_{j,-p-2}( u)d\o}_{\ll{2}}\\
&\ds +\normm{\int_{\S} {b'}^{-1}\trc'(\nabla\log(b')N'+\nabla N'){b'}^{-p-1}{N'}^pF_{k,p+1}( u')d\o'}_{\ll{2}}\\ 
\ds\lesssim & \ds \ep2^{p+j}\ga_j.
\end{array}
\end{equation}

Note that Proposition \ref{bisorthoangle} together with Proposition \ref{bisdiagonal} yields the estimate:
\begin{equation}\label{bisof13}
\norm{E_jf}_{\ll{2}}\lesssim \ep\ga_j,
\end{equation}
for any symbol satisfying the same regularity assumptions than $b^{-1}\trc$ where $b$ satisfies \eqref{estb}, and $\trc$ satisfies \eqref{esttrc}. Now, the terms containing no derivative in \eqref{bisof9}-\eqref{bisof10b} have a symbol given respectively by ${b'}^{-1}{b'}^{-p}{N'}^p$, $b^{-1}\trc b^{p+1}N^{p+1}$, ${b'}^{-1}\trc'{b'}^{-p-1}{N'}^{p+1}$ and $b^{-1}\trc b^{p+1}N^{p+2}$. Since $N$ satisfies regularity assumptions which are at least as good as $\trc$, these symbols satisfies the same regularity assumptions than $b^{-1}\trc$. Applying \eqref{bisof13}, we obtain: 
\bea\label{bisof14}
\nn\ds\normm{\int_{\S}{b'}^{-1}\trc'{b'}^{-p}{N'}^pF_{k,p}( u')d\o'}_{\ll{2}}&&\\
+\normm{\int_{\S}{b'}^{-1}\trc'{b'}^{-p-1}{N'}^{p+1}F_{k,p+1}( u')d\o'}_{\ll{2}}&\lesssim& \ep2^{p}\ga_k,
\eea
and 
\bea\label{bisof15}
\nn\ds\normm{\int_{\S}b^{-1}\trc b^{p+1}N^{p+1}F_{j,-p-1}( u)d\o}_{\ll{2}}&&\\
+\normm{\int_{\S}b^{-1}\trc b^{p+1}N^{p+2}F_{j,-p-2}( u)d\o}_{\ll{2}}&\lesssim& \ep2^{p}\ga_j,
\eea
where we have used the fact that $1/2\leq 2^{-j}\la\leq 2$ on the support of $\psi(2^{-j}\la)$.

Finally, the definition of $A_p^1-A_p^4$ given by \eqref{bisof9}-\eqref{bisof10b} and the estimates \eqref{bisof13b}, \eqref{bisof14} and \eqref{bisof15} yield:
\begin{equation}\label{bisof16}
|A^1_p|\lesssim \ep2^{2p-p(j-k)}\ga_j\ga_k,\,\forall p\geq 0,
\end{equation}
and
\begin{equation}\label{bisof17}
|A^2_p|+|A^3_p|+|A^4_p|\lesssim \ep2^{2p-(p+1)(j-k)}\ga_j\ga_k,\,\forall p\geq 0.
\end{equation}
\eqref{bisof16} and \eqref{bisof17} imply:
\begin{equation}\label{bisof18}
\sum_{p\geq 1}|A^1_p|+\sum_{p\geq 0}(|A^2_p|+|A^3_p|+|A^4_p|)\lesssim \ep2^{-(j-k)}\left(\sum_{p\geq 0}2^{-p(j-k-2)}\right)\ga_j\ga_k\lesssim \ep2^{-(j-k)}\ga_j\ga_k,
\end{equation}
where we have used the assumption $j-k-2>0$. \eqref{bisof8} and \eqref{bisof18} will yield \eqref{bisof2} provided we obtain a similar estimate for $A^1_0$. Now, the estimate of $A^1_0$ provided by \eqref{bisof16} is not sufficient since it does not contain any decay in $j-k$. We will need to perform a second integration by parts for this term.

\subsection{A more precise estimate for $A^1_0$} 

From \eqref{bisof9} with $p=0$, we have:
\begin{equation}\label{bisof19}
\ds A^1_0=2^{-j}\int_{\MM}\left(\int_{\S}(\nabn\trc+ \trc\trt)F_{j,-1}( u)d\o\right)\overline{E_kf(t,x)}d\MM.
\end{equation}
Using the geometric Littlewood-Paley projections on the 2-surfaces $\ptu$, we decompose $\nabn\trc$ as:
$$\nabn\trc=P_{\leq \frac{j+k}{2}}(\nabn\trc)+P_{>\frac{j+k}{2}}(\nabn\trc).$$ 
In turn, this yields a decomposition for $A^1_0$:
\begin{equation}\label{bisof21}
A^1_0=A^1_{0,1}+A^1_{0,2}
\end{equation}
where:
\begin{equation}\label{bisof22}
\begin{array}{l}
\ds A^1_{0,1}=2^{-j}\int_{\MM}\left(\int_{\S}P_{> \frac{j+k}{2}}(\nabn\trc)F_{j,0}( u)d\o\right)\overline{E_kf(t,x)}d\MM ,\\
\ds A^1_{0,2}=2^{-j}\int_{\MM}\left(\int_{\S}(P_{\leq \frac{j+k}{2}}(\nabn\trc)+ \trc\trt)F_{j,0}( u)d\o\right)\overline{E_kf(t,x)}d\MM .
\end{array}
\end{equation}

We first estimate $A^1_{0,1}$. The finite band property yields:
\bee
P_{> \frac{j+k}{2}}(\nabn\trc)&=& \sum_{l> \frac{j+k}{2}}P_l(\nabn\trc)\\
&=& \sum_{l> \frac{j+k}{2}}2^{-2l}\lap P_l(\nabn\trc),
\eee
which yields the following decomposition for $A^1_{0,1}$:
\be\lab{eq:A101}
A^1_{0,1} = \sum_{l> \frac{j+k}{2}}A^1_{0,1,l}
\ee
where $A^1_{0,1,l}$ is given by:
$$A^1_{0,1} =2^{-j-2l}\int_{\MM}\left(\int_{\S}\lap P_l(\nabn\trc)F_{j,0}( u)d\o\right)\overline{E_kf(t,x)}d\MM.$$
Now, the decomposition of the volume element \eqref{coarea} yields:
$$A^1_{0,1} =2^{-j-2l}\int_{\S}\int_{t,u}\left(\int_{\ptu}\lap P_l(\nabn\trc)\overline{E_kf(t,x)}b\dmt \right)F_{j,0}( u)du\,dt\,d\o.$$
Integrating by parts $\lap$ on $\ptu$, we obtain:
\bee
A^1_{0,1,l} &=&-2^{-j-2l}\int_{\S}\int_{t,u}\left(\int_{\ptu}\nabb P_l(\nabn\trc)\nabb(\overline{E_kf(t,x)}b)\dmt \right)F_{j,0}( u)du\,dt\,d\o\\
&=& -2^{-j-2l}\int_{\S}\left(\int_{\MM}\nabb P_l(\nabn\trc)F_{j,0}( u)\nabb(\overline{E_kf(t,x)}b)b^{-1}d\MM\right)d\o,
\eee
where we used again the decomposition of the volume element \eqref{coarea} in the last equality. 
We apply Cauchy-Schwartz to the integral on $\MM$ and obtain:
\bea\label{bisof23}
|A^1_{0,1,l}| &\leq& 2^{-j-2l}\int_{\S}\norm{\nabb P_l(\nabn\trc)F_{j,0}( u)}_{L^2(\MM)}\norm{\nabb(E_k b)b^{-1}}_{L^2(\MM)}d\o\\
\nn&\les& 2^{-j-2l}\int_{\S}\norm{\nabb P_l(\nabn\trc)}_{\li{\infty}{2}}\norm{F_{j,0}( u)}_{L^2_u}\norm{\nabb(E_k b)}_{L^2(\MM)}\norm{b^{-1}}_{L^\infty}d\o\\
\nn&\les& 2^{-j-l}\int_{\S}\norm{\nabn\trc}_{\li{\infty}{2}}\norm{F_{j,0}( u)}_{L^2_u}\norm{\nabb(E_k b)}_{L^2(\MM)}d\o\\
\nn&\les& \ep 2^{-j-l}\int_{\S}\norm{F_{j,0}( u)}_{L^2_u}\norm{\nabb(E_k b)}_{L^2(\MM)}d\o,
\eea
where we used the finite band property for $P_l$, the estimates \eqref{estb} for $b$, and the estimates \eqref{esttrc} for $\trc$. Plancherel yields:
\begin{equation}\label{bisof24}
\norm{F_{j,0}}_{L^2_{\o, u}}\leq\norm{\psi(2^{-j}\la)f(\la\o)\la}_{\le{2}}\lesssim 2^j\ga_j.
\end{equation}
In view of \eqref{bisof23}, we also need to estimate $\norm{\nabb(E_k b)}_{L^2(\MM)}$. We have:
\bea\label{bisof25}
\norm{\nabb(E_k b)}_{L^2(\MM)}&\les& \norm{E_k \nabb b}_{L^2(\MM)}+\norm{b\nabb E_k}_{L^2(\MM)}\\
\nn &\les & \norm{\nabb b}_{L^4(\MM)}\norm{E_k}_{L^4(\MM)}+\norm{b}_{L^\infty}\norm{\dd E_k}_{L^2(\MM)}\\
\nn &\les & \norm{\dd E_k}_{L^2(\MM)},
\eea
where we used in the last inequality the estimates \eqref{estb} for $b$, the Sobolev embedding on $\H_u$ \eqref{sobineq}, and the Sobolev inequality on the 4-dimensional manifold $\MM$ \eqref{sobineqm}. We still need to estimate $\norm{\dd E_k}_{\ll{2}}$. We have:
\begin{equation}\label{bisof26}
\begin{array}{ll}
\ds\dd E_kf(t,x)= & \ds\int_{\S}\int_0^{+\infty} e^{i\la u}\dd (b^{-1}\trc)\psi(2^{-k}\la)f(\la\o)\la^2d\la d\o\\
& \ds +i2^k\int_{\S}\int_0^{+\infty} e^{i\la u} \trc L\psi(2^{-k}\la)(2^{-k}\la)f(\la\o)\la^2d\la d\o.
\end{array}
\end{equation}
Using the basic computation \eqref{bisb5} for the first term together with the fact that $\dd (b^{-1}\trc)\in\l{\infty}{2}$ from the estimates \eqref{estb} for $b$ and \eqref{esttrc} for $\trc$, and \eqref{bisof13} for the second term together with the fact that $\trc N$ satisfies the same regularity assumptions than $b^{-1}\trc$, we obtain:
\begin{equation}\label{bisof27}
\norm{\dd E_k}_{\ll{2}}\lesssim \ep2^k\ga_k.
\end{equation}
\eqref{bisof23}, \eqref{bisof24}, \eqref{bisof25} and \eqref{bisof27} yield:
$$|A^1_{0,1,l}|\lesssim \ep^2 2^{-l+k}\ep^2\ga_j\ga_k.$$
Together with \eqref{eq:A101}, this yields:
\begin{equation}\label{bisof28}
|A^1_{0,1}|\lesssim \left(\sum_{l>\frac{j+k}{2}}2^{-l}\right)\ep2^k\ep^2\ga_j\ga_k\les \ep^2 2^{-\frac{j-k}{2}}\ga_j\ga_k.
\end{equation}

\subsection{A second integration by parts} We now estimate the term $A^1_{0,2}$ defined in \eqref{bisof22}. 

We perform a second integration by parts relying again on \eqref{bisof4}. We obtain: 
\bea\label{bisof29}
&&\ds A^1_{0,2}\\
\nn&= &\ds 2^{-2j}\int_{\MM}\bigg(\int_{\S}\Big(b(\nabn P_{\leq \frac{j+k}{2}}(\nabn\trc)+\nabn(\trc\trt))+\nabn(b)(P_{\leq \frac{j+k}{2}}(\nabn\trc)\\
\nn&&+ \trc\trt)+b\trt(P_{\leq \frac{j+k}{2}}(\nabn\trc)+ \trc\trt)\Big)F_{j,0}( u)d\o\bigg)\overline{E_kf(t,x)}d\MM \\
\nn&&\ds +2^{-2j}\int_{\MM}\left(\int_{\S}(P_{\leq \frac{j+k}{2}}(\nabn\trc)+ \trc\trt)bN F_{j,0}( u)d\o\right)\cdot\overline{\nabla E_kf(t,x)}d\MM +\cdots,
\eea
where we only mention the first term generated by the expansion \eqref{bisof6}. In fact, the other terms generated by \eqref{bisof6} and the ones generated by \eqref{bisof6bis} are estimated in the same way and generate more decay in $j-k$ similarly to the estimates \eqref{bisof16} \eqref{bisof17}. In view of \eqref{bisof29}, we decompose the main part of $A^1_{0,2}$ as the sum of three terms:
\be\lab{elias10}
A^1_{0,2}=A^1_{0,2,1}+A^1_{0,2,2}+A^1_{0,2,3}+\cdots,
\ee
where $A^1_{0,2,1}$ is given by:
\bea\lab{elias11}
\ds A^1_{0,2,1}&= &\ds 2^{-2j}\int_{\MM}\bigg(\int_{\S}\Big(b\nabn P_{\leq \frac{j+k}{2}}(\nabn\trc)\\
\nn&&+\nabn(b)P_{\leq \frac{j+k}{2}}(\nabn\trc)+b\trt P_{\leq \frac{j+k}{2}}(\nabn\trc)\Big)F_{j,0}( u)d\o\bigg)\overline{E_kf(t,x)}d\MM,
\eea
where $A^1_{0,2,2}$ is given by:
\bea\lab{elias12}
\ds A^1_{0,2,2}&= &\ds 2^{-2j}\int_{\MM}\bigg(\int_{\S}\Big(b\nabn(\trc\trt)\\
\nn&&+\nabn(b) \trc\trt+b\trt^2\trc\Big)F_{j,0}( u)d\o\bigg)\overline{E_kf(t,x)}d\MM,
\eea
and where $A^1_{0,2,3}$ is given by:
\be\lab{elias13}
 A^1_{0,2,3}= \ds 2^{-2j}\int_{\MM}\left(\int_{\S}(P_{\leq \frac{j+k}{2}}(\nabn\trc)+ \trc\trt)bN F_{j,0}( u)d\o\right)\cdot\overline{\nabla E_kf(t,x)}d\MM.
\ee

We first estimate $A^1_{0,2,1}$. We have:
\bee
&&\normm{b\nabn P_{\leq \frac{j+k}{2}}(\nabn\trc)+\nabn(b)P_{\leq \frac{j+k}{2}}(\nabn\trc)+b\trt P_{\leq \frac{j+k}{2}}(\nabn\trc)}_{L^2(\H_u)}\\
&\les & \norm{b}_{L^\infty}\left(\sum_{l\leq \frac{j+k}{2}}\norm{\nabn P_l(\nabn\trc)}_{L^2(\H_u)}\right)\\
&&+(\norm{\nabn b}_{\li{\infty}{2}}+\norm{b}_{L^\infty}\norm{\trt}_{\li{\infty}{2}})\norm{P_{\leq \frac{j+k}{2}}(\nabn\trc)}_{L^\infty}\\
&\les& \sum_{l\leq \frac{j+k}{2}}(2^l\ep+2^{\frac{l}{2}}\ep\mu(u))+\ep 2^{\frac{j+k}{2}}\norm{\nabn\trc}_{\tx{\infty}{2}}\\
&\les& \ep 2^{\frac{j+k}{2}}+ 2^{\frac{j+k}{4}}\ep\mu(u),
\eee
where we used the estimates \eqref{estb} for $b$, the estimates \eqref{estth} for $\th$, the estimate 
\eqref{estlblbtrc} for $\nabn P_l(\nabn\trc)$, the strong Bernstein inequality \eqref{eq:strongbernscalarbis}, and the estimates \eqref{esttrc} for $\trc$, and where $\mu$ in a function satisfying:
$$\norm{\mu}_{L^2(\R)}\les 1,$$
according to \eqref{estlblbtrc}.
In view of \eqref{elias11}, this yields:
\bea\lab{elias14}
\ds |A^1_{0,2,1}|&\les &\ds 2^{-2j}\norm{E_k}_{L^2(\MM)}\int_{\S}\bigg\|\Big\|b\nabn P_{\leq \frac{j+k}{2}}(\nabn\trc)\\
\nn&&+\nabn(b)P_{\leq \frac{j+k}{2}}(\nabn\trc)+b\trt P_{\leq \frac{j+k}{2}}(\nabn\trc)\Big\|_{L^2(\H_u)}F_{j,0}( u)\bigg\|_{L^2_u}d\o\\
\nn&\les& 2^{-2j}\ep\gamma_k\int_{\S}\norm{(\ep 2^{\frac{j+k}{2}}+\ep 2^{\frac{j+k}{4}}\mu(u))F_{j,0}(u)}_{L^2_u}d\o\\
\nn&\les& 2^{-2j}\ep\gamma_k\bigg(\ep 2^{\frac{j+k}{2}}\int_{\S}\norm{F_{j,0}(u)}_{L^2_u}d\o
+2^{\frac{j+k}{4}}\int_{\S}\norm{\mu}_{L^2(\R)}\norm{F_{j,0}(u)}_{L^\infty_u}d\o\bigg)\\
\nn&\les& 2^{-\frac{j-k}{4}}\ep^2\gamma_k\gamma_j,
\eea
where we used \eqref{bisof13} for $E_kf$, Plancherel with respect to $\la$ for $\norm{F_{j,0}(u)}_{L^2_u}$, Cauchy-Schwarz in $\la$ for $\norm{F_{j,0}(u)}_{L^\infty_u}$, and Cauchy-Schwarz in $\o$.

Next, we estimate $A^1_{0,2,2}$. We have:
\bee 
&&\normm{b\nabn(\trc\trt)+\nabn(b) \trc\trt+b\trt^2\trc}_{\li{\infty}{\frac{3}{2}}}\\
&\les & \norm{b}_{L^\infty}(\norm{\nabn\trt}_{\li{\infty}{2}}\norm{\trc}_{\li{\infty}{6}}+\norm{\nabn\trc}_{\li{\infty}{2}}\norm{\trt}_{\li{\infty}{6}})\\
&&+\norm{\nabn b}_{\li{\infty}{2}}\norm{\trt}_{\li{\infty}{6}}\norm{\trc}_{L^\infty}+\norm{b}_{L^\infty}\norm{\trt}^2_{\li{\infty}{3}}\norm{\trc}_{L^\infty}\\
&\les&\ep,
\eee
where we used the Sobolev embedding \eqref{sobineq} on $\H_u$, the estimates \eqref{estb} for $b$, the estimates \eqref{esttrc} for $\trc$ and the estimates \eqref{estth} for $\th$. In view of \eqref{elias12}, this yields:
\bea\lab{elias15}
&&\ds |A^1_{0,2,2}|\\
\nn&\les  &\ds 2^{-2j}\int_{\S}\normm{b\nabn(\trc\trt)+\nabn(b) \trc\trt+b\trt^2\trc}_{\li{\infty}{\frac{3}{2}}}\norm{F_{j,0}( u)}_{L^2_u}\norm{E_kf}_{\li{2}{3}}d\o\\
\nn&\les  &\ds 2^{-2j}\ep\left(\int_{\S}\norm{F_{j,0}( u)}_{L^2_u}d\o\right)\norm{E_kf}_{L^4(\MM)}\\
\nn&\les  &\ds 2^{-j}\ep\gamma_j\norm{\dd E_kf}_{L^2(\MM)}\\
\nn&\les  &\ds 2^{-(j-k)}\ep^2\gamma_j\gamma_k,
\eea
where we used Plancherel with respect to $\la$ for $\norm{F_{j,0}(u)}_{L^2_u}$, Cauchy-Schwarz in $\o$, the Sobolev embedding on $\MM$ \eqref{sobineqm}, and \eqref{bisof27} for $\dd E_kf$.

Finally, we estimate $A^1_{0,2,3}$. We have:
\bee
&&\norm{(P_{\leq \frac{j+k}{2}}(\nabn\trc)+ \trc\trt)b}_{\li{\infty}{2}}\\
&\les& \norm{b}_{L^\infty}(\norm{P_{\leq \frac{j+k}{2}}(\nabn\trc)}_{\li{\infty}{2}}+\norm{\trc}_{L^\infty}\norm{\trt}_{\li{\infty}{2}})\\
&\les& \norm{\nabn\trc}_{\li{\infty}{2}}+\ep\\
&\les&\ep,
\eee
where we used the estimates \eqref{estb} for $b$, \eqref{esttrc} for $\trc$ and \eqref{estth} for $\th$. 
 In view of \eqref{elias13}, this yields:
\bea\lab{elias16}
&&\ds |A^1_{0,2,3}|\\
\nn&\les &\ds 2^{-2j}\left(\int_{\S}\norm{(P_{\leq \frac{j+k}{2}}(\nabn\trc)+ \trc\trt)bN}_{\li{\infty}{2}}\norm{F_{j,0}( u)}_{L^2_u}d\o\right)\norm{\nabla E_kf}_{L^2(\MM)}\\
\nn&\les &\ds \ep^22^{k-2j}\gamma_k\left(\int_{\S}\norm{F_{j,0}( u)}_{L^2_u}d\o\right)\\
\nn&\les &\ds \ep^22^{-(j-k)}\gamma_k\gamma_j,
\eea
where we used \eqref{bisof13} for $E_kf$, Plancherel with respect to $\la$ for $\norm{F_{j,0}(u)}_{L^2_u}$, and Cauchy-Schwarz in $\o$. Finally, \eqref{elias10}, \eqref{elias14}, \eqref{elias15} and \eqref{elias16} imply:
\begin{equation}\label{bisof33}
|A^1_{0,2}|\lesssim \ep^22^{-\frac{j-k}{4}}\ga_j\ga_k.
\end{equation}

\subsection{End of the proof of Proposition \ref{bisorthofreq}} 

Since $A^1_0=A^1_1+A^1_2$, the estimate \eqref{bisof28} of $A^1_{0,1}$ and the estimate \eqref{bisof33} of $A^1_{0,2}$ yield:
\begin{equation}\label{bisof34}
|A^1_{0}|\lesssim \ep^22^{-\frac{j-k}{4}}\ga_j\ga_k.
\end{equation}
Together with \eqref{bisof8} and \eqref{bisof18}, this implies:
\begin{equation}\label{bisof35}
\left|\int_{\MM}E_jf(t,x)\overline{E_kf(t,x)}d\MM \right|\lesssim \ep^22^{-\frac{|j-k|}{4}}\ga_j\ga_k\textrm{ for }|j-k|>2.
\end{equation}
Finally, \eqref{bisof35} together with Shur's Lemma yields:
\begin{equation}\label{bisof36}
\norm{Ef}_{\ll{2}}^2\lesssim\sum_{j\geq -1}\norm{E_jf}_{\ll{2}}^2+\ep^2\norm{f}^2_{\le{2}}.
\end{equation}
This concludes the proof of Proposition \ref{bisorthofreq}. \QED

\section{Proof of Proposition \ref{bisdiagonal} (control of the diagonal term)}\label{bissec:diagonal}

Since the orthogonality argument in angle in the core of the paper, we choose to deal first with 
the control of the diagonal term in this section. We will then proceed with the orthogonality argument in angle in the rest of the paper.

In order to control the diagonal term, we have to prove \eqref{bisdiagonal1}:
\begin{equation}\label{bisdi1}
\norm{E^\nu_jf}_{\ll{2}}\lesssim \ep\ga^\nu_j.
\end{equation}
Recall that $E^\nu_j$ is given by:
\begin{equation}\label{bisdi2}
E^\nu_jf(t,x)=\int_{\S}b^{-1}(t,x,\o)\trc(t,x,\o)F_j( u)\eta_j^\nu(\o)d\o,
\end{equation}
where $F_j( u)$ is defined by:
\begin{equation}\label{bisdi3}
F_j( u)=\int_0^{+\infty}e^{i\la u}\psi(2^{-j}\la)f(\la\o)\la^2d\la.
\end{equation}
In view of the decompositions \eqref{dectrcom}, and the decomposition \eqref{decbom} with $p=-1$, we have the following decomposition for $b^{-1}\trc$:
\be\lab{wimby1}
b^{-1}\trc=f^j_1+f^j_2,
\ee
where $f^j_1$ only depends on $(t,x,\nu)$ and satisfies:
\be\label{wimby2}
\norm{f^j_1}_{L^\infty}\les \ep,
\ee
and where $f^j_2$ satisfies
\be\label{wimby3}
\norm{f^j_2}_{L^\infty_u\lh{2}}\les 2^{-\frac{j}{2}}\ep.
\ee
In view of \eqref{bisdi2}, \eqref{wimby1} yields the following decomposition for $E^\nu_jf$:
\bee
E^\nu_jf(t,x)= f^j_1(t,x,\nu)\int_{\S}F_j( u)\eta_j^\nu(\o)d\o+\int_{\S}F_j( u)f^j_2(t,x,\o,\nu)\eta_j^\nu(\o)d\o.
\eee
which together with the estimates \eqref{wimby2} and \eqref{wimby3} implies:
\bea\label{wimby4}
&&\norm{E^\nu_jf}_{\ll{2}}\\
\nn&\les& \norm{f^j_1}_{\ll{\infty}}\normm{\int_{\S}F_j( u)\eta_j^\nu(\o)d\o}_{\ll{2}}+\int_{\S}\norm{F_j(u)}_{L^2_u}\norm{f^j_2}_{\li{\infty}{2}}\eta_j^\nu(\o)d\o\\
\nn&\les& \ep\normm{\int_{\S}F_j( u)\eta_j^\nu(\o)d\o}_{\ll{2}}+\ep\ga^\nu_j,
\eea
where we used in the last inequality Plancherel in $\lambda$, Cauchy-Schwarz in $\o$, and the size of the patch.

The following proposition allows us to estimate the right-hand side of \eqref{wimby4}.
\begin{proposition}\label{bisdiprop1}
We have the following bound:
\begin{equation}\label{bisdiprop2}
\normm{\int_{\S}F_j( u)\eta_j^\nu(\o)d\o}_{L^2_{u_\nu,x_\nu}L^\infty_{t_\nu}}\lesssim\gamma_j^\nu.
\end{equation}
\end{proposition}

The proof of Proposition \ref{bisdiprop1} is postponed to the end of this section. \eqref{bisdiprop2} and \eqref{wimby4} yield:
$$\norm{E_j^\nu f}_{\ll{2}}\les \ep\ga^\nu_j.$$
which is the wanted estimate \eqref{bisdi1}. This concludes the proof of Proposition \ref{bisdiagonal}. \QED

\begin{remark}
In order to control the diagonal term, it suffices to have a bound of the $\ll{2}$ norm for the left-hand side of \eqref{bisdiprop2}. The improvement to a bound for the $L^2_{u_\nu,x_\nu}L^\infty_{t_\nu}$ norm will be crucial when proving the almost orthogonality in angle.
\end{remark}

We still need to prove Proposition \ref{bisdiprop1}. Note that it suffices to show:
\be\label{rg1}
\normm{L_\nu\left(\int_{\S}F_j( u)\eta_j^\nu(\o)d\o\right)}_{\ll{2}}\lesssim\gamma_j^\nu.
\ee
Now, since the space-time gradient of $u$ is given by $b^{-1}L$, we have:
\be\label{rg2}
L_\nu\left(\int_{\S}F_j( u)\eta_j^\nu(\o)d\o\right)=\int_{\S}b^{-1}\gg(L(t,x,\o),L(t,x,\nu))F^1_j( u)\eta_j^\nu(\o)d\o,
\ee
where $F^1_j$ is given by:
$$F^1_j(u)=\int_0^{+\infty}e^{i\lambda u}\psi(2^{-j}\lambda)f(\lambda\o)\lambda^3d\lambda.$$
We have:
\be\label{rg3}
\gg(L(t,x,\o),L(t,x,\nu))=\gg(N(t,x,\o)-N(t,x,\nu),N(t,x,\o)-N(t,x,\nu)).
\ee
Thus, the estimate \eqref{estNomega} for $\po N$ and the size of the patch yields:
\be\lab{rg4}
\norm{\gg(L(t,x,\o),L(t,x,\nu))}_{\lh{\infty}}\les 2^{-j},
\ee
which implies:
\bea\label{rg5}
&&\normm{\int_{\S}(b^{-1}(t,x,\o)-b^{-1}(t,x,\nu))\gg(L(t,x,\o),L(t,x,\nu))F^1_j( u)\eta_j^\nu(\o)d\o}_{\ll{2}}\\
\nn&\les& \int_{\S}\norm{b^{-1}(t,x,\o)-b^{-1}(t,x,\nu)}_{\li{\infty}{2}}\norm{\gg(L(t,x,\o),L(t,x,\nu))}_{L^\infty}\norm{F^1_j( u)}_{L^2_u}\eta_j^\nu(\o)d\o\\
\nn&\les& \ep\gamma_j^\nu,
\eea
where we used in the last inequality \eqref{rg4}, the estimate \eqref{estricciomega} for $\po b$, Plancherel in $\lambda$, Cauchy-Schwarz in $\o$, and the size of the patch.

Now, in view of \eqref{rg2}, we have:
\bee
L_\nu\left(\int_{\S}F_j( u)\eta_j^\nu(\o)d\o\right)=b^{-1}(t,x,\nu)\int_{\S}b^{-1}\gg(L(t,x,\o),L(t,x,\nu))F^1_j( u)\eta_j^\nu(\o)d\o\\
\nn +\int_{\S}(b^{-1}(t,x,\o)-b^{-1}(t,x,\nu))\gg(L(t,x,\o),L(t,x,\nu))F^1_j( u)\eta_j^\nu(\o)d\o.
\eee
which together with \eqref{rg5} and the estimate \eqref{estricciomega} for $\po b$ yields:
\bea\label{rg6}
&&\normm{L_\nu\left(\int_{\S}F_j( u)\eta_j^\nu(\o)d\o\right)}_{\ll{2}}\\
\nn&\les&\normm{\int_{\S}\gg(L(t,x,\o),L(t,x,\nu))F^1_j( u)\eta_j^\nu(\o)d\o}_{\ll{2}}+\ep\gamma_j^\nu.
\eea

Next, we estimate the right-hand side of \eqref{rg6}. Using the decomposition \eqref{decNom}, we have, taking into account \eqref{rg3}:
\be\label{rg7}
\gg(L(t,x,\o),L(t,x,\nu))=(f^j_1+f^j_2)(\o-\nu)^2,
\ee
where $f^j_1$ only depends on $\nu$ and satisfies:
$$\norm{f^j_1}_{L^\infty}\les 1,$$
and where $f^j_2$ satisfies:
$$\norm{f^j_2}_{L^\infty_u\lh{2}}\les 2^{-\frac{j}{2}},$$
where we took into account the size of the patch in the last inequality. Thus, we may rewrite the 
oscillatory integral in the right-hand side of \eqref{rg6} as:
\bee
\int_{\S}\gg(L(t,x,\o),L(t,x,\nu))F^1_j( u)\eta_j^\nu(\o)d\o &=& f^1_j(t,x,\nu)\int_{\S}(\o-\nu)^2F^1_j( u)\eta_j^\nu(\o)d\o\\
&&+\int_{\S}f^2_j(t,x,\o,\nu)(\o-\nu)^2F^1_j( u)\eta_j^\nu(\o)d\o,
\eee
which yields:
\bee
&&\normm{\int_{\S}\gg(L(t,x,\o),L(t,x,\nu))F^1_j( u)\eta_j^\nu(\o)d\o}_{\ll{2}} \\
&\les & \norm{f^1_j}_{L^\infty}\normm{\int_{\S}(\o-\nu)^2F^1_j( u)\eta_j^\nu(\o)d\o}_{\ll{2}}+\int_{\S}\norm{f^2_j}_{\li{\infty}{2}}|\o-\nu|^2\norm{F^1_j( u)}_{L^2_u}\eta_j^\nu(\o)d\o\\
&\les & \normm{\int_{\S}(\o-\nu)^2F^1_j( u)\eta_j^\nu(\o)d\o}_{\ll{2}}+\gamma_j^\nu,
\eee
where we used in the last inequality the estimates for $f^1_j$ and $f^2_j$, Plancherel in $\lambda$, Cauchy-Schwarz in $\o$, and the size of the patch. Together with \eqref{rg6}, this implies:
\bea\label{rg6bis}
\normm{L_\nu\left(\int_{\S}F_j( u)\eta_j^\nu(\o)d\o\right)}_{\ll{2}}&\les&\normm{\int_{\S}(\o-\nu)^2F^1_j( u)\eta_j^\nu(\o)d\o}_{\ll{2}}+\gamma_j^\nu.
\eea

Finally, we need to estimate the first term in the right-hand side of \eqref{rg6}. We will rely on the energy estimate for the wave equation\footnote{Let us note that in \cite{SmTa}, the authors also rely on the energy estimate for the wave equation to estimate the diagonal term in their parametrix}. Recall from \eqref{symbolE} that: 
$$\square_{\bf g}u=b^{-1}\trc.$$
Thus, we have:
\be\lab{rg7bis}
\square_{\gg}\left(\int_{\S}(\o-\nu)^2F_j( u)\eta_j^\nu(\o)d\o\right)=\int_{\S}b^{-1}(t,x,\o)\trc(t,x,\o)(\o-\nu)^2F^1_j( u)\eta_j^\nu(\o)d\o.
\ee
Arguing as in \eqref{wimby1}-\eqref{wimby4}, we have:
\bee
&&\normm{\int_{\S}b^{-1}(t,x,\o)\trc(t,x,\o)(\o-\nu)^2F^1_j( u)\eta_j^\nu(\o)d\o}_{\ll{2}}\\
&\les& \ep\normm{\int_{\S}(\o-\nu)^2F^1_j( u)\eta_j^\nu(\o)d\o}_{\ll{2}}+\ep\gamma^\nu_j
\eee
which together with \eqref{rg7bis} implies:
\bea\lab{rg8}
&&\normm{\square_{\gg}\left(\int_{\S}(\o-\nu)^2F_j( u)\eta_j^\nu(\o)d\o\right)}_{\ll{2}}\\
\nn&\les& \ep\normm{\int_{\S}(\o-\nu)^2F^1_j( u)\eta_j^\nu(\o)d\o}_{\ll{2}}+\ep\gamma^\nu_j.
\eea
Let us now define the scalar function $\phi$ on $\MM$ as:
\be\lab{rg8bis}
\phi(t,x)=\int_{\S}(\o-\nu)^2F_j( u)\eta_j^\nu(\o)d\o.
\ee
Then, $\phi$ satisfies the following wave equation on $\MM$:
\bea\lab{eq:wave}
\left\{\begin{array}{l}
\square_{\gg}\phi=F,\\
\phi|_{\Si_0}=\phi_0,\, \pr_0(\phi)|_{\Si_0}=\phi_1, 
\end{array}\right.
\eea
where in view of \eqref{rg8}, $F$ satisfies:
\be\lab{rg9}
\norm{F}_{\ll{2}}\les \ep\normm{\int_{\S}(\o-\nu)^2F^1_j( u)\eta_j^\nu(\o)d\o}_{\ll{2}}+\ep\gamma^\nu_j.
\ee
Note also that $\phi_0$ and $\phi_1$ correspond to the initial data of the half wave parametrix $\phi$. The corresponding control is the subject of step {\bf C2} and has been obtained in \cite{param2}:
\be\lab{rg10}
\norm{\nabla\phi_0}_{L^2(\Si_0)}+\norm{\phi_1}_{L^2(\Si_0)}\les \gamma^\nu_j.
\ee

Next, we recall how to derive the energy estimate for the wave equation \eqref{eq:wave}. Recall that $T$, the future unit normal to the $\Si_t$ foliation. Let $\pi$ be the deformation tensor of $T$, that is the symmetric 2-tensor on $\MM$ defined as:
$$\pi_{\a\b}=\dd_\a T_\b+\dd_\b T_\a.$$
In view of the definition of the second fundamental form $k$ and the lapse $n$, we have:
\bea\lab{defpi}
\pi_{ij}=-2k_{ij},\, \pi_{iT}=\pi_{Ti}=n^{-1}\nabla_in,\,\pi_{TT}=0.
\eea
We also introduce the energy momentum tensor $Q_{\a\b}$ on $\MM$ given by:
$$Q_{\a\b}=Q_{\a\b}[\phi]=\pr_\a\phi\pr_\b\phi-\frac{1}{2}\gg_{\a\b}\left(\gg^{\mu\nu}\pr_\mu\phi\pr_\nu\phi\right).$$
We have the following energy estimate for the scalar wave equation:
\begin{lemma}\lab{lemma:energyestimate}
Let $F$ a scalar function on $\MM$, and let $\phi_0$ and $\phi_1$ two scalar functions on $\Si_0$. Let $\phi$ the solution of the wave equation \eqref{eq:wave}. Then, $\phi$ satisfies the following energy estimate:
$$\norm{\dd\phi}_{\lsit{\infty}{2}}\les \norm{\nabla\phi_0}_{L^2(\Si_0)}+\norm{\phi_1}_{L^2(\Si_0)}+\norm{F}_{L^2(\MM)}+\left|\int_{\MM}Q_{\a\b}\pi^{\a\b}d\MM\right|^{\frac{1}{2}},$$
where $Q_{\a\b}$ is the energy momentum tensor of $\phi$, and where $\pi$ is the deformation tensor of $T$.
\end{lemma}

\begin{proof}
In view of the equation \eqref{eq:wave} satisfied by $\phi$, we have:
$$\dd^\a Q_{\a\b}=F\pr_b\phi.$$
Now, we form the 1-tensor $P$:
$$P_\a=Q_{\a 0},$$
and we obtain:
$$\dd^\a P_\a=\dd^\a Q_{\a 0}+Q_{\a\b}\dd^\a T^\b=F\pr_0\phi+\frac12 Q_{\a\b}\pi^{\a\b},$$
where $\pi$ is the deformation tensor of $e_0$. Integrating over the region $0\leq t\leq 1$, we obtain:
\bea\lab{nrj1}
&&\norm{\dd\phi}^2_{\lsit{\infty}{2}}\\
\nn&\les& \norm{\nabla\phi_0}^2_{L^2(\Si_0)}+\norm{\phi_1}^2_{L^2(\Si_0)}+\left|\int_{\MM}FT(\phi )d\MM\right|+\left|\int_{\MM}Q_{\a\b}\pi^{\a\b}d\MM\right|\\
\nn&\les& \norm{\nabla\phi_0}^2_{L^2(\Si_0)}+\norm{\phi_1}^2_{L^2(\Si_0)}+\norm{F}_{L^2(\MM)}\norm{T(\phi)}_{L^2(\MM)}+\left|\int_{\MM}Q_{\a\b}\pi^{\a\b}d\MM\right|,
\eea
which concludes the proof of the lemma.
\end{proof}

We are now in position to estimate $\phi$ given by \eqref{rg8bis}. In view of the estimate \eqref{rg9} for $F$ and \eqref{rg10} for $\phi_0, \phi_1$, Lemma \ref{lemma:energyestimate} implies:
\be\lab{rg11}
\norm{\dd\phi}_{\lsit{\infty}{2}}\les \ep\normm{\int_{\S}(\o-\nu)^2F^1_j( u)\eta_j^\nu(\o)d\o}_{\ll{2}}+\left|\int_{\MM}Q_{\a\b}\pi^{\a\b}d\MM\right|^{\frac{1}{2}}+\gamma^\nu_j.
\ee
Then, note from the decomposition of $\pi$ \eqref{defpi} and the maximal foliation assumption \eqref{maxfoliation} that:
$$g_{\a\b}\pi^{\a\b}=0$$
which together with the definition of the energy momentum tensor $Q$ yields:
$$Q_{\a\b}\pi^{\a\b}=\pr_\a\phi\pr_\b\phi\pi^{\a\b}.$$
Together with the definition of $\phi$ \eqref{rg8bis}, we obtain:
\bea\lab{rg12}
\int_{\MM}Q_{\a\b}\pi^{\a\b}d\MM&=&\pi^{\a\b}\left(\int_{\S}(\o-\nu)^2b^{-1}(t,x,\o)L_\a(t,x,\o)F^1_j( u)\eta_j^\nu(\o)d\o\right)\\
\nn&&\times\left(\int_{\S}(\o-\nu)^2b^{-1}(t,x,\o)L_\b(t,x,\o)F^1_j( u)\eta_j^\nu(\o)d\o\right)\\
\nn&=& 2n^{-1}\nabla_in\left(\int_{\S}(\o-\nu)^2b^{-1}(t,x,\o)F^1_j( u)\eta_j^\nu(\o)d\o\right)\\
\nn&&\times\left(\int_{\S}(\o-\nu)^2b^{-1}(t,x,\o)N_i(t,x,\o)F^1_j( u)\eta_j^\nu(\o)d\o\right)\\
\nn&&-2k_{ij}\left(\int_{\S}(\o-\nu)^2b^{-1}(t,x,\o)N_i(t,x,\o)F^1_j( u)\eta_j^\nu(\o)d\o\right)\\
\nn&&\times\left(\int_{\S}(\o-\nu)^2b^{-1}(t,x,\o)N_j(t,x,\o)F^1_j( u)\eta_j^\nu(\o)d\o\right),
\eea
where we used in the last equality the decomposition of $\pi$ \eqref{defpi} and the fact that $\gg(T,L)=-1$ and $L_i=N_i$. Now, we have:
\bee
&&\normm{k_{ij}\left(\int_{\S}(\o-\nu)^2(b^{-1}(t,x,\o)N_i(t,x,\o)-b^{-1}(t,x,\nu)N_i(t,x,\nu))F^1_j( u)\eta_j^\nu(\o)d\o\right)}_{\ll{2}}\\
&\les& \int_{\S}(\o-\nu)^2\norm{b^{-1}(t,x,\o)N_i(t,x,\o)-b^{-1}(t,x,\nu)N_i(t,x,\nu)}_{L^\infty}\norm{k}_{\li{\infty}{2}}\norm{F^1_j( u)}_{L^2_u}\eta_j^\nu(\o)d\o\\
&\les & \ep\gamma^\nu_j,
\eee
where we used in the last inequality the estimates \eqref{estk} for $k$, \eqref{estb} for $b$, \eqref{estricciomega} for $\po b$ and \eqref{estNomega} for $\po N$, Plancherel in $\lambda$, Cauchy-Schwarz in $\o$, and the size of the patch. Treating the other terms in the right-hand side of \eqref{rg12} similarly, we obtain:
\bee
\int_{\MM}Q_{\a\b}\pi^{\a\b}d\MM&=& 2n^{-1}(t,x)\nabla_{N_\nu}n(t,x)b^{-2}(t,x,\nu)\left(\int_{\S}(\o-\nu)^2F^1_j( u)\eta_j^\nu(\o)d\o\right)^2\\
\nn&&-2\d(t,x,\nu)b^{-2}(t,x,\nu)\left(\int_{\S}(\o-\nu)^2F^1_j( u)\eta_j^\nu(\o)d\o\right)^2+O(\ep\gamma^\nu_j).
\eee
which yields:
\bee
\left|\int_{\MM}Q_{\a\b}\pi^{\a\b}d\MM\right|&\les& \norm{n^{-1}\nabla_{N_\nu}nb^{-2}}_{L^\infty}\normm{\int_{\S}(\o-\nu)^2F^1_j( u)\eta_j^\nu(\o)d\o}^2_{\ll{2}}\\
\nn&&+\norm{\d}_{L^\infty_{x_\nu}L^2_{t_\nu}}\norm{b^{-2}}_{L^\infty}\normm{\int_{\S}(\o-\nu)^2F^1_j( u)\eta_j^\nu(\o)d\o}_{L^2_{u_\nu,x_\nu}L^\infty_{t_\nu}}\\
\nn&&\times\normm{\int_{\S}(\o-\nu)^2F^1_j( u)\eta_j^\nu(\o)d\o}_{\ll{2}}+\ep\gamma^\nu_j\\
&\les& \ep\normm{\int_{\S}(\o-\nu)^2F^1_j( u)\eta_j^\nu(\o)d\o}^2_{\ll{2}}\\
\nn&&+\ep\normm{\int_{\S}(\o-\nu)^2F^1_j( u)\eta_j^\nu(\o)d\o}_{L^2_{u_\nu,x_\nu}L^\infty_{t_\nu}}^2+\ep\gamma^\nu_j,
\eee
where we used in the last inequality the estimates \eqref{estk} for $\d$, \eqref{estn} for $n$, and \eqref{estb} for $b$. Together with \eqref{rg11}, this yields:
\be\lab{rg13}
\norm{\dd\phi}_{\lsit{\infty}{2}}\les \ep\normm{\int_{\S}(\o-\nu)^2F^1_j( u)\eta_j^\nu(\o)d\o}_{L^2_{u_\nu,x_\nu}L^\infty_{t_\nu}}+\gamma^\nu_j.
\ee

In view of the definition of $\phi$ \eqref{rg8bis}, we have:
$$\dd\phi(t,x)=\int_{\S}(\o-\nu)^2b^{-1}(t,x,\o)L(t,x,\o)F^1_j( u)\eta_j^\nu(\o)d\o.$$
Also:
\bee
&&\normm{\int_{\S}(\o-\nu)^2(b^{-1}(t,x,\o)N_i(t,x,\o)-b^{-1}(t,x,\nu)N_i(t,x,\nu))F^1_j( u)\eta_j^\nu(\o)d\o}_{\ll{2}}\\
&\les & \gamma^\nu_j,
\eee
where we used in the last inequality the estimates \eqref{estb} for $b$, \eqref{estricciomega} for $\po b$ and \eqref{estNomega} for $\po L=\po N$, Plancherel in $\lambda$, Cauchy-Schwarz in $\o$, and the size of the patch. Thus, we obtain:
$$\normm{\dd\phi(t,x)-b^{-1}(t,x,\nu)L(t,x,\nu)\int_{\S}(\o-\nu)^2F^1_j( u)\eta_j^\nu(\o)d\o}_{\ll{2}}\les \gamma^\nu_j,$$
which together with the estimate \eqref{estb} for $b$ implies:
$$\normm{\int_{\S}(\o-\nu)^2F^1_j( u)\eta_j^\nu(\o)d\o}_{\ll{2}}\les \norm{\dd\phi}_{\ll{2}}+\gamma^\nu_j.$$
Together with \eqref{rg13}, we obtain:
$$\normm{\int_{\S}(\o-\nu)^2F^1_j( u)\eta_j^\nu(\o)d\o}_{\ll{2}}\les \ep\normm{\int_{\S}(\o-\nu)^2F^1_j( u)\eta_j^\nu(\o)d\o}_{L^2_{u_\nu,x_\nu}L^\infty_{t_\nu}}+\gamma^\nu_j.$$
Together with \eqref{rg6bis}, this implies:
\bee
\normm{L_\nu\left(\int_{\S}F_j( u)\eta_j^\nu(\o)d\o\right)}_{\ll{2}}&\les&\ep\normm{\int_{\S}(\o-\nu)^2F^1_j( u)\eta_j^\nu(\o)d\o}_{L^2_{u_\nu,x_\nu}L^\infty_{t_\nu}}+\gamma_j^\nu,
\eee
and thus:
\bee
\normm{\int_{\S}F_j( u)\eta_j^\nu(\o)d\o}_{L^2_{u_\nu,x_\nu}L^\infty_{t_\nu}}&\leq &C\ep\normm{\int_{\S}(\o-\nu)^2F^1_j( u)\eta_j^\nu(\o)d\o}_{L^2_{u_\nu,x_\nu}L^\infty_{t_\nu}}+C\gamma_j^\nu,
\eee
for some universal constant $C>0$. Iterating, we obtain for any $q\geq 0$:
\bea\lab{rg14}
&&\normm{\int_{\S}F_j( u)\eta_j^\nu(\o)d\o}_{L^2_{u_\nu,x_\nu}L^\infty_{t_\nu}}\\
\nn&\leq &C^q\ep^q\normm{\int_{\S}(\o-\nu)^{2q}F^q_j( u)\eta_j^\nu(\o)d\o}_{L^2_{u_\nu,x_\nu}L^\infty_{t_\nu}}+C\left(\sum_{l=0}^{q-1}C^l\ep^l\right)\gamma_j^\nu,
\eea
where $F^q_j(u)$ is defined as:
$$F^q_j(u)=\int_0^{+\infty}e^{i\lambda u}\psi(2^{-j}\lambda)f(\lambda\o)\lambda^{2+q}d\lambda.$$
We have:
$$\normm{\int_{\S}(\o-\nu)^{2q}F^q_j( u)\eta_j^\nu(\o)d\o}_{L^2_{u_\nu,x_\nu}L^\infty_{t_\nu}}\les \normm{L_\nu\left(\int_{\S}(\o-\nu)^{2q}F^q_j( u)\eta_j^\nu(\o)d\o\right)}_{\ll{2}},$$
which together with the analog of \eqref{rg6bis} yields:
\bee
\normm{\int_{\S}(\o-\nu)^{2q}F^q_j( u)\eta_j^\nu(\o)d\o}_{L^2_{u_\nu,x_\nu}L^\infty_{t_\nu}}&\les&\normm{\int_{\S}(\o-\nu)^{2(q+1)}F^{q+1}_j( u)\eta_j^\nu(\o)d\o}_{\ll{2}}+\gamma_j^\nu.
\eee
This implies the non sharp estimate:
\bee
\normm{\int_{\S}(\o-\nu)^{2q}F^q_j( u)\eta_j^\nu(\o)d\o}_{L^2_{u_\nu,x_\nu}L^\infty_{t_\nu}}&\les& 2^{\frac{j}{2}}\gamma_j^\nu,
\eee
where we used Plancherel in $\lambda$, Cauchy-Schwarz in $\o$, and the size of the patch. Thus, we have:
\be\lab{rg15}
C^q\ep^q\normm{\int_{\S}(\o-\nu)^{2q}F^q_j( u)\eta_j^\nu(\o)d\o}_{L^2_{u_\nu,x_\nu}L^\infty_{t_\nu}}\les C^q\ep^q 2^{\frac{j}{2}}\gamma_j^\nu\rightarrow 0\textrm{ as }q\rightarrow +\infty,
\ee
since $\ep>0$ is small and may be chosen to ensure $0<C\ep<1$. Finally, letting $q\rightarrow +\infty$ in \eqref{rg14} and taking \eqref{rg15} into account yields:
$$\normm{\int_{\S}F_j( u)\eta_j^\nu(\o)d\o}_{L^2_{u_\nu,x_\nu}L^\infty_{t_\nu}}\les\gamma_j^\nu.$$
This concludes the proof of Proposition \ref{bisdiprop1}.

\section{Proof of Proposition \ref{bisorthoangle} (almost orthogonality in angle)}\label{bissec:orthoangle}

We have to prove \eqref{bisorthoangle1}:
\begin{equation}\label{bisoa1}
\norm{E_jf}_{\ll{2}}^2\lesssim\sum_{\nu\in\Gamma}\norm{E^\nu_jf}_{\ll{2}}^2+\ep^2\ga_j^2.
\end{equation}
This will result from an estimate for:
\be\lab{pc}
\left|\int_{\MM}E^\nu_jf(t,x)\overline{E^{\nu'}_jf(t,x)}d\MM \right|.
\ee

\begin{remark}
In \cite{SmTa}, the authors rely on a partial Fourier transform with respect to a coordinate system on $\ptu$ to prove almost orthogonality in angle for their parametrix. In our case, coordinate systems on $\ptu$ are not regular enough,  which forces us to work invariantly. More precisely, we will use geometric integrations by parts tied to the $u$-foliation on $\mathcal{M}$ in order to estimate \eqref{pc}. 
\end{remark}

Let us first explain why proceeding directly by integration by parts in \eqref{pc} results in a log-loss.

\subsection{Presence of a log-loss}

Let us first introduce integrations by parts with respect to tangential derivatives. By definition of $\nabb$, we have $\nabb h=\nabla h-(\nabn h)N$ for any function $h$ on $\s$. In particular, we have $\nabb( u)=0$ and $\nabb( u')={b'}^{-1}N'-{b'}^{-1}\gg(N',N) N$. Now, since $|N'-\gg(N',N)N|^2=1-\gg(N',N)^2$, this yields:
\begin{equation}\label{bisoa15} 
e^{i\la u-i\la' u'}=\frac{ib'}{\la'(1-\gg(N',N)^2)}\nabla_{N'-\gn N}(e^{i\la u-i\la' u'}),
\end{equation}
where we have used the fact that $N'-\gn N$ is a tangent vector with respect of the level surfaces of $ u$. Similarly, we have:
\begin{equation}\label{bisoa17} 
e^{i\la u-i\la' u'}=-\frac{ib}{\la(1-\gn^2)}\nabla_{N-\gn N'}(e^{i\la u-i\la' u'}),
\end{equation}
where we have used the fact that $N-\gn N'$ is a tangent vector with respect of the level surfaces of $ u'$. 

Next, we also introduce integrations by parts with respect to $L$. Since $L(u)=0$ and $L(u')={b'}^{-1}\gg(L,L')$, we have:
\begin{equation}\label{ibpl} 
e^{i\la u-i\la' u'}=\frac{ib'}{\la'\gg(L,L')}L(e^{i\la u-i\la' u'}).
\end{equation}
Similarly, we have:
\begin{equation}\label{ibpl'} 
e^{i\la u-i\la' u'}=\frac{ib}{\la\gg(L,L')}L'(e^{i\la u-i\la' u'}).
\end{equation}

We have:
\bee
\int_{\MM}E^\nu_jf(t,x)\overline{E^{\nu'}_jf(t,x)}d\MM&=&   \int_{\S\times\S}\int_0^{+\infty}\int_0^{+\infty}
\left(\int_{\MM}e^{i\la u-i\la' u'} b^{-1}\trc {b'}^{-1}\trc'd\MM\right)\\
&&\times\eta_j^\nu(\o)\eta_j^{\nu'}(\o')\psi(2^{-j}\la)\psi(2^{-j}\la') f(\la\o)f(\la'\o')\la^2 {\la'}^2d\la d\la' d\o d\o'.
\eee
We integrate by parts tangentially using \eqref{bisoa15}. Since $\la'\sim 2^j$, and 
\be\lab{borek}
1-\gn=\frac{\gg(N-N',N-N')}{2}\sim |\o-\o'|^2\sim |\nu-\nu'|^2
\ee
in view of \eqref{threomega1ter}, we see that integrating by parts using \eqref{bisoa15} gains 
roughly $2^j|\nu-\nu'|$ at the expense of a tangential derivative. Consider the term where the tangential derivative falls on $\trc$, which is roughly of the form:
\bee
\frac{1}{2^j|\nu-\nu'|}\int_{\S\times\S}\int_0^{+\infty}\int_0^{+\infty}\left(\int_{\MM}e^{i\la u-i\la' u'} b^{-1}\nabb\trc {b'}^{-1}\trc'd\MM\right)\\
\times\eta_j^\nu(\o)\eta_j^{\nu'}(\o')\psi(2^{-j}\la)\psi(2^{-j}\la') f(\la\o)f(\la'\o')\la^2 {\la'}^2d\la d\la' d\o d\o'.
\eee
Since $L\nabb\trc$ is the only derivative of $\nabb\trc$ for which we have an estimate, our next integration by parts must be with respect to $L$, that is we use \eqref{ibpl}. Since $\la'\sim 2^j$, and since
\be\lab{borek1}
\gg(L,L')=-1+\gn\sim |\nu-\nu'|^2,
\ee
in view of \eqref{borek}, we see that integrating by parts using \eqref{ibpl} gains 
roughly $2^j|\nu-\nu'|^2$ at the expense of an $L$ derivative. Consider the term where the $L$ derivative falls on $\trc'$, which is roughly of the form:
\bee
\frac{1}{2^{2j}|\nu-\nu'|^3}\int_{\S\times\S}\int_0^{+\infty}\int_0^{+\infty}\left(\int_{\MM}e^{i\la u-i\la' u'} b^{-1}\nabb\trc {b'}^{-1}L(\trc')d\MM\right)\\
\times\eta_j^\nu(\o)\eta_j^{\nu'}(\o')\psi(2^{-j}\la)\psi(2^{-j}\la') f(\la\o)f(\la'\o')\la^2 {\la'}^2d\la d\la' d\o d\o'.
\eee
Now, note in view of \eqref{borek1} and the estimate \eqref{estNomega} for $\po N$, that:
$$\gg(L,L')\sim |\nu-\nu'|^2,\,\gg(L,e'_A)=\gg(L-L',e'_A)\sim |\nu-\nu'|\textrm{ and }\gg(L,\lb')=-2+\gg(L,L')\sim 1.$$
Thus, decomposing $L$ on the frame $L', \lb', e'_A$, we obtain:
\be\label{encoreuneffort}
L\sim L'+|\nu-\nu'|\nabb'+|\nu-\nu'|^2\lb'.
\ee
We finally consider the term $|\nu-\nu'|\nabb'\trc'$ in the expansion of $L(\trc')$, and we obtain a term which is roughly of the form:
\bee
\frac{1}{2^{2j}|\nu-\nu'|^2}\int_{\MM}\left(\int_{\S}b^{-1}\nabb\trc F_j(u)\eta^\nu_j(\o)d\o\right)\c\left(\int_{\S}{b'}^{-1}\nabb'\trc' F_j(u')\eta^{\nu'}_j(\o')d\o'\right)d\MM.
\eee
We claim that such a term leads to a log-loss. Indeed, we have:
\bea
\nn&&\frac{1}{2^{2j}|\nu-\nu'|^2}\int_{\MM}\left(\int_{\S}b^{-1}\nabb\trc F_j(u)\eta^\nu_j(\o)d\o\right)\c\left(\int_{\S}{b'}^{-1}\nabb'\trc' F_j(u')\eta^{\nu'}_j(\o')d\o'\right)d\MM\\
\nn&\les& \frac{1}{2^{2j}|\nu-\nu'|^2}\normm{\int_{\S}b^{-1}\nabb\trc F_j(u)\eta^\nu_j(\o)d\o}_{L^2(\MM)}\normm{\int_{\S}{b'}^{-1}\nabb'\trc' F_j(u')\eta^{\nu'}_j(\o')d\o'}_{L^2(\MM)}\\
\nn&\les& \frac{1}{2^{2j}|\nu-\nu'|^2}\left(\int_{\S}\norm{b^{-1}\nabb\trc F_j(u)}_{L^2(\MM)}\eta^\nu_j(\o)d\o\right)\\
\nn&&\times\left(\int_{\S}\norm{{b'}^{-1}\nabb'\trc' F_j(u')}_{L^2(\MM)}\eta^{\nu'}_j(\o')d\o'\right)\\
\nn&\les& \frac{1}{2^{2j}|\nu-\nu'|^2}\left(\int_{\S}\norm{b^{-1}}_{L^\infty}\norm{\nabb\trc}_{\li{\infty}{2}}\norm{F_j(u)}_{L^2_u}\eta^\nu_j(\o)d\o\right)\\
\nn&& \times\left(\int_{\S}\norm{{b'}^{-1}}_{L^\infty}\norm{\nabb'\trc'}_{\li{\infty}{2}} \norm{F_j(u')}_{L^2_{u'}}\eta^{\nu'}_j(\o')d\o'\right)\\
\lab{chuddington}&\les &\frac{\ep^2\gamma^\nu_j\gamma^{\nu'}_j}{(2^{\frac{j}{2}}|\nu-\nu'|)^2},
\eea
where we used in the last inequality Plancherel in $\la$ and $\la'$, Cauchy-Schwartz in $\o$ and $\o'$ which gains the square root of the volume of the patch, the estimates \eqref{estb} for $b$, and the estimates \eqref{esttrc} for $\trc$. This corresponds to a log-loss since we have\footnote{The log divergence in \eqref{bisoa5} is due to the fact that we are working at the level of $H^2$ solutions for Einstein equations. Indeed, summations similar to \eqref{bisoa5} appear in particular in \cite{SmTa} in the context of $H^{2+\epsilon}$ solutions of quasilinear wave equations, albeit with a power strictly larger than 2, and hence without log divergence}:
\begin{equation}\label{bisoa5}
\sup_{\nu}\sum_{\nu'\,/\, 1\leq 2^{j/2}|\nu-\nu'|\leq 2^{j/2}} \frac{1}{(2^{j/2}|\nu-\nu'|)^{2}}\sim j.
\end{equation}
Indeed, note that $\nu'$ runs on a lattice on $\S$ of basic size $2^{-j/2}$ so that \eqref{bisoa5} corresponds to the sum 
$$\sum_{l\in\mathbb{Z}^2,\, 1\leq |l|\leq 2^{j/2}}\frac{1}{|l|^2}\sim j.$$

\subsection{A physical space decomposition for $E^\nu_jf$}

To remove the log-loss exhibited in \eqref{chuddington} \eqref{bisoa5}, we need a further decomposition. The same problem was present when dealing with the parametrix at initial time in \cite{param2}. In that case, we introduced a second decomposition in $\la$ and exploited the corresponding gain of the size of the patch in $\la$ when estimating $F_j(u)$ in $L^\infty$ and taking Cauchy-Schwartz in $\la$. In the present situation, all norms for $\trc$ are estimated by taking the $L^\infty$ norm in $u$. In turn, we always estimate $F_j(u)$ in $L^2$ using Plancherel and can therefore not exploit the size of the patch in $\la$. 

Instead, we rely here on a decomposition of $\trc$ using the geometric Littlewood-Paley projections $P_j$. We have:
$$\trc=P_{\leq j/2}(\trc)+\sum_{l>j/2}P_l\trc$$
which in turn yields the following decomposition for $E^\nu_jf$:
\begin{equation}\label{bisoa7}
E^\nu_jf(t,x)=\sum_{l\geq j/2}E^{\nu,l}_jf(t,x),
\end{equation}
where:
\begin{equation}\label{bisoa8}
E^{\nu,l}_jf(t,x)=\int_{\S}b(t,x,\o)^{-1}P_l\trc(t,x,\o)F_j(u)\eta^\nu_j(\o) d\o\,\,\,\forall l>\frac{j}{2}
\end{equation}
and:
\begin{equation}\label{bisoa8:1}
E^{\nu,j/2}_jf(t,x)=\int_{\S}b(t,x,\o)^{-1}P_{\leq j/2}\trc(t,x,\o)F_j(u)\eta^\nu_j(\o) d\o.
\end{equation}

\subsection{The mechanism to remove the log-loss}\lab{sec:logremoval}

In order to prove almost orthogonality in angle, i.e. \eqref{bisoa1}, we will estimate:
\be\lab{fff6}
\left|\sum_{l,m}\int_{\MM}E^{\nu,l}_jf(t,x)\overline{E^{\nu',m}_jf(t,x)}d\MM \right|.
\ee
Let us assume for convenience that $m\leq l$ in \eqref{fff6}. In order to remove the log-loss, our goal will be to always put more tangential derivatives on the lowest frequency, i.e. $P_m\trc'$ (as opposed to the higher frequency $P_l\trc$). This will be achieved as follows:
\begin{enumerate}
\item Integrate by parts with respect to $L$ using \eqref{ibpl}. 

\item One term corresponds to the case where the $L$ derivative falls on the largest frequency $P_l\trc$, while the other term corresponds to the case where $L$ falls on the lowest frequency $P_m\trc'$. For the second term,  decompose the $L$ derivative on the frame $L', N', e'_A$ as in \eqref{encoreuneffort}. 

\item We claim the the terms involving $L$ and $L'$ do not contain any log-loss. Indeed, instead of the sum \eqref{bisoa5} containing the log-loss, they will ultimately yield 
$$\sup_{\nu}\sum_{\nu'\,/\, 1\leq 2^{j/2}|\nu-\nu'|\leq 2^{j/2}} \frac{1}{(2^{j/2}|\nu-\nu'|)^{3}}\leq 1.$$

\item We claim the the term involving $N'$ does not contain any log-loss. Indeed, instead of the sum \eqref{bisoa5} containing the log-loss, it will ultimately yield 
$$\sup_{\nu}\sum_{\nu'\,/\, 1\leq 2^{j/2}|\nu-\nu'|\leq 2^{j/2}} \frac{1}{2^{j/2}(2^{j/2}|\nu-\nu'|)}\leq 1.$$

\item Finally, the last term is the one containing the $\nabb'$ derivative. This term is the only one which contains the log-loss exhibited in \eqref{bisoa5}. Now, we have achieved our goal since after integration by parts, the tangential derivative fell on $P_m\trc'$ which is the lowest frequency.
\end{enumerate}
  
\begin{remark}
Due to the decomposition \eqref{bisoa7}, we now not only need to obtain summability in $(\nu, \nu')$, but also in $(l, m)$.
\end{remark}

\subsection{The main estimates}

Recall that in order to prove almost orthogonality in angle, i.e. \eqref{bisoa1}, we will estimate:
$$\left|\sum_{l,m}\int_{\MM}E^{\nu,l}_jf(t,x)\overline{E^{\nu',m}_jf(t,x)}d\MM \right|.$$
We will distinguish the following two regions:
$$2^{\min(l,m)}>2^j|\nu-\nu'|\textrm{ and }2^{\min(l,m)}\leq 2^j|\nu-\nu'|.$$
We start with the estimate in the first region.

\begin{proposition}\lab{prop:tsonga1}
If $\nu\neq\nu'$ and $2^{\min(l,m)}>2^j|\nu-\nu'|$, we have the following estimate:
\be\lab{tsonga}
\left|\int_{\MM}E^{\nu,l}_jf(t,x)\overline{E^{\nu',m}_jf(t,x)}d\MM \right|\les 2^{-j}\norm{\mu_{j,\nu,l}}_{L^2(\R\times\S)}\norm{\mu_{j,\nu',m}}_{L^2(\R\times\S)},
\ee
where the sequence of functions $(\mu_{j,\nu,l})_{l> j/2}$ on $\R\times\S$ satisfies:
$$\sum_{\nu}\sum_{l> j/2}2^{2l}\norm{\mu_{j,\nu,l}}^2_{L^2(\R\times\S)}\les \ep^2 2^{2j}\norm{f}^2_{L^2(\R^3)}.$$
\end{proposition}

\begin{proof}
We have:
\bea\lab{tsonga1}
&&\left|\int_{\MM}E^{\nu,l}_jf(t,x)\overline{E^{\nu',m}_jf(t,x)}d\MM \right|\\
\nn&\les& \norm{E^{\nu,l}_jf}_{L^2(\MM)}\norm{E^{\nu',m}_jf}_{L^2(\MM)}\\
\nn&\les& \left(\int_{\S}\normm{b(t,x,\o)^{-1}P_l\trc(t,x,\o)F_j(u)}_{L^2(\MM)}\eta^\nu_j(\o)d\o\right)\\
\nn&&\times\left(\int_{\S}\normm{b(t,x,\o')^{-1}P_m'\trc(t,x,\o')F_j(u')}_{L^2(\MM)}\eta^{\nu'}_j(\o')d\o'\right)\\
\nn&\les & \norm{b^{-1}}^2_{L^\infty}\left(\int_{\S}\normm{\norm{P_l\trc}_{L^2(\H_u)}F_j(u)}_{L^2_u}\eta^\nu_j(\o) d\o\right)\\
\nn&&\times\left(\int_{\S}\normm{\norm{P_m'\trc'}_{L^2(\H_{u'})}F_j(u')}_{L^2_{u'}}\eta^{\nu'}_j(\o') d\o'\right)\\
\nn&\les & 2^{-j}\normm{\norm{P_l\trc}_{L^2(\H_u)}F_j(u)\sqrt{\eta^\nu_j(\o)}}_{L^2_{\o,u}}\normm{\norm{P_m'\trc}_{L^2(\H_u)}F_j(u)\sqrt{\eta^{\nu'}_j(\o)}}_{L^2_{\o,u}},
\eea
where we used in the last inequality the estimates \eqref{estb} for $b$, Cauchy Schwarz in $\o$ and $\o'$, and the size of the patch. 

Now, we have:
\bea\lab{tsonga2}
&&\sum_{\nu}\sum_{l> j/2}2^{2l}\normm{\norm{P_l\trc}_{L^2(\H_u)}F_j(u)\sqrt{\eta^\nu_j(\o)}}_{L^2_{\o,u}}^2\\
\nn&=& \sum_{\nu}\int_{\S}\left(\int_u\left(\sum_{l>j/2}2^{2l}\norm{P_l\trc}_{L^2(\H_u)}^2\right)|F_j(u)|^2du\right)\eta^\nu_j(\o)d\o\\
\nn&\les & \sum_{\nu}\int_{\S}\left(\int_u\norm{\nabb\trc}_{L^2(\H_u)}^2|F_j(u)|^2du\right)\eta^\nu_j(\o)d\o\\
\nn&\les & \sum_{\nu}\int_{\S}\norm{\nabb\trc}_{\li{\infty}{2}}^2\norm{F_j(u)}_{L^2_u}^2\eta^\nu_j(\o)d\o\\
\nn&\les & \ep^22^{2j}\sum_\nu(\gamma^\nu_j)^2\\
\nn&\les & 2^{2j}\ep^2\norm{f}^2_{L^2(\R^3)},
\eea
where we used the finite band property for $P_l$, the estimates \eqref{esttrc} for $\trc$ and  Plancherel in $\la$. \eqref{tsonga1} and \eqref{tsonga2} yield the proof of the proposition.
\end{proof}

\begin{remark}
In \eqref{tsonga2}, we used the estimate:
$$\sup_{\o, u}\left(\sum_{l>j/2}2^{2l}\norm{P_l\trc}_{L^2(\H_u)}^2\right)\les \sup_{\o, u}\norm{\nabb\trc}_{L^2(\H_u)}^2\les \ep^2,$$
which is true in view of the finite band property for $P_l$ and the estimates \eqref{esttrc} for $\trc$. 
Note that the sum in $l$ has to be taken \emph{before} the $\sup$ in $\o$ and $u$ for the estimate to hold. 
This explains why the $L^2$ norm on $\R\times\S$ is present in \eqref{tsonga} and estimated only after letting the sum in $l$ enter the integral as in \eqref{tsonga2}.
\end{remark}

Next, we consider the second region. We have the following decomposition:
\begin{proposition}\lab{prop:tsonga2bis}
If $\nu\neq\nu'$ and $2^{\min(l,m)}\leq 2^j|\nu-\nu'|$, we have the following decomposition:
\be\lab{tsonga3bis}
\int_{\MM}E^{\nu,l}_jf(t,x)\overline{E^{\nu',m}_jf(t,x)}d\MM=A_{j,\nu,\nu',l,m}+B_{j,\nu,\nu',l,m},
\ee
where $B_{j,\nu,\nu',l,m}$ satisfies:
\bea\lab{tsonga3ter}
&&\left|\sum_{(l,m)/2^{\min(l,m)}\leq 2^j|\nu-\nu'|}(B_{j,\nu,\nu',l,m}+B_{j,\nu',\nu,l,m})\right|\\
\nn&\les&\bigg[\frac{1}{(2^{\frac{j}{2}}|\nu-\nu'|)^3}+\frac{1}{(2^{\frac{j}{2}}|\nu-\nu'|)^{\frac{5}{2}}}+\frac{1}{2^{\frac{j}{4}}(2^{\frac{j}{2}}|\nu-\nu'|)^{\frac{3}{2}}}+\frac{2^{-(\frac{1}{12})_-j}}{(2^{\frac{j}{2}}|\nu-\nu'|)^2}+\frac{1}{2^{\frac{j}{2}}(2^{\frac{j}{2}}|\nu-\nu'|)}\\
\nn&&+\frac{1}{2^{\frac{3j}{4}}(2^{\frac{j}{2}}|\nu-\nu'|)^{\frac{1}{2}}}+2^{-j}\bigg]\ep^2\gamma^\nu_j\gamma^{\nu'}_j.
\eea
\end{proposition}

Next, we estimate $A_{j,\nu,\nu',l,m}$ for $(l,m)$ such that $2^{\min(l,m)}\leq 2^j|\nu-\nu'|$. We consider the following two subregions:
$$2^{\min(l,m)}\leq 2^j|\nu-\nu'|<2^{\max(l,m)}\textrm{ and }2^{\max(l,m)}\leq 2^j|\nu-\nu'|$$
starting with the first one:
\begin{proposition}\lab{prop:tsonga2}
If $\nu\neq\nu'$ and $2^{\min(l,m)}\leq 2^j|\nu-\nu'|<2^{\max(l,m)}$, we have the following estimate:
\bea\lab{tsonga3}
&&\left|\sum_{(l,m)/2^{\min(l,m)}\leq 2^j|\nu-\nu'|<2^{\max(l,m)}}A_{j,\nu,\nu',l,m}\right|\\
\nn&\les& \sum_{(l,m)/2^{\min(l,m)}\leq 2^j|\nu-\nu'|<2^{\max(l,m)}}\frac{2^{-2j}2^{2\min(l,m)}}{(2^{\frac{j}{2}}|\nu-\nu'|)^2}\norm{\mu_{j,\nu,l}}_{L^2(\R\times\S)}\norm{\mu_{j,\nu',m}}_{L^2(\R\times\S)}\\
\nn&&+\left[\frac{1}{(2^{\frac{j}{2}}|\nu-\nu'|)^3}+\frac{2^{-\frac{j}{4}}}{(2^{\frac{j}{2}}|\nu-\nu'|)^2}+\frac{1}{(2^{\frac{j}{2}}|\nu-\nu'|)2^{\frac{j}{2}}}\right]\ep^2\gamma^\nu_j\gamma^{\nu'}_j,
\eea
where the sequence of functions $(\mu_{j,\nu,l})_{l>j/2}$ on $\R\times\S$ satisfies:
$$\sum_{\nu}\sum_{l>j/2}2^{2l}\norm{\mu_{j,\nu,l}}^2_{L^2(\R\times\S)}\les \ep^2 2^{2j}\norm{f}^2_{L^2(\R^3)}.$$
\end{proposition}

Finally, we estimate $A_{j,\nu,\nu',l,m}$ for $(l,m)$ such that $2^{\max(l,m)}\leq 2^j|\nu-\nu'|$.
\begin{proposition}\lab{prop:tsonga3}
If $\nu\neq\nu'$ and $2^{\max(l,m)}\leq 2^j|\nu-\nu'|$, we have the following estimate:
\bea\lab{tsonga4}
&&\left|\sum_{(l,m)/2^{\max(l,m)}\leq 2^j|\nu-\nu'|}A_{j,\nu,\nu',l,m}\right|\\
\nn&\les& \sum_{(l,m)/2^{\max(l,m)}\leq 2^j|\nu-\nu'|}\frac{2^{-\frac{5j}{2}}2^{l+m+\min(l,m)}}{(2^{\frac{j}{2}}|\nu-\nu'|)^3}\norm{\mu_{j,\nu,l}}_{L^2(\R\times\S)}\norm{\mu_{j,\nu',m}}_{L^2(\R\times\S)}\\
\nn&&+\bigg[\frac{1}{(2^{\frac{j}{2}}|\nu-\nu'|)^3}+\frac{1}{(2^{\frac{j}{2}}|\nu-\nu'|)^{\frac{5}{2}}}+ \frac{2^{-(\frac{1}{6})_-j}}{(2^{\frac{j}{2}}|\nu-\nu'|)^2}+\frac{1}{2^{\frac{j}{2}}(2^{\frac{j}{2}}|\nu-\nu'|)}+2^{-j}\bigg]\ep^2\gamma^\nu_j\gamma^{\nu'}_j,
\eea
where the sequence of functions $(\mu_{j,\nu,l})_{l>j/2}$ on $\R\times\S$ satisfies:
$$\sum_{\nu}\sum_{l>j/2}2^{2l}\norm{\mu_{j,\nu,l}}^2_{L^2(\R\times\S)}\les \ep^2 2^{2j}\norm{f}^2_{L^2(\R^3)}.$$
\end{proposition}

The proof of Proposition \ref{prop:tsonga2bis} is postponed to section \ref{sec:tsonga2bis}, the proof of Proposition \ref{prop:tsonga2} is postponed to section \ref{sec:tsonga2}, and the proof of Proposition \ref{prop:tsonga3} is postponed to section \ref{sec:tsonga3}.

\subsection{End of the proof of Proposition \ref{bisorthoangle}}

We conclude the proof of Proposition \ref{bisorthoangle} by using Proposition \ref{prop:tsonga1},  Proposition \ref{prop:tsonga2bis}, Proposition \ref{prop:tsonga2} and Proposition \ref{prop:tsonga3}. In view of \eqref{bisoa7}, we have:
\bee
&&\sum_{\nu\neq\nu'}\left|\int_{\MM}E^{\nu}_jf(t,x)\overline{E^{\nu'}_jf(t,x)}d\MM \right|\\
&\les & 
\sum_{\nu\neq\nu'}\left|\sum_{l,m}\int_{\MM}E^{\nu,l}_jf(t,x)\overline{E^{\nu',m}_jf(t,x)}d\MM \right|\\
&\les & 
\sum_{\nu\neq\nu'}\sum_{2^{\min(l,m)}>2^j|\nu-\nu'|}\left|\int_{\MM}E^{\nu,l}_jf(t,x)\overline{E^{\nu',m}_jf(t,x)}d\MM \right|\\
&&+\sum_{\nu\neq\nu'}\left|\int_{\MM}E^{\nu,\leq j/2}_jf(t,x)\overline{E^{\nu',\leq j/2}_jf(t,x)}d\MM \right|\\
&&+
\sum_{\nu\neq\nu'}\left|\sum_{2^{\min(l,m)}\leq 2^j|\nu-\nu'|}\int_{\MM}E^{\nu,l}_jf(t,x)\overline{E^{\nu',m}_jf(t,x)}d\MM \right|.
\eee
In view of Proposition \ref{prop:tsonga1} and Proposition \ref{prop:tsonga2bis}, we obtain:
\bea\lab{tsonga5}
&&\sum_{\nu\neq\nu'}\left|\int_{\MM}E^{\nu}_jf(t,x)\overline{E^{\nu'}_jf(t,x)}d\MM \right|\\
\nn&\les & 
\sum_{\nu\neq\nu'}\sum_{2^{\min(l,m)}>2^j|\nu-\nu'|}2^{-j}\norm{\mu_{j,\nu,l}}_{L^2(\R\times\S)}\norm{\mu_{j,\nu',m}}_{L^2(\R\times\S)} +\sum_{\nu\neq\nu'}\left|\sum_{2^{\min(l,m)}\leq 2^j|\nu-\nu'|}\int_{\MM}A_{j,\nu,\nu',l,m}\right|\\
\nn&&+\sum_{\nu\neq\nu'}\bigg[\frac{1}{(2^{\frac{j}{2}}|\nu-\nu'|)^3}+\frac{1}{(2^{\frac{j}{2}}|\nu-\nu'|)^{\frac{5}{2}}}+\frac{1}{2^{\frac{j}{4}}(2^{\frac{j}{2}}|\nu-\nu'|)^{\frac{3}{2}}}+\frac{2^{-(\frac{1}{12})_-j}}{(2^{\frac{j}{2}}|\nu-\nu'|)^2}+\frac{1}{2^{\frac{j}{2}}(2^{\frac{j}{2}}|\nu-\nu'|)}\\
\nn&&+\frac{1}{2^{\frac{3j}{4}}(2^{\frac{j}{2}}|\nu-\nu'|)^{\frac{1}{2}}}+2^{-j}\bigg]\ep^2\gamma^\nu_j\gamma^{\nu'}_j.
\eea

Now, we have:
$$\left(\sum_{2^{\frac{j}{2}}|\nu-\nu'|<2^{\min(l,m)-\frac{j}{2}}}1\right)^{\frac{1}{2}}\les 2^{2\min(l,m)-j},$$
and: 
\bee
&&\sum_{\nu\neq\nu'}\bigg[\frac{1}{(2^{\frac{j}{2}}|\nu-\nu'|)^3}+\frac{1}{(2^{\frac{j}{2}}|\nu-\nu'|)^{\frac{5}{2}}}+\frac{1}{2^{\frac{j}{4}}(2^{\frac{j}{2}}|\nu-\nu'|)^{\frac{3}{2}}}\\
\nn&&+\frac{2^{-(\frac{1}{12})_-j}}{(2^{\frac{j}{2}}|\nu-\nu'|)^2}+\frac{1}{2^{\frac{j}{2}}(2^{\frac{j}{2}}|\nu-\nu'|)}+\frac{1}{2^{\frac{3j}{4}}(2^{\frac{j}{2}}|\nu-\nu'|)^{\frac{1}{2}}}+2^{-j}\bigg]\ep^2\gamma^\nu_j\gamma^{\nu'}_j\les \ep^2\norm{f}^2_{L^2(\R^3)}.
\eee
Together with \eqref{tsonga5}, we obtain:
\bee
&&\sum_{\nu\neq\nu'}\left|\int_{\MM}E^{\nu}_jf(t,x)\overline{E^{\nu'}_jf(t,x)}d\MM \right|\\
\nn&\les & 2^{-2j}\sum_{l,m}2^{-|l-m|}\left(\sum_{\nu}2^{2l}\norm{\mu_{j,\nu,l}}^2_{L^2(\R\times\S)}\right)^{\frac{1}{2}}\left(\sum_{\nu'}2^{2m}\norm{\mu_{j,\nu',m}}^2_{L^2(\R\times\S)}\right)^{\frac{1}{2}}\\
\nn&&+\sum_{\nu\neq\nu'}\left|\sum_{2^{\min(l,m)}\leq 2^j|\nu-\nu'|}\int_{\MM}A_{j,\nu,\nu',l,m}\right|+\ep^2\norm{f}^2_{L^2(\R^3)}\\
\nn&\les & 2^{-2j}\left(\sum_{l,\nu}2^{2l}\norm{\mu_{j,\nu,l}}^2_{L^2(\R\times\S)}\right)^{\frac{1}{2}}\left(\sum_{m,\nu'}2^{2m}\norm{\mu_{j,\nu',m}}^2_{L^2(\R\times\S)}\right)^{\frac{1}{2}}\\
\nn && +\sum_{\nu\neq\nu'}\left|\sum_{2^{\min(l,m)}\leq 2^j|\nu-\nu'|<2^{\max(l,m)}}\int_{\MM}A_{j,\nu,\nu',l,m}\right|+\sum_{\nu\neq\nu'}\left|\sum_{2^{\max(l,m)}\leq 2^j|\nu-\nu'|}\int_{\MM}A_{j,\nu,\nu',l,m}\right|\\
\nn&&+\ep^2\norm{f}^2_{L^2(\R^3)}.
\eee
Together with Proposition \ref{prop:tsonga2} and Proposition \ref{prop:tsonga3}, this yields:
\bea\lab{tsonga6}
&&\sum_{\nu\neq\nu'}\left|\int_{\MM}E^{\nu}_jf(t,x)\overline{E^{\nu'}_jf(t,x)}d\MM \right|\\
\nn&\les & 2^{-2j}\left(\sum_{l,\nu}2^{2l}\norm{\mu_{j,\nu,l}}^2_{L^2(\R\times\S)}\right)^{\frac{1}{2}}\left(\sum_{m,\nu'}2^{2m}\norm{\mu_{j,\nu',m}}^2_{L^2(\R\times\S)}\right)^{\frac{1}{2}}\\
\nn&&+\sum_{\nu\neq\nu'}\sum_{2^{\min(l,m)}\leq 2^j|\nu-\nu'|<2^{\max(l,m)}}\frac{2^{-2j}2^{2\min(l,m)}}{(2^{\frac{j}{2}}|\nu-\nu'|)^2}\norm{\mu_{j,\nu,l}}_{L^2(\R\times\S)}\norm{\mu_{j,\nu',m}}_{L^2(\R\times\S)}\\
\nn&&+\sum_{\nu\neq\nu'}\sum_{2^{\max(l,m)}\leq 2^j|\nu-\nu'|}\frac{2^{-\frac{5j}{2}}2^{l+m+\min(l,m)}}{(2^{\frac{j}{2}}|\nu-\nu'|)^3}\norm{\mu_{j,\nu,l}}_{L^2(\R\times\S)}\norm{\mu_{j,\nu',m}}_{L^2(\R\times\S)}\\
\nn&&
\nn+\sum_{\nu\neq\nu'}\bigg[\frac{1}{(2^{\frac{j}{2}}|\nu-\nu'|)^3}+\frac{1}{(2^{\frac{j}{2}}|\nu-\nu'|)^{\frac{5}{2}}}+ \frac{2^{-(\frac{1}{6})_-j}}{(2^{\frac{j}{2}}|\nu-\nu'|)^2}+\frac{1}{2^{\frac{j}{2}}(2^{\frac{j}{2}}|\nu-\nu'|)}+2^{-j}\bigg]\ep^2\gamma^\nu_j\gamma^{\nu'}_j\\
\nn&&+\ep^2\norm{f}^2_{L^2(\R^3)}.
\eea
Now, we have:
\bee
\left(\sum_{2^{\min(l,m)-\frac{j}{2}}\leq 2^{\frac{j}{2}}|\nu-\nu'|<2^{\max(l,m)-\frac{j}{2}}}\frac{1}{(2^{\frac{j}{2}}|\nu-\nu'|)^2}\right)^{\frac{1}{2}}&\les& \log(2^{\max(l,m)-\frac{j}{2}})-\log(2^{\min(l,m)-\frac{j}{2}})\\
&\les& \max(l,m)-\min(m,l),
\eee
$$\left(\sum_{2^{\frac{j}{2}}|\nu-\nu'|>2^{\max(l,m)-\frac{j}{2}}}\frac{1}{(2^{\frac{j}{2}}|\nu-\nu'|)^3}\right)^{\frac{1}{2}}\les 2^{-\max(l,m)+\frac{j}{2}},$$
and:
\bee
&&\sum_{\nu\neq\nu'}\bigg[\frac{1}{(2^{\frac{j}{2}}|\nu-\nu'|)^3}+\frac{1}{(2^{\frac{j}{2}}|\nu-\nu'|)^{\frac{5}{2}}}+ \frac{2^{-(\frac{1}{6})_-j}}{(2^{\frac{j}{2}}|\nu-\nu'|)^2}+\frac{1}{2^{\frac{j}{2}}(2^{\frac{j}{2}}|\nu-\nu'|)}+2^{-j}\bigg]\ep^2\gamma^\nu_j\gamma^{\nu'}_j\\
&\les& \ep^2\norm{f}^2_{L^2(\R^3)}.
\eee
Together with \eqref{tsonga6}, we obtain:
\bee
&&\sum_{\nu\neq\nu'}\left|\int_{\MM}E^{\nu}_jf(t,x)\overline{E^{\nu'}_jf(t,x)}d\MM \right|\\
&\les & 2^{-2j}\sum_{l,m}(1+|l-m|)2^{-|l-m|}\left(\sum_{\nu}2^{2l}\norm{\mu_{j,\nu,l}}^2_{L^2(\R\times\S)}\right)^{\frac{1}{2}}\left(\sum_{\nu'}2^{2m}\norm{\mu_{j,\nu',m}}^2_{L^2(\R\times\S)}\right)^{\frac{1}{2}}\\
&&+2^{-2j}\left(\sum_{l,\nu}2^{2l}\norm{\mu_{j,\nu,l}}^2_{L^2(\R\times\S)}\right)^{\frac{1}{2}}\left(\sum_{m,\nu'}2^{2m}\norm{\mu_{j,\nu',m}}^2_{L^2(\R\times\S)}\right)^{\frac{1}{2}}+\ep^2\norm{f}^2_{L^2(\R^3)}\\
&\les & 2^{-2j}\left(\sum_{l,\nu}\norm{\mu_{j,\nu,l}}^2_{L^2(\R\times\S)}\right)^{\frac{1}{2}}\left(\sum_{m,\nu'}\norm{\mu_{j,\nu',m}}^2_{L^2(\R\times\S)}\right)^{\frac{1}{2}}+\ep^2\norm{f}^2_{L^2(\R^3)}.
\eee
Since  the sequence of functions $(\mu_{j,\nu,l})_{l>j/2}$ on $\R\times\S$ satisfies:
$$\sum_{\nu}\sum_{l>j/2}2^{2l}\norm{\mu_{j,\nu,l}}^2_{L^2(\R\times\S)}\les \ep^2 2^{2j}\norm{f}^2_{L^2(\R^3)},$$
we finally obtain:
\bee
&&\sum_{\nu\neq\nu'}\left|\int_{\MM}E^{\nu}_jf(t,x)\overline{E^{\nu'}_jf(t,x)}d\MM \right|\les \ep^2\norm{f}_{L^2(\R^3)}^2.
\eee
This concludes the proof of Proposition \ref{bisorthoangle}.\\

The rest of the paper is as follows. In section \ref{sec:keyestimates}, we derive estimates for oscillatory integrals in various norms, as well as integrations by parts formulas tied to the $u$-foliation on $\mathcal{M}$. In section \ref{sec:tsonga2bis}, we prove Proposition \ref{prop:tsonga2bis}. In section \ref{sec:tsonga2}, we prove Proposition \ref{prop:tsonga2}. Finally, we prove Proposition \ref{prop:tsonga3} in section \ref{sec:tsonga3}.

\section{The key estimates}\lab{sec:keyestimates}

\subsection{Estimate of the $L^p(\MM)$ norm of oscillatory integrals}

\begin{lemma}\lab{lemma:osclp}
Let $H$ a tensor on $\MM$. Then, we have the following estimate:
\be\lab{oscl2bis}
\normm{\int_{\S}H F_j(u)\eta_j^{\nu}(\o)d\o}_{L^2(\MM)}\les \left(\sup_{\o}\norm{H}_{\li{\infty}{2}}\right)2^{\frac{j}{2}}\gamma^\nu_j.
\ee
More generally, for $2\leq p\leq +\infty$, we have:
\be\lab{osclpbis}
\normm{\int_{\S}H F_j(u)\eta_j^{\nu}(\o)d\o}_{L^p(\MM)}\les \left(\sup_{\o}\norm{H}_{\li{\infty}{p}}\right)2^{j(1-\frac{1}{p})}\gamma^\nu_j.
\ee
\end{lemma}

\begin{proof}
We have:
\bee
&&\normm{\int_{\S}H F_j(u)\eta_j^{\nu}(\o)d\o}_{L^p(\MM)}\\
&\les& \int_{\S}\norm{H F_j(u)}_{L^p(\MM)}\eta_j^{\nu}(\o)d\o\\
&\les& \int_{\S}\norm{H}_{\li{\infty}{p}} \norm{F_j(u)}_{L^p_u}\eta_j^{\nu}(\o)d\o\\
&\les& \left(\sup_{\o}\norm{H}_{\li{\infty}{p}}\right)\int_{\S} \norm{F_j(u)}^{\frac{2}{p}}_{L^2_u}\norm{F_j(u)}^{1-\frac{2}{p}}_{L^\infty_u}\eta_j^{\nu}(\o)d\o.
\eee
Using Plancherel to estimate $\norm{F_j(u)}_{L^2_u}$, Cauchy-Schwartz in $\la$ to estimate $\norm{F_j(u)}_{L^\infty_u}$, Cauchy-Schwarz in $\o$ and the size of the patch, we obtain:
$$\normm{\int_{\S}H  F_j(u)\eta_j^{\nu}(\o)d\o}_{L^p(\MM)}\les\left(\sup_{\o}\norm{H}_{\li{\infty}{p}}\right)2^{j(1-\frac{1}{p})}\gamma^\nu_j$$
which concludes the proof of the lemma.
\end{proof}

\begin{corollary}\lab{cor:osclp}
Let $H$ a tensor on $\MM$. Then, we have the following estimate:
\be\lab{oscl2}
\normm{\int_{\S}H P_m(\trc) F_j(u)\eta_j^{\nu}(\o)d\o}_{L^2(\MM)}\les \left(\sup_{\o}\norm{H}_{L^\infty}\right)\ep 2^{-m}2^{\frac{j}{2}}\gamma^\nu_j.
\ee
More generally, for $2\leq p\leq +\infty$, we have:
\be\lab{osclp}
\normm{\int_{\S}H P_m(\trc) F_j(u)\eta_j^{\nu}(\o)d\o}_{L^p(\MM)}\les \left(\sup_{\o}\norm{H}_{L^\infty}\right)\ep 2^{-m\frac{2}{p}}2^{j(1-\frac{1}{p})}\gamma^\nu_j.
\ee
\end{corollary}

\begin{proof}
In view of \eqref{osclpbis}, we have:
\bee
&&\normm{\int_{\S}H P_m(\trc) F_j(u)\eta_j^{\nu}(\o)d\o}_{L^p(\MM)}\\
&\les& \left(\sup_{\o}\norm{HP_m(\trc)}_{\li{\infty}{p}}\right)2^{j(1-\frac{1}{p})}\gamma^\nu_j\\
&\les& \left(\sup_{\o}\norm{H}_{L^\infty}\right)\left(\sup_{\o}\norm{P_m(\trc)}_{\li{\infty}{p}}\right)2^{j(1-\frac{1}{p})}\gamma^\nu_j.
\eee
Using Bernstein on $\ptu$ and the finite band property for $P_m$, we obtain:
$$\normm{\int_{\S}H P_m(\trc) F_j(u)\eta_j^{\nu}(\o)d\o}_{L^p(\MM)}\les \left(\sup_{\o}\norm{H}_{L^\infty}\right)2^{-m\frac{2}{p}}\left(\sup_{\o}\norm{\nabb\trc}_{\tx{\infty}{2}}\right)2^{j(1-\frac{1}{p})}\gamma^\nu_j.$$
Together with the estimates \eqref{esttrc} for $\trc$, we obtain:
$$\normm{\int_{\S}H P_m(\trc) F_j(u)\eta_j^{\nu}(\o)d\o}_{L^p(\MM)}\les\left(\sup_{\o}\norm{H}_{L^\infty}\right)\ep 2^{-m\frac{2}{p}}2^{j(1-\frac{1}{p})}\gamma^\nu_j$$
which concludes the proof of the corollary.
\end{proof}

\subsection{Estimates of the $L^1(\MM)$ norm of oscillatory integrals}

\begin{lemma}\lab{lemma:moubarakbis}
Let $\nu, \nu'$ in $\S$ such that $\nu\neq\nu'$. Recall the decomposition $\hch=\chi_1+\chi_2$ in \eqref{dechch}. Let $H$ a tensor on $\MM$. Then, we have the following estimate: 
\bea
\lab{nadalbis}&&\int_{\MM}\left|\int_{\S}H \dd(L(\trc))F_j(u)\eta_j^\nu(\o)d\o\right|d\MM\\
\nn&\les& \left(\sup_{\o\in\textrm{supp}(\eta_j^{\nu})}(\norm{H}_{L^2_uL^4_tL^2_{x'}})+|\nu-\nu'|\norm{H}_{L^{3_+}(\MM)}+\norm{{\chi_2}_{\nu'}H}_{L^2(\MM)}\right)2^{\frac{j}{2}}\ep\gamma_j^\nu.
\eea
\end{lemma}

\begin{proof}
We have:
\bea\lab{reha1}
\int_{\MM}\left|\int_{\S}H\dd(L(\trc)) F_j(u)\eta_j^\nu(\o)d\o\right|d\MM &\les& \int_{\S}\norm{H\dd(L(\trc))F_j(u)}_{L^1(\MM)}\eta_j^\nu(\o)d\o\\
\nn&\les& \int_{\S}\norm{H\dd(L(\trc))}_{\li{2}{1}}\norm{F_j(u)}_{L^2_u}\eta_j^\nu(\o)d\o.
\eea
Differentiating the Raychaudhuri equation \eqref{raychaudhuri}, we obtain:
$$\dd(L(\trc))=-(\trc+\db)\dd\trc-2\hch\dd\hch-\dd(\db)\trc.$$
Using the decomposition $\hch=\chi_1+\chi_2$ in \eqref{dechch}, we obtain:
\be\lab{reha2}
\dd(L(\trc))=G_1+\chi_2 G_2
\ee
where 
$$G_1=-(\trc+\db)\dd\trc-2\chi_1\dd\hch-\dd(\db  )\trc$$
and 
$$G_2=-2\dd\hch.$$
In particular, we have:
\be\lab{reha3}
\norm{G_1}_{L^\infty_uL^{\frac{4}{3}}_tL^2_{x'}}+\norm{G_2}_{\li{\infty}{2}}\les\ep
\ee
where we used the estimate \eqref{esttrc} for $\trc$, the estimate $\hch$ for $\hch$, the estimate \eqref{estk} and \eqref{estn} for $\db$, and the estimate \eqref{dechch2} for $\chi_1$. In view of \eqref{reha1} and \eqref{reha2}, 
we have:
\bea\lab{reha4}
&&\int_{\MM}\left|\int_{\S}H\dd(L(\trc)) F_j(u)\eta_j^\nu(\o)d\o\right|d\MM \\
\nn&\les& \int_{\S}\norm{HG_1}_{\li{2}{1}}\norm{F_j(u)}_{L^2_u}\eta_j^\nu(\o)d\o +\int_{\S}\norm{\chi_2 HG_2}_{\li{2}{1}}\norm{F_j(u)}_{L^2_u}\eta_j^\nu(\o)d\o\\
\nn&\les& \int_{\S}\norm{H}_{L^2_uL^4_tL^2_{x'}}\norm{G_1}_{L^\infty_uL^{\frac{4}{3}}_tL^2_{x'}}\norm{F_j(u)}_{L^2_u}\eta_j^\nu(\o)d\o \\
\nn&&+\int_{\S}\norm{\chi_2 H}_{L^2(\MM)}\norm{G_2}_{\li{\infty}{2}}\norm{F_j(u)}_{L^2_u}\eta_j^\nu(\o)d\o\\
\nn&\les& \sup_{\o\in\textrm{supp}(\eta_j^{\nu})}(\norm{H}_{L^2_uL^4_tL^2_{x'}})\ep\int_{\S}\norm{F_j(u)}_{L^2_u}\eta_j^\nu(\o)d\o+\ep\int_{\S}\norm{\chi_2 H}_{L^2(\MM)}\norm{F_j(u)}_{L^2_u}\eta_j^\nu(\o)d\o,
\eea
where we used in the last inequality the estimate \eqref{reha3}. In view of the estimate \eqref{dechch2} for $\chi_2$, we have:
$$\norm{\chi_2-{\chi_2}_{\nu'}}_{L^{6_-}(\MM)}\les |\nu-\nu'|\norm{\po\chi_2}_{L^{6_-}(\MM)}\les |\nu-\nu'|\ep,$$
which yields:
\bee
\norm{\chi_2 H}_{L^2(\MM)}&\les& \norm{{\chi_2}_\nu H}_{L^2(\MM)}+\norm{(\chi_2-{\chi_2}_\nu) H}_{L^2(\MM)}\\
&\les& \norm{{\chi_2}_\nu H}_{L^2(\MM)}+\norm{\chi_2-{\chi_2}_\nu}_{L^{6_-}(\MM)}\norm{H}_{L^{3_+}(\MM)}\\
&\les& \norm{{\chi_2}_\nu H}_{L^2(\MM)}+\ep|\nu-\nu'|\norm{H}_{L^{3_+}(\MM)}.
\eee
Together with \eqref{reha4}, we obtain:
\bee
&&\int_{\MM}\left|\int_{\S}H \nabb L(\trc)F_j(u)\eta_j^\nu(\o)d\o\right|d\MM\\
\nn&\les& \left(\sup_{\o\in\textrm{supp}(\eta_j^{\nu})}(\norm{H}_{L^2_uL^4_tL^2_{x'}})+|\nu-\nu'|\norm{H}_{L^{3_+}(\MM)}+\norm{{\chi_2}_{\nu'}H}_{L^2(\MM)}\right)\\
\nn&&\times\left(\int_{\S}\norm{F_j(u)}_{L^2_u}\eta_j^\nu(\o)d\o\right)\\
\nn&\les& \left(\sup_{\o\in\textrm{supp}(\eta_j^{\nu})}(\norm{H}_{L^2_uL^4_tL^2_{x'}})+|\nu-\nu'|\norm{H}_{L^{3_+}(\MM)}+\norm{{\chi_2}_{\nu'}H}_{L^2(\MM)}\right)2^{\frac{j}{2}}\ep\gamma_j^\nu,
\eee
where we used in the last inequality Plancherel in $\la$ for $\norm{F_j(u)}_{L^2_u}$, Cauchy-Schwarz in $\o$, and the size of the patch. This concludes the proof of the lemma.
\end{proof}

\begin{lemma}\lab{lemma:moubarak}
Let $\nu, \nu'$ in $\S$ such that $\nu\neq\nu'$. Let $l$ an integer. Recall the decomposition $\hch=\chi_1+\chi_2$ in \eqref{dechch}. Let $H$ a tensor on $\MM$. Then, we have the following estimate: 
\be\lab{nadal}
\int_{\MM}\left|\int_{\S}HL(P_l\trc)F_j(u)\eta_j^\nu(\o)d\o\right|d\MM\les \left(\sup_{\o\in\textrm{supp}(\eta_j^{\nu})}\norm{H}_{L^2_{u, x'}L^\infty_t}\right)2^{\frac{j}{2}-l}\ep\gamma_j^\nu.
\ee
\end{lemma}

\begin{proof}
We have:
\bee
\int_{\MM}\left|\int_{\S}HL(P_l\trc)F_j(u)\eta_j^\nu(\o)d\o\right|d\MM &\les& \int_{\S}\norm{HL(P_l\trc)F_j(u)}_{L^1(\MM)}\eta_j^\nu(\o)d\o\\
\nn&\les& \int_{\S}\norm{HL(P_l\trc)}_{\li{2}{1}}\norm{F_j(u)}_{L^2_u}\eta_j^\nu(\o)d\o\\
\nn&\les& \int_{\S}\norm{HnL(P_l\trc)}_{\li{2}{1}}\norm{F_j(u)}_{L^2_u}\eta_j^\nu(\o)d\o,
\eee
where we used the estimate \eqref{estn} on $n$ in the last inequality. This yields:
\bea\lab{nadal1}
&&\int_{\MM}\left|\int_{\S}HL(P_l\trc)F_j(u)\eta_j^\nu(\o)d\o\right|d\MM \\
\nn&\les& \left(\sup_{\o\in\textrm{supp}(\eta_j^{\nu})}\norm{H}_{L^2_{u, x'}L^\infty_t}\right)\int_{\S}\norm{nL(P_l\trc)}_{\xt{2}{1}}\norm{F_j(u)}_{L^2_u}\eta_j^\nu(\o)d\o.
\eea

Next, we estimate $nL(P_l\trc)$. We have:
$$nL(P_l\trc)=[nL,P_l](\trc)+P_l(nL(\trc)),$$
which yields:
\bea\lab{nadal2}
\norm{nL(P_l\trc)}_{\xt{2}{1}}&\les & \norm{[nL,P_l]\trc)}_{\tx{1}{2}}+\norm{P_l(nL\trc)}_{\xt{2}{1}}\\
\nn&\les& 2^{-l}\ep,
\eea
where we used in the last inequality the commutator estimate \eqref{commlp3} and the estimate \eqref{lievremont1}. 
Now, \eqref{nadal1} and \eqref{nadal2} imply:
\bee
&&\int_{\MM}\left|\int_{\S}HL(P_l\trc)F_j(u)\eta_j^\nu(\o)d\o\right|d\MM \\
\nn&\les& \left(\sup_{\o\in\textrm{supp}(\eta_j^{\nu})}\norm{H}_{L^2_{u, x'}L^\infty_t}\right)2^{-l}\ep\left(\int_{\S}\norm{F_j(u)}_{L^2_u}\eta_j^\nu(\o)d\o\right)\\
&\les&\left(\sup_{\o\in\textrm{supp}(\eta_j^{\nu})}\norm{H}_{L^2_{u, x'}L^\infty_t}\right)2^{\frac{j}{2}-l}\ep\gamma_j^\nu,
\eee
where we used in the last inequality Plancherel in $u$, Cauchy-Schwarz in $\o$ and the size of the patch. This concludes the proof of the lemma.
\end{proof}

\begin{lemma}\lab{lemma:moubarak:1}
Let $\nu, \nu'$ in $\S$ such that $\nu\neq\nu'$. Let $l$ an integer. Recall the decomposition $\hch=\chi_1+\chi_2$ in \eqref{dechch}. Let $H$ a tensor on $\MM$. Then, we have the following estimate: 
\be\lab{nadal:1}
\int_{\MM}\left|\int_{\S}H\nabb(L(P_{\leq l}\trc))F_j(u)\eta_j^\nu(\o)d\o\right|d\MM\les \left(\sup_{\o\in\textrm{supp}(\eta_j^{\nu})}\norm{H}_{L^2_{u, x'}L^\infty_t}\right)2^{\frac{j}{2}}\ep\gamma_j^\nu.
\ee
\end{lemma}

\begin{proof}
We have:
\bee
&&\int_{\MM}\left|\int_{\S}H\nabb(L(P_{\leq l}\trc))F_j(u)\eta_j^\nu(\o)d\o\right|d\MM \\
&\les& \int_{\S}\norm{H \nabb(L(P_{\leq l}\trc))F_j(u)}_{L^1(\MM)}\eta_j^\nu(\o)d\o\\
\nn&\les& \int_{\S}\norm{H\nabb(L(P_{\leq l}\trc))}_{\li{2}{1}}\norm{F_j(u)}_{L^2_u}\eta_j^\nu(\o)d\o\\
\nn&\les& \int_{\S}\norm{Hn\nabb(L(P_{\leq l}\trc))}_{\li{2}{1}}\norm{F_j(u)}_{L^2_u}\eta_j^\nu(\o)d\o,
\eee
where we used the estimate \eqref{estn} for $n$ in the last inequality. This yields:
\bea\lab{nadal1:1}
&&\int_{\MM}\left|\int_{\S}H\nabb(L(P_{\leq l}\trc))F_j(u)\eta_j^\nu(\o)d\o\right|d\MM \\
\nn&\les& \left(\sup_{\o\in\textrm{supp}(\eta_j^{\nu})}\norm{H}_{L^2_{u, x'}L^\infty_t}\right)\int_{\S}\norm{n\nabb(L(P_{\leq l}\trc))}_{\xt{2}{1}}\norm{F_j(u)}_{L^2_u}\eta_j^\nu(\o)d\o.
\eea

Next, we estimate $\nabb(nL(P_{\leq l}\trc))$. We have:
$$n\nabb(L(P_{\leq l}\trc))=-\nabb n L(P_{\leq l}\trc)+\nabb[nL,P_{\leq l}](\trc)+\nabb(P_{\leq l}(nL(\trc))),$$
which yields:
\bea\lab{nadal2:2}
\norm{n\nabb(L(P_{\leq l}\trc))}_{\xt{2}{1}}&\les & \norm{n^{-1}\nabb n}_{L^\infty}\norm{nL(P_{\leq l}\trc)}_{\tx{1}{2}}\\
\nn&&+\norm{\nabb[nL,P_{\leq l}]\trc)}_{\tx{1}{2}}+\norm{\nabb(P_{\leq l}(nL\trc))}_{\xt{2}{1}}\\
\nn&\les& \ep,
\eea
where we used in the last inequality the estimate \eqref{estn} for $n$, the commutator estimate \eqref{commlp3} and the estimate \eqref{lievremont2}. 
Now, \eqref{nadal1:1} and \eqref{nadal2:2} imply:
\bee
&&\int_{\MM}\left|\int_{\S}H\nabb(L(P_{\leq l}\trc))F_j(u)\eta_j^\nu(\o)d\o\right|d\MM \\
\nn&\les& \left(\sup_{\o\in\textrm{supp}(\eta_j^{\nu})}\norm{H}_{L^2_{u, x'}L^\infty_t}\right)\ep\left(\int_{\S}\norm{F_j(u)}_{L^2_u}\eta_j^\nu(\o)d\o\right)\\
&\les&\left(\sup_{\o\in\textrm{supp}(\eta_j^{\nu})}\norm{H}_{L^2_{u, x'}L^\infty_t}\right)2^{\frac{j}{2}}\ep\gamma_j^\nu,
\eee
where we used in the last inequality Plancherel in $u$, Cauchy-Schwarz in $\o$ and the size of the patch. This concludes the proof of the lemma.
\end{proof}

\subsection{Estimate of the $L^2_{u,x'}L^\infty_t$ norm of oscillatory integrals}

\begin{lemma}\lab{lemma:loeb}
Let $p\in \mathbb{N}$. We have:
\bea\lab{loeb}
\normm{\int_{\S} b^{-1}\trc\left(2^{\frac{j}{2}}(N-N_\nu)\right)^pF_j(u)\eta_j^\nu(\o)d\o}_{L^2_{u_\nu {x'}_{\nu}}L^\infty_t}\les (1+p^2)\ep\gamma^\nu_j.
\eea
\end{lemma}

\begin{proof}
Note that it suffices to show:
\bea\lab{loeb1}
\normm{L_\nu\left(\int_{\S} b^{-1}\trc\left(2^{\frac{j}{2}}(N-N_\nu)\right)^pF_j(u)\eta_j^\nu(\o)d\o\right)}_{L^2_{u_\nu {x'}_{\nu}}L^1_t}\les (1+p^2)\ep\gamma^\nu_j.
\eea
We have:
\bea\lab{loeb2}
&& L_\nu\left(\int_{\S} b^{-1}\trc\left(2^{\frac{j}{2}}(N-N_\nu)\right)^pF_j(u)\eta_j^\nu(\o)d\o\right)\\
\nn&=& i\int_{\S} b^{-2}\trc\left(2^{\frac{j}{2}}(N-N_\nu)\right)^p2^j\gg(L,L_\nu)F_{j,1}(u)\eta_j^\nu(\o)d\o\\
\nn&&+\int_{\S} L_\nu(b^{-1}\trc)\left(2^{\frac{j}{2}}(N-N_\nu)\right)^pF_j(u)\eta_j^\nu(\o)d\o\\
\nn&&+p\int_{\S} b^{-1}\trc\left(2^{\frac{j}{2}}(N-N_\nu)\right)^{p-1}\left(2^{\frac{j}{2}}(\dd_{L_\nu}L-\dd_{L_\nu} L_\nu)\right)F_j(u)\eta_j^\nu(\o)d\o,
\eea
where we used the fact that $N-N_\nu=L-L_\nu$ since $T$ does not depend on $\o$, and the fact that $b^{-1}L$ is the space-time gradient of $u$ so that:
$$L_\nu(u)=b^{-1}\gg(L_\nu, L).$$
Next, we evaluate the various terms in the right-hand side of \eqref{loeb2}. First, recall the identity \eqref{rg3}:
\be\lab{loeb3}
\gg(L,L_\nu)=\gg(N-N_\nu,N-N_\nu).
\ee
Next, decompose $L_\nu$ on the frame $L, \lb, e_A, A=1, 2$ which yields:
\be\lab{loeb4}
L_\nu=\frac{1}{2}(1+\gg(N,N_\nu))L+(N_\nu-\gg(N,N_\nu)N)+\frac{1}{2}(1-\gg(N,N_\nu)\lb,
\ee
which yields:
\bea\lab{loeb5}
L_\nu(b^{-1}\trc)&=&\frac{1}{2}(1+\gg(N,N_\nu))L(b^{-1}\trc)+(N_\nu-\gg(N,N_\nu)N)(b^{-1}\trc)\\
\nn&&+\frac{1}{2}(1-\gg(N,N_\nu)\lb(b^{-1}\trc).
\eea
Also, in view of the decomposition \eqref{loeb3} and the Ricci equations \eqref{ricciform}, we have:
\bea\lab{loeb6}
&&\dd_{L_\nu}L-\dd_{L_\nu} L_\nu\\
\nn&=&\frac{1}{2}(1+\gg(N,N_\nu))\dd_LL+\dd_{N_\nu-\gg(N,N_\nu)N}L+\frac{1}{2}(1-\gg(N,N_\nu)\dd_{\lb}L-\dd_{L_\nu} L_\nu\\
\nn&=&-\frac{1}{2}(1+\gg(N,N_\nu))\db L+\chi(N_\nu-\gg(N,N_\nu)N,e_A)e_A-\epsilon_{N_\nu-\gg(N,N_\nu)N}L\\
\nn&&+\frac{1}{2}(1-\gg(N,N_\nu)(\z_Ae_A+(\d+n^{-1}\nabla_Nn)L)-\db_\nu L_\nu.
\eea
Now, \eqref{loeb2}, \eqref{loeb3}, \eqref{loeb5} and \eqref{loeb6} yield:
\bea\lab{loeb7}
&& L_\nu\left(\int_{\S} b^{-1}\trc\left(2^{\frac{j}{2}}(N-N_\nu)\right)^pF_j(u)\eta_j^\nu(\o)d\o\right)\\
\nn&=& A_1+A_2+A_3+A_4+A_5+A_6+A_7,
\eea
where $A_1, A_2, A_3, A_4, A_5, A_6$ and $A_7$ are respectively given by:
\be\lab{loeb8}
A_1=\int_{\S}b^{-2}\trc\left(2^{\frac{j}{2}}(N-N_\nu)\right)^{p+2}F_{j,1}(u)\eta_j^\nu(\o)d\o,
\ee
\be\lab{loeb9}
A_2=\int_{\S}\frac{1}{2}(1+\gg(N,N_\nu))L(b^{-1}\trc)\left(2^{\frac{j}{2}}(N-N_\nu)\right)^pF_j(u)\eta_j^\nu(\o)d\o,
\ee
\be\lab{loeb10}
A_3=\int_{\S}(N_\nu-\gg(N,N_\nu)N)(b^{-1}\trc)\left(2^{\frac{j}{2}}(N-N_\nu)\right)^pF_j(u)\eta_j^\nu(\o)d\o,
\ee
\be\lab{loeb11}
A_4=\int_{\S}\frac{1}{2}(1-\gg(N,N_\nu)\lb(b^{-1}\trc)\left(2^{\frac{j}{2}}(N-N_\nu)\right)^pF_j(u)\eta_j^\nu(\o)d\o,
\ee
\bea\lab{loeb12}
A_5&=&p\int_{\S}\frac{1}{2}(1+\gg(N,N_\nu))b^{-1}\trc\left(2^{\frac{j}{2}}(N-N_\nu)\right)^{p-1}(2^{\frac{j}{2}}(-\db L+\db_\nu L_\nu))\\
\nn&&\times F_j(u)\eta_j^\nu(\o)d\o,
\eea
\bea\lab{loeb13}
A_6&=&p\int_{\S}b^{-1}\trc\left(2^{\frac{j}{2}}(N-N_\nu)\right)^{p-1}\\
\nn&&\times (2^{\frac{j}{2}}(\chi(N_\nu-\gg(N,N_\nu)N,e_A)e_A-\epsilon_{N_\nu-\gg(N,N_\nu)N}L))F_j(u)\eta_j^\nu(\o)d\o,
\eea
and:
\bea\lab{loeb14}
A_7&=&p\int_{\S}\frac{1}{2}(1-\gg(N,N_\nu))b^{-1}\trc\left(2^{\frac{j}{2}}(N-N_\nu)\right)^{p-1}\\
\nn&&\times (2^{\frac{j}{2}}(\z_Ae_A+(\d+n^{-1}\nabla_Nn)L)F_j(u)\eta_j^\nu(\o)d\o.
\eea

We estimate $A_1, A_2, A_3, A_4, A_5, A_6$ and $A_7$ starting with $A_1$. Recall the decomposition \eqref{decNom} for $2^{\frac{j}{2}}(N-N_\nu)$:
\be\lab{loeb15}
2^{\frac{j}{2}}(N-N_\nu)=F^j_1+F^j_2
\ee
where the tensor $F^j_1$ only depends on $\nu$ and satisfies:
\be\lab{loeb15bis}
\norm{F^j_1}_{L^\infty}\les 1,
\ee
and where the tensor $F^j_2$ satisfies:
\be\lab{loeb15ter}
\norm{F^j_2}_{L^\infty_u\lh{2}}\les 2^{-\frac{j}{2}}.
\ee
This yields:
$$\left(2^{\frac{j}{2}}(N-N_\nu)\right)^{p+2}=\sum_{m=0}^{p+1}F_{1,j}^m\left(2^{\frac{j}{2}}(N-N_\nu)\right)^{p-m+1}F_{2,j}+F_{1,j}^{p+2}$$
and thus:
\bee
A_1&=&\sum_{m=0}^{p+1}F_{1,j}^m\left(\int_{\S}b^{-2}\trc\left(2^{\frac{j}{2}}(N-N_\nu)\right)^{p-m+1}F_{2,j}F_{j,1}(u)\eta_j^\nu(\o)d\o\right)\\
&&+F_{1,j}^{p+2}\left(\int_{\S}b^{-2}\trc F_{j,1}(u)\eta_j^\nu(\o)d\o\right),
\eee
where we used the fact that $F^j_1$ does not depend on $\o$. We obtain:
\bee
\nn\norm{A_1}_{L^2(\MM)}&\les& \sum_{m=0}^{p+1}\norm{F_{1,j}}_{L^\infty}^m\normm{\int_{\S}b^{-2}\trc\left(2^{\frac{j}{2}}(N-N_\nu)\right)^{p-m+1}F_{2,j}F_{j,1}(u)\eta_j^\nu(\o)d\o}_{L^2(\MM)}\\
\nn&&+\norm{F_{1,j}}_{L^\infty}^{p+2}\normm{\int_{\S}b^{-2}\trc F_{j,1}(u)\eta_j^\nu(\o)d\o}_{L^2(\MM)}\\
\nn&\les& \sum_{m=0}^{p+1}\normm{\int_{\S}b^{-2}\trc\left(2^{\frac{j}{2}}(N-N_\nu)\right)^{p-m+1}F_{2,j}F_{j,1}(u)\eta_j^\nu(\o)d\o}_{L^2(\MM)}\\
&&+\normm{\int_{\S}b^{-2}\trc F_{j,1}(u)\eta_j^\nu(\o)d\o}_{L^2(\MM)}
\eee
where we used \eqref{loeb15bis} in the last inequality. The estimate in $L^2(\MM)$ \eqref{oscl2bis} for oscillatory integrals yields:
\bea
\nn\norm{A_1}_{L^2(\MM)}&\les& \sum_{m=0}^{p+1}\left(\sup_{\o}\normm{b^{-2}\trc\left(2^{\frac{j}{2}}(N-N_\nu)\right)^{p-m+1}F_{2,j}}_{\li{\infty}{2}}\right)2^{\frac{j}{2}}\gamma^\nu_j\\
\nn&&+\normm{\int_{\S}b^{-2}\trc F_{j,1}(u)\eta_j^\nu(\o)d\o}_{L^2(\MM)}\\
\nn&\les& \sum_{m=0}^{p+1}\left(\sup_{\o}\normm{b^{-2}\trc\left(2^{\frac{j}{2}}(N-N_\nu)\right)^{p-m+1}}_{L^\infty}\norm{F_{2,j}}_{\li{\infty}{2}}\right)2^{\frac{j}{2}}\gamma^\nu_j\\
\nn&&+\normm{\int_{\S}b^{-2}\trc F_{j,1}(u)\eta_j^\nu(\o)d\o}_{L^2(\MM)}\\
\lab{loeb16}&\les& (1+p)\ep\gamma^\nu_j+\normm{\int_{\S}b^{-2}\trc F_{j,1}(u)\eta_j^\nu(\o)d\o}_{L^2(\MM)}
\eea
where we used in the last inequality the estimate \eqref{loeb15ter} and the estimates \eqref{estb} for $b$ and \eqref{esttrc} for $\trc$. 

Next, we estimate the second term in the right-hand side of \eqref{loeb16}. In view of the decomposition \eqref{dectrcom} for $\trc$, and the decomposition \eqref{decbom} for $b$, we have:
\be\lab{loeb17}
b^{-2}\trc=f^j_1+f^j_2
\ee
where the scalar $f^j_1$ only depends on $\nu$ and satisfies:
\be\lab{loeb17bis}
\norm{f^j_1}_{L^\infty}\les \ep,
\ee
and where the scalar $f^j_2$ satisfies:
\be\lab{loeb17ter}
\norm{f^j_2}_{L^\infty_u\lh{2}}\les \ep 2^{-\frac{j}{2}}.
\ee
This yields:
\bee
\int_{\S}b^{-2}\trc F_{j,1}(u)\eta_j^\nu(\o)d\o&=& f^j_1\int_{\S} F_{j,1}(u)\eta_j^\nu(\o)d\o+\int_{\S}f^j_2 F_{j,1}(u)\eta_j^\nu(\o)d\o
\eee
which together with \eqref{loeb17bis} implies:
\bee
&&\normm{\int_{\S}b^{-2}\trc F_{j,1}(u)\eta_j^\nu(\o)d\o}_{L^2(\MM)}\\
&\les & \norm{f^j_1}_{L^\infty}\normm{\int_{\S} F_{j,1}(u)\eta_j^\nu(\o)d\o}_{L^2(\MM)}+\normm{\int_{\S}f^j_2 F_{j,1}(u)\eta_j^\nu(\o)d\o}_{L^2(\MM)}\\
&\les & \ep\normm{\int_{\S} F_{j,1}(u)\eta_j^\nu(\o)d\o}_{L^2(\MM)}+\normm{\int_{\S}f^j_2 F_{j,1}(u)\eta_j^\nu(\o)d\o}_{L^2(\MM)}.
\eee
Using the estimate \eqref{bisdiprop2} and the estimate in $L^2(\MM)$ \eqref{oscl2bis} for oscillatory integrals, we finally obtain:
\bee
\normm{\int_{\S}b^{-2}\trc F_{j,1}(u)\eta_j^\nu(\o)d\o}_{L^2(\MM)}&\les & \ep\gamma^\nu_j+ \ep\gamma^\nu_j 2^{\frac{j}{2}}\left(\sup_\o\norm{f^j_2}_{\li{\infty}{2}}\right)\\
\nn&\les & \ep\gamma^\nu_j
\eee
where we used \eqref{loeb17ter} in the last estimate. Together with \eqref{loeb16}, we obtain:
\be\lab{loeb18}
\norm{A_1}_{L^2(\MM)}\les  (1+p)\ep\gamma^\nu_j.
\ee

Next, we estimate $A_2$ defined by \eqref{loeb9}. In view of the Raychaudhuri equation \eqref{raychaudhuri} satisfied by $\trc$ and the transport equation \eqref{D4a} satisfied by $b$, we have:
$$L(b^{-1}\trc) = - b^{-1}\half (\trc)^2 - b^{-1}|\hch|^2 .$$
Together with the decomposition \eqref{dectrcom} for $\trc$, \eqref{dechch2om} for $|\hch|^2$ and \eqref{decbom} for $b^{-1}$, and with the $L^\infty$ estimates for $b$ and $\trc$ provided respectively by \eqref{estb} and \eqref{esttrc}, we obtain the following decomposition for $L(b^{-1}\trc)$: 
\be\lab{loeb19}
L(b^{-1}\trc)=|{\chi_2}_\nu|^2+{\chi_2}_\nu\c F^j_1+{\chi_2}_\nu\c F^j_2+f^j_3+f^j_4+f^j_5,
\ee
where the tensor $F^j_1$ and the scalar $f^j_3$ only depends on $\nu$ and satisfy:
\be\lab{loeb20}
\norm{F^j_1}_{L^\infty_{u_\nu}L^2_t L^\infty(P_{t,u_\nu})}+\norm{f^j_3}_{L^\infty_{u_\nu}L^2_t L^\infty(P_{t,u_\nu})}\les \ep,
\ee
where the tensor $F^j_2$, and the scalar $f^j_4$ satisfy:
\be\lab{loeb21}
\norm{F^j_2}_{L^\infty_u\lh{2}}+\norm{f^j_4}_{L^\infty_u\lh{2}}\les \ep 2^{-\frac{j}{2}},
\ee
and where the scalar $f^j_5$ satisfies:
\be\lab{loeb22}
\norm{f^j_5}_{L^2(\MM)}\les \ep 2^{-j}.
\ee
Together with the definition \eqref{loeb9} for $A_2$, this yields the following decomposition for $A_2$
\bee
A_2&=&(|{\chi_2}_\nu|^2+{\chi_2}_\nu\c F^j_1+f^j_3)\left(\int_{\S}\frac{1}{2}(1+\gg(N,N_\nu))\left(2^{\frac{j}{2}}(N-N_\nu)\right)^pF_j(u)\eta_j^\nu(\o)d\o\right)\\
&&+{\chi_2}_\nu\c\left(\int_{\S}\frac{1}{2}(1+\gg(N,N_\nu))\left(2^{\frac{j}{2}}(N-N_\nu)\right)^p F^j_2 F_j(u)\eta_j^\nu(\o)d\o\right)\\
&&+\left(\int_{\S}\frac{1}{2}(1+\gg(N,N_\nu))\left(2^{\frac{j}{2}}(N-N_\nu)\right)^p f^j_4 F_j(u)\eta_j^\nu(\o)d\o\right)\\
&&+\left(\int_{\S}\frac{1}{2}(1+\gg(N,N_\nu))\left(2^{\frac{j}{2}}(N-N_\nu)\right)^p f^j_5 F_j(u)\eta_j^\nu(\o)d\o\right).
\eee
We may now estimate $A_2$. We have:
\bee
&&\norm{A_2}_{L^2_{u_\nu {x'}_{\nu}}L^1_t}\\
&\les& (\norm{{\chi_2}_\nu}_{L^\infty_{u_\nu, {x'}_\nu}L^2_t}^2+\norm{F^j_1}_{L^\infty_{u_\nu}L^2_t L^\infty(P_{t,u_\nu})}^2+\norm{f^j_3}_{L^\infty_{u_\nu}L^2_t L^\infty(P_{t,u_\nu})})\\
&&\times\normm{\int_{\S}\frac{1}{2}(1+\gg(N,N_\nu))\left(2^{\frac{j}{2}}(N-N_\nu)\right)^pF_j(u)\eta_j^\nu(\o)d\o}_{L^2_{u_\nu {x'}_{\nu}}L^\infty_t}\\
&&+\norm{{\chi_2}_\nu}_{L^\infty_{u_\nu, {x'}_\nu}L^2_t}\normm{\int_{\S}\frac{1}{2}(1+\gg(N,N_\nu))\left(2^{\frac{j}{2}}(N-N_\nu)\right)^p F^j_2 F_j(u)\eta_j^\nu(\o)d\o}_{L^2(\MM)}\\
&&+\normm{\int_{\S}\frac{1}{2}(1+\gg(N,N_\nu))\left(2^{\frac{j}{2}}(N-N_\nu)\right)^p f^j_4 F_j(u)\eta_j^\nu(\o)d\o}_{L^2(\MM)}\\
&&+\normm{\int_{\S}\frac{1}{2}(1+\gg(N,N_\nu))\left(2^{\frac{j}{2}}(N-N_\nu)\right)^p f^j_5 F_j(u)\eta_j^\nu(\o)d\o}_{L^2(\MM)}\\
&\les& \ep\normm{\int_{\S}\frac{1}{2}(1+\gg(N,N_\nu))\left(2^{\frac{j}{2}}(N-N_\nu)\right)^pF_j(u)\eta_j^\nu(\o)d\o}_{L^2_{u_\nu {x'}_{\nu}}L^\infty_t}\\
&&+\ep\normm{\int_{\S}\frac{1}{2}(1+\gg(N,N_\nu))\left(2^{\frac{j}{2}}(N-N_\nu)\right)^p F^j_2 F_j(u)\eta_j^\nu(\o)d\o}_{L^2(\MM)}\\
&&+\normm{\int_{\S}\frac{1}{2}(1+\gg(N,N_\nu))\left(2^{\frac{j}{2}}(N-N_\nu)\right)^p f^j_4 F_j(u)\eta_j^\nu(\o)d\o}_{L^2(\MM)}\\
&&+\normm{\int_{\S}\frac{1}{2}(1+\gg(N,N_\nu))\left(2^{\frac{j}{2}}(N-N_\nu)\right)^p f^j_5 F_j(u)\eta_j^\nu(\o)d\o}_{L^2(\MM)},
\eee
where we used in the last inequality the estimate \eqref{loeb20} for $F^j_1$ and $F^j_3$ and the estimate \eqref{dechch1} for $\chi_2$. Using the estimate in $L^2(\MM)$ \eqref{oscl2bis} for oscillatory integrals we obtain:
\bea\lab{loeb23}
&&\norm{A_2}_{L^2_{u_\nu {x'}_{\nu}}L^1_t}\\
\nn&\les& \ep\normm{\int_{\S}(1+\gg(N,N_\nu))\left(2^{\frac{j}{2}}(N-N_\nu)\right)^pF_j(u)\eta_j^\nu(\o)d\o}_{L^2_{u_\nu {x'}_{\nu}}L^\infty_t}\\
\nn&&+\sup_\o\left((\norm{F^j_2}_{\li{\infty}{2}}+\norm{f^j_4}_{\li{\infty}{2}})\normm{(1+\gg(N,N_\nu))\left(2^{\frac{j}{2}}(N-N_\nu)\right)^p}_{L^\infty} \right)2^{\frac{j}{2}}\gamma^\nu_j\\
\nn&&+\int_{\S}\normm{(1+\gg(N,N_\nu))\left(2^{\frac{j}{2}}(N-N_\nu)\right)^p}_{L^\infty} \norm{f^j_5}_{L^2(\MM)}\norm{F_j(u)}_{L^\infty_u}\eta_j^\nu(\o)d\o\\
\nn&\les& \ep\normm{\int_{\S}(1+\gg(N,N_\nu))\left(2^{\frac{j}{2}}(N-N_\nu)\right)^pF_j(u)\eta_j^\nu(\o)d\o}_{L^2_{u_\nu {x'}_{\nu}}L^\infty_t}\\
\nn&&+\ep\gamma^\nu_j+\ep 2^{-j}\int_{\S}\norm{F_j(u)}_{L^\infty_u}\eta_j^\nu(\o)d\o\\
\nn&\les& \ep\normm{\int_{\S}(1+\gg(N,N_\nu))\left(2^{\frac{j}{2}}(N-N_\nu)\right)^pF_j(u)\eta_j^\nu(\o)d\o}_{L^2_{u_\nu {x'}_{\nu}}L^\infty_t}+\ep\gamma^\nu_j,
\eea
where we used the estimates \eqref{loeb21} and \eqref{loeb22}, Cauchy-Schwarz in $\la$ to estimate $\norm{F_j(u)}_{L^\infty_u}$, Cauchy-Schwartz in $\o$ and the size of the patch, and the fact that:
$$\norm{2^{\frac{j}{2}}(N-N_\nu)}_{L^\infty}\les 1$$
in view of the estimate \eqref{estNomega} for $\po N$ and the size of the patch.

Next, we estimate $A_3$ and $A_4$ defined respectively by \eqref{loeb10} and \eqref{loeb11}. Using the basic estimate in $L^2(\MM)$ \eqref{oscl2bis}, we obtain:
\bea\lab{loeb24}
&&\norm{A_3}_{L^2(\MM)}+\norm{A_4}_{L^2(\MM)}\\
\nn&\les& \sup_\o\Bigg(\norm{\dd\trc}_{\li{\infty}{2}}\bigg(\normm{(N_\nu-\gg(N,N_\nu)N)\left(2^{\frac{j}{2}}(N-N_\nu)\right)^p}_{L^\infty}\\
\nn&&+\normm{\frac{1}{2}(1-\gg(N,N_\nu))\left(2^{\frac{j}{2}}(N-N_\nu)\right)^p}_{L^\infty}\bigg)\Bigg) 2^{\frac{j}{2}}\ep\gamma^\nu_j\\
\nn&\les& \ep\gamma^\nu_j
\eea
where we used in the last inequality the estimate \eqref{esttrc} for $\trc$, the estimate \eqref{estNomega} for $\po N$ and the size of the patch.

Next, we estimate $A_5$ defined in \eqref{loeb12}. We first decompose $-\db L+\db_\nu L_\nu$. We have:
\bea\lab{loeb25}
2^{\frac{j}{2}}(-\db L+\db_\nu L_\nu)&=&2^{\frac{j}{2}}\Big(-\db(L-L_\nu)+(-\db+\db_\nu)L_\nu\Big)\\
\nn&=& 2^{\frac{j}{2}}\Big(-\db_\nu(N-N_\nu)+(-\db+\db_\nu)(N-N_\nu+L_\nu)\Big).
\eea
Furthermore:
\bea\lab{loeb26}
-\db+\db_\nu&=&-k_{NN}+n^{-1}n\nab_Nn+k_{N_\nu N_\nu}-n^{-1}\nab_{N_\nu}n\\
\nn&=& n^{-1}\nab n\c (N-N_\nu)-k_{N_\nu \c}(N-N_\nu)+k_{N_\nu N}-k_{NN}\\
\nn&=& \Big(n^{-1}\nab n-2k_{N_\nu\c}\Big)\c (N-N_\nu)-k(N-N_\nu, N-N_\nu)\\
\nn&=& \Big(n^{-1}\nab n-2\d_\nu N_\nu-2\epsilon_\nu\Big)\c (N-N_\nu)-k(N-N_\nu, N-N_\nu).
\eea
\eqref{loeb25}, \eqref{loeb26} and the definition \eqref{loeb12} of $A_5$ yield:
\bee
&& A_5\\
&=&-p\db_\nu\int_{\S}\frac{1}{2}(1+\gg(N,N_\nu))b^{-1}\trc\left(2^{\frac{j}{2}}(N-N_\nu)\right)^p F_j(u)\eta_j^\nu(\o)d\o\\
\nn&& +p\Big(n^{-1}\nab n-2\d_\nu N_\nu-2\epsilon_\nu\Big)\int_{\S}(N-N_\nu+L_\nu)\frac{1}{2}(1+\gg(N,N_\nu))b^{-1}\trc\left(2^{\frac{j}{2}}(N-N_\nu)\right)^p\\
\nn&&\times F_j(u)\eta_j^\nu(\o)d\o\\
\nn&& +p2^{\frac{j}{2}}\int_{\S}(N-N_\nu+L_\nu)k(N-N_\nu, N-N_\nu)\frac{1}{2}(1+\gg(N,N_\nu))b^{-1}\trc\left(2^{\frac{j}{2}}(N-N_\nu)\right)^{p-1}\\
\nn&&\times F_j(u)\eta_j^\nu(\o)d\o
\eee
which implies:
\bea\lab{loeb27}
&&\norm{A_5}_{L^2_{u_\nu {x'}_{\nu}}L^1_t}\\
\nn&\les & p\norm{\db_\nu}_{L^\infty_{u_\nu {x'}_{\nu}}L^2_t}\normm{\int_{\S}(1+\gg(N,N_\nu))b^{-1}\trc\left(2^{\frac{j}{2}}(N-N_\nu)\right)^p F_j(u)\eta_j^\nu(\o)d\o}_{L^2(\MM)}\\
\nn&& +p\Big(\norm{n^{-1}\nab n}_{L^\infty(\MM)}+\norm{\d_\nu}_{L^\infty_{u_\nu {x'}_{\nu}}}+\norm{\epsilon_\nu}_{L^\infty_{u_\nu {x'}_{\nu}}}\Big)\\
\nn&&\times \normm{\int_{\S}(N-N_\nu+L_\nu)(1+\gg(N,N_\nu))b^{-1}\trc\left(2^{\frac{j}{2}}(N-N_\nu)\right)^p F_j(u)\eta_j^\nu(\o)d\o}_{L^2(\MM)}\\
\nn&& +p2^{\frac{j}{2}}\bigg\|\int_{\S}(N-N_\nu+L_\nu)k(N-N_\nu, N-N_\nu)(1+\gg(N,N_\nu))b^{-1}\trc\\
\nn&&\times\left(2^{\frac{j}{2}}(N-N_\nu)\right)^{p-1}F_j(u)\eta_j^\nu(\o)d\o\bigg\|_{L^2(\MM)}\\
\nn&\les & p\ep\normm{\int_{\S}(1+\gg(N,N_\nu))b^{-1}\trc\left(2^{\frac{j}{2}}(N-N_\nu)\right)^p F_j(u)\eta_j^\nu(\o)d\o}_{L^2(\MM)}\\
\nn&& +p\ep\normm{\int_{\S}(N-N_\nu+L_\nu)(1+\gg(N,N_\nu))b^{-1}\trc\left(2^{\frac{j}{2}}(N-N_\nu)\right)^p F_j(u)\eta_j^\nu(\o)d\o}_{L^2(\MM)}\\
\nn&& +p2^{\frac{j}{2}}\bigg\|\int_{\S}(N-N_\nu+L_\nu)k(N-N_\nu, N-N_\nu)(1+\gg(N,N_\nu))b^{-1}\trc\\
\nn&&\times\left(2^{\frac{j}{2}}(N-N_\nu)\right)^{p-1}F_j(u)\eta_j^\nu(\o)d\o\bigg\|_{L^2(\MM)}
\eea
where we used in the last inequality the estimates \eqref{estn} for $n$ and the estimates \eqref{estk} for $\d$ and $\epsilon$. The first two terms in the right-hand side of \eqref{loeb27} are similar to $A_1$ and can be estimated in the same way. In view of \eqref{loeb18}, we obtain:
\bee
\norm{A_5}_{L^2_{u_\nu {x'}_{\nu}}L^1_t}&\les & (1+p^2)\gamma^\nu_j+p2^{\frac{j}{2}}\bigg\|\int_{\S}(N-N_\nu+L_\nu) k(N-N_\nu, N-N_\nu)\\
\nn&&\times (1+\gg(N,N_\nu))b^{-1}\trc\left(2^{\frac{j}{2}}(N-N_\nu)\right)^{p-1}F_j(u)\eta_j^\nu(\o)d\o\bigg\|_{L^2(\MM)}.
\eee
Using the basic estimate in $L^2(\MM)$ \eqref{oscl2bis}, this yields:
\bea\lab{loeb28}
&&\norm{A_5}_{L^2_{u_\nu {x'}_{\nu}}L^1_t}\\
\nn&\les & (1+p^2)\gamma^\nu_j+p2^{\frac{j}{2}}\sup_\o\Bigg(\norm{k}_{\li{\infty}{2}}\\
\nn&&\normm{(N-N_\nu+L_\nu)(N-N_\nu)^2(1+\gg(N,N_\nu))b^{-1}\trc\left(2^{\frac{j}{2}}(N-N_\nu)\right)^{p-1}}_{L^\infty}\Bigg)2^{\frac{j}{2}}\gamma^\nu_j\\
\nn&\les & (1+p^2)\gamma^\nu_j,
\eea
where we used in the last inequality the estimates \eqref{estk} for $k$, \eqref{estb} for $b$, \eqref{esttrc} for $\trc$ and \eqref{estNomega} for $\po N$, and the size of the patch.

Next, we estimate $A_6$ defined in \eqref{loeb13}. We first decompose $\kep$. We have, schematically:
\be\lab{loeb29}
\kep=k_{N_\nu\c}+k(N-N_\nu,.)=\d_\nu N_\nu+\ep_\nu+k(N-N_\nu,.).
\ee
Together with the decompositions \eqref{dectrcom} for $\trc$ and \eqref{dechchom} $\hch$, this yields:
\bee
&& 2^{\frac{j}{2}}(\chi(N_\nu-\gg(N,N_\nu)N,e_A)e_A-\epsilon_{N_\nu-\gg(N,N_\nu)N}L)\\
&=& F^j_1 2^{\frac{j}{2}}(N-N_\nu)+F^j_1 2^{\frac{j}{2}}(N-N_\nu)
\eee
where the tensor $F^j_1$ only depends on $\nu$ and satisfies:
\be\lab{loeb30}
\norm{F^j_1}_{L^2_{u_\nu}, x'_{\nu}L^2_t}\les \ep,
\ee
and where the tensor $F^j_2$ satisfies:
\be\lab{loeb31}
\norm{F^j_2}_{L^\infty_u\lh{2}}\les \ep 2^{-\frac{j}{2}}.
\ee
In view of the definition \eqref{loeb13} of $A_6$, we obtain:
\bee
A_6&=&pF^j_1\int_{\S}b^{-1}\trc \left(2^{\frac{j}{2}}(N-N_\nu)\right)^p F_j(u)\eta_j^\nu(\o)d\o\\
\nn&& +p\int_{\S}b^{-1}\trc F^j_2 \left(2^{\frac{j}{2}}(N-N_\nu)\right)^p F_j(u)\eta_j^\nu(\o)d\o.
\eee
This yields:
\bea\lab{loeb32}
\norm{A_6}_{L^2_{u_\nu {x'}_{\nu}}L^1_t}&\les& p\norm{F^j_1}_{L^\infty_{u_\nu {x'}_{\nu}}L^2_t}\normm{\int_{\S}b^{-1}\trc \left(2^{\frac{j}{2}}(N-N_\nu)\right)^p F_j(u)\eta_j^\nu(\o)d\o}_{L^2(\MM)}\\
\nn&& +p\normm{\int_{\S}b^{-1}\trc F^j_2 \left(2^{\frac{j}{2}}(N-N_\nu)\right)^p F_j(u)\eta_j^\nu(\o)d\o}_{L^2(\MM)}\\
\nn&\les& p\ep \normm{\int_{\S}b^{-1}\trc \left(2^{\frac{j}{2}}(N-N_\nu)\right)^p F_j(u)\eta_j^\nu(\o)d\o}_{L^2(\MM)}\\
\nn&& +p\normm{\int_{\S}b^{-1}\trc F^j_2 \left(2^{\frac{j}{2}}(N-N_\nu)\right)^p F_j(u)\eta_j^\nu(\o)d\o}_{L^2(\MM)},
\eea
where we used the estimate \eqref{loeb30} in the last inequality. The first term in the right-hand side of \eqref{loeb32} are similar to $A_1$ and can be estimated in the same way. In view of \eqref{loeb18}, we obtain:
$$\norm{A_6}_{L^2_{u_\nu {x'}_{\nu}}L^1_t} \les (1+p^2)\ep \gamma^\nu_j+p\normm{\int_{\S}b^{-1}\trc F^j_2 \left(2^{\frac{j}{2}}(N-N_\nu)\right)^p F_j(u)\eta_j^\nu(\o)d\o}_{L^2(\MM)}.$$
Using the basic estimate in $L^2(\MM)$ \eqref{oscl2bis}, this yields:
\bea\lab{loeb33}
&&\norm{A_6}_{L^2_{u_\nu {x'}_{\nu}}L^1_t} \\
\nn&\les& (1+p^2)\ep \gamma^\nu_j+p\sup_\o\left(\norm{F^j_2}_{\li{\infty}{2}}\normm{b^{-1}\trc \left(2^{\frac{j}{2}}(N-N_\nu)\right)^p}_{L^\infty}\right)2^{\frac{j}{2}}\gamma^\nu_j\\
\nn&\les& (1+p^2)\ep \gamma^\nu_j,
\eea
where we used in the last inequality the estimate \eqref{loeb31} for $F^j_2$, the estimate \eqref{estb} for $b$, the estimate \eqref{esttrc} for $\trc$, the estimate \eqref{estNomega} for $\po N$, and the size of the patch.

Finally, we estimate $A_7$. In view of the definition \eqref{loeb14} for $A_7$ and the basic estimate in $L^2(\MM)$ \eqref{oscl2bis}, we have:
\bea\lab{loeb34}
\norm{A_7}_{L^2(\MM)}&\les &p 2^{\frac{j}{2}}\sup_\o\Bigg((\norm{\z}_{\li{\infty}{2}}+\norm{\d}_{\li{\infty}{2}}+\norm{n^{-1}\nab n}_{\li{\infty}{2}})\\
\nn&&\times\normm{(1-\gg(N,N_\nu))b^{-1}\trc\left(2^{\frac{j}{2}}(N-N_\nu)\right)^{p-1}}_{L^\infty}\Bigg)2^{\frac{j}{2}}\gamma^\nu_j\\
\nn&\les& p\ep \gamma^\nu_j,
\eea
where we used in the last inequality the estimate \eqref{estzeta} for $\z$, the estimate \eqref{estk} for $\d$, the estimate \eqref{estn} for $n$, the estimate \eqref{estb} for $b$, the estimate \eqref{esttrc} for $\trc$, the estimate \eqref{estNomega} for $\po N$, the size of the patch, and the fact that:
\be\lab{loeb34bis}
1-\gg(N,N_\nu)=\frac{\gg(N-N_\nu, N-N_\nu)}{2}.
\ee
We have:
\bee
&& \normm{\int_{\S} b^{-1}\trc\left(2^{\frac{j}{2}}(N-N_\nu)\right)^pF_j(u)\eta_j^\nu(\o)d\o}_{L^2_{u_\nu {x'}_{\nu}}L^\infty_t}\\
&\les & \norm{L_\nu\left(\int_{\S} b^{-1}\trc\left(2^{\frac{j}{2}}(N-N_\nu)\right)^pF_j(u)\eta_j^\nu(\o)d\o\right)}_{L^2_{u_\nu {x'}_{\nu}}L^1_t}\\
\nn&\les& \norm{A_1}_{L^2_{u_\nu {x'}_{\nu}}L^1_t}+\norm{A_2}_{L^2_{u_\nu {x'}_{\nu}}L^1_t}+\norm{A_3}_{L^2_{u_\nu {x'}_{\nu}}L^1_t}+\norm{A_4}_{L^2_{u_\nu {x'}_{\nu}}L^1_t}+\norm{A_5}_{L^2_{u_\nu {x'}_{\nu}}L^1_t}\\
&&+\norm{A_6}_{L^2_{u_\nu {x'}_{\nu}}L^1_t}+\norm{A_7}_{L^2_{u_\nu {x'}_{\nu}}L^1_t},
\eee
where we used \eqref{loeb7} in the last inequality. Together with \eqref{loeb18}, \eqref{loeb23}, \eqref{loeb24}, \eqref{loeb28}, \eqref{loeb33} and \eqref{loeb34}, we obtain:
\bea\lab{loeb35}
&& \normm{\int_{\S} b^{-1}\trc\left(2^{\frac{j}{2}}(N-N_\nu)\right)^pF_j(u)\eta_j^\nu(\o)d\o}_{L^2_{u_\nu {x'}_{\nu}}L^\infty_t}\\
\nn&\les& \ep\normm{\int_{\S}(1+\gg(N,N_\nu))\left(2^{\frac{j}{2}}(N-N_\nu)\right)^pF_j(u)\eta_j^\nu(\o)d\o}_{L^2_{u_\nu {x'}_{\nu}}L^\infty_t}+(1+p^2)\ep\gamma^\nu_j.
\eea
Now, we have:
\bee
&&\normm{\int_{\S}(1-\gg(N,N_\nu))\left(2^{\frac{j}{2}}(N-N_\nu)\right)^pF_j(u)\eta_j^\nu(\o)d\o}_{L^2_{u_\nu {x'}_{\nu}}L^\infty_t}\\
\nn&\les& \int_{\S}\normm{(1-\gg(N,N_\nu))\left(2^{\frac{j}{2}}(N-N_\nu)\right)^pF_j(u)}\eta_j^\nu(\o)d\o_{L^2_{u_\nu {x'}_{\nu}}L^\infty_t}\\
\nn&\les& \sup_\o\left(\norm{1-\gg(N,N_\nu)}_{L^2_{u_\nu {x'}_{\nu}}L^\infty_t}\normm{\left(2^{\frac{j}{2}}(N-N_\nu)\right)^p}_{L^\infty}\right)\int_{\S}\norm{F_j(u)}_{L^\infty_u}\eta_j^\nu(\o)d\o\\
\nn&\les& 2^{-j}\int_{\S}\norm{F_j(u)}_{L^\infty_u}\eta_j^\nu(\o)d\o\\
\nn&\les& \gamma^\nu_j,
\eee
where we used \eqref{loeb34bis}, the estimate \eqref{estNomega} for $\po N$, the size of the patch, Cauchy-Schwarz in $\la$ to estimate $\norm{F_j(u)}_{L^\infty_u}$, and Cauchy-Schwarz in $\o$. Together with \eqref{loeb35}, this yields:
\bea\lab{loeb36}
&& \normm{\int_{\S} b^{-1}\trc\left(2^{\frac{j}{2}}(N-N_\nu)\right)^pF_j(u)\eta_j^\nu(\o)d\o}_{L^2_{u_\nu {x'}_{\nu}}L^\infty_t}\\
\nn&\les& \ep\normm{\int_{\S}\left(2^{\frac{j}{2}}(N-N_\nu)\right)^pF_j(u)\eta_j^\nu(\o)d\o}_{L^2_{u_\nu {x'}_{\nu}}L^\infty_t}+(1+p^2)\ep\gamma^\nu_j.
\eea
Note that the first term in the right-hand side corresponds to the left-hand side where $b^{-1}\trc$ has been replaced by 1. In particular, we have the analog of \eqref{loeb7}:
\bea\lab{loeb7bis}
&& L_\nu\left(\int_{\S} \left(2^{\frac{j}{2}}(N-N_\nu)\right)^pF_j(u)\eta_j^\nu(\o)d\o\right)\\
\nn&=& A_1'+A_5'+A_6'+A_7',
\eea
where $A_1', A_5', A_6'$ and $A_7'$ are respectively given by:
\be\lab{loeb8bis}
A_1'=\int_{\S}b^{-1}\left(2^{\frac{j}{2}}(N-N_\nu)\right)^{p+2}F_{j,1}(u)\eta_j^\nu(\o)d\o,
\ee
\bea\lab{loeb12bis}
A_5'&=&p\int_{\S}\frac{1}{2}(1+\gg(N,N_\nu))\left(2^{\frac{j}{2}}(N-N_\nu)\right)^{p-1}(2^{\frac{j}{2}}(-\db L+\db_\nu L_\nu))\\
\nn&&\times F_j(u)\eta_j^\nu(\o)d\o,
\eea
\bea\lab{loeb13bis}
A_6'&=&p\int_{\S}\left(2^{\frac{j}{2}}(N-N_\nu)\right)^{p-1}\\
\nn&&\times (2^{\frac{j}{2}}(\chi(N_\nu-\gg(N,N_\nu)N,e_A)e_A-\epsilon_{N_\nu-\gg(N,N_\nu)N}L))F_j(u)\eta_j^\nu(\o)d\o,
\eea
and:
\bea\lab{loeb14bis}
A_7'&=&p\int_{\S}\frac{1}{2}(1-\gg(N,N_\nu))\left(2^{\frac{j}{2}}(N-N_\nu)\right)^{p-1}\\
\nn&&\times (2^{\frac{j}{2}}(\z_Ae_A+(\d+n^{-1}\nabla_Nn)L)F_j(u)\eta_j^\nu(\o)d\o.
\eea
The analog of the estimates \eqref{loeb18}, \eqref{loeb28}, \eqref{loeb33} and \eqref{loeb34} for $A_1', A_5', A_6'$ and $A_7'$ yield:
\bee
&& \normm{\int_{\S}\left(2^{\frac{j}{2}}(N-N_\nu)\right)^pF_j(u)\eta_j^\nu(\o)d\o}_{L^2_{u_\nu {x'}_{\nu}}L^\infty_t}\\
\nn&\les& \norm{A_1'}_{L^2_{u_\nu {x'}_{\nu}}L^1_t}+\norm{A_5'}_{L^2_{u_\nu {x'}_{\nu}}L^1_t}+\norm{A_6'}_{L^2_{u_\nu {x'}_{\nu}}L^1_t}+\norm{A_7'}_{L^2_{u_\nu {x'}_{\nu}}L^1_t}\\
&\les& (1+p^2)\gamma^\nu_j.
\eee
Together with \eqref{loeb36}, we obtain:
\bee
&& \normm{\int_{\S} b^{-1}\trc\left(2^{\frac{j}{2}}(N-N_\nu)\right)^pF_j(u)\eta_j^\nu(\o)d\o}_{L^2_{u_\nu {x'}_{\nu}}L^\infty_t}\\
\nn&\les& (1+p^2)\ep\gamma^\nu_j.
\eee
This concludes the proof of the lemma.
\end{proof}

\begin{lemma}\lab{lemma:loebbis}
Let $p\in \mathbb{N}$ and $q\in\mathbb{Z}$. We have:
\be\lab{loebbis}
\normm{\int_{\S} b^q\left(2^{\frac{j}{2}}(N-N_\nu)\right)^pF_j(u)\eta_j^\nu(\o)d\o}_{L^2_{u_\nu {x'}_{\nu}}L^\infty_t}\les (1+p^2)\gamma^\nu_j.
\ee
and:
\be\lab{loebter}
\normm{\int_{\S} b^q\trc\left(2^{\frac{j}{2}}(N-N_\nu)\right)^pF_j(u)\eta_j^\nu(\o)d\o}_{L^2_{u_\nu {x'}_{\nu}}L^\infty_t}\les (1+p^2)\ep\gamma^\nu_j.
\ee
\end{lemma}

The proof of Lemma \ref{lemma:loebbis} is completely analogous to the proof of Lemma \ref{lemma:loeb} and is left  to the reader.

Next, we obtain estimates evaluating the $L^2_{u,x'}L^\infty_t$ of $H$ where $u=u(t,x,\o)$, and where $H$ is naturally defined with respect to the foliation of $u'=u(t,x,\o')$. We start with a basic lemma.

\begin{lemma}\lab{lemma:messi}
Let $H$ a tensor on $\MM$, and $\o, \o'$ two angles in $\S$. Let $u=u(t,x,\o)$, and $L$ corresponding to $u$. Let $u'=u(t,x,\o')$, and $L', \lb', \nabb'$ corresponding to $u'$. Then, we have the following estimate for the $L^2_uL^\infty_tL^2(\ptu)$ of $H$:
\bea
\nn\norm{H}_{L^2_{u,x'}L^\infty_t}^2&\les& \norm{H}_{L^2(\MM)}\norm{\dd_{\lb'}H}_{L^2(\MM)}|\o-\o'|^2+\norm{H}_{L^2(\MM)}\norm{\nabb'H}_{L^2(\MM)}|\o-\o'|\\
\lab{messi}&&+\int_{\MM}|H||\dd_{L'}H|d\MM+\norm{H}_{L^2(\MM)}^2.
\eea
\end{lemma}

\begin{proof}
Recall the estimate \eqref{murray} on $\H_u$:
$$\norm{H}_{\xt{2}{\infty}}^2\les \int_{\H_u}|H||\dd_LH| dt \dmt+\norm{H}^2_{L^2(\H_u)}.$$
Integrating in $u$, and using the expression of the volume element $d\MM$ in the coordinate system $(u,t,x')$ \eqref{coarea} and the control of $b$ in $L^\infty$ given by \eqref{estb}, we obtain:
\be\lab{messi1}
\norm{H}_{L^2_{u, x'}L^\infty_t}^2\les \int_{\MM}|H||\dd_LH| d\MM+\norm{H}^2_{L^2(\MM)}.
\ee

Next, we decompose $L$ on the frame $L', \lb', e'_A, A=1, 2$. We have:
\be\lab{messi2}
L=-\frac{1}{2}\gg(L,\lb')L'-\frac{1}{2}\gg(L,L')\lb'+\gg(L,e'_A)e'_A.
\ee
Now, we have:
\be\lab{borekbis}
1-\gn=\frac{\gg(N-N',N-N')}{2}\sim |\o-\o'|^2,
\ee
where we used \eqref{threomega1ter}. \eqref{borekbis} and the estimate \eqref{estNomega} for $\po N$ yield:
\be\lab{borekter}
\gg(L,L')\sim |\o-\o'|^2,\,\gg(L,e'_A)=\gg(L-L',e'_A)\sim |\o-\o'|\textrm{ and }\gg(L,\lb')=-2+\gg(L,L').
\ee
Together with \eqref{messi2}, this yields:
\be\lab{messi3}
L=(-2+O(|\o-\o'|^2))L'+O(|\o-\o'|)\nabb'+O(|\o-\o'|^2)\lb'.
\ee

Finally, plugging \eqref{messi3} in \eqref{messi1} and using Cauchy-Schwartz yields \eqref{messi}. This concludes the proof of the lemma.
\end{proof}

We have the following corollary of Lemma \ref{lemma:messi}
\begin{corollary}\lab{cor:messi1}
Let $\o, \nu, \o', \nu'$ four angles in $\S$ such that $\o$ belongs to the patch of center $\nu$ and $\o'$ belongs to the patch of center $\nu'$. Let $u=u(t,x,\o)$. Let $u'=u(t,x,\o')$, and $L', \lb', \nabb'$ corresponding to $u'$. Let a tensor $G$ on $\MM$. Then, we have the following estimate:
\bea\lab{messi4:0}
&&\normm{\int_{\S}G P_l'\trc'F_j(u')\eta^{\nu'}_j(\o') d\o'}_{L^2_{u,x'}L^\infty_t}\\
\nn&\les& 
\left(\sup_{\o'}\norm{G}_{L^\infty}\right)\ep\big(2^{\frac{j}{2}}|\nu-\nu'|2^{-l+\frac{j}{2}}+(2^{\frac{j}{2}}|\nu-\nu'|)^{\frac{1}{2}}2^{-\frac{l}{2}+\frac{j}{4}}\big)\gamma^{\nu'}_j.
\eea
\end{corollary}

\begin{proof}
We have:
\bea
&&\nn\normm{\int_{\S} G P_l'\trc'F_j(u')\eta^{\nu'}_j(\o') d\o'}_{L^2_{u,x'}L^\infty_t}\\
\nn&\les&\int_{\S}\norm{G}_{L^\infty}\norm{P_l'\trc'F_j(u')}_{L^2_uL^\infty_tL^2_{x'}}\eta^{\nu'}_j(\o') d\o'\\
\lab{messi4}&\les& \left(\sup_{\o'}\norm{G}_{L^\infty}\right)2^{-\frac{j}{2}}\normm{\sqrt{\eta^{\nu'}_j(\o')}\norm{P_l'\trc'F_j(u')}_{L^2_uL^\infty_tL^2_{x'}}}_{L^2_{\o'}},
\eea
where we used in the last inequality the estimate \eqref{estb} for $b$, Cauchy-Schwartz in $\o'$ and the size of the patch.

Next, we apply \eqref{messi} with the choice $H=P_l'\trc'F_j(u')$:
\bee
\norm{P_l'\trc'F_j(u')}_{L^2_{u,x'}L^\infty_t}^2&\les& \norm{P_l'\trc'F_j(u')}_{L^2(\MM)}\norm{\lb'(P_l'\trc'F_j(u'))}_{L^2(\MM)}|\nu-\nu'|^2\\
\nn&&+\norm{P_l'\trc'F_j(u')}_{L^2(\MM)}\norm{\nabb'(P_l'\trc')F_j(u')}_{L^2(\MM)}|\nu-\nu'|\\
\nn&&+\int_{\MM}|P_l'\trc'F_j(u')||L'(P_l'\trc')F_j(u')|d\MM+\norm{P_l'\trc'F_j(u')}_{L^2(\MM)}^2,
\eee
where we used the fact that $|\o-\o'|\sim |\nu-\nu'|$ and $L'(u')=\nabb'(u')=0$. Also, since $\lb'(u')=-2{b'}^{-1}$ and $\la'\sim 2^j$, we have:
$$\lb'(F_j(u'))\sim 2^j{b'}^{-1}F_j(u')$$
and we obtain:
\bea
\lab{messi6}\norm{P_l'\trc'F_j(u')}_{L^2_{u,x'}L^\infty_t}^2 &\les& \bigg(\norm{P_l'\trc'}_{\li{\infty}{2}}\big(2^j\norm{P_l'\trc'}_{\li{\infty}{2}}|\nu-\nu'|^2\\
\nn&&+\norm{\lb'(P_l'\trc')}_{\li{\infty}{2}}|\nu-\nu'|^2+\norm{\nabb'P_l'\trc'}_{\li{\infty}{2}}|\nu-\nu'|\big)\\
\nn&&+\int_{\H_u}|P_l'\trc'||L'(P_l'\trc')|d\H_u+\norm{P_l'\trc'}^2_{\li{\infty}{2}}\bigg)\norm{F_j(u')}_{L^2_u}^2\\
\nn&\les& \bigg(2^{-l}\ep\big(2^{j-l}\ep |\nu-\nu'|^2+\norm{\lb'(P_l'\trc')}_{\li{\infty}{2}}|\nu-\nu'|^2+\ep |\nu-\nu'|\big)\\
\nn&&+\int_{\H_u}|P_l'\trc'||L'(P_l'\trc')|d\H_u+2^{-2l}\ep^2\bigg)\norm{F_j(u')}_{L^2_u}^2,
\eea
where we used in the last inequality the finite band property for $P_l'$ and the estimate \eqref{esttrc} for $\trc$.

Next, we evaluate $\norm{\lb'(P_l'\trc')}_{\li{\infty}{2}}$. We have:
\bee
\norm{\lb'(P_l'\trc')}_{\li{\infty}{2}}&\les &\norm{L'(P_l'\trc')}_{\li{\infty}{2}} +\norm{N'(P_l'\trc')}_{\li{\infty}{2}}\\
&\les &\norm{nL'(P_l'\trc')}_{\li{\infty}{2}} +\norm{b'N'(P_l'\trc')}_{\li{\infty}{2}}\\
&\les &\norm{P_l'(nL'(\trc'))}_{\li{\infty}{2}} +\norm{[nL',P_l']\trc'}_{\li{\infty}{2}}\\
&& +\norm{P_l'(b'N'(\trc'))}_{\li{\infty}{2}} +\norm{[b'N',P_l']\trc'}_{\li{\infty}{2}},
\eee
where we used the fact that $\lb'=L'-2N'$, the estimate \eqref{estb} for $b$ and the estimate \eqref{estn} for $n$. Together with the estimates \eqref{esttrc} for $\trc$, \eqref{estb} for $b$ and \eqref{estn} for $n$, and the commutator estimates \eqref{commlp1} and \eqref{commlp2}, we obtain:
\bea
\nn\norm{\lb'(P_l'\trc')}_{\li{\infty}{2}}&\les & \norm{nL'(\trc')}_{\li{\infty}{2}} +\norm{b'N'(\trc')}_{\li{\infty}{2}} +\no(\trc')\\
\lab{messi7}&\les& \ep.
\eea

Next, we estimate the integral over $\H_u$ in the right-hand side of \eqref{messi6}. We have:
\bea\lab{messi8}
\int_{\H_u}|P_l'\trc'||L'(P_l'\trc')|d\H_u&\les & \int_{\H_u}|P_l'\trc'||nL'(P_l'\trc')|d\H_u\\
\nn&\les & \norm{P_l'\trc'}_{L^2_{{x'}'}L^\infty_t}\norm{P_l'(n\trc')}_{L^2_{{x'}'}L^1_t}\\
\nn&&+\norm{P_l'\trc'}_{L^2_{{x'}'}L^\infty_t}\norm{[nL',P_l'](\trc')}_{L^2_{{x'}'}L^1_t}\\
\nn&\les & 2^{-2l}\ep^2,
\eea
where we used in the last inequality the finite band property for $P_l'$, the commutator estimate \eqref{commlp3},  the estimate \eqref{esttrc} for $\trc$, and the estimate \eqref{lievremont1} for $P_l\trc$ and $P_l(nL\trc)$. 

Finally, \eqref{messi6}, \eqref{messi7} and \eqref{messi8} imply:
\bee
\norm{P_l'\trc'F_j(u')}_{L^2_{u,x'}L^\infty_t}^2 &\les& \left(2^j|\nu-\nu'|^2+2^l|\nu-\nu'|+1\right)\ep^22^{-2l}\norm{F_j(u')}_{L^2_u}^2,
\eee
which together with \eqref{messi4} implies:
\bee
&&\nn\normm{\int_{\S}G P_l'\trc'F_j(u')\eta^{\nu'}_j(\o') d\o'}_{L^2_{u,x'}L^\infty_t}\\
&\les& \left(\sup_{\o'}\norm{G}_{L^\infty}\right)\left(2^{\frac{j}{2}}|\nu-\nu'|+2^{\frac{l}{2}}|\nu-\nu'|^{\frac{1}{2}}+1\right)\ep 2^{-l-\frac{j}{2}}\normm{\sqrt{\eta^{\nu'}_j(\o')}\norm{F_j(u')}_{L^2_u}}_{L^2_{\o'}}\\
&\les& \left(\sup_{\o'}\norm{G}_{L^\infty}\right)\left(2^{\frac{j}{2}}|\nu-\nu'|2^{-l+\frac{j}{2}}+2^{-\frac{l}{2}+\frac{j}{4}}(2^{\frac{j}{2}}|\nu-\nu'|)^{\frac{1}{2}}\right)\ep \gamma^{\nu'}_j,
\eee
where we used in the last inequality Plancherel in $u'$ and the fact that 
$$2^{\frac{j}{2}}|\nu-\nu'|\gtrsim 1$$
since $\nu\neq \nu'$. This conclude the proof of the corollary.
\end{proof} 

Lemma \ref{lemma:messi} yields also a second corollary.
\begin{corollary}\lab{cor:messi2}
Let $\o, \nu, \o', \nu'$ four angles in $\S$ such that $\o$ belongs to the patch of center $\nu$ and $\o'$ belongs to the patch of center $\nu'$. Let $u=u(t,x,\o)$. Let $u'=u(t,x,\o')$, and $L', \lb', \nabb'$ corresponding to $u'$. Let a tensor $G$ on $\MM$. Then, we have the following estimate:
\bea\lab{messi11}
&&\normm{\int_{\S}G \nabb'P_{\leq l}'\trc'F_j(u')\eta^{\nu'}_j(\o') d\o'}_{L^2_{u,x'}L^\infty_t}\\
\nn&\les& \left(\sup_{\o'}\norm{G}_{L^\infty}\right)
\ep\big(2^{\frac{j}{2}}|\nu-\nu'|2^{\frac{j}{2}}+(2^{\frac{j}{2}}|\nu-\nu'|)^{\frac{1}{2}}2^{\frac{l}{2}+\frac{j}{4}}\big)\gamma^{\nu'}_j.
\eea
\end{corollary}

\begin{proof}
We have:
\bea
&&\nn\normm{\int_{\S}G \nabb'P_{\leq l}'\trc'F_j(u')\eta^{\nu'}_j(\o') d\o'}_{L^2_{u,x'}L^\infty_t}\\
\nn&\les&
\int_{\S}\norm{G}_{L^\infty}\norm{\nabb'P_{\leq l}'\trc'F_j(u')}_{L^2_uL^\infty_tL^2_{x'}}\eta^{\nu'}_j(\o') d\o'\\
\lab{messibis4}&\les& \left(\sup_{\o'}\norm{G}_{L^\infty}\right)2^{-\frac{j}{2}}\normm{\sqrt{\eta^{\nu'}_j(\o')}\norm{\nabb'P_{\leq l}'\trc'F_j(u')}_{L^2_uL^\infty_tL^2_{x'}}}_{L^2_{\o'}},
\eea
where we used in the last inequality the estimate \eqref{estb} for $b$, Cauchy-Schwartz in $\o'$ and the size of the patch.

Next, we apply \eqref{messi} with the choice $H=\nabb'P_{\leq l}'\trc'F_j(u')$, and we obtain the analogous estimate to \eqref{messi6}:
\bea\lab{messibis6}
&&\norm{\nabb'P_{\leq l}'\trc'F_j(u')}_{L^2_{u,x'}L^\infty_t}^2\\
\nn&\les& \bigg(\norm{\nabb'P_{\leq l}'\trc'}_{\li{\infty}{2}}\big(2^j\norm{\nabb'P_{\leq l}'\trc'}_{\li{\infty}{2}}|\nu-\nu'|^2\\
\nn&&+\norm{\lb'\nabb'(P_{\leq l}'\trc')}_{\li{\infty}{2}}|\nu-\nu'|^2+\norm{{\nabb'}^2P_{\leq l}'\trc'}_{\li{\infty}{2}}|\nu-\nu'|\big)\\
\nn&&+\int_{\H_u}|\nabb'P_{\leq l}'\trc'||L'\nabb'(P_{\leq l}'\trc')|d\H_u+\norm{\nabb'P_{\leq l}'\trc'}^2_{\li{\infty}{2}}\bigg)\norm{F_j(u')}_{L^2_u}^2\\
\nn&\les& \bigg(\ep\big(2^j\ep |\nu-\nu'|^2+\norm{\lb'\nabb'(P_{\leq l}'\trc')}_{\li{\infty}{2}}|\nu-\nu'|^2+\ep 2^l|\nu-\nu'|\big)\\
\nn&&+\int_{\H_u}|\nabb'P_{\leq l}'\trc'||L'\nabb'(P_{\leq l}'\trc')|d\H_u+\ep^2\bigg)\norm{F_j(u')}_{L^2_u}^2,
\eea
where we used in the last inequality the finite band property for $P_{\leq l}'$, the Bochner inequality \eqref{eq:Bochconseqbis}, and the estimate \eqref{esttrc} for $\trc$.

Next, we evaluate $\norm{\lb'\nabb'(P_{\leq l}'\trc')}_{\li{\infty}{2}}$. We have:
\bee
&&\norm{\lb'\nabb'(P_{\leq l}'\trc')}_{\li{\infty}{2}}\\
&\les &\norm{L'\nabb'(P_{\leq l}'\trc')}_{\li{\infty}{2}} +\norm{N'\nabb'(P_{\leq l}'\trc')}_{\li{\infty}{2}}\\
&\les &\norm{nL'\nabb'(P_{\leq l}'\trc')}_{\li{\infty}{2}} +\norm{b'N'\nabb'(P_{\leq l}'\trc')}_{\li{\infty}{2}}\\
&\les &\norm{\nabb'P_{\leq l}'(nL'(\trc'))}_{\li{\infty}{2}} +\norm{[nL',\nabb']P_{\leq l}'(\trc')}_{\li{\infty}{2}}\\
&&+\norm{\nabb'[nL',P_{\leq l}']\trc'}_{\li{\infty}{2}}+\norm{\nabb'P_{\leq l}'(b'N'(\trc'))}_{\li{\infty}{2}}\\  
&&+\norm{[b'N',\nabb'](P_{\leq l}'\trc')}_{\li{\infty}{2}}+\norm{\nabb'[b'N',P_{\leq l}']\trc'}_{\li{\infty}{2}},
\eee
where we used the fact that $\lb'=L'-2N'$, the estimate \eqref{estb} for $b$ and the estimate \eqref{estn} for $n$. Together with the finite band property for $P_{\leq l}'$, the estimates \eqref{esttrc} for $\trc$, \eqref{estb} for $b$ and \eqref{estn} for $n$, and the commutator estimates \eqref{commlp1} and \eqref{commlp2}, we obtain:
\bea
&&\nn\norm{\lb'\nabb'(P_{\leq l}'\trc')}_{\li{\infty}{2}}\\
\nn&\les & 2^l\norm{nL'(\trc')}_{\li{\infty}{2}}+2^l\norm{b'N'(\trc')}_{\li{\infty}{2}}  +\no(\trc')\\
\nn&&+\norm{[nL',\nabb']P_{\leq l}'(\trc')}_{\li{\infty}{2}}+\norm{[b'N',\nabb'](P_{\leq l}'\trc')}_{\li{\infty}{2}}\\
\lab{messibis7bis}&\les& 2^l\ep+\norm{[nL',\nabb']P_{\leq l}'(\trc')}_{\li{\infty}{2}}+\norm{[b'N',\nabb'](P_{\leq l}'\trc')}_{\li{\infty}{2}}.
\eea
To estimate the commutator terms in the right-hand side of \eqref{messibis7bis}, we use the commutator formulas \eqref{comm5} and \eqref{comm7}:
\bee
&&\norm{[nL',\nabb']P_{\leq l}'(\trc')}_{\li{\infty}{2}}+\norm{[b'N',\nabb'](P_{\leq l}'\trc')}_{\li{\infty}{2}}\\
&\les &
\norm{\chi'\nabb'(P_{\leq l}'\trc')}_{\li{\infty}{2}}+\norm{k\nabb'P_{\leq l}'(\trc')}_{\li{\infty}{2}}\\
&\les & (\norm{\chi'}_{\tx{\infty}{4}}+\norm{k}_{\tx{\infty}{4}})\norm{\nabb'(P_{\leq l}'\trc')}_{\tx{2}{4}}.
\eee
Together with the estimate \eqref{sobineq1} and the Gagliardo-Nirenberg inequality \eqref{eq:GNirenberg}, we obtain:
\bee
&&\norm{[nL',\nabb']P_{\leq l}'(\trc')}_{\li{\infty}{2}}+\norm{[b'N',\nabb'](P_{\leq l}'\trc')}_{\li{\infty}{2}}\\
&\les & (\no(\chi')+\no(k))\norm{{\nabb'}^2(P_{\leq l}'\trc')}_{\li{\infty}{2}}\\
&\les & 2^l\ep,
\eee
where we used in the last inequality the Bochner inequality for scalars \eqref{eq:Bochconseqbis}, the finite band property for $P_{\leq l}'$, the estimates \eqref{esttrc} and \eqref{esthch} for $\chi'$ and the estimates \eqref{estk} for $k$. Together with \eqref{messibis7bis}, this yields:
\be\lab{messibis7}
\norm{\lb'\nabb'(P_{\leq l}'\trc')}_{\li{\infty}{2}}\les  2^l\ep.
\ee

Next, we estimate the integral over $\H_u$ in the right-hand side of \eqref{messibis6}. We have:
\bea\lab{messibis8bis}
&&\int_{\H_u}|\nabb'P_{\leq l}'\trc'||L'\nabb'(P_{\leq l}'\trc')|d\H_u\\
\nn&\les & \int_{\H_u}|\nabb'P_{\leq l}'\trc'||nL'\nabb'(P_{\leq l}'\trc')|d\H_u\\
\nn&\les & \norm{\nabb'P_{\leq l}'\trc'}_{L^2_{{x'}'}L^\infty_t}\norm{\nabb'P_{\leq l}'(nL'(\trc'))}_{L^2_{{x'}'}L^1_t}+\norm{\nabb'P_{\leq l}'\trc'}_{\lprime{\infty}{2}}\\
\nn&&\times\norm{[nL',\nabb']P_{\leq l}'(\trc')}_{\lprime{\infty}{2}}+\norm{\nabb'P_{\leq l}'\trc'}_{L^2_{{x'}'}L^\infty_t}\norm{\nabb'[nL',P_{\leq l}'](\trc')}_{L^2_{{x'}'}L^1_t}\\
\nn&\les & \ep\norm{[nL',\nabb']P_{\leq l}'(\trc')}_{\lprime{\infty}{2}}+\ep^2,
\eea
where we used in the last inequality the finite band property for $P_{\leq l}'$, the commutator estimate \eqref{commlp3}, the estimate \eqref{esttrc} for $\trc$, and the estimate \eqref{lievremont2} for $P_{\leq l}\trc$ and $P_{\leq l}(nL\trc)$. 

Next, we estimate the right-hand side of \eqref{messibis8bis}:
\bee
\norm{[nL',\nabb']P_{\leq l}'(\trc')}_{\li{\infty}{2}}&\les &\norm{\chi'\nabb'(P_{\leq l}'\trc')}_{\li{\infty}{2}}\\
&\les &\norm{\chi'}_{\xt{\infty}{2}}\norm{\nabb'(P_{\leq l}'\trc')}_{\xt{2}{\infty}}\\
&\les&\ep,
\eee
where we used in the last inequality the estimates \eqref{esttrc} \eqref{esthch} for $\chi'$ and the estimate \eqref{lievremont2} for $\nabb'(P_{\leq l}'\trc')$. Together with \eqref{messibis8bis}, we obtain:
\be\lab{messibis10}
\int_{\H_u}|\nabb'P_{\leq l}'\trc'||L'\nabb'(P_{\leq l}'\trc')|d\H_u\les  \ep^2.
\ee

Finally, \eqref{messibis6}, \eqref{messibis7} and \eqref{messibis10} imply:
\bee
\norm{P_{\leq l}'\trc'F_j(u')}_{L^2_{u,x'}L^\infty_t}^2 &\les& \left(2^j|\nu-\nu'|^2+2^l|\nu-\nu'|+1\right)\ep^2\norm{F_j(u')}_{L^2_u}^2,
\eee
which together with \eqref{messibis4} implies:
\bee
&&\nn\normm{\int_{\S} G P_{\leq l}'\trc'F_j(u')\eta^{\nu'}_j(\o') d\o'}_{L^2_{u,x'}L^\infty_t}\\
&\les& \left(\sup_{\o'}\norm{G}_{L^\infty}\right)\left(2^{\frac{j}{2}}|\nu-\nu'|+2^{\frac{l}{2}}|\nu-\nu'|^{\frac{1}{2}}+1\right)\ep 2^{-\frac{j}{2}}\normm{\sqrt{\eta^{\nu'}_j(\o')}\norm{F_j(u')}_{L^2_u}}_{L^2_{\o'}}\\
&\les& \left(\sup_{\o'}\norm{G}_{L^\infty}\right)\left(2^{\frac{j}{2}}|\nu-\nu'|2^{\frac{j}{2}}+2^{\frac{l}{2}+\frac{j}{4}}(2^{\frac{j}{2}}|\nu-\nu'|)^{\frac{1}{2}}\right)\ep \gamma^{\nu'}_j,
\eee
where we used in the last inequality Plancherel in $u'$ and the fact that 
$$2^{\frac{j}{2}}|\nu-\nu'|\gtrsim 1$$
since $\nu\neq \nu'$. This conclude the proof of the corollary.
\end{proof}

Lemma \ref{lemma:messi} also yields a third corollary.
\begin{corollary}\lab{cor:koko}
Let $\o, \nu, \o', \nu'$ four angles in $\S$ such that $\o$ belongs to the patch of center $\nu$ and $\o'$ belongs to the patch of center $\nu'$. Let $u=u(t,x,\o)$. Let $u'=u(t,x,\o')$, and $L', \lb', \nabb'$ corresponding to $u'$, and let $L_{\nu'}, \lb_{\nu'}, \nabb_{\nu'}$ corresponding to $u(t,x,\nu')$. Let $q\in\mathbb{N}$. Then, we have the following estimate:
\bea\lab{koko1}
&&\normm{\int_{\S}{b'}^{-1}\trc'\left(2^{\frac{j}{2}}(N'-N_{\nu'})\right)^q F_j(u')\eta^{\nu'}_j(\o') d\o'}_{L^2_{u,x'}L^\infty_t}\\
\nn&\les& (1+q^{\frac{5}{2}})\ep\big(2^{\frac{j}{2}}|\nu-\nu'|+1\big)\gamma^{\nu'}_j.
\eea
\end{corollary}

\begin{proof}
We apply \eqref{messi} with the choice 
\bea\lab{koko2}
H=\int_{\S}{b'}^{-1}\trc'\left(2^{\frac{j}{2}}(N'-N_{\nu'})\right)^qF_j(u')\eta^{\nu'}_j(\o') d\o'.
\eea
We have:
\bea
\nn\norm{H}_{L^2_{u,x'}L^\infty_t}^2&\les& \norm{H}_{L^2(\MM)}\norm{\lb_{\nu'}(H)}_{L^2(\MM)}|\nu-\nu'|^2+\norm{H}_{L^2(\MM)}\norm{\nabb_{\nu'}(H)}_{L^2(\MM)}|\nu-\nu'|\\
\nn&&+\int_{\MM}|H||L_{\nu'}(H)|d\MM+\norm{H}_{L^2(\MM)}^2\\
\nn&\les& \norm{H}_{L^2(\MM)}\norm{\lb_{\nu'}(H)}_{L^2(\MM)}|\nu-\nu'|^2+\norm{H}_{L^2(\MM)}\norm{\nabb_{\nu'}(H)}_{L^2(\MM)}|\nu-\nu'|\\
\lab{koko3}&&+\norm{H}_{L^2_{u_{\nu'},{x'}_{\nu'}}L^\infty_t}\norm{L_{\nu'}(H)}_{L^2_{u_{\nu'},{x'}_{\nu'}}L^1_t}+\norm{H}_{L^2(\MM)}^2,
\eea
where we used the fact that $|\o-\nu'|\sim |\nu-\nu'|$. Now, the estimate of the $L^2_{u,x'}L^\infty_t$ norm of oscillatory integrals \eqref{loeb} yields:
\be\lab{koko4}
\norm{H}_{L^2_{u_{\nu'},{x'}_{\nu'}}L^\infty_t}\les (1+q^2)\ep \ga^{\nu'}_j.
\ee
Furthermore in order to prove the estimate \eqref{loeb}, we actually obtain the following estimate (see \eqref{loeb1}): 
\be\lab{omaoue}
\norm{L_{\nu'}(H)}_{L^2_{u_{\nu'},{x'}_{\nu'}}L^1_t}\les (1+q^2)\ep\gamma^{\nu'}_j.
\ee
Together with \eqref{koko3} and \eqref{koko4}, we obtain:
\bea
\nn\norm{H}_{L^2_{u,x'}L^\infty_t}^2 &\les& (1+q^2)\ep \ga^{\nu'}_j\Big(\norm{\lb_{\nu'}(H)}_{L^2(\MM)}|\nu-\nu'|^2+\norm{\nabb_{\nu'}(H)}_{L^2(\MM)}|\nu-\nu'|\Big)\\
\lab{koko5}&&+(1+q^2)^2\ep^2 (\ga^{\nu'}_j)^2.
\eea

Next, we estimates the various term in the right-hand side of \eqref{koko5} starting with the one involving $\lb_{\nu'}(H)$. In view of \eqref{koko2}, we have:
\bee
\nn\lb_{\nu'}(H)&=& i2^j\int_{\S}\gg(\lb_{\nu'},L'){b'}^{-1}\trc' \left(2^{\frac{j}{2}}(N'-N_{\nu'})\right)^qF_{j,1}(u')\eta^{\nu'}_j(\o') d\o'\\
&&+\int_{\S}\lb_{\nu'}({b'}^{-1}\trc')\left(2^{\frac{j}{2}}(N'-N_{\nu'})\right)^qF_j(u')\eta^{\nu'}_j(\o') d\o'\\
&&+q2^{\frac{j}{2}}\int_{\S}{b'}^{-1}\trc'(\dd_{\lb_{\nu'}}L'-\dd_{\lb_{\nu'}}L_{\nu'})\left(2^{\frac{j}{2}}(N'-N_{\nu'})\right)^{q-1}F_j(u')\eta^{\nu'}_j(\o') d\o'\\
&=& i2^j\int_{\S}\gg(\lb_{\nu'},L'){b'}^{-1}\trc' \left(2^{\frac{j}{2}}(N'-N_{\nu'})\right)^qF_{j,1}(u')\eta^{\nu'}_j(\o') d\o'\\
&&+\int_{\S}\lb_{\nu'}({b'}^{-1}\trc')\left(2^{\frac{j}{2}}(N'-N_{\nu'})\right)^qF_j(u')\eta^{\nu'}_j(\o') d\o'\\
&&+q2^{\frac{j}{2}}\int_{\S}{b'}^{-1}\trc'\dd_{\lb_{\nu'}}L'\left(2^{\frac{j}{2}}(N'-N_{\nu'})\right)^{q-1}F_j(u')\eta^{\nu'}_j(\o') d\o'\\
&&-q2^{\frac{j}{2}}\dd_{\lb_{\nu'}}L_{\nu'}\int_{\S}{b'}^{-1}\trc'\left(2^{\frac{j}{2}}(N'-N_{\nu'})\right)^{q-1}F_j(u')\eta^{\nu'}_j(\o') d\o'
\eee
where we used the fact that $\lb_{\nu'}(u)=\gg(\lb_{\nu'},L')$ and $N'-N_{\nu'}=L'-L_{\nu'}$. This yields:
\bee
&&\nn\norm{\lb_{\nu'}(H)}_{L^2(\MM)}\\
& \les & 2^j\normm{\int_{\S}\gg(\lb_{\nu'},L'){b'}^{-1}\trc' \left(2^{\frac{j}{2}}(N'-N_{\nu'})\right)^qF_{j,1}(u')\eta^{\nu'}_j(\o') d\o'}_{L^2(\MM)}\\
&&+\normm{\int_{\S}\lb_{\nu'}({b'}^{-1}\trc')\left(2^{\frac{j}{2}}(N'-N_{\nu'})\right)^qF_j(u')\eta^{\nu'}_j(\o') d\o'}_{L^2(\MM)}\\
&&+q2^{\frac{j}{2}}\normm{\int_{\S}{b'}^{-1}\trc'\dd_{\lb_{\nu'}}L'\left(2^{\frac{j}{2}}(N'-N_{\nu'})\right)^{q-1}F_j(u')\eta^{\nu'}_j(\o') d\o'}_{L^2(\MM)}\\
&&+q2^{\frac{j}{2}}\norm{\dd_{\lb_{\nu'}}L_{\nu'}}_{L^\infty_{u_{\nu'},{x'}_{\nu'}}L^2_t}\normm{\int_{\S}{b'}^{-1}\trc'\left(2^{\frac{j}{2}}(N'-N_{\nu'})\right)^{q-1}F_j(u')\eta^{\nu'}_j(\o') d\o'}_{L^2_{u_{\nu'},{x'}_{\nu'}}L^\infty_t}
\eee
Using the estimate of the $L^2_{u,x'}L^\infty_t$ norm of oscillatory integrals \eqref{loeb} for the first and the last term in the right-hand side, and the basic estimate in $L^2(\MM)$ \eqref{oscl2bis} for the second and the third term in the right-hand side, we obtain:
\bea
\lab{koko6}&&\norm{\lb_{\nu'}(H)}_{L^2(\MM)}\\
\nn& \les & 2^j(1+q^2)\ep\ga^{\nu'}_j+\sup_\o\left(\norm{\dd(b^{-1}\trc)}_{\li{\infty}{2}}\normm{\left(2^{\frac{j}{2}}(N'-N_{\nu'})\right)^q}_{L^\infty}\right)2^{\frac{j}{2}}\gamma^{\nu'}_j\\
\nn&&+q2^{\frac{j}{2}}\sup_\o\left(\norm{\dd L'}_{\li{\infty}{2}}\normm{\left(2^{\frac{j}{2}}(N'-N_{\nu'})\right)^{q-1}}_{L^\infty}\right)2^{\frac{j}{2}}\gamma^{\nu'}_j\\
\nn&&+q2^{\frac{j}{2}}\norm{\dd L_{\nu'}}_{L^\infty_{u_{\nu'},{x'}_{\nu'}}L^2_t}(1+q^2)\ep\ga^{\nu'}_j\\
\nn& \les & 2^j(1+q^3)\ep\ga^{\nu'}_j,
\eea
where we used in the last inequality the Ricci equations \eqref{ricciform} for $\dd L_{\nu'}$ and $\dd L'$, the estimate \eqref{estb} for $b$, the estimate \eqref{estn} for $n$, the estimate \eqref{estk} for $k$, the estimates \eqref{esttrc} \eqref{esthch} for $\chi$, the estimate \eqref{estzeta} for $\z$, and the estimate \eqref{estNomega} for $\po N$.

Next, we estimate the term in the right-hand side of \eqref{koko5} involving $\nabb_{\nu'}(H)$. In view of \eqref{koko2}, we have for $A=1, 2$:
\bee
\nn\nabb_{(e_{\nu'})_A}(H)&=& i2^j\int_{\S}\gg((e_{\nu'})_A,L'){b'}^{-1}\trc' \left(2^{\frac{j}{2}}(N'-N_{\nu'})\right)^qF_{j,1}(u')\eta^{\nu'}_j(\o') d\o'\\
&&+\int_{\S}(e_{\nu'})_A({b'}^{-1}\trc')\left(2^{\frac{j}{2}}(N'-N_{\nu'})\right)^qF_j(u')\eta^{\nu'}_j(\o') d\o'\\
&&+q2^{\frac{j}{2}}\int_{\S}{b'}^{-1}\trc'(\dd_{(e_{\nu'})_A}L'-\dd_{(e_{\nu'})_A}L_{\nu'})\left(2^{\frac{j}{2}}(N'-N_{\nu'})\right)^{q-1}F_j(u')\eta^{\nu'}_j(\o') d\o'\\
&=& i2^{\frac{j}{2}}(e_{\nu'})_A\c\int_{\S}{b'}^{-1}\trc' \left(2^{\frac{j}{2}}(N'-N_{\nu'})\right)^{q+1}F_{j,1}(u')\eta^{\nu'}_j(\o') d\o'\\
&&+\int_{\S}(e_{\nu'})_A({b'}^{-1}\trc')\left(2^{\frac{j}{2}}(N'-N_{\nu'})\right)^qF_j(u')\eta^{\nu'}_j(\o') d\o'\\
&&+q2^{\frac{j}{2}}\int_{\S}{b'}^{-1}\trc'\dd_{(e_{\nu'})_A}L'\left(2^{\frac{j}{2}}(N'-N_{\nu'})\right)^{q-1}F_j(u')\eta^{\nu'}_j(\o') d\o'\\
&&-q2^{\frac{j}{2}}\dd_{(e_{\nu'})_A}L_{\nu'}\int_{\S}{b'}^{-1}\trc'\left(2^{\frac{j}{2}}(N'-N_{\nu'})\right)^{q-1}F_j(u')\eta^{\nu'}_j(\o') d\o'
\eee
where we used the fact that: 
$$(e_{\nu'})_A(u)=\gg((e_{\nu'})_A,L')=\gg((e_{\nu'})_A,L'-L_{\nu'})=\gg((e_{\nu'})_A,N'-N_{\nu'})$$ 
and $N'-N_{\nu'}=L'-L_{\nu'}$. This yields:
\bee
&&\nn\norm{\nabb_{(e_{\nu'})_A}(H)}_{L^2(\MM)}\\
& \les & 2^{\frac{j}{2}}\normm{\int_{\S}{b'}^{-1}\trc' \left(2^{\frac{j}{2}}(N'-N_{\nu'})\right)^{q+1}F_{j,1}(u')\eta^{\nu'}_j(\o') d\o'}_{L^2(\MM)}\\
&&+\normm{\int_{\S}(e_{\nu'})_A({b'}^{-1}\trc')\left(2^{\frac{j}{2}}(N'-N_{\nu'})\right)^qF_j(u')\eta^{\nu'}_j(\o') d\o'}_{L^2(\MM)}\\
&&+q2^{\frac{j}{2}}\normm{\int_{\S}{b'}^{-1}\trc'\dd_{(e_{\nu'})_A}L'\left(2^{\frac{j}{2}}(N'-N_{\nu'})\right)^{q-1}F_j(u')\eta^{\nu'}_j(\o') d\o'}_{L^2(\MM)}\\
&&+q2^{\frac{j}{2}}\norm{\dd_{(e_{\nu'})_A}L_{\nu'}}_{L^\infty_{u_{\nu'},{x'}_{\nu'}}L^2_t}\normm{\int_{\S}{b'}^{-1}\trc'\left(2^{\frac{j}{2}}(N'-N_{\nu'})\right)^{q-1}F_j(u')\eta^{\nu'}_j(\o') d\o'}_{L^2_{u_{\nu'},{x'}_{\nu'}}L^\infty_t}
\eee
Using the estimate of the $L^2_{u,x'}L^\infty_t$ norm of oscillatory integrals \eqref{loeb} for the first and the last term in the right-hand side, and the basic estimate in $L^2(\MM)$ \eqref{oscl2bis} for the second term in the right-hand side, we obtain:
\bea
\lab{koko7}&&\norm{\nabb_{(e_{\nu'})_A}(H)}_{L^2(\MM)}\\
\nn& \les & 2^{\frac{j}{2}}(1+q^2)\ep\ga^{\nu'}_j+\sup_\o\left(\norm{\dd(b^{-1}\trc)}_{\li{\infty}{2}}\normm{\left(2^{\frac{j}{2}}(N'-N_{\nu'})\right)^q}_{L^\infty}\right)2^{\frac{j}{2}}\gamma^{\nu'}_j\\
\nn&&+q2^{\frac{j}{2}}\normm{\int_{\S}{b'}^{-1}\trc'\dd_{(e_{\nu'})_A}L'\left(2^{\frac{j}{2}}(N'-N_{\nu'})\right)^{q-1}F_j(u')\eta^{\nu'}_j(\o') d\o'}_{L^2(\MM)}\\
\nn&& +q2^{\frac{j}{2}}\norm{\dd_{(e_{\nu'})_A}L_{\nu'}}_{L^\infty_{u_{\nu'},{x'}_{\nu'}}L^2_t}(1+q^2)\ep\ga^{\nu'}_j\\
\nn& \les & 2^{\frac{j}{2}}(1+q^3)\ep\ga^{\nu'}_j\\
\nn&&+q2^{\frac{j}{2}}\ep\normm{\int_{\S}{b'}^{-1}\trc'\dd_{(e_{\nu'})_A}L'\left(2^{\frac{j}{2}}(N'-N_{\nu'})\right)^{q-1}F_j(u')\eta^{\nu'}_j(\o') d\o'}_{L^2(\MM)},
\eea
where we used in the last inequality the Ricci equations \eqref{ricciform} for $\dd_{(e_{\nu'})_A}L_{\nu'}$, the estimate \eqref{estb} for $b$, the estimate \eqref{estk} for $k$, the estimates \eqref{esttrc} \eqref{esthch} for $\chi$, and the estimate \eqref{estNomega} for $\po N$. Next, we decompose the term $\dd_{(e_{\nu'})_A}L'$ in the right-hand side of \eqref{koko7}. First, we decompose $(e_{\nu'})_A$ on the frame $L', \lb', e_B'$, $B=1, 2$. We have:
$$(e_{\nu'})_A=-\frac{1}{2}\gg((e_{\nu'})_A, \lb')L'-\frac{1}{2}\gg((e_{\nu'})_A, L')\lb'+((e_{\nu'})_A-\gg((e_{\nu'})_A,N')N').$$
Together with the Ricci equations \eqref{ricciform}, this yields, schematically:
\bee
\dd_{(e_{\nu'})_A}L'=(\chi+\kep)(e_{\nu'})_A+(N'-N_{\nu'})(\chi+\kep+\z+\d+n^{-1}\nabla_Nn).
\eee
In view of the decompositions \eqref{dechchom} \eqref{dectrcom} for $\chi$, the fact that $\kep_A=k_{NA}$ with $k$ independent of $\o$, the estimate \eqref{estn} for $n$, the estimate \eqref{estk} for $\db$, $\kep$ and $k$, the estimate \eqref{esttrc} \eqref{esthch} for $\chi$, and the estimate \eqref{estzeta} for $\z$, we obtain the following decomposition:
\be\lab{koko8}
\dd_{(e_{\nu'})_A}L'=F^j_1+F^j_2,
\ee
where the tensor $F^1_j$ only depends on $\nu'$ and satisfies:
\be\lab{koko9}
\norm{F^j_1}_{L^\infty_{u_{\nu'},x_{\nu'}}L^2_t}\les \ep
\ee
and the tensor $F^j_2$ satisfies:
\be\lab{koko10}
\norm{F^j_2}_{\li{\infty}{2}}\les 2^{-\frac{j}{2}}\ep.
\ee
In view of \eqref{koko8}, we obtain:
\bee
&&\int_{\S}{b'}^{-1}\trc'\dd_{(e_{\nu'})_A}L'\left(2^{\frac{j}{2}}(N'-N_{\nu'})\right)^{q-1}F_j(u')\eta^{\nu'}_j(\o') d\o'\\
 &=& F^j_1\int_{\S}{b'}^{-1}\trc'\left(2^{\frac{j}{2}}(N'-N_{\nu'})\right)^{q-1}F_j(u')\eta^{\nu'}_j(\o') d\o'\\
&& +\int_{\S}{b'}^{-1}\trc'F^j_2\left(2^{\frac{j}{2}}(N'-N_{\nu'})\right)^{q-1}F_j(u')\eta^{\nu'}_j(\o') d\o'
\eee
which yields:
\bee
&&\normm{\int_{\S}{b'}^{-1}\trc'\dd_{(e_{\nu'})_A}L'\left(2^{\frac{j}{2}}(N'-N_{\nu'})\right)^{q-1}F_j(u')\eta^{\nu'}_j(\o') d\o'}_{L^2(\MM)}\\
 &\les& \norm{F^j_1}_{L^\infty_{u_{\nu'},x_{\nu'}}L^2_t}\normm{\int_{\S}{b'}^{-1}\trc'\left(2^{\frac{j}{2}}(N'-N_{\nu'})\right)^{q-1}F_j(u')\eta^{\nu'}_j(\o') d\o'}_{L^2_{u_{\nu'},x_{\nu'}}L^\infty_t}\\
&& +\normm{\int_{\S}{b'}^{-1}\trc'F^j_2\left(2^{\frac{j}{2}}(N'-N_{\nu'})\right)^{q-1}F_j(u')\eta^{\nu'}_j(\o') d\o'}_{L^2(\MM)}\\
 &\les& \ep\normm{\int_{\S}{b'}^{-1}\trc'\left(2^{\frac{j}{2}}(N'-N_{\nu'})\right)^{q-1}F_j(u')\eta^{\nu'}_j(\o') d\o'}_{L^2_{u_{\nu'},x_{\nu'}}L^\infty_t}\\
&& +\normm{\int_{\S}{b'}^{-1}\trc'F^j_2\left(2^{\frac{j}{2}}(N'-N_{\nu'})\right)^{q-1}F_j(u')\eta^{\nu'}_j(\o') d\o'}_{L^2(\MM)}
\eee 
where we used in the last inequality the estimate \eqref{koko9}. Using the estimate of the $L^2_{u,x'}L^\infty_t$ norm of oscillatory integrals \eqref{loeb} for the first term in the right-hand side, and the basic estimate in $L^2(\MM)$ \eqref{oscl2bis} for the second term in the right-hand side, we obtain:
\bee
&&\norm{\int_{\S}{b'}^{-1}\trc'\dd_{(e_{\nu'})_A}L'\left(2^{\frac{j}{2}}(N'-N_{\nu'})\right)^{q-1}F_j(u')\eta^{\nu'}_j(\o') d\o'}_{L^2(\MM)}\\
 &\les& \ep(1+q^2)\gamma^{\nu'}_j+\sup_{\o'}\left(\norm{F^j_2}_{\li{\infty}{2}}\normm{\left(2^{\frac{j}{2}}(N'-N_{\nu'})\right)^{q-1}}_{L^\infty}\right)2^{\frac{j}{2}}\gamma^{\nu'}_j\\
&\les& \ep(1+q^2)\gamma^{\nu'}_j,
\eee 
where we used in the last inequality the estimate \eqref{koko10}, the estimate \eqref{estNomega} for $\po N$, and the size of the patch. Together with \eqref{koko7}, we finally obtain:
\be\lab{omaoue1}
\norm{\nabb_{(e_{\nu'})_A}(H)}_{L^2(\MM)}\les  2^{\frac{j}{2}}(1+q^3)\ep\ga^{\nu'}_j.
\ee
Together with \eqref{koko5} and \eqref{koko6}, this yields:
$$\norm{H}_{L^2_{u,x'}L^\infty_t}^2 \les (1+q^5)\ep^2(\ga^{\nu'}_j)^2\Big(2^j|\nu-\nu'|^2+2^{\frac{j}{2}}|\nu-\nu'|+1\Big).$$
This concludes the proof of the lemma.
\end{proof}

We have finally a last corollary of Lemma \ref{lemma:messi}.
\begin{corollary}\lab{cor:bis:koko}
Let $\o, \nu, \o', \nu'$ four angles in $\S$ such that $\o$ belongs to the patch of center $\nu$ and $\o'$ belongs to the patch of center $\nu'$. Let $u=u(t,x,\o)$. Let $u'=u(t,x,\o')$, and $L', \lb', \nabb'$ corresponding to $u'$, and let $L_{\nu'}, \lb_{\nu'}, \nabb_{\nu'}$ corresponding to $u(t,x,\nu')$. Let $q\in\mathbb{N}$. Then, we have the following estimate:
\bea\lab{bis:koko1}
&&\normm{\int_{\S}(b'-b_{\nu'})\trc'\left(2^{\frac{j}{2}}(N'-N_{\nu'})\right)^q F_j(u')\eta^{\nu'}_j(\o') d\o'}_{L^2_{u,x'}L^\infty_t}\\
\nn&\les& 2^{-\frac{j}{4}}(1+q^{\frac{5}{2}})\ep\big(2^{\frac{j}{2}}|\nu-\nu'|+1\big)\gamma^{\nu'}_j.
\eea
\end{corollary}

\begin{proof}
We apply \eqref{messi} with the choice 
\bea\lab{bis:koko2}
H=\int_{\S}(b'-b_{\nu'})\trc'\left(2^{\frac{j}{2}}(N'-N_{\nu'})\right)^qF_j(u')\eta^{\nu'}_j(\o') d\o'.
\eea
As in \eqref{koko3}, we have:
\bea
\nn\norm{H}_{L^2_{u,x'}L^\infty_t}^2&\les& \norm{H}_{L^2(\MM)}\norm{\lb_{\nu'}(H)}_{L^2(\MM)}|\nu-\nu'|^2+\norm{H}_{L^2(\MM)}\norm{\nabb_{\nu'}(H)}_{L^2(\MM)}|\nu-\nu'|\\
\lab{bis:koko3}&&+\norm{H}_{L^2_{u_{\nu'},{x'}_{\nu'}}L^\infty_t}\norm{L_{\nu'}(H)}_{L^2_{u_{\nu'},{x'}_{\nu'}}L^1_t}+\norm{H}_{L^2(\MM)}^2.
\eea
We first estimate $\norm{H}_{L^2(\MM)}$. Recall the decomposition \eqref{decbpom} for $b-b_\nu$. We have:
\be\lab{facilfamille}
b-b_\nu=2^{-\frac{j}{2}}(f^j_1+f^j_2)
\ee
where the tensor $f^j_1$ only depends on $\nu$ and satisfies:
\be\lab{facilfamille1}
\norm{f^j_1}_{L^\infty}\les \ep,
\ee
and where the tensor $f^j_2$ satisfies:
\be\lab{facilfamille2}
\norm{f^j_2}_{L^\infty_u\lh{2}}\les \ep 2^{-\frac{j}{4}}.
\ee
This yields:
\bea\lab{facilfamille3}
&&\normm{\int_{\S}(b'-b_{\nu'})\trc'\left(2^{\frac{j}{2}}(N'-N_{\nu'})\right)^qF_j(u')\eta^{\nu'}_j(\o') d\o'}_{L^2(\MM)}\\
\nn&\les & 2^{-\frac{j}{2}}\norm{f^j_1}_{L^\infty(\MM)}\normm{\int_{\S}\trc'\left(2^{\frac{j}{2}}(N'-N_{\nu'})\right)^qF_j(u')\eta^{\nu'}_j(\o') d\o'}_{L^2(\MM)}\\
\nn&&+2^{-\frac{j}{2}}\normm{\int_{\S}f^j_2\trc'\left(2^{\frac{j}{2}}(N'-N_{\nu'})\right)^qF_j(u')\eta^{\nu'}_j(\o') d\o'}_{L^2(\MM)}\\
\nn&\les & 2^{-\frac{j}{2}}\ep\normm{\int_{\S}\trc'\left(2^{\frac{j}{2}}(N'-N_{\nu'})\right)^qF_j(u')\eta^{\nu'}_j(\o') d\o'}_{L^2(\MM)}\\
\nn&&+2^{-\frac{j}{2}}\normm{\int_{\S}f^j_2\trc'\left(2^{\frac{j}{2}}(N'-N_{\nu'})\right)^qF_j(u')\eta^{\nu'}_j(\o') d\o'}_{L^2(\MM)},
\eea
where we used in the last inequality the estimate \eqref{facilfamille1} for $f^j_1$. We control the first term in the right-hand side of \eqref{facilfamille3} using the estimate \eqref{loebter}, and the second term in the right-hand side of \eqref{facilfamille3} using the basic estimate in $L^2(\MM)$ \eqref{oscl2bis}. We obtain:
\bea
&&\nn\normm{\int_{\S}(b'-b_{\nu'})\trc'\left(2^{\frac{j}{2}}(N'-N_{\nu'})\right)^qF_j(u')\eta^{\nu'}_j(\o') d\o'}_{L^2(\MM)}\\
\nn&\les & 2^{-\frac{j}{2}}\ep(1+q^2)\gamma^{\nu'}_j+\left(\sup_{\o'}\normm{f^j_2\trc'\left(2^{\frac{j}{2}}(N'-N_{\nu'})\right)^q}_{\lprime{\infty}{2}}\right)\gamma^{\nu'}_j\\
\nn&\les & 2^{-\frac{j}{2}}\ep(1+q^2)\gamma^{\nu'}_j+\left(\sup_{\o'}\norm{f^j_2}_{\lprime{\infty}{2}}\norm{\trc'}_{L^\infty}\normm{\left(2^{\frac{j}{2}}(N'-N_{\nu'})\right)^q}_{L^\infty}\right)\gamma^{\nu'}_j\\
\lab{facilfamille4bis}&\les & 2^{-\frac{j}{4}}\ep(1+q^2)\gamma^{\nu'}_j,
\eea
where we used in the last inequality the estimate \eqref{facilfamille2} for $f^j_2$, the estimate \eqref{esttrc} for $\trc'$, the estimate \eqref{estNomega} for $\po N$ and the size of the patch. In view of the definition \eqref{bis:koko2} of $H$, this yields:
\be\lab{facilfamille4}
\norm{H}_{L^2(\MM)}\les  2^{-\frac{j}{4}}\ep(1+q^2)\gamma^{\nu'}_j.
\ee
Together with \eqref{bis:koko3}, we obtain:
\bea
\nn\norm{H}_{L^2_{u,x'}L^\infty_t}^2&\les& (\norm{\lb_{\nu'}(H)}_{L^2(\MM)}|\nu-\nu'|^2+\norm{\nabb_{\nu'}(H)}_{L^2(\MM)}|\nu-\nu'|)2^{-\frac{j}{2}}\ep(1+q^2)\gamma^{\nu'}_j\\
\lab{facilfamille4ter}&&+\norm{H}_{L^2_{u_{\nu'},{x'}_{\nu'}}L^\infty_t}\norm{L_{\nu'}(H)}_{L^2_{u_{\nu'},{x'}_{\nu'}}L^1_t}+2^{-\frac{j}{2}}\ep^2(1+q^4)(\gamma^{\nu'}_j)^2.
\eea

Next, we define:
\be\lab{cars}
H_1=\int_{\S}(\lb_{\nu'}(b')-\lb_{\nu'}(b_{\nu'}))\trc'\left(2^{\frac{j}{2}}(N'-N_{\nu'})\right)^qF_j(u')\eta^{\nu'}_j(\o') d\o',
\ee
\be\lab{cars1}
H_2=\int_{\S}(\nabb_{\nu'}(b')-\nabb_{\nu'}(b_{\nu'}))\trc'\left(2^{\frac{j}{2}}(N'-N_{\nu'})\right)^qF_j(u')\eta^{\nu'}_j(\o') d\o',
\ee
\be\lab{cars2}
H_3=\int_{\S}(L_{\nu'}(b')-L_{\nu'}(b_{\nu'}))\trc'\left(2^{\frac{j}{2}}(N'-N_{\nu'})\right)^qF_j(u')\eta^{\nu'}_j(\o') d\o',
\ee
and:
\be\lab{cars3} 
H_1'=\lb_{\nu'}(H)-H_1,\, H_2'=\nabb_{\nu'}(H)-H_2\textrm{ and }H_3'=L_{\nu'}(H)-H_3.
\ee
The terms $H_1, H_2, H_3$ denote the contributions in the right-hand side of \eqref{bis:koko5} where the derivatives  $\lb_{\nu'}, \nabb_{\nu'}, L_{\nu'}$ fall on $b$. The terms $H_1', H_2', H_3'$ are the ones already treated in the proof of Corollary \ref{cor:koko} up to the presence of the extra term $b'-b_{\nu'}$ which is evaluated in $L^\infty$ norm. In view of the estimate \eqref{estricciomega} for $\po b$ and the size of the patch, we have:
\be\lab{cars4}
\norm{b'-b_{\nu'}}_{L^\infty}\les 2^{-\frac{j}{2}}.
\ee
Thus, in view of the estimates \eqref{omaoue}, \eqref{koko6} and \eqref{omaoue1} of the proof of Corollary \ref{cor:koko}, and taking into account the extra $2^{-\frac{j}{2}}$ factor coming from \eqref{cars4}, we obtain the analog of \eqref{omaoue} \eqref{koko6} \eqref{omaoue1}:
$$\norm{H_3'}_{L^2_{u_{\nu'},{x'}_{\nu'}}L^1_t}\les 2^{-\frac{j}{2}}(1+q^2)\ep\gamma^{\nu'}_j,$$
$$\norm{H_1'}_{L^2(\MM)} \les  2^{\frac{j}{2}}(1+q^3)\ep\ga^{\nu'}_j,$$
and:
$$\norm{H_2'}_{L^2(\MM)}\les  (1+q^3)\ep\ga^{\nu'}_j.$$
Together with \eqref{facilfamille4ter} and in view of the decompositions \eqref{cars3}, we get:
\bea\lab{bis:koko5}
\norm{H}_{L^2_{u,x'}L^\infty_t}^2&\les& (\norm{H_1}_{L^2(\MM)}|\nu-\nu'|^2+\norm{H_2}_{L^2(\MM)}|\nu-\nu'|)2^{-\frac{j}{2}}\ep(1+q^2)\gamma^{\nu'}_j\\
\nn&&+\norm{H}_{L^2_{u_{\nu'},{x'}_{\nu'}}L^\infty_t}\norm{H_3}_{L^2_{u_{\nu'},{x'}_{\nu'}}L^1_t}+(1+q^5)2^{-\frac{j}{2}}(1+(2^{\frac{j}{2}}|\nu-\nu'|)^2)\ep^2(\ga^{\nu'}_j)^2.
\eea

Next, we estimates the various term in the right-hand side of \eqref{bis:koko5} starting with $H_1$. In view of the definition \eqref{cars} for $H_1$, we have:
\bea\lab{cars5}
\norm{H_1}_{L^2(\MM)}&\les & \normm{\int_{\S}\lb_{\nu'}(b')\trc'\left(2^{\frac{j}{2}}(N'-N_{\nu'})\right)^qF_j(u')\eta^{\nu'}_j(\o') d\o'}_{L^2(\MM)}\\
\nn&&+\normm{\lb_{\nu'}(b_{\nu'})\left(\int_{\S}\trc'\left(2^{\frac{j}{2}}(N'-N_{\nu'})\right)^qF_j(u')\eta^{\nu'}_j(\o') d\o'\right)}_{L^2(\MM)}.
\eea
Next, we evaluate both terms in the right-hand side of \eqref{cars5} starting with the first one. Using the basic estimate in $L^2(\MM)$, we have:
\bea\lab{cars6}
&&\normm{\int_{\S}\lb_{\nu'}(b')\trc'\left(2^{\frac{j}{2}}(N'-N_{\nu'})\right)^qF_j(u')\eta^{\nu'}_j(\o') d\o'}_{L^2(\MM)}\\
\nn&\les &\left(\sup_{\o'}\normm{\lb_{\nu'}(b')\trc'\left(2^{\frac{j}{2}}(N'-N_{\nu'})\right)^q}_{\li{\infty}{2}}\right) 2^{\frac{j}{2}}\gamma^{\nu'}_j.
\eea
Now, we have:
\bee
\normm{\lb_{\nu'}(b')\trc'\left(2^{\frac{j}{2}}(N'-N_{\nu'})\right)^q}_{\li{\infty}{2}}&\les& \norm{\dd b'}_{\li{\infty}{2}}\norm{\trc'}_{L^\infty}\normm{\left(2^{\frac{j}{2}}(N'-N_{\nu'})\right)^q}_{L^\infty}\\
&\les& \ep
\eee
where we used in the last inequality the estimate \eqref{estb} for $b$, the estimate \eqref{esttrc} for $\trc$, the estimate \eqref{estNomega} for $\po N$ and the size of the patch. Together with \eqref{cars6}, we obtain:
\be\lab{cars7}
\normm{\int_{\S}\lb_{\nu'}(b')\trc'\left(2^{\frac{j}{2}}(N'-N_{\nu'})\right)^qF_j(u')\eta^{\nu'}_j(\o') d\o'}_{L^2(\MM)}\les \ep 2^{\frac{j}{2}}\gamma^{\nu'}_j.
\ee
Next, we estimate the second term in the right-hand side of \eqref{cars5}. We have:
\bea\lab{cars8}
&&\normm{\lb_{\nu'}(b_{\nu'})\left(\int_{\S}\trc'\left(2^{\frac{j}{2}}(N'-N_{\nu'})\right)^qF_j(u')\eta^{\nu'}_j(\o') d\o'\right)}_{L^2(\MM)}\\
\nn&\les& \normm{\lb_{\nu'}(b_{\nu'})}_{L^4(\MM)}\normm{\int_{\S}\trc'\left(2^{\frac{j}{2}}(N'-N_{\nu'})\right)^qF_j(u')\eta^{\nu'}_j(\o') d\o'}_{L^4(\MM)}\\
\nn&\les& \ep\normm{\int_{\S}\trc'\left(2^{\frac{j}{2}}(N'-N_{\nu'})\right)^qF_j(u')\eta^{\nu'}_j(\o') d\o'}_{L^4(\MM)},
\eea
where we used in the last inequality the estimate \eqref{estb} for $b$. Now, we have:
\bea
&&\nn\normm{\int_{\S}\trc'\left(2^{\frac{j}{2}}(N'-N_{\nu'})\right)^qF_j(u')\eta^{\nu'}_j(\o') d\o'}_{L^2(\MM)}\\
\nn&\les& \normm{\int_{\S}\trc'\left(2^{\frac{j}{2}}(N'-N_{\nu'})\right)^qF_{j,-1}(u')\eta_j^{\nu'}(\o')d\o'}_{L^2_{u_{\nu'},{x'}_{\nu'}}L^\infty_t}\\
\lab{nice50}&\les & \ep \gamma^{\nu'}_j,
\eea
where we used in the last inequality the estimate \eqref{loebter} of the $L^2_{u_\nu,{x'}_\nu}L^\infty_t$ of oscillatory integrals. Also, we have:
\bea
&&\nn\normm{\int_{\S}\trc'\left(2^{\frac{j}{2}}(N'-N_{\nu'})\right)^qF_j(u')\eta^{\nu'}_j(\o') d\o'}_{L^\infty(\MM)}\\
\nn&\les& \int_{\S}\normm{\trc'\left(2^{\frac{j}{2}}(N'-N_{\nu'})\right)^q}_{L^\infty(\MM)}\norm{F_{j,-1}(u')}_{L^\infty_{u'}}\eta_j^{\nu'}(\o')d\o'\\
\nn&\les & \ep \left(\int_{\S}\norm{F_{j,-1}(u')}_{L^\infty_{u'}}\eta_j^{\nu'}(\o')d\o'\right)\\
\lab{nice51}&\les& \ep 2^j\gamma^{\nu'}_j,
\eea
where we used, the estimate \eqref{esttrc} for $\trc$, the estimate \eqref{estNomega} for $\po N$, Cauchy-Schwarz in $\la'$ to estimate $\norm{F_{j,-1}(u')}_{L^\infty_{u'}}$, Cauchy-Schwarz in $\o'$ and the size of the patch. Interpolating \eqref{nice50} and \eqref{nice51}, we obtain:
\be\lab{cars9}
\normm{\int_{\S}\trc'\left(2^{\frac{j}{2}}(N'-N_{\nu'})\right)^qF_j(u')\eta^{\nu'}_j(\o') d\o'}_{L^4(\MM)}\les  2^{\frac{j}{2}}\ep \gamma^{\nu'}_j.
\ee
Finally, \eqref{cars5}, \eqref{cars7} and \eqref{cars9} imply:
\be\lab{cars10}
\norm{H_1}_{L^2(\MM)}\les  2^{\frac{j}{2}}\ep \gamma^{\nu'}_j.
\ee

Next, we estimates the term $H_2$ in the right-hand side of \eqref{bis:koko5}. In view of the definition \eqref{cars1} for $H_2$, we have:
\bea\lab{cars11}
\norm{H_2}_{L^2(\MM)}&\les & \normm{\int_{\S}\nabb_{\nu'}(b')\trc'\left(2^{\frac{j}{2}}(N'-N_{\nu'})\right)^qF_j(u')\eta^{\nu'}_j(\o') d\o'}_{L^2(\MM)}\\
\nn&&+\normm{\nabb_{\nu'}(b_{\nu'})\left(\int_{\S}\trc'\left(2^{\frac{j}{2}}(N'-N_{\nu'})\right)^qF_j(u')\eta^{\nu'}_j(\o') d\o'\right)}_{L^2(\MM)}.
\eea
Next, we evaluate both terms in the right-hand side of \eqref{cars11} starting with the last one. We have:
\bea\lab{cars12}
&&\normm{\nabb_{\nu'}(b_{\nu'})\left(\int_{\S}\trc'\left(2^{\frac{j}{2}}(N'-N_{\nu'})\right)^qF_j(u')\eta^{\nu'}_j(\o') d\o'\right)}_{L^2(\MM)}\\
\nn&\les& \norm{\nabb_{\nu'}(b_{\nu'})}_{L^\infty_{u_{\nu'},x_{\nu'}'}L^2_t}\normm{\int_{\S}\trc'\left(2^{\frac{j}{2}}(N'-N_{\nu'})\right)^qF_j(u')\eta^{\nu'}_j(\o') d\o'}_{L^2_{u_{\nu'},x_{\nu'}'}L^\infty_t}\\
\nn&\les& \ep(1+q^2)\gamma^{\nu'}_j,
\eea
where we used in the last inequality the estimate \eqref{estb} for $b$ and the estimate in $L^2_{u_{\nu'},x_{\nu'}'}L^\infty_t$ \eqref{loebter}. Next, we evaluate the first term in the right-hand side of \eqref{cars5}. Decomposing $\nabb_{\nu'}$ on the frame $L', \lb', e_A'$ and using the fact that:
$$|\o'-\nu'|\les 2^{\frac{j}{2}},$$
we have schematically:
$$\nabb_{\nu'}(b')=\nabb'(b')+2^{-\frac{j}{2}}\dd (b')$$ 
and thus:
\bea\lab{cars13}
&& \normm{\int_{\S}\nabb_{\nu'}(b')\trc'\left(2^{\frac{j}{2}}(N'-N_{\nu'})\right)^qF_j(u')\eta^{\nu'}_j(\o') d\o'}_{L^2(\MM)}\\
\nn&\les&\normm{\int_{\S}\nabb'(b')\trc'\left(2^{\frac{j}{2}}(N'-N_{\nu'})\right)^qF_j(u')\eta^{\nu'}_j(\o') d\o'}_{L^2(\MM)}\\
\nn&&+2^{-\frac{j}{2}}\normm{\int_{\S}\dd(b')\trc'\left(2^{\frac{j}{2}}(N'-N_{\nu'})\right)^qF_j(u')\eta^{\nu'}_j(\o') d\o'}_{L^2(\MM)}\\
\nn&\les& \normm{\int_{\S}\nabb'(b')\trc'\left(2^{\frac{j}{2}}(N'-N_{\nu'})\right)^qF_j(u')\eta^{\nu'}_j(\o') d\o'}_{L^2(\MM)}+\ep\gamma^{\nu'}_j,
\eea
where we used in the last inequality an estimate analog to \eqref{cars7}. In order to estimate the right-hand side of \eqref{cars13}, we use the decomposition \eqref{deczetaom} of $\nabb'(b')$. We have:
\be\lab{cars14}
\nabb(b)=F^j_1+F^j_2
\ee
where the tensor $F^j_1$ only depends on $\nu$ and satisfies:
\be\lab{cars15}
\norm{F^j_1}_{L^\infty_{u_{\nu'}}L^2_t,L^8_{x'_{\nu'}}}\les \ep,
\ee
and where the tensor $F^j_2$ satisfies:
\be\lab{cars16}
\norm{F^j_2}_{L^\infty_{u'}L^2(\H_{u'})}\les \ep 2^{-\frac{j}{4}}.
\ee
In view of \eqref{cars14}, we have:
\bea\lab{cars17}
&&\normm{\int_{\S}\nabb'(b')\trc'\left(2^{\frac{j}{2}}(N'-N_{\nu'})\right)^qF_j(u')\eta^{\nu'}_j(\o') d\o'}_{L^2(\MM)}\\
\nn&\les& \norm{F^j_1}_{L^2_{u_{\nu'}}L^2_t,L^8_{x'_{\nu'}}}\normm{\int_{\S}\trc'\left(2^{\frac{j}{2}}(N'-N_{\nu'})\right)^qF_j(u')\eta^{\nu'}_j(\o') d\o'}_{L^2_{u_{\nu'}}L^\infty_t,L^{\frac{8}{3}}_{x'_{\nu'}}}\\
\nn&& +\normm{\int_{\S}F^j_2\trc'\left(2^{\frac{j}{2}}(N'-N_{\nu'})\right)^qF_j(u')\eta^{\nu'}_j(\o') d\o'}_{L^2(\MM)}\\
\nn&\les& \ep\normm{\int_{\S}\trc'\left(2^{\frac{j}{2}}(N'-N_{\nu'})\right)^qF_j(u')\eta^{\nu'}_j(\o') d\o'}_{L^2_{u_{\nu'}}L^\infty_t,L^{\frac{8}{3}}_{x'_{\nu'}}}\\
\nn&& +\normm{\int_{\S}F^j_2\trc'\left(2^{\frac{j}{2}}(N'-N_{\nu'})\right)^qF_j(u')\eta^{\nu'}_j(\o') d\o'}_{L^2(\MM)},
\eea
where we used the estimate \eqref{cars15} for $F^j_1$ in the last inequality. Now, interpolating between the the estimate in $L^2_{u_{\nu'},x_{\nu'}'}L^\infty_t$ \eqref{loebter} and the $L^\infty$ estimate \eqref{nice51}, we obtain:
\be\lab{cars18}
\normm{\int_{\S}\trc'\left(2^{\frac{j}{2}}(N'-N_{\nu'})\right)^qF_j(u')\eta^{\nu'}_j(\o') d\o'}_{L^{\frac{8}{3}}_{u_{\nu'}, x_{\nu'}'}L^\infty_t}\les  2^{\frac{j}{4}}\ep \gamma^{\nu'}_j.
\ee
For the second term in the right-hand side of \eqref{cars17}, we have:
\bee
&& \normm{\int_{\S}F^j_2\trc'\left(2^{\frac{j}{2}}(N'-N_{\nu'})\right)^qF_j(u')\eta^{\nu'}_j(\o') d\o'}_{L^2(\MM)}\\
\nn&\les& \int_{\S}\norm{F^j_2\trc'\left(2^{\frac{j}{2}}(N'-N_{\nu'})\right)^qF_j(u')}_{L^2(\MM)}\eta^{\nu'}_j(\o') d\o'\\
\nn&\les& \int_{\S}\norm{F^j_2}_{L^\infty_{u'}L^2(\H_{u'}}\norm{F_j(u')}_{L^2_{u'}}\normm{\trc'\left(2^{\frac{j}{2}}(N'-N_{\nu'})\right)^q}_{L^\infty(\MM)}\eta^{\nu'}_j(\o') d\o'.
\eee
Together with the estimate \eqref{cars16} for $F^j_2$, the estimate \eqref{esttrc} for $\trc$, the estimate \eqref{estNomega} for $\po N$ and the size of the patch, we obtain:
\bea\lab{cars19}
&& \normm{\int_{\S}F^j_2\trc'\left(2^{\frac{j}{2}}(N'-N_{\nu'})\right)^qF_j(u')\eta^{\nu'}_j(\o') d\o'}_{L^2(\MM)}\\
\nn&\les& \ep\int_{\S}2^{-\frac{j}{4}}\ep\norm{F_j(u')}_{L^2_{u'}}\eta^{\nu'}_j(\o') d\o'\\
\nn&\les& \ep 2^{\frac{j}{4}}\gamma^{\nu'}_j,
\eea
where we used in the last inequality Plancherel in $\la'$ for $\norm{F_j(u')}_{L^2_{u'}}$, Cauchy Schwartz in $\o'$ and the size of the patch. Finally, \eqref{cars13}, \eqref{cars17}, \eqref{cars18} and \eqref{cars19} imply:
\be\lab{cars20}
 \normm{\int_{\S}\nabb_{\nu'}(b')\trc'\left(2^{\frac{j}{2}}(N'-N_{\nu'})\right)^qF_j(u')\eta^{\nu'}_j(\o') d\o'}_{L^2(\MM)}\les \ep 2^{\frac{j}{4}} \gamma^{\nu'}_j.
\ee
Together with \eqref{cars11} and \eqref{cars12}, this yields:
\be\lab{cars21}
\norm{H_2}_{L^2(\MM)}\les \ep 2^{\frac{j}{4}} \gamma^{\nu'}_j.
\ee

Next, we estimates the term $H_3$ in the right-hand side of \eqref{bis:koko5}. Decomposing $L_{\nu'}$ on the frame $L', \lb', e_A'$ and using the fact that:
$$|\o'-\nu'|\les 2^{\frac{j}{2}},$$
we have schematically:
$$L_{\nu'}(b')=L'(b')+2^{-\frac{j}{2}}\nabb'(b')+2^{-j}\lb'(b').$$
Together with the transport equation \eqref{D4a} satisfied by $b$, we obtain:
$$L_{\nu'}(b')-L_{\nu'}(b_{\nu'})=-\db' b'+\db_{\nu'}b_{\nu'}+2^{-\frac{j}{2}}\nabb'(b')+2^{-j}\lb'(b').$$
In view of the definition \eqref{cars2} for $H_3$, this yields:
\bee
\norm{H_3}_{L^2(\MM)}&\les&\normm{\int_{\S}(-\db' b'+\db_{\nu'}b_{\nu'})\trc'\left(2^{\frac{j}{2}}(N'-N_{\nu'})\right)^qF_j(u')\eta^{\nu'}_j(\o') d\o'}_{L^2(\MM)}\\
\nn&&+2^{-\frac{j}{2}}\normm{\int_{\S}\nabb'(b')\trc'\left(2^{\frac{j}{2}}(N'-N_{\nu'})\right)^qF_j(u')\eta^{\nu'}_j(\o') d\o'}_{L^2(\MM)}\\
\nn&&+2^{-j}\normm{\int_{\S}\lb'(b')\trc'\left(2^{\frac{j}{2}}(N'-N_{\nu'})\right)^qF_j(u')\eta^{\nu'}_j(\o') d\o'}_{L^2(\MM)}.
\eee
Together with an estimate analog to \eqref{cars20} and an estimate analog to \eqref{cars7}, we get:
\bea\lab{cars22}
\norm{H_3}_{L^2(\MM)}&\les&\normm{\int_{\S}(-\db' b' +\db_{\nu'}b_{\nu'})\trc'\left(2^{\frac{j}{2}}(N'-N_{\nu'})\right)^qF_j(u')\eta^{\nu'}_j(\o') d\o'}_{L^2(\MM)}\\
\nn&&+\ep 2^{-\frac{j}{4}}\gamma^{\nu'}_j+\ep 2^{-\frac{j}{2}}\gamma^{\nu'}_j.
\eea
Next, we estimate the right-hand side of \eqref{cars22}. We have:
\bea\lab{cars23}
&&\normm{\int_{\S}(-\db' b' +\db_{\nu'}b_{\nu'})\trc'\left(2^{\frac{j}{2}}(N'-N_{\nu'})\right)^qF_j(u')\eta^{\nu'}_j(\o') d\o'}_{L^2(\MM)}\\
\nn&\les&\normm{\int_{\S}(-\db' +\db_{\nu'})b'\trc'\left(2^{\frac{j}{2}}(N'-N_{\nu'})\right)^qF_j(u')\eta^{\nu'}_j(\o') d\o'}_{L^2(\MM)}\\
\nn&&+\normm{\db_{\nu'}}_{L^\infty_{u_{\nu'},x_{\nu'}'}L^2_t}\normm{\int_{\S}(-b'+b_{\nu'})\trc'\left(2^{\frac{j}{2}}(N'-N_{\nu'})\right)^qF_j(u')\eta^{\nu'}_j(\o') d\o'}_{L^2_{u_{\nu'},x_{\nu'}'}L^\infty_t}\\
\nn&\les&\normm{\int_{\S}(-\db' +\db_{\nu'})b'\trc'\left(2^{\frac{j}{2}}(N'-N_{\nu'})\right)^qF_j(u')\eta^{\nu'}_j(\o') d\o'}_{L^2(\MM)}+\ep\normm{H}_{L^2_{u_{\nu'},x_{\nu'}'}L^\infty_t},
\eea
where we used in the last inequality the estimates \eqref{estn} \eqref{estk} for $\db_{\nu'}$ and the definition \eqref{bis:koko2} for $H$. Now, recall the decomposition \eqref{loeb26}:
\be\lab{beig}
-\db'+\db_{\nu'}= \Big(n^{-1}\nab n-2\d_{\nu'} N_{\nu'}-2\epsilon_{\nu'}\Big)\c (N'-N_{\nu'})-k(N'-N_{\nu'}, N'-N_{\nu'}).
\ee
This yields:
\bea\lab{cars24}
&&\normm{\int_{\S}(-\db' +\db_{\nu'})b'\trc'\left(2^{\frac{j}{2}}(N'-N_{\nu'})\right)^qF_j(u')\eta^{\nu'}_j(\o') d\o'}_{L^2(\MM)}\\
\nn&\les& 2^{-\frac{j}{2}}\norm{n^{-1}\nab n-2\d_{\nu'} N_{\nu'}-2\epsilon_{\nu'}}_{L^\infty_{u_{\nu'},x_{\nu'}'}L^2_t}\\
\nn&&\times\normm{\int_{\S}b'\trc'\left(2^{\frac{j}{2}}(N'-N_{\nu'})\right)^{q+1}F_j(u')\eta^{\nu'}_j(\o') d\o'}_{L^2_{u_{\nu'},x_{\nu'}'}L^\infty_t}\\
\nn&& + 2^{-j}\norm{k}_{L^6(\MM)}\normm{\int_{\S}b'\trc'\left(2^{\frac{j}{2}}(N'-N_{\nu'})\right)^{q+2}F_j(u')\eta^{\nu'}_j(\o') d\o'}_{L^3(\MM)}\\
\nn&\les& 2^{-\frac{j}{2}}\ep\normm{\int_{\S}b'\trc'\left(2^{\frac{j}{2}}(N'-N_{\nu'})\right)^{q+1}F_j(u')\eta^{\nu'}_j(\o') d\o'}_{L^2_{u_{\nu'},x_{\nu'}'}L^\infty_t}\\
\nn&& + 2^{-j}\ep\normm{\int_{\S}b'\trc'\left(2^{\frac{j}{2}}(N'-N_{\nu'})\right)^{q+2}F_j(u')\eta^{\nu'}_j(\o') d\o'}_{L^3(\MM)},
\eea
where we used in the last inequality the estimate \eqref{estn} for $n$, the estimates \eqref{estn} \eqref{estk} for $\db$ and $\kep$ and the estimate \eqref{estk} for $k$. \eqref{cars24} together with the estimate \eqref{nice50} and the interpolation of \eqref{nice50} with \eqref{nice51} implies:
\be\lab{cars25}
\normm{\int_{\S}(-\db' +\db_{\nu'})b'\trc'\left(2^{\frac{j}{2}}(N'-N_{\nu'})\right)^qF_j(u')\eta^{\nu'}_j(\o') d\o'}_{L^2(\MM)}\les 2^{-\frac{j}{2}}\ep \gamma^{\nu'}_j.
\ee
Finally, \eqref{cars22}, \eqref{cars23} and \eqref{cars25} imply:
\be\lab{cars26}
\norm{H_3}_{L^2(\MM)}\les \ep\normm{H}_{L^2_{u_{\nu'},x_{\nu'}'}L^\infty_t}+\ep 2^{-\frac{j}{4}}\gamma^{\nu'}_j.
\ee

Finally, \eqref{bis:koko5}, \eqref{cars10}, \eqref{cars21} and \eqref{cars26} yield:
\bee
\norm{H}_{L^2_{u,x'}L^\infty_t}^2&\les& (2^{-\frac{j}{2}}(2^{\frac{j}{2}}|\nu-\nu'|^2)^2+2^{-\frac{j}{4}}(2^{\frac{j}{2}}|\nu-\nu'|))2^{-\frac{j}{2}}\ep^2(1+q^2)(\gamma^{\nu'}_j)^2\\
\nn&&+\ep\norm{H}_{L^2_{u_{\nu'},{x'}_{\nu'}}L^\infty_t}^2+\ep 2^{-\frac{j}{4}}\gamma^{\nu'}_j\norm{H}_{L^2_{u_{\nu'},{x'}_{\nu'}}L^\infty_t}\\
\nn&&+(1+q^5)2^{-\frac{j}{2}}(1+(2^{\frac{j}{2}}|\nu-\nu'|)^2)\ep^2(\ga^{\nu'}_j)^2.
\eee
This implies:
$$\norm{H}_{L^2_{u,x'}L^\infty_t}^2\les \Big(2^{-\frac{j}{2}}+2^{-\frac{3j}{4}}(2^{\frac{j}{2}}|\nu-\nu'|)+2^{-\frac{j}{2}}(2^{\frac{j}{2}}|\nu-\nu'|^2)^2\Big)(1+q^5)\ep^2(\ga^{\nu'}_j)^2.$$
This concludes the proof of the lemma.
\end{proof}

\subsection{Integration by parts}

\subsubsection{Integration by parts in tangential directions}

\begin{lemma}
We consider an oscillatory integral of the following form:
$$\int_{\MM}\int_{\S\times\S}b^{-1}{b'}^{-1}h(t,x)F_j(u)F_j(u')\eta_j^\nu(\o)\eta_j^{\nu'}(\o')d\o d\o'd\MM,$$
where $h$ is a scalar function on $\MM$. Integrating by parts once using \eqref{bisoa15} yields:
\bea\lab{fete}
&&\int_{\MM}\int_{\S\times\S}b^{-1}{b'}^{-1}h(t,x)F_j(u)F_j(u')\eta_j^\nu(\o)\eta_j^{\nu'}(\o')d\o d\o'd\MM\\
\nn&=&   -i2^{-j}\int_{\MM}\int_{\S\times\S}\frac{b^{-1}}{1-\gn^2}\bigg((N'-\gn N)(h)+\Big(\trt'-\gn \trt\\
\nn&&-\th'(N-\gn N', N-\gn N')-\gn {b'}^{-1}(N-\gn N')(b')\Big)h\\
\nn&&+\frac{2\gn}{1-\gn^2}\Big(\th(N'-\gn N,N'-\gn N)\\
\nn&&-\gn \th'(N-\gn N', N-\gn N')\Big)h\bigg)\\
\nn&& \times F_j(u)F_{j,-1}(u')\eta_j^\nu(\o)\eta_j^{\nu'}(\o')d\o d\o'd\MM.
\eea
Also, integrating by parts once using \eqref{bisoa17} yields:
\bea\lab{fete1}
&&\int_{\MM}\int_{\S\times\S}b^{-1}{b'}^{-1}h(t,x)F_j(u)F_j(u')\eta_j^\nu(\o)\eta_j^{\nu'}(\o')d\o d\o'd\MM\\
\nn&=&   i2^{-j}\int_{\MM}\int_{\S\times\S}\frac{{b'}^{-1}}{1-\gn^2}\bigg((N-\gn N')(h)+\Big(\trt-\gn \trt'\\
\nn&&-\th(N'-\gn N, N'-\gn N)-\gn b^{-1}(N'-\gn N)(b)\Big)h\\
\nn&&+\frac{2\gn}{1-\gn^2}\Big(\th'(N-\gn N',N-\gn N')\\
\nn&&-\gn \th(N'-\gn N, N'-\gn N)\Big)h\bigg)\\
\nn&& \times F_{j,-1}(u)F_j(u')\eta_j^\nu(\o)\eta_j^{\nu'}(\o')d\o d\o'd\MM.
\eea
\end{lemma}

\begin{proof}
We have:
\bee
&&\int_{\MM}\int_{\S\times\S}b^{-1}{b'}^{-1}h(t,x)F_j(u)F_j(u')\eta_j^\nu(\o)\eta_j^{\nu'}(\o')d\o d\o'd\MM\\
&=&\int_{\S\times\S}\int_0^{+\infty}\int_0^{+\infty}
\left(\int_{\MM}e^{i\la u-i\la' u'}b^{-1}{b'}^{-1}hd\MM\right)\\
\nn&&\times\eta_j^\nu(\o)\eta_j^{\nu'}(\o')\psi(2^{-j}\la)(2^{-j}\la')\psi(2^{-j}\la') f(\la\o)f(\la'\o')\la^2 {\la'}^2d\la d\la' d\o d\o'.
\eee
We integrate by parts in tangential directions using \eqref{bisoa15}. We obtain:
\bea\lab{fete2}
&&\int_{\MM}\int_{\S\times\S}h(t,x)F_j(u)F_j(u')\eta_j^\nu(\o)\eta_j^{\nu'}(\o')d\o d\o'd\MM\\
\nn&=&   -i2^{-j}\int_{\S\times\S}\int_0^{+\infty}\int_0^{+\infty}
\Bigg(\int_{\MM}e^{i\la u-i\la' u'}\frac{b^{-1}}{1-\gn^2}\bigg((N'-\gn N)(h)\\
\nn&&+(b(N'-\gn N)(b^{-1})+\textrm{div}_{\gg}(N'-\gn N))h\\
\nn&&+2\frac{\gn(N'-\gn N)(\gn)}{1-\gn^2}h\bigg) d\MM\Bigg)\\
\nn&&\times\eta_j^\nu(\o)\eta_j^{\nu'}(\o')(2^{-j}\la')^{-1}\psi(2^{-j}\la)(2^{-j}\la')\psi(2^{-j}\la') f(\la\o)f(\la'\o')\la^2 {\la'}^2d\la d\la' d\o d\o',
\eea
where $\textrm{div}_{\gg}(N'-\gn N)$ denotes the space-time divergence of $N'-\gn N$. 

Next, we consider the various terms in the right-hand side of \eqref{fete2}. 
Using \eqref{frame}, we have:
\bea\lab{fete3}
&&b(N'-\gn N)(b^{-1})+\textrm{div}_{\gg}(N'-\gn N)\\
\nn&=&\trt'-\gn \trt-\th'(N-\gn N', N-\gn N')\\
\nn&&+\gn (N-\gn N')(b'),
\eea
where we used the decomposition of $N$ in the frame $N', e'_A$:
\be\lab{fete4}
N=\gn N'+(N-\gn N'),
\ee
and the decomposition of $N'$ in the frame $N, e_A$:
\be\lab{fete5}
N'=\gn N+(N'-\gn N),
\ee
and where $\th$ is the second fundamental form of $\ptu$ in $\Sit$. We also have in view of \eqref{frame}, \eqref{fete4} and \eqref{fete5}:
\bea\lab{fete6}
&&(N'-\gn N)(\gn)\\
\nn&=& (\gn^2-1){b'}^{-1}(N-\gn N')(b')+\th(N'-\gn N,N'-\gn N)\\
\nn&&-\gn \th'(N-\gn N', N-\gn N').
\eea
Using \eqref{fete2}, \eqref{fete3} and \eqref{fete6}, we obtain:
\bee
&&\int_{\MM}\int_{\S\times\S}h(t,x)F_j(u)F_j(u')\eta_j^\nu(\o)\eta_j^{\nu'}(\o')d\o d\o'd\MM\\
\nn&=&   -i2^{-j}\int_{\MM}\int_{\S\times\S}\frac{b^{-1}}{1-\gn^2}\bigg((N'-\gn N)(h)+\Big(\trt'-\gn \trt\\
\nn&&-\th'(N-\gn N', N-\gn N')-\gn {b'}^{-1}(N-\gn N')(b')\Big)h\\
\nn&&+\frac{2\gn}{(1-\gn^2)}\Big(\th(N'-\gn N,N'-\gn N)\\
\nn&&-\gn \th'(N-\gn N', N-\gn N')\Big)h\\
\nn&& \times F_j(u)F_{j,-1}(u')\eta_j^\nu(\o)\eta_j^{\nu'}(\o')d\o d\o'd\MM,
\eee
which concludes the proof of \eqref{fete}.

In order to obtain \eqref{fete1}, we integrate by parts in tangential directions using \eqref{bisoa17} instead of \eqref{bisoa15}. The proof is completely analogous by exchanging the role played by $N$ and $N'$, so we omit it. This concludes the proof of the lemma.
\end{proof}

\subsubsection{Integration by parts in $L$}

\begin{lemma}
We consider an oscillatory integral of the following form:
$$\int_{\MM}\int_{\S\times\S}b^{-1}{b'}^{-1}h(t,x)F_j(u)F_j(u')\eta_j^\nu(\o)\eta_j^{\nu'}(\o')d\o d\o'd\MM,$$
where $h$ is a scalar function on $\MM$. Integrating by parts once using \eqref{ibpl} yields:
\bea\lab{fetebis}
&&\int_{\MM}\int_{\S\times\S}b^{-1}{b'}^{-1}h(t,x)F_j(u)F_j(u')\eta_j^\nu(\o)\eta_j^{\nu'}(\o')d\o d\o'd\MM\\
\nn&=&   -i2^{-j}\int_{\MM}\int_{\S\times\S}\frac{b^{-1}}{\gl}\bigg(L(h)+\trc h-\db h-\db'h-(1-\gn)\d'h\\
\nn&& -2\z'_{N-\gn N'}h-\frac{\chi'(N-\gn N', N-\gn N')}{\gl}h\bigg)\\
\nn&& F_j(u)F_{j,-1}(u')\eta_j^\nu(\o)\eta_j^{\nu'}(\o')d\o d\o'd\MM.
\eea
Also, integrating by parts once using \eqref{ibpl'} yields:
\bea\lab{fete1bis}
&&\int_{\MM}\int_{\S\times\S}b^{-1}{b'}^{-1}h(t,x)F_j(u)F_j(u')\eta_j^\nu(\o)\eta_j^{\nu'}(\o')d\o d\o'd\MM\\
\nn&=&   -i2^{-j}\int_{\MM}\int_{\S\times\S}\frac{{b'}^{-1}}{\gl}\bigg(L'(h)+\trc' h-\db h-\db'h-(1-\gn)\d h\\
\nn&& -2\z_{N'-\gn N}h-\frac{\chi(N'-\gn N, N'-\gn N)}{\gl}h\bigg)\\
\nn&& F_{j,-1}(u)F_j(u')\eta_j^\nu(\o)\eta_j^{\nu'}(\o')d\o d\o'd\MM.
\eea
\end{lemma}

\begin{proof}
We have:
\bee
&&\int_{\MM}\int_{\S\times\S}b^{-1}{b'}^{-1}h(t,x)F_j(u)F_j(u')\eta_j^\nu(\o)\eta_j^{\nu'}(\o')d\o d\o'd\MM\\
&=&\int_{\S\times\S}\int_0^{+\infty}\int_0^{+\infty}
\left(\int_{\MM}e^{i\la u-i\la' u'}b^{-1}{b'}^{-1}hd\MM\right)\\
\nn&&\times\eta_j^\nu(\o)\eta_j^{\nu'}(\o')\psi(2^{-j}\la)(2^{-j}\la')\psi(2^{-j}\la') f(\la\o)f(\la'\o')\la^2 {\la'}^2d\la d\la' d\o d\o'.
\eee
We integrate by parts in $L$ using \eqref{ibpl}. We obtain:
\bea\lab{fete2bis}
&&\int_{\MM}\int_{\S\times\S}h(t,x)F_j(u)F_j(u')\eta_j^\nu(\o)\eta_j^{\nu'}(\o')d\o d\o'd\MM\\
\nn&=&   -i2^{-j}\int_{\S\times\S}\int_0^{+\infty}\int_0^{+\infty}
\Bigg(\int_{\MM}e^{i\la u-i\la' u'}\frac{b^{-1}}{\gg(L,L')}\bigg(L(h)+(bL(b^{-1})+\textrm{div}_{\gg}(L))h\\
\nn&&-\frac{L(\gg(L,L'))}{g(L,L')}h\bigg) d\MM\Bigg)\\
\nn&&\times\eta_j^\nu(\o)\eta_j^{\nu'}(\o')\psi(2^{-j}\la)(2^{-j}\la')\psi(2^{-j}\la') f(\la\o)f(\la'\o')\la^2 {\la'}^2d\la d\la' d\o d\o',
\eea
where $\textrm{div}_{\gg}(L)$ denotes the space-time divergence of $L$. 

Next, we consider the various terms in the right-hand side of \eqref{fete2bis}. 
Using the Ricci equations \eqref{ricciform}, we have:
\be\lab{fete3bis}
L(b^{-1})+\textrm{div}_{\gg}(L)=b^{-1}\trc
\ee
and:
\be\lab{fete4bis}
L(\gg(L,L'))=-\db\gg(L,L')+\gg(L,\dd_LL').
\ee
We decompose $L$ on the frame $(L',\lb', e'_A)$:
\be\lab{fete5bis}
L=\frac{1}{2}(1+\gn)L'+\frac{1}{2}(1-\gn)\lb'+N-\gn N',
\ee
where the vector $N-\gn N'$ is tangent to $P_{t,u'}$. \eqref{fete4bis}, \eqref{fete5bis} and the Ricci equations \eqref{ricciform} yields:
\bea\lab{fete6bis}
&& L(\gg(L,L'))\\
\nn&=&-\db\gg(L,L')-\frac{1}{2}(1+\gn)\db'\gg(L,L')+(1-\gn)\z'_{N-\gn N'}\\
\nn&&+\half(1-\gn)(\d'+n^{-1}\nab_{N'} n)\gg(L,L')+\chi'(N-\gn N', N-\gn N')\\
\nn&&-\z'_{N-\gn N'}\gg(L,L')
\eea
Using \eqref{fete2bis}, \eqref{fete3bis} and \eqref{fete6bis}, we obtain:
\bee
&&\int_{\MM}\int_{\S\times\S}h(t,x)F_j(u)F_j(u')\eta_j^\nu(\o)\eta_j^{\nu'}(\o')d\o d\o'd\MM\\
\nn&=&   -i2^{-j}\int_{\MM}\int_{\S\times\S}\frac{b^{-1}}{\gl}\bigg(L(h)+\trc h-\db h-\db'h-(1-\gn)\d'h\\
\nn&& -2\z'_{N-\gn N'}h-\frac{\chi'(N-\gn N', N-\gn N')}{\gl}h\bigg)\\
\nn&& F_j(u)F_{j,-1}(u')\eta_j^\nu(\o)\eta_j^{\nu'}(\o')d\o d\o'd\MM,
\eee
where we also used the identity:
$$\gl=-1+\gn.$$
This concludes the proof of \eqref{fetebis}.

In order to obtain \eqref{fete1bis}, we integrate by parts in $L'$ using \eqref{ibpl'} instead of \eqref{ibpl}. The proof is completely analogous by exchanging the role played by $L$ and $L'$, so we omit it. This concludes the proof of the lemma.
\end{proof}

\section{Proof of Proposition \ref{prop:tsonga2bis}}\lab{sec:tsonga2bis}

Since $2^{\min(l,m)}\leq 2^j|\nu-\nu'|$, we may assume that $l>m$ and thus:
\be\lab{uso1}
m<l\textrm{ and }2^m\leq 2^j|\nu-\nu'|.
\ee
In order to prove Proposition \ref{prop:tsonga2bis}, recall that we need to exhibit a decomposition:
\be\lab{tsonga3bis:1}
\int_{\MM}E^{\nu,l}_jf(t,x)\overline{E^{\nu',m}_jf(t,x)}d\MM=A_{j,\nu,\nu',l,m}+B_{j,\nu,\nu',l,m},
\ee
where $B_{j,\nu,\nu',l,m}$ satisfies:
\bea\lab{tsonga3ter:1}
&&\left|\sum_{(l,m)/2^{\min(l,m)}\leq 2^j|\nu-\nu'|}(B_{j,\nu,\nu',l,m}+B_{j,\nu',\nu,l,m})\right|\\
\nn&\les&\bigg[\frac{1}{(2^{\frac{j}{2}}|\nu-\nu'|)^3}+\frac{1}{(2^{\frac{j}{2}}|\nu-\nu'|)^{\frac{5}{2}}}+\frac{1}{2^{\frac{j}{4}}(2^{\frac{j}{2}}|\nu-\nu'|)^{\frac{3}{2}}}+\frac{2^{-(\frac{1}{12})_-j}}{(2^{\frac{j}{2}}|\nu-\nu'|)^2}+\frac{1}{2^{\frac{j}{2}}(2^{\frac{j}{2}}|\nu-\nu'|)}\\
\nn&&+\frac{1}{2^{\frac{3j}{4}}(2^{\frac{j}{2}}|\nu-\nu'|)^{\frac{1}{2}}}+2^{-j}\bigg]\ep^2\gamma^\nu_j\gamma^{\nu'}_j.
\eea

We have:
\bee
\int_{\MM}E^{\nu,l}_jf(t,x)\overline{E^{\nu',m}_jf(t,x)}d\MM&=&   \int_{\MM}\int_{\S\times\S}
b^{-1}P_l\trc {b'}^{-1}P_m\trc'\\
&&\times F_j(u)F_j(u')\eta_j^\nu(\o)\eta_j^{\nu'}(\o')d\o d\o' d\MM.
\eee
We first integrate by parts in $L$ using \eqref{fetebis} with the choice $h=P_l\trc P_m\trc'$. We obtain:
\bea\lab{nice}
&&\int_{\MM}E^{\nu,l}_jf(t,x)\overline{E^{\nu',m}_jf(t,x)}d\MM\\
\nn&=&   -i2^{-j}\int_{\MM}\int_{\S\times\S}\frac{b^{-1}}{\gl}\Bigg(L(P_l\trc)\trc P_m\trc'+P_l\trc L(P_m\trc')+\bigg(\trc -\db -\db'\\
\nn&&-(1-\gn)\d' -2\z'_{N-\gn N'}-\frac{\chi'(N-\gn N', N-\gn N')}{\gl}\bigg)\\
\nn&&\times P_l\trc P_m\trc'\Bigg) F_j(u)F_{j,-1}(u')\eta_j^\nu(\o)\eta_j^{\nu'}(\o')d\o d\o'd\MM.
\eea
Next, we decompose $L$ on the frame $(L',N', e'_A)$:
\be\lab{nice5}
L=L'+(\gn-1)N'+N-\gn N',
\ee
which yields:
\be\lab{nice6}
L(P'_m\trc')=L'(P'_m\trc')+(\gn-1)N'(P'_m\trc')+(N-\gn N')(P'_m\trc').
\ee
Now, \eqref{nice} and \eqref{nice6} yield:
\be\lab{nice7}
\int_{\MM}E^{\nu,l}_jf(t,x)\overline{E^{\nu',m}_jf(t,x)}d\MM=A_{j,\nu,\nu',l,m}+B_{j,\nu,\nu',l,m}
\ee
where $A_{j,\nu,\nu',l,m}$ is given by:
\bea\lab{nice8}
A_{j,\nu,\nu',l,m}&=& -i2^{-j}\int_{\MM}\int_{\S\times\S} \frac{P_l\trc (N-\gn N')(P_m\trc')}{\gg(L,L')}\\
\nn&&\times F_j(u)F_{j,-1}(u')\eta_j^\nu(\o)\eta_j^{\nu'}(\o')d\o d\o' d\MM
\eea
and $B_{j,\nu,\nu',l,m}$ may be decomposed as:
\be\lab{nice8bis}
B_{j,\nu,\nu',l,m}=B^1_{j,\nu,\nu',l,m}+B^2_{j,\nu,\nu',l,m},
\ee
where $B^1_{j,\nu,\nu',l,m}$ and $B^2_{j,\nu,\nu',l,m}$ are given by:
\bea\lab{nice9}
B^1_{j,\nu,\nu',l,m}&=& -i2^{-j}\int_{\MM}\int_{\S\times\S} \frac{b^{-1}}{\gg(L,L')}\bigg(L(P_l\trc)P_m\trc'+P_l\trc L'(P_m\trc')\bigg)\\
\nn&&\times F_j(u)F_{j,-1}(u')\eta_j^\nu(\o)\eta_j^{\nu'}(\o')d\o d\o' d\MM,
\eea
and:
\bea
\nn B^2_{j,\nu,\nu',l,m}&=& -i2^{-j}\int_{\MM}\int_{\S\times\S} \frac{b^{-1}}{\gg(L,L')}\Bigg((\gn-1)P_l\trc N'(P_m\trc')+\bigg(\trc-\db-\db'\\
\nn&&-(1-\gn)\d' -2\z'_{N-\gn N'}-\frac{\chi'(N-\gn N', N-\gn N')}{\gl}\bigg)\\
\lab{nice10}&& P_l\trc P_m\trc'\Bigg)F_j(u)F_{j,-1}(u')\eta_j^\nu(\o)\eta_j^{\nu'}(\o')d\o d\o' d\MM.
\eea

The estimates satisfied by $B^1_{j,\nu,\nu',l,m}$ and $B^2_{j,\nu,\nu',l,m}$ are provided by the following propositions.
\begin{proposition}\lab{prop:labexfsmp}
Let $B^1_{j,\nu,\nu',l,m}$ be given by \eqref{nice9}. Then, we have the following estimate:
\bea\lab{bizu39}
&&\left|\sum_{(l,m)/ 2^{\min(l,m)}\leq 2^j|\nu-\nu'|}(B^1_{j,\nu,\nu',l,m}+B^1_{j,\nu',\nu,l,m})\right| \\
\nn&\les&  \bigg[\frac{1}{(2^{\frac{j}{2}}|\nu-\nu'|)^3}+\frac{1}{(2^{\frac{j}{2}}|\nu-\nu'|)^{\frac{5}{2}}}+\frac{2^{-(\frac{1}{12})_-j}}{(2^{\frac{j}{2}}|\nu-\nu'|)^2}+\frac{1}{2^{\frac{j}{2}}(2^{\frac{j}{2}}|\nu-\nu'|)}\\
\nn&&+\frac{1}{2^{\frac{3j}{4}}(2^{\frac{j}{2}}|\nu-\nu'|)^{\frac{1}{2}}}+ 2^{-j}\bigg]\ep^2\gamma^\nu_j\gamma^{\nu'}_j.
\eea
\end{proposition}

\begin{proposition}\lab{prop:labexfsmp1}
Let $B^2_{j,\nu,\nu',l,m}$ be given by \eqref{nice10}. Then, we have the following estimate:
\bea\lab{duc28}
&&\left|\sum_{(l,m)/2^{\min(l,m)}\leq 2^j|\nu-\nu'|}(B^2_{j,\nu,\nu',l,m}+B^2_{j,\nu',\nu,l,m})\right|\\
\nn&\les& \bigg[\frac{1}{(2^{\frac{j}{2}}|\nu-\nu'|)^3}+\frac{1}{(2^{\frac{j}{2}}|\nu-\nu'|)^{\frac{5}{2}}}+\frac{1}{2^{\frac{j}{4}}(2^{\frac{j}{2}}|\nu-\nu'|)^{\frac{3}{2}}}+\frac{2^{-\frac{j}{4}}}{(2^{\frac{j}{2}}|\nu-\nu'|)^2}+\frac{1}{2^{\frac{j}{2}}(2^{\frac{j}{2}}|\nu-\nu'|)}\\
\nn&&+\frac{1}{2^{\frac{3j}{4}}(2^{\frac{j}{2}}|\nu-\nu'|)^{\frac{1}{2}}}+2^{-j}\bigg]\ep^2\gamma^\nu_j\gamma^{\nu'}_j.
\eea
\end{proposition}

Now, in view of the decomposition \eqref{nice8bis}, we have:
\bee
\left|\sum_{(l,m)/ 2^{\min(l,m)}\leq 2^j|\nu-\nu'|}(B_{j,\nu,\nu',l,m}+B_{j,\nu',\nu,l,m})\right|&\les& 
\left|\sum_{(l,m)/ 2^{\min(l,m)}\leq 2^j|\nu-\nu'|}(B^1_{j,\nu,\nu',l,m}+B^1_{j,\nu',\nu,l,m})\right|\\
&&+\left|\sum_{(l,m)/ 2^{\min(l,m)}\leq 2^j|\nu-\nu'|}(B^2_{j,\nu,\nu',l,m}+B^2_{j,\nu',\nu,l,m})\right|.
\eee
Together with \eqref{bizu39} and \eqref{duc28}, this yields the estimate \eqref{tsonga3ter:1} and thus concludes the proof of Proposition \ref{prop:tsonga2bis}. The rest of this section is devoted to the proof of  Proposition \ref{prop:labexfsmp} and Proposition \ref{prop:labexfsmp1}.

We start with the proof of Proposition \ref{prop:labexfsmp}. We rewrite $B^1_{j,\nu,\nu',l,m}$ as:
\bea
\nn B^1_{j,\nu,\nu',l,m}&=& -i2^{-j}\int_{\MM}\int_{\S\times\S}\int_0^{\infty}\int_0^{\infty} \frac{b^{-1}}{\gg(L,L')}\bigg(L(P_l\trc)P_m\trc'+P_l\trc L'(P_m\trc')\bigg)\\
\nn&&\times \eta_j^\nu(\o)\eta_j^{\nu'}(\o')(2^{-j}\la')^{-1}\psi(2^{-j}\la)(2^{-j}\la')\psi(2^{-j}\la') f(\la\o)f(\la'\o')\la^2 {\la'}^2d\la d\la'\\
\lab{nice12} &=& B^{1,1}_{j,\nu,\nu',l,m}+B^{1,2}_{j,\nu,\nu',l,m},
\eea
where $B^{1,1}_{j,\nu,\nu',l,m}$ and $B^{1,2}_{j,\nu,\nu',l,m}$ are given by:
\bea
\nn B^{1,1}_{j,\nu,\nu',l,m} &=& -i2^{-j}\int_{\MM}\int_{\S\times\S}\int_0^{\infty}\int_0^{\infty} \frac{b^{-1}}{\gg(L,L')}\bigg(L(P_l\trc)P_m\trc'+P_l\trc L'(P_m\trc')\bigg)\\
\nn&&\times \eta_j^\nu(\o)\eta_j^{\nu'}(\o')\frac{(2^{-j}\la')^{-1}+(2^{-j}\la)^{-1}}{2}\psi(2^{-j}\la)(2^{-j}\la')\psi(2^{-j}\la') f(\la\o)f(\la'\o')\\
\lab{nice13} &&\times\la^2 {\la'}^2d\la d\la'd\o d\o' d\MM
\eea
and:
\bea
\nn B^{1,2}_{j,\nu,\nu',l,m} &=& -i2^{-j}\int_{\MM}\int_{\S\times\S}\int_0^{\infty}\int_0^{\infty} \frac{b^{-1}}{\gg(L,L')}\bigg(L(P_l\trc)P_m\trc'+P_l\trc L'(P_m\trc')\bigg)\\
\nn&&\times \eta_j^\nu(\o)\eta_j^{\nu'}(\o')\frac{(2^{-j}\la')^{-1}-(2^{-j}\la)^{-1}}{2}\psi(2^{-j}\la)(2^{-j}\la')\psi(2^{-j}\la') f(\la\o)f(\la'\o')\\
\lab{nice14} &&\times\la^2 {\la'}^2d\la d\la'd\o d\o' d\MM.
\eea
$B^{1,1}_{j,\nu,\nu',l,m}$ and $B^{1,2}_{j,\nu,\nu',l,m}$ satisfy the following estimates:
\begin{proposition}\lab{prop:labexfsmp3}
Let $B^{1,1}_{j,\nu,\nu',l,m}$ be given by \eqref{nice13}. Then, we have the following estimate:
\bea\lab{nice80}
&&\left|\sum_{(m,l)\ 2^m\leq 2^j|\nu-\nu'|}B^{1,1}_{j,\nu,\nu',l,m}+\sum_{(m,l)\ 2^m\leq 2^j|\nu-\nu'|}B^{1,1}_{j,\nu,\nu',l,m}\right|\\
\nn&\les& |B^{1,1,1}_{j,\nu,\nu'}+B^{1,1,1}_{j,\nu',\nu}|+|B^{1,1,2}_{j,\nu,\nu'}+B^{1,1,2}_{j,\nu',\nu}|+\frac{\ep^2\gamma_j^\nu\gamma_j^{\nu'}}{(2^{\frac{j}{2}}|\nu-\nu'|)^{3_-}}\\
\nn&\les& \left[\frac{2^{-(\frac{1}{12})_-j}}{(2^{\frac{j}{2}}|\nu-\nu'|)^2}+\frac{1}{(2^{\frac{j}{2}}|\nu-\nu'|)^{\frac{5}{2}}}\right]\ep^2\gamma_j^\nu\gamma^j_{\nu'}.
\eea
\end{proposition}

\begin{proposition}\lab{prop:labexfsmp4}
Let $B^{1,2}_{j,\nu,\nu',l,m}$ be given by \eqref{nice14}. Then, we have the following estimate:
\bea\lab{bizu38}
&&\left|\sum_{(l,m)/ 2^{\min(l,m)}\leq 2^j|\nu-\nu'|}(B^{1,2}_{j,\nu,\nu',l,m}+B^{1,2}_{j,\nu',\nu,l,m})\right| \\
\nn&\les&  \bigg[\frac{1}{(2^{\frac{j}{2}}|\nu-\nu'|)^3}+\frac{j 2^{-\frac{j}{12}}}{(2^{\frac{j}{2}}|\nu-\nu'|)^2}+\frac{1}{2^{\frac{j}{2}}(2^{\frac{j}{2}}|\nu-\nu'|)}+\frac{1}{2^{\frac{3j}{4}}(2^{\frac{j}{2}}|\nu-\nu'|)^{\frac{1}{2}}}+ 2^{-j}\bigg]\ep^2\gamma^\nu_j\gamma^{\nu'}_j.
\eea
\end{proposition}

Now, the decomposition \eqref{nice12} of $B^1_{j,\nu,\nu',l,m}$ yields:
$$\left|\sum_{(l,m)/ 2^{\min(l,m)}\leq 2^j|\nu-\nu'|}B^1_{j,\nu,\nu',l,m}\right| \les \left|\sum_{(l,m)/ 2^{\min(l,m)}\leq 2^j|\nu-\nu'|}B^{1,1}_{j,\nu,\nu',l,m}\right|+\left|\sum_{(l,m)/ 2^{\min(l,m)}\leq 2^j|\nu-\nu'|}B^{1,2}_{j,\nu,\nu',l,m}\right|.$$
Together with the estimates \eqref{nice80} and \eqref{bizu38}, we obtain:
\bee
&&\left|\sum_{(l,m)/ 2^{\min(l,m)}\leq 2^j|\nu-\nu'|}(B^1_{j,\nu,\nu',l,m}+B^1_{j,\nu',\nu,l,m})\right| \\
\nn&\les&  \bigg[\frac{1}{(2^{\frac{j}{2}}|\nu-\nu'|)^3}+\frac{1}{(2^{\frac{j}{2}}|\nu-\nu'|)^{\frac{5}{2}}}+\frac{2^{-(\frac{1}{12})_-j}}{(2^{\frac{j}{2}}|\nu-\nu'|)^2}+\frac{1}{2^{\frac{j}{2}}(2^{\frac{j}{2}}|\nu-\nu'|)}\\
\nn&&+\frac{1}{2^{\frac{3j}{4}}(2^{\frac{j}{2}}|\nu-\nu'|)^{\frac{1}{2}}}+ 2^{-j}\bigg]\ep^2\gamma^\nu_j\gamma^{\nu'}_j.
\eee
This concludes the proof of Proposition \ref{prop:labexfsmp}.

The rest of this section is organized as follows. Proposition \ref{prop:labexfsmp3} is proved in section \ref{sec:labex}, Proposition \ref{prop:labexfsmp4} is proved in section \ref{sec:labex1}, and Proposition \ref{prop:labexfsmp1} is proved in section \ref{sec:labex2}. 

\subsection{Proof of Proposition \ref{prop:labexfsmp3} (Control of $B^{1,1}_{j,\nu,\nu',l,m}$)}\lab{sec:labex}

Recall the definition \eqref{nice13} of $B^{1,1}_{j,\nu,\nu',l,m}$:
\bee
\nn B^{1,1}_{j,\nu,\nu',l,m} &=& -i2^{-j}\int_{\MM}\int_{\S\times\S}\int_0^{\infty}\int_0^{\infty} \frac{b^{-1}}{\gg(L,L')}\bigg(L(P_l\trc)P_m\trc'+P_l\trc L'(P_m\trc')\bigg)\\
\nn&&\times \eta_j^\nu(\o)\eta_j^{\nu'}(\o')\frac{(2^{-j}\la')^{-1}+(2^{-j}\la)^{-1}}{2}\psi(2^{-j}\la)(2^{-j}\la')\psi(2^{-j}\la') f(\la\o)f(\la'\o')\\
&&\times\la^2 {\la'}^2d\la d\la'd\o d\o' d\MM
\eee
We have:
\bea\lab{nice15}
&&\sum_{(l,m)/ 2^{\min(l,m)}\leq 2^j|\nu-\nu'|}\Big(L(P_l\trc)P_m\trc'+P_l\trc L'(P_m\trc')\Big)\\
\nn&=&L(\trc)\trc'+\trc L(\trc')-\sum_{(l,m)/ 2^{\min(l,m)}> 2^j|\nu-\nu'|}\Big(L(P_l\trc)P_m\trc'+P_l\trc L'(P_m\trc')\Big).
\eea
Now, the difference between $B^1_{j,\nu,\nu',l,m}$ and $B^{1,1}_{j,\nu,\nu',l,m}$ is the fact that the term $(2^{-j}\la')^{-1}$ has been replaced by:
$$\frac{(2^{-j}\la')^{-1}+(2^{-j}\la)^{-1}}{2}$$
such as to obtain an expression which is totally symmetric in $(\la,\la')$ and $(\o, \o')$. In turn, we may sum over $l,m$ belonging to the region $2^m\leq 2^j|\nu-\nu'|$. Together with \eqref{nice15} we obtain:
\bea\lab{nice16}
&&\sum_{(m,l)\ 2^m\leq 2^j|\nu-\nu'|}B^{1,1}_{j,\nu,\nu',l,m}\\
\nn&=&B^{1,1,1}_{j,\nu,\nu'}+B^{1,1,2}_{j,\nu,\nu'}+\sum_{(m,l)/ 2^m> 2^j|\nu-\nu'|}B^{1,1,3}_{j,\nu,\nu',l,m}+\sum_{(m,l)/ 2^m > 2^j|\nu-\nu'|}B^{1,1,4}_{j,\nu,\nu',l,m}
\eea
where $B^{1,1,1}_{j,\nu,\nu'}$, $B^{1,1,2}_{j,\nu,\nu'}$, $B^{1,1,3}_{j,\nu,\nu',l,m}$ and $B^{1,1,4}_{j,\nu,\nu',l,m}$ are given by:
\bea\lab{nice17}
B^{1,1,1}_{j,\nu,\nu'}&=& -i2^{-j-1}\int_{\MM}\int_{\S\times\S} \frac{b^{-1}}{\gg(L,L')}\bigg(L(\trc)\trc'+\trc L'(\trc')\bigg)\\
\nn&&\times F_j(u)F_{j,-1}(u')\eta_j^\nu(\o)\eta_j^{\nu'}(\o')d\o d\o' d\MM,
\eea
\bea\lab{nice18}
B^{1,1,2}_{j,\nu,\nu'}&=& -i2^{-j-1}\int_{\MM}\int_{\S\times\S} \frac{b^{-1}}{\gg(L,L')}\bigg(L(\trc)\trc'+\trc L'(\trc')\bigg)\\
\nn&&\times F_{j,-1}(u)F_j(u')\eta_j^\nu(\o)\eta_j^{\nu'}(\o')d\o d\o' d\MM,
\eea
\bea\lab{nice19}
B^{1,1,3}_{j,\nu,\nu',l,m}&=& i2^{-j-1}\int_{\MM}\int_{\S\times\S} \frac{b^{-1}}{\gg(L,L')}\bigg(L(P_l\trc)P_m\trc'+P_l\trc L'(P_m\trc')\bigg)\\
\nn&&\times F_j(u)F_{j,-1}(u')\eta_j^\nu(\o)\eta_j^{\nu'}(\o')d\o d\o' d\MM,
\eea
and:
\bea\lab{nice20}
B^{1,1,4}_{j,\nu,\nu',l,m}&=& i2^{-j-1}\int_{\MM}\int_{\S\times\S} \frac{b^{-1}}{\gg(L,L')}\bigg(L(P_l\trc)P_m\trc'+P_l\trc L'(P_m\trc')\bigg)\\
\nn&&\times F_{j,-1}(u)F_j(u')\eta_j^\nu(\o)\eta_j^{\nu'}(\o')d\o d\o' d\MM.
\eea

We have the following propositions:
\begin{proposition}\lab{prop:labexfsmp5}
Let $B^{1,1,3}_{j,\nu,\nu',l,m}$ be given by \eqref{nice19}, and let $B^{1,1,4}_{j,\nu,\nu',l,m}$ be given by \eqref{nice20}. Then, we have the following estimate:
\be\lab{eq:labexfsmp}
\sum_{(m,l)/ 2^{\min(m,l)}> 2^j|\nu-\nu'|}(|B^{1,1,3}_{j,\nu,\nu',l,m}|+|B^{1,1,4}_{j,\nu,\nu',l,m}|)\les \frac{\ep^2\gamma_j^\nu\gamma_j^{\nu'}}{(2^{\frac{j}{2}}|\nu-\nu'|)^{3_-}}.
\ee
\end{proposition}

\begin{proposition}\lab{prop:labexfsmp6}
Let $B^{1,1,1}_{j,\nu,\nu'}$ be given by \eqref{nice17}, and let $B^{1,1,2}_{j,\nu,\nu'}$ be given by \eqref{nice18}. Then, we have the following estimate:
\be\lab{eq:labexfsmp1}
|B^{1,1,1}_{j,\nu,\nu'}+B^{1,1,1}_{j,\nu',\nu}|+|B^{1,1,2}_{j,\nu,\nu'}+B^{1,1,2}_{j,\nu',\nu}|\les \left[\frac{2^{-(\frac{1}{12})_-j}}{(2^{\frac{j}{2}}|\nu-\nu'|)^2}+\frac{1}{(2^{\frac{j}{2}}|\nu-\nu'|)^{\frac{5}{2}}}\right]\ep^2\gamma_j^\nu\gamma^j_{\nu'}.
\ee
\end{proposition}

In view of the decomposition \eqref{nice16}, we have:
\bee
&&\left|\sum_{(m,l)/ 2^m\leq 2^j|\nu-\nu'|}B^{1,1}_{j,\nu,\nu',l,m}-(B^{1,1,1}_{j,\nu,\nu'}+B^{1,1,2}_{j,\nu,\nu'})\right|\\
\nn&\les& \sum_{(m,l)/ 2^m> 2^j|\nu-\nu'|}|B^{1,1,3}_{j,\nu,\nu',l,m}|+\sum_{(m,l)/ 2^m > 2^j|\nu-\nu'|}|B^{1,1,4}_{j,\nu,\nu',l,m}|\\
\nn&\les & \frac{\ep^2\gamma_j^\nu\gamma_j^{\nu'}}{(2^{\frac{j}{2}}|\nu-\nu'|)^{3_-}},
\eee
where we used the estimate \eqref{eq:labexfsmp} in the last inequality. Together with \eqref{eq:labexfsmp1}, this yields:
\bee
&&\left|\sum_{(m,l)/ 2^m\leq 2^j|\nu-\nu'|}B^{1,1}_{j,\nu,\nu',l,m}+\sum_{(m,l)/ 2^m\leq 2^j|\nu-\nu'|}B^{1,1}_{j,\nu,\nu',l,m}\right|\\
\nn&\les& |B^{1,1,1}_{j,\nu,\nu'}+B^{1,1,1}_{j,\nu',\nu}|+|B^{1,1,2}_{j,\nu,\nu'}+B^{1,1,2}_{j,\nu',\nu}|+\frac{\ep^2\gamma_j^\nu\gamma_j^{\nu'}}{(2^{\frac{j}{2}}|\nu-\nu'|)^{3_-}}\\
\nn&\les& \left[\frac{2^{-(\frac{1}{12})_-j}}{(2^{\frac{j}{2}}|\nu-\nu'|)^2}+\frac{1}{(2^{\frac{j}{2}}|\nu-\nu'|)^{\frac{5}{2}}}\right]\ep^2\gamma_j^\nu\gamma^j_{\nu'}.
\eee
This concludes the proof of Proposition \ref{prop:labexfsmp3}. 

The rest of this section is as follows. In section \ref{sec:lobotomisation}, we give a proof of Proposition \ref{prop:labexfsmp5}, and In section \ref{sec:lobotomisation1}, we give a proof of Proposition \ref{prop:labexfsmp6}.

\subsubsection{Proof of Proposition \ref{prop:labexfsmp5} (Control of $B^{1,1,3}_{j,\nu,\nu',l,m}$ and $B^{1,1,4}_{j,\nu,\nu',l,m}$)}\lab{sec:lobotomisation}

We further decompose. We have:
\bea\lab{nice21}
B^{1,1,3}_{j,\nu,\nu',l,m}= B^{1,1,3,1}_{j,\nu,\nu',l,m}+B^{1,1,3,2}_{j,\nu,\nu',l,m},
\eea
where $B^{1,1,3,1}_{j,\nu,\nu',l,m}$ and $B^{1,1,3,2}_{j,\nu,\nu',l,m}$ are given by:
\bea\lab{nice22}
B^{1,1,3,1}_{j,\nu,\nu',l,m}&=& i2^{-j-1}\int_{\MM}\int_{\S\times\S} \frac{b^{-1}}{\gg(L,L')}L(P_l\trc)P_m\trc' F_j(u)F_{j,-1}(u')\\
\nn&&\times\eta_j^\nu(\o)\eta_j^{\nu'}(\o')d\o d\o' d\MM,
\eea
and:
\bea\lab{nice23}
B^{1,1,3,2}_{j,\nu,\nu',l,m}&=& i2^{-j-1}\int_{\MM}\int_{\S\times\S} \frac{b^{-1}}{\gg(L,L')}P_l\trc L'(P_m\trc')F_j(u)F_{j,-1}(u')\\
\nn&&\times \eta_j^\nu(\o)\eta_j^{\nu'}(\o')d\o d\o' d\MM,
\eea
The terms $B^{1,1,3,1}_{j,\nu,\nu',l,m}$ and $B^{1,1,3,2}_{j,\nu,\nu',l,m}$ are estimated in the same way, so we focus on $B^{1,1,3,1}_{j,\nu,\nu',l,m}$. We first deal with $\gl$. We have the identities:
\be\lab{nice24}
\gg(L,L')=-1+\gn
\ee
and
\be\lab{nice25}
1-\gn=\frac{\gg(N-N',N-N')}{2}.
\ee
Furthermore, the estimates on $N$ \eqref{estNomega} and \eqref{threomega1ter} yield:
\be\lab{nice26}
|N-N_\nu|\les |\o-\nu|,\, |N'-N_{\nu'}|\les |\o'-\nu'|\textrm{ and }|N_\nu-N_{\nu'}|\gtrsim |\nu-\nu'|,
\ee
where we have used the following notation for any vectorfield tangent to $\Si_t$:
$$|X|=\gg(X,X)^{\frac{1}{2}}.$$
Since $\o$ belongs to the patch of center $\nu$, $\o'$ belongs to the patch of center $\nu'$, and $\nu\neq\nu'$, we obtain in view of \eqref{nice24}, \eqref{nice25} and \eqref{nice26}:
\be\lab{nice27}
\frac{1}{\gl}=\frac{1}{|N_\nu-N_{\nu'}|^2}\left(\sum_{p, q\geq 0}c_{pq}\left(\frac{N-N_\nu}{|N_\nu-N_{\nu'}|}\right)^p\left(\frac{N'-N_{\nu'}}{|N_\nu-N_{\nu'}|}\right)^q\right),
\ee
for some explicit real coefficients $c_{pq}$ such that the series 
$$\sum_{p, q\geq 0}c_{pq}x^py^q$$
has radius of convergence 1.

In view of \eqref{nice22} and \eqref{nice27}, we may rewrite $B^{1,1,3,1}_{j,\nu,\nu',l,m}$ as:
\bee
\nn B^{1,1,3,1}_{j,\nu,\nu',l,m}&=& i2^{-j-1}\sum_{p, q\geq 0}c_{pq}\int_{\MM}\frac{1}{|N_\nu-N_{\nu'}|^2}\left(\int_{\S} b^{-1}L(P_l\trc)\left(\frac{N-N_\nu}{|N_\nu-N_{\nu'}|}\right)^pF_j(u)\eta_j^\nu(\o)d\o\right)\\
&&\times\left(\int_{\S}P_m\trc'\left(\frac{N'-N_{\nu'}}{|N_\nu-N_{\nu'}|}\right)^qF_{j,-1}(u')\eta_j^{\nu'}(\o')d\o'\right) d\MM.
\eee
Using the estimate \eqref{nadal} with the choice:
\be\lab{nice28}
H_{pq} = \frac{b^{-1}}{|N_\nu-N_{\nu'}|^2}\left(\frac{N-N_\nu}{|N_\nu-N_{\nu'}|}\right)^p\left(\int_{\S}P_m\trc'\left(\frac{N'-N_{\nu'}}{|N_\nu-N_{\nu'}|}\right)^qF_{j,-1}(u')\eta_j^{\nu'}(\o')d\o'\right),
\ee
we obtain:
\be\lab{nice29} 
|B^{1,1,3,1}_{j,\nu,\nu',l,m}|\les  \left(\sum_{p, q\geq 0}c_{pq}\left(\sup_{\o\in\textrm{supp}(\eta_j^{\nu})}(\norm{H_{pq}}_{L^2_{u, x'}L^\infty_t})\right)\right)2^{-\frac{j}{2}-l}\ep\gamma_j^\nu.
\ee

Next, we evaluate the right-hand side of \eqref{nice29}. In view of \eqref{nice26}, we have:
\bea\lab{nice30}
\normm{\frac{b^{-1}}{|N_\nu-N_{\nu'}|^2}\left(\frac{N-N_\nu}{|N_\nu-N_{\nu'}|}\right)^p}_{L^\infty(\MM)}&\les& \frac{\norm{b^{-1}}_{L^\infty(\MM)}}{|\nu-\nu'|^2}\left(\frac{|\o-\nu|}{|\nu-\nu'|}\right)^p\\
\nn &\les& \frac{1}{|\nu-\nu'|^2}\left(\frac{1}{2^{\frac{j}{2}}|\nu-\nu'|}\right)^p,
\eea
where we used in the last inequality the estimate \eqref{estb} for $b$, and the fact that $\o$ is in the patch centered around $\nu$ of diameter $\sim 2^{\frac{j}{2}}$. Let:
$$H^1_{pq}=\int_{\S}P_m\trc'\left(\frac{N'-N_{\nu'}}{|N_\nu-N_{\nu'}|}\right)^qF_{j,-1}(u')\eta_j^{\nu'}(\o')d\o'.$$
Then, \eqref{nice28}, \eqref{nice29} and \eqref{nice30} yield:
\be\lab{nice32}
|B^{1,1,3,1}_{j,\nu,\nu',l,m}|\les  \left(\sum_{p, q\geq 0}c_{pq}\left(\frac{1}{2^{\frac{j}{2}}|\nu-\nu'|}\right)^p\left(\sup_{\o\in\textrm{supp}(\eta_j^{\nu})}(\norm{H^1_{pq}}_{L^2_{u, x'}L^\infty_t})\right)\right)\frac{2^{\frac{j}{2}-l}\ep\gamma_j^\nu}{(2^{\frac{j}{2}}|\nu-\nu'|)^2}.
\ee

Next, we evaluate $H^1_{pq}$. In view of \eqref{nice26}, we have:
\be\lab{nice33}
\normm{\left(\frac{N'-N_{\nu'}}{|N_\nu-N_{\nu'}|}\right)^q}_{L^\infty(\MM)}\les \left(\frac{|\o'-\nu'|}{|\nu-\nu'|}\right)^q\les \left(\frac{1}{2^{\frac{j}{2}}|\nu-\nu'|}\right)^q,
\ee
where we used in the last inequality the fact that $\o'$ is in the patch centered around $\nu'$ of diameter $\sim 2^{\frac{j}{2}}$. Now, \eqref{nice33} together with Corollary \ref{cor:messi1} yields:
\be\lab{nice34}
\norm{H^1_{pq}}_{L^2_{u, x'}L^\infty_t}\les \left(\frac{1}{2^{\frac{j}{2}}|\nu-\nu'|}\right)^q\ep\big(2^{\frac{j}{2}}|\nu-\nu'|2^{-m+\frac{j}{2}}+(2^{\frac{j}{2}}|\nu-\nu'|)^{\frac{1}{2}}2^{-\frac{m}{2}+\frac{j}{4}}\big)\gamma^{\nu'}_j,
\ee

Finally, \eqref{nice32} and \eqref{nice34} imply: 
\bea
&&\lab{nice39} |B^{1,1,3,1}_{j,\nu,\nu',l,m}|\\
\nn&\les & \left(\sum_{p, q\geq 0}c_{pq}\left(\frac{1}{2^{\frac{j}{2}}|\nu-\nu'|}\right)^{p+q}\right)\bigg(2^{\frac{j}{2}}|\nu-\nu'|2^{-m+\frac{j}{2}}+(2^{\frac{j}{2}}|\nu-\nu'|)^{\frac{1}{2}}2^{-\frac{m}{2}+\frac{j}{4}}\bigg)\frac{2^{\frac{j}{2}-l}\ep^2\gamma_j^\nu\gamma_j^{\nu'}}{(2^{\frac{j}{2}}|\nu-\nu'|)^2}\\
\nn&\les & \bigg(2^{\frac{j}{2}}|\nu-\nu'|2^{-m+\frac{j}{2}}+(2^{\frac{j}{2}}|\nu-\nu'|)^{\frac{1}{2}}2^{-\frac{m}{2}+\frac{j}{4}}\bigg)\frac{2^{\frac{j}{2}-l}\ep^2\gamma_j^\nu\gamma_j^{\nu'}}{(2^{\frac{j}{2}}|\nu-\nu'|)^2}.
\eea
\eqref{nice39} implies:
$$\sum_{(m,l)/ 2^{\min(m,l)}> 2^j|\nu-\nu'|}|B^{1,1,3,1}_{j,\nu,\nu',l,m}|\les \frac{\ep^2\gamma_j^\nu\gamma_j^{\nu'}}{(2^{\frac{j}{2}}|\nu-\nu'|)^{3_-}}.$$
The term $B^{1,1,3,2}_{j,\nu,\nu',l,m}$ is completely analogous, so we obtain in view of \eqref{nice21}:
$$\sum_{(m,l)/ 2^{\min(m,l)}> 2^j|\nu-\nu'|}|B^{1,1,3}_{j,\nu,\nu',l,m}|\les \frac{\ep^2\gamma_j^\nu\gamma_j^{\nu'}}{(2^{\frac{j}{2}}|\nu-\nu'|)^{3}}.$$
The term $B^{1,1,4}_{j,\nu,\nu',l,m}$ is completely analogous to $B^{1,1,3}_{j,\nu,\nu',l,m}$. This concludes the proof of Proposition \ref{prop:labexfsmp6}.

\subsubsection{Proof of Proposition \ref{prop:labexfsmp6} (Control of $B^{1,1,1}_{j,\nu,\nu'}$ and $B^{1,1,2}_{j,\nu,\nu'}$)}\lab{sec:lobotomisation1}

We need to estimate $B^{1,1,1}_{j,\nu,\nu'}$ and $B^{1,1,2}_{j,\nu,\nu'}$. These terms are estimated in the same way, so we focus on $B^{1,1,1}_{j,\nu,\nu'}$. We further decompose:
\be\label{nice41}
B^{1,1,1}_{j,\nu,\nu'}=B^{1,1,1,1}_{j,\nu,\nu'}+B^{1,1,1,2}_{j,\nu,\nu'}
\ee
where $B^{1,1,1,1}_{j,\nu,\nu'}$ and $B^{1,1,1,2}_{j,\nu,\nu'}$ are given by:
\bea\lab{nice42}
B^{1,1,1,1}_{j,\nu,\nu'}&=& -i2^{-j-1}\int_{\MM}\int_{\S\times\S} \frac{b^{-1}}{\gg(L,L')}L(\trc)\trc'\\
\nn&&\times F_j(u)F_{j,-1}(u')\eta_j^\nu(\o)\eta_j^{\nu'}(\o')d\o d\o' d\MM,
\eea
and:
\bea\lab{nice43}
B^{1,1,1,2}_{j,\nu,\nu'}&=& -i2^{-j-1}\int_{\MM}\int_{\S\times\S} \frac{b^{-1}}{\gg(L,L')}\trc L'(\trc')\\
\nn&&\times F_j(u)F_{j,-1}(u')\eta_j^\nu(\o)\eta_j^{\nu'}(\o')d\o d\o' d\MM.
\eea
The terms $B^{1,1,1,1}_{j,\nu,\nu',l,m}$ and $B^{1,1,1,2}_{j,\nu,\nu',l,m}$ are estimated in the same way, so we focus on $B^{1,1,1,1}_{j,\nu,\nu',l,m}$. We integrate by parts in $B^{1,1,1,1}_{j,\nu,\nu',l,m}$ using \eqref{fete1}. 
\begin{lemma}\lab{lemma:app1}
Let $B^{1,1,1,1}_{j,\nu,\nu'}$ be defined by \eqref{nice42}. Integrating by parts using \eqref{fete1} yields:
\bea\lab{app1}
&&B^{1,1,1,1}_{j,\nu,\nu'}\\
\nn&=& 2^{-2j}\sum_{p, q\geq 0}c_{pq}\int_{\MM}\frac{1}{(2^{\frac{j}{2}}|N_\nu-N_{\nu'}|)^{p+q}}\\
\nn&&\times\left[\frac{1}{|N_\nu-N_{\nu'}|^2}h_{1,p,q}+\frac{1}{|N_\nu-N_{\nu'}|^3}(h_{2,p,q}+h_{3,p,q}+h_{4,p,q})\right] d\MM\\
\nn&&+2^{-2j}\int_{\MM}\int_{\S\times\S} \frac{(\chi-\chi')L(\trc)\trc'}{\gg(L,L')^2}F_{j,-1}(u)F_{j,-1}(u')\eta_j^\nu(\o)\eta_j^{\nu'}(\o')d\o d\o' d\MM,
\eea
where $c_{pq}$ are explicit real coefficients such that the series 
$$\sum_{p, q\geq 0}c_{pq}x^py^q$$
has radius of convergence 1, where the scalar functions $h_{1,p,q}, h_{2,p,q}, h_{3,p,q}, h_{4,p,q}$ on $\MM$ are given by:
\bea\lab{app2}
h_{1,p,q}&=& \left(\int_{\S} N(L(\trc))\left(2^{\frac{j}{2}}(N-N_\nu)\right)^pF_{j,-1}(u)\eta_j^\nu(\o)d\o\right)\\
\nn&&\times\left(\int_{\S}\trc'\left(2^{\frac{j}{2}}(N'-N_{\nu'})\right)^qF_{j,-1}(u')\eta_j^{\nu'}(\o')d\o'\right),
\eea
\bea\lab{app3}
h_{2,p,q}&=& \left(\int_{\S} \nabb L(\trc)\left(2^{\frac{j}{2}}(N-N_\nu)\right)^pF_{j,-1}(u)\eta_j^\nu(\o)d\o\right)\\
\nn&&\times\left(\int_{\S}\trc'\left(2^{\frac{j}{2}}(N'-N_{\nu'})\right)^qF_{j,-1}(u')\eta_j^{\nu'}(\o')d\o'\right),
\eea
\bea\lab{app4}
h_{3,p,q}&=& \left(\int_{\S} L(\trc)\left(2^{\frac{j}{2}}(N-N_\nu)\right)^pF_{j,-1}(u)\eta_j^\nu(\o)d\o\right)\\
\nn&&\times\left(\int_{\S}H_1\left(2^{\frac{j}{2}}(N'-N_{\nu'})\right)^qF_{j,-1}(u')\eta_j^{\nu'}(\o')d\o'\right),
\eea
\bea\lab{app5}
h_{4,p,q}&=& \left(\int_{\S} H_2\left(2^{\frac{j}{2}}(N-N_\nu)\right)^pF_{j,-1}(u)\eta_j^\nu(\o)d\o\right)\\
\nn&&\times\left(\int_{\S}\trc'\left(2^{\frac{j}{2}}(N'-N_{\nu'})\right)^qF_{j,-1}(u')\eta_j^{\nu'}(\o')d\o'\right),
\eea
where the tensor $H_1$ on $\MM$ involved in the definition of $h_{3,p,q}$ is a linear combination of terms in the following list:
\be\lab{app7}
{b'}^{-1}\nabb'(b'\trc'),\, \th'\trc',
\ee
and where the tensor $H_2$ on $\MM$ involved in the definition of $h_{4,p,q}$ is a linear combination of terms in the following list:
\be\lab{app8}
\th L(\trc),\, b^{-1}\nabb(b) L(\trc).
\ee
\end{lemma}
The proof of lemma \ref{lemma:app1} is postponed to Appendix A. In the rest of this section, we use Lemma \ref{lemma:app1} to obtain the control of $B^{1,1,1,1}_{j,\nu,\nu'}$.

We first estimate $h_{1,p,q}$. We have:
\be\lab{lucho}
h_{1,p,q}=  \int_{\S} HN(L(\trc))F_{j,-1}(u)\eta_j^\nu(\o)d\o,
\ee
where the tensor $H$ is given by:
\be\lab{lucho1}
H=\left(2^{\frac{j}{2}}(N-N_\nu)\right)^p\left(\int_{\S}\trc'\left(2^{\frac{j}{2}}(N'-N_{\nu'})\right)^qF_{j,-1}(u')\eta_j^{\nu'}(\o')d\o'\right),
\ee
In view of \eqref{lucho}, the estimate \eqref{nadalbis} in $L^1(\MM)$ yields:
\bea\lab{lucho2}
&&\norm{h_{1,p,q}}_{L^1(\MM)}\\
\nn&\les& \left(\sup_{\o\in\textrm{supp}(\eta_j^{\nu})}(\norm{H}_{L^2_uL^4_tL^2_{x'}})+|\nu-\nu'|\norm{H}_{L^{3_+}(\MM)}+\norm{{\chi_2}_{\nu'}H}_{L^2(\MM)}\right)2^{\frac{j}{2}}\ep\gamma_j^\nu.
\eea
Let the tensor $H_1$ be defined by:
\be\lab{nice47}
H_1= \int_{\S}\trc'\left(2^{\frac{j}{2}}(N'-N_{\nu'})\right)^qF_{j,-1}(u')\eta_j^{\nu'}(\o')d\o'.
\ee
Then, we have:
$$H= \left(2^{\frac{j}{2}}(N-N_\nu)\right)^pH_1$$
which together with \eqref{lucho2} yields:
\bea\lab{nice48}
&&\norm{h_{1,p,q}}_{L^1(\MM)}\\
\nn&\les& \left(\sup_{\o\in\textrm{supp}(\eta_j^{\nu})}(\norm{H_1}_{L^2_uL^4_tL^2_{x'}})+|\nu-\nu'|\norm{H_1}_{L^{3_+}(\MM)}+\norm{{\chi_2}_{\nu'}H_1}_{L^2(\MM)}\right)2^{\frac{j}{2}}\ep\gamma_j^\nu,
\eea
where we used the estimate \eqref{estNomega} for $\po N$, and the size of the patch.

Next, we estimate the various terms in the right-hand side of \eqref{nice48} starting with the last one. In view of \eqref{nice47}, we have:
\bea
\nn\norm{{\chi_2}_{\nu'}H_1}_{L^2(\MM)}&\les& \norm{{\chi_2}_{\nu'}}_{L^\infty_{u_{\nu'},{x'}_{\nu'}}L^2_t}\normm{\int_{\S}\trc'\left(2^{\frac{j}{2}}(N'-N_{\nu'})\right)^qF_{j,-1}(u')\eta_j^{\nu'}(\o')d\o'}_{L^2_{u_{\nu'},{x'}_{\nu'}}L^\infty_t}\\
\lab{nice49}&\les & \ep(1+q^2) \gamma^{\nu'}_j,
\eea
where we used in the last inequality the estimate \eqref{loebter} of the $L^2_{u_\nu,{x'}_\nu}L^\infty_t$ of oscillatory integrals together with the estimate \eqref{dechch1} for $\chi_2$.

Next, we estimate the second term in the right-hand side of \eqref{nice48}. In view of the definition \eqref{nice47} of $H_1$, and in view of the estimates \eqref{nice50} and \eqref{nice51}, we have:
$$\norm{H_1}_{L^2(\MM)}\les  \ep \gamma^{\nu'}_j,$$
and 
$$\norm{H_1}_{L^\infty(\MM)}\les  2^j\ep \gamma^{\nu'}_j.$$
Interpolating between these two estimates, we obtain:
\be\lab{nice52}
\norm{H_1}_{L^{3_+}(\MM)}\les  2^{(\frac{1}{3})_+j}\ep \gamma^{\nu'}_j.
\ee

Next, we estimate the first term in the right-hand side of \eqref{nice48}. The estimate \eqref{koko1} applied to $H$ yields:
$$\sup_{\o\in\textrm{supp}(\eta_j^{\nu})}(\norm{H_1}_{L^2_uL^\infty_tL^2_{x'}})\les  (1+q^{\frac{5}{2}})\ep 2^{\frac{j}{2}}|\nu-\nu'|\gamma^{\nu'}_j.$$
Interpolating with \eqref{nice50}, we obtain:
\be\lab{nice53}
\sup_{\o\in\textrm{supp}(\eta_j^{\nu})}(\norm{H_1}_{L^2_uL^4_tL^2_{x'}})\les  (1+q^2)\ep (2^{\frac{j}{2}}|\nu-\nu'|)^{\frac{1}{2}}\gamma^{\nu'}_j.
\ee
Finally, \eqref{nice48}, \eqref{nice49}, \eqref{nice52} and \eqref{nice53} imply:
\be\lab{nice54}
\norm{h_{1,p,q}}_{L^1(\MM)}\les (1+q^2)\Big(2^{\frac{j}{2}}|\nu-\nu'|2^{-(\frac{1}{6})_-j}+(2^{\frac{j}{2}}|\nu-\nu'|)^{\frac{1}{2}}\Big)2^{\frac{j}{2}}\ep^2\gamma_j^\nu\gamma^j_{\nu'}.
\ee

Next, we estimate $h_{2,p,q}$ defined in \eqref{app3}. We have:
\be\lab{nice55}
h_{2,p,q}=  \int_{\S} H\nabb L(\trc)F_{j,-1}(u)\eta_j^\nu(\o)d\o,
\ee
where the tensor $H$ is given by:
\be\lab{nice56}
H= \left(2^{\frac{j}{2}}(N-N_\nu)\right)^p\left(\int_{\S}\trc'\left(2^{\frac{j}{2}}(N'-N_{\nu'})\right)^qF_{j,-1}(u')\eta_j^{\nu'}(\o')d\o'\right).
\ee
In view of \eqref{nice55}, the estimate \eqref{nadalbis} in $L^1(\MM)$ yields:
\bea\lab{nice57}
&&\norm{h_{2,p,q}}_{L^1(\MM)}\\
\nn&\les& \left(\sup_{\o\in\textrm{supp}(\eta_j^{\nu})}(\norm{H}_{L^2_uL^4_tL^2_{x'}})+|\nu-\nu'|\norm{H}_{L^{3_+}(\MM)}+\norm{{\chi_2}_{\nu'}H}_{L^2(\MM)}\right)2^{\frac{j}{2}}\ep\gamma_j^\nu.
\eea
In view, of \eqref{lucho1}, \eqref{lucho2}, \eqref{nice56} and \eqref{nice57}, $h_{2,p,q}$ satisfies the same estimate as $h_{1,p,q}$. Thus, we have in view of \eqref{nice54}: 
\be\lab{nice65}
\norm{h_{2,p,q}}_{L^1(\MM)}\les (1+q^2)\Big(2^{\frac{j}{2}}|\nu-\nu'|2^{-(\frac{1}{6})_-j}+(2^{\frac{j}{2}}|\nu-\nu'|)^{\frac{1}{2}}\Big)2^{\frac{j}{2}}\ep^2\gamma_j^\nu\gamma^j_{\nu'}.
\ee

Next, we estimate $h_{3,p,q}$. In view of the Raychaudhuri equation \eqref{raychaudhuri} satisfied by $\trc$, the decomposition \eqref{dectrcom} for $\trc$, the decomposition \eqref{dechch2om} for $|\hch|^2$, and with the $L^\infty$ estimates for $b$ and $\trc$ provided respectively by \eqref{estb} and \eqref{esttrc}, we obtain the following decomposition for $L(\trc)$: 
\be\lab{nice44}
L(\trc)={\chi_2}_\nu\c (2\chi_1+\hch)+f^j_1+f^j_2,
\ee
where the scalar $f^j_1$ only depends on $\nu$ and satisfies:
\be\lab{nice45}
\norm{f^j_1}_{L^\infty_{u_\nu}L^2_t L^\infty(P_{t,u_\nu})}\les \ep,
\ee
where the scalar $f^j_2$ satisfies:
\be\lab{nice46}
\norm{f^j_2}_{L^\infty_u\lh{2}}\les \ep 2^{-\frac{j}{2}}.
\ee
This implies the following decomposition:
\bea\lab{bbbbbb}
&&\int_{\S} L(\trc)\left(2^{\frac{j}{2}}(N-N_\nu)\right)^pF_{j,-1}(u)\eta_j^\nu(\o)d\o \\
\nn&= & -{\chi_2}_\nu\c\left(\int_{\S} (2\chi_1+\hch)\left(2^{\frac{j}{2}}(N-N_\nu)\right)^pF_{j,-1}(u)\eta_j^\nu(\o)d\o\right)\\
\nn&& +f^j_1\left(\int_{\S} \left(2^{\frac{j}{2}}(N-N_\nu)\right)^pF_{j,-1}(u)\eta_j^\nu(\o)d\o\right)\\
\nn&&+\int_{\S} f^j_2 \left(2^{\frac{j}{2}}(N-N_\nu)\right)^pF_{j,-1}(u)\eta_j^\nu(\o)d\o\\
\nn&= & -\chi_2'\c\left(\int_{\S} (2\chi_1+\hch)\left(2^{\frac{j}{2}}(N-N_\nu)\right)^pF_{j,-1}(u)\eta_j^\nu(\o)d\o\right)\\
\nn&& -({\chi_2}_\nu-\chi_2')\c\left(\int_{\S} (2\chi_1+\hch)\left(2^{\frac{j}{2}}(N-N_\nu)\right)^pF_{j,-1}(u)\eta_j^\nu(\o)d\o\right)\\
\nn&& +f^j_1\left(\int_{\S} \left(2^{\frac{j}{2}}(N-N_\nu)\right)^pF_{j,-1}(u)\eta_j^\nu(\o)d\o\right)\\
\nn&&+\int_{\S} f^j_2 \left(2^{\frac{j}{2}}(N-N_\nu)\right)^pF_{j,-1}(u)\eta_j^\nu(\o)d\o
\eea
We obtain the following estimate for $h_{3,p,q}$:
\bee
&&\norm{h_{3,p,q}}_{L^1(\MM)}\\
\nn&\les & \normm{\int_{\S} (2\chi_1+\hch)\left(2^{\frac{j}{2}}(N-N_\nu)\right)^pF_{j,-1}(u)\eta_j^\nu(\o)d\o}_{L^2(\MM)}\\
\nn&&\times\normm{\int_{\S}\chi_2'H_1\left(2^{\frac{j}{2}}(N'-N_{\nu'})\right)^qF_{j,-1}(u')\eta_j^{\nu'}(\o')d\o'}_{L^2(\MM)}\\
\nn&&+\int_{\S}\norm{{\chi_2}_\nu-\chi_2'}_{L^{6_-}(\MM)}\normm{\int_{\S} (2\chi_1+\hch)\left(2^{\frac{j}{2}}(N-N_\nu)\right)^pF_{j,-1}(u)\eta_j^\nu(\o)d\o}_{L^{3_+}(\MM)}\\
\nn&&\times\normm{H_1\left(2^{\frac{j}{2}}(N'-N_{\nu'})\right)^qF_{j,-1}(u')}_{L^2(\MM)}\eta_j^{\nu'}(\o')d\o'\\
\nn&&+\Bigg(\norm{f^j_1}^2_{L^\infty_{u_\nu,{x'}_\nu}L^2_t}\normm{\int_{\S} \left(2^{\frac{j}{2}}(N-N_\nu)\right)^pF_{j,-1}(u)\eta_j^\nu(\o)d\o}_{L^2_{u_\nu,{x'}_\nu}L^\infty_t}\\
\nn&&+\normm{\int_{\S} f^j_2 \left(2^{\frac{j}{2}}(N-N_\nu)\right)^pF_{j,-1}(u)\eta_j^\nu(\o)d\o}_{L^2(\MM)}\Bigg)\\
\nn&&\times\normm{\int_{\S}H_1\left(2^{\frac{j}{2}}(N'-N_{\nu'})\right)^qF_{j,-1}(u')\eta_j^{\nu'}(\o')d\o'}_{L^2(\MM)}
\eee
which together with the estimate \eqref{nice45} for $f^j_1$, the estimate in $L^2(\MM)$ \eqref{nice49}, the estimates \eqref{dechch1} and \eqref{dechch2} for $\chi_2$, and the estimate \eqref{loebbis} of the $L^2_{u_\nu,{x'}_\nu}L^\infty_t$ of oscillatory integrals yields:
\bea
\lab{nice66}&&\norm{h_{3,p,q}}_{L^1(\MM)}\\
\nn&\les & \normm{\int_{\S} (2\chi_1+\hch)\left(2^{\frac{j}{2}}(N-N_\nu)\right)^pF_{j,-1}(u)\eta_j^\nu(\o)d\o}_{L^2(\MM)}\\
\nn&&\times\normm{\int_{\S}\chi_2'H_1\left(2^{\frac{j}{2}}(N'-N_{\nu'})\right)^qF_{j,-1}(u')\eta_j^{\nu'}(\o')d\o'}_{L^2(\MM)}\\
\nn&&+\ep|\nu-\nu'|\Bigg(\int_{\S}\normm{\int_{\S} (2\chi_1+\hch)\left(2^{\frac{j}{2}}(N-N_\nu)\right)^pF_{j,-1}(u)\eta_j^\nu(\o)d\o}_{L^{3_+}(\MM)}\\
\nn&&\times\norm{H_1}_{\lprime{\infty}{2}}\normm{\left(2^{\frac{j}{2}}(N'-N_{\nu'})\right)^q}_{L^\infty}\norm{F_{j,-1}(u')}_{L^2_{u'}}\eta_j^{\nu'}(\o')d\o'\Bigg)\\
\nn&&+\ep\gamma^j_\nu\normm{\int_{\S}H_1\left(2^{\frac{j}{2}}(N'-N_{\nu'})\right)^qF_{j,-1}(u')\eta_j^{\nu'}(\o')d\o'}_{L^2(\MM)}
\eea

Next, we estimate the $L^{3_+}(\MM)$ norm in the right-hand side of \eqref{nice66}. Using the estimate for the $L^p(\MM)$ norm \eqref{osclpbis} with $p=6$, we have:
\bea\lab{cannes51}
&&\normm{\int_{\S} (2\chi_1+\hch)\left(2^{\frac{j}{2}}(N-N_\nu)\right)^pF_j(u)\eta_j^\nu(\o)d\o}_{L^6(\MM)}\\
\nn&\les& \sup_\o\left((\norm{\chi_1}_{\li{\infty}{6}}+\norm{\hch}_{\li{\infty}{6}})\normm{\left(2^{\frac{j}{2}}(N-N_\nu)\right)^p}_{L^\infty}\right)2^{\frac{5j}{6}}\gamma^\nu_j\\
\nn&\les& 2^{\frac{5j}{6}}\ep\gamma^\nu_j,
\eea
where we used in the last inequality the estimate \eqref{esthch} for $\hch$, the estimate \eqref{estNomega} for $\po N$, and the estimate \eqref{dechch1} for $\chi_1$. Next, recall the  decomposition \eqref{dechchom} for $\hch$ and the decomposition \eqref{decchi2om} for $\chi_2$ which yield:
$$2\chi_1+\hch=F^j_1+F^j_2$$
where the tensor $F^j_1$ only depends on $\nu$ and satisfies:
\be\lab{cannes52}
\norm{F^j_1}_{L^\infty_{u_\nu,{x'}_\nu}L^2_t}\les \ep,
\ee
where the scalar $F^j_2$ satisfies:
\be\lab{cannes53}
\norm{F^j_2}_{\li{\infty}{2}}\les \ep 2^{-\frac{j}{2}}.
\ee
This yields:
\bee
&&\int_{\S} (2\chi_1+\hch)\left(2^{\frac{j}{2}}(N-N_\nu)\right)^pF_j(u)\eta_j^\nu(\o)d\o\\
&=&F^j_1\left(\int_{\S}\left(2^{\frac{j}{2}}(N-N_\nu)\right)^pF_j(u)\eta_j^\nu(\o)d\o\right)
+\int_{\S} F^j_2\left(2^{\frac{j}{2}}(N-N_\nu)\right)^pF_j(u)\eta_j^\nu(\o)d\o
\eee
and thus:
\bee
&&\normm{\int_{\S} (2\chi_1+\hch)\left(2^{\frac{j}{2}}(N-N_\nu)\right)^pF_j(u)\eta_j^\nu(\o)d\o}_{L^2(\MM)}\\
&\les& \norm{F^j_1}_{L^\infty_{u_\nu,{x'}_\nu}L^2_t}\normm{\int_{\S} \left(2^{\frac{j}{2}}(N-N_\nu)\right)^pF_j(u)\eta_j^\nu(\o)d\o}_{L^2_{u_\nu,{x'}_\nu}L^\infty_t}\\
&&+\normm{\int_{\S} F^j_2 \left(2^{\frac{j}{2}}(N-N_\nu)\right)^pF_j(u)\eta_j^\nu(\o)d\o}_{L^2(\MM)}\\
&\les& \ep\gamma^\nu_j+\normm{\int_{\S} F^j_2 \left(2^{\frac{j}{2}}(N-N_\nu)\right)^pF_j(u)\eta_j^\nu(\o)d\o}_{L^2(\MM)}
\eee
where we used in the last inequality the estimate \eqref{cannes52} and the estimate \eqref{loebbis} of the $L^2_{u_\nu,{x'}_\nu}L^\infty_t$ of oscillatory integrals. Then, using the basic estimate in $L^2(\MM)$ \eqref{oscl2bis}, we obtain:
\bea\lab{succulent}
&&\normm{\int_{\S} (2\chi_1+\hch)b^{-1}\left(2^{\frac{j}{2}}(N-N_\nu)\right)^pF_j(u)\eta_j^\nu(\o)d\o}_{L^2(\MM)}\\
\nn&\les& \ep\gamma^\nu_j+\sup_\o\left(\norm{F^j_2}_{\li{\infty}{2}}\normm{b^{-1}\left(2^{\frac{j}{2}}(N-N_\nu)\right)^p}_{L^\infty}\right)2^{\frac{j}{2}}\gamma^\nu_j\\
\nn&\les& \ep\gamma^\nu_j,
\eea
where we used in the last inequality the estimate \eqref{cannes53}. Next, interpolating \eqref{cannes51} and \eqref{succulent}, we obtain:
$$\normm{\int_{\S} (2\chi_1+\hch)b^{-1}\left(2^{\frac{j}{2}}(N-N_\nu)\right)^pF_{j,-1}(u)\eta_j^\nu(\o)d\o}_{L^{3_+}(\MM)}\les 2^{(\frac{5}{12})_+j}\ep\gamma^\nu_j.$$
Together with \eqref{nice66}, the estimate \eqref{estNomega} for $\po N$, the size of the patch, and the estimate in $L^2(\MM)$ \eqref{succulent}, we obtain:
\bea
\lab{nice67}\norm{h_{3,p,q}}_{L^1(\MM)}&\les&  \ep\gamma^\nu_j\normm{\int_{\S}\chi_2'H_1\left(2^{\frac{j}{2}}(N'-N_{\nu'})\right)^qF_{j,-1}(u')\eta_j^{\nu'}(\o')d\o'}_{L^2(\MM)}\\
\nn&&+\ep|\nu-\nu'|2^{(\frac{5}{12})_+j}\gamma^\nu_j\Bigg(\int_{\S}\norm{H_1}_{\lprime{\infty}{2}}\norm{F_{j,-1}(u')}_{L^2_{u'}}\eta_j^{\nu'}(\o')d\o'\Bigg)\\
\nn&&+\ep\gamma^j_\nu\normm{\int_{\S}H_1\left(2^{\frac{j}{2}}(N'-N_{\nu'})\right)^qF_{j,-1}(u')\eta_j^{\nu'}(\o')d\o'}_{L^2(\MM)}.
\eea
Next, using the definition \eqref{app7} of $H_1$, the estimate \eqref{estb} for $b$, the estimate \eqref{esttrc} and \eqref{esthch} for $\chi$, and the estimate \eqref{dechch1} for $\chi_2$, we have:
\be\lab{nice68}
\norm{H_1}_{\lprime{\infty}{2}}+\norm{\chi_2'H_1}_{\lprime{\infty}{2}}\les \ep.
\ee
Using the basic estimate in $L^2(\MM)$ \eqref{oscl2bis}, we have:
\bee
&&\normm{\int_{\S}\chi_2'H_1\left(2^{\frac{j}{2}}(N'-N_{\nu'})\right)^qF_{j,-1}(u')\eta_j^{\nu'}(\o')d\o'}_{L^2(\MM)}\\
&&+\normm{\int_{\S}H_1\left(2^{\frac{j}{2}}(N'-N_{\nu'})\right)^qF_{j,-1}(u')\eta_j^{\nu'}(\o')d\o'}_{L^2(\MM)}\\
&\les& \sup_\o\left(\Big(\norm{\chi_2'H_1}_{\lprime{\infty}{2}}+\norm{H_1}_{\lprime{\infty}{2}}\Big)\normm{\left(2^{\frac{j}{2}}(N'-N_{\nu'})\right)^q}_{L^\infty}\right)2^{\frac{j}{2}}\gamma^\nu_j\\
&\les& 2^{\frac{j}{2}}\ep\gamma^\nu_j\\
\eee
where we used in the last inequality the estimate \eqref{nice68}, the estimate \eqref{estNomega} for $\po N$ and the size of the patch. Together with \eqref{nice67} and \eqref{nice68}, this yields:
\bea
\lab{nice69}\norm{h_{3,p,q}}_{L^1(\MM)}&\les&  2^{\frac{j}{2}}\ep^2\gamma^\nu_j\gamma^j_{\nu'}+\ep^2|\nu-\nu'|2^{(\frac{5}{12})_+j}\gamma^\nu_j\left(\int_{\S}\norm{F_{j,-1}(u')}_{L^2_{u'}}\eta_j^{\nu'}(\o')d\o'\right)\\
\nn&\les& 2^{\frac{j}{2}}\left(1+2^{\frac{j}{2}}|\nu-\nu'| 2^{-(\frac{1}{12})_-j}\right)\ep^2\gamma^\nu_j\gamma^j_{\nu'},
\eea
where we used in the last inequality Cauchy-Schwarz in $\la'$ to evaluate $\norm{F_{j,-1}(u')}_{L^2_{u'}}$, Cauchy-Schwarz in $\o'$, and the size of the patch.

Next, we estimate $h_{4,p,q}$. We have:
\bea\lab{nice70}
\norm{h_{4,p,q}}_{L^1(\MM)}&\les& \normm{\int_{\S} H_2\left(2^{\frac{j}{2}}(N-N_\nu)\right)^pF_{j,-1}(u)\eta_j^\nu(\o)d\o}_{L^2(\MM)}\\
\nn&&\times\normm{\int_{\S}\trc'\left(2^{\frac{j}{2}}(N'-N_{\nu'})\right)^qF_{j,-1}(u')\eta_j^{\nu'}(\o')d\o'}_{L^2(\MM)},
\eea
The basic estimate in $L^2(\MM)$ \eqref{oscl2bis} yields:
\bea\lab{nice71}
&&\normm{\int_{\S} H_2\left(2^{\frac{j}{2}}(N-N_\nu)\right)^pF_{j,-1}(u)\eta_j^\nu(\o)d\o}_{L^2(\MM)}\\
\nn&\les&\sup_\o\left(\norm{H_2}_{\li{\infty}{2}}\normm{\left(2^{\frac{j}{2}}(N-N_\nu)\right)^p}_{L^\infty}\right)2^{\frac{j}{2}}\gamma^\nu_j\\
\nn&\les&\sup_\o\left(\norm{H_2}_{\li{\infty}{2}}\right)2^{\frac{j}{2}}\gamma^\nu_j,
\eea
where we used in the last inequality the estimate \eqref{estNomega} for $\po N$ and the size of the patch. In view of \eqref{app8}, the estimate \eqref{estb} for $b$, the estimates \eqref{estn} \eqref{estk} for $\db$, the estimates \eqref{esttrc} \eqref{esthch} for $\chi$, and the Raychaudhuri equation \eqref{raychaudhuri} satisfied by $\trc$, we have:
$$\norm{H_2}_{\li{\infty}{2}}\les \ep,$$
which together with \eqref{nice71} yields:
$$\normm{\int_{\S} H_2\left(2^{\frac{j}{2}}(N-N_\nu)\right)^pF_{j,-1}(u)\eta_j^\nu(\o)d\o}_{L^2(\MM)}\les \ep 2^{\frac{j}{2}}\gamma^\nu_j.$$
Together with \eqref{nice70} and the estimate \eqref{loebbis} of the $L^2_{u_\nu,{x'}_\nu}L^\infty_t$ of oscillatory integrals, we obtain:
\be\lab{nice72}
\norm{h_{4,p,q}}_{L^1(\MM)}\les (1+q^2) 2^{\frac{j}{2}}\ep^2\gamma^\nu_j\gamma^j_{\nu'}.
\ee

Finally, in view of \eqref{app1}, \eqref{nice54}, \eqref{nice65}, \eqref{nice69} and \eqref{nice72}, we obtain for $B^{1,1,1,1}_{j,\nu,\nu'}$ the following decomposition:
\bea\lab{nice73}
&&B^{1,1,1,1}_{j,\nu,\nu'}\\
\nn&=& 2^{-2j}\sum_{p, q\geq 0}c_{pq}\int_{\MM}\frac{1}{(2^{\frac{j}{2}}|N_\nu-N_{\nu'}|)^{p+q}}\\
\nn&&\times\left[\frac{1}{|N_\nu-N_{\nu'}|^2}h_{1,p,q}+\frac{1}{|N_\nu-N_{\nu'}|^3}(h_{2,p,q}+h_{3,p,q}+h_{4,p,q})\right] d\MM\\
\nn&&+2^{-2j}\int_{\MM}\int_{\S\times\S} \frac{(\chi-\chi')L(\trc)\trc'}{\gg(L,L')^2}F_{j,-1}(u)F_{j,-1}(u')\eta_j^\nu(\o)\eta_j^{\nu'}(\o')d\o d\o' d\MM,
\eea
where $c_{pq}$ are explicit real coefficients such that the series 
$$\sum_{p, q\geq 0}c_{pq}x^py^q$$
has radius of convergence 1, and where $h_{1,p,q}$, $h_{2,p,q}$, $h_{3,p,q}$ and $h_{4,p,q}$ satisfy the following estimate:
\bea\lab{nice74}
&&\norm{h_{1,p,q}}_{L^1(\MM)}+\norm{h_{2,p,q}}_{L^1(\MM)}+\norm{h_{3,p,q}}_{L^1(\MM)}+\norm{h_{4,p,q}}_{L^1(\MM)}\\
\nn&\les& (1+q^2)\Big(1+2^{\frac{j}{2}}|\nu-\nu'|2^{-(\frac{1}{12})_-j}+(2^{\frac{j}{2}}|\nu-\nu'|)^{\frac{1}{2}}\Big)2^{\frac{j}{2}}\ep^2\gamma_j^\nu\gamma^j_{\nu'}.
\eea

The term $B^{1,1,1,2}_{j,\nu,\nu'}$ defined by \eqref{nice43} is estimated in the same way. Indeed, 
proceeding as for $B^{1,1,1,1}_{j,\nu,\nu'}$, we integrate by parts in $B^{1,1,1,2}_{j,\nu,\nu',l,m}$ using \eqref{fete1}. 
\begin{lemma}\lab{lemma:app1:2}
Let $B^{1,1,1,2}_{j,\nu,\nu'}$ be defined by \eqref{nice43}. Integrating by parts using \eqref{fete1} yields:
\bea\lab{app1:2}
&&B^{1,1,1,2}_{j,\nu,\nu'}\\
\nn&=& 2^{-2j}\sum_{p, q\geq 0}c_{pq}\int_{\MM}\frac{1}{(2^{\frac{j}{2}}|N_\nu-N_{\nu'}|)^{p+q}}\\
\nn&&\times\left[\frac{1}{|N_\nu-N_{\nu'}|^2}h_{1,p,q}'+\frac{1}{|N_\nu-N_{\nu'}|^3}(h_{2,p,q}'+h_{3,p,q}'+h_{4,p,q}')\right] d\MM\\
\nn&&+2^{-2j}\int_{\MM}\int_{\S\times\S} \frac{(\chi-\chi')\trc L'(\trc')}{\gg(L,L')^2}F_{j,-1}(u)F_{j,-1}(u')\eta_j^\nu(\o)\eta_j^{\nu'}(\o')d\o d\o' d\MM,
\eea
where $c_{pq}$ are explicit real coefficients such that the series 
$$\sum_{p, q\geq 0}c_{pq}x^py^q$$
has radius of convergence 1, where the scalar functions $h_{1,p,q}', h_{2,p,q}', h_{3,p,q}', h_{4,p,q}'$ on $\MM$ are given by:
\bea\lab{app2:2}
h_{1,p,q}'&=& \left(\int_{\S} N(\trc)\left(2^{\frac{j}{2}}(N-N_\nu)\right)^pF_{j,-1}(u)\eta_j^\nu(\o)d\o\right)\\
\nn&&\times\left(\int_{\S}L'(\trc')\left(2^{\frac{j}{2}}(N'-N_{\nu'})\right)^qF_{j,-1}(u')\eta_j^{\nu'}(\o')d\o'\right),
\eea
\bea\lab{app3:2}
h_{2,p,q}'&=& \left(\int_{\S} \trc\left(2^{\frac{j}{2}}(N-N_\nu)\right)^pF_{j,-1}(u)\eta_j^\nu(\o)d\o\right)\\
\nn&&\times\left(\int_{\S}\nabb' L'(\trc')\left(2^{\frac{j}{2}}(N'-N_{\nu'})\right)^qF_{j,-1}(u')\eta_j^{\nu'}(\o')d\o'\right),
\eea
\bea\lab{app4:2}
h_{3,p,q}'&=& \left(\int_{\S} H_1'\left(2^{\frac{j}{2}}(N-N_\nu)\right)^pF_{j,-1}(u)\eta_j^\nu(\o)d\o\right)\\
\nn&&\times\left(\int_{\S}L'(\trc')\left(2^{\frac{j}{2}}(N'-N_{\nu'})\right)^qF_{j,-1}(u')\eta_j^{\nu'}(\o')d\o'\right),
\eea
\bea\lab{app5:2}
h_{4,p,q}'&=& \left(\int_{\S} \trc\left(2^{\frac{j}{2}}(N-N_\nu)\right)^pF_{j,-1}(u)\eta_j^\nu(\o)d\o\right)\\
\nn&&\times\left(\int_{\S}H_2'\left(2^{\frac{j}{2}}(N'-N_{\nu'})\right)^qF_{j,-1}(u')\eta_j^{\nu'}(\o')d\o'\right),
\eea
where the tensor $H_1'$ on $\MM$ involved in the definition of $h_{3,p,q}'$ is a linear combination of terms in the following list:
\be\lab{app7:2}
b^{-1}\nabb(b\trc),\, \th\trc,
\ee
and where the tensor $H_2'$ on $\MM$ involved in the definition of $h_{4,p,q}$ is a linear combination of terms in the following list:
\be\lab{app8:2}
\th' L'(\trc'),\, {b'}^{-1}\nabb'(b') L'(\trc').
\ee
\end{lemma}
The proof of lemma \ref{lemma:app1:2} is postponed to Appendix B. Next, we use Lemma \ref{lemma:app1:2} to obtain the control of $B^{1,1,1,2}_{j,\nu,\nu'}$. Now, note that exchanging the role of $\o$ and $\o'$, we obtain that $h_{2,p,q}'$ corresponds to $h_{2,p,q}$, $h_{3,p,q}'$ corresponds to $h_{3,p,q}$, and $h_{4,p,q}'$ corresponds to $h_{4,p,q}$. Also, exchanging the role of $\o$ and $\o'$, we obtain that $h_{1,p,q}'$ corresponds to $h_{3,p,q}$ where $H_1$ has been replaced with $N(\trc)$ which satisfies \eqref{nice68} in view of the estimate \eqref{esttrc} for $\trc$, and the estimate \eqref{dechch1} for $\chi_2$. Thus, since $h_{1,p,q}$, $h_{2,p,q}$, $h_{3,p,q}$ and $h_{4,p,q}$ satisfy the estimate \eqref{nice74}, we obtain that $h_{1,p,q}'$, $h_{2,p,q}'$, $h_{3,p,q}'$ and $h_{4,p,q}'$ satisfy the following estimate:
\bea\lab{nice76}
&&\norm{h_{1,p,q}'}_{L^1(\MM)}+\norm{h_{2,p,q}'}_{L^1(\MM)}+\norm{h_{3,p,q}'}_{L^1(\MM)}+\norm{h_{4,p,q}'}_{L^1(\MM)}\\
\nn&\les& (1+q^2)\Big(1+2^{\frac{j}{2}}|\nu-\nu'|2^{-(\frac{1}{12})_-j}+(2^{\frac{j}{2}}|\nu-\nu'|)^{\frac{1}{2}}\Big)2^{\frac{j}{2}}\ep^2\gamma_j^\nu\gamma^j_{\nu'}.
\eea
Summing \eqref{nice73} and \eqref{app1:2}, we obtain:
\bea\lab{nice77}
&&B^{1,1,1,1}_{j,\nu,\nu'}+B^{1,1,1,2}_{j,\nu,\nu'}\\
\nn&=& 2^{-2j}\sum_{p, q\geq 0}c_{pq}\int_{\MM}\frac{1}{(2^{\frac{j}{2}}|N_\nu-N_{\nu'}|)^{p+q}}\bigg[\frac{1}{|N_\nu-N_{\nu'}|^2}(h_{1,p,q}+h_{1,p,q}')\\
\nn&&+\frac{1}{|N_\nu-N_{\nu'}|^3}(h_{2,p,q}+h_{2,p,q}'+h_{3,p,q}+h_{3,p,q}'+h_{4,p,q}+h_{4,p,q}')\bigg] d\MM\\
\nn&&+2^{-2j}\int_{\MM}\int_{\S\times\S} \frac{(\chi-\chi')(L(\trc)\trc'+\trc L'(\trc'))}{\gg(L,L')^2}\\
\nn&&\times F_{j,-1}(u)F_{j,-1}(u')\eta_j^\nu(\o)\eta_j^{\nu'}(\o')d\o d\o' d\MM.
\eea
Note that the last term in the right-hand side is antisymmetric in $(\nu, \nu')$ and thus vanishes when considering the sum:
\be\lab{nice78}
B^{1,1,1,1}_{j,\nu,\nu'}+B^{1,1,1,2}_{j,\nu,\nu'}+B^{1,1,1,1}_{j,\nu',\nu}+B^{1,1,1,2}_{j,\nu',\nu}.
\ee
This cancellation together with \eqref{nice77} yields:
\bee
&&|B^{1,1,1,1}_{j,\nu,\nu'}+B^{1,1,1,2}_{j,\nu,\nu'}+B^{1,1,1,1}_{j,\nu',\nu}+B^{1,1,1,2}_{j,\nu',\nu}|\\
\nn&\les&  2^{-2j}\sum_{p, q\geq 0}c_{pq}\normm{\frac{1}{(2^{\frac{j}{2}}|N_\nu-N_{\nu'}|)^{p+q}}}_{L^\infty(\MM)}\Bigg[\normm{\frac{1}{|N_\nu-N_{\nu'}|^2}}_{L^\infty(\MM)}\\
\nn&&\times(\norm{h_{1,p,q}}_{L^1(\MM)}+\norm{h_{1,p,q}'}_{L^1(\MM)})+\normm{\frac{1}{|N_\nu-N_{\nu'}|^3}}_{L^\infty(\MM)}(\norm{h_{2,p,q}}_{L^1(\MM)}+\norm{h_{2,p,q}'}_{L^1(\MM)}\\
\nn&&+\norm{h_{3,p,q}}_{L^1(\MM)}+\norm{h_{3,p,q}'}_{L^1(\MM)}+\norm{h_{4,p,q}}_{L^1(\MM)}+\norm{h_{4,p,q}'}_{L^1(\MM)})\Bigg]. 
\eee
Together with the estimate \eqref{nice26} for $N_\nu-N_{\nu'}$, and the estimates \eqref{nice74} and \eqref{nice76}, 
we obtain:
\bee
&&|B^{1,1,1,1}_{j,\nu,\nu'}+B^{1,1,1,2}_{j,\nu,\nu'}+B^{1,1,1,1}_{j,\nu',\nu}+B^{1,1,1,2}_{j,\nu',\nu}|\\
\nn&\les&  \left(\sum_{p, q\geq 0}c_{pq}\frac{1+q^2}{(2^{\frac{j}{2}}|\nu-\nu'|)^{p+q}}\right)\left[\frac{2^{-j}}{(2^{\frac{j}{2}}|\nu-\nu'|)^2}+\frac{2^{-\frac{j}{2}}}{(2^{\frac{j}{2}}|\nu-\nu'|)^3}\right]\\
\nn&&\times\Big(1+2^{\frac{j}{2}}|\nu-\nu'|2^{-(\frac{1}{12})_-j}+(2^{\frac{j}{2}}|\nu-\nu'|)^{\frac{1}{2}}\Big)2^{\frac{j}{2}}\ep^2\gamma_j^\nu\gamma^j_{\nu'}\\
\nn&\les& \left[\frac{2^{-(\frac{1}{12})_-j}}{(2^{\frac{j}{2}}|\nu-\nu'|)^2}+\frac{1}{(2^{\frac{j}{2}}|\nu-\nu'|)^{\frac{5}{2}}}\right]\ep^2\gamma_j^\nu\gamma^j_{\nu'}. 
\eee
Since we have:
$$B^{1,1,1}_{j,\nu,\nu'}=B^{1,1,1,1}_{j,\nu,\nu'}+B^{1,1,1,2}_{j,\nu,\nu'}$$
in view of the decomposition \eqref{nice41}, this yields:
\bea
\lab{nice79}&&|B^{1,1,1}_{j,\nu,\nu'}+B^{1,1,1}_{j,\nu',\nu}|\\
\nn&\les&  \left(\sum_{p, q\geq 0}c_{pq}\frac{1+q^2}{(2^{\frac{j}{2}}|\nu-\nu'|)^{p+q}}\right)\left[\frac{2^{-j}}{(2^{\frac{j}{2}}|\nu-\nu'|)^2}+\frac{2^{-\frac{j}{2}}}{(2^{\frac{j}{2}}|\nu-\nu'|)^3}\right]\\
\nn&&\times\Big(1+2^{\frac{j}{2}}|\nu-\nu'|2^{-(\frac{1}{12})_-j}+(2^{\frac{j}{2}}|\nu-\nu'|)^{\frac{1}{2}}\Big)2^{\frac{j}{2}}\ep^2\gamma_j^\nu\gamma^j_{\nu'}\\
\nn&\les& \left[\frac{2^{-(\frac{1}{12})_-j}}{(2^{\frac{j}{2}}|\nu-\nu'|)^2}+\frac{1}{(2^{\frac{j}{2}}|\nu-\nu'|)^{\frac{5}{2}}}\right]\ep^2\gamma_j^\nu\gamma^j_{\nu'}. 
\eea

\begin{remark}
The cancellation of the last term of \eqref{nice77} when considering the sum \eqref{nice78} in view of the antisymmetry in $(\nu, \nu')$ is crucial. Indeed, we would not be able to estimate this term directly.
\end{remark}

Note that exchanging the role of $\o$ and $\o'$, we obtain that the term $B^{1,1,2}_{j,\nu,\nu',l,m}$ corresponds to $B^{1,1,1}_{j,\nu,\nu',l,m}$. Thus, we obtain in view of \eqref{nice79}:
\bee
&&|B^{1,1,2}_{j,\nu,\nu'}+B^{1,1,2}_{j,\nu',\nu}|\\
\nn&\les&  \left(\sum_{p, q\geq 0}c_{pq}\frac{1+q^2}{(2^{\frac{j}{2}}|\nu-\nu'|)^{p+q}}\right)\left[\frac{2^{-j}}{(2^{\frac{j}{2}}|\nu-\nu'|)^2}+\frac{2^{-\frac{j}{2}}}{(2^{\frac{j}{2}}|\nu-\nu'|)^3}\right]\\
\nn&&\times\Big(1+2^{\frac{j}{2}}|\nu-\nu'|2^{-(\frac{1}{12})_-j}+(2^{\frac{j}{2}}|\nu-\nu'|)^{\frac{1}{2}}\Big)2^{\frac{j}{2}}\ep^2\gamma_j^\nu\gamma^j_{\nu'}\\
\nn&\les& \left[\frac{2^{-(\frac{1}{12})_-j}}{(2^{\frac{j}{2}}|\nu-\nu'|)^2}+\frac{1}{(2^{\frac{j}{2}}|\nu-\nu'|)^{\frac{5}{2}}}\right]\ep^2\gamma_j^\nu\gamma^j_{\nu'}. 
\eee
Together with \eqref{nice79}, this concludes the proof of Proposition \ref{prop:labexfsmp6}.

\subsection{Proof of Proposition \ref{prop:labexfsmp4} (Control of $B^{1,2}_{j,\nu,\nu',l,m}$)}\lab{sec:labex1}

Recall from \eqref{nice14} that $B^{1,2}_{j,\nu,\nu',l,m}$ is given by:
\bea
\nn B^{1,2}_{j,\nu,\nu',l,m} &=& -i2^{-j}\int_{\MM}\int_{\S\times\S}\int_0^{\infty}\int_0^{\infty} \frac{b^{-1}}{\gg(L,L')}\bigg(L(P_l\trc)P_m\trc'+P_l\trc L'(P_m\trc')\bigg)\\
\nn&&\times \eta_j^\nu(\o)\eta_j^{\nu'}(\o')\frac{(2^{-j}\la')^{-1}-(2^{-j}\la)^{-1}}{2}\psi(2^{-j}\la)(2^{-j}\la')\psi(2^{-j}\la') f(\la\o)f(\la'\o')\\
\lab{nyc} &&\times\la^2 {\la'}^2d\la d\la'd\o d\o' d\MM.
\eea
Since $\nab u=b^{-1}N$ and $\nab u'={b'}^{-1}u'$, we have:
\bee
-iN(e^{i\la u-i\la'u'})&=& e^{i\la u-i\la'u'}\left(b^{-1}\la -{b'}^{-1}\gn \la'\right)\\
&=&e^{i\la u-i\la'u'} b^{-1}(\la-\la')+e^{i\la u-i\la'u'}(b^{-1}-{b'}^{-1}\gn)\la'.
\eee
This yields:
\bee
&&((2^{-j}\la')^{-1}-(2^{-j}\la))e^{i\la u-i\la'u'} b^{-1}\\
&=&-i\frac{2^j}{\la\la'}N(e^{i\la u-i\la'u'})+\frac{2^j}{\la}e^{i\la u-i\la'u'}(b^{-1}-{b'}^{-1}\gn)\\
&=&-i\frac{2^j}{\la\la'}N(e^{i\la u-i\la'u'})+\frac{2^j}{\la}e^{i\la u-i\la'u'}(b^{-1}-{b'}^{-1})+\frac{2^j}{\la}e^{i\la u-i\la'u'}{b'}^{-1}(1-\gn).
\eee
In view of \eqref{nyc}, this implies the following decomposition for $B^{1,2}_{j,\nu,\nu',l,m}$:
\be\lab{nyc1}
B^{1,2}_{j,\nu,\nu',l,m}=B^{1,2,1}_{j,\nu,\nu',l,m}+B^{1,2,2}_{j,\nu,\nu',l,m}+B^{1,2,3}_{j,\nu,\nu',l,m}
\ee
where $B^{1,2,1}_{j,\nu,\nu',l,m}$, $B^{1,2,2}_{j,\nu,\nu',l,m}$ and $B^{1,2,3}_{j,\nu,\nu',l,m}$ are respectively given by:
\bea\lab{nyc2}
&& B^{1,2,1}_{j,\nu,\nu',l,m}\\
\nn&=& -2^{-2j-1}\int_{\MM}\int_{\S\times\S}\int_0^{\infty}\int_0^{\infty} \frac{N(e^{i\la u-i\la'u'})}{\gg(L,L')}\bigg(L(P_l\trc)P_m\trc'+P_l\trc L'(P_m\trc')\bigg)\\
\nn&&\times \eta_j^\nu(\o)\eta_j^{\nu'}(\o')(2^{-j}\la')^{-1}(2^{-j}\la)^{-1}\psi(2^{-j}\la)(2^{-j}\la')\psi(2^{-j}\la') f(\la\o)f(\la'\o')\\
\nn &&\times\la^2 {\la'}^2d\la d\la'd\o d\o' d\MM,
\eea
\bea\lab{nyc3}
\nn B^{1,2,2}_{j,\nu,\nu',l,m} &=&  -i2^{-j-1}\int_{\MM}\int_{\S\times\S} \frac{1}{\gg(L,L')}\bigg(L(P_l\trc)P_m\trc'+P_l\trc L'(P_m\trc')\bigg)\\
&&\times (b^{-1}-{b'}^{-1}) F_{j,-1}(u)F_j(u')\eta_j^\nu(\o)\eta_j^{\nu'}(\o')d\o d\o' d\MM.
\eea
and:
\bea\lab{nyc3bis}
\nn B^{1,2,3}_{j,\nu,\nu',l,m} &=&  -i2^{-j-1}\int_{\MM}\int_{\S\times\S} \frac{1}{\gg(L,L')}\bigg(L(P_l\trc)P_m\trc'+P_l\trc L'(P_m\trc')\bigg)\\
&&\times {b'}^{-1}(1-\gn) F_{j,-1}(u)F_j(u')\eta_j^\nu(\o)\eta_j^{\nu'}(\o')d\o d\o' d\MM.
\eea

We have the following propositions:
\begin{proposition}\lab{prop:labexfsmp7}
Let $B^{1,2,1}_{j,\nu,\nu',l,m}$ be given by \eqref{nyc2}. Then, we have the following estimate:
\bea\lab{nyc80}
&&\left|\sum_{(l,m)/m<l\textrm{ and }2^m\leq 2^j|\nu-\nu'|} B^{1,2,1}_{j,\nu,\nu',l,m}+ \sum_{(l,m)/m<l\textrm{ and }2^m\leq 2^j|\nu-\nu'|}B^{1,2,1}_{j,\nu',\nu,l,m}\right|\\
\nn &\les& \frac{\ep^2\gamma^\nu_j\gamma^{\nu'}_j}{2^j |\nu-\nu'|}+\frac{j2^{-\frac{j}{12}}\ep^2\gamma^\nu_j\gamma^{\nu'}_j}{(2^{\frac{j}{2}}|\nu-\nu'|)^2}+\frac{\ep^2\gamma^\nu_j\gamma^{\nu'}_j}{(2^{\frac{j}{2}}|\nu-\nu'|)^3}.
\eea
\end{proposition}

\begin{proposition}\lab{prop:labexfsmp8}
Let $B^{1,2,2}_{j,\nu,\nu',l,m}$ be given by \eqref{nyc3}. Then, we have the following estimate:
\bea\lab{mgen50}
&&\left|\sum_{(l,m)/2^{\min(l,m)}\leq 2^j|\nu-\nu'|}(B^{1,2,2}_{j,\nu,\nu',l,m}+B^{1,2,2}_{j,\nu',\nu,l,m})\right|\\
\nn&\les& \left[2^{-j}+\frac{1}{2^{\frac{j}{2}}(2^{\frac{j}{2}}|\nu-\nu'|)}+\frac{1}{2^{\frac{j}{4}}(2^{\frac{j}{2}}|\nu-\nu'|)^{\frac{3}{2}}}+\frac{2^{-(\frac{1}{4})_-j}}{(2^{\frac{j}{2}}|\nu-\nu'|)^2}+\frac{1}{(2^{\frac{j}{2}}|\nu-\nu'|)^3}\right]\ep^2\gamma^\nu_j\gamma^{\nu'}_j.
\eea
\end{proposition}

\begin{proposition}\lab{prop:labexfsmp9}
Let $B^{1,2,3}_{j,\nu,\nu',l,m}$ be given by \eqref{nyc3bis}. Then, we have the following estimate:
\bea\lab{bizu37}
&&\left|\sum_{(l,m)/ 2^{\min(l,m)}\leq 2^j|\nu-\nu'|}B^{1,2,3}_{j,\nu,\nu',l,m}\right|\\ 
\nn&\les&\bigg[\frac{j^2 2^{-\frac{j}{2}}}{(2^{\frac{j}{2}}|\nu-\nu'|)^2}+\frac{1}{2^{\frac{j}{2}}(2^{\frac{j}{2}}|\nu-\nu'|)}+\frac{1}{2^{\frac{3j}{4}}(2^{\frac{j}{2}}|\nu-\nu'|)^{\frac{1}{2}}}+ 2^{-j}\bigg]\ep^2\gamma^\nu_j\gamma^{\nu'}_j.
\eea
\end{proposition}

In view of the decomposition \eqref{nyc1} of $B^{1,2}_{j,\nu,\nu',l,m}$, we have:
\bee
&&\left|\sum_{(l,m)/ 2^{\min(l,m)}\leq 2^j|\nu-\nu'|}B^{1,2}_{j,\nu,\nu',l,m}\right|\\ 
&\les&  \left|\sum_{(l,m)/ 2^{\min(l,m)}\leq 2^j|\nu-\nu'|}B^{1,2,1}_{j,\nu,\nu',l,m}\right|+\left|\sum_{(l,m)/ 2^{\min(l,m)}\leq 2^j|\nu-\nu'|}B^{1,2,2}_{j,\nu,\nu',l,m}\right|\\
\nn&&+\left|\sum_{(l,m)/ 2^{\min(l,m)}\leq 2^j|\nu-\nu'|}B^{1,2,3}_{j,\nu,\nu',l,m}\right|.
\eee
Together with the estimates \eqref{nyc80}, \eqref{mgen50} and \eqref{bizu37}, we obtain:
\bee
&&\left|\sum_{(l,m)/ 2^{\min(l,m)}\leq 2^j|\nu-\nu'|}(B^{1,2}_{j,\nu,\nu',l,m}+B^{1,2}_{j,\nu',\nu,l,m})\right| \\
\nn&\les&  \bigg[\frac{1}{(2^{\frac{j}{2}}|\nu-\nu'|)^3}+\frac{j 2^{-\frac{j}{12}}}{(2^{\frac{j}{2}}|\nu-\nu'|)^2}+\frac{1}{2^{\frac{j}{2}}(2^{\frac{j}{2}}|\nu-\nu'|)}+\frac{1}{2^{\frac{3j}{4}}(2^{\frac{j}{2}}|\nu-\nu'|)^{\frac{1}{2}}}+ 2^{-j}\bigg]\ep^2\gamma^\nu_j\gamma^{\nu'}_j.
\eee
This concludes the proof of Proposition \ref{prop:labexfsmp4}. 

The rest of this section is as follows. In section \ref{sec:lobotomisation2}, we give a proof of Proposition \ref{prop:labexfsmp7}, in section \ref{sec:lobotomisation3}, we give a proof of Proposition \ref{prop:labexfsmp8}, 
and in section \ref{sec:lobotomisation4}, we give a proof of Proposition \ref{prop:labexfsmp9}.

\subsubsection{Proof of Proposition \ref{prop:labexfsmp7} (Control of $B^{1,2,1}_{j,\nu,\nu',l,m}$)}\lab{sec:lobotomisation2}

Integrating by parts the $N$ derivative in \eqref{nyc2}, we obtain:
\bea\lab{nyc4}
B^{1,2,1}_{j,\nu,\nu',l,m} &=&  2^{-2j-1}\int_{\MM}\int_{\S\times\S} \frac{1}{\gg(L,L')}\Bigg(N(L(P_l\trc))P_m\trc'\\
\nn&&+P_l\trc N(L'(P_m\trc'))+L(P_l\trc)N(P_m\trc')+N(P_l\trc)L'(P_m\trc')\\
\nn&& +\bigg(-\frac{N(\gl)}{\gl}+\trt\bigg)\Big(L(P_l\trc)P_m\trc'+P_l\trc L'(P_m\trc')\Big)\Bigg)\\
\nn&&\times  F_{j,-1}(u)F_{j,-1}(u')\eta_j^\nu(\o)\eta_j^{\nu'}(\o')d\o d\o' d\MM.
\eea
Recall the decomposition of $N$ in the frame $N', e'_A$:
\be\lab{nyc5}
N=\gn N'+(N-\gn N').
\ee
and the decomposition of $N'$ in the frame $N, e_A$:
\be\lab{nyc6}
N'=\gn N+(N'-\gn N).
\ee
\eqref{nyc5} yields:
\bea\lab{nyc7}
&&P_l\trc N(L'(P_m\trc'))+L(P_l\trc)N(P_m\trc')\\
\nn&=&\gn P_l\trc N'(L'(P_m\trc'))+P_l\trc (N-\gn N')(L'(P_m\trc'))\\
\nn&&+\gn L(P_l\trc)N'(P_m\trc')+L(P_l\trc)(N-\gn N')(P_m\trc')
\eea
Also, recall that:
$$\gl=-1+\gn$$
which together with \eqref{frame}, \eqref{nyc5} and \eqref{nyc6} yields:
\bea\lab{nyc8}
N(\gl)&=&-\gg(\nab_NN,N')-\gg(N,\nab_NN')\\
\nn&=&b^{-1}\nabb_{N'-\gn N}(b)+\gn {b'}^{-1}\nabb_{N-\gn N'}(b')\\
\nn&& -\th'(N-\gn N', N-\gn N').
\eea
In view of \eqref{nyc4}, \eqref{nyc7} and \eqref{nyc8}, we obtain:
\be\lab{nyc9}
B^{1,2,1}_{j,\nu,\nu',l,m} = B^{1,2,1,1}_{j,\nu,\nu',l,m}+B^{1,2,1,2}_{j,\nu,\nu',l,m}+B^{1,2,1,3}_{j,\nu,\nu',l,m},
\ee
where $B^{1,2,1,1}_{j,\nu,\nu',l,m}$ is given by:
\bea\lab{nyc9:1}
B^{1,2,1,1}_{j,\nu,\nu',l,m} & = &  2^{-2j-1}\int_{\MM}\int_{\S\times\S} \frac{N(L(P_l\trc))P_m\trc'+ P_l\trc N'(L'(P_m\trc'))}{\gl}\\
\nn&&\times F_{j,-1}(u)F_{j,-1}(u')\eta_j^\nu(\o)\eta_j^{\nu'}(\o')d\o d\o' d\MM,
\eea
where $B^{1,2,1,2}_{j,\nu,\nu',l,m}$ is given by:
\be\lab{nyc9bis}
B^{1,2,1,2}_{j,\nu,\nu',l,m}  =  2^{-2j}\int_{\MM}\int_{\S\times\S} H F_{j,-1}(u)F_{j,-1}(u')\eta_j^\nu(\o)\eta_j^{\nu'}(\o')d\o d\o' d\MM,
\ee
with the tensor $H$ on $\MM$ given, schematically, by:
\bea\lab{nyc10}
H &=&  \frac{1}{\gg(L,L')}\Bigg(P_l\trc \nabb'(L'(P_m\trc'))(N-N')+ L(P_l\trc)N'(P_m\trc')\\
\nn&&+L(P_l\trc)\nabb'(P_m\trc')(N-N')+N(P_l\trc)L'(P_m\trc')\\
\nn&& +\left(\frac{({b'}^{-1}\nabb (b')+\th')(N-N')^2}{\gl}+\trt\right)\left(L(P_l\trc)P_m\trc'+P_l\trc L'(P_m\trc')\right)\Bigg),
\eea
and where $B^{1,2,1,3}_{j,\nu,\nu',l,m}$ is given by:
\bea\lab{nyc9ter}
&&B^{1,2,1,3}_{j,\nu,\nu',l,m} \\
\nn& = &  -2^{-2j-1}\int_{\MM}\int_{\S\times\S} \frac{b^{-1}\nabb_{N'-\gn N}(b)+{b'}^{-1}\nabb_{N-\gn N'}(b')}{\gl^2}\\
\nn&&\times \left(L(P_l\trc)P_m\trc'+P_l\trc L'(P_m\trc')\right) F_{j,-1}(u)F_{j,-1}(u')\eta_j^\nu(\o)\eta_j^{\nu'}(\o')d\o d\o' d\MM.
\eea

Next, we estimate the three terms in the right-hand side of \eqref{nyc9} starting with $B^{1,2,1,1}_{j,\nu,\nu',l,m}$. Recall from \eqref{uso1} that $(l, m)$ satisfy:
$$m<l\textrm{ and }2^m\leq 2^j|\nu-\nu'|.$$
Summing in $(l,m)$, we obtain:
\bee
&& \sum_{(l,m)/m<l\textrm{ and }2^m\leq 2^j|\nu-\nu'|}\frac{N(L(P_l\trc))P_m\trc'+ P_l\trc N'(L'(P_m\trc'))}{\gl}\\
&&+\sum_{(l,m)/m<l\textrm{ and }2^m\leq 2^j|\nu-\nu'|}\frac{N(L(P_m\trc))P_l\trc'+ P_m\trc N'(L'(P_l\trc'))}{\gl}\\
&=& \frac{N(L(\trc))\trc'+ \trc N'(L'(\trc'))}{\gl}\\
&&-\frac{N(L(P_{>2^j|\nu-\nu'|}\trc))P_{>2^j|\nu-\nu'|}\trc'+ P_{>2^j|\nu-\nu'|}\trc N'(L'(P_{>2^j|\nu-\nu'|}\trc'))}{\gl}.
\eee
Thus, using the symmetry in $(\o, \o')$ of the integrant in $B^{1,2,1,1}_{j,\nu,\nu',l,m}$, we obtain:
 \bea\lab{nycc}
&&\sum_{(l,m)/m<l\textrm{ and }2^m\leq 2^j|\nu-\nu'|} B^{1,2,1,1}_{j,\nu,\nu',l,m}+ \sum_{(l,m)/m<l\textrm{ and }2^m\leq 2^j|\nu-\nu'|}B^{1,2,1,1}_{j,\nu',\nu,l,m}\\
\nn&=& 2^{-2j-1}\int_{\MM}\int_{\S\times\S} \frac{N(L(\trc))\trc'+ \trc N'(L'(P_m\trc'))}{\gl}\\
\nn&&\times F_{j,-1}(u)F_{j,-1}(u')\eta_j^\nu(\o)\eta_j^{\nu'}(\o')d\o d\o' d\MM-2^{-2j-1}\int_{\MM}\int_{\S\times\S}\\ 
\nn&&\times\frac{N(L(P_{>2^j|\nu-\nu'|}\trc))P_{>2^j|\nu-\nu'|}\trc'+ P_{>2^j|\nu-\nu'|}\trc N'(L'(P_{>2^j|\nu-\nu'|}\trc'))}{\gl}\\
\nn&&\times F_{j,-1}(u)F_{j,-1}(u')\eta_j^\nu(\o)\eta_j^{\nu'}(\o')d\o d\o' d\MM.
\eea
Estimating the terms $N(L(P_{>2^j|\nu-\nu'|}\trc))$ and $N'(L'(P_{>2^j|\nu-\nu'|}\trc'))$ would involve commutator terms which are difficult to handle. To avoid this issue, we commute $L$ with $N$ and $L'$ with $N'$, and then integrate the $L$ and the $L'$ derivative by parts. We obtain schematically in view of \eqref{nycc}, :
 \bea\lab{nycc1}
&&\sum_{(l,m)/m<l\textrm{ and }2^m\leq 2^j|\nu-\nu'|} B^{1,2,1,1}_{j,\nu,\nu',l,m}+ \sum_{(l,m)/m<l\textrm{ and }2^m\leq 2^j|\nu-\nu'|}B^{1,2,1,1}_{j,\nu',\nu,l,m}\\
\nn&=& 2^{-2j-1}\int_{\MM}\int_{\S\times\S} \frac{N(L(\trc))\trc'}{\gl}\\
\nn&&\times F_{j,-1}(u)F_{j,-1}(u')\eta_j^\nu(\o)\eta_j^{\nu'}(\o')d\o d\o' d\MM\\
\nn&& -2^{-2j-1}\int_{\MM}\int_{\S\times\S} \frac{[N,L](P_{>2^j|\nu-\nu'|}\trc)P_{>2^j|\nu-\nu'|}\trc'}{\gl}\\
\nn&&\times F_{j,-1}(u)F_{j,-1}(u')\eta_j^\nu(\o)\eta_j^{\nu'}(\o')d\o d\o' d\MM\\
\nn&& -2^{-2j-1}\int_{\MM}\int_{\S\times\S} \frac{N(P_{>2^j|\nu-\nu'|}\trc)L(P_{>2^j|\nu-\nu'|}\trc')}{\gl}\\
\nn&&\times F_{j,-1}(u)F_{j,-1}(u')\eta_j^\nu(\o)\eta_j^{\nu'}(\o')d\o d\o' d\MM\\
\nn&& -2^{-2j-1}\int_{\MM}\int_{\S\times\S} \left(\frac{\textrm{div}_{\gg}(L)}{\gl}-\frac{L(\gl)}{\gl^2}\right)N(P_{>2^j|\nu-\nu'|}\trc)P_{>2^j|\nu-\nu'|}\trc'\\
\nn&&\times F_{j,-1}(u)F_{j,-1}(u')\eta_j^\nu(\o)\eta_j^{\nu'}(\o')d\o d\o' d\MM\\
\nn&& -2^{-j-1}\int_{\MM}\left(\int_{\S} N(P_{>2^j|\nu-\nu'|}\trc) F_{j,-1}(u)\eta_j^\nu(\o)d\o\right)\\
\nn&&\times\left(\int_{\S} {b'}^{-1}P_{>2^j|\nu-\nu'|}\trc' F_j(u')\eta_j^{\nu'}(\o')d\o'\right) d\MM+\textrm{ terms  interverting }(\nu, \nu'),
\eea
where the last term in the right-hand side of \eqref{nycc1} appears when the $L$ derivative falls on the phase in 
view of \eqref{ibpl}, and where we chose to ignore the terms which are obtained by interverting $\nu$ and $\nu'$ since they are treated in the exact same way.  

We decompose $L$ in the frame $L', N', e'_A$:
\be\lab{dialo}
L=L'+(N-\gn N')+(\gn-1)N',
\ee
which yields the following decomposition:
\bea\lab{dialo1}
L(P_{>2^j|\nu-\nu'|}\trc')&=&L'(P_{>2^j|\nu-\nu'|}\trc')+(N-\gn N')(P_{>2^j|\nu-\nu'|}\trc')\\
\nn&&+(\gn-1)N'(P_{>2^j|\nu-\nu'|}\trc').
\eea
Recall the identities \eqref{nice24} and \eqref{nice25}:
$$\gg(L,L')=-1+\gn\textrm{ and }1-\gn=\frac{\gg(N-N',N-N')}{2}.$$
We may thus expand 
$$\frac{1}{\gl}\textrm{ and }\frac{1}{\gl^2}$$ 
in the same fashion than \eqref{nice27}, and in view of \eqref{nycc1}, \eqref{dialo1}, the formula \eqref{fete3bis} for $\textrm{div}_{\gg}(L)$ and the formula \eqref{fete6bis} for $L(\gg(L,L'))$, we obtain, schematically:
\bea\lab{nycc2}
&&\sum_{(l,m)/m<l\textrm{ and }2^m\leq 2^j|\nu-\nu'|} B^{1,2,1,1}_{j,\nu,\nu',l,m}+ \sum_{(l,m)/m<l\textrm{ and }2^m\leq 2^j|\nu-\nu'|}B^{1,2,1,1}_{j,\nu',\nu,l,m}\\
\nn &=& 2^{-j}\sum_{p, q\geq 0}c_{pq}\int_{\MM}\frac{1}{(2^{\frac{j}{2}}|N_\nu-N_{\nu'}|)^{p+q}}\Bigg[\frac{1}{(2^{\frac{j}{2}}|N_\nu-N_{\nu'}|)^2}(h_{1,p,q}+h_{2,p,q}+h_{3,p,q})\\
\nn&&+\frac{1}{2^{\frac{j}{2}}(2^{\frac{j}{2}}|N_\nu-N_{\nu'}|)}h_{4,p,q}+2^{-j}h_{5,p,q}\Bigg] d\MM\\
\nn&& -2^{-j-1}\int_{\MM}\left(\int_{\S} N(P_{>2^j|\nu-\nu'|}\trc) F_{j,-1}(u)\eta_j^\nu(\o)d\o\right)\\
\nn&&\times\left(\int_{\S} {b'}^{-1}P_{>2^j|\nu-\nu'|}\trc' F_j(u')\eta_j^{\nu'}(\o')d\o'\right) d\MM+\textrm{ terms interverting }(\nu, \nu'),
\eea
where the scalar functions $h_{1,p,q}, h_{2,p,q}, h_{3,p,q}, h_{4,p,q}, h_{5,p,q}$ on $\MM$ are given by:
\bea\lab{nycc3}
h_{1,p,q}&=& \left(\int_{\S} N(L(\trc))\left(2^{\frac{j}{2}}(N-N_\nu)\right)^pF_{j,-1}(u)\eta_j^\nu(\o)d\o\right)\\
\nn&&\times\left(\int_{\S}\trc'\left(2^{\frac{j}{2}}(N'-N_{\nu'})\right)^qF_{j,-1}(u')\eta_j^{\nu'}(\o')d\o'\right),
\eea
\bea\lab{nycc4}
h_{2,p,q}&=& \left(\int_{\S} G_1\left(2^{\frac{j}{2}}(N-N_\nu)\right)^pF_{j,-1}(u)\eta_j^\nu(\o)d\o\right)\\
\nn&&\times\left(\int_{\S}P_{>2^j|\nu-\nu'|}\trc'\left(2^{\frac{j}{2}}(N'-N_{\nu'})\right)^qF_{j,-1}(u')\eta_j^{\nu'}(\o')d\o'\right),
\eea
\bea\lab{nycc5}
h_{3,p,q}&=& \left(\int_{\S} N(P_{>2^j|\nu-\nu'|}\trc)\left(2^{\frac{j}{2}}(N-N_\nu)\right)^pF_{j,-1}(u)\eta_j^\nu(\o)d\o\right)\\
\nn&&\times\left(\int_{\S}G_2\left(2^{\frac{j}{2}}(N'-N_{\nu'})\right)^qF_{j,-1}(u')\eta_j^{\nu'}(\o')d\o'\right),
\eea
\bea\lab{nycc5bis}
h_{4,p,q}&=& \left(\int_{\S} N(P_{>2^j|\nu-\nu'|}\trc)\left(2^{\frac{j}{2}}(N-N_\nu)\right)^pF_{j,-1}(u)\eta_j^\nu(\o)d\o\right)\\
\nn&&\times\left(\int_{\S}\nabb'(P_{>2^j|\nu-\nu'|}\trc')\left(2^{\frac{j}{2}}(N'-N_{\nu'})\right)^qF_{j,-1}(u')\eta_j^{\nu'}(\o')d\o'\right),
\eea
and:
\bea\lab{nycc5ter}
h_{5,p,q}&=& \left(\int_{\S} N(P_{>2^j|\nu-\nu'|}\trc)\left(2^{\frac{j}{2}}(N-N_\nu)\right)^pF_{j,-1}(u)\eta_j^\nu(\o)d\o\right)\\
\nn&&\times\left(\int_{\S}N'(P_{>2^j|\nu-\nu'|}\trc')\left(2^{\frac{j}{2}}(N'-N_{\nu'})\right)^qF_{j,-1}(u')\eta_j^{\nu'}(\o')d\o'\right),
\eea
where the tensors $G_1$ and $G_2$ are schematically given by:
\be\lab{nycc6}
G_1=[N, L](P_{>2^j|\nu-\nu'|}\trc)+\db N(P_{>2^j|\nu-\nu'|}\trc)
\ee
and:
\be\lab{nycc7}
G_2=L'(P_{>2^j|\nu-\nu'|}\trc')+(\db'+\chi'+\zeta')P_{>2^j|\nu-\nu'|}\trc',
\ee
and where $c_{pq}$ are explicit real coefficients such that the series 
$$\sum_{p, q\geq 0}c_{pq}x^py^q$$
has radius of convergence 1. 

Next, we estimate the $L^1(\MM)$ norm of $h_{1,p,q}, h_{2,p,q}, h_{3,p,q}, h_{4,p,q}, h_{5,p,q}$ starting with $h_{1,p,q}$. We have:
\bea\lab{nycc8}
\norm{h_{1,p,q}}_{L^1(\MM)}&\les& \normm{\int_{\S} N(L(\trc))\left(2^{\frac{j}{2}}(N-N_\nu)\right)^pF_{j,-1}(u)\eta_j^\nu(\o)d\o}_{L^{\frac{3}{2}}(\MM)}\\
\nn &&\normm{\int_{\S}\trc'\left(2^{\frac{j}{2}}(N'-N_{\nu'})\right)^qF_{j,-1}(u')\eta_j^{\nu'}(\o')d\o'}_{L^3(\MM)}.
\eea
We estimate the $L^\infty(\MM)$ norm of the last term:
\bea\lab{nycc13bis}
&&\normm{\int_{\S}\trc'\left(2^{\frac{j}{2}}(N'-N_{\nu'})\right)^qF_{j,-1}(u')\eta_j^{\nu'}(\o')d\o'}_{L^\infty(\MM)}\\
\nn&\les & \int_{\S}\normm{\trc'\left(2^{\frac{j}{2}}(N'-N_{\nu'})\right)^q F_{j,-1}(u')}_{L^\infty}\eta_j^{\nu'}(\o')d\o'\\
\nn&\les & \ep \int_{\S}\norm{F_{j,-1}(u')}_{L^\infty_{u'}}\eta_j^{\nu'}(\o')d\o'\\
\nn&\les & 2^j\ep\gamma^{\nu'}_j,
\eea
where we used the estimate \eqref{esttrc} for $\trc$, the estimate \eqref{estNomega} for $\po N$, the size of the patch,  
Cauchy Schwartz in $\lambda'$ for $\norm{F_{j,-1}(u')}_{L^\infty_{u'}}$, and Cauchy Schwartz in $\o'$. On the other hand, the estimate \eqref{loebter} yields:
$$\normm{\int_{\S}\trc'\left(2^{\frac{j}{2}}(N'-N_{\nu'})\right)^qF_{j,-1}(u')\eta_j^{\nu'}(\o')d\o'}_{L^2(\MM)}\les  (1+q^2)\ep\gamma^{\nu'}_j.$$
Interpolating these two estimates, we obtain:
\be\lab{nycc9}
\normm{\int_{\S}\trc'\left(2^{\frac{j}{2}}(N'-N_{\nu'})\right)^qF_{j,-1}(u')\eta_j^{\nu'}(\o')d\o'}_{L^3(\MM)}\les 2^{\frac{j}{3}}(1+q)\ep\gamma^{\nu'}_j
\ee
Next, we estimate the first term in the right-hand side of \eqref{nycc8}. We have:
\bea\lab{nycc10}
&&\normm{\int_{\S} N(L(\trc))\left(2^{\frac{j}{2}}(N-N_\nu)\right)^pF_{j,-1}(u)\eta_j^\nu(\o)d\o}_{L^{\frac{3}{2}}(\MM)}\\
\nn &\les& \int_{\S} \norm{N(L(\trc))}_{\li{\infty}{\frac{3}{2}}}\normm{\left(2^{\frac{j}{2}}(N-N_\nu)\right)^p}_{L^\infty}\norm{F_{j,-1}(u)}_{L^{\frac{3}{2}}_u}\eta_j^\nu(\o)d\o\\
\nn &\les& \int_{\S} \norm{N(L(\trc))}_{\li{\infty}{\frac{3}{2}}}\norm{F_{j,-1}(u)}_{L^2_u}\eta_j^\nu(\o)d\o,
\eea
where we used in the last inequality the estimate \eqref{estNomega} for $\po N$ and the size of the patch. Next, we estimate $N(L(\trc))$. In view of the Raychaudhuri equation \eqref{raychaudhuri}, we have:
$$N(L(\trc))=  -\trc N(\trc)- 2\hch\c \dd_N\hch -N(\db)   \trc-\db N(\trc),$$
which together with the Sobolev embedding \eqref{sobineq}, and the estimates \eqref{esttrc} for $\trc$, \eqref{esthch} for $\hch$, and \eqref{estn} \eqref{estk} for $\db$ yields:
$$\norm{N(L(\trc))}_{\li{\infty}{\frac{3}{2}}}\les (\norm{\dd\chi}_{\li{\infty}{2}}+\norm{\dd\db}_{\li{\infty}{2}})((\norm{\chi}_{\li{\infty}{6}}+\norm{\db}_{\li{\infty}{6}})\les \ep.$$
Together with \eqref{nycc10}, we obtain:
\bea\lab{nycc11}
&&\normm{\int_{\S} N(L(\trc))\left(2^{\frac{j}{2}}(N-N_\nu)\right)^pF_{j,-1}(u)\eta_j^\nu(\o)d\o}_{L^{\frac{3}{2}}(\MM)}\\
\nn &\les& \ep\int_{\S} \norm{F_{j,-1}(u)}_{L^2_u}\eta_j^\nu(\o)d\o\\
\nn &\les& 2^{\frac{j}{2}}\ep\gamma^\nu_j,
\eea
where we used in the last inequality Plancherel in $u$, Cauchy Schwarz in $\o$ and the size of the patch. Finally, \eqref{nycc8}, \eqref{nycc9} and \eqref{nycc11} imply:
\be\lab{nycc12}
\norm{h_{1,p,q}}_{L^1(\MM)}\les (1+q)2^{\frac{5j}{6}}\ep^2\gamma^\nu_j\gamma^{\nu'}_j.
\ee

Next, we estimate $h_{2,p,q}$.  We have:
\bea\lab{nycc13}
\norm{h_{2,p,q}}_{L^1(\MM)}&\les& \normm{\int_{\S} G_1\left(2^{\frac{j}{2}}(N-N_\nu)\right)^pF_{j,-1}(u)\eta_j^\nu(\o)d\o}_{L^{\frac{3}{2}}(\MM)}\\
\nn &&\normm{\int_{\S}P_{>2^j|\nu-\nu'|}\trc'\left(2^{\frac{j}{2}}(N'-N_{\nu'})\right)^qF_{j,-1}(u')\eta_j^{\nu'}(\o')d\o'}_{L^3(\MM)}.
\eea
Arguing as in \eqref{nycc13bis}, we have:
$$\normm{\int_{\S}P_{>2^j|\nu-\nu'|}\trc'\left(2^{\frac{j}{2}}(N'-N_{\nu'})\right)^qF_{j,-1}(u')\eta_j^{\nu'}(\o')d\o'}_{L^\infty(\MM)}\les  2^j\ep\gamma^{\nu'}_j.$$
On the other hand, we have in view of the $L^2$ estimate \eqref{oscl2}:
\bea\lab{nycc133}
&&\normm{\int_{\S}P_{>2^j|\nu-\nu'|}\trc'\left(2^{\frac{j}{2}}(N'-N_{\nu'})\right)^qF_{j,-1}(u')\eta_j^{\nu'}(\o')d\o'}_{L^2(\MM)}\\
\nn&\les&  \left(\sup_\o\normm{\left(2^{\frac{j}{2}}(N'-N_{\nu'})\right)^q}_{L^\infty}\right)\frac{2^{\frac{j}{2}}}{2^j|\nu-\nu'|}\ep\gamma^{\nu'}_j\\
\nn&\les & \frac{\ep \gamma^{\nu'}_j}{2^{\frac{j}{2}}|\nu-\nu'|},
\eea
where we used in the last inequality the estimate \eqref{estNomega} for $\po N$ and the size of the patch. Interpolating these two estimates, and using the fact that:
$$2^{\frac{j}{2}}|\nu-\nu'|\gtrsim 1,$$
we obtain:
\be\lab{nycc14}
\normm{\int_{\S}P_{>2^j|\nu-\nu'|}\trc'\left(2^{\frac{j}{2}}(N'-N_{\nu'})\right)^qF_{j,-1}(u')\eta_j^{\nu'}(\o')d\o'}_{L^3(\MM)}\les 2^{\frac{j}{3}}\ep\gamma^{\nu'}_j.
\ee
Next, we estimate the first term in the right-hand side of \eqref{nycc13}. Arguing as in \eqref{nycc10}, we have:
\bea\lab{nycc15}
&&\normm{\int_{\S} G_1\left(2^{\frac{j}{2}}(N-N_\nu)\right)^pF_{j,-1}(u)\eta_j^\nu(\o)d\o}_{L^{\frac{3}{2}}(\MM)}\\
\nn &\les& \int_{\S} \norm{G_1}_{\li{\infty}{\frac{3}{2}}}\norm{F_{j,-1}(u)}_{L^2_u}\eta_j^\nu(\o)d\o.
\eea
Next, we estimate $G_1$. In view of the definition of $G_1$ \eqref{nycc6}, the commutator formulas \eqref{comm3} for $[L, \lb]$ and the fact that $2N=L-\lb$, we have schematically:
 \bee
G_1=\db N(P_{>2^j|\nu-\nu'|}\trc)+n^{-1}\nab_Nn LP_{>2^j|\nu-\nu'|}\trc+(\zeta-\zb)\c \nabb(P_{>2^j|\nu-\nu'|}\trc
\eee
This yields:
\bee
\norm{G_1}_{\li{\infty}{\frac{3}{2}}}&\les& (\norm{\db}_{\li{\infty}{6}}+\norm{n^{-1}\nab_Nn}_{\li{\infty}{6}}+\norm{\zeta}_{\li{\infty}{6}}+\norm{\zb}_{\li{\infty}{6}})\\
&&\times\norm{\dd P_{>2^j|\nu-\nu'|}\trc}_{\li{\infty}{2}}\\
&\les& \ep \norm{\dd P_{>2^j|\nu-\nu'|}\trc}_{\li{\infty}{2}},
\eee
where we used in the last inequality the Sobolev embedding \eqref{sobineq}, and the estimates \eqref{estk} \eqref{estn} for $n$, $\db$ and $\zb$, and the estimate \eqref{estzeta} for $\zeta$. Together with the basic properties of $P_{>2^j|\nu-\nu'|}$, the commutator estimates \eqref{commlp1} and \eqref{commlp2}, and the estimate \eqref{esttrc} for $\trc$, this implies:
\be\lab{vinoverde}
\norm{G_1}_{\li{\infty}{\frac{3}{2}}}\les \ep.
\ee
Together with \eqref{nycc15}, and arguing as in \eqref{nycc11}, we obtain:
\be\lab{nycc16}
\normm{\int_{\S} G_1\left(2^{\frac{j}{2}}(N-N_\nu)\right)^pF_{j,-1}(u)\eta_j^\nu(\o)d\o}_{L^{\frac{3}{2}}(\MM)}\les 2^{\frac{j}{2}}\ep\gamma^\nu_j.
\ee
Finally, \eqref{nycc13}, \eqref{nycc14} and \eqref{nycc16} imply:
\be\lab{nycc17}
\norm{h_{2,p,q}}_{L^1(\MM)}\les 2^{\frac{5j}{6}}\ep^2\gamma^\nu_j\gamma^{\nu'}_j.
\ee

Next, we estimate $h_{3,p,q}$. We have:
\bea\lab{nycc18}
\norm{h_{3,p,q}}_{L^1(\MM)}&\les& \normm{\int_{\S} N(P_{>2^j|\nu-\nu'|}\trc)\left(2^{\frac{j}{2}}(N-N_\nu)\right)^pF_{j,-1}(u)\eta_j^\nu(\o)d\o}_{L^2(\MM)}\\
\nn &&\normm{\int_{\S}G_2\left(2^{\frac{j}{2}}(N'-N_{\nu'})\right)^qF_{j,-1}(u')\eta_j^{\nu'}(\o')d\o'}_{L^2(\MM)}.
\eea
We estimate the first term in the right-hand side of \eqref{nycc18}. Using the basic estimate in $L^2(\MM)$ \eqref{oscl2bis}, we have:
\bee
&&\normm{\int_{\S} N(P_{>2^j|\nu-\nu'|}\trc)\left(2^{\frac{j}{2}}(N-N_\nu)\right)^pF_{j,-1}(u)\eta_j^\nu(\o)d\o}_{L^2(\MM)}\\
&\les& \left(\sup_{\o}\normm{N(P_{>2^j|\nu-\nu'|}\trc)\left(2^{\frac{j}{2}}(N-N_\nu)\right)^p}_{\li{\infty}{2}}\right)2^{\frac{j}{2}}\gamma^\nu_j\\
&\les& \left(\sup_{\o}\norm{N(P_{>2^j|\nu-\nu'|}\trc)}_{\li{\infty}{2}}\normm{\left(2^{\frac{j}{2}}(N-N_\nu)\right)^p}_{L^\infty}\right)2^{\frac{j}{2}}\gamma^\nu_j.
\eee
Together with the estimate \eqref{esttrc} for $\trc$, the commutator estimate \eqref{commlp1}, the estimate \eqref{estNomega} for $\po N$ and the size of the patch, we obtain:
\be\lab{nycc19}
\normm{\int_{\S} N(P_{>2^j|\nu-\nu'|}\trc)\left(2^{\frac{j}{2}}(N-N_\nu)\right)^pF_{j,-1}(u)\eta_j^\nu(\o)d\o}_{L^2(\MM)}\les \ep 2^{\frac{j}{2}}\gamma^\nu_j.
\ee
Next, we estimate the second term in the right-hand side of \eqref{nycc18}. Using the basic estimate in $L^2(\MM)$ \eqref{oscl2bis}, we have:
\bea\lab{nycc20}
&&\normm{\int_{\S}G_2\left(2^{\frac{j}{2}}(N'-N_{\nu'})\right)^qF_{j,-1}(u')\eta_j^{\nu'}(\o')d\o'}_{L^2(\MM)}\\
\nn&\les& \left(\sup_{\o}\normm{G_2\left(2^{\frac{j}{2}}(N'-N_{\nu'})\right)^q}_{\li{\infty}{2}}\right)2^{\frac{j}{2}}\gamma^{\nu'}_j\\
\nn&\les& \left(\sup_{\o}\norm{G_2}_{\li{\infty}{2}}\normm{\left(2^{\frac{j}{2}}(N'-N_{\nu'})\right)^q}_{L^\infty}\right)2^{\frac{j}{2}}\gamma^{\nu'}_j.
\eea
In view of the definition of $G_2$ \eqref{nycc7}, we have:
\bea\lab{dialo2}
\norm{G_2}_{\li{\infty}{2}}&\les& (\norm{\db}_{\tx{\infty}{4}}+\norm{\chi}_{\tx{\infty}{4}}+\norm{\z}_{\tx{\infty}{4}})\norm{P_{>2^j|\nu-\nu'|}\trc}_{\tx{2}{4}}\\
\nn&&+\norm{L(P_{>2^j|\nu-\nu'|}\trc)}_{\li{\infty}{2}}\\
\nn&\les& \ep\norm{P_{>2^j|\nu-\nu'|}\trc}_{\tx{2}{4}}+\norm{L(P_{>2^j|\nu-\nu'|}\trc)}_{\li{\infty}{2}},
\eea
where we used in the last inequality the embedding \eqref{sobineq1}, and the estimates \eqref{estn} \eqref{estk} for $\db$, the estimates \eqref{esttrc} \eqref{esthch} for $\chi$, and the estimate \eqref{estzeta} for $\z$. Using the Bernstein inequality and the finite band property for $P_l$, we have:
\bea\lab{dialo3}
\norm{P_{>2^j|\nu-\nu'|}\trc}_{\tx{2}{4}}&\les& \sum_{l>2^j|\nu-\nu'|}\norm{P_l\trc}_{\tx{2}{4}}\\
\nn&\les& \sum_{l>2^j|\nu-\nu'|}2^{\frac{l}{2}}\norm{P_l\trc}_{\li{\infty}{2}}\\
\nn&\les& \sum_{l>2^j|\nu-\nu'|}2^{-\frac{l}{2}}\norm{\nabb\trc}_{\li{\infty}{2}}\\
\nn&\les& \frac{\ep}{(2^j|\nu-\nu'|)^{\frac{1}{2}}},
\eea
where we used the estimate \eqref{esttrc} for $\trc$ in the last inequality. In view of \eqref{dialo2}, we also need to estimate $L(P_{>2^j|\nu-\nu'|}\trc)$. Using the estimate \eqref{estn} for $n$, we have:
\bea\lab{nyc22}
\norm{L(P_l\trc)}_{\li{\infty}{2}}&\les&\norm{nL(P_l\trc)}_{\li{\infty}{2}}\\
\nn&\les& \norm{P_l(nL\trc)}_{\li{\infty}{2}}+\norm{[nL, P_l]\trc)}_{\li{\infty}{2}}\\
\nn&\les& \norm{P_l(nL\trc)}_{\li{\infty}{2}}+2^{-\frac{l}{2}}\ep,
\eea 
where we used in the last inequality the commutator estimate \eqref{commlp3ter}. Now, in view of the Raychaudhuri equation \eqref{raychaudhuri}, the worst term in $P_l(nL\trc)$ is $P_l(n|\hch|^2)$. In view of the finite band property, we have:
\bee
\norm{P_l(n|\hch|^2)}_{\li{\infty}{2}}&=&2^{-2l}\norm{P_l(\divb(\nabb (n|\hch|^2))}_{\li{\infty}{2}}\\
&\les& 2^{-2l}2^{\frac{3l}{2}}\norm{\nabb(n|\hch|^2)}_{\tx{2}{\frac{4}{3}}},
\eee
where we used \eqref{lbz14bis} with $p=\frac{4}{3}$ in the last inequality. This yields:
\bee
\norm{P_l(n|\hch|^2)}_{\li{\infty}{2}}&\les& 2^{-\frac{l}{2}}\norm{\hch}_{\tx{\infty}{4}}(\norm{n}_{L^\infty(\MM)}\norm{\hch}_{\li{\infty}{2}}+\norm{\nabla n}_{\li{\infty}{3}}\norm{\hch}_{\li{\infty}{6}})\\
&\les& 2^{-\frac{l}{2}}\ep,
\eee
where we used in the last inequality the Sobolev embedding \eqref{sobineq} and the embedding \eqref{sobineq1}, and the estimates \eqref{estn} for $n$ and \eqref{esthch} for $\hch$. Since $P_l(n|\hch|^2)$ is the worst term in $P_l(nL\trc)$ in view of the Raychaudhuri equation \eqref{raychaudhuri}, we obtain:
\be\lab{celeri}
\norm{P_l(nL\trc)}_{\li{\infty}{2}}\les 2^{-\frac{l}{2}}\ep,
\ee
which together with \eqref{nyc22} yields:
\be\lab{celeri1}
\norm{L(P_l\trc)}_{\li{\infty}{2}}\les 2^{-\frac{l}{2}}\ep.
\ee
Now, in view of \eqref{dialo2}, \eqref{dialo3} and \eqref{celeri1}, and since $2^{\frac{j}{2}}|\nu-\nu'|\gtrsim 1$, we obtain:
$$\norm{G_2}_{\li{\infty}{2}}\les 2^{-\frac{j}{4}}\ep.$$
Together with \eqref{nycc20}, the estimate \eqref{estNomega} for $\po N$ and the size of the patch, we deduce: 
\be\lab{nycc21}
\normm{\int_{\S}G_2\left(2^{\frac{j}{2}}(N'-N_{\nu'})\right)^qF_{j,-1}(u')\eta_j^{\nu'}(\o')d\o'}_{L^2(\MM)}\les \ep 2^{\frac{j}{4}}\gamma^{\nu'}_j,
\ee
where we also used the estimate \eqref{estNomega} for $\po N$ and the size of the patch. Finally, \eqref{nycc18}, \eqref{nycc19} and \eqref{nycc21} imply:
\be\lab{nycc22}
\norm{h_{3,p,q}}_{L^1(\MM)}\les \ep^2 2^{\frac{3j}{4}}\gamma^\nu_j\gamma^{\nu'}_j.
\ee

Next, we estimate the $L^1(\MM)$ norm of $h_{4,p,q}$. In view of the definition of $h_{4,p,q}$ \eqref{nycc5bis}, we have:
\bee
\norm{h_{4,p,q}}_{L^1(\MM)} &\les & \normm{\int_{\S} N(P_{>2^j|\nu-\nu'|}\trc)\left(2^{\frac{j}{2}}(N-N_\nu)\right)^pF_{j,-1}(u)\eta_j^\nu(\o)d\o}_{L^2(\MM)}\\
\nn&&\times\normm{\int_{\S}\nabb'(P_{>2^j|\nu-\nu'|}\trc')\left(2^{\frac{j}{2}}(N'-N_{\nu'})\right)^qF_{j,-1}(u')\eta_j^{\nu'}(\o')d\o'}_{L^2(\MM)}.
\eee
Using the estimate \eqref{nycc19} for the first term, and the basic estimate in $L^2(\MM)$ \eqref{oscl2bis} for the second term, we obtain:
\bea
\nn\norm{h_{4,p,q}}_{L^1(\MM)} &\les & \left(\sup_{\o'}\normm{\nabb'(P_{>2^j|\nu-\nu'|}\trc')\left(2^{\frac{j}{2}}(N'-N_{\nu'})\right)^q}_{\lprime{\infty}{2}}\right) \ep 2^j\gamma^\nu_j\gamma^{\nu'}_j\\
\lab{souka}&\les & \left(\sup_{\o'}\normm{\nabb'(P_{>2^j|\nu-\nu'|}\trc')}_{\lprime{\infty}{2}}\right) \ep 2^j\gamma^\nu_j\gamma^{\nu'}_j
\eea
where we used in the last inequality the estimate \eqref{estNomega} for $\po N$ and the size of the patch. Using the finite band property and the estimate \eqref{esttrc} for $\trc$, we obtain:
$$\normm{\nabb'(P_{>2^j|\nu-\nu'|}\trc')}_{\lprime{\infty}{2}}\les\ep.$$
Together with \eqref{souka}, we finally obtain:
\be\lab{souka2}
\norm{h_{4,p,q}}_{L^1(\MM)} \les   2^j\ep^2\gamma^\nu_j\gamma^{\nu'}_j
\ee

Next, we estimate the $L^1(\MM)$ norm of $h_{5,p,q}$.  In view of the definition of $h_{5,p,q}$ \eqref{nycc5ter}, we have:
\bee
\norm{h_{5,p,q}}_{L^1(\MM)} &\les & \normm{\int_{\S} N(P_{>2^j|\nu-\nu'|}\trc)\left(2^{\frac{j}{2}}(N-N_\nu)\right)^pF_{j,-1}(u)\eta_j^\nu(\o)d\o}_{L^2(\MM)}\\
\nn&&\times\normm{\int_{\S}N'(P_{>2^j|\nu-\nu'|}\trc')\left(2^{\frac{j}{2}}(N'-N_{\nu'})\right)^qF_{j,-1}(u')\eta_j^{\nu'}(\o')d\o'}_{L^2(\MM)}.
\eee
Then, using the estimate \eqref{nycc19} for both terms, we obtain:
\be\lab{souka3}
\norm{h_{5,p,q}}_{L^1(\MM)} \les   2^j\ep^2\gamma^\nu_j\gamma^{\nu'}_j
\ee

Now, we have in view of \eqref{nycc2}, we have:
\bee
&&\left|\sum_{(l,m)/m<l\textrm{ and }2^m\leq 2^j|\nu-\nu'|} B^{1,2,1,1}_{j,\nu,\nu',l,m}+ \sum_{(l,m)/m<l\textrm{ and }2^m\leq 2^j|\nu-\nu'|}B^{1,2,1,1}_{j,\nu',\nu,l,m}\right|\\
\nn &\les& 2^{-2j}\sum_{p, q\geq 0}c_{pq}\normm{\frac{1}{(2^{\frac{j}{2}}|N_\nu-N_{\nu'}|)^{p+q}}}_{L^\infty(\MM)}\Bigg[\normm{\frac{1}{(2^{\frac{j}{2}}|N_\nu-N_{\nu'}|)^2}}_{L^\infty(\MM)}\\
\nn&&\times(\norm{h_{1,p,q}}_{L^1(\MM)}+\norm{h_{2,p,q}}_{L^1(\MM)}+\norm{h_{3,p,q}}_{L^1(\MM)})\\
\nn&&+\normm{\frac{1}{2^{\frac{j}{2}}(2^{\frac{j}{2}}|N_\nu-N_{\nu'}|)}}_{L^\infty(\MM)}\norm{h_{4,p,q}}_{L^1(\MM)}+2^{-j}\norm{h_{5,p,q}}_{L^1(\MM)}\Bigg]\\
\nn&& +2^{-j}\normm{\int_{\S} N(P_{>2^j|\nu-\nu'|}\trc) F_{j,-1}(u)\eta_j^\nu(\o)d\o}_{L^2(\MM)}\\
\nn&&\times\normm{\int_{\S} {b'}^{-1}P_{>2^j|\nu-\nu'|}\trc' F_j(u')\eta_j^{\nu'}(\o')d\o'}_{L^2(\MM)},
\eee
which together with \eqref{nice26}, \eqref{nycc12}, \eqref{nycc17}, \eqref{nycc22}, \eqref{souka2} and \eqref{souka3}  yields:
\bee
&&\left|\sum_{(l,m)/m<l\textrm{ and }2^m\leq 2^j|\nu-\nu'|} B^{1,2,1,1}_{j,\nu,\nu',l,m}+ \sum_{(l,m)/m<l\textrm{ and }2^m\leq 2^j|\nu-\nu'|}B^{1,2,1,1}_{j,\nu',\nu,l,m}\right|\\
\nn &\les& 2^{-j}\sum_{p, q\geq 0}c_{pq}\frac{1}{(2^{\frac{j}{2}}|\nu-\nu'|)^{p+q}}\left[\frac{(1+q)2^{\frac{5j}{6}}}{(2^{\frac{j}{2}}|\nu-\nu'|)^2}+\frac{1}{2^{\frac{j}{2}}(2^{\frac{j}{2}}|\nu-\nu'|)}\right]\ep^2\gamma^\nu_j\gamma^{\nu'}_j\\
\nn&& +2^{-j}\normm{\int_{\S} N(P_{>2^j|\nu-\nu'|}\trc) F_{j,-1}(u)\eta_j^\nu(\o)d\o}_{L^2(\MM)}\\
\nn&&\times\normm{\int_{\S} {b'}^{-1}P_{>2^j|\nu-\nu'|}\trc' F_j(u')\eta_j^{\nu'}(\o')d\o'}_{L^2(\MM)}\\
\nn &\les& \frac{2^{-\frac{j}{6}}\ep^2\gamma^\nu_j\gamma^{\nu'}_j}{(2^{\frac{j}{2}}|\nu-\nu'|)^2}+\frac{\ep^2\gamma^\nu_j\gamma^{\nu'}_j}{2^j |\nu-\nu'|}+2^{-j}\normm{\int_{\S} N(P_{>2^j|\nu-\nu'|}\trc) F_{j,-1}(u)\eta_j^\nu(\o)d\o}_{L^2(\MM)}\\
\nn &&\times\normm{\int_{\S} {b'}^{-1}P_{>2^j|\nu-\nu'|}\trc' F_j(u')\eta_j^{\nu'}(\o')d\o'}_{L^2(\MM)}.
\eee
Using the corresponding analog of \eqref{nycc19} and the corresponding analog of \eqref{nycc133} to estimate the last term in the right-hand side, we deduce:
\bea\lab{nycc23}
&&\left|\sum_{(l,m)/m<l\textrm{ and }2^m\leq 2^j|\nu-\nu'|} B^{1,2,1,1}_{j,\nu,\nu',l,m}+ \sum_{(l,m)/m<l\textrm{ and }2^m\leq 2^j|\nu-\nu'|}B^{1,2,1,1}_{j,\nu',\nu,l,m}\right|\\
\nn &\les& \frac{2^{-\frac{j}{6}}\ep^2\gamma^\nu_j\gamma^{\nu'}_j}{(2^{\frac{j}{2}}|\nu-\nu'|)^2}+\frac{\ep^2\gamma^\nu_j\gamma^{\nu'}_j}{2^j |\nu-\nu'|}.
\eea

Next, we estimate $B^{1,2,1,2}_{j,\nu,\nu',l,m}$. Recall from \eqref{nyc9bis} and \eqref{nyc10} that $B^{1,2,1,2}_{j,\nu,\nu',l,m}$ is given by:
$$B^{1,2,1,2}_{j,\nu,\nu',l,m}  =  2^{-2j}\int_{\MM}\int_{\S\times\S} H F_{j,-1}(u)F_{j,-1}(u')\eta_j^\nu(\o)\eta_j^{\nu'}(\o')d\o d\o' d\MM,$$
with the tensor $H$ on $\MM$ given, schematically, by:
\bee
H &=&  \frac{1}{\gg(L,L')}\Bigg(P_l\trc \nabb'(L'(P_m\trc'))(N-N')+ L(P_l\trc)N'(P_m\trc')\\
\nn&&+L(P_l\trc)\nabb'(P_m\trc')(N-N')+N(P_l\trc)L'(P_m\trc')\\
\nn&& +\left(\frac{({b'}^{-1}\nabb (b')+\th')(N-N')^2}{\gl}+\trt\right)\left(L(P_l\trc)P_m\trc'+P_l\trc L'(P_m\trc')\right)\Bigg).
\eee
Expanding 
$$\frac{1}{\gl}\textrm{ and }\frac{1}{\gl^2}$$
in the same fashion than \eqref{nice27}, and in view of \eqref{nyc10}, we obtain, schematically:
\bea\lab{nyc11}
H &=& \frac{1}{|N_\nu-N_{\nu'}|^2}\left(\sum_{p, q\geq 0}c_{pq}\left(\frac{N-N_\nu}{|N_\nu-N_{\nu'}|}\right)^p\left(\frac{N'-N_{\nu'}}{|N_\nu-N_{\nu'}|}\right)^q\right)\\
\nn && \times(L(P_l\trc) H_1+H_2 L'(P_m\trc')+P_l\trc (\nabb'(L'(P_m\trc'))\\
\nn&&+({b'}^{-1}\nabb (b')+\th') L'(P_m\trc'))+\trt L(P_l\trc)P_m\trc'),
\eea
where the tensors $H_1, H_2$ on $\MM$ are schematically given by:
\be\lab{nyc13}
H_1 = N'(P_m\trc')+\nabb'(P_m\trc')+\th' P_m\trc',
\ee
and:
\be\lab{nyc14}
H_2 = N(P_l\trc)+\trt P_l\trc,
\ee
and where $c_{pq}$ are explicit real coefficients such that the series 
$$\sum_{p, q\geq 0}c_{pq}x^py^q$$
has radius of convergence 1. In turn, this yields in view of \eqref{nyc9bis} and \eqref{nyc10} the following decomposition for $B^{1,2,1,2}_{j,\nu,\nu',l,m}$:
\bea\lab{nyc15}
B^{1,2,1,2}_{j,\nu,\nu',l,m} &=& 2^{-j}\sum_{p, q\geq 0}c_{pq}\int_{\MM}\frac{1}{(2^{\frac{j}{2}}|N_\nu-N_{\nu'}|)^{p+q+2}}\\
\nn&&\times\left[h_{1,p,q,l,m}+h_{2,p,q,l,m}+h_{3,p,q,l,m}+h_{4,p,q,l,m}\right] d\MM,
\eea
where the scalar functions $h_{1,p,q,l,m}, h_{2,p,q,l,m}, h_{3,p,q,l,m}, h_{4,p,q,l,m}$ on $\MM$ are schematically given by:
\bea\lab{nyc17}
h_{1,p,q,l,m}&=& \left(\int_{\S} L(P_l\trc)\left(2^{\frac{j}{2}}(N-N_\nu)\right)^pF_{j,-1}(u)\eta_j^\nu(\o)d\o\right)\\
\nn&&\times\left(\int_{\S}H_1\left(2^{\frac{j}{2}}(N'-N_{\nu'})\right)^qF_{j,-1}(u')\eta_j^{\nu'}(\o')d\o'\right),
\eea
\bea\lab{nyc18}
h_{2,p,q,l,m}&=& \left(\int_{\S} H_2\left(2^{\frac{j}{2}}(N-N_\nu)\right)^pF_{j,-1}(u)\eta_j^\nu(\o)d\o\right)\\
\nn&&\times\left(\int_{\S}L'(P_m\trc') \left(2^{\frac{j}{2}}(N'-N_{\nu'})\right)^qF_{j,-1}(u')\eta_j^{\nu'}(\o')d\o'\right),
\eea
\bea\lab{nyc19}
h_{3,p,q,l,m}&=& \left(\int_{\S} P_l\trc \left(2^{\frac{j}{2}}(N-N_\nu)\right)^pF_{j,-1}(u)\eta_j^\nu(\o)d\o\right)\\
\nn&&\times\left(\int_{\S}({b'}^{-1}\nabb (b')+\th') L'(P_m\trc')\left(2^{\frac{j}{2}}(N'-N_{\nu'})\right)^qF_{j,-1}(u')\eta_j^{\nu'}(\o')d\o'\right)\\
\nn &&+\left(\int_{\S} \trt L(P_l\trc) \left(2^{\frac{j}{2}}(N-N_\nu)\right)^pF_{j,-1}(u)\eta_j^\nu(\o)d\o\right)\\
\nn&&\times\left(\int_{\S} P_m\trc'\left(2^{\frac{j}{2}}(N'-N_{\nu'})\right)^qF_{j,-1}(u')\eta_j^{\nu'}(\o')d\o'\right).
\eea
and:
\bea\lab{nyc19bis}
h_{4,p,q,l,m}&=& \left(\int_{\S} P_l\trc \left(2^{\frac{j}{2}}(N-N_\nu)\right)^pF_{j,-1}(u)\eta_j^\nu(\o)d\o\right)\\
\nn&&\times\left(\int_{\S}\nabb'(L'(P_m\trc'))\left(2^{\frac{j}{2}}(N'-N_{\nu'})\right)^qF_{j,-1}(u')\eta_j^{\nu'}(\o')d\o'\right).
\eea

Next, we evaluate the $L^1(\MM)$ norm of $h_{1,p,q,l,m}, h_{2,p,q,l,m}, h_{3,p,q,l,m}, h_{4,p,q,l,m}$, starting with $h_{1,p,q,l,m}$. We have:
\bea\lab{nyc20}
\norm{h_{1,p,q,l,m}}_{L^1(\MM)}&\les& \normm{\int_{\S} L(P_l\trc)\left(2^{\frac{j}{2}}(N-N_\nu)\right)^pF_{j,-1}(u)\eta_j^\nu(\o)d\o}_{L^2(\MM)}\\
\nn&&\times\normm{\int_{\S}H_1\left(2^{\frac{j}{2}}(N'-N_{\nu'})\right)^qF_{j,-1}(u')\eta_j^{\nu'}(\o')d\o'}_{L^2(\MM)}.
\eea
Next, we evaluate both terms in the right-hand side of \eqref{nyc20} starting with the first one. Assume first that $l>j/2$. Then, the basic estimate in $L^2(\MM)$ \eqref{oscl2bis} yields:
\bea\lab{nyc21}
&&\normm{\int_{\S} L(P_l\trc)\left(2^{\frac{j}{2}}(N-N_\nu)\right)^pF_{j,-1}(u)\eta_j^\nu(\o)d\o}_{L^2(\MM)}\\
\nn&\les& \left(\sup_{\o}\normm{L(P_l\trc)\left(2^{\frac{j}{2}}(N-N_\nu)\right)^p}_{\li{\infty}{2}}\right)2^{\frac{j}{2}}\gamma^\nu_j\\
\nn&\les& \left(\sup_{\o}\norm{L(P_l\trc)}_{\li{\infty}{2}}\right)2^{\frac{j}{2}}\gamma^\nu_j,
\eea
where we used in the last inequality the estimate \eqref{estNomega} for $\po N$ and the size of the patch. In view of \eqref{nyc21} and the estimate \eqref{celeri1} for $L(P_l\trc)$, we obtain, in the case $l>j/2$:
\be\lab{nyc23}
\normm{\int_{\S} L(P_l\trc)\left(2^{\frac{j}{2}}(N-N_\nu)\right)^pF_{j,-1}(u)\eta_j^\nu(\o)d\o}_{L^2(\MM)}\les  \ep 2^{\frac{j}{2}-\frac{l}{2}}\gamma^\nu_j.
\ee

Next, we evaluate the first term in the right-hand side of \eqref{nyc20} in the case $l=j/2$, which is given by:
$$\normm{\int_{\S} L(P_{\leq j/2}\trc)\left(2^{\frac{j}{2}}(N-N_\nu)\right)^pF_{j,-1}(u)\eta_j^\nu(\o)d\o}_{L^2(\MM)}.$$
We first decompose $L(P_{\leq j/2}\trc)$ as:
$$L(P_{\leq j/2}\trc)=L(\trc)-\sum_{l>\frac{j}{2}}L(P_l\trc),$$
which together with \eqref{nyc23} yields:
\bea\lab{nyc23bis}
&&\normm{\int_{\S} L(P_{\leq \frac{j}{2}}\trc)\left(2^{\frac{j}{2}}(N-N_\nu)\right)^pF_{j,-1}(u)\eta_j^\nu(\o)d\o}_{L^2(\MM)}\\
\nn &\les& \normm{\int_{\S} L(\trc)\left(2^{\frac{j}{2}}(N-N_\nu)\right)^pF_{j,-1}(u)\eta_j^\nu(\o)d\o}_{L^2(\MM)}+ \sum_{l>\frac{j}{2}}\ep 2^{\frac{j}{2}-\frac{l}{2}}\gamma^\nu_j\\
\nn &\les& \normm{\int_{\S} L(\trc)\left(2^{\frac{j}{2}}(N-N_\nu)\right)^pF_{j,-1}(u)\eta_j^\nu(\o)d\o}_{L^2(\MM)}+ \ep 2^{\frac{j}{4}}\gamma^\nu_j.
\eea
Now, recall the decomposition \eqref{nice44} \eqref{nice45} \eqref{nice46} for $L(\trc)$. We have: 
\be\lab{nyc24}
L(\trc)={\chi_2}_\nu\c (2\chi_1+\hch)+f^j_1+f^j_2,
\ee
where the scalar $f^j_1$ only depends on $\nu$ and satisfies:
\be\lab{nyc25}
\norm{f^j_1}_{L^\infty_{u_\nu}L^2_t L^\infty(P_{t,u_\nu})}\les \ep,
\ee
where the scalar $f^j_2$ satisfies:
\be\lab{nyc26}
\norm{f^j_2}_{L^\infty_u\lh{2}}\les \ep 2^{-\frac{j}{2}}.
\ee
Together with \eqref{nyc23}, this yields:
\bea\lab{nyc27}
&&\normm{\int_{\S} L(P_{\leq \frac{j}{2}}\trc)\left(2^{\frac{j}{2}}(N-N_\nu)\right)^pF_{j,-1}(u)\eta_j^\nu(\o)d\o}_{L^2(\MM)}\\
\nn &\les& \normm{{\chi_2}_\nu\c\int_{\S} (2\chi_1+\hch)\left(2^{\frac{j}{2}}(N-N_\nu)\right)^pF_{j,-1}(u)\eta_j^\nu(\o)d\o}_{L^2(\MM)}\\
\nn&&+\normm{f^j_1\int_{\S}\left(2^{\frac{j}{2}}(N-N_\nu)\right)^pF_{j,-1}(u)\eta_j^\nu(\o)d\o}_{L^2(\MM)}\\
\nn&&+\normm{\int_{\S}f^j_2\left(2^{\frac{j}{2}}(N-N_\nu)\right)^pF_{j,-1}(u)\eta_j^\nu(\o)d\o}_{L^2(\MM)}
+ \ep 2^{\frac{j}{4}}\gamma^\nu_j.
\eea
Next, we evaluate the various terms in the right-hand side of \eqref{nyc27} starting with the first one. We have:
\bea\lab{nyc28}
&& \normm{{\chi_2}_\nu\c\int_{\S} (2\chi_1+\hch)\left(2^{\frac{j}{2}}(N-N_\nu)\right)^pF_{j,-1}(u)\eta_j^\nu(\o)d\o}_{L^2(\MM)}\\
\nn&\les & \norm{\chi_2}_{L^6(\MM)}\normm{\int_{\S} (2\chi_1+\hch)\left(2^{\frac{j}{2}}(N-N_\nu)\right)^pF_{j,-1}(u)\eta_j^\nu(\o)d\o}_{L^3(\MM)}\\
\nn&\les & \ep\normm{\int_{\S} (2\chi_1+\hch)\left(2^{\frac{j}{2}}(N-N_\nu)\right)^pF_{j,-1}(u)\eta_j^\nu(\o)d\o}_{L^3(\MM)},
\eea
where we used in the last inequality the Sobolev embedding \eqref{sobineq} and the estimate \eqref{dechch1} for $\chi_2$. Interpolating \eqref{cannes51} and \eqref{succulent}, we obtain:
$$\normm{\int_{\S} (2\chi_1+\hch)\left(2^{\frac{j}{2}}(N-N_\nu)\right)^pF_j(u)\eta_j^\nu(\o)d\o}_{L^3(\MM)}\les 2^{\frac{5j}{12}}\ep\gamma^\nu_j.$$
Together with \eqref{nyc28}, this yields:
\be\lab{nyc29}
 \normm{{\chi_2}_\nu\c\int_{\S} (2\chi_1+\hch)\left(2^{\frac{j}{2}}(N-N_\nu)\right)^pF_{j,-1}(u)\eta_j^\nu(\o)d\o}_{L^2(\MM)}\les  2^{\frac{5j}{12}}\ep\gamma^\nu_j.
\ee

Next, we evaluate the second term in the right-hand side of \eqref{nyc27}. We have:
\bea\lab{nyc30}
&&\normm{f^j_1\int_{\S}\left(2^{\frac{j}{2}}(N-N_\nu)\right)^pF_{j,-1}(u)\eta_j^\nu(\o)d\o}_{L^2(\MM)}\\
\nn&\les & \ep\normm{\int_{\S}\left(2^{\frac{j}{2}}(N-N_\nu)\right)^pF_{j,-1}(u)\eta_j^\nu(\o)d\o}_{L^2_{u_\nu, {x'}_\nu}L^\infty_t}\\
\nn &\les & (1+p^2)\ep\gamma^\nu_j,
\eea
where we used the estimate \eqref{nyc25} for $f^j_1$, and the estimate \eqref{loebbis} to bound the $L^2_{u_\nu, {x'}_\nu}L^\infty_t$ norm.

Next, we evaluate the third term in the right-hand side of \eqref{nyc27}. In view of the basic estimate in $L^2(\MM)$ \eqref{oscl2bis}, we have:
\bea\lab{nyc31}
&&\normm{\int_{\S}f^j_2\left(2^{\frac{j}{2}}(N-N_\nu)\right)^pF_{j,-1}(u)\eta_j^\nu(\o)d\o}_{L^2(\MM)}\\
\nn&\les & \left(\sup_\o\normm{f^j_2\left(2^{\frac{j}{2}}(N-N_\nu)\right)^p}_{\li{\infty}{2}}\right) 2^{\frac{j}{2}}\gamma^\nu_j\\
\nn&\les & \ep\gamma^\nu_j,
\eea
where we used in the last inequality the estimate \eqref{nyc26} for $f^j_2$, the estimate \eqref{estNomega} for $\po N$, and the size of the patch. Finally, \eqref{nyc27}, \eqref{nyc29}, \eqref{nyc30} and \eqref{nyc31} imply:
\be\lab{nyc32bis}
\normm{\int_{\S} L(P_{\leq \frac{j}{2}}\trc)\left(2^{\frac{j}{2}}(N-N_\nu)\right)^pF_{j,-1}(u)\eta_j^\nu(\o)d\o}_{L^2(\MM)}\les (1+p^2) 2^{\frac{5j}{12}}\ep\gamma^\nu_j.
\ee
Now, in view of \eqref{nyc23} in the case $l>j/2$ and \eqref{nyc32bis} in the case $l=j/2$, we finally obtain for any $l\geq j/2$:
\be\lab{nyc32}
\normm{\int_{\S} L(P_l\trc)\left(2^{\frac{j}{2}}(N-N_\nu)\right)^pF_{j,-1}(u)\eta_j^\nu(\o)d\o}_{L^2(\MM)}\les (1+p^2) 2^{\frac{5j}{12}}\ep\gamma^\nu_j.
\ee

Next, we evaluate the second term in the right-hand side of \eqref{nyc20}. In view of the basic estimate in $L^2(\MM)$ \eqref{oscl2bis}, we have:
\bea\lab{nyc33}
&&\normm{\int_{\S}H_1\left(2^{\frac{j}{2}}(N'-N_{\nu'})\right)^qF_{j,-1}(u')\eta_j^{\nu'}(\o')d\o'}_{L^2(\MM)}\\
\nn&\les& \left(\sup_{\o'}\normm{H_1\left(2^{\frac{j}{2}}(N'-N_{\nu'})\right)^q}_{\li{\infty}{2}}\right)2^{\frac{j}{2}}\ep\gamma^{\nu'}_j\\
\nn&\les& \left(\sup_{\o'}\norm{H_1}_{\li{\infty}{2}}\right)2^{\frac{j}{2}}\ep\gamma^{\nu'}_j,
\eea
where we used in the last inequality the estimate \eqref{estNomega} for $\po N$ and the size of the patch. In view of the definition \eqref{nyc13} of $H_1$, we have:
\bea\lab{tomater}
\norm{H_1}_{\li{\infty}{2}}&\les & \norm{N'(P_m\trc')}_{\li{\infty}{2}}+\norm{\nabb'(P_m\trc')}_{\li{\infty}{2}}+\norm{\th' P_m\trc'}_{\li{\infty}{2}}\\
\nn&\les & \norm{P_m(b'N'(\trc'))}_{\li{\infty}{2}}+\norm{[b'N',P_m]\trc'}_{\li{\infty}{2}}\\
\nn&&+\norm{\nabb'\trc'}_{\li{\infty}{2}}+\norm{\th'}_{\li{\infty}{2}} \norm{P_m\trc'}_{L^\infty}\\
\nn&\les &\ep
\eea
where we used the finite band property and the boundedness on $L^p(\ptu)$ for $P_m$, the commutator estimate \eqref{commlp1} for $[b'N',P_m]$, and the estimates \eqref{estb} for $b$, \eqref{estk} \eqref{esttrc} \eqref{esthch} for $\th$, and \eqref{esttrc} for $\trc$. In view of \eqref{nyc33}, this yields:
\be\lab{nyc34}
\normm{\int_{\S}H_1\left(2^{\frac{j}{2}}(N'-N_{\nu'})\right)^qF_{j,-1}(u')\eta_j^{\nu'}(\o')d\o'}_{L^2(\MM)}\les  2^{\frac{j}{2}}\ep\gamma^{\nu'}_j.
\ee
Finally, \eqref{nyc20}, \eqref{nyc32} and \eqref{nyc34} imply:
\be\lab{nyc35}
\norm{h_{1,p,q,l,m}}_{L^1(\MM)}\les (1+p^2) 2^{\frac{11j}{12}}\ep^2\gamma^\nu_j\gamma^{\nu'}_j.
\ee

Next, we evaluate the $L^1(\MM)$ norm of $h_{2,p,q,l,m}$. In view of \eqref{nyc18}, we have:
\bea\lab{nyc36}
\norm{h_{2,p,q,l,m}}_{L^1(\MM)}&\les& \normm{\int_{\S} H_2\left(2^{\frac{j}{2}}(N-N_\nu)\right)^pF_{j,-1}(u)\eta_j^\nu(\o)d\o}_{L^2(\MM)}\\
\nn&&\times\normm{\int_{\S}L'(P_m\trc') \left(2^{\frac{j}{2}}(N'-N_{\nu'})\right)^qF_{j,-1}(u')\eta_j^{\nu'}(\o')d\o'}_{L^2(\MM)}.
\eea
Next, we evaluate both terms in the right-hand side of \eqref{nyc20} starting with the first one. In view of the definition \eqref{nyc14} of $H_2$, and proceeding as in \eqref{tomater}, we have:
$$\norm{H_2}_{\li{\infty}{2}}\les \norm{N(P_l\trc)}_{\li{\infty}{2}}+\norm{\trt P_l\trc}_{\li{\infty}{2}}\les \ep.$$
Thus, proceeding as in \eqref{nyc34}, we obtain:
\be\lab{nyc37}
\normm{\int_{\S} H_2\left(2^{\frac{j}{2}}(N-N_\nu)\right)^pF_{j,-1}(u)\eta_j^\nu(\o)d\o}_{L^2(\MM)}\les  2^{\frac{j}{2}}\ep\gamma^\nu_j.
\ee
Also, the analog of \eqref{nyc32} yields:
$$\normm{\int_{\S}L'(P_m\trc') \left(2^{\frac{j}{2}}(N'-N_{\nu'})\right)^qF_{j,-1}(u')\eta_j^{\nu'}(\o')d\o'}_{L^2(\MM)}\les (1+q^2) 2^{\frac{5j}{12}}\ep\gamma^{\nu'}_j.$$
Together with \eqref{nyc36} and \eqref{nyc37}, we deduce:
\be\lab{nyc38}
\norm{h_{2,p,q,l,m}}_{L^1(\MM)}\les (1+q^2) 2^{\frac{11j}{12}}\ep^2\gamma^\nu_j\gamma^{\nu'}_j.
\ee

Next, we evaluate the $L^1(\MM)$ norm of $h_{3,p,q,l,m}$. In view of \eqref{nyc19}, we have:
\bea\lab{nyc39}
&&\norm{h_{3,p,q,l,m}}_{L^1(\MM)}\\
\nn&\les& \normm{\int_{\S} P_l\trc \left(2^{\frac{j}{2}}(N-N_\nu)\right)^pF_{j,-1}(u)\eta_j^\nu(\o)d\o}_{L^2(\MM)}\\
\nn&&\times\normm{\int_{\S}({b'}^{-1}\nabb (b')+\th') L'(P_m\trc')\left(2^{\frac{j}{2}}(N'-N_{\nu'})\right)^qF_{j,-1}(u')\eta_j^{\nu'}(\o')d\o'}_{L^2(\MM)}\\
\nn&& +\normm{\int_{\S} \trt L(P_l\trc) \left(2^{\frac{j}{2}}(N-N_\nu)\right)^pF_{j,-1}(u)\eta_j^\nu(\o)d\o}_{L^2(\MM)}\\\nn&&\times\normm{\int_{\S} P_m\trc'\left(2^{\frac{j}{2}}(N'-N_{\nu'})\right)^qF_{j,-1}(u')\eta_j^{\nu'}(\o')d\o'}_{L^2(\MM)}.
\eea
Next, we estimate the various terms in the right-hand side of \eqref{nyc39} starting with the first one. Assume first that $l>j/2$. Then, the basic estimate in $L^2(\MM)$ \eqref{oscl2bis} yields:
\bee
&&\normm{\int_{\S} P_l\trc \left(2^{\frac{j}{2}}(N-N_\nu)\right)^pF_{j,-1}(u)\eta_j^\nu(\o)d\o}_{L^2(\MM)}\\
\nn&\les& \left(\sup_{\o}\normm{P_l\trc\left(2^{\frac{j}{2}}(N-N_\nu)\right)^p}_{\li{\infty}{2}}\right)2^{\frac{j}{2}}\gamma^\nu_j\\
\nn&\les& \left(\sup_{\o}\norm{P_l\trc}_{\li{\infty}{2}}\right)2^{\frac{j}{2}}\gamma^\nu_j,
\eee
where we used in the last inequality the estimate \eqref{estNomega} for $\po N$ and the size of the patch. Together with the finite band property for $P_l$, this yields:
\bea\lab{nyc40}
&&\normm{\int_{\S} P_l\trc \left(2^{\frac{j}{2}}(N-N_\nu)\right)^pF_{j,-1}(u)\eta_j^\nu(\o)d\o}_{L^2(\MM)}\\
\nn&\les& 2^{-l}\left(\sup_{\o}\norm{\nabb\trc}_{\li{\infty}{2}}\right)2^{\frac{j}{2}}\gamma^\nu_j\\
\nn&\les& \ep 2^{\frac{j}{2}-l}\gamma^\nu_j,
\eea
where we used the estimate \eqref{esttrc} for $\trc$ in the last inequality.

Next, we evaluate the first term in the right-hand side of \eqref{nyc39} in the case $l=j/2$, which is given by:
$$\normm{\int_{\S} P_{\leq j/2}\trc \left(2^{\frac{j}{2}}(N-N_\nu)\right)^pF_{j,-1}(u)\eta_j^\nu(\o)d\o}_{L^2(\MM)}.$$
We first decompose $P_{\leq j/2}\trc$ as:
$$P_{\leq j/2}\trc=\trc-\sum_{l>\frac{j}{2}}P_l\trc,$$
which together with \eqref{nyc40} yields:
\bee
&&\normm{\int_{\S} P_{\leq \frac{j}{2}}\trc \left(2^{\frac{j}{2}}(N-N_\nu)\right)^pF_{j,-1}(u)\eta_j^\nu(\o)d\o}_{L^2(\MM)}\\
\nn &\les& \normm{\int_{\S} \trc\left(2^{\frac{j}{2}}(N-N_\nu)\right)^pF_{j,-1}(u)\eta_j^\nu(\o)d\o}_{L^2(\MM)}+ \sum_{l>\frac{j}{2}}\ep 2^{\frac{j}{2}-l}\gamma^\nu_j\\
\nn &\les& \normm{\int_{\S} \trc\left(2^{\frac{j}{2}}(N-N_\nu)\right)^pF_{j,-1}(u)\eta_j^\nu(\o)d\o}_{L^2(\MM)}+ \ep\gamma^\nu_j.
\eee
Together with the estimate \eqref{loebter}, this yields:
\be\lab{nyc41}
\normm{\int_{\S} P_{\leq \frac{j}{2}}\trc \left(2^{\frac{j}{2}}(N-N_\nu)\right)^pF_{j,-1}(u)\eta_j^\nu(\o)d\o}_{L^2(\MM)}\les (1+p^2)\ep\gamma^\nu_j.
\ee
Now, in view of \eqref{nyc40} in the case $l>j/2$ and \eqref{nyc41} in the case $l=j/2$, we finally obtain for any $l\geq l/2$:
\be\lab{nyc42}
\normm{\int_{\S} P_l\trc\left(2^{\frac{j}{2}}(N-N_\nu)\right)^pF_{j,-1}(u)\eta_j^\nu(\o)d\o}_{L^2(\MM)}\les (1+p^2)\ep\gamma^\nu_j.
\ee
Arguing similarly for the third term in the right-hand side of \eqref{nyc39}, we obtain: 
$$\normm{\int_{\S} P_m\trc'\left(2^{\frac{j}{2}}(N'-N_{\nu'})\right)^qF_{j,-1}(u')\eta_j^{\nu'}(\o')d\o'}_{L^2(\MM)}\les (1+q^2)\ep\gamma^{\nu'}_j,$$
which together with \eqref{nyc39} and \eqref{nyc42} yields:
\bea\lab{nyc43}
&&\norm{h_{3,p,q,l,m}}_{L^1(\MM)}\\
\nn&\les& (1+p^2)\ep\gamma^\nu_j\normm{\int_{\S}(\th'+{b'}^{-1}\nabb (b')) L'(P_m\trc')\left(2^{\frac{j}{2}}(N'-N_{\nu'})\right)^qF_{j,-1}(u')\eta_j^{\nu'}(\o')d\o'}_{L^2(\MM)}\\
\nn&& +(1+q^2)\ep\gamma^{\nu'}_j\normm{\int_{\S} \trt L(P_l\trc) \left(2^{\frac{j}{2}}(N-N_\nu)\right)^pF_{j,-1}(u)\eta_j^\nu(\o)d\o}_{L^2(\MM)}.
\eea

We estimate the two terms in the right-hand side of \eqref{nyc43} starting with the first one.  Using the basic estimate in $L^2(\MM)$ \eqref{oscl2bis}, we have:
\bea\lab{nyc44}
&&\normm{\int_{\S}({b'}^{-1}\nabb (b')+\th') L'(P_m\trc')\left(2^{\frac{j}{2}}(N'-N_{\nu'})\right)^qF_{j,-1}(u')\eta_j^{\nu'}(\o')d\o'}_{L^2(\MM)}\\
\nn&\les& \left(\sup_\o\normm{(\nabb'(L'(P_m\trc'))+({b'}^{-1}\nabb (b')+\th') L'(P_m\trc'))\left(2^{\frac{j}{2}}(N'-N_{\nu'})\right)^q}_{\li{\infty}{2}}\right) 2^{\frac{j}{2}}\gamma^{\nu'}_j\\
\nn&\les& \left(\sup_\o\normm{(\th'+{b'}^{-1}\nabb (b')) L'(P_m\trc')}_{\li{\infty}{2}}\right) 2^{\frac{j}{2}}\gamma^{\nu'}_j,
\eea
where we used in the last inequality the estimate \eqref{estNomega} for $\po N$ and the size of the patch. Now, using the estimate \eqref{estn} for $n$, we have for any tensor $G$ and any integer $r$:
\bee
&&\norm{GL P_r(\trc)}_{\li{\infty}{2}}\\
&\les& \norm{G nLP_r(\trc)}_{\li{\infty}{2}}\\
&\les & \norm{G P_r(nL\trc)}_{\li{\infty}{2}}+\norm{G[nL,P_r](\trc)}_{\li{\infty}{2}}\\
&\les & \norm{G}_{\li{\infty}{6}}\norm{P_r(nL\trc)}_{\li{\infty}{3}}+\norm{G}_{\tx{\infty}{4}}\norm{[nL,P_r](\trc)}_{\tx{2}{4}}.
\eee
Together with the embeddings \eqref{sobineq} and \eqref{sobineq1}, the $L^p$ boundedness of $P_r$, the Gagliardo-Nirenberg inequality \eqref{eq:GNirenberg} and the estimate \eqref{estn} for $n$, we obtain:
\bee
&&\norm{GLP_r(\trc)}_{\li{\infty}{2}}\\
&\les&  \no(G)(\norm{L\trc}_{\li{\infty}{3}}+\norm{[nL,P_r](\trc)}_{\li{\infty}{2}}^{\frac{1}{2}}\norm{\nabb[nL,P_r](\trc)}^{\frac{1}{2}}_{\li{\infty}{2}}).
\eee
Together with the commutator estimate \eqref{commlp3ter}, we deduce:
\be\lab{raviole}
\norm{GLP_r(\trc)}_{\li{\infty}{2}}\les  (\norm{L\trc}_{\li{\infty}{3}}+\ep)\no(G)\les \ep\no(G),
\ee
where we used the fact that:
\be\lab{raviole1}
\norm{L\trc}_{\li{\infty}{3}}\les \ep,
\ee
in view of the Raychaudhuri equation \eqref{raychaudhuri}, and the estimates \eqref{estn} \eqref{estk} for $\db$ and \eqref{esttrc} \eqref{esthch} for $\chi$. Choosing $G=\th'+{b'}^{-1}\nabb (b')$, we obtain:
\be\lab{nyc46}
\normm{(\th'+{b'}^{-1}\nabb (b')) L'(P_m\trc')}_{\li{\infty}{2}}\les  \ep\no(\th'))\les\ep,
\ee
where we used in the last inequality the estimates \eqref{estn} \eqref{estk} \eqref{esttrc} \eqref{esthch} for $\th'$.  Together with \eqref{nyc44}, we obtain:
\be\lab{nyc47}
\normm{\int_{\S}(\th'+{b'}^{-1}\nabb (b')) L'(P_m\trc')\left(2^{\frac{j}{2}}(N'-N_{\nu'})\right)^qF_{j,-1}(u')\eta_j^{\nu'}(\o')d\o'}_{L^2(\MM)}\les \ep 2^{\frac{j}{2}}\gamma^{\nu'}_j.
\ee
Arguing similarly, we obtain:
\be\lab{nyc48}
\normm{\int_{\S} \trt L(P_l\trc) \left(2^{\frac{j}{2}}(N-N_\nu)\right)^pF_{j,-1}(u)\eta_j^\nu(\o)d\o}_{L^2(\MM)}\les \ep 2^{\frac{j}{2}}\gamma^{\nu}_j.
\ee
Finally, \eqref{nyc39}, \eqref{nyc47}, and \eqref{nyc48} imply:
\be\lab{nyc49}
\norm{h_{3,p,q,l,m}}_{L^1(\MM)}\les (1+p^2+q^2)2^{\frac{j}{2}}\ep^2\gamma^\nu_j\gamma^{\nu'}_j.
\ee

Next, we evaluate the $L^1(\MM)$ norm of $h_{4,p,q,l,m}$. In view of \eqref{nyc19}, we have:
\bea\lab{nyc50}
\norm{h_{4,p,q,l,m}}_{L^1(\MM)}&\les& \normm{\int_{\S} P_l\trc \left(2^{\frac{j}{2}}(N-N_\nu)\right)^pF_{j,-1}(u)\eta_j^\nu(\o)d\o}_{L^2(\MM)}\\
\nn&&\times\normm{\int_{\S}\nabb' L'(P_m\trc')\left(2^{\frac{j}{2}}(N'-N_{\nu'})\right)^qF_{j,-1}(u')\eta_j^{\nu'}(\o')d\o'}_{L^2(\MM)}.
\eea
Let us first estimate the last term in the right-hand side of \eqref{nyc50}. Using the basic estimate in $L^2(\MM)$ \eqref{oscl2bis}, we have:
\bea\lab{nyc51}
&&\normm{\int_{\S}\nabb' L'(P_m\trc')\left(2^{\frac{j}{2}}(N'-N_{\nu'})\right)^qF_{j,-1}(u')\eta_j^{\nu'}(\o')d\o'}_{L^2(\MM)}\\
\nn&\les& \left(\sup_\o\normm{(\nabb'(L'(P_m\trc'))+\th' L'(P_m\trc'))\left(2^{\frac{j}{2}}(N'-N_{\nu'})\right)^q}_{\li{\infty}{2}}\right) 2^{\frac{j}{2}}\gamma^{\nu'}_j\\
\nn&\les& \left(\sup_\o\normm{\nabb' L'(P_m\trc')}_{\li{\infty}{2}}\right) 2^{\frac{j}{2}}\gamma^{\nu'}_j,
\eea
where we used in the last inequality the estimate \eqref{estNomega} for $\po N$ and the size of the patch. Using the estimate \eqref{estn} for $n$, we estimate the right-hand side of \eqref{nyc51}:
\bee
\nn\normm{\nabb'(L'(P_m\trc'))}_{\li{\infty}{2}}&\les & \norm{\nabb'P_m(nL'\trc')}_{\li{\infty}{2}}+\norm{\nabb'[nL',P_m](\trc')}_{\li{\infty}{2}}\\
&&+\norm{n^{-1}\nabb'n L'P_m(\trc')}_{\li{\infty}{2}}.
\eee
Applying  \eqref{raviole} with the choice $G=n^{-1}\nabb'n$, we obtain:
\bea
\nn\normm{\nabb'(L'(P_m\trc'))}_{\li{\infty}{2}}&\les & \norm{\nabb'P_m(nL'\trc')}_{\li{\infty}{2}}+\norm{\nabb'[nL',P_m](\trc')}_{\li{\infty}{2}}\\
\nn&&+\ep\no(n^{-1}\nabb'n)\\
\lab{nyc52}&\les & \norm{\nabb'P_m(nL'\trc')}_{\li{\infty}{2}}+2^{\frac{m}{2}}\ep,
\eea
where we used in the last inequality the estimate \eqref{estn} for $n$ and the commutator estimate \eqref{commlp3ter}. Next, using the finite band property for $P_m$, we have:
\bee
\normm{\nabb'(P_m(nL'\trc'))}_{\li{\infty}{2}}&\les & 2^m\norm{P_m(nL'\trc')}_{\li{\infty}{2}}\\
\nn&\les& 2^{\frac{m}{2}}\ep,
\eee
where we used \eqref{celeri} in the last inequality. Together with \eqref{nyc52}, we obtain:
\be\lab{youkoulele}
\normm{\nabb'(L'(P_m\trc'))}_{\li{\infty}{2}}\les  2^{\frac{m}{2}}\ep,
\ee
which in view of \eqref{nyc51} yields:
$$\normm{\int_{\S}\nabb' L'(P_m\trc')\left(2^{\frac{j}{2}}(N'-N_{\nu'})\right)^qF_{j,-1}(u')\eta_j^{\nu'}(\o')d\o'}_{L^2(\MM)}\les \ep  2^{\frac{m}{2}+\frac{j}{2}}\gamma^{\nu'}_j.$$
Together with \eqref{nyc50}, we obtain:
\bea\lab{nyc53}
\norm{h_{4,p,q,l,m}}_{L^1(\MM)}&\les& \ep  2^{\frac{m}{2}+\frac{j}{2}}\gamma^{\nu'}_j\normm{\int_{\S} P_l\trc \left(2^{\frac{j}{2}}(N-N_\nu)\right)^pF_{j,-1}(u)\eta_j^\nu(\o)d\o}_{L^2(\MM)}.
\eea
Assume first that $l>j/2$. Then, \eqref{nyc40} and \eqref{nyc53} yield:
$$\norm{h_{4,p,q,l,m}}_{L^1(\MM)}\les \ep^2  2^{j+\frac{m}{2}-l}\gamma^\nu_j\gamma^{\nu'}_j,$$
which together with the fact that $l>m$ from \eqref{uso1} and the assumption $l>j/2$ yields:
\be\lab{nyc54}
\norm{h_{4,p,q,l,m}}_{L^1(\MM)}\les \ep^2  2^{\frac{3j}{4}}\gamma^\nu_j\gamma^{\nu'}_j.
\ee
Next, assume that $l=j/2$. Then, \eqref{nyc41} and \eqref{nyc53} yield:
$$\norm{h_{4,p,q,l,m}}_{L^1(\MM)}\les (1+p^2)\ep^2 2^{\frac{m}{2}+\frac{j}{2}}\gamma^\nu_j\gamma^{\nu'}_j,$$
which together with the fact that $l\geq m$ from \eqref{uso1} and the assumption $l=j/2$ yields:
\be\lab{nyc55}
\norm{h_{4,p,q,l,m}}_{L^1(\MM)}\les (1+p^2)\ep^2  2^{\frac{3j}{4}}\gamma^\nu_j\gamma^{\nu'}_j.
\ee
In view of \eqref{nyc54} and \eqref{nyc55}, we finally obtain in all cases:
\be\lab{nyc56}
\norm{h_{4,p,q,l,m}}_{L^1(\MM)}\les (1+p^2)\ep^2  2^{\frac{3j}{4}}\gamma^\nu_j\gamma^{\nu'}_j.
\ee

We are now ready to estimate $B^{1,2,1,2}_{j,\nu,\nu',l,m}$. In view of \eqref{nyc15}, we have:
\bee
|B^{1,2,1,2}_{j,\nu,\nu',l,m}| &\les & 2^{-j}\sum_{p, q\geq 0}c_{pq}\normm{\frac{1}{(2^{\frac{j}{2}}|N_\nu-N_{\nu'}|)^{p+q+2}}}_{L^\infty(\MM)}\\
\nn&&\times\left[\norm{h_{1,p,q,l,m}}_{L^1(\MM)}+\norm{h_{2,p,q,l,m}}_{L^1(\MM)}+\norm{h_{3,p,q,l,m}}_{L^1(\MM)}+\norm{h_{4,p,q,l,m}}_{L^1(\MM)}\right],
\eee
which together with \eqref{nice26}, \eqref{nyc35}, \eqref{nyc38}, \eqref{nyc49} and \eqref{nyc56} yields:
\bea\lab{nyc57}
|B^{1,2,1,2}_{j,\nu,\nu',l,m}| &\les & \sum_{p, q\geq 0}c_{pq}\frac{1}{(2^{\frac{j}{2}}|\nu-\nu'|)^{p+q+2}}(1+p^2+q^2) 2^{-\frac{j}{12}}\ep^2\gamma^\nu_j\gamma^{\nu'}_j\\
\nn&\les & \frac{2^{-\frac{j}{12}}\ep^2\gamma^\nu_j\gamma^{\nu'}_j}{(2^{\frac{j}{2}}|\nu-\nu'|)^2}.
\eea
Note that summing the estimate \eqref{nyc57} in $m$ is not a problem. Indeed, we have from \eqref{uso1}:
$$2^m\leq 2^j|\nu-\nu'|.$$
Now, we have:
\be\lab{nyc57bis}
\#\{m\,/\,2^m\leq 2^j|\nu-\nu'|\}\les j
\ee
so that the sum in $m$ generates a $O(j)$ term which is absorbed by the extra gain $2^{-\frac{j}{12}}$ in \eqref{nyc57}. On the other hand, there is no a priori bound on $l$ so that summing the estimate \eqref{nyc57} in $l$ is 
 problematic. To fix this issue, it suffices, since $l\geq m$ in view of \eqref{uso1}, to obtain an upper bound for
 $$\left|\sum_{l/l\geq m}B^{1,2,1,2}_{j,\nu,\nu',l,m}\right|.$$
To this end, it suffices to replace $P_l$ with $P_{\geq m}$ in the definition and the estimate of $h_{1,p,q,l,m}, h_{2,p,q,l,m}, h_{3,p,q,l,m}, h_{4,p,q,l,m}$. The estimates are completely analogous, and we obtain as in \eqref{nyc57}:
\be\lab{nyc58}
\left|\sum_{l/l\geq m}B^{1,2,1,2}_{j,\nu,\nu',l,m}\right|\les  \frac{2^{-\frac{j}{12}}\ep^2\gamma^\nu_j\gamma^{\nu'}_j}{(2^{\frac{j}{2}}|\nu-\nu'|)^2}.
\ee
Now, we have:
$$\left|\sum_{(l,m)/m<l\textrm{ and }2^m\leq 2^j|\nu-\nu'|}B^{1,2,1,2}_{j,\nu,\nu',l,m}\right|\les \sum_{2^m\leq 2^j|\nu-\nu'|}\left|\sum_{l/l\geq m}B^{1,2,1,2}_{j,\nu,\nu',l,m}\right|$$
which together with \eqref{nyc57bis} and \eqref{nyc58} implies: 
\be\lab{nyc58bis}
\left|\sum_{(l,m)/m<l\textrm{ and }2^m\leq 2^j|\nu-\nu'|}B^{1,2,1,2}_{j,\nu,\nu',l,m}\right|\les \frac{j2^{-\frac{j}{12}}\ep^2\gamma^\nu_j\gamma^{\nu'}_j}{(2^{\frac{j}{2}}|\nu-\nu'|)^2}
\ee

Next, we estimate $B^{1,2,1,3}_{j,\nu,\nu',l,m}$ defined in \eqref{nyc9ter}. Recall from \eqref{uso1} that $(l, m)$ satisfy:
$$m<l\textrm{ and }2^m\leq 2^j|\nu-\nu'|.$$
Summing in $(l,m)$, we obtain:
\bee
&& \sum_{(l,m)/m<l\textrm{ and }2^m\leq 2^j|\nu-\nu'|}L(P_l\trc)P_m\trc'+P_l\trc L'(P_m\trc')\\
&&+\sum_{(l,m)/m<l\textrm{ and }2^m\leq 2^j|\nu-\nu'|}L(P_l\trc)P_m\trc'+P_l\trc L'(P_m\trc')\\
&=& L(\trc)\trc'+ \trc L'(\trc')\\
&&-L(P_{>2^j|\nu-\nu'|}\trc)P_{>2^j|\nu-\nu'|}\trc'+ P_{>2^j|\nu-\nu'|}\trc L'(P_{>2^j|\nu-\nu'|}\trc').
\eee
Thus, using the symmetry in $(\o, \o')$ of the integrant in $B^{1,2,1,3}_{j,\nu,\nu',l,m}$ defined in \eqref{nyc9ter}, we obtain:
\bea\lab{nyc59}
&&\sum_{(l,m)/m<l\textrm{ and }2^m\leq 2^j|\nu-\nu'|} B^{1,2,1,3}_{j,\nu,\nu',l,m}+ \sum_{(l,m)/m<l\textrm{ and }2^m\leq 2^j|\nu-\nu'|}B^{1,2,1,3}_{j,\nu',\nu,l,m}\\
\nn&=& 2^{-2j-1}\int_{\MM}\int_{\S\times\S} \frac{b^{-1}\nabb_{N'-\gn N}(b)}{\gl^2}(L(\trc)\trc'+ \trc L'(\trc'))\\
\nn&&\times F_{j,-1}(u)F_{j,-1}(u')\eta_j^\nu(\o)\eta_j^{\nu'}(\o')d\o d\o' d\MM\\
\nn&& -2^{-2j-1}\int_{\MM}\int_{\S\times\S} \frac{b^{-1}\nabb_{N'-\gn N}(b)}{\gl^2}\\
\nn&&\times(L(P_{>2^j|\nu-\nu'|}\trc)P_{>2^j|\nu-\nu'|}\trc'+ P_{>2^j|\nu-\nu'|}\trc L'(P_{>2^j|\nu-\nu'|}\trc'))\\
\nn&&\times F_{j,-1}(u)F_{j,-1}(u')\eta_j^\nu(\o)\eta_j^{\nu'}(\o')d\o d\o' d\MM+\textrm{ terms interverting }(\nu, \nu').
\eea
We estimate the two terms in the right-hand side of \eqref{nyc59} starting with the last one. We have:
\bee
&&\Bigg|2^{-2j-1}\int_{\MM}\int_{\S\times\S} \frac{b^{-1}\nabb_{N'-\gn N}(b)}{\gl^2}(L(P_{>2^j|\nu-\nu'|}\trc)P_{>2^j|\nu-\nu'|}\trc'\\
\nn&&+ P_{>2^j|\nu-\nu'|}\trc L'(P_{>2^j|\nu-\nu'|}\trc')) F_{j,-1}(u)F_{j,-1}(u')\eta_j^\nu(\o)\eta_j^{\nu'}(\o')d\o d\o' d\MM\Bigg|\\
&\les & 2^{-2j}\int_{\S\times\S} \normm{\frac{N'-\gn N}{\gl^2}}_{L^\infty(\MM)}\Big(\normm{b^{-1}\nabb(b)L(P_{>2^j|\nu-\nu'|}\trc)F_{j,-1}(u)}_{L^2(\MM)}\\
&&\times\normm{P_{>2^j|\nu-\nu'|}\trc'F_{j,-1}(u')}_{L^2(\MM)}+\normm{b^{-1}\nabb(b) P_{>2^j|\nu-\nu'|}\trc F_{j,-1}(u)}_{L^2(\MM)}\\
&&\times\normm{L'(P_{>2^j|\nu-\nu'|}\trc')F_{j,-1}(u')}_{L^2(\MM)}\Big)\eta_j^\nu(\o)\eta_j^{\nu'}(\o')d\o d\o'.
\eee
In view of the identities \eqref{nice24} \eqref{nice25} for $\gl$ and $\gn$, and in view of the estimate \eqref{nice26}, we obtain:
\bea\lab{nyc60}
&&\Bigg|2^{-2j-1}\int_{\MM}\int_{\S\times\S} \frac{b^{-1}\nabb_{N'-\gn N}(b)}{\gl^2}(L(P_{>2^j|\nu-\nu'|}\trc)P_{>2^j|\nu-\nu'|}\trc'\\
\nn&&+ P_{>2^j|\nu-\nu'|}\trc L'(P_{>2^j|\nu-\nu'|}\trc')) F_{j,-1}(u)F_{j,-1}(u')\eta_j^\nu(\o)\eta_j^{\nu'}(\o')d\o d\o' d\MM\Bigg|\\
\nn&\les & \frac{1}{2^{\frac{j}{2}}(2^{\frac{j}{2}}|\nu-\nu'|)^3}\int_{\S\times\S}\Big(\normm{b^{-1}\nabb(b)L(P_{>2^j|\nu-\nu'|}\trc)}_{\li{\infty}{2}}\norm{F_{j,-1}(u)}_{L^2_u}\\
\nn&&\times\normm{P_{>2^j|\nu-\nu'|}\trc'}_{\lprime{\infty}{2}}\norm{F_{j,-1}(u')}_{L^2_{u'}}+\normm{b^{-1}\nabb(b) P_{>2^j|\nu-\nu'|}\trc}_{\li{\infty}{2}}\\
\nn&&\times\norm{F_{j,-1}(u)}_{L^2_u}\normm{L'(P_{>2^j|\nu-\nu'|}\trc')}_{\lprime{\infty}{2}}\norm{F_{j,-1}(u')}_{L^2_{u'}}\Big)\eta_j^\nu(\o)\eta_j^{\nu'}(\o')d\o d\o'.
\eea

Next, we evaluate the various terms in the right-hand side of \eqref{nyc60}. Choosing $G=b^{-1}\nabb (b)$ in \eqref{raviole}, we have:
\be\lab{nyc61}
\normm{b^{-1}\nabb(b)L(P_{>2^j|\nu-\nu'|}\trc)}_{\li{\infty}{2}}\les  \ep\no(b^{-1}\nabb(b))\les \ep,
\ee
where we used the estimate \eqref{estb} for $b$ in the last inequality. Also, \eqref{celeri1} together with the estimate \eqref{estn} for $n$ yields:
\be\lab{nyc62}
\normm{L'(P_{>2^j|\nu-\nu'|}\trc')}_{\lprime{\infty}{2}}\les \sum_{2^m>2^j|\nu-\nu'|}2^{-\frac{m}{2}}\ep\les \frac{\ep}{(2^j|\nu-\nu'|)^{\frac{1}{2}}}.
\ee
Using the finite band property for $P_m$, we have:
\bea\lab{nyc63}
\normm{P_{>2^j|\nu-\nu'|}\trc'}_{\lprime{\infty}{2}}&\les& \sum_{2^m>2^j|\nu-\nu'|}\normm{P_m\trc'}_{\lprime{\infty}{2}}\\
\nn&\les & \sum_{2^m>2^j|\nu-\nu'|}2^{-m}\normm{\nabb'\trc'}_{\lprime{\infty}{2}}\\
\nn&\les & \frac{\ep}{2^j|\nu-\nu'|},
\eea
where we used the estimate \eqref{esttrc} for $\trc$ in the last inequality. Also, we have:
\bee
\normm{b^{-1}\nabb(b) P_{>2^j|\nu-\nu'|}\trc}_{\li{\infty}{2}}&\les& \norm{b^{-1}\nabb(b)}_{\tx{\infty}{4}} \norm{P_{>2^j|\nu-\nu'|}\trc}_{\tx{2}{4}}\\
&\les& \no(b^{-1}\nabb(b)) \left(\sum_{2^m>2^j|\nu-\nu'|}\norm{P_m\trc}_{\tx{2}{4}}\right)\\
&\les& \ep \left(\sum_{2^m>2^j|\nu-\nu'|}2^{\frac{m}{2}}\norm{P_m\trc}_{\li{\infty}{2}}\right),
\eee
where we used the embedding \eqref{sobineq1}, the estimate \eqref{estb} for $b$ and the Bernstein inequality for $P_m$. Together with the finite band property for $P_m$, this yields:
\bea\lab{nyc64}
\normm{b^{-1}\nabb(b) P_{>2^j|\nu-\nu'|}\trc}_{\li{\infty}{2}}&\les& \ep \left(\sum_{2^m>2^j|\nu-\nu'|}2^{-\frac{m}{2}}\norm{\nabb\trc}_{\li{\infty}{2}}\right)\\
\nn&\les& \frac{\ep}{(2^j|\nu-\nu'|)^{\frac{1}{2}}},
\eea 
where we used the estimate \eqref{esttrc} for $\trc$ in the last inequality. 

Finally, \eqref{nyc60} \eqref{nyc61} \eqref{nyc62} \eqref{nyc63} \eqref{nyc64} yield:
\bea\lab{nyc65}
&&\Bigg|2^{-2j-1}\int_{\MM}\int_{\S\times\S} \frac{b^{-1}\nabb_{N'-\gn N}(b)}{\gl^2}(L(P_{>2^j|\nu-\nu'|}\trc)P_{>2^j|\nu-\nu'|}\trc'\\
\nn&&+ P_{>2^j|\nu-\nu'|}\trc L'(P_{>2^j|\nu-\nu'|}\trc')) F_{j,-1}(u)F_{j,-1}(u')\eta_j^\nu(\o)\eta_j^{\nu'}(\o')d\o d\o' d\MM\Bigg|\\
\nn&\les & \frac{\ep^2}{2^j(2^{\frac{j}{2}}|\nu-\nu'|)^4}\left(\int_{\S}\norm{F_{j,-1}(u)}_{L^2_u}\eta_j^\nu(\o) d\o\right)\left(\int_{\S}\norm{F_{j,-1}(u')}_{L^2_{u'}}\eta_j^{\nu'}(\o') d\o'\right)\\
\nn&\les & \frac{\ep^2\gamma^\nu_j\gamma^{\nu'}_j}{(2^{\frac{j}{2}}|\nu-\nu'|)^4},
\eea
where we used in the last inequality Plancherel in $\la$ for $\norm{F_{j,-1}(u)}_{L^2_u}$, Plancherel in $\la'$ for $\norm{F_{j,-1}(u')}_{L^2_{u'}}$, Cauchy Schwarz in $\o$ and $\o'$, and the size of the patches.

Next, we estimate the first term in the right-hand side of \eqref{nyc59}, which is given by:
\bea\lab{nyc66}
&& 2^{-2j-1}\int_{\MM}\int_{\S\times\S} \frac{b^{-1}\nabb_{N'-\gn N}(b)}{\gl^2}(L(\trc)\trc'+ \trc L'(\trc'))\\
\nn&&\times F_{j,-1}(u)F_{j,-1}(u')\eta_j^\nu(\o)\eta_j^{\nu'}(\o')d\o d\o' d\MM.
\eea
Recall the identities \eqref{nice24} and \eqref{nice25}:
$$\gg(L,L')=-1+\gn\textrm{ and }1-\gn=\frac{\gg(N-N',N-N')}{2}.$$
We may thus expand 
$$\frac{1}{\gl^2}$$ 
in the same fashion than \eqref{nice27}, and in view of \eqref{nyc66}, we obtain, schematically:
\bea\lab{nyc67}
&& 2^{-2j-1}\int_{\MM}\int_{\S\times\S} \frac{b^{-1}\nabb_{N'-\gn N}(b)}{\gl^2}(L(\trc)\trc'+ \trc L'(\trc'))\\
\nn&&\times F_{j,-1}(u)F_{j,-1}(u')\eta_j^\nu(\o)\eta_j^{\nu'}(\o')d\o d\o' d\MM\\
\nn &=& 2^{-\frac{j}{2}}\sum_{p, q\geq 0}c_{pq}\int_{\MM}\frac{1}{(2^{\frac{j}{2}}|N_\nu-N_{\nu'}|)^{p+q+3}}\left[h_{1,p,q}+h_{2,p,q}\right] d\MM,
\eea
where the scalar functions $h_{1,p,q}, h_{2,p,q}$ on $\MM$ are given by:
\bea\lab{nyc68}
h_{1,p,q}&=& \left(\int_{\S}b^{-1}\nabb(b)L(\trc)\left(2^{\frac{j}{2}}(N-N_\nu)\right)^pF_{j,-1}(u)\eta_j^\nu(\o)d\o\right)\\
\nn&&\times\left(\int_{\S}\trc'\left(2^{\frac{j}{2}}(N'-N_{\nu'})\right)^qF_{j,-1}(u')\eta_j^{\nu'}(\o')d\o'\right),
\eea
and:
\bea\lab{nyc69}
h_{2,p,q}&=& \left(\int_{\S}b^{-1}\nabb(b)\trc \left(2^{\frac{j}{2}}(N-N_\nu)\right)^pF_{j,-1}(u)\eta_j^\nu(\o)d\o\right)\\
\nn&&\times\left(\int_{\S}L'(\trc')\left(2^{\frac{j}{2}}(N'-N_{\nu'})\right)^qF_{j,-1}(u')\eta_j^{\nu'}(\o')d\o'\right),
\eea
and where $c_{pq}$ are explicit real coefficients such that the series 
$$\sum_{p, q\geq 0}c_{pq}x^py^q$$
has radius of convergence 1. 

Next, we estimate the $L^1(\MM)$ norm of $h_{1,p,q}, h_{2,p,q}$ starting with $h_{1,p,q}$. We have:
\bee
\norm{h_{1,p,q}}_{L^1(\MM)}&\les& (1+q^2)\ep\gamma^{\nu'}_j\normm{\int_{\S}b^{-1}\nabb(b)L(\trc)\left(2^{\frac{j}{2}}(N-N_\nu)\right)^pF_{j,-1}(u)\eta_j^\nu(\o)d\o}_{L^2(\MM)},
\eee
where we used the estimate \eqref{loebter} in the last inequality. Together with the basic estimate in $L^2(\MM)$  \eqref{oscl2bis}, we obtain:
\bea\lab{nyc70}
\norm{h_{1,p,q}}_{L^1(\MM)}&\les& (1+q^2)\ep\gamma^{\nu'}_j\left(\sup_{\o}\norm{b^{-1}\nabb(b)L(\trc)}_{\li{\infty}{2}}\right)2^{\frac{j}{2}}\gamma^\nu_j.
\eea
Next, we estimate $b^{-1}\nabb(b)L(\trc)$. We have:
\bee
\norm{b^{-1}\nabb(b)L(\trc)}_{\li{\infty}{2}}&\les& \norm{b^{-1}\nabb(b)}_{\li{\infty}{6}}\norm{L(\trc)}_{\li{\infty}{3}}\\
\nn&\les& \no(b^{-1}\nabb(b))\norm{L(\trc)}_{\li{\infty}{3}}\\
\nn&\les& \ep,
\eee
where we used the Sobolev embedding \eqref{sobineq}, the estimate \eqref{estb} for $b$ and the estimate \eqref{raviole1} for $L(\trc)$. Together with \eqref{nyc70}, we deduce:
\be\lab{nyc71}
\norm{h_{1,p,q}}_{L^1(\MM)}\les (1+q^2)\ep^2 2^{\frac{j}{2}}\gamma^\nu_j\gamma^{\nu'}_j.
\ee

Next, we estimate the $L^1(\MM)$ norm of $h_{2,p,q}$. Recall the decomposition \eqref{bbbbbb}: 
\bee
&&\int_{\S} L'(\trc')\left(2^{\frac{j}{2}}(N'-N_{\nu'})\right)^qF_{j,-1}(u')\eta_j^{\nu'}(\o')d\o' \\
\nn&= & -\chi_2\c\left(\int_{\S} (2{\chi_1}'+\hch')\left(2^{\frac{j}{2}}(N'-N_{\nu'})\right)^qF_{j,-1}(u')\eta_j^{\nu'}(\o')d\o'\right)\\
\nn&& -({\chi_2}_{\nu'}-\chi_2)\c\left(\int_{\S} (2{\chi_1}'+\hch')\left(2^{\frac{j}{2}}(N'-N_{\nu'})\right)^qF_{j,-1}(u')\eta_j^{\nu'}(\o')d\o'\right)\\
\nn&& +f^j_1\left(\int_{\S} \left(2^{\frac{j}{2}}(N'-N_{\nu'})\right)^qF_{j,-1}(u')\eta_j^{\nu'}(\o')d\o'\right)\\
\nn&&+\int_{\S} f^j_2 \left(2^{\frac{j}{2}}(N'-N_{\nu'})\right)^qF_{j,-1}(u')\eta_j^{\nu'}(\o')d\o',
\eee
where the scalar $f^j_1$ only depends on $\nu'$ and satisfies:
\be\lab{nyc72}
\norm{f^j_1}_{L^\infty_{u_{\nu'}}L^2_t L^\infty(P_{t,u_{\nu'}})}\les \ep,
\ee
where the scalar $f^j_2$ satisfies:
\be\lab{nyc73}
\norm{f^j_2}_{\lprime{\infty}{2}}\les \ep 2^{-\frac{j}{2}}.
\ee
Together with the definition \eqref{nyc69} of $h_{2,p,q}$, this yields:
\bee
\norm{h_{2,p,q}}_{L^1(\MM)}&\les& \Bigg(\normm{\int_{\S}\chi_2 b^{-1}\nabb(b)\trc \left(2^{\frac{j}{2}}(N-N_\nu)\right)^pF_{j,-1}(u)\eta_j^\nu(\o)d\o}_{L^2(\MM)}\\
\nn&&+\normm{\int_{\S}({\chi_2}_{\nu'}-\chi_2)b^{-1}\nabb(b)\trc \left(2^{\frac{j}{2}}(N-N_\nu)\right)^pF_{j,-1}(u)\eta_j^\nu(\o)d\o}_{L^2(\MM)}\Bigg)\\
\nn&&\times\normm{\int_{\S} (2{\chi_1}'+\hch')\left(2^{\frac{j}{2}}(N'-N_{\nu'})\right)^qF_{j,-1}(u')\eta_j^{\nu'}(\o')d\o'}_{L^2(\MM)}\\
\nn&&+\normm{\int_{\S}b^{-1}\nabb(b)\trc \left(2^{\frac{j}{2}}(N-N_\nu)\right)^pF_{j,-1}(u)\eta_j^\nu(\o)d\o}_{L^2(\MM)}\\
\nn&&\times\Bigg(\normm{f^j_1\int_{\S} \left(2^{\frac{j}{2}}(N'-N_{\nu'})\right)^qF_{j,-1}(u')\eta_j^{\nu'}(\o')d\o'}_{L^2(\MM)}\\
\nn&&+\normm{\int_{\S} f^j_2 \left(2^{\frac{j}{2}}(N'-N_{\nu'})\right)^qF_{j,-1}(u')\eta_j^{\nu'}(\o')d\o'}_{L^2(\MM)}\Bigg).
\eee
Together with the estimates \eqref{loebbis} and \eqref{succulent}, and the estimate \eqref{nyc72} for $f^j_1$, we obtain:
\bea\lab{nyc74}
&&\norm{h_{2,p,q}}_{L^1(\MM)}\\
\nn&\les& \Bigg(\normm{\int_{\S}\chi_2 b^{-1}\nabb(b)\trc \left(2^{\frac{j}{2}}(N-N_\nu)\right)^pF_{j,-1}(u)\eta_j^\nu(\o)d\o}_{L^2(\MM)}\\
\nn&&+\normm{\int_{\S}({\chi_2}_{\nu'}-\chi_2)b^{-1}\nabb(b)\trc \left(2^{\frac{j}{2}}(N-N_\nu)\right)^pF_{j,-1}(u)\eta_j^\nu(\o)d\o}_{L^2(\MM)}\Bigg)\ep\gamma^{\nu'}_j\\
\nn&&+\normm{\int_{\S}b^{-1}\nabb(b)\trc \left(2^{\frac{j}{2}}(N-N_\nu)\right)^pF_{j,-1}(u)\eta_j^\nu(\o)d\o}_{L^2(\MM)}\\
\nn&&\times\Bigg((1+q^2)\ep\gamma^{\nu'}_j+\normm{\int_{\S} f^j_2 \left(2^{\frac{j}{2}}(N'-N_{\nu'})\right)^qF_{j,-1}(u')\eta_j^{\nu'}(\o')d\o'}_{L^2(\MM)}\Bigg).
\eea
Next, we we apply the basic estimate in $L^2(\MM)$ to the first, the third and the last term in the right-hand side of \eqref{nyc74}. We obtain: 
\bea\lab{nyc75}
&&\norm{h_{2,p,q}}_{L^1(\MM)}\\
\nn&\les& \Bigg(\left(\sup_\o\norm{\chi_2 b^{-1}\nabb(b)\trc}_{\li{\infty}{2}}\right)2^{\frac{j}{2}}\gamma^\nu_j\\
\nn&&+\normm{\int_{\S}({\chi_2}_{\nu'}-\chi_2)b^{-1}\nabb(b)\trc \left(2^{\frac{j}{2}}(N-N_\nu)\right)^pF_{j,-1}(u)\eta_j^\nu(\o)d\o}_{L^2(\MM)}\Bigg)\gamma^{\nu'}_j\\
\nn&&+\left(\sup_\o\norm{b^{-1}\nabb(b)\trc}_{\li{\infty}{2}}\right)2^{\frac{j}{2}}\ep\gamma^\nu_j\\
\nn&&\times\Bigg((1+q^2)\ep\gamma^{\nu'}_j+\left(\sup_\o\norm{f^j_2}_{\lprime{\infty}{2}}\right)2^{\frac{j}{2}}\gamma^{\nu'}_j\Bigg).
\eea
Now, the Sobolev embedding \eqref{sobineq} :
\bee
&&\norm{\chi_2 b^{-1}\nabb(b)\trc}_{\li{\infty}{2}}+\norm{b^{-1}\nabb(b)\trc}_{\li{\infty}{2}}\\
&\les & (\norm{\chi_2}_{\li{\infty}{4}}\norm{b^{-1}\nabb(b)}_{\li{\infty}{4}}+\norm{b^{-1}\nabb(b)}_{\li{\infty}{2}}\norm{\trc}_{L^\infty(\MM)}\\
&\les & (\no(\chi_2)+1)\no(b^{-1}\nabb(b))\norm{\trc}_{L^\infty(\MM)}\\
&\les &\ep,
\eee
where we used in the last inequality the estimate \eqref{estb} for $b$, the estimate \eqref{esttrc} for $\trc$ and the estimate \eqref{dechch1} for $\chi_2$. Together with \eqref{nyc75} and the estimate \eqref{nyc73} for $f^j_2$, this yields:
\bea\lab{nyc76}
&&\norm{h_{2,p,q}}_{L^1(\MM)}\\
\nn&\les& \Bigg(2^{\frac{j}{2}}\ep\gamma^\nu_j+\normm{\int_{\S}({\chi_2}_{\nu'}-\chi_2)b^{-1}\nabb(b)\trc \left(2^{\frac{j}{2}}(N-N_\nu)\right)^pF_{j,-1}(u)\eta_j^\nu(\o)d\o}_{L^2(\MM)}\Bigg)\ep\gamma^{\nu'}_j\\
\nn&&+2^{\frac{j}{2}}(1+q^2)\ep^2\gamma^\nu_j\gamma^{\nu'}_j.
\eea
Next, we estimate the right-hand side of \eqref{nyc76}. We have:
\bee
&&\normm{\int_{\S}({\chi_2}_{\nu'}-\chi_2)b^{-1}\nabb(b)\trc \left(2^{\frac{j}{2}}(N-N_\nu)\right)^pF_{j,-1}(u)\eta_j^\nu(\o)d\o}_{L^2(\MM)}\\
&\les& \int_{\S}\normm{({\chi_2}_{\nu'}-\chi_2)b^{-1}\nabb(b)\trc \left(2^{\frac{j}{2}}(N-N_\nu)\right)^pF_{j,-1}(u)}_{L^2(\MM)}\eta_j^\nu(\o)d\o\\
&\les& \int_{\S}\norm{{\chi_2}_{\nu'}-\chi_2}_{L^3(\MM)}\norm{b^{-1}\nabb(b)}_{\li{\infty}{6}}\norm{\trc}_{L^\infty(\MM)} \normm{\left(2^{\frac{j}{2}}(N-N_\nu)\right)^p}_{L^\infty(\MM)}\\
&&\norm{F_{j,-1}(u)}_{L^6_u}\eta_j^\nu(\o)d\o,
\eee
which together with the Sobolev embedding \eqref{sobineq} yields:
\bee
&&\normm{\int_{\S}({\chi_2}_{\nu'}-\chi_2)b^{-1}\nabb(b)\trc \left(2^{\frac{j}{2}}(N-N_\nu)\right)^pF_{j,-1}(u)\eta_j^\nu(\o)d\o}_{L^2(\MM)}\\
&\les& \int_{\S}|\o-\nu'|\norm{\po\chi_2}_{L^3(\MM)}\no(b^{-1}\nabb(b))\norm{\trc}_{L^\infty(\MM)} \normm{\left(2^{\frac{j}{2}}(N-N_\nu)\right)^p}_{L^\infty(\MM)}\\
&&\norm{F_{j,-1}(u)}^{\frac{2}{3}}_{L^\infty_u}\norm{F_{j,-1}(u)}^{\frac{1}{3}}_{L^2_u}\eta_j^\nu(\o)d\o\\
&\les& |\nu-\nu'|\ep\int_{\S}\norm{F_{j,-1}(u)}^{\frac{2}{3}}_{L^\infty_u}\norm{F_{j,-1}(u)}^{\frac{1}{3}}_{L^2_u}\eta_j^\nu(\o)d\o,
\eee
where we used in the last inequality the estimate \eqref{dechch2} for $\chi_2$, the estimate \eqref{estb} for $b$, the estimate \eqref{esttrc} for $\trc$, the estimate \eqref{estNomega} for $\po N$, and the size of the patch. Using 
Cauchy Schwartz in $\la$ for $\norm{F_{j,-1}(u)}^{\frac{1}{2}}_{L^\infty_u}$, Plancherel in $u$ for $\norm{F_{j,-1}(u)}^{\frac{1}{2}}_{L^\infty_2}$, Cauchy Schwarz in $\o$ and the volume of the patch, we finally obtain:
$$\normm{\int_{\S}({\chi_2}_{\nu'}-\chi_2)b^{-1}\nabb(b)\trc \left(2^{\frac{j}{2}}(N-N_\nu)\right)^pF_{j,-1}(u)\eta_j^\nu(\o)d\o}_{L^2(\MM)}\les |\nu-\nu'|\ep 2^{\frac{5j}{6}}\gamma^\nu_j.$$
In view of \eqref{nyc76}, we deduce:
\be\lab{nyc77}
\norm{h_{2,p,q}}_{L^1(\MM)}\les (1+q^2)(1+|\nu-\nu'| 2^{\frac{j}{3}})2^{\frac{j}{2}}\ep^2\gamma^\nu_j\gamma^{\nu'}_j.
\ee

In view of \eqref{nyc67}, we have:
\bee
&& \Bigg|2^{-2j-1}\int_{\MM}\int_{\S\times\S} \frac{b^{-1}\nabb_{N'-\gn N}(b)}{\gl^2}(L(\trc)\trc'+ \trc L'(\trc'))\\
\nn&&\times F_{j,-1}(u)F_{j,-1}(u')\eta_j^\nu(\o)\eta_j^{\nu'}(\o')d\o d\o' d\MM\Bigg|\\
\nn &\les & 2^{-\frac{j}{2}}\sum_{p, q\geq 0}c_{pq}\normm{\frac{1}{(2^{\frac{j}{2}}|N_\nu-N_{\nu'}|)^{p+q+3}}}_{L\infty(\MM)}\left[\norm{h_{1,p,q}}_{L^1(\MM)}+\norm{h_{2,p,q}}_{L^1(\MM)}\right],
\eee
which together with \eqref{nice26}, \eqref{nyc71} and \eqref{nyc72} yields:
\bea\lab{nyc78}
&& \Bigg|2^{-2j-1}\int_{\MM}\int_{\S\times\S} \frac{b^{-1}\nabb_{N'-\gn N}(b)}{\gl^2}(L(\trc)\trc'+ \trc L'(\trc'))\\
\nn&&\times F_{j,-1}(u)F_{j,-1}(u')\eta_j^\nu(\o)\eta_j^{\nu'}(\o')d\o d\o' d\MM\Bigg|\\
\nn &\les & \sum_{p, q\geq 0}c_{pq}\frac{1}{(2^{\frac{j}{2}}|\nu-\nu'|)^{p+q+3}}(1+q^2)(1+|\nu-\nu'| 2^{\frac{j}{3}})\ep^2\gamma^\nu_j\gamma^{\nu'}_j\\
\nn &\les & \frac{\ep^2\gamma^\nu_j\gamma^{\nu'}_j}{(2^{\frac{j}{2}}|\nu-\nu'|)^3}+\frac{\ep^2\gamma^\nu_j\gamma^{\nu'}_j}{2^{\frac{j}{6}}(2^{\frac{j}{2}}|\nu-\nu'|)^2}.
\eea
Finally, \eqref{nyc59}, \eqref{nyc65} and \eqref{nyc78} imply:
\bea\lab{nyc79}
&&\Bigg|\sum_{(l,m)/m<l\textrm{ and }2^m\leq 2^j|\nu-\nu'|} B^{1,2,1,3}_{j,\nu,\nu',l,m}+ \sum_{(l,m)/m<l\textrm{ and }2^m\leq 2^j|\nu-\nu'|}B^{1,2,1,3}_{j,\nu',\nu,l,m}\Bigg|\\
\nn&\les& \frac{\ep^2\gamma^\nu_j\gamma^{\nu'}_j}{(2^{\frac{j}{2}}|\nu-\nu'|)^3}+\frac{\ep^2\gamma^\nu_j\gamma^{\nu'}_j}{2^{\frac{j}{6}}(2^{\frac{j}{2}}|\nu-\nu'|)^2}.
\eea
Now, recall \eqref{nyc9}:
$$B^{1,2,1}_{j,\nu,\nu',l,m} = B^{1,2,1,1}_{j,\nu,\nu',l,m}+B^{1,2,1,2}_{j,\nu,\nu',l,m}+B^{1,2,1,3}_{j,\nu,\nu',l,m}.$$
Together with the estimates \eqref{nycc23}, \eqref{nyc58bis} and \eqref{nyc79}, we finally obtain:
\bee
&&\left|\sum_{(l,m)/m<l\textrm{ and }2^m\leq 2^j|\nu-\nu'|} B^{1,2,1}_{j,\nu,\nu',l,m}+ \sum_{(l,m)/m<l\textrm{ and }2^m\leq 2^j|\nu-\nu'|}B^{1,2,1}_{j,\nu',\nu,l,m}\right|\\
\nn &\les& \frac{\ep^2\gamma^\nu_j\gamma^{\nu'}_j}{2^j |\nu-\nu'|}+\frac{j2^{-\frac{j}{12}}\ep^2\gamma^\nu_j\gamma^{\nu'}_j}{(2^{\frac{j}{2}}|\nu-\nu'|)^2}+\frac{\ep^2\gamma^\nu_j\gamma^{\nu'}_j}{(2^{\frac{j}{2}}|\nu-\nu'|)^3}.
\eee
This concludes the proof of Proposition \ref{prop:labexfsmp8}.

\subsubsection{Proof of Proposition \ref{prop:labexfsmp8} (Control of $B^{1,2,2}_{j,\nu,\nu',l,m}$)}\lab{sec:lobotomisation3}

Recall that $B^{1,2,2}_{j,\nu,\nu',l,m}$ is defined by \eqref{nyc3}:
\bee
B^{1,2,2}_{j,\nu,\nu',l,m} &=&  -i2^{-j-1}\int_{\MM}\int_{\S\times\S} \frac{1}{\gg(L,L')}\bigg(L(P_l\trc)P_m\trc'+P_l\trc L'(P_m\trc')\bigg)\\
\nn&&\times (b^{-1}-{b'}^{-1}) F_{j,-1}(u)F_j(u')\eta_j^\nu(\o)\eta_j^{\nu'}(\o')d\o d\o' d\MM.
\eee
Recall \eqref{uso1}:
$$m<l\textrm{ and }2^m\leq 2^j|\nu-\nu'|.$$
We first consider the range of $(l,m)$ such that:
$$2^m\leq 2^j|\nu-\nu'|<2^l.$$
This yields:
\bea\lab{stosur}
&&\sum_{m/2^m\leq 2^j|\nu-\nu'|<2^l}B^{1,2,2}_{j,\nu,\nu',l,m} \\
\nn&=&  -i2^{-j-1}\sum_{2^l> 2^j|\nu-\nu'|}\int_{\MM}\int_{\S\times\S} \frac{1}{\gg(L,L')}\bigg(L(P_l\trc)P_{\leq 2^j|\nu-\nu'|}\trc'+P_l\trc L'(P_{\leq 2^j|\nu-\nu'|}\trc')\bigg)\\
\nn&&\times (b^{-1}-{b'}^{-1}) F_{j,-1}(u)F_j(u')\eta_j^\nu(\o)\eta_j^{\nu'}(\o')d\o d\o' d\MM.
\eea
Next, recall the identities \eqref{nice24} and \eqref{nice25}:
$$\gg(L,L')=-1+\gn\textrm{ and }1-\gn=\frac{\gg(N-N',N-N')}{2}.$$
We may thus expand 
$$\frac{1}{\gl}$$ 
in the same fashion than \eqref{nice27}, and in view of \eqref{stosur}, we obtain, schematically:
\bea\lab{stosur1}
&&\sum_{m/2^m\leq 2^j|\nu-\nu'|<2^l}B^{1,2,2}_{j,\nu,\nu',l,m}\\
\nn &=& \sum_{p, q\geq 0}c_{pq}\int_{\MM}\frac{1}{(2^{\frac{j}{2}}|N_\nu-N_{\nu'}|)^{p+q+2}}\left[h_{1,p,q,l}+h_{2,p,q,l}+h_{3,p,q,l}+h_{4,p,q,l}\right] d\MM,
\eea
where the scalar functions $h_{1,p,q,l}, h_{2,p,q,l}, h_{3,p,q,l}, h_{4,p,q,l}$ on $\MM$ are given by:
\bea\lab{stosur2}
h_{1,p,q,l}&=& \left(\int_{\S} (b^{-1}-b^{-1}_{\nu'})L(P_l\trc)\left(2^{\frac{j}{2}}(N-N_\nu)\right)^pF_{j,-1}(u)\eta_j^\nu(\o)d\o\right)\\
\nn&&\times\left(\int_{\S}P_{\leq 2^j|\nu-\nu'|}\trc'\left(2^{\frac{j}{2}}(N'-N_{\nu'})\right)^qF_{j,-1}(u')\eta_j^{\nu'}(\o')d\o'\right),
\eea
\bea\lab{stosur3}
h_{2,p,q,l}&=& \left(\int_{\S} (b^{-1}-b^{-1}_{\nu'})P_l\trc\left(2^{\frac{j}{2}}(N-N_\nu)\right)^pF_{j,-1}(u)\eta_j^\nu(\o)d\o\right)\\
\nn&&\times\left(\int_{\S}L'(P_{\leq 2^j|\nu-\nu'|}\trc')\left(2^{\frac{j}{2}}(N'-N_{\nu'})\right)^qF_{j,-1}(u')\eta_j^{\nu'}(\o')d\o'\right),
\eea
\bea\lab{stosur4}
h_{3,p,q,l}&=& \left(\int_{\S} L(P_l\trc)\left(2^{\frac{j}{2}}(N-N_\nu)\right)^pF_{j,-1}(u)\eta_j^\nu(\o)d\o\right)\\
\nn&&\times\left(\int_{\S}(b^{-1}_{\nu'}-{b'}^{-1})P_{\leq 2^j|\nu-\nu'|}\trc'\left(2^{\frac{j}{2}}(N'-N_{\nu'})\right)^qF_{j,-1}(u')\eta_j^{\nu'}(\o')d\o'\right),
\eea
and:
\bea\lab{stosur5}
h_{4,p,q,l}&=& -\left(\int_{\S} P_l\trc\left(2^{\frac{j}{2}}(N-N_\nu)\right)^pF_{j,-1}(u)\eta_j^\nu(\o)d\o\right)\\
\nn&&\times\left(\int_{\S}(b^{-1}_{\nu'}-{b'}^{-1})L'(P_{\leq 2^j|\nu-\nu'|}\trc')\left(2^{\frac{j}{2}}(N'-N_{\nu'})\right)^qF_{j,-1}(u')\eta_j^{\nu'}(\o')d\o'\right),
\eea
and where $c_{pq}$ are explicit real coefficients such that the series 
$$\sum_{p, q\geq 0}c_{pq}x^py^q$$
has radius of convergence 1. 

Next, we evaluate the $L^1(\MM)$ norm of $h_{1,p,q,l}, h_{2,p,q,l}, h_{3,p,q,l}, h_{4,p,q,l}$ starting with $h_{1,p,q,l}$. 
We first estimate $L(P_l\trc)$. We have:
$$nL(P_l\trc)=P_l(nL\trc)+[nL,P_l]\trc$$
which together with the estimate \eqref{estn} for $n$ yields:
$$\norm{L(P_l\trc)}_{\xt{2}{1}}\les \norm{[nL,P_l]\trc}_{\tx{1}{2}}+\norm{P_l(nL\trc)}_{\xt{2}{1}}.$$
Together with the commutator estimate \eqref{commlp3} for $[nL,P_l]\trc$ and the estimate \eqref{lievremont1} for $P_l(nL\trc)$, we obtain:
\be\lab{stosur10}
\norm{L(P_l\trc)}_{\xt{2}{1}}\les 2^{-l}\ep.
\ee
Now, in view of the definition of $h_{1,p,q,l}$ \eqref{stosur2}, we have:
\bee
&&\norm{h_{1,p,q,l}}_{L^1(\MM)} \\
\nn& \les & \int_{\S} \Bigg\|(b^{-1}-b^{-1}_{\nu'})L(P_l\trc)\left(2^{\frac{j}{2}}(N-N_\nu)\right)^pF_{j,-1}(u)\\
\nn&&\times\left(\int_{\S}P_{\leq 2^j|\nu-\nu'|}\trc'\left(2^{\frac{j}{2}}(N'-N_{\nu'})\right)^qF_{j,-1}(u')\eta_j^{\nu'}(\o')d\o'\right)\Bigg\|_{L^1(\MM)}\eta_j^\nu(\o)d\o\\
\nn& \les & \int_{\S} \norm{b^{-1}-b^{-1}_{\nu'}}_{L^\infty(\MM)}\norm{L(P_l\trc)}_{\xt{2}{1}}\normm{\left(2^{\frac{j}{2}}(N-N_\nu)\right)^p}_{L^\infty(\MM)}\norm{F_{j,-1}(u)}_{L^2_u}\\
\nn&&\times\normm{\int_{\S}P_{\leq 2^j|\nu-\nu'|}\trc'\left(2^{\frac{j}{2}}(N'-N_{\nu'})\right)^qF_{j,-1}(u')\eta_j^{\nu'}(\o')d\o'}_{L^2_{u,x'}L^\infty_t}\eta_j^\nu(\o)d\o.
\eee
Together with the estimate \eqref{estricciomega} for $\po b$, the estimate \eqref{stosur10} for $L(P_l\trc)$, the estimate \eqref{estNomega} for $\po N$, and the size of the patch, we obtain:
\bea\lab{stosur11}
\norm{h_{1,p,q,l}}_{L^1(\MM)}&\les & 2^{-l}|\nu-\nu'|\ep\int_{\S}\norm{F_{j,-1}(u)}_{L^2_u}\\
\nn&\times&\normm{\int_{\S}P_{\leq 2^j|\nu-\nu'|}\trc'\left(2^{\frac{j}{2}}(N'-N_{\nu'})\right)^qF_{j,-1}(u')\eta_j^{\nu'}(\o')d\o'}_{L^2_{u,x'}L^\infty_t}\eta_j^\nu(\o)d\o.
\eea
Next, we estimate the last term in the right-hand side of \eqref{stosur11}. We have:
$$P_{\leq 2^j|\nu-\nu'|}\trc'  =  \trc'-\sum_{2^m> 2^j|\nu-\nu'|}P_m\trc'$$
which yields the decomposition:
\bee
&&\int_{\S}P_{\leq 2^j|\nu-\nu'|}\trc'\left(2^{\frac{j}{2}}(N'-N_{\nu'})\right)^qF_{j,-1}(u')\eta_j^{\nu'}(\o')d\o'\\
&=& \int_{\S}\trc'\left(2^{\frac{j}{2}}(N'-N_{\nu'})\right)^qF_{j,-1}(u')\eta_j^{\nu'}(\o')d\o' \\
&& - \sum_{2^m> 2^j|\nu-\nu'|}\int_{\S}P_m\trc'\left(2^{\frac{j}{2}}(N'-N_{\nu'})\right)^qF_{j,-1}(u')\eta_j^{\nu'}(\o')d\o'.
\eee
Together with \eqref{messi4:0} and \eqref{koko1}, we obtain:
\bea\lab{stosur12}
&&\normm{\int_{\S}P_{\leq 2^j|\nu-\nu'|}\trc'\left(2^{\frac{j}{2}}(N'-N_{\nu'})\right)^qF_{j,-1}(u')\eta_j^{\nu'}(\o')d\o'}_{L^2_{u,x'}L^\infty_t}\\
\nn&\les& (1+q^{\frac{5}{2}})\ep (2^{\frac{j}{2}}|\nu-\nu'|)\gamma^{\nu'}_j+\left(\sup_{\o'}\normm{\left(2^{\frac{j}{2}}(N'-N_{\nu'})\right)^q}_{L^\infty(\MM)}\right)\\
\nn&&\times\sum_{2^m> 2^j|\nu-\nu'|}\ep\big(2^{\frac{j}{2}}|\nu-\nu'|2^{-m+\frac{j}{2}}+(2^{\frac{j}{2}}|\nu-\nu'|)^{\frac{1}{2}}2^{-\frac{m}{2}+\frac{j}{4}}\big)\gamma^{\nu'}_j\\
\nn&\les & (1+q^{\frac{5}{2}})\ep (2^{\frac{j}{2}}|\nu-\nu'|)\gamma^{\nu'}_j,
\eea
where we used in the last inequality the estimate \eqref{estNomega} for $\po N$ and the size of the patch. Finally, \eqref{stosur11} and \eqref{stosur12} imply:
\bea\lab{stosur13}
\norm{h_{1,p,q,l}}_{L^1(\MM)}&\les & (1+q^{\frac{5}{2}})\ep^2 2^{-l}|\nu-\nu'| (2^{\frac{j}{2}}|\nu-\nu'|)\gamma^{\nu'}_j\int_{\S}\norm{F_{j,-1}(u)}_{L^2_u}\eta_j^\nu(\o)d\o\\
\nn&\les & (1+q^{\frac{5}{2}})\ep^2 2^{-l}(2^{\frac{j}{2}}|\nu-\nu'|)^2\gamma^\nu_j\gamma^{\nu'}_j,
\eea
where we used in the last inequality Plancherel in $\la$ for $\norm{F_{j,-1}(u)}_{L^2_u}$, Cauchy Swartz in $\o$ and the size of the patch.

Next, we evaluate the $L^1(\MM)$ norm of $h_{2,p,q,l}$. In view of the definition \eqref{stosur3} of $h_{2,p,q,l}$, we have:
\bea\lab{stosur14}
\norm{h_{2,p,q,l}}_{L^1(\MM)}&\les& \normm{\int_{\S} (b^{-1}-b^{-1}_{\nu'})P_l\trc\left(2^{\frac{j}{2}}(N-N_\nu)\right)^pF_{j,-1}(u)\eta_j^\nu(\o)d\o}_{L^2_{u_{\nu'}, x_{\nu'}'}L^\infty_t}\\
\nn&&\times\normm{\int_{\S}L'(P_{\leq 2^j|\nu-\nu'|}\trc')\left(2^{\frac{j}{2}}(N'-N_{\nu'})\right)^qF_{j,-1}(u')\eta_j^{\nu'}(\o')d\o'}_{L^2_{u_{\nu'}, x_{\nu'}'}L^1_t}.
\eea
Next, we estimate the two terms in the right-hand side of \eqref{stosur14} starting with the first one. Using the estimate \eqref{messi4:0} with $G=b^{-1}-b^{-1}_{\nu'}$, we have:
\bea\lab{stosur15}
&&\normm{\int_{\S} (b^{-1}-b^{-1}_{\nu'})P_l\trc\left(2^{\frac{j}{2}}(N-N_\nu)\right)^pF_{j,-1}(u)\eta_j^\nu(\o)d\o}_{L^2_{u_{\nu'}, x_{\nu'}'}L^\infty_t}\\
\nn&\les& \left(\sup_{\o}\norm{b^{-1}-b^{-1}_{\nu'}}_{L^\infty}\right)\ep\big(2^{\frac{j}{2}}|\nu-\nu'|2^{-l+\frac{j}{2}}+(2^{\frac{j}{2}}|\nu-\nu'|)^{\frac{1}{2}}2^{-\frac{l}{2}+\frac{j}{4}}\big)\gamma^\nu_j\\
\nn&\les& |\nu-\nu'|\ep\big(2^{\frac{j}{2}}|\nu-\nu'|2^{-l+\frac{j}{2}}+(2^{\frac{j}{2}}|\nu-\nu'|)^{\frac{1}{2}}2^{-\frac{l}{2}+\frac{j}{4}}\big)\gamma^\nu_j,
\eea
where we used in the last inequality the estimate \eqref{estb} for $b$ and the estimate \eqref{estricciomega} for $\po b$. Next, we estimate the second term in the right-hand side of \eqref{stosur14}. We have:
\bea\lab{stosur16}
&&\normm{\int_{\S}L'(P_{\leq 2^j|\nu-\nu'|}\trc')\left(2^{\frac{j}{2}}(N'-N_{\nu'})\right)^qF_{j,-1}(u')\eta_j^{\nu'}(\o')d\o'}_{L^2_{u_{\nu'}, x_{\nu'}'}L^1_t}\\
\nn & \les & \normm{\int_{\S}L'(\trc')\left(2^{\frac{j}{2}}(N'-N_{\nu'})\right)^qF_{j,-1}(u')\eta_j^{\nu'}(\o')d\o'}_{L^2_{u_{\nu'}, x_{\nu'}'}L^1_t}\\
\nn&&+\normm{\int_{\S}L'(P_{> 2^j|\nu-\nu'|}\trc')\left(2^{\frac{j}{2}}(N'-N_{\nu'})\right)^qF_{j,-1}(u')\eta_j^{\nu'}(\o')d\o'}_{L^2(\MM)}.
\eea
In view of \eqref{loeb1}, we have:
\be\lab{stosur17}
\normm{\int_{\S}L'(\trc')\left(2^{\frac{j}{2}}(N'-N_{\nu'})\right)^qF_{j,-1}(u')\eta_j^{\nu'}(\o')d\o'}_{L^2_{u_{\nu'}, x_{\nu'}'}L^1_t}\les (1+q^2)\ep\gamma^{\nu'}_j.
\ee
Also, the basic estimate in $L^2(\MM)$ \eqref{oscl2bis} yields:
\bea\lab{stosur18}
&&\normm{\int_{\S}L'(P_{> 2^j|\nu-\nu'|}\trc')\left(2^{\frac{j}{2}}(N'-N_{\nu'})\right)^qF_{j,-1}(u')\eta_j^{\nu'}(\o')d\o'}_{L^2(\MM)}\\
\nn&\les& \left(\sup_{\o'}\normm{L'(P_{> 2^j|\nu-\nu'|}\trc')\left(2^{\frac{j}{2}}(N'-N_{\nu'})\right)^q}_{\lprime{\infty}{2}}\right)2^{\frac{j}{2}}\gamma^{\nu'}_j.
\eea
Now, we have:
\bee
&&\normm{L'(P_{> 2^j|\nu-\nu'|}\trc')\left(2^{\frac{j}{2}}(N'-N_{\nu'})\right)^q}_{\lprime{\infty}{2}}\\
&\les& \norm{L'(P_{> 2^j|\nu-\nu'|}\trc')}_{\lprime{\infty}{2}}\normm{\left(2^{\frac{j}{2}}(N'-N_{\nu'})\right)^q}_{L^\infty(\MM)}\\
&\les&  \frac{\ep}{(2^j|\nu-\nu'|)^{\frac{1}{2}}},
\eee
where we used in the last inequality the estimate \eqref{nyc62}, the estimate \eqref{estNomega} for $\po N$ and the size of the patch. Together with \eqref{stosur18}, we obtain:
\bea\lab{stosur19}
&&\normm{\int_{\S}L'(P_{> 2^j|\nu-\nu'|}\trc')\left(2^{\frac{j}{2}}(N'-N_{\nu'})\right)^qF_{j,-1}(u')\eta_j^{\nu'}(\o')d\o'}_{L^2(\MM)}\\
\nn&\les& \frac{\ep 2^{\frac{j}{4}}}{(2^{\frac{j}{2}}|\nu-\nu'|)^{\frac{1}{2}}}\gamma^{\nu'}_j.
\eea
\eqref{stosur16}, \eqref{stosur17} and \eqref{stosur19} imply:
\bea\lab{stosur20}
&&\normm{\int_{\S}L'(P_{\leq 2^j|\nu-\nu'|}\trc')\left(2^{\frac{j}{2}}(N'-N_{\nu'})\right)^qF_{j,-1}(u')\eta_j^{\nu'}(\o')d\o'}_{L^2_{u_{\nu'}, x_{\nu'}'}L^1_t}\\
\nn& \les & (1+q^2)\ep\gamma^{\nu'}_j+\frac{\ep 2^{\frac{j}{4}}}{(2^{\frac{j}{2}}|\nu-\nu'|)^{\frac{1}{2}}}\gamma^{\nu'}_j.
\eea
Finally, \eqref{stosur14}, \eqref{stosur15} and \eqref{stosur20} yield:
\bea\lab{stosur21}
&&\norm{h_{2,p,q,l}}_{L^1(\MM)}\\
\nn&\les& \Big((2^{\frac{j}{2}}|\nu-\nu'|)^2 2^{-l}+(2^{\frac{j}{2}}|\nu-\nu'|)^{\frac{3}{2}}2^{-\frac{l}{2}-\frac{j}{4}}\Big)\left((1+q^2)+\frac{2^{\frac{j}{4}}}{(2^{\frac{j}{2}}|\nu-\nu'|)^{\frac{1}{2}}}\right)\ep^2\gamma^\nu_j\gamma^{\nu'}_j
\eea

Next, we evaluate the $L^1(\MM)$ norm of $h_{3,p,q,l}$. We decompose 
\bee
&&\int_{\S}(b^{-1}_{\nu'}-{b'}^{-1})P_{\leq 2^j|\nu-\nu'|}\trc'\left(2^{\frac{j}{2}}(N'-N_{\nu'})\right)^qF_{j,-1}(u')\eta_j^{\nu'}(\o')d\o'\\
&=& H+\int_{\S}(b^{-1}_{\nu'}-{b'}^{-1})P_{> 2^j|\nu-\nu'|}\trc'\left(2^{\frac{j}{2}}(N'-N_{\nu'})\right)^qF_{j,-1}(u')\eta_j^{\nu'}(\o')d\o',
\eee
where $H$ is given by
$$H=\int_{\S}(b^{-1}_{\nu'}-{b'}^{-1})\trc'\left(2^{\frac{j}{2}}(N'-N_{\nu'})\right)^qF_{j,-1}(u')\eta_j^{\nu'}(\o')d\o'.$$
Together with the definition \eqref{stosur4} of $h_{3,p,q,l}$, we obtain:
\bea\lab{stosur22}
&&\norm{h_{3,p,q,l}}_{L^1(\MM)}\\
\nn&\les & \normm{\int_{\S} HL(P_l\trc)\left(2^{\frac{j}{2}}(N-N_\nu)\right)^pF_{j,-1}(u)\eta_j^\nu(\o)d\o}_{L^1(\MM)}\\
\nn&&+\normm{\int_{\S} L(P_l\trc)\left(2^{\frac{j}{2}}(N-N_\nu)\right)^pF_{j,-1}(u)\eta_j^\nu(\o)d\o}_{L^2(\MM)}\\
\nn&&\times\normm{\int_{\S}(b^{-1}_{\nu'}-{b'}^{-1})P_{> 2^j|\nu-\nu'|}\trc'\left(2^{\frac{j}{2}}(N'-N_{\nu'})\right)^qF_{j,-1}(u')\eta_j^{\nu'}(\o')d\o'}_{L^2(\MM)}.
\eea
Next, we evaluate the three terms in the right-hand side of \eqref{stosur22} starting with the first one. In view of the estimate \eqref{nadal}, we have
\bea\lab{pastaefagioli}
&&\normm{\int_{\S} HL(P_l\trc)\left(2^{\frac{j}{2}}(N-N_\nu)\right)^pF_{j,-1}(u)\eta_j^\nu(\o)d\o}_{L^1(\MM)}\\
\nn&\les& \sup_{\o\in\textrm{supp}(\eta^\nu_j)}\left(\normm{H\left(2^{\frac{j}{2}}(N-N_\nu)\right)^p}_{L^2_{u, x'}L^\infty_t}\right)\ep 2^{\frac{j}{2}-l}\gamma^\nu_j\\
\nn&\les& \sup_{\o\in\textrm{supp}(\eta^\nu_j)}\left(\norm{H}_{L^2_{u, x'}L^\infty_t}\normm{\left(2^{\frac{j}{2}}(N-N_\nu)\right)^p}_{L^\infty}\right)\ep 2^{\frac{j}{2}-l}\gamma^\nu_j\\
\nn&\les& \sup_{\o\in\textrm{supp}(\eta^\nu_j)}\left(\norm{H}_{L^2_{u, x'}L^\infty_t}\right)\ep 2^{\frac{j}{2}-l}\gamma^\nu_j,
\eea
where we used in the last inequality the estimate \eqref{estNomega} for $\po N$ and the size of the patch. In view of the estimate \eqref{bis:koko1} and the definition of $H$, we have:
$$\norm{H}_{L^2_{u, x'}L^\infty_t}\les 2^{-\frac{j}{4}}(1+q^{\frac{5}{2}})\ep\big(2^{\frac{j}{2}}|\nu-\nu'|+1\big)\gamma^{\nu'}_j,$$
which together with \eqref{pastaefagioli} implies
\bea\lab{pastaefagioli1}
&&\normm{\int_{\S} HL(P_l\trc)\left(2^{\frac{j}{2}}(N-N_\nu)\right)^pF_{j,-1}(u)\eta_j^\nu(\o)d\o}_{L^1(\MM)}\\
\nn&\les& (1+q^{\frac{5}{2}})\big(2^{\frac{j}{2}}|\nu-\nu'|+1\big)\ep^22^{\frac{j}{4}-l}\gamma^\nu_j\gamma^{\nu'}_j.
\eea
Next, we evaluate the second term in the right-hand side of \eqref{stosur22}. Using the basic estimate in $L^2(\MM)$ \eqref{oscl2bis}, we have:
\bea\lab{stosur23}
&&\normm{\int_{\S} L(P_l\trc)\left(2^{\frac{j}{2}}(N-N_\nu)\right)^pF_{j,-1}(u)\eta_j^\nu(\o)d\o}_{L^2(\MM)}\\
\nn&\les& \left(\sup_\o\normm{L(P_l\trc)\left(2^{\frac{j}{2}}(N-N_\nu)\right)^p}_{\li{\infty}{2}}\right)2^{\frac{j}{2}}\gamma^\nu_j\\
\nn&\les& \left(\sup_\o\normm{L(P_l\trc)}_{\li{\infty}{2}}\right)2^{\frac{j}{2}}\gamma^\nu_j,
\eea
where we used in the last inequality the estimate \eqref{estNomega} for $\po N$ and the size of the patch. Also, \eqref{celeri1} together with the estimate \eqref{estn} for $n$ yields:
\be\lab{stosur24}
\normm{L(P_l\trc)}_{\li{\infty}{2}}\les 2^{-\frac{l}{2}}\ep,
\ee
which together with \eqref{stosur23} yields:
\be\lab{stosur25}
\normm{\int_{\S} L(P_l\trc)\left(2^{\frac{j}{2}}(N-N_\nu)\right)^pF_{j,-1}(u)\eta_j^\nu(\o)d\o}_{L^2(\MM)}\les 2^{\frac{j}{2}-\frac{l}{2}}\ep\gamma^\nu_j.
\ee
Next, we evaluate the third term in the right-hand side of \eqref{stosur22}. Using the basic estimate in $L^2(\MM)$ \eqref{oscl2}, we have:
\bea\lab{stosur28}
&&\normm{\int_{\S}(b^{-1}_{\nu'}-{b'}^{-1})P_{> 2^j|\nu-\nu'|}\trc'\left(2^{\frac{j}{2}}(N'-N_{\nu'})\right)^qF_{j,-1}(u')\eta_j^{\nu'}(\o')d\o'}_{L^2(\MM)}\\
\nn&\les& \left(\sup_{\o'}\normm{(b^{-1}_{\nu'}-{b'}^{-1})\left(2^{\frac{j}{2}}(N'-N_{\nu'})\right)^q}_{L^\infty}\right)\left(\sum_{2^m>2^j|\nu-\nu'|}2^{-m}\right)\ep 2^{\frac{j}{2}}\gamma^{\nu'}_j\\
\nn&\les & \frac{\ep 2^{-\frac{j}{2}}\gamma^{\nu'}_j}{2^{\frac{j}{2}}|\nu-\nu'|},
\eea
where we used in the last inequality the estimate \eqref{estricciomega} for $\po b$, the estimate \eqref{estNomega} for $\po N$ and the size of the patch. Finally, \eqref{stosur22}, \eqref{pastaefagioli1}, \eqref{stosur25} and \eqref{stosur28} imply:
\be\lab{stosur29}
\norm{h_{3,p,q,l}}_{L^1(\MM)}\les (1+q^{\frac{5}{2}})\big(2^{\frac{j}{2}}|\nu-\nu'|+1\big)\ep^22^{\frac{j}{4}-l}\gamma^\nu_j\gamma^{\nu'}_j+\frac{\ep^2 2^{-\frac{l}{2}}\gamma^\nu_j \gamma^{\nu'}_j}{2^{\frac{j}{2}}|\nu-\nu'|}.
\ee

Next, we evaluate the $L^1(\MM)$ norm of $h_{4,p,q,l}$. In view of the definition \eqref{stosur5} of $h_{4,p,q,l}$, we have:
\bea\lab{stosur30}
&&\norm{h_{4,p,q,l}}_{L^1(\MM)}\\
\nn&\les& \normm{\int_{\S} P_l\trc\left(2^{\frac{j}{2}}(N-N_\nu)\right)^pF_{j,-1}(u)\eta_j^\nu(\o)d\o}_{L^2(\MM)}\\
\nn&&\times\normm{\int_{\S}(b^{-1}_{\nu'}-{b'}^{-1})L'(P_{\leq 2^j|\nu-\nu'|}\trc')\left(2^{\frac{j}{2}}(N'-N_{\nu'})\right)^qF_{j,-1}(u')\eta_j^{\nu'}(\o')d\o'}_{L^2(\MM)}.
\eea
Next, we estimate the two terms in the right-hand side of \eqref{stosur30} starting with the first one. Using the basic estimate in $L^2(\MM)$ \eqref{oscl2}, we have:
\bea\lab{stosur31}
&&\normm{\int_{\S} P_l\trc\left(2^{\frac{j}{2}}(N-N_\nu)\right)^pF_{j,-1}(u)\eta_j^\nu(\o)d\o}_{L^2(\MM)}\\
\nn&\les& \left(\sup_\o \normm{\left(2^{\frac{j}{2}}(N-N_\nu)\right)^p}_{L^\infty}\right) \ep 2^{\frac{j}{2}-l}\gamma^\nu_j\\
\nn&\les& \ep 2^{\frac{j}{2}-l}\gamma^\nu_j,
\eea
where we used in the last inequality the estimate \eqref{estNomega} for $\po N$ and the size of the patch. Next, we estimate the second term in the right-hand side of \eqref{stosur30}. Using the basic estimate in $L^2(\MM)$ \eqref{oscl2bis}, we have:
\bea\lab{stosur32}
&&\normm{\int_{\S}(b^{-1}_{\nu'}-{b'}^{-1})L'(P_{\leq 2^j|\nu-\nu'|}\trc')\left(2^{\frac{j}{2}}(N'-N_{\nu'})\right)^qF_{j,-1}(u')\eta_j^{\nu'}(\o')d\o'}_{L^2(\MM)}\\
\nn&\les& \left(\sup_{\o'}\normm{(b^{-1}_{\nu'}-{b'}^{-1})L'(P_{\leq 2^j|\nu-\nu'|}\trc')\left(2^{\frac{j}{2}}(N'-N_{\nu'})\right)^q}_{\lprime{\infty}{2}}\right)2^{\frac{j}{2}}\gamma^{\nu'}_j.
\eea
Now, we have:
\bea\lab{stosur33}
&&\normm{(b^{-1}_{\nu'}-{b'}^{-1})L'(P_{\leq 2^j|\nu-\nu'|}\trc')\left(2^{\frac{j}{2}}(N'-N_{\nu'})\right)^q}_{\lprime{\infty}{2}}\\
\nn&\les& \normm{b^{-1}_{\nu'}-{b'}^{-1}}_{L^\infty}\norm{L'(P_{\leq 2^j|\nu-\nu'|}\trc')}_{\lprime{\infty}{2}}\normm{\left(2^{\frac{j}{2}}(N'-N_{\nu'})\right)^q}_{L^\infty}\\
\nn&\les& 2^{-\frac{j}{2}}\ep\norm{L'(P_{\leq 2^j|\nu-\nu'|}\trc')}_{\lprime{\infty}{2}},
\eea
where we used in the last inequality the estimate \eqref{estricciomega} for $\po b$, the estimate \eqref{estNomega} for $\po N$ and the size of the patch. Furthermore, we have:
\bee
\norm{L'(P_{\leq 2^j|\nu-\nu'|}\trc')}_{\lprime{\infty}{2}}&\les& \norm{L'(\trc')}_{\lprime{\infty}{2}}+\norm{L'(P_{> 2^j|\nu-\nu'|}\trc')}_{\lprime{\infty}{2}}\\
&\les & \ep+\frac{\ep}{(2^j|\nu-\nu'|)^{\frac{1}{2}}}
\eee
where we used in the last inequality the estimate \eqref{esttrc} for $\trc'$ and the estimate \eqref{nyc62} for $L'(P_{> 2^j|\nu-\nu'|}\trc')$. Together with \eqref{stosur33} and the fact that $2^{\frac{j}{2}}|\nu-\nu'|\gtrsim 1$, we obtain:
$$\normm{(b^{-1}_{\nu'}-{b'}^{-1})L'(P_{\leq 2^j|\nu-\nu'|}\trc')\left(2^{\frac{j}{2}}(N'-N_{\nu'})\right)^q}_{\lprime{\infty}{2}}\les 2^{-\frac{j}{2}}\ep.$$
Together with \eqref{stosur32}, this yields:
\be\lab{stosur34}
\normm{\int_{\S}(b^{-1}_{\nu'}-{b'}^{-1})L'(P_{\leq 2^j|\nu-\nu'|}\trc')\left(2^{\frac{j}{2}}(N'-N_{\nu'})\right)^qF_{j,-1}(u')\eta_j^{\nu'}(\o')d\o'}_{L^2(\MM)}\les \ep\gamma^{\nu'}_j.
\ee
Finally, \eqref{stosur30}, \eqref{stosur31} and \eqref{stosur34} yield:
\be\lab{stosur35}
\norm{h_{4,p,q,l}}_{L^1(\MM)}\les \ep^2 2^{\frac{j}{2}-l}\gamma^\nu_j\gamma^{\nu'}_j.
\ee

Finally, in view of \eqref{stosur1}, we have:
\bee
\left|\sum_{m/2^m\leq 2^j|\nu-\nu'|<2^l}B^{1,2,2}_{j,\nu,\nu',l,m}\right|&\les & \sum_{p, q\geq 0}c_{pq}\normm{\frac{1}{(2^{\frac{j}{2}}|N_\nu-N_{\nu'}|)^{p+q+2}}}_{L^\infty(\MM)}\Big[\norm{h_{1,p,q,l}}_{L^1(\MM)}\\
&&+\norm{h_{2,p,q,l}}_{L^1(\MM)}+\norm{h_{3,p,q,l}}_{L^1(\MM)}+\norm{h_{4,p,q,l}}_{L^1(\MM)}\Big],
\eee
which together with \eqref{nice26}, \eqref{stosur13}, \eqref{stosur21}, \eqref{stosur29} and \eqref{stosur35} yields:
\bee
&&\left|\sum_{m/2^m\leq 2^j|\nu-\nu'|<2^l}B^{1,2,2}_{j,\nu,\nu',l,m}\right|\\
\nn&\les & \sum_{p, q\geq 0}c_{pq}\frac{(1+q^{\frac{5}{2}})\ep^2\gamma^\nu_j\gamma^{\nu'}_j}{(2^{\frac{j}{2}}|\nu-\nu'|)^{p+q+2}}\Bigg[2^{-l}(2^{\frac{j}{2}}|\nu-\nu'|)^2\\
\nn&&+\Big((2^{\frac{j}{2}}|\nu-\nu'|)^2 2^{-l}+(2^{\frac{j}{2}}|\nu-\nu'|)^{\frac{3}{2}}2^{-\frac{l}{2}-\frac{j}{4}}\Big)\left(1+\frac{2^{\frac{j}{4}}}{(2^{\frac{j}{2}}|\nu-\nu'|)^{\frac{1}{2}}}\right)\\
\nn&&+(1+q^{\frac{5}{2}})\big(2^{\frac{j}{2}}|\nu-\nu'|+1\big)2^{\frac{j}{4}-l}+\frac{ 2^{-\frac{l}{2}}}{2^{\frac{j}{2}}|\nu-\nu'|}+  2^{\frac{j}{2}-l}\Bigg]\\
\nn&\les &\frac{\ep^2\gamma^\nu_j\gamma^{\nu'}_j}{(2^{\frac{j}{2}}|\nu-\nu'|)^2}\Bigg[2^{-l}(2^{\frac{j}{2}}|\nu-\nu'|)^2\\
\nn&&+\Big((2^{\frac{j}{2}}|\nu-\nu'|)^2 2^{-l}+(2^{\frac{j}{2}}|\nu-\nu'|)^{\frac{3}{2}}2^{-\frac{l}{2}-\frac{j}{4}}\Big)\left(1+\frac{2^{\frac{j}{4}}}{(2^{\frac{j}{2}}|\nu-\nu'|)^{\frac{1}{2}}}\right)\\
\nn&&+\big(2^{\frac{j}{2}}|\nu-\nu'|+1\big)2^{\frac{j}{4}-l}+\frac{ 2^{-\frac{l}{2}}}{2^{\frac{j}{2}}|\nu-\nu'|}+  2^{\frac{j}{2}-l}\Bigg].
\eee
Summing in $l$, we obtain:
\bea\lab{stosur36}
&&\left|\sum_{(l,m)/2^m\leq 2^j|\nu-\nu'|<2^l}B^{1,2,2}_{j,\nu,\nu',l,m}\right|\\
\nn&\les&  \left[\frac{2^{-\frac{j}{2}}}{(2^{\frac{j}{2}}|\nu-\nu'|)}+\frac{2^{-\frac{j}{4}}}{(2^{\frac{j}{2}}|\nu-\nu'|)^{\frac{3}{2}}}+\frac{2^{-\frac{j}{4}}}{(2^{\frac{j}{2}}|\nu-\nu'|)^2}+\frac{1}{(2^{\frac{j}{2}}|\nu-\nu'|)^3}\right]\ep^2\gamma^\nu_j\gamma^{\nu'}_j.
\eea

In view of \eqref{uso1}, we still need to estimate $B^{1,2,2}_{j,\nu,\nu',l,m}$ in the range of $(l,m)$ such that:
$$2^m\leq 2^l\leq 2^j|\nu-\nu'|.$$
Recall the definition  \eqref{nyc3} of $B^{1,2,2}_{j,\nu,\nu',l,m}$:
\bee
B^{1,2,2}_{j,\nu,\nu',l,m} &=&  -i2^{-j-1}\int_{\MM}\int_{\S\times\S} \frac{1}{\gg(L,L')}\bigg(L(P_l\trc)P_m\trc'+P_l\trc L'(P_m\trc')\bigg)\\
\nn&&\times (b^{-1}-{b'}^{-1}) F_{j,-1}(u)F_j(u')\eta_j^\nu(\o)\eta_j^{\nu'}(\o')d\o d\o' d\MM.
\eee
We integrate by parts tangentially using \eqref{fete}. 
\begin{lemma}\lab{lemma:vino}
Let $B^{1,2,2}_{j,\nu,\nu',l,m}$ defined by \eqref{nyc3}. Integrating by parts using \eqref{fete} yields:
\bea\lab{vino1}
B^{1,2,2}_{j,\nu,\nu',l,m}&=& 2^{-2j}\sum_{p, q\geq 0}c_{pq}\int_{\MM}\frac{1}{(2^{\frac{j}{2}}|N_\nu-N_{\nu'}|)^{p+q}}\bigg[\frac{1}{|N_\nu-N_{\nu'}|^2}h_{1,p,q,l,m}\\
\nn&&+\frac{1}{|N_\nu-N_{\nu'}|^3}h_{2,p,q,l,m})\bigg] d\MM+B^{1,2,2,1}_{j,\nu,\nu',l,m}+B^{1,2,2,2}_{j,\nu,\nu',l,m}+B^{1,2,2,3}_{j,\nu,\nu',l,m},
\eea
where $c_{pq}$ are explicit real coefficients such that the series 
$$\sum_{p, q\geq 0}c_{pq}x^py^q$$
has radius of convergence 1, where the scalar functions $h_{1,p,q,l,m}, h_{2,p,q,l,m}$ on $\MM$ are given by:
\bea\lab{vino2}
h_{1,p,q,l,m}&=& \left(\int_{\S} L(P_l\trc)\left(2^{\frac{j}{2}}(N-N_\nu)\right)^pF_{j,-1}(u)\eta_j^\nu(\o)d\o\right)\\
\nn&&\times\left(\int_{\S}N'(P_m\trc')(b'-b_\nu)\left(2^{\frac{j}{2}}(N'-N_{\nu'})\right)^qF_{j,-1}(u')\eta_j^{\nu'}(\o')d\o'\right)\\
\nn&& +\left(\int_{\S} L(P_l\trc)(b_\nu-b)\left(2^{\frac{j}{2}}(N-N_\nu)\right)^pF_{j,-1}(u)\eta_j^\nu(\o)d\o\right)\\
\nn&&\times\left(\int_{\S}N'(P_m\trc')\left(2^{\frac{j}{2}}(N'-N_{\nu'})\right)^qF_{j,-1}(u')\eta_j^{\nu'}(\o')d\o'\right),
\eea
and:
\bea\lab{vino4}
h_{2,p,q,l,m}&=& \left(\int_{\S} (\th+b^{-1}\nabb(b))L(P_l\trc)\left(2^{\frac{j}{2}}(N-N_\nu)\right)^pF_{j,-1}(u)\eta_j^\nu(\o)d\o\right)\\
\nn&&\times\left(\int_{\S}P_m\trc' (b'-b_{\nu'})\left(2^{\frac{j}{2}}(N'-N_{\nu'})\right)^qF_{j,-1}(u')\eta_j^{\nu'}(\o')d\o'\right)\\
\nn&& +\left(\int_{\S}(\th+b^{-1}\nabb(b)) P_l\trc\left(2^{\frac{j}{2}}(N-N_\nu)\right)^pF_{j,-1}(u)\eta_j^\nu(\o)d\o\right)\\
\nn&&\times\left(\int_{\S}L'(P_m\trc')(b'-b_\nu)\left(2^{\frac{j}{2}}(N'-N_{\nu'})\right)^qF_{j,-1}(u')\eta_j^{\nu'}(\o')d\o'\right)\\
\nn&& +\left(\int_{\S} (\th+b^{-1}\nabb(b))L(P_l\trc)(b_{\nu'}-b)\left(2^{\frac{j}{2}}(N-N_\nu)\right)^pF_{j,-1}(u)\eta_j^\nu(\o)d\o\right)\\
\nn&&\times\left(\int_{\S}P_m\trc'\left(2^{\frac{j}{2}}(N'-N_{\nu'})\right)^qF_{j,-1}(u')\eta_j^{\nu'}(\o')d\o'\right)\\
\nn&& +\left(\int_{\S} (\th+b^{-1}\nabb(b)) P_l\trc (b_\nu-b)\left(2^{\frac{j}{2}}(N-N_\nu)\right)^pF_{j,-1}(u)\eta_j^\nu(\o)d\o\right)\\
\nn&&\times\left(\int_{\S}L'(P_m\trc')\left(2^{\frac{j}{2}}(N'-N_{\nu'})\right)^qF_{j,-1}(u')\eta_j^{\nu'}(\o')d\o'\right),
\eea
and where $B^{1,2,2,1}_{j,\nu,\nu',l,m}, B^{1,2,2,2}_{j,\nu,\nu',l,m}, B^{1,2,2,3}_{j,\nu,\nu',l,m}$ are given, schematically, by:
\bea\lab{vino5bis}
&& B^{1,2,2,1}_{j,\nu,\nu',l,m}\\
\nn&=& 2^{-2j}\int_{\S\times \S}\frac{(N'-N)(b'-b)}{\gl^2}\Big(\nabb(L(P_l\trc))P_m\trc'+\nabb(P_l\trc)L'(P_m\trc')\\
\nn&&+L(P_l\trc)\nabb'(P_m\trc')+P_l\trc\nabb'(L'(P_m\trc'))\Big)F_{j,-1}(u)\eta_j^\nu(\o)F_{j,-1}(u')\eta_j^{\nu'}(\o') d\o d\o',
\eea
\be\lab{vino5ter}
B^{1,2,2,2}_{j,\nu,\nu',l,m}= 2^{-2j}\int_{\S\times \S} \frac{P_l\trc\nab_{N'}(L'(P_m\trc'))(b'-b)}{\gg(L,L')}F_{j,-1}(u)\eta_j^\nu(\o)F_{j,-1}(u')\eta_j^{\nu'}(\o') d\o d\o',
\ee
and:
\bea\lab{vino5quatre}
\nn B^{1,2,2,3}_{j,\nu,\nu',l,m}&=& 2^{-2j}\int_{\S\times \S}\frac{(\chi'-\chi)(b'-b)}{\gl^2}\Big(L(P_l\trc)P_m\trc'+P_l\trc L'(P_m\trc')\Big)\\
&& \times F_{j,-1}(u)\eta_j^\nu(\o)F_{j,-1}(u')\eta_j^{\nu'}(\o') d\o d\o',
\eea
\end{lemma}
The proof of lemma \ref{lemma:vino} is postponed to Appendix C. In the rest of this section, we use Lemma \ref{lemma:vino} to obtain the control of $B^{1,2,2}_{j,\nu,\nu',l,m}$.

We first estimate the $L^1(\MM)$ norm of $h_{1,p,q,l,m}, h_{2,p,q,l,m}$ starting with $h_{1,p,q,l,m}$. In view of the definition \eqref{vino2} of $h_{1,p,q,l,m}$, we have:
\bee
&&\norm{h_{1,p,q,l,m}}_{L^1(\MM)}\\
&\les & \normm{\int_{\S} L(P_l\trc)\left(2^{\frac{j}{2}}(N-N_\nu)\right)^pF_{j,-1}(u)\eta_j^\nu(\o)d\o}_{L^2(\MM)}\\
\nn&&\times\normm{\int_{\S}N'(P_m\trc')(b'-b_\nu)\left(2^{\frac{j}{2}}(N'-N_{\nu'})\right)^qF_{j,-1}(u')\eta_j^{\nu'}(\o')d\o'}_{L^2(\MM)}\\
\nn&& +\normm{\int_{\S} L(P_l\trc)(b_\nu-b)\left(2^{\frac{j}{2}}(N-N_\nu)\right)^pF_{j,-1}(u)\eta_j^\nu(\o)d\o}_{L^2(\MM)}\\
\nn&&\times\normm{\int_{\S}N'(P_m\trc')\left(2^{\frac{j}{2}}(N'-N_{\nu'})\right)^qF_{j,-1}(u')\eta_j^{\nu'}(\o')d\o'}_{L^2(\MM)}.
\eee
Using the basic estimate in $L^2(\MM)$ \eqref{oscl2bis} and the estimate \eqref{nyc32}, we obtain:
\bea\lab{vino6}
&&\norm{h_{1,p,q,l,m}}_{L^1(\MM)}\\
\nn&\les & \left(\sup_{\o'}\normm{N'(P_m\trc')(b'-b_\nu)\left(2^{\frac{j}{2}}(N'-N_{\nu'})\right)^q}_{\lprime{\infty}{2}}\right) (1+p^2)2^{\frac{11j}{12}}\ep\gamma^\nu_j\gamma^{\nu'}_j\\
\nn&& +\left(\sup_\o\normm{L(P_l\trc)(b_\nu-b)\left(2^{\frac{j}{2}}(N-N_\nu)\right)^p}_{\li{\infty}{2}}\right)\\
\nn&&\times\left(\sup_{\o'}\normm{N'(P_m\trc')\left(2^{\frac{j}{2}}(N'-N_{\nu'})\right)^q}_{\lprime{\infty}{2}}\right)2^j\gamma^\nu_j\gamma^{\nu'}_j.
\eea
Using the estimate \eqref{estb} for $b'$, we have:
$$\norm{N'(P_m\trc')}_{\lprime{\infty}{2}}\les \norm{P_m(b'N'\trc')}_{\lprime{\infty}{2}}+\norm{[b'N',P_m]\trc'}_{\lprime{\infty}{2}}$$
which together with the estimate \eqref{esttrc} for $\trc$, the commutator estimate \eqref{commlp1}, and the boundedness of $P_m$ on $L^2(P_{t,u'})$ yields:
\be\lab{paysdegalles}
\norm{N'(P_m\trc')}_{\lprime{\infty}{2}}\les\ep.
\ee 
Now, we have:
\bee
&&\normm{N'(P_m\trc')(b'-b_\nu)\left(2^{\frac{j}{2}}(N'-N_{\nu'})\right)^q}_{\lprime{\infty}{2}}\\
&\les& \normm{N'(P_m\trc')}_{\lprime{\infty}{2}}\norm{b'-b_\nu}_{L^\infty}\normm{\left(2^{\frac{j}{2}}(N'-N_{\nu'})\right)^q}_{L^\infty}\\
\nn&\les & \ep|\nu-\nu'|,
\eee
where we used in the last inequality the estimate \eqref{paysdegalles}, the estimate \eqref{estricciomega} for $\po b$, the estimate \eqref{estNomega} for $\po N$, and the size of the patch. Using the same estimates, we obtain similar estimates for the last two terms in the right-hand side of \eqref{vino6}:
$$\normm{L(P_l\trc)(b_\nu-b)\left(2^{\frac{j}{2}}(N-N_\nu)\right)^p}_{\li{\infty}{2}}\les 2^{-\frac{j}{2}}\ep,$$
and:
$$\normm{N'(P_m\trc')\left(2^{\frac{j}{2}}(N'-N_{\nu'})\right)^q}_{\lprime{\infty}{2}}\les \ep.$$
In the end, we obtain:
\bea\lab{vino7}
\norm{h_{1,p,q,l,m}}_{L^1(\MM)}&\les&  ((1+p^2)2^{\frac{j}{11}}|\nu-\nu'|+2^{\frac{j}{2}})\ep^2\gamma^\nu_j\gamma^{\nu'}_j\\
\nn&\les & (1+p^2)2^{\frac{j}{11}}\ep^2\gamma^\nu_j\gamma^{\nu'}_j, 
\eea
where we used in the last inequality the fact that $|\nu-\nu'|\les 1$.

Next, we estimate the $L^1(\MM)$ norm of $h_{2,p,q,l,m}$. In view of the definition \eqref{vino4} of $h_{2,p,q,l,m}$, we have:
\bea\lab{vino8}
&&\norm{h_{2,p,q,l,m}}_{L^1(\MM)}\\
\nn&\les& \normm{\int_{\S} (\th+b^{-1}\nabb(b))L(P_l\trc)\left(2^{\frac{j}{2}}(N-N_\nu)\right)^pF_{j,-1}(u)\eta_j^\nu(\o)d\o}_{L^2(\MM)}\\
\nn&&\times\normm{\int_{\S}P_m\trc' (b'-b_{\nu'})\left(2^{\frac{j}{2}}(N'-N_{\nu'})\right)^qF_{j,-1}(u')\eta_j^{\nu'}(\o')d\o'}_{L^2(\MM)}\\
\nn&& +\normm{\int_{\S}(\th+b^{-1}\nabb(b)) P_l\trc\left(2^{\frac{j}{2}}(N-N_\nu)\right)^pF_{j,-1}(u)\eta_j^\nu(\o)d\o}_{L^2(\MM)}\\
\nn&&\times\normm{\int_{\S}L'(P_m\trc')(b'-b_\nu)\left(2^{\frac{j}{2}}(N'-N_{\nu'})\right)^qF_{j,-1}(u')\eta_j^{\nu'}(\o')d\o'}_{L^2(\MM)}\\
\nn&& +\normm{\int_{\S} (\th+b^{-1}\nabb(b))L(P_l\trc)(b_{\nu'}-b)\left(2^{\frac{j}{2}}(N-N_\nu)\right)^pF_{j,-1}(u)\eta_j^\nu(\o)d\o}_{L^2(\MM)}\\
\nn&&\times\normm{\int_{\S}P_m\trc'\left(2^{\frac{j}{2}}(N'-N_{\nu'})\right)^qF_{j,-1}(u')\eta_j^{\nu'}(\o')d\o'}_{L^2(\MM)}\\
\nn&& +\normm{\int_{\S} (\th+b^{-1}\nabb(b)) P_l\trc (b_\nu-b)\left(2^{\frac{j}{2}}(N-N_\nu)\right)^pF_{j,-1}(u)\eta_j^\nu(\o)d\o}_{L^2(\MM)}\\
\nn&&\times\normm{\int_{\S}L'(P_m\trc')\left(2^{\frac{j}{2}}(N'-N_{\nu'})\right)^qF_{j,-1}(u')\eta_j^{\nu'}(\o')d\o'}_{L^2(\MM)}.
\eea
Using the basic estimate in $L^2(\MM)$ \eqref{oscl2bis} for the first term, the fourth term, the fifth term and the eighth term in the right-hand side of \eqref{vino8}, the estimate \eqref{facilfamille4bis} for the second term in the right-hand side of \eqref{vino8}, and the estimate \eqref{nyc42} for the sixth term in the right-hand side of \eqref{vino8}, we obtain:
\bee
&&\norm{h_{2,p,q,l,m}}_{L^1(\MM)}\\
&\les& \left(\sup_\o\normm{(\th+b^{-1}\nabb(b))L(P_l\trc)\left(2^{\frac{j}{2}}(N-N_\nu)\right)^p}_{\li{\infty}{2}}\right)2^{\frac{j}{2}}\gamma^\nu_j 2^{-\frac{j}{4}}\ep(1+q^2)\gamma^{\nu'}_j\\
\nn&& +\normm{\int_{\S}(\th+b^{-1}\nabb(b)) P_l\trc\left(2^{\frac{j}{2}}(N-N_\nu)\right)^pF_{j,-1}(u)\eta_j^\nu(\o)d\o}_{L^2(\MM)}\\
\nn&&\times \left(\sup_{\o'}\normm{L'(P_m\trc')(b'-b_\nu)\left(2^{\frac{j}{2}}(N'-N_{\nu'})\right)^q}_{\lprime{\infty}{2}}\right)2^{\frac{j}{2}}\gamma^{\nu'}_j\\
\nn&& + \left(\sup_\o\normm{(\th+b^{-1}\nabb(b))L(P_l\trc)(b_{\nu'}-b)\left(2^{\frac{j}{2}}(N-N_\nu)\right)^p}_{\li{\infty}{2}}\right)2^{\frac{j}{2}}\gamma^\nu_j\ep(1+q^2)\gamma^{\nu'}_j\\
\nn&& +\normm{\int_{\S} (\th+b^{-1}\nabb(b)) P_l\trc (b_\nu-b)\left(2^{\frac{j}{2}}(N-N_\nu)\right)^pF_{j,-1}(u)\eta_j^\nu(\o)d\o}_{L^2(\MM)}\\
\nn&&\times\left(\sup_{\o'}\normm{L'(P_m\trc')\left(2^{\frac{j}{2}}(N'-N_{\nu'})\right)^q}_{\lprime{\infty}{2}}\right)2^{\frac{j}{2}}\gamma^{\nu'}_j.
\eee
Together with the estimate \eqref{estb} for $b$, the estimates \eqref{estk} \eqref{esttrc} \eqref{esthch} for $\th$, the estimate \eqref{esttrc} for $\trc$, the estimate \eqref{nyc46} for $L(P_l\trc)$ and $L'(P_m\trc')$, the estimate \eqref{estricciomega} for $\po b$, the estimate \eqref{estNomega} for $\po N$ and the size of the patch, we obtain:
\bea\lab{vino9}
&&\norm{h_{2,p,q,l,m}}_{L^1(\MM)}\\
\nn&\les&  (1+q^2)2^{\frac{j}{4}}\ep^2\gamma^\nu_j\gamma^{\nu'}_j\\
\nn&& +\normm{\int_{\S}(\th+b^{-1}\nabb(b)) P_l\trc\left(2^{\frac{j}{2}}(N-N_\nu)\right)^pF_{j,-1}(u)\eta_j^\nu(\o)d\o}_{L^2(\MM)} 2^{\frac{j}{2}}|\nu-\nu'|\gamma^{\nu'}_j\\
\nn&& + 2^{\frac{j}{2}}|\nu-\nu'|(1+q^2)\ep^2\gamma^\nu_j\gamma^{\nu'}_j\\
\nn&& +\normm{\int_{\S} (\th+b^{-1}\nabb(b)) P_l\trc (b_\nu-b)\left(2^{\frac{j}{2}}(N-N_\nu)\right)^pF_{j,-1}(u)\eta_j^\nu(\o)d\o}_{L^2(\MM)}\ep 2^{\frac{j}{2}}\gamma^{\nu'}_j.
\eea
Next, we estimate the two $L^2(\MM)$ norms in the right-hand side of \eqref{vino9}. Recall the definition of $\th$ \eqref{def:theta}:
$$\th=\chi+k.$$
Now, since $k$ does not depend on $\o$, and in view of the decomposition \eqref{dectrcom} \eqref{dechchom} for $\chi$, and the decomposition \eqref{deczetaom} for $b^{-1}\nabb(b)$, we have the following decomposition for $\th+b^{-1}\nabb(b)$:
\be\lab{vino10}
\th+b^{-1}\nabb(b)=F^j_1+F^j_2
\ee
where the tensor $F^j_1$ only depends on $\nu$ and satisfies:
\be\lab{vino11}
\norm{F^j_1}_{L^\infty_{u_\nu}L^2_t L^8_{x'_{\nu}}}\les \ep,
\ee
and where the tensor $F^j_2$ satisfies:
\be\lab{vino12}
\norm{F^j_2}_{L^\infty_u\lh{2}}\les \ep 2^{-\frac{1}{4}}.
\ee
We estimate the first term in the right-hand side of \eqref{vino9}. In view of \eqref{vino10}, we have:
\bea\lab{vino19}
&&\normm{\int_{\S} (\th+b^{-1}\nabb(b)) P_l\trc \left(2^{\frac{j}{2}}(N-N_\nu)\right)^pF_{j,-1}(u)\eta_j^\nu(\o)d\o}_{L^2(\MM)}\\
\nn&\les& \norm{F^j_1}_{L^2_{u_\nu}L^2_t,L^8_{x'_\nu}}\normm{\int_{\S}P_l\trc \left(2^{\frac{j}{2}}(N-N_{\nu})\right)^pF_j(u)\eta^{\nu}_j(\o) d\o}_{L^2_{u_{\nu}}L^{\frac{8}{3}}_{x'_{\nu}}L^\infty_t}\\
\nn&& +\normm{\int_{\S}F^j_2P_l\trc \left(2^{\frac{j}{2}}(N-N_{\nu})\right)^pF_j(u)\eta^{\nu}_j(\o) d\o}_{L^2(\MM)}\\
\nn&\les& \ep\normm{\int_{\S}P_l\trc \left(2^{\frac{j}{2}}(N-N_{\nu})\right)^pF_j(u)\eta^{\nu}_j(\o) d\o}_{L^2_{u_{\nu}} L^{\frac{8}{3}}_{x'_{\nu}}L^\infty_t}\\
\nn&& +\normm{\int_{\S}F^j_2 P_l\trc \left(2^{\frac{j}{2}}(N-N_{\nu})\right)^pF_j(u)\eta^{\nu}_j(\o) d\o}_{L^2(\MM)},
\eea
where we used the estimate \eqref{vino11} for $F^j_1$ in the last inequality. Now, using \eqref{messi4:0} in the case $l>j/2$, and:
$$P_{\leq j/2}\trc=\trc-\sum_{l>j/2}P_l\trc$$
together with \eqref{messi4:0} and \eqref{loeb} in the case $l=j/2$, we obtain:
\be\lab{vino20}
\normm{\int_{\S}P_l\trc \left(2^{\frac{j}{2}}(N-N_{\nu})\right)^pF_j(u)\eta^{\nu}_j(\o) d\o}_{L^2_{u_{\nu}, x_{\nu}'}L^\infty_t}\les  (1+p^2)\ep\gamma^{\nu}_j.
\ee
Also, we have:
\bea
&&\nn\normm{\int_{\S}P_l\trc \left(2^{\frac{j}{2}}(N-N_{\nu})\right)^pF_j(u)\eta^{\nu}_j(\o) d\o}_{L^\infty(\MM)}\\
\nn&\les& \int_{\S}\normm{P_l\trc \left(2^{\frac{j}{2}}(N-N_{\nu})\right)^pF_j(u)}_{L^\infty}\eta_j^{\nu}(\o)d\o\\
\nn&\les & \ep \left(\int_{\S}\norm{F_{j,-1}(u)}_{L^\infty_u}\eta_j^{\nu}(\o)d\o'\right)\\
\lab{vino21}&\les& \ep 2^j\gamma^{\nu}_j,
\eea
where we used, the estimate \eqref{esttrc} for $\trc$, the estimate \eqref{estNomega} for $\po N$, Cauchy-Schwarz in $\la$ to estimate $\norm{F_{j,-1}(u)}_{L^\infty_{u}}$, Cauchy-Schwarz in $\o$ and the size of the patch. Interpolating between \eqref{vino20} and \eqref{vino21}, we obtain:
\be\lab{vino22}
\normm{\int_{\S}P_l\trc \left(2^{\frac{j}{2}}(N-N_{\nu})\right)^pF_j(u)\eta^{\nu}_j(\o) d\o}_{L^2_{u_{\nu}}L^{\frac{8}{3}}_{x'_{\nu}}L^\infty_t}\les  2^{\frac{j}{4}}\ep \gamma^{\nu}_j.
\ee
For the second term in the right-hand side of \eqref{vino19}, we have:
\bea\lab{vino23}
\nn\normm{\int_{\S}F^j_2 P_l\trc \left(2^{\frac{j}{2}}(N-N_{\nu})\right)^pF_j(u)\eta^{\nu}_j(\o) d\o}_{L^2(\MM)}&\les& \ep \int_{\S}2^{-\frac{j}{4}}\norm{F_j(u)}_{L^2_{u}}\eta^{\nu}_j(\o) d\o\\
&\les& \ep 2^{\frac{j}{4}}\gamma^{\nu}_j,
\eea
where we used in the last inequality Plancherel in $\la$, Cauchy Schwartz in $\o$ and the size of the patch. Finally, \eqref{vino19}, \eqref{vino22} and \eqref{vino23} imply:
\bea\lab{vino24}
&&\normm{\int_{\S} (\th+b^{-1}\nabb(b)) P_l\trc \left(2^{\frac{j}{2}}(N-N_\nu)\right)^pF_{j,-1}(u)\eta_j^\nu(\o)d\o}_{L^2(\MM)}\\
\nn&\les&  2^{\frac{j}{4}}\ep \gamma^{\nu}_j.
\eea

Next, we estimate the second term in the right-hand side of \eqref{vino9}. In view of \eqref{vino10}, we have:
\bea\lab{vino13}
&&\normm{\int_{\S} (\th+b^{-1}\nabb(b)) P_l\trc(b-b_\nu) \left(2^{\frac{j}{2}}(N-N_\nu)\right)^pF_{j,-1}(u)\eta_j^\nu(\o)d\o}_{L^2(\MM)}\\
\nn&\les& \norm{F^j_1}_{L^2_{u_\nu}L^2_t,L^8_{x'_\nu}}\normm{\int_{\S}P_l\trc(b-b_\nu) \left(2^{\frac{j}{2}}(N-N_{\nu})\right)^pF_j(u)\eta^{\nu}_j(\o) d\o}_{L^2_{u_{\nu}}L^{\frac{8}{3}}_{x'_{\nu}}L^\infty_t}\\
\nn&& +\normm{\int_{\S}F^j_2P_l\trc (b-b_\nu)\left(2^{\frac{j}{2}}(N-N_{\nu})\right)^pF_j(u)\eta^{\nu}_j(\o) d\o}_{L^2(\MM)}\\
\nn&\les& \ep\normm{\int_{\S}P_l\trc (b-b_\nu)\left(2^{\frac{j}{2}}(N-N_{\nu})\right)^pF_j(u)\eta^{\nu}_j(\o) d\o}_{L^2_{u_{\nu}} L^{\frac{8}{3}}_{x'_{\nu}}L^\infty_t}\\
\nn&& +\normm{\int_{\S}F^j_2 P_l\trc(b-b_\nu) \left(2^{\frac{j}{2}}(N-N_{\nu})\right)^pF_j(u)\eta^{\nu}_j(\o) d\o}_{L^2(\MM)},
\eea
where we used the estimate \eqref{vino11} for $F^j_1$ in the last inequality. Now, using \eqref{messi4:0} in the case $l>j/2$, and:
$$P_{\leq j/2}\trc=\trc-\sum_{l>j/2}P_l\trc$$
together with \eqref{messi4:0} and \eqref{bis:koko1} in the case $l=j/2$, we obtain:
\be\lab{vino14}
\normm{\int_{\S}P_l\trc (b_\nu-b)\left(2^{\frac{j}{2}}(N-N_{\nu})\right)^pF_j(u)\eta^{\nu}_j(\o) d\o}_{L^2_{u_{\nu}, x_{\nu}'}L^\infty_t}\les  2^{-\frac{j}{4}}(1+p^{\frac{5}{2}})\ep\gamma^{\nu}_j.
\ee
Also, we have:
\bea
&&\nn\normm{\int_{\S}P_l\trc (b_\nu-b)\left(2^{\frac{j}{2}}(N-N_{\nu})\right)^pF_j(u)\eta^{\nu}_j(\o) d\o}_{L^\infty(\MM)}\\
\nn&\les& \int_{\S}\normm{P_l\trc (b_\nu-b)\left(2^{\frac{j}{2}}(N-N_{\nu})\right)^pF_j(u)}_{L^\infty}\eta_j^{\nu}(\o)d\o\\
\nn&\les & \ep 2^{-\frac{j}{2}}\left(\int_{\S}\norm{F_{j,-1}(u)}_{L^\infty_u}\eta_j^{\nu}(\o)d\o'\right)\\
\lab{vino15}&\les& \ep 2^{\frac{j}{2}}\gamma^{\nu}_j,
\eea
where we used, the estimate \eqref{esttrc} for $\trc$, the estimate \eqref{estricciomega} for $\po b$, the estimate \eqref{estNomega} for $\po N$, Cauchy-Schwarz in $\la$ to estimate $\norm{F_{j,-1}(u)}_{L^\infty_{u}}$, Cauchy-Schwarz in $\o$ and the size of the patch. Interpolating between \eqref{vino14} and \eqref{vino15}, we obtain:
\be\lab{vino16}
\normm{\int_{\S}P_l\trc (b_\nu-b)\left(2^{\frac{j}{2}}(N-N_{\nu})\right)^pF_j(u)\eta^{\nu}_j(\o) d\o}_{L^2_{u_{\nu}}L^{\frac{8}{3}}_{x'_{\nu}}L^\infty_t}\les  2^{-\frac{j}{16}}\ep \gamma^{\nu}_j.
\ee
For the second term in the right-hand side of \eqref{vino13}, we have:
\bee
&& \normm{\int_{\S}F^j_2 P_l\trc (b_\nu-b)\left(2^{\frac{j}{2}}(N-N_{\nu})\right)^pF_j(u)\eta^{\nu}_j(\o) d\o}_{L^2(\MM)}\\
\nn&\les& \int_{\S}\norm{F^j_2 P_l\trc (b_\nu-b)\left(2^{\frac{j}{2}}(N-N_{\nu})\right)^pF_j(u)}_{L^2(\MM)}\eta^{\nu}_j(\o) d\o\\
\nn&\les& \int_{\S}\norm{F^j_2}_{L^\infty_uL^2(\H_{u})}\norm{F_j(u)}_{L^2_{u}}\normm{P_l\trc (b_\nu-b)\left(2^{\frac{j}{2}}(N-N_{\nu})\right)^p}_{L^\infty(\MM)}\eta^{\nu}_j(\o) d\o.
\eee
Together with the estimate \eqref{vino12} for $F^j_2$, the estimate \eqref{esttrc} for $\trc$, the estimate \eqref{estricciomega} for $b$, the estimate \eqref{estNomega} for $\po N$ and the size of the patch, we obtain:
\bea\lab{vino17}
&&\normm{\int_{\S}F^j_2 P_l\trc (b_\nu-b)\left(2^{\frac{j}{2}}(N-N_{\nu})\right)^pF_j(u)\eta^{\nu}_j(\o) d\o}_{L^2(\MM)}\\
\nn&\les& \ep 2^{-\frac{j}{2}}\int_{\S}2^{-\frac{j}{4}}\norm{F_j(u)}_{L^2_{u}}\eta^{\nu}_j(\o) d\o\\
\nn&\les& \ep 2^{-\frac{j}{4}}\gamma^{\nu}_j,
\eea
where we used in the last inequality Plancherel in $\la$, Cauchy Schwartz in $\o$ and the size of the patch. Finally, \eqref{vino13}, \eqref{vino16} and \eqref{vino17} imply:
\be\lab{vino18}
\normm{\int_{\S} (\th+b^{-1}\nabb(b)) P_l\trc (b_\nu-b)\left(2^{\frac{j}{2}}(N-N_\nu)\right)^pF_{j,-1}(u)\eta_j^\nu(\o)d\o}_{L^2(\MM)}\les \ep 2^{-\frac{j}{16}}\ep \gamma^{\nu}_j.
\ee
Finally, \eqref{vino9}, \eqref{vino18} and \eqref{vino24} yield:
\be\lab{vino25}
\norm{h_{2,p,q,l,m}}_{L^1(\MM)}\les  (1+q^2)(2^{\frac{j}{4}}(2^{\frac{j}{2}}|\nu-\nu'|)+2^{\frac{7}{16}})\ep^2\gamma^\nu_j\gamma^{\nu'}_j.
\ee

Next, we estimate $B^{1,2,2,1}_{j,\nu,\nu',l,m}$. Recall that we are considering the range of $(l,m)$:
$$2^m\leq 2^l\leq 2^j|\nu-\nu'|.$$
Summing in $(l,m)$, we have:
\bee
&&\sum_{(l,m)/2^m\leq 2^l\leq 2^j|\nu-\nu'|}\Big(\nabb(L(P_l\trc))P_m\trc'+\nabb(P_l\trc)L'(P_m\trc')\\
\nn&&+L(P_l\trc)\nabb'(P_m\trc')+P_l\trc\nabb'(L'(P_m\trc'))\Big)\\
&=& \nabb(L(P_{\leq 2^j|\nu-\nu'|}\trc))P_{\leq 2^j|\nu-\nu'|}\trc'+\nabb(P_{\leq 2^j|\nu-\nu'|}\trc)L'(P_{\leq 2^j|\nu-\nu'|}\trc')\\
\nn&&+L(P_{\leq 2^j|\nu-\nu'|}\trc)\nabb'(P_{\leq 2^j|\nu-\nu'|}\trc')+P_{\leq 2^j|\nu-\nu'|}\trc\nabb'(L'(P_{\leq 2^j|\nu-\nu'|}\trc').
\eee
Thus, using the symmetry in $(\o, \o')$ of the integrant in $B^{1,2,2,1}_{j,\nu,\nu',l,m}$, we obtain in view of the definition \eqref{vino5bis} of $B^{1,2,2,1}_{j,\nu,\nu',l,m}$:
\bea\lab{vino26}
&&\sum_{(l,m)/2^m\leq 2^l\leq 2^j|\nu-\nu'|}(B^{1,2,2,1}_{j,\nu',\nu,l,m}+B^{1,2,2,1}_{j,\nu,\nu',l,m})\\
\nn&=& 2^{-2j}\int_{\S\times \S}\frac{(N'-N)(b'-b)}{\gl^2}\Big(\nabb(L(P_{\leq 2^j|\nu-\nu'|}\trc))P_{\leq 2^j|\nu-\nu'|}\trc'\\
\nn&&+\nabb(P_{\leq 2^j|\nu-\nu'|}\trc)L'(P_{\leq 2^j|\nu-\nu'|}\trc')+L(P_{\leq 2^j|\nu-\nu'|}\trc)\nabb'(P_{\leq 2^j|\nu-\nu'|}\trc')\\
\nn&&+P_{\leq 2^j|\nu-\nu'|}\trc\nabb'(L'(P_{\leq 2^j|\nu-\nu'|}\trc'))\Big)F_{j,-1}(u)\eta_j^\nu(\o)F_{j,-1}(u')\eta_j^{\nu'}(\o') d\o d\o'.
\eea
Now, note that the right-hand side of \eqref{vino26} is the analog of the right-hand side of \eqref{stosur} provided one replaces 
$$P_l\trc,\, 2^l>2^j|\nu-\nu'|$$
with:
$$\frac{2^{-j}(N'-N)}{\gl}\nabb (P_{\leq 2^j|\nu-\nu'|}\trc).$$
We obtain the analog of \eqref{stosur36}:
\bea\lab{vino27}
&&\left|\sum_{(l,m)/2^m\leq 2^l\leq 2^j|\nu-\nu'|}(B^{1,2,2,1}_{j,\nu',\nu,l,m}+B^{1,2,2,1}_{j,\nu,\nu',l,m})\right|\\
\nn&\les&  \left[\frac{2^{-\frac{j}{2}}}{(2^{\frac{j}{2}}|\nu-\nu'|)}+\frac{2^{-\frac{j}{4}}}{(2^{\frac{j}{2}}|\nu-\nu'|)^{\frac{3}{2}}}+\frac{1}{(2^{\frac{j}{2}}|\nu-\nu'|)^3}\right]\ep^2\gamma^\nu_j\gamma^{\nu'}_j.
\eea
The proof of \eqref{vino27} is essentially the same as the proof of the estimate \eqref{stosur36} and is left to the reader. The similarity in these proofs originates from the fact that 
$$\sum_{2^l>2^j|\nu-\nu'|} P_l\trc$$
and
$$\frac{2^{-j}(N'-N)}{\gl}\nabb (P_{\leq 2^j|\nu-\nu'|}\trc)$$
satisfy the same estimates. For instance, in view of the finite band property, the identities \eqref{nice24} \eqref{nice25}, and the estimate \eqref{nice26}, we have:
$$\sum_{2^l>2^j|\nu-\nu'|}P_l\trc\sim \frac{\nabb\trc}{2^j|\nu-\nu'|}\textrm{ and }\frac{2^{-j}(N'-N)}{\gl}\nabb (P_{\leq 2^j|\nu-\nu'|}\trc)\sim \frac{\nabb\trc}{2^j|\nu-\nu'|}.$$

Next, we estimate $B^{1,2,2,2}_{j,\nu,\nu',l,m}$. Recall the definition \eqref{vino5ter} of $B^{1,2,2,2}_{j,\nu,\nu',l,m}$:
$$B^{1,2,2,2}_{j,\nu,\nu',l,m}= 2^{-2j}\int_{\S\times \S} \frac{P_l\trc N'(L'(P_m\trc'))(b'-b)}{\gg(L,L')}F_{j,-1}(u)\eta_j^\nu(\o)F_{j,-1}(u')\eta_j^{\nu'}(\o') d\o d\o'.$$
Estimating the term $N'(L'(P_m\trc'))$ would involve commutator terms which are difficult to handle. To avoid this issue, we commute $L'$ with $N'$, and then integrate the $L'$ derivative by parts. We obtain schematically in view of 
the definition \eqref{vino5ter} of $B^{1,2,2,2}_{j,\nu,\nu',l,m}$:
\bea\lab{vino28}
&& B^{1,2,2,2}_{j,\nu',\nu,l,m}\\
\nn&=& 2^{-2j}\int_{\S\times \S} \frac{P_l\trc [N',L'](P_m\trc'))(b'-b)}{\gg(L,L')}F_{j,-1}(u)\eta_j^\nu(\o)F_{j,-1}(u')\eta_j^{\nu'}(\o') d\o d\o'\\
\nn&& -2^{-2j}\int_{\S\times \S} \left(\frac{\textrm{div}_{\gg}(L')}{\gl}-\frac{L'(\gl)}{\gl^2}\right)\frac{P_l\trc N'(P_m\trc')(b'-b)}{\gg(L,L')}\\
\nn&&\times F_{j,-1}(u)\eta_j^\nu(\o)F_{j,-1}(u')\eta_j^{\nu'}(\o') d\o d\o'\\
\nn&& -2^{-2j}\int_{\S\times \S} \frac{L'(P_l\trc) N'(P_m\trc')(b'-b)}{\gg(L,L')}F_{j,-1}(u)\eta_j^\nu(\o)F_{j,-1}(u')\eta_j^{\nu'}(\o') d\o d\o'\\
\nn&& -2^{-2j}\int_{\S\times \S} \frac{P_l\trc N'(P_m\trc')(L'(b')-L'(b))}{\gg(L,L')}F_{j,-1}(u)\eta_j^\nu(\o)F_{j,-1}(u')\eta_j^{\nu'}(\o') d\o d\o'\\
\nn&& -i2^{-j}\int_{\S\times \S} b^{-1} P_l\trc N'(P_m\trc')(b'-b)F_j(u)\eta_j^\nu(\o)F_{j,-1}(u')\eta_j^{\nu'}(\o') d\o d\o',
\eea
where the last term in the right-hand side of \eqref{vino28} appears when the $L'$ derivative falls on the phase in 
view of \eqref{ibpl'}.  

We decompose $L'$ in the frame $L, N, e_A$:
\be\lab{dialobis}
L'=L+(N'-\gn N)+(\gn-1)N,
\ee
which yields the following decompositions:
\be\lab{dialo1bis}
L'(P_l\trc)= L(P_l\trc)+(N'-\gn N)(P_l\trc)+(\gn-1)N(P_l\trc),
\ee
and:
\be\lab{dialo1ter}
L'(b)= L(b)+(N'-\gn N)(b)+(\gn-1)N(b).
\ee
Recall the identities \eqref{nice24} and \eqref{nice25}:
$$\gg(L,L')=-1+\gn\textrm{ and }1-\gn=\frac{\gg(N-N',N-N')}{2}.$$
We may thus expand 
$$\frac{1}{\gl}\textrm{ and }\frac{1}{\gl^2}$$ 
in the same fashion than \eqref{nice27}, and in view of \eqref{vino28}, \eqref{dialo1bis}, \eqref{dialo1ter}, the formula \eqref{fete3bis} for $\textrm{div}_{\gg}(L')$ and the formula \eqref{fete6bis} for $L'(\gg(L,L'))$, we obtain, schematically:
\bea\lab{vino29}
&& \sum_{m/m\leq l}B^{1,2,2,2}_{j,\nu,\nu',l,m}\\
\nn &=& 2^{-j}\sum_{p, q\geq 0}c_{pq}\int_{\MM}\frac{1}{(2^{\frac{j}{2}}|N_\nu-N_{\nu'}|)^{p+q}}\Bigg[\frac{1}{(2^{\frac{j}{2}}|N_\nu-N_{\nu'}|)^2}(h_{1,p,q}+h_{2,p,q})\\
\nn&&+\frac{1}{2^{\frac{j}{2}}(2^{\frac{j}{2}}|N_\nu-N_{\nu'}|)}h_{3,p,q}+2^{-j}h_{4,p,q}\Bigg] d\MM+\sum_{m/m\leq l}(B^{1,2,2,2,1}_{j,\nu,\nu',l,m}+B^{1,2,2,2,2}_{j,\nu,\nu',l,m}+B^{1,2,2,2,3}_{j,\nu,\nu',l,m}),
\eea
where the scalar functions $h_{1,p,q}, h_{2,p,q}, h_{3,p,q}, h_{4,p,q}$ on $\MM$ are given by:
\bea\lab{vino30}
h_{1,p,q}&=& \left(\int_{\S} G_1(b-b_{\nu})^r\left(2^{\frac{j}{2}}(N-N_\nu)\right)^pF_{j,-1}(u)\eta_j^\nu(\o)d\o\right)\\
\nn&&\times\left(\int_{\S}N'(P_{\leq l}\trc')(b_{\nu}-b')^s\left(2^{\frac{j}{2}}(N'-N_{\nu'})\right)^qF_{j,-1}(u')\eta_j^{\nu'}(\o')d\o'\right),
\eea
\bea\lab{vino31}
h_{2,p,q}&=& \left(\int_{\S} P_l\trc  \left(2^{\frac{j}{2}}(N-N_\nu)\right)^pF_{j,-1}(u)\eta_j^\nu(\o)d\o\right)\\
\nn&&\times\left(\int_{\S}G_2 \left(2^{\frac{j}{2}}(N'-N_{\nu'})\right)^qF_{j,-1}(u')\eta_j^{\nu'}(\o')d\o'\right),
\eea
\bea\lab{vino32}
h_{3,p,q}&=& \left(\int_{\S} \nabb(b)P_l\trc\left(2^{\frac{j}{2}}(N-N_\nu)\right)^pF_{j,-1}(u)\eta_j^\nu(\o)d\o\right)\\
\nn&&\times\left(\int_{\S}N'(P_{\leq l}\trc')\left(2^{\frac{j}{2}}(N'-N_{\nu'})\right)^qF_{j,-1}(u')\eta_j^{\nu'}(\o')d\o'\right),
\eea
and:
\bea\lab{vino33}
h_{4,p,q}&=& \left(\int_{\S} N(b) P_l\trc\left(2^{\frac{j}{2}}(N-N_\nu)\right)^pF_{j,-1}(u)\eta_j^\nu(\o)d\o\right)\\
\nn&&\times\left(\int_{\S}N'(P_{\leq l}\trc')\left(2^{\frac{j}{2}}(N'-N_{\nu'})\right)^qF_{j,-1}(u')\eta_j^{\nu'}(\o')d\o'\right),
\eea
where the integer $r, s$ satisfy:
$$r+s=1,$$
where the tensors $G_1$ and $G_2$ are schematically given by:
\be\lab{vino34}
G_1=L(P_l\trc)+(\db+\chi+\zeta+L(b))P_l\trc,
\ee
and:
\be\lab{vino35}
G_2=[N', L'](P_{\leq l}\trc')+(\db'+L'(b')) N'(P_{\leq l}\trc'),
\ee
where $B^{1,2,2,2,1}_{j,\nu,\nu',l,m}, B^{1,2,2,2,2}_{j,\nu,\nu',l,m}, B^{1,2,2,2,3}_{j,\nu,\nu',l,m}$ are given by:
\bea\lab{vinoroja1}
\nn B^{1,2,2,2,1}_{j,\nu,\nu',l,m}&=& -2^{-2j}\int_{\MM}\int_{\S\times \S} \frac{(N'-\gn N)(P_l\trc) N'(P_m\trc')(b'-b)}{\gg(L,L')}\\
&&\times F_{j,-1}(u)\eta_j^\nu(\o)F_{j,-1}(u')\eta_j^{\nu'}(\o') d\o d\o'd\MM,
\eea
\bea\lab{vinoroja2}
&& B^{1,2,2,2,2}_{j,\nu,\nu',l,m}\\
\nn&=& -2^{-2j}\int_{\MM}\int_{\S\times \S} N(P_l\trc) N'(P_m\trc')(b'-b) F_{j,-1}(u)\eta_j^\nu(\o)F_{j,-1}(u')\eta_j^{\nu'}(\o') d\o d\o'd\MM,
\eea
\bea\lab{vinoroja3}
&& B^{1,2,2,2,3}_{j,\nu,\nu',l,m}\\
\nn &=& -i2^{-j}\int_{\MM}\int_{\S\times \S} b^{-1} P_l\trc N'(P_m\trc')(b'-b)F_j(u)\eta_j^\nu(\o)F_{j,-1}(u')\eta_j^{\nu'}(\o') d\o d\o'd\MM,
\eea
and where $c_{pq}$ are explicit real coefficients such that the series 
$$\sum_{p, q\geq 0}c_{pq}x^py^q$$
has radius of convergence 1. Note that the term $b-b'$ present in all terms of \eqref{vino28} is present in the definition of $h_{1,p,q}$, but absent from the definition of $h_{2,p,q}, h_{3,p,q}, h_{4,p,q}$. Indeed, we do not need to exploit the gain $b-b'$ in $h_{2,p,q} h_{3,p,q}, h_{4,p,q}$, and we just separate $b$ and $b'$ and estimate them in $L^\infty$ using the estimate \eqref{estb} for $b$. To simplify the notations, we chose not to specify these factors of $b$ and $b'$.

Next, we estimate the $L^1(\MM)$ norm of $h_{1,p,q}, h_{2,p,q}, h_{3,p,q}, h_{4,p,q}$ starting with $h_{1,p,q}$. We have:
\bea\lab{vino36}
\norm{h_{1,p,q}}_{L^1(\MM)}&\les& \normm{\int_{\S} G_1(b-b_{\nu})^r\left(2^{\frac{j}{2}}(N-N_\nu)\right)^pF_{j,-1}(u)\eta_j^\nu(\o)d\o}_{L^2(\MM)}\\
\nn &&\normm{\int_{\S}N'(P_{\leq l}\trc')(b_{\nu}-b')^s\left(2^{\frac{j}{2}}(N'-N_{\nu'})\right)^qF_{j,-1}(u')\eta_j^{\nu'}(\o')d\o'}_{L^2(\MM)}.
\eea
We estimate the second term in the right-hand side of \eqref{vino36}. Using the basic estimate in $L^2(\MM)$ \eqref{oscl2bis}, we have:
\bee
&&\normm{\int_{\S}N'(P_{\leq l}\trc')(b_{\nu}-b')^s\left(2^{\frac{j}{2}}(N'-N_{\nu'})\right)^qF_{j,-1}(u')\eta_j^{\nu'}(\o')d\o'}_{L^2(\MM)}\\
&\les& \left(\sup_{\o'}\normm{N'(P_{\leq l}\trc')(b_{\nu}-b')^s\left(2^{\frac{j}{2}}(N'-N_{\nu'})\right)^q}_{\lprime{\infty}{2}}\right)2^{\frac{j}{2}}\gamma^{\nu'}_j\\
&\les& \left(\sup_{\o}\norm{N'(P_{\leq l}\trc')}_{\lprime{\infty}{2}}\right) |\nu-\nu'|^s2^{\frac{j}{2}}\gamma^\nu_j,
\eee
where we used in the last inequality the estimate \eqref{estricciomega} for $\po b$, the estimate \eqref{estNomega} for $\po N$ and the size of the patch. Together with the estimate \eqref{esttrc} for $\trc$ and the commutator estimate \eqref{commlp1}, we obtain:
\be\lab{vino37}
\normm{\int_{\S}N'(P_{\leq l}\trc')(b_{\nu}-b')^s\left(2^{\frac{j}{2}}(N'-N_{\nu'})\right)^qF_{j,-1}(u')\eta_j^{\nu'}(\o')d\o'}_{L^2(\MM)}\les \ep |\nu-\nu'|^s2^{\frac{j}{2}}\gamma^\nu_j.
\ee
Next, we estimate the first term in the right-hand side of \eqref{vino36}. Using the basic estimate in $L^2(\MM)$ \eqref{oscl2bis}, we have:
\bea\lab{vino38}
&&\normm{\int_{\S} G_1(b-b_{\nu})^r\left(2^{\frac{j}{2}}(N-N_\nu)\right)^pF_{j,-1}(u)\eta_j^\nu(\o)d\o}_{L^2(\MM)}\\
\nn&\les& \left(\sup_{\o}\normm{G_1(b-b_{\nu})^r\left(2^{\frac{j}{2}}(N-N_{\nu})\right)^p}_{\li{\infty}{2}}\right)2^{\frac{j}{2}}\gamma^{\nu}_j\\
\nn&\les& \left(\sup_{\o}\norm{G_1}_{\li{\infty}{2}}\right) 2^{-\frac{rj}{2}}2^{\frac{j}{2}}\gamma^{\nu'}_j,
\eea
where we used in the last inequality the estimate \eqref{estricciomega} for $\po b$, the estimate \eqref{estNomega} for $\po N$ and the size of the patch. In view of the definition of $G_1$ \eqref{vino34}, we have:
\bea\lab{vino39}
\nn\norm{G_1}_{\li{\infty}{2}}&\les& (\norm{\db}_{\tx{\infty}{4}}+\norm{\chi}_{\tx{\infty}{4}}+\norm{\z}_{\tx{\infty}{4}}+\norm{L(b)}_{\tx{\infty}{4}})\norm{P_l\trc}_{\tx{2}{4}}\\
\nn&&+\norm{L(P_l\trc)}_{\li{\infty}{2}}\\
&\les& \ep\norm{P_l\trc}_{\tx{2}{4}}+\norm{L(P_l\trc)}_{\li{\infty}{2}},
\eea
where we used in the last inequality the embedding \eqref{sobineq1}, and the estimates \eqref{estn} \eqref{estk} for $\db$, the estimates \eqref{esttrc} \eqref{esthch} for $\chi$, the estimate \eqref{estb} for $b$, and the estimate \eqref{estzeta} for $\z$. If $l>j/2$, we have the analog of \eqref{dialo3}:
\be\lab{vino40}
\norm{P_l\trc}_{\tx{2}{4}}\les 2^{-\frac{l}{2}}\ep,
\ee
while in the case $l=j/2$, the boundedness of $P_{\leq j/2}$ on $L^4(\ptu)$ and the estimate \eqref{esttrc} for $\trc$ yields:
\be\lab{vino41}
\norm{P_{\leq j/2}\trc}_{\tx{2}{4}}\les \ep.
\ee
In view of \eqref{vino39}, we also need to estimate $L(P_l\trc)$. In the case $l>j/2$, the estimate \eqref{celeri1} yields:
\be\lab{vino42}
\norm{L(P_l\trc)}_{\li{\infty}{2}}\les 2^{-\frac{l}{2}}\ep,
\ee
which together with the estimate \eqref{esttrc} for $\trc$ and the decomposition:
$$P_{\leq j/2}\trc=\trc-\sum_{l>\frac{j}{2}}P_l\trc$$
implies in the case $l=j/2$:
\be\lab{vino43}
\norm{L(P_{\leq j/2}\trc)}_{\li{\infty}{2}}\les \ep.
\ee
Now, in view of \eqref{vino39}, \eqref{vino40}, \eqref{vino41}, \eqref{vino42} and \eqref{vino43}, we obtain:
$$\norm{G_1}_{\li{\infty}{2}}\les\ep 2^{-l(1-\delta_{l,j/2})},$$
where we defined:
$$\delta_{l,j/2}=1\textrm{ if }l=j/2\textrm{ and }\delta_{l,j/2}=0\textrm{ otherwise}.$$
Together with \eqref{vino38}, we deduce: 
\be\lab{vino44}
\normm{\int_{\S} G_1(b-b_{\nu})^r\left(2^{\frac{j}{2}}(N-N_\nu)\right)^pF_{j,-1}(u)\eta_j^\nu(\o)d\o}_{L^2(\MM)}\les \ep 2^{-\frac{rj}{2}}2^{\frac{j}{2}}2^{-l(1-\delta_{l,j/2})}\gamma^{\nu'}_j.
\ee
Finally, \eqref{vino36}, \eqref{vino37} and \eqref{vino44} imply:
$$\norm{h_{1,p,q}}_{L^1(\MM)}\les \ep^2 |\nu-\nu'|^s2^{-\frac{rj}{2}} 2^{j-l(1-\delta_{l,j/2})}\gamma^\nu_j\gamma^{\nu'}_j.$$
Since we have:
$$2^{\frac{j}{2}}|\nu-\nu'|\gtrsim 1\textrm{ and }r+s=1,$$
this yields:
\be\lab{vino45}
\norm{h_{1,p,q}}_{L^1(\MM)}\les \ep^2 |\nu-\nu'|2^{j-l(1-\delta_{l,j/2})}\gamma^\nu_j\gamma^{\nu'}_j.
\ee
 
Next, we estimate the $L^1(\MM)$ norm of $h_{2,p,q}$. In view of the definition \eqref{vino31} of $h_{2,p,q}$, we have:
\bea\lab{vino46}
\norm{h_{2,p,q}}_{L^2(\MM)}&\les& \normm{\int_{\S} P_l\trc  \left(2^{\frac{j}{2}}(N-N_\nu)\right)^pF_{j,-1}(u)\eta_j^\nu(\o)d\o}_{L^3(\MM)}\\
\nn&&\times\normm{\int_{\S}G_2 \left(2^{\frac{j}{2}}(N'-N_{\nu'})\right)^qF_{j,-1}(u')\eta_j^{\nu'}(\o')d\o'}_{L^\frac{3}{2}(\MM)},
\eea
Recall \eqref{vino21}:
\be\lab{vino47}
\normm{\int_{\S} P_l\trc \left(2^{\frac{j}{2}}(N-N_\nu)\right)^pF_{j,-1}(u)\eta_j^\nu(\o)d\o}_{L^\infty(\MM)}\les  \ep 2^j\gamma^\nu_j.
\ee
On the other hand, \eqref{nyc40} and \eqref{nyc41} imply:
\be\lab{vino48}
\normm{\int_{\S} P_l\trc  \left(2^{\frac{j}{2}}(N-N_\nu)\right)^pF_{j,-1}(u)\eta_j^\nu(\o)d\o}_{L^2(\MM)}\les (1+p^2)2^{\frac{j}{2}-l}\ep\gamma^\nu_j.
\ee
Interpolating between \eqref{vino47} and \eqref{vino48}, we obtain:
\be\lab{vino51}
\normm{\int_{\S} P_l\trc  \left(2^{\frac{j}{2}}(N-N_\nu)\right)^pF_{j,-1}(u)\eta_j^\nu(\o)d\o}_{L^3(\MM)}\les  \ep 2^{\frac{2j}{3}}2^{-\frac{2l}{3}}\gamma^\nu_j.
\ee
Next, we estimate the second term in the right-hand side of \eqref{vino46}. We have:
\bea\lab{vino52}
&&\normm{\int_{\S}G_2 \left(2^{\frac{j}{2}}(N'-N_{\nu'})\right)^qF_{j,-1}(u')\eta_j^{\nu'}(\o')d\o'}_{L^{\frac{3}{2}}(\MM)}\\
\nn &\les& \int_{\S} \norm{G_2}_{\lprime{\infty}{\frac{3}{2}}}\normm{\left(2^{\frac{j}{2}}(N'-N_{\nu'})\right)^q}_{L^\infty}\norm{F_{j,-1}(u')}_{L^2_{u'}}\eta_j^\nu(\o')d\o'\\
\nn &\les& \int_{\S} \norm{G_2}_{\lprime{\infty}{\frac{3}{2}}}\norm{F_{j,-1}(u')}_{L^2_{u'}}\eta_j^\nu(\o')d\o',
\eea
where we used in the last inequality the estimate \eqref{estNomega} for $\po N$ and the size of the patch. Next, we estimate $G_2$. In view of the definition of $G_2$ \eqref{vino35}, the commutator formulas \eqref{comm3} for $[L, \lb]$ and the fact that $2N=L-\lb$, we have schematically:
 \bee
G_2= n^{-1}\nab_{N'}n L'(P_{\leq l}\trc')+(\zeta'-\zb')\c \nabb'(P_{\leq l}\trc')+(\db'+L'(b'))N'(P_{\leq l}\trc').
\eee
This yields:
\bee
\norm{G_2}_{\lprime{\infty}{\frac{3}{2}}}&\les& (\norm{\db'}_{\lprime{\infty}{6}}+\norm{n^{-1}\nab_{N'}n}_{\lprime{\infty}{6}}+\norm{\zeta'}_{\lprime{\infty}{6}}\\
\nn&&+\norm{\zb}_{\lprime{\infty}{6}}+\norm{L'(b')}_{\lprime{\infty}{6}})\norm{\dd P_{\leq l}\trc'}_{\lprime{\infty}{2}}\\
&\les& \ep \norm{\dd P_{\leq l}\trc'}_{\lprime{\infty}{2}},
\eee
where we used in the last inequality the Sobolev embedding \eqref{sobineq}, and the estimates \eqref{estk} \eqref{estn} for $n$, $\db$ and $\zb$, the estimate \eqref{estb} for $b$, and the estimate \eqref{estzeta} for $\zeta$. Together with the basic properties of $P_{\leq l}$, the commutator estimates \eqref{commlp1} and \eqref{commlp2}, and the estimate \eqref{esttrc} for $\trc$, this implies:
$$\norm{G_2}_{\lprime{\infty}{\frac{3}{2}}}\les \ep.$$
Together with \eqref{vino52}, we obtain:
\bea\lab{vino53}
&&\normm{\int_{\S}G_2 \left(2^{\frac{j}{2}}(N'-N_{\nu'})\right)^qF_{j,-1}(u')\eta_j^{\nu'}(\o')d\o'}_{L^{\frac{3}{2}}(\MM)}\\
\nn&\les& \ep \left(\int_{\S} \norm{F_{j,-1}(u)}_{L^2_u}\eta_j^\nu(\o)d\o\right)\\
\nn&\les& 2^{\frac{j}{2}}\ep\gamma^{\nu'}_j,
\eea
where we used in the last inequality Plancherel in $\la$ for $\normm{F_{j,-1}(u')}_{L^2_{u'}}$, Cauchy Schwarz in $\o'$ and the size of the patch. Finally, \eqref{vino46}, \eqref{vino51} and \eqref{vino53} imply:
\be\lab{vino54}
\norm{h_{2,p,q}}_{L^1(\MM)}\les 2^{\frac{7j}{6}}2^{-\frac{2l}{3}}\ep^2\gamma^\nu_j\gamma^{\nu'}_j.
\ee

Next, we estimate the $L^1(\MM)$ norm of $h_{3,p,q}$. In view of the definition of $h_{3,p,q}$ \eqref{vino32}, we have:
\bea\lab{vino55}
\norm{h_{3,p,q}}_{L^1(\MM)}&\les & \normm{\int_{\S} \nabb(b)P_l\trc\left(2^{\frac{j}{2}}(N-N_\nu)\right)^pF_{j,-1}(u)\eta_j^\nu(\o)d\o}_{L^2(\MM)}\\
\nn&&\times\normm{\int_{\S}N'(P_{\leq l}\trc')\left(2^{\frac{j}{2}}(N'-N_{\nu'})\right)^qF_{j,-1}(u')\eta_j^{\nu'}(\o')d\o'}_{L^2(\MM)}.
\eea
Arguing as for the proof of \eqref{nycc19}, we have:
\be\lab{vino56}
\normm{\int_{\S}N'(P_{\leq l}\trc')\left(2^{\frac{j}{2}}(N'-N_{\nu'})\right)^qF_{j,-1}(u')\eta_j^{\nu'}(\o')d\o'}_{L^2(\MM)}\les \ep 2^{\frac{j}{2}}\gamma^{\nu'}_j.
\ee
Also, using the basic estimate in $L^2(\MM)$ \eqref{oscl2bis}, we have:
\bee
&&\normm{\int_{\S} \nabb(b)P_l\trc\left(2^{\frac{j}{2}}(N-N_\nu)\right)^pF_{j,-1}(u)\eta_j^\nu(\o)d\o}_{L^2(\MM)}\\
\nn&\les& \left(\sup_{\o}\normm{\nabb(b)P_l\trc\left(2^{\frac{j}{2}}(N-N_\nu)\right)^p}_{\lprime{\infty}{2}}\right)2^{\frac{j}{2}}\gamma^{\nu}_j\\
\nn&\les& \left(\sup_{\o}\normm{\nabb(b)}_{\tx{\infty}{4}}\norm{P_l\trc}_{\tx{2}{4}}\right)2^{\frac{j}{2}}\gamma^{\nu}_j,
\eee
where we used in the last inequality the estimate \eqref{estNomega} for $\po N$ and the size of the patch. Together with the estimate \eqref{estb} for $b$, the embedding \eqref{sobineq1}, and the estimates \eqref{vino40} and \eqref{vino41} for $P_l\trc$, we obtain:
\be\lab{vino57}
\normm{\int_{\S} \nabb(b)P_l\trc\left(2^{\frac{j}{2}}(N-N_\nu)\right)^pF_{j,-1}(u)\eta_j^\nu(\o)d\o}_{L^2(\MM)}\les \ep 2^{-\frac{l}{2}(1-\delta_{j/2,l})}2^{\frac{j}{2}}\gamma^{\nu}_j.
\ee
Finally, \eqref{vino55}, \eqref{vino56} and \eqref{vino57} imply:
\be\lab{vino58}
\norm{h_{3,p,q}}_{L^1(\MM)}\les \ep^2 2^{-\frac{l}{2}(1-\delta_{j/2,l})}2^j\gamma^{\nu}_j\gamma^{\nu'}_j.
\ee

Next, we estimate the $L^1(\MM)$ norm of $h_{4,p,q}$. In view of the definition of $h_{4,p,q}$ \eqref{vino33}, we have:
\bea\lab{vino59}
\norm{h_{4,p,q}}_{L^1(\MM)}&\les & \normm{\int_{\S}N(b)P_l\trc\left(2^{\frac{j}{2}}(N-N_\nu)\right)^pF_{j,-1}(u)\eta_j^\nu(\o)d\o}_{L^2(\MM)}\\
\nn&&\times\normm{\int_{\S}N'(P_{\leq l}\trc')\left(2^{\frac{j}{2}}(N'-N_{\nu'})\right)^qF_{j,-1}(u')\eta_j^{\nu'}(\o')d\o'}_{L^2(\MM)}\\
\nn&\les & \ep 2^{\frac{j}{2}}\gamma^{\nu'}_j\normm{\int_{\S}N(b)P_l\trc\left(2^{\frac{j}{2}}(N-N_\nu)\right)^pF_{j,-1}(u)\eta_j^\nu(\o)d\o}_{L^2(\MM)},
\eea
where we used in the last inequality the estimate \eqref{vino56}. In view of the estimate \eqref{estb} for $b$ , we have:
$$\norm{N(b)}_{\tx{\infty}{4}}\les \ep.$$
Thus, arguing as for the proof of \eqref{vino57}, we obtain:
\bee
\normm{\int_{\S} N(b)P_l\trc\left(2^{\frac{j}{2}}(N-N_\nu)\right)^pF_{j,-1}(u)\eta_j^\nu(\o)d\o}_{L^2(\MM)}\les \ep 2^{-\frac{l}{2}(1-\delta_{j/2,l})}2^{\frac{j}{2}}\gamma^{\nu}_j.
\eee
Together with \eqref{vino59}, this yields:
\be\lab{vino60}
\norm{h_{4,p,q}}_{L^1(\MM)}\les \ep^2 2^{-\frac{l}{2}(1-\delta_{j/2,l})}2^j\gamma^{\nu}_j\gamma^{\nu'}_j.
\ee

Now, in view of the decomposition \eqref{vino29} of $B^{1,2,2,2}_{j,\nu,\nu',l,m}$, we have:
\bee
&& \left|\sum_{m/m\leq l}B^{1,2,2,2}_{j,\nu,\nu',l,m}-\sum_{m/m\leq l}(B^{1,2,2,2,1}_{j,\nu,\nu',l,m}+B^{1,2,2,2,2}_{j,\nu,\nu',l,m}+B^{1,2,2,2,3}_{j,\nu,\nu',l,m})\right|\\
\nn &\les & 2^{-j}\sum_{p, q\geq 0}c_{pq}\normm{\frac{1}{(2^{\frac{j}{2}}|N_\nu-N_{\nu'}|)^{p+q}}}_{L^\infty(\MM)}\Bigg[\normm{\frac{1}{(2^{\frac{j}{2}}|N_\nu-N_{\nu'}|)^2}}_{L^\infty(\MM)}(\norm{h_{1,p,q}}_{L^1(\MM)}\\
\nn&&+\norm{h_{2,p,q}}_{L^1(\MM)})+\normm{\frac{1}{2^{\frac{j}{2}}(2^{\frac{j}{2}}|N_\nu-N_{\nu'}|)}}_{L^\infty(\MM)}\norm{h_{3,p,q}}_{L^1(\MM)}+2^{-j}\norm{h_{4,p,q}}_{L^1(\MM)}\Bigg],
\eee
which together with \eqref{nice26}, \eqref{vino45}, \eqref{vino54}, \eqref{vino58} and \eqref{vino60} yields:
\bee
&& \left|\sum_{m/m\leq l}B^{1,2,2,2}_{j,\nu,\nu',l,m}-\sum_{m/m\leq l}(B^{1,2,2,2,1}_{j,\nu,\nu',l,m}+B^{1,2,2,2,2}_{j,\nu,\nu',l,m}+B^{1,2,2,2,3}_{j,\nu,\nu',l,m})\right|\\
\nn &\les & 2^{-j}\sum_{p, q\geq 0}c_{pq}\frac{1}{(2^{\frac{j}{2}}|\nu-\nu'|)^{p+q}}\Bigg[\frac{1}{(2^{\frac{j}{2}}|\nu-\nu'|)^2}( |\nu-\nu'|2^{j-l(1-\delta_{l,j/2})}+2^{\frac{7j}{6}}2^{-\frac{2l}{3}})\\
\nn&&+\frac{1}{2^{\frac{j}{2}}(2^{\frac{j}{2}}|\nu-\nu'|)}  2^{-\frac{l}{2}(1-\delta_{j/2,l})}2^j+2^{-j} 2^{-\frac{l}{2}(1-\delta_{j/2,l})}2^j\Bigg]\ep^2\gamma^{\nu}_j\gamma^{\nu'}_j\\
\nn &\les & 2^{-j}\Bigg[\frac{1}{(2^{\frac{j}{2}}|\nu-\nu'|)^2}( |\nu-\nu'|2^{j-l(1-\delta_{l,j/2})}+2^{\frac{7j}{6}}2^{-\frac{2l}{3}})\\
\nn&&+\frac{1}{2^{\frac{j}{2}}(2^{\frac{j}{2}}|\nu-\nu'|)}  2^{-\frac{l}{2}(1-\delta_{j/2,l})}2^j+2^{-j} 2^{-\frac{l}{2}(1-\delta_{j/2,l})}2^j\Bigg]\ep^2\gamma^{\nu}_j\gamma^{\nu'}_j.
\eee
Summing in $l$, we obtain:
\bea
\nn&& \left|\sum_{(l,m)/2^m\leq 2^l\leq 2^j|\nu-\nu'|}B^{1,2,2,2}_{j,\nu,\nu',l,m}-\sum_{(l,m)/2^m\leq 2^l\leq 2^j|\nu-\nu'|}(B^{1,2,2,2,1}_{j,\nu,\nu',l,m}+B^{1,2,2,2,2}_{j,\nu,\nu',l,m}+B^{1,2,2,2,3}_{j,\nu,\nu',l,m})\right|\\
\lab{vino61} &\les & \Bigg[\frac{1}{2^{\frac{j}{2}}(2^{\frac{j}{2}}|\nu-\nu'|)}+\frac{2^{-\frac{j}{6}}}{(2^{\frac{j}{2}}|\nu-\nu'|)^2}+2^{-j}\Bigg]\ep^2\gamma^{\nu}_j\gamma^{\nu'}_j.
\eea

In view of \eqref{vino61}, we need to estimate $B^{1,2,2,2,1}_{j,\nu,\nu',l,m}, B^{1,2,2,2,2}_{j,\nu,\nu',l,m}, B^{1,2,2,2,3}_{j,\nu,\nu',l,m}$. We start with $B^{1,2,2,2,3}_{j,\nu,\nu',l,m}$ which is defined in \eqref{vinoroja3} as:
$$B^{1,2,2,2,3}_{j,\nu,\nu',l,m}= -i2^{-j}\int_{\MM}\int_{\S\times \S} b^{-1} P_l\trc N'(P_m\trc')(b'-b)F_j(u)\eta_j^\nu(\o)F_{j,-1}(u')\eta_j^{\nu'}(\o') d\o d\o'd\MM.$$
We integrate by parts tangentially using \eqref{fete1}.
\begin{lemma}\lab{lemma:vinoroja}
Let $B^{1,2,2,2,3}_{j,\nu,\nu',l,m}$ defined in \eqref{vinoroja3}. Integrating by parts using \eqref{fete1} yields:
\bea\lab{vinoroja29}
&& \sum_{m/m\leq l}B^{1,2,2,2,3}_{j,\nu,\nu',l,m}\\
\nn &=& 2^{-j}\sum_{p, q\geq 0}c_{pq}\int_{\MM}\frac{1}{(2^{\frac{j}{2}}|N_\nu-N_{\nu'}|)^{p+q}}\Bigg[\frac{1}{(2^{\frac{j}{2}}|N_\nu-N_{\nu'}|)^2}(h_{1,p,q}'+h_{2,p,q}')\\
\nn&&+\frac{1}{2^{\frac{j}{2}}(2^{\frac{j}{2}}|N_\nu-N_{\nu'}|)}(h_{3,p,q}+h_{5,p,q}')+2^{-j}h_{4,p,q}\Bigg] d\MM+\sum_{m/m\leq l}(B^{1,2,2,2,1}_{j,\nu,\nu',l,m}+B^{1,2,2,2,2}_{j,\nu,\nu',l,m}),
\eea
where the scalar functions $h_{3,p,q}, h_{4,p,q}$ are given respectively by \eqref{vino32} and \eqref{vino33}, where the scalar functions $h_{1,p,q}', h_{2,p,q}', h_{5,p,q}'$ on $\MM$ are given by:
\bea\lab{vinoroja30}
h_{1,p,q}'&=& \left(\int_{\S} G_1'(b-b_{\nu})^r\left(2^{\frac{j}{2}}(N-N_\nu)\right)^pF_{j,-1}(u)\eta_j^\nu(\o)d\o\right)\\
\nn&&\times\left(\int_{\S}N'(P_{\leq l}\trc')(b_{\nu}-b')^s\left(2^{\frac{j}{2}}(N'-N_{\nu'})\right)^qF_{j,-1}(u')\eta_j^{\nu'}(\o')d\o'\right),
\eea
\bea\lab{vinoroja31}
h_{2,p,q}'&=& \left(\int_{\S} P_l\trc  \left(2^{\frac{j}{2}}(N-N_\nu)\right)^pF_{j,-1}(u)\eta_j^\nu(\o)d\o\right)\\
\nn&&\times\left(\int_{\S}G_2' \left(2^{\frac{j}{2}}(N'-N_{\nu'})\right)^qF_{j,-1}(u')\eta_j^{\nu'}(\o')d\o'\right),
\eea
\bea\lab{vinoroja33bis}
h_{5,p,q}'&=& \left(\int_{\S} P_l\trc  \left(2^{\frac{j}{2}}(N-N_\nu)\right)^pF_{j,-1}(u)\eta_j^\nu(\o)d\o\right)\\
\nn&&\times\left(\int_{\S}G_3' \left(2^{\frac{j}{2}}(N'-N_{\nu'})\right)^qF_{j,-1}(u')\eta_j^{\nu'}(\o')d\o'\right),
\eea
where the integer $r, s$ satisfy:
$$r+s=1,$$
where the tensors $G_1', G_2'$ and $G_3'$ are schematically given by:
\be\lab{vinoroja34}
G_1'=(\chi+\th) P_l\trc,
\ee
\be\lab{vinoroja35}
G_2'=(\chi'+L'(b')+\th'+\nabb'(b)) N'(P_{\leq l}\trc'),
\ee
and:
\be\lab{vinoroja36}
G_3'=\nabb' N'(P_{\leq l}\trc'),
\ee
where $B^{1,2,2,2,1}_{j,\nu,\nu',l,m}, B^{1,2,2,2,2}_{j,\nu,\nu',l,m}$ are given respectively by \eqref{vinoroja1} and \eqref{vinoroja2}, and where $c_{pq}$ are explicit real coefficients such that the series 
$$\sum_{p, q\geq 0}c_{pq}x^py^q$$
has radius of convergence 1.
\end{lemma}
The proof of Lemma \ref{lemma:vinoroja} is postponed to Appendix D. We now use this lemma to estimate $B^{1,2,2,2,3}_{j,\nu,\nu',l,m}$. 

We estimate the $L^1(\MM)$ norm of $h_{1,p,q}', h_{2,p,q}', h_{5,p,q}'$. The estimate of $h_{1,p,q}'$ is completely analogous to the one of $h_{1,p,q}$ defined in \eqref{vino30}. Thus, we obtain in view of \eqref{vino45}:
\be\lab{vino62}
\norm{h_{1,p,q}'}_{L^1(\MM)}\les \ep^2 |\nu-\nu'|2^{j-l(1-\delta_{l,j/2})}\gamma^\nu_j\gamma^{\nu'}_j.
\ee
Also, the estimate of $h_{2,p,q}'$ is completely analogous to the one of $h_{2,p,q}$ defined in \eqref{vino31}. Thus, we obtain in view of \eqref{vino54}:
\be\lab{vino63}
\norm{h_{2,p,q}'}_{L^1(\MM)}\les 2^{\frac{7j}{6}}2^{-\frac{2l}{3}}\ep^2\gamma^\nu_j\gamma^{\nu'}_j.
\ee
Next, we estimate the $L^1(\MM)$ norm of $h_{5,p,q}'$. In view of \eqref{vinoroja33bis}, we have:
\bee
&&\norm{h_{5,p,q}'}_{L^1(\MM)}\\
\nn&\les& \normm{\int_{\S} P_l\trc  \left(2^{\frac{j}{2}}(N-N_\nu)\right)^pF_{j,-1}(u)\eta_j^\nu(\o)d\o}_{L^2(\MM)}\\
\nn&&\times\normm{\int_{\S}G_3' \left(2^{\frac{j}{2}}(N'-N_{\nu'})\right)^qF_{j,-1}(u')\eta_j^{\nu'}(\o')d\o'}_{L^2(\MM)}\\
\nn&\les& \int_{\S} \normm{P_l\trc  \left(2^{\frac{j}{2}}(N-N_\nu)\right)^pF_{j,-1}(u)}_{L^2(\MM)}\eta_j^\nu(\o)d\o\\
\nn&&\times\int_{\S}\normm{G_3' \left(2^{\frac{j}{2}}(N'-N_{\nu'})\right)^qF_{j,-1}(u')}_{L^2(\MM)}\eta_j^{\nu'}(\o')d\o'\\
\nn&\les& \left(\int_{\S} \normm{P_l\trc F_{j,-1}(u)}_{L^2(\MM)}\eta_j^\nu(\o)d\o\right)\left(\int_{\S}\normm{G_3'F_{j,-1}(u')}_{L^2(\MM)}\eta_j^{\nu'}(\o')d\o'\right),
\eee
where we used in the last inequality the estimate \eqref{estNomega} for $\po N$ and the size of the patch. Together with the definition of $G_3'$ \eqref{vinoroja36}, this yields:
\bea\lab{vino64}
\norm{h_{5,p,q}'}_{L^1(\MM)}&\les& \sum_{m/m\leq l}\left(\int_{\S} \normm{\norm{P_l\trc}_{L^2(\H_u)} F_{j,-1}(u)}_{L^2_u}\eta_j^\nu(\o)d\o\right)\\
\nn&&\times\left(\int_{\S}\normm{\norm{\nabb' N'(P_m\trc')}_{L^2(\H_{u'})}F_{j,-1}(u')}_{L^2_{u'}}\eta_j^{\nu'}(\o')d\o'\right).
\eea
Next, we estimate $\nabb' N'(P_m\trc')$. Using the estimate \eqref{estb} for $b'$, we have:
\bee
&&\norm{\nabb' N'(P_m\trc')}_{L^2(\H_{u'})}\\
&\les& \norm{{b'}^{-1}\nabb'(b')N'(P_m\trc')}_{L^2(\H_{u'})}+\norm{{b'}^{-1}\nabb' (b'N'(P_m\trc'))}_{L^2(\H_{u'})}\\
&\les& \norm{{b'}^{-1}\nabb'(b')}_{\tx{\infty}{4}}\norm{bN'(P_m\trc')}_{\tx{2}{4}}+\norm{\nabb' P_m(b'N'\trc')}_{L^2(\H_{u'})}\\
&&+\norm{\nabb'[b'N',P_m]\trc')}_{\lprime{\infty}{2}}\\
&\les& \norm{P_m(b'N'\trc')}_{\tx{2}{4}}+\norm{[bN',P_m]\trc'}_{\tx{2}{4}}+\norm{\nabb' P_m(b'N'\trc')}_{L^2(\H_{u'})}\\
&&+\norm{\nabb'[b'N',P_m]\trc')}_{\lprime{\infty}{2}}\\
&\les& \norm{P_m(b'N'\trc')}_{L^2(\H_{u'})}+\norm{\nabb' P_m(b'N'\trc')}_{L^2(\H_{u'})}\\
&&+\norm{[b'N',P_m]\trc')}_{\lprime{\infty}{2}}+\norm{\nabb'[b'N',P_m]\trc')}_{\lprime{\infty}{2}},
\eee
where we used in the last inequality the Gagliardo-Nirenberg inequality \eqref{eq:GNirenberg}. Together with the finite band property for $P_m$ and the commutator estimate \eqref{commlp3bis}, we obtain:
\be\lab{getafe}
\norm{\nabb' N'(P_m\trc')}_{L^2(\H_{u'})}\les 2^m\norm{P_m(b'N'\trc')}_{L^2(\H_{u'})} +2^{\frac{m}{2}}\ep.
\ee
Together with \eqref{vino64}, this yields:
\bea\lab{vino65}
\norm{h_{5,p,q}'}_{L^1(\MM)}&\les& \sum_{m/m\leq l}2^{-|l-m|}\left(\int_{\S} 2^l\normm{\norm{P_l\trc}_{L^2(\H_u)} F_{j,-1}(u)}_{L^2_u}\eta_j^\nu(\o)d\o\right)\\
\nn&&\times\left(\int_{\S}\normm{(\norm{P_m(b'N'\trc')}_{L^2(\H_{u'})}+2^{-\frac{m}{2}}\ep)F_{j,-1}(u')}_{L^2_{u'}}\eta_j^{\nu'}(\o')d\o'\right)\\
\nn&\les& 2^{-j}\sum_{m/m\leq l}2^{-|l-m|}\normm{2^l\norm{P_l\trc}_{L^2(\H_u)} F_{j,-1}(u)\sqrt{\eta_j^\nu(\o)}}_{L^2_{\o, u}}\\
\nn&&\times\normm{(\norm{P_m(b'N'\trc')}_{L^2(\H_{u'})}+2^{-\frac{m}{2}}\ep)F_{j,-1}(u')\sqrt{\eta_j^{\nu'}(\o')}}_{L^2_{\o', u'}},
\eea
where we used in the last inequality Cauchy Schwarz in $\o$ and $\o'$, and the size of the patch. Now, we have:
\bea\lab{vino66}
&& \sum_l\normm{2^l\norm{P_l\trc}_{L^2(\H_u)} F_{j,-1}(u)\sqrt{\eta_j^\nu(\o)}}_{L^2_{\o, u}}^2\\
\nn&=& \int_{\S}\left(\int_u\left(\sum_l2^{2l}\norm{P_l\trc}_{L^2(\H_u)}^2\right)|F_j(u)|^2du\right)\eta^\nu_j(\o)d\o\\
\nn&\les & \int_{\S}\norm{\nabb\trc}_{\li{\infty}{2}}^2\norm{F_j(u)}_{L^2_u}^2\eta^\nu_j(\o)d\o\\
\nn&\les & \ep^22^{2j}(\gamma^\nu_j)^2,
\eea
where we used the finite band property for $P_l$, the estimates \eqref{esttrc} for $\trc$ and  Plancherel in $\la$. Also, we have:
\bea\lab{vino67}
&& \sum_m\normm{(\norm{P_m(b'N'\trc')}_{L^2(\H_{u'})}+2^{-\frac{m}{2}}\ep)F_{j,-1}(u')\sqrt{\eta_j^{\nu'}(\o')}}_{L^2_{\o', u'}}^2\\
\nn&=& \int_{\S}\left(\int_{u'}\left(\sum_m\left(\norm{P_m(b'N'\trc')}_{L^2(\H_{u'})}+2^{-\frac{m}{2}}\ep\right)^2\right)|F_j(u')|^2du'\right)\eta^{\nu'}_j(\o')d\o'\\
\nn&\les & \int_{\S}\left(\norm{b'N'\trc}_{\lprime{\infty}{2}}^2+\ep^2\right)\norm{F_j(u')}_{L^2_{u'}}^2\eta^{\nu'}_j(\o')d\o'\\
\nn&\les & \ep^22^{2j}(\gamma^{\nu'}_j)^2,
\eea
where we used the finite band property for $P_m$, the estimates \eqref{esttrc} for $\trc'$ and  Plancherel in $\la'$. Finally, \eqref{vino65}, \eqref{vino66} and \eqref{vino67} yield:
\bea\lab{vino68}
&&\sum_{(l,m)/m\leq l}\norm{h_{5,p,q}'}_{L^1(\MM)}\\
\nn&\les& 2^{-j}\left(\sum_l\normm{2^l\norm{P_l\trc}_{L^2(\H_u)} F_{j,-1}(u)\sqrt{\eta_j^\nu(\o)}}_{L^2_{\o, u}}^2\right)^{\frac{1}{2}}\\
\nn&&\times\left(\sum_m\normm{(\norm{P_m(b'N'\trc')}_{L^2(\H_{u'})}+2^{-\frac{m}{2}}\ep)F_{j,-1}(u')\sqrt{\eta_j^{\nu'}(\o')}}_{L^2_{\o', u'}}^2\right)^{\frac{1}{2}}\\
\nn&\les& \ep^22^j\gamma^\nu_j\gamma^{\nu'}_j.
\eea

Now, in view of the decomposition \eqref{vinoroja29} of $B^{1,2,2,2,3}_{j,\nu,\nu',l,m}$, we have:
\bee
&& \left|\sum_{m/m\leq l}B^{1,2,2,2,3}_{j,\nu,\nu',l,m}-\sum_{m/m\leq l}(B^{1,2,2,2,1}_{j,\nu,\nu',l,m}+B^{1,2,2,2,2}_{j,\nu,\nu',l,m})\right|\\
\nn &\les& 2^{-j}\sum_{p, q\geq 0}c_{pq}\normm{\frac{1}{(2^{\frac{j}{2}}|N_\nu-N_{\nu'}|)^{p+q}}}_{L^\infty(\MM)}\Bigg[\normm{\frac{1}{(2^{\frac{j}{2}}|N_\nu-N_{\nu'}|)^2}}_{L^\infty(\MM)}(\norm{h_{1,p,q}'}_{L^1(\MM)}\\
\nn&&+\norm{h_{2,p,q}'}_{L^1(\MM)})+\normm{\frac{1}{2^{\frac{j}{2}}(2^{\frac{j}{2}}|N_\nu-N_{\nu'}|)}}_{L^\infty(\MM)}(\norm{h_{3,p,q}}_{L^1(\MM)}+\norm{h_{5,p,q}'}_{L^1(\MM)})\\
\nn&&+2^{-j}\norm{h_{4,p,q}}_{L^1(\MM)}\Bigg],
\eee
which together with \eqref{nice26}, \eqref{vino58}, \eqref{vino60}, \eqref{vino62} and \eqref{vino63} yields:
\bee
&& \left|\sum_{m/m\leq l}B^{1,2,2,2,3}_{j,\nu,\nu',l,m}-\sum_{m/m\leq l}(B^{1,2,2,2,1}_{j,\nu,\nu',l,m}+B^{1,2,2,2,2}_{j,\nu,\nu',l,m})\right|\\
\nn &\les& 2^{-j}\sum_{p, q\geq 0}c_{pq}\frac{1}{(2^{\frac{j}{2}}|\nu-\nu'|)^{p+q}}\Bigg[\bigg(\frac{1}{(2^{\frac{j}{2}}|\nu-\nu'|)^2}( |\nu-\nu'|2^{j-l(1-\delta_{l,j/2})}+2^{\frac{7j}{6}}2^{-\frac{2l}{3}})\\
\nn&&+\frac{1}{2^{\frac{j}{2}}(2^{\frac{j}{2}}|\nu-\nu'|)}2^{-\frac{l}{2}(1-\delta_{j/2,l})}2^j+2^{-j}2^{-\frac{l}{2}(1-\delta_{j/2,l})}2^j\bigg)\ep^2\gamma^{\nu}_j\gamma^{\nu'}_j+\frac{1}{2^{\frac{j}{2}}(2^{\frac{j}{2}}|\nu-\nu'|)}\norm{h_{5,p,q}'}_{L^1(\MM)}\Bigg]\\
\nn &\les& 2^{-j}\Bigg[\frac{1}{(2^{\frac{j}{2}}|\nu-\nu'|)^2}( |\nu-\nu'|2^{j-l(1-\delta_{l,j/2})}+2^{\frac{7j}{6}}2^{-\frac{2l}{3}})\\
\nn&&+\frac{1}{2^{\frac{j}{2}}(2^{\frac{j}{2}}|\nu-\nu'|)}2^{-\frac{l}{2}(1-\delta_{j/2,l})}2^j+2^{-j}2^{-\frac{l}{2}(1-\delta_{j/2,l})}2^j\Bigg]\ep^2\gamma^{\nu}_j\gamma^{\nu'}_j\\
\nn&&+2^{-\frac{3j}{2}}\sum_{p, q\geq 0}c_{pq}\frac{1}{(2^{\frac{j}{2}}|\nu-\nu'|)^{p+q+1}}\norm{h_{5,p,q}'}_{L^1(\MM)}.
\eee
Summing in $l$, we obtain:
\bee
&& \left|\sum_{(l,m)/2^m\leq 2^l\leq 2^j|\nu-nu'|}B^{1,2,2,2,3}_{j,\nu,\nu',l,m}-\sum_{(l,m)/2^m\leq 2^l\leq 2^j|\nu-\nu'|}(B^{1,2,2,2,1}_{j,\nu,\nu',l,m}+B^{1,2,2,2,2}_{j,\nu,\nu',l,m})\right|\\
\nn &\les&  \Bigg[\frac{1}{2^{\frac{j}{2}}(2^{\frac{j}{2}}|\nu-\nu'|)}+\frac{2^{-\frac{j}{6}}}{(2^{\frac{j}{2}}|\nu-\nu'|)^2}+2^{-j}\Bigg]\ep^2\gamma^{\nu}_j\gamma^{\nu'}_j\\
\nn&&+2^{-\frac{3j}{2}}\sum_{p, q\geq 0}c_{pq}\frac{1}{(2^{\frac{j}{2}}|\nu-\nu'|)^{p+q+1}}\left(\sum_{(l,m)/m\leq l}\norm{h_{5,p,q}'}_{L^1(\MM)}\right).
\eee
Together with \eqref{vino68}, we get:
\bea
\nn&& \left|\sum_{(l,m)/2^m\leq 2^l\leq 2^j|\nu-\nu'|}B^{1,2,2,2,3}_{j,\nu,\nu',l,m}-\sum_{(l,m)/2^m\leq 2^l\leq 2^j|\nu-\nu'|}(B^{1,2,2,2,1}_{j,\nu,\nu',l,m}+B^{1,2,2,2,2}_{j,\nu,\nu',l,m})\right|\\
\nn &\les&  \Bigg[\frac{1}{2^{\frac{j}{2}}(2^{\frac{j}{2}}|\nu-\nu'|)}+\frac{2^{-\frac{j}{6}}}{(2^{\frac{j}{2}}|\nu-\nu'|)^2}+2^{-j}\Bigg]\ep^2\gamma^{\nu}_j\gamma^{\nu'}_j\\
\nn&&+2^{-\frac{3j}{2}}\sum_{p, q\geq 0}c_{pq}\frac{1}{(2^{\frac{j}{2}}|\nu-\nu'|)^{p+q+1}}\ep^22^j\gamma^\nu_j\gamma^{\nu'}_j\\
\lab{vino69bis} &\les&  \Bigg[\frac{1}{2^{\frac{j}{2}}(2^{\frac{j}{2}}|\nu-\nu'|)}+\frac{2^{-\frac{j}{6}}}{(2^{\frac{j}{2}}|\nu-\nu'|)^2}+2^{-j}\Bigg]\ep^2\gamma^{\nu}_j\gamma^{\nu'}_j.
\eea
Finally, \eqref{vino61} and \eqref{vino69bis} imply:
\bea
\nn&& \left|\sum_{(l,m)/2^m\leq 2^l\leq 2^j|\nu-\nu'|}B^{1,2,2,2}_{j,\nu,\nu',l,m}-\sum_{(l,m)/2^m\leq 2^l\leq 2^j|\nu-\nu'|}(B^{1,2,2,2,1}_{j,\nu,\nu',l,m}+B^{1,2,2,2,2}_{j,\nu,\nu',l,m})\right|\\
\lab{vino69} &\les & \Bigg[\frac{1}{2^{\frac{j}{2}}(2^{\frac{j}{2}}|\nu-\nu'|)}+\frac{2^{-\frac{j}{6}}}{(2^{\frac{j}{2}}|\nu-\nu'|)^2}+2^{-j}\Bigg]\ep^2\gamma^{\nu}_j\gamma^{\nu'}_j.
\eea

In view of \eqref{vino69}, we still need to estimate $B^{1,2,2,2,1}_{j,\nu,\nu',l,m}$ and $B^{1,2,2,2,2}_{j,\nu,\nu',l,m}$. We start with $B^{1,2,2,2,1}_{j,\nu,\nu',l,m}$ which is defined by \eqref{vinoroja1} as:
\bee
B^{1,2,2,2,1}_{j,\nu,\nu',l,m}&=& -2^{-2j}\int_{\MM}\int_{\S\times \S} \frac{(N'-\gn N)(P_l\trc) N'(P_m\trc')(b'-b)}{\gg(L,L')}\\
\nn&&\times F_{j,-1}(u)\eta_j^\nu(\o)F_{j,-1}(u')\eta_j^{\nu'}(\o') d\o d\o'd\MM,
\eee
We integrate by parts tangentially using \eqref{fete1}.
\begin{lemma}\lab{lemma:vinoverde}
Let $B^{1,2,2,2,1}_{j,\nu,\nu',l,m}$ defined in \eqref{vinoroja1}. Integrating by parts using \eqref{fete1} yields:
\bea\lab{vinoverde29}
&& \sum_{m/m\leq l}B^{1,2,2,2,1}_{j,\nu,\nu',l,m}\\
\nn &=& 2^{-\frac{3j}{2}}\sum_{p, q\geq 0}c_{pq}\int_{\MM}\frac{1}{(2^{\frac{j}{2}}|N_\nu-N_{\nu'}|)^{p+q+1}}\Bigg[\frac{1}{(2^{\frac{j}{2}}|N_\nu-N_{\nu'}|)^2}(h_{1,p,q}''+h_{2,p,q}'')\\
\nn&&+\frac{1}{2^{\frac{j}{2}}(2^{\frac{j}{2}}|N_\nu-N_{\nu'}|)}(h_{3,p,q}''+h_{4,p,q}'')+2^{-j}h_{5,p,q}''\Bigg] d\MM,
\eea
where the scalar functions $h_{1,p,q}'', h_{2,p,q}'', h_{3,p,q}'', h_{4,p,q}'', h_{5,p,q}''$ on $\MM$ are given by:
\bea\lab{vinoverde30}
h_{1,p,q}''&=& \left(\int_{\S} G_1''\left(2^{\frac{j}{2}}(N-N_\nu)\right)^pF_{j,-1}(u)\eta_j^\nu(\o)d\o\right)\\
\nn&&\times\left(\int_{\S}N'(P_{\leq l}\trc')\left(2^{\frac{j}{2}}(N'-N_{\nu'})\right)^qF_{j,-1}(u')\eta_j^{\nu'}(\o')d\o'\right),
\eea
\bea\lab{vinoverde31}
h_{2,p,q}''&=& \left(\int_{\S} \nabb(P_l\trc)  \left(2^{\frac{j}{2}}(N-N_\nu)\right)^pF_{j,-1}(u)\eta_j^\nu(\o)d\o\right)\\
\nn&&\times\left(\int_{\S}G_2'' \left(2^{\frac{j}{2}}(N'-N_{\nu'})\right)^qF_{j,-1}(u')\eta_j^{\nu'}(\o')d\o'\right),
\eea
\bea\lab{vinoverde32}
h_{3,p,q}''&=& \left(\int_{\S} \nabb^2(P_l\trc)  \left(2^{\frac{j}{2}}(N-N_\nu)\right)^pF_{j,-1}(u)\eta_j^\nu(\o)d\o\right)\\
\nn&&\times\left(\int_{\S} N'(P_{\leq l}\trc')\left(2^{\frac{j}{2}}(N'-N_{\nu'})\right)^qF_{j,-1}(u')\eta_j^{\nu'}(\o')d\o'\right),
\eea
\bea\lab{vinoverde33}
h_{4,p,q}''&=& \left(\int_{\S} \nabb(P_l\trc)  \left(2^{\frac{j}{2}}(N-N_\nu)\right)^pF_{j,-1}(u)\eta_j^\nu(\o)d\o\right)\\
\nn&&\times\left(\int_{\S}\nabb'(N'(P_{\leq l}\trc')) \left(2^{\frac{j}{2}}(N'-N_{\nu'})\right)^qF_{j,-1}(u')\eta_j^{\nu'}(\o')d\o'\right),
\eea
and:
\bea\lab{vinoverde33bis}
h_{5,p,q}''&=& \left(\int_{\S} \nabb(N(P_l\trc))  \left(2^{\frac{j}{2}}(N-N_\nu)\right)^pF_{j,-1}(u)\eta_j^\nu(\o)d\o\right)\\
\nn&&\times\left(\int_{\S}N'(P_{\leq l}\trc') \left(2^{\frac{j}{2}}(N'-N_{\nu'})\right)^qF_{j,-1}(u')\eta_j^{\nu'}(\o')d\o'\right),
\eea
where the tensors $G_1''$ and $G_2''$ are schematically given by:
\be\lab{vinoverde34}
G_1''=(\chi+\th+\nabb(b)) \nabb(P_l\trc),
\ee
and:
\be\lab{vinoverde35}
G_2''=(\chi'+\th'+\nabb'(b')) N'(P_{\leq l}\trc'),
\ee
and where $c_{pq}$ are explicit real coefficients such that the series 
$$\sum_{p, q\geq 0}c_{pq}x^py^q$$
has radius of convergence 1.
\end{lemma}
The proof of Lemma \ref{lemma:vinoverde} is postponed to Appendix E. We now use this lemma to estimate $B^{1,2,2,2,1}_{j,\nu,\nu',l,m}$. 

We estimate the $L^1(\MM)$ norm of $h_{1,p,q}'', h_{2,p,q}'', h_{3,p,q}'', h_{4,p,q}'', h_{5,p,q}''$ starting with $h_{1,p,q}''$. In view of \eqref{vinoverde30}, we have:
\bee
\norm{h_{1,p,q}''}_{L^1(\MM)}&\les& \normm{\int_{\S} G_1''\left(2^{\frac{j}{2}}(N-N_\nu)\right)^pF_{j,-1}(u)\eta_j^\nu(\o)d\o}_{L^2(\MM)}\\
\nn&&\times\normm{\int_{\S}N'(P_{\leq l}\trc')\left(2^{\frac{j}{2}}(N'-N_{\nu'})\right)^qF_{j,-1}(u')\eta_j^{\nu'}(\o')d\o'}_{L^2(\MM)}.
\eee
Together with the basic estimate in $L^2(\MM)$ \eqref{oscl2bis}, this yields:
\bea\lab{vino70}
\norm{h_{1,p,q}''}_{L^1(\MM)}&\les& \left(\sup_\o\normm{G_1''\left(2^{\frac{j}{2}}(N-N_\nu)\right)^p}_{\li{\infty}{2}}\right)\\
\nn&&\times\left(\sup_{\o'}\normm{N'(P_{\leq l}\trc')\left(2^{\frac{j}{2}}(N'-N_{\nu'})\right)^q}_{\lprime{\infty}{2}}\right) 2^j\gamma^\nu_j\gamma^{\nu'}_j\\
\nn&\les& \left(\sup_\o\norm{G_1''}_{\li{\infty}{2}}\right)\left(\sup_{\o'}\norm{N'(P_{\leq l}\trc')}_{\lprime{\infty}{2}}\right) 2^j\gamma^\nu_j\gamma^{\nu'}_j,
\eea
where we used in the last inequality the estimate \eqref{estNomega} for $\po N$ and the size of the patch. Now, the estimate \eqref{esttrc} for $\trc$, the boundedness of $P_m$ on $L^2(\ptu)$ and the commutator estimate \eqref{commlp3bis} yields:
\be\lab{vino71}
\norm{N'(P_{\leq l}\trc')}_{\lprime{\infty}{2}}\les\ep.
\ee
Also, we have in view of the definition \eqref{vinoverde34} of $G_1''$:
\bee
\norm{G_1''}_{\li{\infty}{2}}&\les& (\norm{\chi}_{\tx{\infty}{4}}+\norm{\th}_{\tx{\infty}{4}}+\norm{\nabb(b)}_{\tx{\infty}{4}})\norm{\nabb(P_l\trc)}_{\tx{2}{4}}\\
&\les& (\no(\chi)+\no(\th)+\no(\nabb(b)))\norm{\nabb(P_l\trc)}_{\li{\infty}{2}}^{\frac{1}{2}}\norm{\nabb^2(P_l\trc)}_{\li{\infty}{2}}^{\frac{1}{2}},
\eee
where we used in the last inequality the embedding \eqref{sobineq1} and the Gagliardo-Nirenberg inequality \eqref{eq:GNirenberg}. Together with the Bochner inequality \eqref{eq:Bochconseqbis} and the finite band property for $P_l$, we obtain:
\bea\lab{vino72}
\norm{G_1''}_{\li{\infty}{2}}&\les& (\no(\chi)+\no(\th)+\no(\nabb(b)))2^{\frac{l}{2}}\norm{\nabb\trc}_{\li{\infty}{2}}\\
\nn&\les& 2^{\frac{l}{2}}\ep,
\eea
where we used in the last inequality the estimates \eqref{esttrc} \eqref{esthch} for $\chi$, the estimate \eqref{estb} for $b$ and the estimates \eqref{estk} \eqref{esttrc} \eqref{esthch} for $\th$. Finally, \eqref{vino70}, \eqref{vino71} and \eqref{vino72} imply:
\be\lab{vino73}
\norm{h_{1,p,q}''}_{L^1(\MM)}\les \ep^2 2^{j+\frac{l}{2}}\gamma^\nu_j\gamma^{\nu'}_j,
\ee

Next, we estimate $h_{2,p,q}''$. In view of its definition \eqref{vinoverde31}, we have the analog of the estimate \eqref{vino70}:
\be\lab{vino74}
\norm{h_{2,p,q}''}_{L^1(\MM)}\les \left(\sup_\o\norm{\nabb(P_l\trc)}_{\li{\infty}{2}}\right)\left(\sup_{\o'}\norm{G_2''}_{\lprime{\infty}{2}}\right) 2^j\gamma^\nu_j\gamma^{\nu'}_j.
\ee
The estimate \eqref{esttrc} for $\trc$ together with the finite band property for $P_l$ yields:
\be\lab{vino75}
\norm{\nabb(P_l\trc)}_{\li{\infty}{2}}\les \ep.
\ee
Also, in view of the definition \eqref{vinoverde35} of $G_2''$ and the estimate \eqref{vino71}, the analog of the estimate \eqref{vino72} yields:
$$\norm{G_2''}_{\lprime{\infty}{2}}\les 2^{\frac{l}{2}}\ep,$$
which together with \eqref{vino74} and \eqref{vino75} implies:
\be\lab{vino76}
\norm{h_{2,p,q}''}_{L^1(\MM)}\les \ep^2 2^{j+\frac{l}{2}}\gamma^\nu_j\gamma^{\nu'}_j.
\ee

Next, we estimate $h_{3,p,q}''$. In view of its definition \eqref{vinoverde32}, we have the analog of the estimate \eqref{vino70}:
\be\lab{vino77}
\norm{h_{3,p,q}''}_{L^1(\MM)}\les \left(\sup_\o\norm{\nabb^2(P_l\trc)}_{\li{\infty}{2}}\right)\left(\sup_{\o'}\norm{N'(P_{\leq l}\trc')}_{\lprime{\infty}{2}}\right) 2^j\gamma^\nu_j\gamma^{\nu'}_j.
\ee
Now, the Bochner inequality \eqref{eq:Bochconseqbis}, the finite band property for $P_l$, and the estimate \eqref{esttrc} for $\trc$ yield:
\be\lab{vino78}
\norm{\nabb^2(P_l\trc)}_{\li{\infty}{2}}\les 2^l\ep.
\ee
Finally, \eqref{vino77}, \eqref{vino78} and \eqref{vino71} yield:
\be\lab{vino79}
\norm{h_{3,p,q}''}_{L^1(\MM)}\les \ep^2 2^{j+l}\gamma^\nu_j\gamma^{\nu'}_j.
\ee

Next, we estimate $h_{4,p,q}''$. In view of its definition \eqref{vinoverde33}, we have the analog of the estimate \eqref{vino70}:
\be\lab{vino80}
\norm{h_{4,p,q}''}_{L^1(\MM)}\les \left(\sup_\o\norm{\nabb(P_l\trc)}_{\li{\infty}{2}}\right)\left(\sup_{\o'}\norm{\nabb'(N'(P_{\leq l}\trc'))}_{\lprime{\infty}{2}}\right) 2^j\gamma^\nu_j\gamma^{\nu'}_j.
\ee
Now, we have in view of the estimate \eqref{getafe}:
\bea\lab{vino81}
\norm{\nabb' N'(P_{\leq l}\trc')}_{L^2(\H_{u'})}&\les& 2^l\norm{P_{\leq l}(b'N'\trc')}_{L^2(\H_{u'})} +2^{\frac{l}{2}}\ep\\
\nn&\les& 2^l\ep,
\eea
where we used in the last inequality the boundedness of $P_m$ on $L^2(\ptu)$, the estimate \eqref{estb} for $b$ and the estimate \eqref{esttrc} for $\trc$. Finally, \eqref{vino80}, \eqref{vino81} and \eqref{vino75} yield:
\be\lab{vino82}
\norm{h_{4,p,q}''}_{L^1(\MM)}\les \ep^2 2^{j+l}\gamma^\nu_j\gamma^{\nu'}_j.
\ee

Next, we estimate $h_{5,p,q}''$. In view of its definition \eqref{vinoverde33bis}, we have the analog of the estimate \eqref{vino70}:
$$\norm{h_{5,p,q}''}_{L^1(\MM)}\les \left(\sup_\o\norm{\nabb(N(P_l\trc))}_{\li{\infty}{2}}\right)\left(\sup_{\o'}\norm{N'(P_{\leq l}\trc')}_{\lprime{\infty}{2}}\right) 2^j\gamma^\nu_j\gamma^{\nu'}_j.$$
Together with the estimate \eqref{vino71} and the estimate \eqref{vino81}, we obtain:
\be\lab{vino83}
\norm{h_{5,p,q}''}_{L^1(\MM)}\les \ep^2 2^{j+l}\gamma^\nu_j\gamma^{\nu'}_j.
\ee

Now, in view of the decomposition \eqref{vinoverde29} of $B^{1,2,2,2,1}_{j,\nu,\nu',l,m}$, we have:
\bee
&& \left|\sum_{m/m\leq l}B^{1,2,2,2,1}_{j,\nu,\nu',l,m}\right|\\
\nn &\les& 2^{-\frac{3j}{2}}\sum_{p, q\geq 0}c_{pq}\normm{\frac{1}{(2^{\frac{j}{2}}|N_\nu-N_{\nu'}|)^{p+q+1}}}_{L^\infty(\MM)}\Bigg[\normm{\frac{1}{(2^{\frac{j}{2}}|N_\nu-N_{\nu'}|)^2}}_{L^\infty(\MM)}(\norm{h_{1,p,q}''}_{L^1(\MM)}\\
\nn&&+\norm{h_{2,p,q}''}_{L^1(\MM)})+\normm{\frac{1}{2^{\frac{j}{2}}(2^{\frac{j}{2}}|N_\nu-N_{\nu'}|)}}_{L^\infty(\MM)}(\norm{h_{3,p,q}''}_{L^1(\MM)}+\norm{h_{4,p,q}''}_{L^1(\MM)})\\
\nn&&+2^{-j}\norm{h_{5,p,q}''}_{L^1(\MM)}\Bigg],
\eee
which together with \eqref{nice26}, \eqref{vino73}, \eqref{vino76}, \eqref{vino79}, \eqref{vino82} and \eqref{vino83}  implies:
\bee
&& \left|\sum_{m/m\leq l}B^{1,2,2,2,1}_{j,\nu,\nu',l,m}\right|\\
\nn &\les& 2^{-\frac{j}{2}}\sum_{p, q\geq 0}c_{pq}\frac{1}{(2^{\frac{j}{2}}|\nu-\nu'|)^{p+q+1}}\Bigg[\frac{1}{(2^{\frac{j}{2}}|\nu-\nu'|)^2}2^{\frac{l}{2}}+\frac{1}{2^{\frac{j}{2}}(2^{\frac{j}{2}}|\nu-\nu'|)}2^l+2^{-j+l}\Bigg]\ep^2 \gamma^\nu_j\gamma^{\nu'}_j\\
\nn &\les& 2^{-\frac{j}{2}}\frac{1}{(2^{\frac{j}{2}}|\nu-\nu'|)}\Bigg[\frac{1}{(2^{\frac{j}{2}}|\nu-\nu'|)^2}2^{\frac{l}{2}}+\frac{1}{2^{\frac{j}{2}}(2^{\frac{j}{2}}|\nu-\nu'|)}2^l+2^{-j+l}\Bigg]\ep^2 \gamma^\nu_j\gamma^{\nu'}_j.
\eee
Summing in $l$, we obtain:
\be\lab{vino84}
\left|\sum_{(l,m)/ 2^m\leq 2^l\leq 2^j|\nu-\nu'|}B^{1,2,2,2,1}_{j,\nu,\nu',l,m}\right|\les  \Bigg[\frac{2^{-\frac{j}{4}}}{(2^{\frac{j}{2}}|\nu-\nu'|)^{\frac{5}{2}}}+\frac{1}{2^{\frac{j}{2}}(2^{\frac{j}{2}}|\nu-\nu'|)}+2^{-j}\Bigg]\ep^2\gamma^{\nu}_j\gamma^{\nu'}_j.
\ee

Now, \eqref{vino69} and \eqref{vino84}, together with the fact that 
$$2^{\frac{j}{2}}|\nu-\nu'|\gtrsim 1$$
implies:
\bea
\nn&& \left|\sum_{(l,m)/2^m\leq 2^l\leq 2^j|\nu-\nu'|}B^{1,2,2,2}_{j,\nu,\nu',l,m}-\sum_{(l,m)/2^m\leq 2^l\leq 2^j|\nu-\nu'|}B^{1,2,2,2,2}_{j,\nu,\nu',l,m}\right|\\
\lab{vino85} &\les & \Bigg[\frac{1}{2^{\frac{j}{2}}(2^{\frac{j}{2}}|\nu-\nu'|)}+\frac{2^{-\frac{j}{6}}}{(2^{\frac{j}{2}}|\nu-\nu'|)^2}+2^{-j}\Bigg]\ep^2\gamma^{\nu}_j\gamma^{\nu'}_j.
\eea
In view of \eqref{vino85}, we still need to estimate $B^{1,2,2,2,2}_{j,\nu,\nu',l,m}$. Recall the definition \eqref{vinoroja2} of $B^{1,2,2,2,2}_{j,\nu,\nu',l,m}$:
$$B^{1,2,2,2,2}_{j,\nu,\nu',l,m}= -2^{-2j}\int_{\MM}\int_{\S\times \S} N(P_l\trc) N'(P_m\trc')(b'-b) F_{j,-1}(u)\eta_j^\nu(\o)F_{j,-1}(u')\eta_j^{\nu'}(\o') d\o d\o'd\MM.$$
Note that the integrant in the definition of $B^{1,2,2,2,2}_{j,\nu,\nu',l,m}$ is antisymmetric in $((l,\o,\nu),(m,\o',\nu'))$, and thus we have the following cancellation:
$$B^{1,2,2,2,2}_{j,\nu,\nu',l,m}+B^{1,2,2,2,2}_{j,\nu',\nu,m,l}=0.$$
This yields:
$$\sum_{(l,m)/ 2^{\max(m,l)}\leq 2^j|\nu-\nu'|}(B^{1,2,2,2,2}_{j,\nu,\nu',l,m}+B^{1,2,2,2,2}_{j,\nu',\nu,m,l})=0,$$
which together with \eqref{vino85} implies:
\be\lab{vino86}
\left|\sum_{(l,m)/2^{\max(m,l)}\leq 2^j|\nu-\nu'|}B^{1,2,2,2}_{j,\nu,\nu',l,m}\right|\les  \Bigg[\frac{1}{2^{\frac{j}{2}}(2^{\frac{j}{2}}|\nu-\nu'|)}+\frac{2^{-\frac{j}{6}}}{(2^{\frac{j}{2}}|\nu-\nu'|)^2}+2^{-j}\Bigg]\ep^2\gamma^{\nu}_j\gamma^{\nu'}_j.
\ee

Next, we estimate $B^{1,2,2,3}_{j,\nu,\nu',l,m}$. Recall the definition \eqref{vino5quatre} of $B^{1,2,2,3}_{j,\nu,\nu',l,m}$:
\bee
B^{1,2,2,3}_{j,\nu,\nu',l,m}&=& 2^{-2j}\int_{\S\times \S}\frac{(\chi'-\chi)(b'-b)}{\gl^2}\Big(L(P_l\trc)P_m\trc'+P_l\trc L'(P_m\trc')\Big)\\
\nn&& \times F_{j,-1}(u)\eta_j^\nu(\o)F_{j,-1}(u')\eta_j^{\nu'}(\o') d\o d\o'.
\eee
Recall also that we are considering the range of $(l,m)$:
$$2^m\leq 2^l\leq 2^j|\nu-\nu'|.$$
Summing in $(l,m)$, we have:
\bee
&&\sum_{(l,m)/2^m\leq 2^l\leq 2^j|\nu-\nu'|}\Big(L(P_l\trc)P_m\trc'+P_l\trc L'(P_m\trc')\Big)\\
&=& L(P_{\leq 2^j|\nu-\nu'|}\trc)P_{\leq 2^j|\nu-\nu'|}\trc'+P_{\leq 2^j|\nu-\nu'|}\trc L'(P_{\leq 2^j|\nu-\nu'|}\trc').
\eee
Thus, using the symmetry in $(\o, \o')$ of the integrant in $B^{1,2,2,3}_{j,\nu,\nu',l,m}$, we obtain in view of the definition \eqref{vino5quatre} of $B^{1,2,2,3}_{j,\nu,\nu',l,m}$:
\bea\lab{mgen}
&&\sum_{(l,m)/2^{\max(l,m)}l\leq 2^j|\nu-\nu'|}(B^{1,2,2,3}_{j,\nu',\nu,l,m}+B^{1,2,2,3}_{j,\nu,\nu',l,m})\\
\nn&=& 2^{-2j}\int_{\S\times \S}\frac{(\chi'-\chi)(b'-b)}{\gl^2}L(P_{\leq 2^j|\nu-\nu'|}\trc)P_{\leq 2^j|\nu-\nu'|}\trc'\\
\nn&&\times F_{j,-1}(u)\eta_j^\nu(\o)F_{j,-1}(u')\eta_j^{\nu'}(\o') d\o d\o'+\textrm{ terms interverting }(\nu,\nu'),
\eea
where we chose to ignore the terms which are obtained by interverting $\nu$ and $\nu'$ since they are treated in the exact same way.

Recall the identities \eqref{nice24} and \eqref{nice25}:
$$\gg(L,L')=-1+\gn\textrm{ and }1-\gn=\frac{\gg(N-N',N-N')}{2}.$$
We may thus expand 
$$\frac{1}{\gl^2}$$ 
in the same fashion than \eqref{nice27}, and in view of \eqref{mgen}, we obtain, schematically:
\bea\lab{mgen1}
&& \sum_{(l,m)/2^{\max(l,m)}l\leq 2^j|\nu-\nu'|}(B^{1,2,2,3}_{j,\nu',\nu,l,m}+B^{1,2,2,3}_{j,\nu,\nu',l,m})\\
\nn &=& \sum_{p, q\geq 0}c_{pq}\int_{\MM}\frac{1}{(2^{\frac{j}{2}}|N_\nu-N_{\nu'}|)^{p+q+4}}[h_{1,p,q}+h_{2,p,q}] d\MM+\textrm{ terms interverting }(\nu,\nu'),
\eea
where the scalar functions $h_{1,p,q}, h_{2,p,q}$ on $\MM$ are given by:
\bea\lab{mgen2}
h_{1,p,q}&=& \left(\int_{\S} \chi L(P_{\leq 2^j|\nu-\nu'|}\trc)(b-b_{\nu'})^r\left(2^{\frac{j}{2}}(N-N_\nu)\right)^pF_{j,-1}(u)\eta_j^\nu(\o)d\o\right)\\
\nn&&\times\left(\int_{\S}P_{\leq 2^j|\nu-\nu'|}\trc'(b_{\nu'}-b')^s\left(2^{\frac{j}{2}}(N'-N_{\nu'})\right)^qF_{j,-1}(u')\eta_j^{\nu'}(\o')d\o'\right),
\eea
and:
\bea\lab{mgen3bis}
h_{2,p,q}&=& \left(\int_{\S} L(P_{\leq 2^j|\nu-\nu'|}\trc)(b-b_{\nu'})^r\left(2^{\frac{j}{2}}(N-N_\nu)\right)^pF_{j,-1}(u)\eta_j^\nu(\o)d\o\right)\\
\nn&&\times\left(\int_{\S}\chi' P_{\leq 2^j|\nu-\nu'|}\trc'(b_{\nu'}-b')^s\left(2^{\frac{j}{2}}(N'-N_{\nu'})\right)^qF_{j,-1}(u')\eta_j^{\nu'}(\o')d\o'\right),
\eea
where the integer $r, s$ satisfy:
$$r+s=1,$$
and where $c_{pq}$ are explicit real coefficients such that the series 
$$\sum_{p, q\geq 0}c_{pq}x^py^q$$
has radius of convergence 1. 

Next, we estimate the $L^1(\MM)$ norm of $h_{1,p,q}, h_{2,p,q}$ starting with $h_{1,p,q}$. We consider first the case $r=1$ and $s=0$ which is easier. In view of the definition \eqref{mgen2} of $h_{1,p,q}$ in the case $r=1$ and $s=0$,  we have:
\bea\lab{mgen3}
\nn\norm{h_{1,p,q}}_{L^1(\MM)}&\les & \normm{\int_{\S} \chi L(P_{\leq 2^j|\nu-\nu'|}\trc)(b-b_{\nu'})\left(2^{\frac{j}{2}}(N-N_\nu)\right)^pF_{j,-1}(u)\eta_j^\nu(\o)d\o}_{L^2(\MM)}\\
&&\times\normm{\int_{\S}P_{\leq 2^j|\nu-\nu'|}\trc'\left(2^{\frac{j}{2}}(N'-N_{\nu'})\right)^qF_{j,-1}(u')\eta_j^{\nu'}(\o')d\o'}_{L^2(\MM)}.
\eea
We estimate the two terms in the right-hand side of \eqref{mgen3} starting with the first one. The basic estimate in $L^2(\MM)$ \eqref{oscl2bis} yields:
\bea\lab{mgen4}
&&\normm{\int_{\S} \chi L(P_{\leq 2^j|\nu-\nu'|}\trc)(b-b_{\nu'})\left(2^{\frac{j}{2}}(N-N_\nu)\right)^pF_{j,-1}(u)\eta_j^\nu(\o)d\o}_{L^2(\MM)}\\
\nn&\les& \left(\sup_\o\normm{\chi L(P_{\leq 2^j|\nu-\nu'|}\trc)(b-b_{\nu'})\left(2^{\frac{j}{2}}(N-N_\nu)\right)^p}_{\li{\infty}{2}}\right)2^{\frac{j}{2}}\gamma^\nu_j\\
\nn&\les& |\nu-\nu'|\left(\sup_\o\norm{\chi L(P_{\leq 2^j|\nu-\nu'|}\trc)}_{\li{\infty}{2}}\right)2^{\frac{j}{2}}\gamma^\nu_j.
\eea
where we used in the last inequality the estimate \eqref{estricciomega} for $\po b$, the estimate \eqref{estNomega} for $\po N$, and the size of the patch. Now, the estimate \eqref{raviole} yields:
$$\norm{\chi L(P_{\leq 2^j|\nu-\nu'|}\trc)}_{\li{\infty}{2}}\les \ep \no(\chi)\les \ep,$$
where we used in the last inequality the estimates \eqref{esttrc} \eqref{esthch} for $\chi$. Together with \eqref{mgen4}, we obtain:
\be\lab{mgen5}
\normm{\int_{\S} \chi L(P_{\leq 2^j|\nu-\nu'|}\trc)(b-b_{\nu'})\left(2^{\frac{j}{2}}(N-N_\nu)\right)^pF_{j,-1}(u)\eta_j^\nu(\o)d\o}_{L^2(\MM)}\les \ep |\nu-\nu'| 2^{\frac{j}{2}}\gamma^\nu_j.
\ee
Next, we estimate the second term in the right-hand side of \eqref{mgen3}. We have:
\bee
&&\normm{\int_{\S}P_{\leq 2^j|\nu-\nu'|}\trc'\left(2^{\frac{j}{2}}(N'-N_{\nu'})\right)^qF_{j,-1}(u')\eta_j^{\nu'}(\o')d\o'}_{L^2(\MM)}\\
&\les&\normm{\int_{\S}\trc'\left(2^{\frac{j}{2}}(N'-N_{\nu'})\right)^qF_{j,-1}(u')\eta_j^{\nu'}(\o')d\o'}_{L^2(\MM)}\\
&&+\normm{\int_{\S}P_{> 2^j|\nu-\nu'|}\trc'\left(2^{\frac{j}{2}}(N'-N_{\nu'})\right)^qF_{j,-1}(u')\eta_j^{\nu'}(\o')d\o'}_{L^2(\MM)},
\eee
which together with the estimate \eqref{nycc133}, the estimate \eqref{loebter} and the fact that $2^{\frac{j}{2}}|\nu-\nu'|\gtrsim 1$ yields: 
\be\lab{mgen7}
\normm{\int_{\S}P_{\leq 2^j|\nu-\nu'|}\trc'\left(2^{\frac{j}{2}}(N'-N_{\nu'})\right)^qF_{j,-1}(u')\eta_j^{\nu'}(\o')d\o'}_{L^2(\MM)}\les \ep (1+q^2)\ep\gamma^{\nu'}_j.
\ee
Finally, \eqref{mgen3}, \eqref{mgen5} and \eqref{mgen7} imply in the case $r=1$ and $s=0$:
\be\lab{mgen9prems}
\norm{h_{1,p,q}}_{L^1(\MM)}\les  (1+q^2)|\nu-\nu'|2^{\frac{j}{2}}\ep^2\gamma^\nu_j\gamma^{\nu'}_j.
\ee
Next, we consider the case $r=0$ and $s=1$. We decompose $h_{1,p,q}$ as:
\be\lab{fix}
h_{1,p,q}=h_{1,p,q,1}+h_{1,p,q,2}+h_{1,p,q,3},
\ee
where $h_{1,p,q,1}$, $h_{1,p,q,2}$ and $h_{1,p,q,3}$ are given respectively by
\bea\lab{fix1}
h_{1,p,q,1}&=& \left(\int_{\S} \chi L(P_{\leq 2^j|\nu-\nu'|}\trc)\left(2^{\frac{j}{2}}(N-N_\nu)\right)^pF_{j,-1}(u)\eta_j^\nu(\o)d\o\right)\\
\nn&&\times\left(\int_{\S}P_{>2^j|\nu-\nu'|}\trc'(b_{\nu'}-b')\left(2^{\frac{j}{2}}(N'-N_{\nu'})\right)^qF_{j,-1}(u')\eta_j^{\nu'}(\o')d\o'\right),
\eea
\bea\lab{fix2}
h_{1,p,q,2}&=& \left(\int_{\S} \chi L(P_{> 2^j|\nu-\nu'|}\trc)\left(2^{\frac{j}{2}}(N-N_\nu)\right)^pF_{j,-1}(u)\eta_j^\nu(\o)d\o\right)\\
\nn&&\times\left(\int_{\S}\trc'(b_{\nu'}-b')\left(2^{\frac{j}{2}}(N'-N_{\nu'})\right)^qF_{j,-1}(u')\eta_j^{\nu'}(\o')d\o'\right).
\eea
and
\bea\lab{fix3}
h_{1,p,q,3}&=& \left(\int_{\S} \chi L(\trc)\left(2^{\frac{j}{2}}(N-N_\nu)\right)^pF_{j,-1}(u)\eta_j^\nu(\o)d\o\right)\\
\nn&&\times\left(\int_{\S}\trc'(b_{\nu'}-b')\left(2^{\frac{j}{2}}(N'-N_{\nu'})\right)^qF_{j,-1}(u')\eta_j^{\nu'}(\o')d\o'\right).
\eea
Next, we estimate the $L^1(\MM)$ norm of $h_{1,p,q,1}$, $h_{1,p,q,2}$ and $h_{1,p,q,3}$ starting with $h_{1,p,q,1}$.  In view of the definition \eqref{fix1} of $h_{1,p,q,1}$, we have
\bea\lab{mgen3repet}
\norm{h_{1,p,q,1}}_{L^1(\MM)}&\les & \normm{\int_{\S} \chi L(P_{\leq 2^j|\nu-\nu'|}\trc)\left(2^{\frac{j}{2}}(N-N_\nu)\right)^pF_{j,-1}(u)\eta_j^\nu(\o)d\o}_{L^2(\MM)}\\
\nn&&\times\normm{\int_{\S}P_{> 2^j|\nu-\nu'|}\trc'(b_{\nu'}-b')\left(2^{\frac{j}{2}}(N'-N_{\nu'})\right)^qF_{j,-1}(u')\eta_j^{\nu'}(\o')d\o'}_{L^2(\MM)}.
\eea
We estimate the two terms in the right-hand side of \eqref{mgen3} starting with the first one. Proceeding as for the proof of \eqref{mgen5}, and noticing that the only difference is the missing factor of $b-b_{\nu'}$, we obtain:
\be\lab{mgen5repet}
\normm{\int_{\S} \chi L(P_{\leq 2^j|\nu-\nu'|}\trc)\left(2^{\frac{j}{2}}(N-N_\nu)\right)^pF_{j,-1}(u)\eta_j^\nu(\o)d\o}_{L^2(\MM)}\les \ep  2^{\frac{j}{2}}\gamma^\nu_j.
\ee
Next, we estimate the second term in the right-hand side of \eqref{mgen3}. The analog of \eqref{stosur28} yields:
\be\lab{mgen8repet}
\normm{\int_{\S}P_{> 2^j|\nu-\nu'|}\trc'(b_{\nu'}-b')\left(2^{\frac{j}{2}}(N'-N_{\nu'})\right)^qF_{j,-1}(u')\eta_j^{\nu'}(\o')d\o'}_{L^2(\MM)}\les \frac{\ep 2^{-\frac{j}{2}}\gamma^{\nu'}_j}{2^{\frac{j}{2}}|\nu-\nu'|}.
\ee
\eqref{mgen3repet}, \eqref{mgen5repet} and \eqref{mgen8repet} imply:
\be\lab{fix4}
\norm{h_{1,p,q,1}}_{L^1(\MM)}\les \frac{\ep^2\gamma^\nu_j\gamma^{\nu'}_j}{2^{\frac{j}{2}}|\nu-\nu'|}.
\ee
Next, we estimate the $L^1(\MM)$ norm of $h_{1,p,q,2}$. In view of the definition \eqref{fix2} of $h_{1,p,q,2}$, we have
$$h_{1,p,q,2}=\int_{\S} \chi HL(P_{> 2^j|\nu-\nu'|}\trc)\left(2^{\frac{j}{2}}(N-N_\nu)\right)^pF_{j,-1}(u)\eta_j^\nu(\o)d\o,$$
where $H$ is given by
$$H=\int_{\S}\trc'(b_{\nu'}-b')\left(2^{\frac{j}{2}}(N'-N_{\nu'})\right)^qF_{j,-1}(u')\eta_j^{\nu'}(\o')d\o'.$$
This yields
\bea\lab{fix5}
&&\norm{h_{1,p,q,2}}_{L^1(\MM)}\\
\nn&\les& \int_{\S} \normm{\chi HL(P_{> 2^j|\nu-\nu'|}\trc)\left(2^{\frac{j}{2}}(N-N_\nu)\right)^pF_{j,-1}(u)}_{L^1(\MM)}\eta_j^\nu(\o)d\o\\
\nn&\les&  \int_{\S} \normm{\chi}_{L^\infty_{u, x'}L^2_t}\norm{H}_{L^2_{u,x'}L^\infty_t}\norm{L(P_{> 2^j|\nu-\nu'|}\trc)}_{\li{\infty}{2}}\normm{\left(2^{\frac{j}{2}}(N-N_\nu)\right)^p}_{L^\infty}\\
\nn&&\times\norm{F_{j,-1}(u)}_{L^2_u}\eta_j^\nu(\o)d\o\\
\nn&\les& \frac{\ep}{(2^j|\nu-\nu'|)^{\frac{1}{2}}}\int_{\S} \norm{H}_{L^2_{u,x'}L^\infty_t}\norm{F_{j,-1}(u)}_{L^2_u}\eta_j^\nu(\o)d\o,
\eea
where we used in the last inequality the estimates \eqref{esttrc} \eqref{esthch} for $\chi$, the estimate \eqref{estNomega} for $\po N$, the estimate \eqref{celeri1} for $L(P_{> 2^j|\nu-\nu'|}\trc)$, and the size of the patch. In view of the definition of $H$ and the estimate \eqref{bis:koko1}, we have
$$\norm{H}_{L^2_{u,x'}L^\infty_t}\les 2^{-\frac{j}{4}}(1+q^{\frac{5}{2}})\ep\big(2^{\frac{j}{2}}|\nu-\nu'|+1\big)\gamma^{\nu'}_j,$$
which together with \eqref{fix5} implies
$$\norm{h_{1,p,q,2}}_{L^1(\MM)}\les\frac{\ep}{(2^j|\nu-\nu'|)^{\frac{1}{2}}}2^{-\frac{j}{4}}(1+q^{\frac{5}{2}})\ep\big(2^{\frac{j}{2}}|\nu-\nu'|+1\big)\gamma^{\nu'}_j\int_{\S}\norm{F_{j,-1}(u)}_{L^2_u}\eta_j^\nu(\o)d\o.$$
Taking Plancherel in $\la$, Cauchy-Schwarz in $\o$ and using the size of the patch, we finally obtain:
\be\lab{fix6}
\norm{h_{1,p,q,2}}_{L^1(\MM)}\les (1+q^{\frac{5}{2}})\big(2^{\frac{j}{2}}|\nu-\nu'|\big)^{\frac{1}{2}}\ep^2\gamma^\nu_j\gamma^{\nu'}_j.
\ee
Next, we estimate the $L^1(\MM)$ norm of $h_{1,p,q,3}$. In view of the Raychaudhuri equation \eqref{raychaudhuri} satisfied by $\trc$, the worst term in $\chi L(\trc)$ is of the form $\hch^3$. In view of the decomposition \eqref{dechch3om} for $\hch^3$, we obtain the following decomposition for $\chi L(\trc)$: 
\bea\lab{fix7}
\chi L(\trc) &=&{\chi_2}_\nu^3+{\chi_2}_\nu^2F^j_1+{\chi_2}_\nu^2F^j_2+{\chi_2}_\nu F^j_3+{\chi_2}_\nu F^j_4\\
\nn&&+{\chi_2}_\nu F^j_5+F^j_6+F^j_7+F^j_8+F^j_9
\eea
where $F^j_1$, $F^j_3$ and $F^j_6$ only depend on $(t,x)$ and $\nu$ and satisfy:
\be\lab{fix8}
\norm{F^j_1}_{L^\infty_{u_\nu}L^2_tL^\infty(P_{t, u_\nu})}+\norm{F^j_3}_{L^\infty_{u_\nu}L^2_tL^\infty(P_{t, u_\nu})}+\norm{F^j_6}_{L^\infty_{u_\nu}L^2_tL^\infty(P_{t, u_\nu})}\les \ep,
\ee
where $F^j_2$, $F^j_4$ and $F^j_7$ satisfy:
\be\lab{fix9}
\norm{F^j_2}_{L^\infty_u\lh{2}}+\norm{F^j_4}_{L^\infty_u\lh{2}}+\norm{F^j_7}_{L^\infty_u\lh{2}}\les 2^{-\frac{j}{2}}\ep,
\ee
where $F^j_5$ and $F^j_8$ satisfy
\be\lab{fix10}
\norm{F^j_5}_{L^2(\mathcal{M})}+\norm{F^j_8}_{L^2(\mathcal{M})}\les \ep 2^{-j}.
\ee
and where $F^j_9$ satisfies
\be\lab{fix11}
\norm{F^j_9}_{L^{2_-}(\mathcal{M})}\les \ep 2^{-\frac{3j}{2}}.
\ee
In view of the definition \eqref{fix3} of $h_{1,p,q,3}$, this yields the following decomposition
\be\lab{fix12}
h_{1,p,q,3}=h_{1,p,q,3,1}+h_{1,p,q,3,2}+h_{1,p,q,3,3}+h_{1,p,q,3,4}+h_{1,p,q,3,5}+h_{1,p,q,3,6}+h_{1,p,q,3,7},
\ee
where $h_{1,p,q,3,1}$, $h_{1,p,q,3,2}$, $h_{1,p,q,3,3}$, $h_{1,p,q,3,4}$, $h_{1,p,q,3,5}$, $h_{1,p,q,3,6}$ and $h_{1,p,q,3,7}$ are given by: 
\be\lab{fix13}
h_{1,p,q,3,1}= ({\chi_2}_\nu^3+{\chi_2}_\nu^2F^j_1+{\chi_2}_\nu F^j_3+F^j_6)\left(\int_{\S}\left(2^{\frac{j}{2}}(N-N_\nu)\right)^pF_{j,-1}(u)\eta_j^\nu(\o)d\o\right)H,
\ee
\be\lab{fix14}
h_{1,p,q,3,2}= {\chi_2}_\nu^2\left(\int_{\S}F^j_2\left(2^{\frac{j}{2}}(N-N_\nu)\right)^pF_{j,-1}(u)\eta_j^\nu(\o)d\o\right)H,
\ee
\be\lab{fix15}
h_{1,p,q,3,3}= {\chi_2}_\nu\left(\int_{\S}F^j_4\left(2^{\frac{j}{2}}(N-N_\nu)\right)^pF_{j,-1}(u)\eta_j^\nu(\o)d\o\right)H,
\ee
\be\lab{fix16}
h_{1,p,q,3,4}= {\chi_2}_\nu\left(\int_{\S}F^j_5\left(2^{\frac{j}{2}}(N-N_\nu)\right)^pF_{j,-1}(u)\eta_j^\nu(\o)d\o\right)H
\ee
\be\lab{fix17}
h_{1,p,q,3,5}= \left(\int_{\S}F^j_7\left(2^{\frac{j}{2}}(N-N_\nu)\right)^pF_{j,-1}(u)\eta_j^\nu(\o)d\o\right)H,
\ee
\be\lab{fix18}
h_{1,p,q,3,6}= \left(\int_{\S}F^j_8\left(2^{\frac{j}{2}}(N-N_\nu)\right)^pF_{j,-1}(u)\eta_j^\nu(\o)d\o\right)H,
\ee
and
\be\lab{fix19}
h_{1,p,q,3,7}= \left(\int_{\S}F^j_9\left(2^{\frac{j}{2}}(N-N_\nu)\right)^pF_{j,-1}(u)\eta_j^\nu(\o)d\o\right)H,
\ee
with $H$ given by:
\be\lab{fix20}
H=\int_{\S}\trc'(b_{\nu'}-b')\left(2^{\frac{j}{2}}(N'-N_{\nu'})\right)^qF_{j,-1}(u')\eta_j^{\nu'}(\o')d\o'.
\ee
Using the basic estimate \eqref{oscl2bis}, we have for $n=2, 4, 7$
\bea\lab{fix21}
&&\normm{\int_{\S}F^j_n\left(2^{\frac{j}{2}}(N-N_\nu)\right)^pF_{j,-1}(u)\eta_j^\nu(\o)d\o}_{L^2(\MM)}\\
\nn&\les& \left(\sup_\o\normm{F^j_n\left(2^{\frac{j}{2}}(N-N_\nu)\right)^p}_{\li{\infty}{2}}\right)2^{\frac{j}{2}}\gamma^\nu_j\\
\nn&\les& \left(\sup_\o\normm{F^j_n}_{\li{\infty}{2}}\normm{\left(2^{\frac{j}{2}}(N-N_\nu)\right)^p}_{L^\infty}\right)2^{\frac{j}{2}}\gamma^\nu_j\\
\nn&\les& \ep\gamma^\nu_j,
\eea
where we used in the last inequality the estimate \eqref{fix9}, the estimate \eqref{estNomega} for $\po N$ and the size of the patch. Also, we have for $n=5, 8$
\bee
&&\normm{\int_{\S}F^j_n\left(2^{\frac{j}{2}}(N-N_\nu)\right)^pF_{j,-1}(u)\eta_j^\nu(\o)d\o}_{L^2(\MM)}\\
&\les&\int_{\S}\normm{F^j_n\left(2^{\frac{j}{2}}(N-N_\nu)\right)^pF_{j,-1}(u)}_{L^2(\MM)}\eta_j^\nu(\o)d\o\\
&\les&\int_{\S}\norm{F^j_n}_{L^2(\MM)}\normm{\left(2^{\frac{j}{2}}(N-N_\nu)\right)^p}_{L^\infty}\norm{F_{j,-1}(u)}_{L^\infty_u}\eta_j^\nu(\o)d\o\\
&\les& 2^{-j}\ep\int_{\S}\norm{F_{j,-1}(u)}_{L^\infty_u}\eta_j^\nu(\o)d\o,
\eee
where we used in the last inequality the estimate \eqref{fix10}, the estimate \eqref{estNomega} for $\po N$ and the size of the patch. Taking Cauchy-Schwarz both in $\la$ and $\o$ and using the size of the patch, we obtain  for $n=5, 8$:
\be\lab{fix22}
\normm{\int_{\S}F^j_n\left(2^{\frac{j}{2}}(N-N_\nu)\right)^pF_{j,-1}(u)\eta_j^\nu(\o)d\o}_{L^2(\MM)}\les\ep\gamma^\nu_j.
\ee
Also, we have
\bee
&&\normm{\int_{\S}F^j_9\left(2^{\frac{j}{2}}(N-N_\nu)\right)^pF_{j,-1}(u)\eta_j^\nu(\o)d\o}_{L^{2_-}(\MM)}\\
&\les&\int_{\S}\normm{F^j_9\left(2^{\frac{j}{2}}(N-N_\nu)\right)^pF_{j,-1}(u)}_{L^{2_-}(\MM)}\eta_j^\nu(\o)d\o\\
&\les&\int_{\S}\norm{F^j_9}_{L^{2_-}(\MM)}\normm{\left(2^{\frac{j}{2}}(N-N_\nu)\right)^p}_{L^\infty}\norm{F_{j,-1}(u)}_{L^\infty_u}\eta_j^\nu(\o)d\o\\
&\les& 2^{-\frac{3j}{2}}\ep\int_{\S}\norm{F_{j,-1}(u)}_{L^\infty_u}\eta_j^\nu(\o)d\o,
\eee
where we used in the last inequality the estimate \eqref{fix11}, the estimate \eqref{estNomega} for $\po N$ and the size of the patch. Taking Cauchy-Schwarz both in $\la$ and $\o$ and using the size of the patch, we obtain:
\be\lab{fix23}
\normm{\int_{\S}F^j_9\left(2^{\frac{j}{2}}(N-N_\nu)\right)^pF_{j,-1}(u)\eta_j^\nu(\o)d\o}_{L^{2_-}(\MM)}\les 2^{-\frac{j}{2}}\ep\gamma^\nu_j.
\ee
Finally, let $G$ be given by
$$G=\int_{\S}\left(2^{\frac{j}{2}}(N-N_\nu)\right)^pF_{j,-1}(u)\eta_j^\nu(\o)d\o.$$
Then, we have in view of Lemma \ref{lemma:loeb}, we have
\be\lab{fix24}
\norm{G}_{L^2_{u_\nu, {x'}_\nu}L^\infty_t}\les (1+p^2)\gamma^\nu_j.
\ee
\eqref{fix12}-\eqref{fix23} imply
\bea\lab{fix25}
&&\norm{h_{1,p,q,3}}_{L^1(\MM)}\\
\nn&\les& \norm{({\chi_2}^2_\nu, {\chi_2}_\nu, 1)H}_{L^2(\MM)}(\norm{({\chi_2}_\nu, F^j_1, F^j_3, F^j_6)G}_{L^2(\MM)}+\ep\gamma^\nu_j)+\norm{H}_{L^{2_+}(\MM)}2^{-\frac{j}{2}}\ep\gamma^\nu_j\\
\nn&\les& (\norm{{\chi_2}^2_\nu H}_{L^2(\MM)}+\norm{{\chi_2}^2_\nu}_{L^6(\MM)}\norm{H}_{L^3(\MM)}+\norm{H}_{L^2(\MM)})\\
\nn&&\times(\norm{{\chi_2}_\nu, F^j_1, F^j_3, F^j_6}_{L^\infty_{u_\nu, {x'}_\nu}L^2_t}\norm{G}_{L^2_{u_\nu, {x'}_\nu}L^\infty_t}+\ep\gamma^\nu_j)+\norm{H}_{L^{2_+}(\MM)}2^{-\frac{j}{2}}\ep\gamma^\nu_j\\
\nn&\les& (\norm{{\chi_2}^2_\nu H}_{L^2(\MM)}+\norm{H}_{L^3(\MM)}+\norm{H}_{L^2(\MM)})(1+p^2)\ep\gamma^\nu_j,
\eea
where we used in the last inequality the estimate \eqref{dechch1} for ${\chi_2}_\nu$, the estimate \eqref{fix8} for $F^j_1$, $F^j_3$ and $F^j_6$, and the estimate \eqref{fix24} for $G$. Next, we estimate $H$. In view of its definition \eqref{fix20}, we have from \eqref{facilfamille4}
\be\lab{fix26}
\norm{H}_{L^2(\MM)}\les 2^{-\frac{j}{4}}\ep(1+q^2)\gamma^{\nu'}_j.
\ee
Also, we have
\bee
\norm{H}_{L^\infty(\MM)}&\les& \int_{\S}\normm{\trc'(b_{\nu'}-b')\left(2^{\frac{j}{2}}(N'-N_{\nu'})\right)^qF_{j,-1}(u')}_{L^\infty}\eta_j^{\nu'}(\o')d\o'\\
&\les& 2^{-\frac{j}{2}}\ep\int_{\S}\norm{F_{j,-1}(u')}_{L^\infty_{u'}}\eta_j^{\nu'}(\o')d\o',
\eee
where we used in the last inequality the estimate \eqref{esttrc} for $\trc'$, the estimate \eqref{estricciomega} for $\po b'$, the estimate \eqref{estNomega} for $\po N$, and the size of the patch. Taking Cauchy-Schwarz in $\la'$ and $\o'$, and using the size of the patch, we obtain
\be\lab{fix27}
\norm{H}_{L^\infty(\MM)}\les 2^{\frac{j}{2}}\ep\gamma^{\nu'}_j.
\ee
In particular, interpolating \eqref{fix26} and \eqref{fix27}, we obtain
$$\norm{H}_{L^3(\MM)}\les (1+q^2)^{\frac{2}{3}}\ep\gamma^{\nu'}_j,$$
which together with \eqref{fix25} and \eqref{fix26} implies
\be\lab{fix28}
\norm{h_{1,p,q,3}}_{L^1(\MM)}\les \left(\norm{{\chi_2}^2_\nu H}_{L^2(\MM)}+(1+q^2)^{\frac{2}{3}}\ep\gamma^{\nu'}_j\right)(1+p^2)\ep\gamma^\nu_j.
\ee
Thus, in view of \eqref{fix28}, it remains to estimate $\norm{{\chi_2}^2_\nu H}_{L^2(\MM)}$. We decompose ${\chi_2}^2_\nu H$:
\be\lab{fix29}
{\chi_2}^2_\nu H={\chi_2}_\nu H_1+H_2+H_3,
\ee
where $H_1$, $H_2$ and $H_3$ are given by
\be\lab{fix30}
H_1=\int_{\S}({\chi_2}_\nu-{\chi_2}')\trc'(b_{\nu'}-b')\left(2^{\frac{j}{2}}(N'-N_{\nu'})\right)^qF_{j,-1}(u')\eta_j^{\nu'}(\o')d\o',
\ee
\be\lab{fix31}
H_2=\int_{\S}({\chi_2}_\nu-{\chi_2}'){\chi_2}'\trc'(b_{\nu'}-b')\left(2^{\frac{j}{2}}(N'-N_{\nu'})\right)^qF_{j,-1}(u')\eta_j^{\nu'}(\o')d\o',
\ee
and
\be\lab{fix32}
H_3=\int_{\S}{{\chi_2}'}^2\trc'(b_{\nu'}-b')\left(2^{\frac{j}{2}}(N'-N_{\nu'})\right)^qF_{j,-1}(u')\eta_j^{\nu'}(\o')d\o'.
\ee
In view of \eqref{fix29}, we have
\bea\lab{fix33}
\norm{{\chi_2}^2_\nu H}_{L^2(\MM)}&\les& \norm{{\chi_2}_\nu H_1}_{L^2(\MM)}+\norm{H_2}_{L^2(\MM)}+\norm{H_3}_{L^2(\MM)}\\
\nn&\les& \norm{{\chi_2}_\nu}_{L^6(\MM)}\norm{H_1}_{L^3(\MM)}+\norm{H_2}_{L^2(\MM)}+\norm{H_3}_{L^2(\MM)}\\
\nn&\les& \ep\norm{H_1}_{L^3(\MM)}+\norm{H_2}_{L^2(\MM)}+\norm{H_3}_{L^2(\MM)},
\eea
where we used in the last inequality the estimate \eqref{dechch1} for ${\chi_2}_\nu$. Next, we estimate each term in the right-hand side of \eqref{fix33} starting with the first one. In view of the definition \eqref{fix30} of $H_1$ and the estimate \eqref{osclpbis}, we have
\bea\lab{fix34}
\nn\norm{H_1}_{L^3(\MM)}&\les& \left(\sup_{\o'}\normm{({\chi_2}_\nu-{\chi_2}')\trc'(b_{\nu'}-b')\left(2^{\frac{j}{2}}(N'-N_{\nu'})\right)^q}_{L^\infty_{u'}L^3(\H_{u'})}\right)2^{\frac{2j}{3}}\gamma^{\nu'}_j\\
\nn&\les& \left(\sup_{\o'}\norm{{\chi_2}_\nu-{\chi_2}'}_{L^\infty_{u'}L^3(\H_{u'})}\normm{\trc'(b_{\nu'}-b')\left(2^{\frac{j}{2}}(N'-N_{\nu'})\right)^q}_{L^\infty}\right)2^{\frac{2j}{3}}\gamma^{\nu'}_j\\
&\les& |\nu-\nu'|2^{\frac{j}{6}}\ep\gamma^{\nu'}_j,
\eea
where we used in the last inequality the estimate \eqref{decchi2om} for ${\chi_2}_\nu-{\chi_2}'$, the estimate \eqref{esttrc} for $\trc'$, the estimate \eqref{estricciomega} for $\po b$, the estimate \eqref{estNomega} for $\po N$, and the size of the patch. Next, we estimate $H_2$. In view of the definition \eqref{fix31} of $H_2$ and the estimate \eqref{oscl2bis}, we have
\bea\lab{fix35}
&&\norm{H_2}_{L^2(\MM)}\\
\nn&\les& \left(\sup_{\o'}\normm{({\chi_2}_\nu-{\chi_2}'){\chi_2}'\trc'(b_{\nu'}-b')\left(2^{\frac{j}{2}}(N'-N_{\nu'})\right)^q}_{L^\infty_{u'}L^2(\H_{u'})}\right)2^{\frac{j}{2}}\gamma^{\nu'}_j\\
\nn&\les& \left(\sup_{\o'}\norm{{\chi_2}_\nu-{\chi_2}'}_{L^\infty_{u'}L^3(\H_{u'})}\norm{{\chi_2}'}_{L^\infty_{u'}L^6(\H_{u'})}\normm{\trc'(b_{\nu'}-b')\left(2^{\frac{j}{2}}(N'-N_{\nu'})\right)^q}_{L^\infty}\right)2^{\frac{j}{2}}\gamma^{\nu'}_j\\
\nn&\les& |\nu-\nu'|\ep\gamma^{\nu'}_j,
\eea
where we used in the last inequality the estimate \eqref{decchi2om} for ${\chi_2}_\nu-{\chi_2}'$, the estimate \eqref{dechch1} for ${\chi_2}'$, the estimate \eqref{esttrc} for $\trc'$, the estimate \eqref{estricciomega} for $\po b$, the estimate \eqref{estNomega} for $\po N$, and the size of the patch. Next, we estimate $H_3$. In view of the definition \eqref{fix32} of $H_3$ and the estimate \eqref{oscl2bis}, we have
\bea\lab{fix36}
\norm{H_3}_{L^2(\MM)}&\les& \left(\sup_{\o'}\normm{{{\chi_2}'}^2\trc'(b_{\nu'}-b')\left(2^{\frac{j}{2}}(N'-N_{\nu'})\right)^q}_{L^\infty_{u'}L^2(\H_{u'})}\right)2^{\frac{j}{2}}\gamma^{\nu'}_j\\
\nn&\les& \left(\sup_{\o'}\norm{{\chi_2}'}_{L^\infty_{u'}L^4(\H_{u'})}^2\normm{\trc'(b_{\nu'}-b')\left(2^{\frac{j}{2}}(N'-N_{\nu'})\right)^q}_{L^\infty}\right)2^{\frac{j}{2}}\gamma^{\nu'}_j\\
\nn&\les& \ep\gamma^{\nu'}_j,
\eea
where we used in the last inequality the estimate \eqref{dechch1} for ${\chi_2}'$, the estimate \eqref{esttrc} for $\trc'$, the estimate \eqref{estricciomega} for $\po b$, the estimate \eqref{estNomega} for $\po N$, and the size of the patch. Finally, \eqref{fix33}-\eqref{fix36} yield
$$\norm{{\chi_2}^2_\nu H}_{L^2(\MM)}\les (1+|\nu-\nu'|2^{\frac{j}{6}})\ep\gamma^{\nu'}_j.$$
In view \eqref{fix28}, we obtain:
$$\norm{h_{1,p,q,3}}_{L^1(\MM)}\les \left(|\nu-\nu'|2^{\frac{j}{6}}+(1+q^2)^{\frac{2}{3}}\right)(1+p^2)\ep^2\gamma^\nu_j\gamma^{\nu'}_j,$$
which together with \eqref{fix}, \eqref{fix4}, \eqref{fix6}, and the fact that $2^{\frac{j}{2}}|\nu-\nu'|\gtrsim 1$, implies in the case $r=0$ and $s=1$
\bea\lab{fix37}
\norm{h_{1,p,q}}_{L^1(\MM)}&\les& \norm{h_{1,p,q,1}}_{L^1(\MM)}+\norm{h_{1,p,q,2}}_{L^1(\MM)}+\norm{h_{1,p,q,3}}_{L^1(\MM)}\\
\nn&\les& (1+p^2)(1+q^{\frac{5}{2}})(2^{\frac{j}{2}}|\nu-\nu'|)^{\frac{1}{2}}\ep^2\gamma^\nu_j\gamma^{\nu'}_j.
\eea
Using \eqref{mgen9prems} in the case $r=1$ and $s=0$, and \eqref{fix37} in the case $r=0$ and $s=1$, together with the fact that $2^{\frac{j}{2}}|\nu-\nu'|\gtrsim 1$, we finally obtain
\be\lab{mgen9}
\norm{h_{1,p,q}}_{L^1(\MM)}\les (1+p^2)(1+q^{\frac{5}{2}})2^{\frac{j}{2}}|\nu-\nu'|\ep^2\gamma^\nu_j\gamma^{\nu'}_j
\ee

Next, we estimate the $L^1(\MM)$ norm of $h_{2,p,q}$. We start with the case $r=1$ and $s=0$ which is easier. In view of definition \eqref{mgen3bis}, we have:
\bea\lab{mgen10}
\norm{h_{2,p,q}}_{L^2(\MM)}&\les& \normm{\int_{\S} L(P_{\leq 2^j|\nu-\nu'|}\trc)(b-b_{\nu'})\left(2^{\frac{j}{2}}(N-N_\nu)\right)^pF_{j,-1}(u)\eta_j^\nu(\o)d\o}_{L^2(\MM)}\\
\nn&&\times\normm{\int_{\S}\chi' P_{\leq 2^j|\nu-\nu'|}\trc'\left(2^{\frac{j}{2}}(N'-N_{\nu'})\right)^qF_{j,-1}(u')\eta_j^{\nu'}(\o')d\o'}_{L^2(\MM)},
\eea
We estimate the two terms in the right-hand side of \eqref{mgen3} starting with the first one. The basic estimate in $L^2(\MM)$ \eqref{oscl2bis} yields:
\bea\lab{mgen11}
&&\normm{\int_{\S} L(P_{\leq 2^j|\nu-\nu'|}\trc)(b-b_{\nu'})\left(2^{\frac{j}{2}}(N-N_\nu)\right)^pF_{j,-1}(u)\eta_j^\nu(\o)d\o}_{L^2(\MM)}\\
\nn&\les& \left(\sup_\o\normm{L(P_{\leq 2^j|\nu-\nu'|}\trc)(b-b_{\nu'})\left(2^{\frac{j}{2}}(N-N_\nu)\right)^p}_{\li{\infty}{2}}\right)2^{\frac{j}{2}}\gamma^\nu_j\\
\nn&\les& |\nu-\nu'|\left(\sup_\o\norm{L(P_{\leq 2^j|\nu-\nu'|}\trc)}_{\li{\infty}{2}}\right)2^{\frac{j}{2}}\gamma^\nu_j.
\eea
where we used in the last inequality the estimate \eqref{estricciomega} for $\po b$, the estimate \eqref{estNomega} for $\po N$, and the size of the patch. Now, the estimate \eqref{raviole} yields:
\be\lab{mgen11bis}
\norm{L(P_{\leq 2^j|\nu-\nu'|}\trc)}_{\li{\infty}{2}}\les \ep.
\ee
Together with \eqref{mgen11}, we obtain:
\be\lab{mgen12}
\normm{\int_{\S} L(P_{\leq 2^j|\nu-\nu'|}\trc)(b-b_{\nu'})\left(2^{\frac{j}{2}}(N-N_\nu)\right)^pF_{j,-1}(u)\eta_j^\nu(\o)d\o}_{L^2(\MM)}\les \ep |\nu-\nu'| 2^{\frac{j}{2}}\gamma^\nu_j.
\ee
Next, we estimate the second term in the right-hand side of \eqref{mgen10}. Recall the decomposition \eqref{decchiom} for $\chi'$:
\be\lab{mgen13}
\chi'={\chi_2}_{\nu'}+F^j_1+F^j_2
\ee
where the tensor $F^j_1$ only depends on $\nu'$ and satisfies for any $2\leq p<+\infty$:
\be\lab{mgen14}
\norm{F^j_1}_{L^\infty_{u_{\nu'}}L^p_t L^\infty_{x'_{\nu'}}}\les \ep,
\ee
and where the tensor $F^j_2$ satisfies:
\be\lab{mgen15}
\norm{F^j_2}_{\lprime{\infty}{2}}\les \ep 2^{-\frac{j}{2}}.
\ee
The decomposition \eqref{mgen13} implies:
\bea\lab{mgen16}
&&\normm{\int_{\S}\chi' P_{\leq 2^j|\nu-\nu'|}\trc'\left(2^{\frac{j}{2}}(N'-N_{\nu'})\right)^qF_{j,-1}(u')\eta_j^{\nu'}(\o')d\o'}_{L^2(\MM)}\\
\nn&\les& (\normm{{\chi_2}_{\nu'}}_{L^\infty_{u_{\nu'}}, x'_{\nu'}L^2_t}+\normm{F^j_1}_{L^\infty_{u_{\nu'}}, x'_{\nu'}L^2_t})\\
\nn&&\times\normm{\int_{\S}P_{\leq 2^j|\nu-\nu'|}\trc'\left(2^{\frac{j}{2}}(N'-N_{\nu'})\right)^qF_{j,-1}(u')\eta_j^{\nu'}(\o')d\o'}_{L^2_{u_{\nu'}}, x'_{\nu'}L^\infty_t}\\
\nn&&+\normm{\int_{\S}F^j_2 P_{\leq 2^j|\nu-\nu'|}\trc'\left(2^{\frac{j}{2}}(N'-N_{\nu'})\right)^qF_{j,-1}(u')\eta_j^{\nu'}(\o')d\o'}_{L^2(\MM)}\\
\nn&\les& \ep\normm{\int_{\S}P_{\leq 2^j|\nu-\nu'|}\trc'\left(2^{\frac{j}{2}}(N'-N_{\nu'})\right)^qF_{j,-1}(u')\eta_j^{\nu'}(\o')d\o'}_{L^2_{u_{\nu'}}, x'_{\nu'}L^\infty_t}\\
\nn&&+\normm{\int_{\S}F^j_2 P_{\leq 2^j|\nu-\nu'|}\trc'\left(2^{\frac{j}{2}}(N'-N_{\nu'})\right)^qF_{j,-1}(u')\eta_j^{\nu'}(\o')d\o'}_{L^2(\MM)},
\eea
where we used in the last inequality the estimate \eqref{mgen14} for $F^j_1$ and the estimate \eqref{dechch1} for ${\chi_2}_{\nu'}$. Now, in view of the estimate \eqref{messi4:0}, we have for $m>j/2$:
\bea\lab{mgen17}
&&\normm{\int_{\S}P_m\trc'\left(2^{\frac{j}{2}}(N'-N_{\nu'})\right)^qF_{j,-1}(u')\eta_j^{\nu'}(\o')d\o'}_{L^2_{u_{\nu'}}, x'_{\nu'}L^\infty_t}\\
\nn&\les& \left(\sup_{\o'}\normm{\left(2^{\frac{j}{2}}(N'-N_{\nu'})\right)^q}_{L^\infty}\right)\ep\big(2^{-l+\frac{j}{2}}+2^{-\frac{l}{2}+\frac{j}{4}}\big)\gamma^{\nu'}_j\\
\nn&\les& \ep\big(2^{-l+\frac{j}{2}}+2^{-\frac{l}{2}+\frac{j}{4}}\big)\gamma^{\nu'}_j,
\eea
where we used in the last inequality the estimate \eqref{estNomega} for $\po N$ and the size of the patch. Also, in view of the estimate \eqref{koko1}, we have:
\be\lab{mgen18}
\normm{\int_{\S}\trc'\left(2^{\frac{j}{2}}(N'-N_{\nu'})\right)^qF_{j,-1}(u')\eta_j^{\nu'}(\o')d\o'}_{L^2_{u_{\nu'}}, x'_{\nu'}L^\infty_t}\les (1+q^{\frac{5}{2}})\ep\gamma^{\nu'}_j.
\ee
\eqref{mgen17}, \eqref{mgen18}, the fact that $2^{\frac{j}{2}}|\nu-\nu'|\gtrsim 1$, and the fact that:
$$P_{\leq 2^j|\nu-\nu'|}\trc'=\trc'-\sum_{m/2^m>2^j|\nu-\nu'|}P_m\trc'$$
implies:
\be\lab{mgen19}
\normm{\int_{\S}P_{\leq 2^j|\nu-\nu'|}\trc'\left(2^{\frac{j}{2}}(N'-N_{\nu'})\right)^qF_{j,-1}(u')\eta_j^{\nu'}(\o')d\o'}_{L^2_{u_{\nu'}}, x'_{\nu'}L^\infty_t}\les (1+q^{\frac{5}{2}})\ep\gamma^{\nu'}_j.
\ee
On the other hand, the basic estimate in $L^2(\MM)$ \eqref{oscl2bis} yields:
\bea\lab{mgen20}
&&\normm{\int_{\S}F^j_2 P_{\leq 2^j|\nu-\nu'|}\trc'\left(2^{\frac{j}{2}}(N'-N_{\nu'})\right)^qF_{j,-1}(u')\eta_j^{\nu'}(\o')d\o'}_{L^2(\MM)}\\
\nn&\les& \left(\sup_{\o'}\normm{F^j_2 P_{\leq 2^j|\nu-\nu'|}\trc'\left(2^{\frac{j}{2}}(N'-N_{\nu'})\right)^q}_{\lprime{\infty}{2}}\right)2^{\frac{j}{2}}\gamma^{\nu'}_j\\
\nn&\les& \left(\sup_{\o'}\norm{F^j_2}_{\lprime{\infty}{2}}\normm{P_{\leq 2^j|\nu-\nu'|}\trc'\left(2^{\frac{j}{2}}(N'-N_{\nu'})\right)^q}_{L^\infty(\MM)}\right)2^{\frac{j}{2}}\gamma^{\nu'}_j\\
\nn&\les& \ep \gamma^{\nu'}_j,
\eea
where we used in the last inequality the estimate \eqref{mgen15} for $F^j_2$, the boundedness of $P_{\leq 2^j|\nu-\nu'|}$ on $L^\infty(\ptu)$, the estimate \eqref{esttrc} for $\trc$, the estimate \eqref{estNomega} for $\po N$, and the size of the patch. Now, \eqref{mgen16}, \eqref{mgen19} and \eqref{mgen20} imply:
\be\lab{mgen21}
\normm{\int_{\S}\chi' P_{\leq 2^j|\nu-\nu'|}\trc'\left(2^{\frac{j}{2}}(N'-N_{\nu'})\right)^qF_{j,-1}(u')\eta_j^{\nu'}(\o')d\o'}_{L^2(\MM)}\les \ep (1+q^{\frac{5}{2}})\ep\gamma^{\nu'}_j.
\ee
Finally, \eqref{mgen10}, \eqref{mgen12} and \eqref{mgen21} imply in the case $r=1$ and $s=0$:
\be\lab{mgen22}
\norm{h_{2,p,q}}_{L^2(\MM)}\les  2^{\frac{j}{2}}|\nu-\nu'|  (1+q^{\frac{5}{2}})\ep^2\gamma^\nu_j\gamma^{\nu'}_j.
\ee

Next, we estimate $h_{2,p,q}$ in the case $r=0$ and $s=1$. In view of the definition \eqref{mgen3bis}, we may decompose $h_{2,p,q}$ as:
\be\lab{but:fix}
h_{2,p,q}=h_{2,p,q,1}+h_{2,p,q,2}+h_{2,p,q,3},
\ee
where $h_{2,p,q,1}$, $h_{2,p,q,2}$ and $h_{2,p,q,3}$ are given respectively by
\bea\lab{but:fix1}
h_{2,p,q,1}&=& \left(\int_{\S} L(P_{\leq 2^j|\nu-\nu'|}\trc)\left(2^{\frac{j}{2}}(N-N_\nu)\right)^pF_{j,-1}(u)\eta_j^\nu(\o)d\o\right)\\
\nn&&\times\left(\int_{\S}\chi' P_{>2^j|\nu-\nu'|}\trc'(b_{\nu'}-b')\left(2^{\frac{j}{2}}(N'-N_{\nu'})\right)^qF_{j,-1}(u')\eta_j^{\nu'}(\o')d\o'\right),
\eea
\bea\lab{but:fix2}
h_{2,p,q,2}&=& \left(\int_{\S} L(P_{> 2^j|\nu-\nu'|}\trc)\left(2^{\frac{j}{2}}(N-N_\nu)\right)^pF_{j,-1}(u)\eta_j^\nu(\o)d\o\right)\\
\nn&&\times\left(\int_{\S}\chi'\trc'(b_{\nu'}-b')\left(2^{\frac{j}{2}}(N'-N_{\nu'})\right)^qF_{j,-1}(u')\eta_j^{\nu'}(\o')d\o'\right).
\eea
and
\bea\lab{but:fix3}
h_{2,p,q,3}&=& \left(\int_{\S}L(\trc)\left(2^{\frac{j}{2}}(N-N_\nu)\right)^pF_{j,-1}(u)\eta_j^\nu(\o)d\o\right)\\
\nn&&\times\left(\int_{\S}\chi' \trc'(b_{\nu'}-b')\left(2^{\frac{j}{2}}(N'-N_{\nu'})\right)^qF_{j,-1}(u')\eta_j^{\nu'}(\o')d\o'\right).
\eea
Next, we estimate the $L^1(\MM)$ norm of $h_{2,p,q,1}$, $h_{2,p,q,2}$ and $h_{2,p,q,3}$ starting with $h_{2,p,q,1}$.  In view of the definition \eqref{but:fix1} of $h_{2,p,q,1}$, we have
\bea\lab{but:mgen3repet}
&&\norm{h_{2,p,q,1}}_{L^1(\MM)}\\
\nn&\les & \normm{\int_{\S} L(P_{\leq 2^j|\nu-\nu'|}\trc)\left(2^{\frac{j}{2}}(N-N_\nu)\right)^pF_{j,-1}(u)\eta_j^\nu(\o)d\o}_{L^2(\MM)}\\
\nn&&\times\normm{\int_{\S}\chi' P_{> 2^j|\nu-\nu'|}\trc'(b_{\nu'}-b')\left(2^{\frac{j}{2}}(N'-N_{\nu'})\right)^qF_{j,-1}(u')\eta_j^{\nu'}(\o')d\o'}_{L^2(\MM)}.
\eea
We estimate the two terms in the right-hand side of \eqref{but:mgen3repet} starting with the first one. The basic estimate \eqref{oscl2bis} yields
 \bea\lab{but:mgen5repet}
 &&\normm{\int_{\S} L(P_{\leq 2^j|\nu-\nu'|}\trc)\left(2^{\frac{j}{2}}(N-N_\nu)\right)^pF_{j,-1}(u)\eta_j^\nu(\o)d\o}_{L^2(\MM)}\\
\nn&\les& \left(\sup_\o\normm{L(P_{\leq 2^j|\nu-\nu'|}\trc)\left(2^{\frac{j}{2}}(N-N_\nu)\right)^p}_{\li{\infty}{2}}\right) 2^{\frac{j}{2}}\gamma^\nu_j\\
\nn&\les& \left(\sup_\o\norm{L(P_{\leq 2^j|\nu-\nu'|}\trc)}_{\li{\infty}{2}}\normm{\left(2^{\frac{j}{2}}(N-N_\nu)\right)^p}_{L^\infty}\right) 2^{\frac{j}{2}}\gamma^\nu_j\\
\nn&\les& 2^{\frac{j}{2}}\ep\gamma^\nu_j,
\eea
where we used in the last inequality the estimate \eqref{estNomega} for $\po N$, the size of the patch, and the following estimate
$$\norm{L(P_{\leq 2^j|\nu-\nu'|}\trc)}_{\li{\infty}{2}}\les \ep,$$
which follows from the estimates \eqref{vino42} and \eqref{vino43}. Next, we estimate the second term in the right-hand side of \eqref{but:mgen3repet}. Using the basic estimate \eqref{oscl2bis}, we have
\bea\lab{but:mgen8repet}
&&\normm{\int_{\S}\chi' P_{> 2^j|\nu-\nu'|}\trc'(b_{\nu'}-b')\left(2^{\frac{j}{2}}(N'-N_{\nu'})\right)^qF_{j,-1}(u')\eta_j^{\nu'}(\o')d\o'}_{L^2(\MM)}\\
\nn&\les& \left(\sup_{\o'}\normm{\chi' P_{> 2^j|\nu-\nu'|}\trc'(b_{\nu'}-b')\left(2^{\frac{j}{2}}(N'-N_{\nu'})\right)^q}_{L^\infty_{u'}L^2(\H_{u'})}\right)2^{\frac{j}{2}}\gamma^{\nu'}_j\\
\nn&\les& \sum_{2^l>2^j\nu-\nu'}\left(\sup_{\o'}\norm{\chi'}_{L^\infty_{u', {x.'}'}L^2_t}\norm{P_l\trc'}_{L^\infty_{u'}L^2_{{x'}'}L^\infty_t}\normm{(b_{\nu'}-b')\left(2^{\frac{j}{2}}(N'-N_{\nu'})\right)^q}_{L^\infty}\right)2^{\frac{j}{2}}\gamma^{\nu'}_j\\
\nn&\les& \left(\sum_{2^l>2^j\nu-\nu'}2^{-l}\right)\ep\gamma^{\nu'}_j\\
\nn&\les& \frac{2^{-\frac{j}{2}}\ep}{(2^j|\nu-\nu'|)^\frac{1}{2}}\gamma^\nu_j,
\eea
where we used in the last inequality the estimates \eqref{esttrc} \eqref{esthch} for $\chi'$, the estimate \eqref{lievremont1} for $P_l\trc'$, the estimate \eqref{estricciomega} for $\po b$, the estimate \eqref{estNomega} for $po N$, and the size of the patch. \eqref{but:mgen3repet}, \eqref{but:mgen5repet} and \eqref{but:mgen8repet} imply:
\be\lab{but:fix4}
\norm{h_{2,p,q,1}}_{L^1(\MM)}\les \frac{\ep^2\gamma^\nu_j\gamma^{\nu'}_j}{(2^{\frac{j}{2}}|\nu-\nu'|)^{\frac{1}{2}}}.
\ee
Next, we estimate the $L^1(\MM)$ norm of $h_{2,p,q,2}$. In view of the definition \eqref{but:fix2} of $h_{2,p,q,2}$, we have
\bea\lab{but:fix5}
\norm{h_{2,p,q,2}}_{L^1(\MM)}&\les&  \normm{\int_{\S} L(P_{> 2^j|\nu-\nu'|}\trc)\left(2^{\frac{j}{2}}(N-N_\nu)\right)^pF_{j,-1}(u)\eta_j^\nu(\o)d\o}_{L^2(\MM)}\\
\nn&&\times\normm{\int_{\S}\chi'\trc'(b_{\nu'}-b')\left(2^{\frac{j}{2}}(N'-N_{\nu'})\right)^qF_{j,-1}(u')\eta_j^{\nu'}(\o')d\o'}_{L^2(\MM)}.
\eea
We estimate both terms in the right-hand side of \eqref{but:fix5} starting with the first one. In view of \eqref{nyc23}, we have
\bea\lab{but:fix5bis}
&&\normm{\int_{\S} L(P_{> 2^j|\nu-\nu'|}\trc)\left(2^{\frac{j}{2}}(N-N_\nu)\right)^pF_{j,-1}(u)\eta_j^\nu(\o)d\o}_{L^2(\MM)}\\
\nn&\les& \left(\sum_{2^l>2^j|\nu-\nu'|}2^{\frac{j}{2}-\frac{l}{2}}\right)\ep\gamma^\nu_j\\
\nn&\les& \frac{2^{\frac{j}{4}}}{(2^{\frac{j}{2}}|\nu-\nu'|)^{\frac{1}{2}}}\ep\gamma^\nu_j.
\eea
For the second term in the right-hand side of \eqref{but:fix5}, we use the decomposition \eqref{decchiom} for $\chi'$ which yields
\be\lab{but:fix5:2}
\chi'=F^j_1+F^2_j,
\ee
where the tensor $F^j_1$ depends only on $(t,x)$ and $\nu'$ and satisfies 
\be\lab{but:fix5:3}
\norm{F^j_1}_{L^\infty_{u_{\nu'},{x'}_{\nu'}}L^2_t}\les\ep,
\ee
and where the tensor $F^2_j$ satisfies
\be\lab{but:fix5:4}
\norm{F^j_2}_{L^\infty_{u'}L^2(\H_{u'})}\les 2^{-\frac{j}{2}}\ep.
\ee
Using the decomposition \eqref{but:fix5:2}, we have
\bee
&&\int_{\S}\chi'\trc'(b_{\nu'}-b')\left(2^{\frac{j}{2}}(N'-N_{\nu'})\right)^qF_{j,-1}(u')\eta_j^{\nu'}(\o')d\o'\\
&=& F^j_1\left(\int_{\S}\trc'(b_{\nu'}-b')\left(2^{\frac{j}{2}}(N'-N_{\nu'})\right)^qF_{j,-1}(u')\eta_j^{\nu'}(\o')d\o'\right)\\
&&+\int_{\S}F^j_2\trc'(b_{\nu'}-b')\left(2^{\frac{j}{2}}(N'-N_{\nu'})\right)^qF_{j,-1}(u')\eta_j^{\nu'}(\o')d\o'.
\eee
This yields
\bea\lab{but:fix5:5}
&&\normm{\int_{\S}\chi'\trc'(b_{\nu'}-b')\left(2^{\frac{j}{2}}(N'-N_{\nu'})\right)^qF_{j,-1}(u')\eta_j^{\nu'}(\o')d\o'}_{L^2(\MM)}\\
\nn&\les& \norm{F^j_1}_{L^\infty_{u_{\nu'},{x'}_{\nu'}}L^2_t}\normm{\int_{\S}\trc'(b_{\nu'}-b')\left(2^{\frac{j}{2}}(N'-N_{\nu'})\right)^qF_{j,-1}(u')\eta_j^{\nu'}(\o')d\o'}_{L^2_{u_{\nu'},{x'}_{\nu'}}L^\infty_t}\\
\nn&&+\normm{\int_{\S}F^j_2\trc'(b_{\nu'}-b')\left(2^{\frac{j}{2}}(N'-N_{\nu'})\right)^qF_{j,-1}(u')\eta_j^{\nu'}(\o')d\o'}_{L^2(\MM)}\\
\nn&\les& 2^{-\frac{j}{4}}(1+q^{\frac{5}{2}})\ep\gamma^{\nu'}_j+\normm{\int_{\S}F^j_2\trc'(b_{\nu'}-b')\left(2^{\frac{j}{2}}(N'-N_{\nu'})\right)^qF_{j,-1}(u')\eta_j^{\nu'}(\o')d\o'}_{L^2(\MM)},
\eea
where we used in the last inequality the estimate \eqref{but:fix5:3} for $F^j_1$ and the estimate \eqref{bis:koko1}. The basic estimate \eqref{oscl2bis} yields
\bea\lab{but:fix5:6}
&&\normm{\int_{\S}F^j_2\trc'(b_{\nu'}-b')\left(2^{\frac{j}{2}}(N'-N_{\nu'})\right)^qF_{j,-1}(u')\eta_j^{\nu'}(\o')d\o'}_{L^2(\MM)}\\
\nn&\les& \left(\sup_{\o'}\normm{F^j_2\trc'(b_{\nu'}-b')\left(2^{\frac{j}{2}}(N'-N_{\nu'})\right)^q}_{L^\infty_{u'}L^2(\H_{u'})}\right)2^{\frac{j}{2}}\gamma^{\nu'}_j\\
\nn&\les& \left(\sup_{\o'}\norm{F^j_2}_{L^\infty_{u'}L^2(\H_{u'})}\normm{\trc'(b_{\nu'}-b')\left(2^{\frac{j}{2}}(N'-N_{\nu'})\right)^q}_{L^\infty}\right)2^{\frac{j}{2}}\gamma^{\nu'}_j\\
\nn&\les& 2^{-\frac{j}{2}}\ep\gamma^{\nu'}_j,
\eea
where we used in the last inequality the estimate \eqref{but:fix5:4} for $F^j_2$, the estimate \eqref{esttrc} for $\trc'$, the estimate \eqref{estricciomega} for $\po b$, the estimate \eqref{estNomega} for $\po N$, and the size of the patch. Finally, \eqref{but:fix5:5} and \eqref{but:fix5:6} yield
\be\lab{but:fix5:7}
\normm{\int_{\S}\chi'\trc'(b_{\nu'}-b')\left(2^{\frac{j}{2}}(N'-N_{\nu'})\right)^qF_{j,-1}(u')\eta_j^{\nu'}(\o')d\o'}_{L^2(\MM)}\les 2^{-\frac{j}{4}}(1+q^{\frac{5}{2}})\ep\gamma^{\nu'}_j.
\ee
\eqref{but:fix5}, \eqref{but:fix5bis} and \eqref{but:fix5:7}, together with the fact that $2^{\frac{j}{2}}|\nu-\nu'|\gtrsim 1$, imply
\be\lab{but:fix6}
\norm{h_{2,p,q,2}}_{L^1(\MM)}\les (1+q^{\frac{5}{2}})\ep^2\gamma^\nu_j\gamma^{\nu'}_j.
\ee
Next, we estimate the $L^1(\MM)$ norm of $h_{2,p,q,3}$. Recall the decomposition \eqref{nice44} \eqref{nice45} \eqref{nice46} for $L(\trc)$. We have: 
\be\lab{but:fix7}
L(\trc)={\chi_2}_\nu\c (2\chi_1+\hch)+f^j_1+f^j_2,
\ee
where the scalar $f^j_1$ only depends on $\nu$ and satisfies:
\be\lab{but:fix8}
\norm{f^j_1}_{L^\infty_{u_\nu}L^2_t L^\infty(P_{t,u_\nu})}\les \ep,
\ee
where the scalar $f^j_2$ satisfies:
\be\lab{but:fix9}
\norm{f^j_2}_{L^\infty_u\lh{2}}\les \ep 2^{-\frac{j}{2}}.
\ee
In view of the definition \eqref{but:fix3} of $h_{2,p,q,3}$, this yields the following decomposition
\be\lab{but:fix12}
h_{2,p,q,3}=h_{2,p,q,3,1}+h_{2,p,q,3,2}+h_{2,p,q,3,3},
\ee
where $h_{2,p,q,3,1}$, $h_{2,p,q,3,2}$ and $h_{2,p,q,3,3}$ are given by: 
\be\lab{but:fix13}
h_{2,p,q,3,1}= {\chi_2}_\nu\left(\int_{\S}(2\chi_1+\hch)\left(2^{\frac{j}{2}}(N-N_\nu)\right)^pF_{j,-1}(u)\eta_j^\nu(\o)d\o\right)H,
\ee
\be\lab{but:fix14}
h_{2,p,q,3,2}= f^j_1\left(\int_{\S}\left(2^{\frac{j}{2}}(N-N_\nu)\right)^pF_{j,-1}(u)\eta_j^\nu(\o)d\o\right)H,
\ee
and
\be\lab{but:fix15}
h_{2,p,q,3,3}= \left(\int_{\S}f^j_2\left(2^{\frac{j}{2}}(N-N_\nu)\right)^pF_{j,-1}(u)\eta_j^\nu(\o)d\o\right)H,
\ee
with $H$ given by:
\be\lab{but:fix20}
H=\int_{\S}\chi'\trc'(b_{\nu'}-b')\left(2^{\frac{j}{2}}(N'-N_{\nu'})\right)^qF_{j,-1}(u')\eta_j^{\nu'}(\o')d\o'.
\ee
Using the basic estimate \eqref{oscl2bis}, we have 
\bea\lab{but:fix21}
&&\normm{\int_{\S}f^j_2\left(2^{\frac{j}{2}}(N-N_\nu)\right)^pF_{j,-1}(u)\eta_j^\nu(\o)d\o}_{L^2(\MM)}\\
\nn&\les& \left(\sup_\o\normm{f^j_2\left(2^{\frac{j}{2}}(N-N_\nu)\right)^p}_{\li{\infty}{2}}\right)2^{\frac{j}{2}}\gamma^\nu_j\\
\nn&\les& \left(\sup_\o\normm{f^j_2}_{\li{\infty}{2}}\normm{\left(2^{\frac{j}{2}}(N-N_\nu)\right)^p}_{L^\infty}\right)2^{\frac{j}{2}}\gamma^\nu_j\\
\nn&\les& \ep\gamma^\nu_j,
\eea
where we used in the last inequality the estimate \eqref{but:fix9}, the estimate \eqref{estNomega} for $\po N$ and the size of the patch. Also, we have
\bea\lab{but:fix22}
&&\normm{f^j_1\left(\int_{\S}\left(2^{\frac{j}{2}}(N-N_\nu)\right)^pF_{j,-1}(u)\eta_j^\nu(\o)d\o\right)}_{L^2(\MM)}\\
\nn&\les& \norm{f^j_1}_{L^\infty_{u_\nu}L^2_t L^\infty(P_{t,u_\nu})}\normm{\int_{\S}\left(2^{\frac{j}{2}}(N-N_\nu)\right)^pF_{j,-1}(u)\eta_j^\nu(\o)d\o}_{L^2_{u_\nu, {x'}_\nu}L^2_t}\\
\nn&\les& (1+p^2)\ep\gamma^\nu_j,
\eea
where we used in the last inequality the estimate \eqref{but:fix8} for $f^j_1$, and Lemma \ref{lemma:loeb}. Also, recall \eqref{succulent}:
\be\lab{but:fix23}
\normm{\int_{\S}(2\chi_1+\hch)\left(2^{\frac{j}{2}}(N-N_\nu)\right)^pF_{j,-1}(u)\eta_j^\nu(\o)d\o}_{L^2(\MM)}\les  (1+p^2)\ep\gamma^\nu_j.
\ee
\eqref{but:fix12}-\eqref{but:fix23} imply
\be\lab{but:fix25}
\norm{h_{2,p,q,3}}_{L^1(\MM)}\les (\norm{{\chi_2}_\nu H}_{L^2(\MM)}+\norm{H}_{L^2(\MM)})(1+p^2)\gamma^\nu_j.
\ee
Next, we estimate $H$. In view of the definition \eqref{but:fix20} of $H$, the analog of \eqref{succulent} yields
\be\lab{but:fix26}
\norm{H}_{L^2(\MM)}\les (1+q^2)\ep\gamma^{\nu'}_j.
\ee
Also, in view of the definition \eqref{but:fix20} of $H$, we have
$${\chi_2}_\nu H=H_2+H_3,$$
where $H_2$ and $H_3$ have been defined respectively in \eqref{fix31} and \eqref{fix32}. We deduce 
\bea\lab{but:fix27}
\norm{{\chi_2}_\nu H}_{L^2(\MM)}&\les& \norm{H_2}_{L^2(\MM)}+\norm{H_3}_{L^2(\MM)}\\
\nn&\les& \ep\gamma^{\nu'}_j,
\eea
where we used in the last inequality the estimate \eqref{fix35} for $H_2$ and the estimate \eqref{fix36} for $H_3$. 
\eqref{but:fix25}, \eqref{but:fix26} and \eqref{but:fix27} imply
$$\norm{h_{2,p,q,3}}_{L^1(\MM)}\les (1+p^2)(1+q^2)\ep^2\gamma^\nu_j\gamma^{\nu'}_j,$$
which together with \eqref{but:fix}, \eqref{but:fix4}, \eqref{but:fix6}, and the fact that $2^{\frac{j}{2}}|\nu-\nu'|\gtrsim 1$, implies in the case $r=0$ and $s=1$
\bea\lab{but:fix37}
\norm{h_{2,p,q}}_{L^1(\MM)}&\les& \norm{h_{2,p,q,1}}_{L^1(\MM)}+\norm{h_{2,p,q,2}}_{L^1(\MM)}+\norm{h_{2,p,q,3}}_{L^1(\MM)}\\
\nn&\les& (1+p^2)(1+q^{\frac{5}{2}})\ep^2\gamma^\nu_j\gamma^{\nu'}_j.
\eea
Using \eqref{mgen22} in the case $r=1$ and $s=0$, and \eqref{but:fix37} in the case $r=0$ and $s=1$, together with the fact that $2^{\frac{j}{2}}|\nu-\nu'|\gtrsim 1$, we finally obtain
\be\lab{mgen48}
\norm{h_{2,p,q}}_{L^1(\MM)}\les (1+p^2)(1+q^{\frac{5}{2}})2^{\frac{j}{2}}|\nu-\nu'|\ep^2\gamma^\nu_j\gamma^{\nu'}_j.
\ee
Now, in view of the decomposition \eqref{mgen1}, we have:
\bee
&& \left|\sum_{(l,m)/2^{\max(l,m)}l\leq 2^j|\nu-\nu'|}(B^{1,2,2,3}_{j,\nu',\nu,l,m}+B^{1,2,2,3}_{j,\nu,\nu',l,m})\right|\\
\nn &\les& \sum_{p, q\geq 0}c_{pq}\normm{\frac{1}{(2^{\frac{j}{2}}|N_\nu-N_{\nu'}|)^{p+q+4}}}_{L^\infty(\MM)}[\norm{h_{1,p,q}}_{L^1(\MM)}+\norm{h_{2,p,q}}_{L^1(\MM)}],
\eee
which together with \eqref{nice26}, \eqref{mgen9} and \eqref{mgen48} implies:
\bea\lab{mgen49}
&& \left|\sum_{(l,m)/2^{\max(l,m)}\leq 2^j|\nu-\nu'|}(B^{1,2,2,3}_{j,\nu',\nu,l,m}+B^{1,2,2,3}_{j,\nu,\nu',l,m})\right|\\
\nn &\les& \sum_{p, q\geq 0}c_{pq}\frac{1}{(2^{\frac{j}{2}}|\nu-\nu'|)^{p+q+4}}(1+p^2)(1+q^{\frac{5}{2}})2^{\frac{j}{2}}|\nu-\nu'|\ep^2\gamma^\nu_j\gamma^{\nu'}_j\\
\nn&\les& \frac{\ep^2\gamma^\nu_j\gamma^{\nu'}_j}{(2^{\frac{j}{2}}|\nu-\nu'|)^3}.
\eea 

In view of \eqref{vino1}, we have in the range $2^{\max(l,m)}\leq 2^j|\nu-\nu'|$:
\bee
&&\left|B^{1,2,2}_{j,\nu,\nu',l,m}-(B^{1,2,2,1}_{j,\nu,\nu',l,m}+B^{1,2,2,2}_{j,\nu,\nu',l,m}+B^{1,2,2,3}_{j,\nu,\nu',l,m})\right|\\
\nn&\les& 2^{-2j}\sum_{p, q\geq 0}c_{pq}\normm{\frac{1}{(2^{\frac{j}{2}}|N_\nu-N_{\nu'}|)^{p+q}}}_{L^\infty(\MM)}\\
\nn&&\times\bigg[\normm{\frac{1}{|N_\nu-N_{\nu'}|^2}}_{L^\infty(\MM)}\norm{h_{1,p,q,l,m}}_{L^1(\MM)}+\normm{\frac{1}{|N_\nu-N_{\nu'}|^3}}_{L^\infty(\MM)}\norm{h_{2,p,q,l,m}}_{L^1(\MM)}\bigg].
\eee
Together with \eqref{nice26}, \eqref{vino7} and \eqref{vino25}, we obtain:
\bee
&&\left|B^{1,2,2}_{j,\nu,\nu',l,m}-(B^{1,2,2,1}_{j,\nu,\nu',l,m}+B^{1,2,2,2}_{j,\nu,\nu',l,m}+B^{1,2,2,3}_{j,\nu,\nu',l,m})\right|\\
\nn&\les& 2^{-2j}\sum_{p, q\geq 0}c_{pq}\frac{1}{(2^{\frac{j}{2}}|\nu-\nu'|)^{p+q}}\bigg[\frac{1}{|\nu-\nu'|^2}((1+p^2)2^{\frac{j}{11}}|\nu-\nu'|+2^{\frac{j}{2}})\ep^2\gamma^\nu_j\gamma^{\nu'}_j\\
\nn&&+\frac{1}{|\nu-\nu'|^3}(1+q^2)(2^{\frac{j}{4}}(2^{\frac{j}{2}}|\nu-\nu'|)+2^{\frac{7j}{16}})\ep^2\gamma^\nu_j\gamma^{\nu'}_j\bigg]\\
\nn&\les& \bigg[\frac{2^{-\frac{10j}{11}}}{2^{\frac{j}{2}}(2^{\frac{j}{2}}|\nu-\nu'|)}+\frac{2^{-\frac{j}{4}}}{(2^{\frac{j}{2}}|\nu-\nu'|)^2}+\frac{2^{-\frac{j}{16}}}{(2^{\frac{j}{2}}|\nu-\nu'|)^3}\bigg]\ep^2\gamma^\nu_j\gamma^{\nu'}_j.
\eee
Together with \eqref{vino27}, \eqref{vino86} and \eqref{mgen49}, we finally obtain:
\bee
&&\left|\sum_{(l,m)/2^{\max(l,m)}\leq 2^j|nu-\nu'|}(B^{1,2,2}_{j,\nu,\nu',l,m}+B^{1,2,2}_{j,\nu',\nu,l,m})\right|\\
\nn&\les& \left|\sum_{(l,m)/2^{\max(l,m)}}(B^{1,2,2,1}_{j,\nu,\nu',l,m}+B^{1,2,2,1}_{j,\nu',\nu,l,m})\right|+\left|\sum_{(l,m)/2^{\max(l,m)}}(B^{1,2,2,2}_{j,\nu,\nu',l,m}+B^{1,2,2,2}_{j,\nu',\nu,l,m})\right|\\
\nn&&+\left|\sum_{(l,m)/2^{\max(l,m)}}(B^{1,2,2,3}_{j,\nu,\nu',l,m}+B^{1,2,2,3}_{j,\nu',\nu,l,m})\right|+ \bigg[\frac{j2^{-\frac{j}{4}}}{(2^{\frac{j}{2}}|\nu-\nu'|)^2}+\frac{j2^{-\frac{j}{16}}}{(2^{\frac{j}{2}}|\nu-\nu'|)^3}\bigg]\ep^2\gamma^\nu_j\gamma^{\nu'}_j\\
\nn&\les& \left[2^{-j}+\frac{1}{2^{\frac{j}{2}}(2^{\frac{j}{2}}|\nu-\nu'|)}+\frac{1}{2^{\frac{j}{4}}(2^{\frac{j}{2}}|\nu-\nu'|)^{\frac{3}{2}}}+\frac{2^{-(\frac{1}{4})_-j}}{(2^{\frac{j}{2}}|\nu-\nu'|)^2}+\frac{1}{(2^{\frac{j}{2}}|\nu-\nu'|)^3}\right]\ep^2\gamma^\nu_j\gamma^{\nu'}_j.
\eee
Together with \eqref{stosur36}, we obtain the estimate on the whole range $2^{\min(l,m)}\leq 2^j|\nu-\nu'|$:
\bee
&&\left|\sum_{(l,m)/2^{\min(l,m)}\leq 2^j|\nu-\nu'|}(B^{1,2,2}_{j,\nu,\nu',l,m}+B^{1,2,2}_{j,\nu',\nu,l,m})\right|\\
\nn&\les& \left[2^{-j}+\frac{1}{2^{\frac{j}{2}}(2^{\frac{j}{2}}|\nu-\nu'|)}+\frac{1}{2^{\frac{j}{4}}(2^{\frac{j}{2}}|\nu-\nu'|)^{\frac{3}{2}}}+\frac{2^{-(\frac{1}{4})_-j}}{(2^{\frac{j}{2}}|\nu-\nu'|)^2}+\frac{1}{(2^{\frac{j}{2}}|\nu-\nu'|)^3}\right]\ep^2\gamma^\nu_j\gamma^{\nu'}_j.
\eee
This concludes the proof of Proposition \ref{prop:labexfsmp8}.

\subsubsection{Proof of Proposition \ref{prop:labexfsmp9} (Control of $B^{1,2,3}_{j,\nu,\nu',l,m}$)}\lab{sec:lobotomisation4}

Recall that $B^{1,2,3}_{j,\nu,\nu',l,m}$ is defined by \eqref{nyc3bis}:
\bee
B^{1,2,3}_{j,\nu,\nu',l,m} &=&  -i2^{-j-1}\int_{\MM}\int_{\S\times\S} \frac{1}{\gg(L,L')}\bigg(L(P_l\trc)P_m\trc'+P_l\trc L'(P_m\trc')\bigg)\\
\nn&&\times {b'}^{-1}(1-\gn) F_{j,-1}(u)F_j(u')\eta_j^\nu(\o)\eta_j^{\nu'}(\o')d\o d\o' d\MM.
\eee
Together with the identity \eqref{nice24}:
$$\gg(L,L')=-1+\gn,$$
this yields:
\bea\lab{bizu}
B^{1,2,3}_{j,\nu,\nu',l,m} &=&  i2^{-j-1}\int_{\MM}\int_{\S\times\S}  {b'}^{-1}\bigg(L(P_l\trc)P_m\trc'+P_l\trc L'(P_m\trc')\bigg)\\
\nn&&\times F_{j,-1}(u)F_j(u')\eta_j^\nu(\o)\eta_j^{\nu'}(\o')d\o d\o' d\MM.
\eea
Recall \eqref{uso1}:
$$m<l\textrm{ and }2^m\leq 2^j|\nu-\nu'|.$$
We first consider the range of $(l,m)$ such that:
$$2^m\leq 2^j|\nu-\nu'|<2^l.$$
This yields:
\be\lab{bizu1}
\sum_{(l,m)/2^m\leq 2^j|\nu-\nu'|<2^l}B^{1,2,3}_{j,\nu,\nu',l,m} =  -i2^{-j-1}\int_{\MM}(h_1+h_2)d\MM,
\ee
where the scalar functions $h_1, h_2$ on $\MM$ are given by:
\be\lab{bizu2}
h_1= \left(\int_{\S}L(P_{> 2^j|\nu-\nu'|}\trc)F_{j,-1}(u)\eta_j^\nu(\o)d\o\right)\left(\int_{\S}{b'}^{-1}P_{\leq 2^j|\nu-\nu'|}\trc'F_j(u')\eta_j^{\nu'}(\o')d\o'\right),
\ee
and:
\be\lab{bizu3}
h_2=\left(\int_{\S}P_{> 2^j|\nu-\nu'|}\trc F_{j,-1}(u)\eta_j^\nu(\o)d\o\right)\left(\int_{\S}{b'}^{-1}L'(P_{\leq 2^j|\nu-\nu'|}\trc')F_j(u')\eta_j^{\nu'}(\o')d\o'\right).
\ee

Next, we estimate the $L^1(\MM)$ norm of $h_1$ and $h_2$ starting with $h_1$. In view of the definition of $h_1$ \eqref{bizu2}, we have:
\bea\lab{bizu4}
\norm{h_1}_{L^1(\MM)}&\les& \normm{\int_{\S}L(P_{> 2^j|\nu-\nu'|}\trc)F_{j,-1}(u)\eta_j^\nu(\o)d\o}_{L^2(\MM)}\\
\nn&&\times\normm{\int_{\S}{b'}^{-1}P_{\leq 2^j|\nu-\nu'|}\trc'F_j(u')\eta_j^{\nu'}(\o')d\o'}_{L^2(\MM)}\\
\nn&\les& \frac{2^{\frac{j}{4}}}{(2^{\frac{j}{2}}|\nu-\nu'|)^{\frac{1}{2}}}\ep^2\gamma^\nu_j\gamma^{\nu'}_j,
\eea
where we used in the last inequality the estimates \eqref{stosur19} and \eqref{mgen7}.

Next, we estimate the $L^1(\MM)$ norm of $h_2$. In view of the definition of $h_2$ \eqref{bizu3}, we have:
\bea\lab{bizu5}
\norm{h_2}_{L^1(\MM)}&\les& \normm{\int_{\S}P_{> 2^j|\nu-\nu'|}\trc F_{j,-1}(u)\eta_j^\nu(\o)d\o}_{L^2(\MM)}\\
\nn&&\times\normm{\int_{\S}{b'}^{-1}L'(P_{\leq 2^j|\nu-\nu'|}\trc')F_j(u')\eta_j^{\nu'}(\o')d\o'}_{L^2(\MM)}\\
\nn&\les& \frac{2^{\frac{j}{2}}}{(2^{\frac{j}{2}}|\nu-\nu'|)}\ep^2\gamma^\nu_j\gamma^{\nu'}_j,
\eea
where we used in the last inequality the estimates \eqref{nycc133} and \eqref{but:mgen5repet}.

Finally, \eqref{bizu1}, \eqref{bizu4} and \eqref{bizu5} imply the following estimate in the range of $(l,m)$ such that $2^m\leq 2^j|\nu-\nu'|<2^l$:
\bea\lab{bizu6}
\left|\sum_{(l,m)/2^m\leq 2^j|\nu-\nu'|<2^l}B^{1,2,3}_{j,\nu,\nu',l,m}\right| &\les&  2^{-j}(\norm{h_1}_{L^1(\MM)}+\norm{h_2}_{L^1(\MM)})\\
\nn&\les& \left[\frac{1}{2^{\frac{3j}{4}}(2^{\frac{j}{2}}|\nu-\nu'|)^{\frac{1}{2}}}+\frac{1}{2^{\frac{j}{2}}(2^{\frac{j}{2}}|\nu-\nu'|)}\right]\ep^2\gamma^\nu_j\gamma^{\nu'}_j.
\eea

Next, we estimate $B^{1,2,3}_{j,\nu,\nu',l,m}$ in the range of $(l,m)$ such that:
$$2^m\leq 2^l\leq 2^j|\nu-\nu'|.$$
We have the following decomposition for $B^{1,2,3}_{j,\nu,\nu',l,m}$:
\be\lab{bizu7}
B^{1,2,3}_{j,\nu,\nu',l,m}=B^{1,2,3,1}_{j,\nu,\nu',l,m}+B^{1,2,3,2}_{j,\nu,\nu',l,m},
\ee
where $B^{1,2,3,1}_{j,\nu,\nu',l,m}$ and $B^{1,2,3,2}_{j,\nu,\nu',l,m}$ are given by:
\be\lab{bizu8}
B^{1,2,3,1}_{j,\nu,\nu',l,m} =  i2^{-j-1}\int_{\MM}\int_{\S\times\S}  {b'}^{-1}L(P_l\trc)P_m\trc' F_{j,-1}(u)F_j(u')\eta_j^\nu(\o)\eta_j^{\nu'}(\o')d\o d\o' d\MM,
\ee
and
\be\lab{bizu9}
B^{1,2,3,2}_{j,\nu,\nu',l,m} =  i2^{-j-1}\int_{\MM}\int_{\S\times\S}  {b'}^{-1} P_l\trc L'(P_m\trc') F_{j,-1}(u)F_j(u')\eta_j^\nu(\o)\eta_j^{\nu'}(\o')d\o d\o' d\MM.
\ee
We first estimate $B^{1,2,3,1}_{j,\nu,\nu',l,m}$. We integrate by parts tangentially using \eqref{fete}. 
\begin{lemma}\lab{lemma:bizu}
Let $B^{1,2,3,1}_{j,\nu,\nu',l,m}$ defined by \eqref{bizu8}. Integrating by parts using \eqref{fete} yields:
\bea\lab{bizu10}
&& B^{1,2,3,1}_{j,\nu,\nu',l,m}\\
\nn&=& 2^{-2j}\sum_{p, q\geq 0}c_{pq}\int_{\MM}\frac{1}{(2^{\frac{j}{2}}|N_\nu-N_{\nu'}|)^{p+q}}\bigg[\frac{1}{|N_\nu-N_{\nu'}|^2}(h_{1,p,q,l,m}+h_{2,p,q,l,m})\\
\nn&&+\frac{1}{|N_\nu-N_{\nu'}|}(h_{3,p,q,l,m}+h_{4,p,q,l,m})+h_{5,p,q,l,m}\bigg] d\MM,
\eea
where $c_{pq}$ are explicit real coefficients such that the series 
$$\sum_{p, q\geq 0}c_{pq}x^py^q$$
has radius of convergence 1, where the scalar functions $h_{1,p,q,l,m}$, $h_{2,p,q,l,m}$, $h_{3,p,q,l,m}$, $h_{4,p,q,l,m}$, $h_{5,p,q,l,m}$ on $\MM$ are given by:
\bea\lab{bizu11}
h_{1,p,q,l,m}&=& \left(\int_{\S} \chi L(P_l\trc)\left(2^{\frac{j}{2}}(N-N_\nu)\right)^pF_{j,-1}(u)\eta_j^\nu(\o)d\o\right)\\
\nn&&\times\left(\int_{\S}P_m\trc'\left(2^{\frac{j}{2}}(N'-N_{\nu'})\right)^qF_{j,-1}(u')\eta_j^{\nu'}(\o')d\o'\right),
\eea
\bea\lab{bizu12}
h_{2,p,q,l,m}&=& \left(\int_{\S} L(P_l\trc)\left(2^{\frac{j}{2}}(N-N_\nu)\right)^pF_{j,-1}(u)\eta_j^\nu(\o)d\o\right)\\
\nn&&\times\left(\int_{\S}\chi' P_m\trc' \left(2^{\frac{j}{2}}(N'-N_{\nu'})\right)^qF_{j,-1}(u')\eta_j^{\nu'}(\o')d\o'\right),
\eea
\bea\lab{bizu13}
h_{3,p,q,l,m}&=& \left(\int_{\S}G_1\left(2^{\frac{j}{2}}(N-N_\nu)\right)^pF_{j,-1}(u)\eta_j^\nu(\o)d\o\right)\\
\nn&&\times\left(\int_{\S}P_m\trc'\left(2^{\frac{j}{2}}(N'-N_{\nu'})\right)^qF_{j,-1}(u')\eta_j^{\nu'}(\o')d\o'\right),
\eea
\bea\lab{bizu14}
h_{4,p,q,l,m}&=& \left(\int_{\S} L(P_l\trc)\left(2^{\frac{j}{2}}(N-N_\nu)\right)^pF_{j,-1}(u)\eta_j^\nu(\o)d\o\right)\\
\nn&&\times\left(\int_{\S}G_2  \left(2^{\frac{j}{2}}(N'-N_{\nu'})\right)^qF_{j,-1}(u')\eta_j^{\nu'}(\o')d\o'\right),
\eea
and:
\bea\lab{bizu15}
h_{5,p,q,l,m}&=& \left(\int_{\S} L(P_l\trc)\left(2^{\frac{j}{2}}(N-N_\nu)\right)^pF_{j,-1}(u)\eta_j^\nu(\o)d\o\right)\\
\nn&&\times\left(\int_{\S}N'(P_m\trc')\left(2^{\frac{j}{2}}(N'-N_{\nu'})\right)^qF_{j,-1}(u')\eta_j^{\nu'}(\o')d\o'\right),
\eea
and where the tensors $G_1$ and $G_2$ on $\MM$ are given by:
\be\lab{bizu16}
G_1=\nabb(L(P_l\trc))+(\th+b^{-1}\nabb(b))L(P_l\trc),
\ee
and:
\be\lab{bizu17}
G_2=\nabb'(P_m\trc)+(\th'+{b'}^{-1}\nabb(b'))P_m\trc'.
\ee
\end{lemma}
The proof of lemma \ref{lemma:bizu} is postponed to Appendix F. In the rest of this section, we use Lemma \ref{lemma:bizu} to obtain the control of $B^{1,2,3,1}_{j,\nu,\nu',l,m}$.

We estimate the $L^1(\MM)$ norm of $h_{1,p,q,l,m}, h_{2,p,q,l,m}, h_{3,p,q,l,m}, h_{4,p,q,l,m}, h_{5,p,q,l,m}$ starting  with $h_{1,p,q,l,m}$. In view of the definition \eqref{bizu11} of $h_{1,p,q,l,m}$, we have:
\bea\lab{bizu18}
\norm{h_{1,p,q,l,m}}_{L^1(\MM)}&\les & \normm{\int_{\S} \chi L(P_l\trc)\left(2^{\frac{j}{2}}(N-N_\nu)\right)^pF_{j,-1}(u)\eta_j^\nu(\o)d\o}_{L^2(\MM)}\\
\nn&&\times\normm{\int_{\S}P_m\trc'\left(2^{\frac{j}{2}}(N'-N_{\nu'})\right)^qF_{j,-1}(u')\eta_j^{\nu'}(\o')d\o'}_{L^2(\MM)}\\
\nn&\les & (1+q^2)\ep\gamma^{\nu'}_j\normm{\int_{\S} \chi L(P_l\trc)\left(2^{\frac{j}{2}}(N-N_\nu)\right)^pF_{j,-1}(u)\eta_j^\nu(\o)d\o}_{L^2(\MM)},
\eea
where we used in the last inequality the estimate \eqref{nyc42}. Now, the analog of the estimate \eqref{mgen5} with $r=0$ yields:
\be\lab{bizu19}
\normm{\int_{\S} \chi L(P_l\trc)\left(2^{\frac{j}{2}}(N-N_\nu)\right)^pF_{j,-1}(u)\eta_j^\nu(\o)d\o}_{L^2(\MM)}\les \ep 2^{\frac{j}{2}}\gamma^\nu_j,
\ee
which together with \eqref{bizu18} implies:
\be\lab{bizu20}
\norm{h_{1,p,q,l,m}}_{L^1(\MM)}\les  (1+q^2)2^{\frac{j}{2}}\ep^2\gamma^\nu_j\gamma^{\nu'}_j.
\ee

Next, we estimate the $L^1(\MM)$ norm of $h_{2,p,q,l,m}$. In view of the definition \eqref{bizu12} of $h_{2,p,q,l,m}$, we have:
\bea\lab{bizu21}
\norm{h_{2,p,q,l,m}}_{L^1(\MM)}&\les& \normm{\int_{\S} L(P_l\trc)\left(2^{\frac{j}{2}}(N-N_\nu)\right)^pF_{j,-1}(u)\eta_j^\nu(\o)d\o}_{L^2(\MM)}\\
\nn&&\times\normm{\int_{\S}\chi' P_m\trc' \left(2^{\frac{j}{2}}(N'-N_{\nu'})\right)^qF_{j,-1}(u')\eta_j^{\nu'}(\o')d\o'}_{L^2(\MM)}.
\eea
The basic estimate in $L^2(\MM)$ \eqref{oscl2bis} yields:
\bea\lab{bizu22}
&& \normm{\int_{\S} L(P_l\trc)\left(2^{\frac{j}{2}}(N-N_\nu)\right)^pF_{j,-1}(u)\eta_j^\nu(\o)d\o}_{L^2(\MM)}\\
\nn&\les& \left(\sup_\o\normm{L(P_l\trc)\left(2^{\frac{j}{2}}(N-N_\nu)\right)^p}_{\li{\infty}{2}}\right)2^{\frac{j}{2}}\gamma^\nu_j\\
\nn&\les& \ep  2^{\frac{j}{2}}\gamma^\nu_j,
\eea
where we used in the last inequality the analog of the estimate \eqref{mgen11bis}, the estimate \eqref{estNomega} for $\po N$, and the size of the patch. Also, the analog of the estimate \eqref{mgen21} yields:
\be\lab{bizu23}
\normm{\int_{\S}\chi' P_m\trc' \left(2^{\frac{j}{2}}(N'-N_{\nu'})\right)^qF_{j,-1}(u')\eta_j^{\nu'}(\o')d\o'}_{L^2(\MM)}\les \ep(1+q^{\frac{5}{2}})\gamma^{\nu'}_j.
\ee
Finally, \eqref{bizu21}, \eqref{bizu22} and \eqref{bizu23} imply:
\be\lab{bizu24}
\norm{h_{2,p,q,l,m}}_{L^1(\MM)}\les (1+q^{\frac{5}{2}})2^{\frac{j}{2}}\ep^2\gamma^\nu_j\gamma^{\nu'}_j.
\ee

Next, we estimate the $L^1(\MM)$ norm of $h_{3,p,q,l,m}$. In view of the definition \eqref{bizu13} of $h_{3,p,q,l,m}$, we have:
\bee
\sum_{m\leq l}h_{3,p,q,l,m}&=& \left(\int_{\S}G_1\left(2^{\frac{j}{2}}(N-N_\nu)\right)^pF_{j,-1}(u)\eta_j^\nu(\o)d\o\right)\\
\nn&&\times\left(\int_{\S}P_{\leq l}\trc'\left(2^{\frac{j}{2}}(N'-N_{\nu'})\right)^qF_{j,-1}(u')\eta_j^{\nu'}(\o')d\o'\right),
\eee
which yields:
\bea\lab{bizu25}
\normm{\sum_{m\leq l}h_{3,p,q,l,m}}_{L^1(\MM)}&\les& \normm{\int_{\S}G_1\left(2^{\frac{j}{2}}(N-N_\nu)\right)^pF_{j,-1}(u)\eta_j^\nu(\o)d\o}_{L^2(\MM)}\\
\nn&&\times\normm{\int_{\S}P_{\leq l}\trc'\left(2^{\frac{j}{2}}(N'-N_{\nu'})\right)^qF_{j,-1}(u')\eta_j^{\nu'}(\o')d\o'}_{L^2(\MM)}.
\eea
In view of the estimates \eqref{nyc40} and \eqref{nyc41}, we have:
\be\lab{bizu26}
\normm{\int_{\S}P_{\leq l}\trc'\left(2^{\frac{j}{2}}(N'-N_{\nu'})\right)^qF_{j,-1}(u')\eta_j^{\nu'}(\o')d\o'}_{L^2(\MM)}\les (1+q^2)\ep\gamma^{\nu'}_j.
\ee
Also, using the basic estimate in $L^2(\MM)$ \eqref{oscl2bis}, we have:
\bea\lab{bizu27}
&&\normm{\int_{\S}G_1\left(2^{\frac{j}{2}}(N-N_\nu)\right)^pF_{j,-1}(u)\eta_j^\nu(\o)d\o}_{L^2(\MM)}\\
\nn&\les& \left(\sup_\o\normm{G_1\left(2^{\frac{j}{2}}(N-N_\nu)\right)^p}_{\li{\infty}{2}}\right)2^{\frac{j}{2}}\gamma^\nu_j\\
\nn&\les& \left(\sup_\o\norm{G_1}_{\li{\infty}{2}}\right)2^{\frac{j}{2}}\gamma^\nu_j,
\eea
where we used in the last inequality the estimate \eqref{estNomega} for $\po N$ and the size of the patch. In view of the definition \eqref{bizu16} of $G_1$, we have: 
\bea\lab{bizu28}
\nn\norm{G_1}_{\li{\infty}{2}}&\les&\norm{\nabb(L(P_l\trc))}_{\li{\infty}{2}}+\norm{(\th+b^{-1}\nabb(b))L(P_l\trc)}_{\li{\infty}{2}}\\
&\les& 2^{\frac{l}{2}}\ep,
\eea
where we used the estimate \eqref{youkoulele} for the first term and the estimate \eqref{nyc46} for the second term. Now, \eqref{bizu27} and \eqref{bizu28} imply:
$$\normm{\int_{\S}G_1\left(2^{\frac{j}{2}}(N-N_\nu)\right)^pF_{j,-1}(u)\eta_j^\nu(\o)d\o}_{L^2(\MM)}\les \ep 2^{\frac{j}{2}+\frac{l}{2}}\gamma^\nu_j,$$
which together with \eqref{bizu25} and \eqref{bizu26} yields:
\be\lab{bizu29}
\normm{\sum_{m\leq l}h_{3,p,q,l,m}}_{L^1(\MM)}\les (1+q^2)2^{\frac{j}{2}+\frac{l}{2}}\ep^2\gamma^\nu_j\gamma^{\nu'}_j.
\ee

Next, we estimate the $L^1(\MM)$ norm of $h_{4,p,q,l,m}$. In view of the definition \eqref{bizu14} of $h_{4,p,q,l,m}$, we have:
\bea\lab{bizu31}
\normm{h_{4,p,q,l,m}}_{L^1(\MM)}&\les& \normm{\int_{\S}L(P_l\trc)\left(2^{\frac{j}{2}}(N-N_\nu)\right)^pF_{j,-1}(u)\eta_j^\nu(\o)d\o}_{L^2(\MM)}\\
\nn&&\times\normm{\int_{\S}G_2\left(2^{\frac{j}{2}}(N'-N_{\nu'})\right)^qF_{j,-1}(u')\eta_j^{\nu'}(\o')d\o'}_{L^2(\MM)}\\
\nn&\les& (1+p^2)2^{\frac{5j}{12}}\ep\gamma^\nu_j\normm{\int_{\S}G_2\left(2^{\frac{j}{2}}(N'-N_{\nu'})\right)^qF_{j,-1}(u')\eta_j^{\nu'}(\o')d\o'}_{L^2(\MM)},
\eea
where we used in the last inequality the estimate \eqref{nyc32}. Now, the basic estimate in $L^2(\MM)$ \eqref{oscl2bis} yields:
\bea\lab{bizu32}
&&\normm{\int_{\S}G_2\left(2^{\frac{j}{2}}(N'-N_{\nu'})\right)^qF_{j,-1}(u')\eta_j^{\nu'}(\o')d\o'}_{L^2(\MM)}\\
\nn&\les& \left(\sup_{\o'}\normm{G_2\left(2^{\frac{j}{2}}(N'-N_{\nu'})\right)^q}_{\lprime{\infty}{2}}\right)2^{\frac{j}{2}}\gamma^{\nu'}_j\\
\nn&\les& \left(\sup_{\o'}\norm{G_2}_{\lprime{\infty}{2}}\right)2^{\frac{j}{2}}\gamma^{\nu'}_j,
\eea
where we used in the last inequality the estimate \eqref{estNomega} for $\po N$ and the size of the patch. Now, in view of the definition \eqref{bizu17} for $G_2$, we have:
\bee
\norm{G_2}_{\lprime{\infty}{2}}&\les & \norm{\nabb'(P_{\leq l}\trc)}_{\lprime{\infty}{2}}+\norm{\th'+{b'}^{-1}\nabb(b')}_{\lprime{\infty}{2}}\norm{P_{\leq l}\trc'}_{L^\infty(\MM)}\\
&\les &\ep,
\eee
where we used in the last inequality the finite band property and the boundedness on $L^\infty(P_{t,u'})$ for $P_{\leq l}$, the estimate \eqref{esttrc} for $\trc'$, the estimate \eqref{estb} for $b'$, and the estimates \eqref{estk} \eqref{esttrc} \eqref{esthch} for $\th'$. Together with \eqref{bizu32}, this yields:
$$\normm{\int_{\S}G_2\left(2^{\frac{j}{2}}(N'-N_{\nu'})\right)^qF_{j,-1}(u')\eta_j^{\nu'}(\o')d\o'}_{L^2(\MM)}\les \ep 2^{\frac{j}{2}}\gamma^{\nu'}_j,$$
which in view of \eqref{bizu31} implies:
\be\lab{bizu33}
\normm{h_{4,p,q,l,m}}_{L^1(\MM)}\les  (1+p^2)2^{\frac{11j}{12}}\ep^2\gamma^\nu_j\gamma^{\nu'}_j.
\ee

Next, we estimate the $L^1(\MM)$ norm of $h_{5,p,q,l,m}$. In view of the definition \eqref{bizu15} of $h_{5,p,q,l,m}$, we have:
\bea\lab{bizu34}
\normm{h_{5,p,q,l,m}}_{L^1(\MM)}&\les& \normm{\int_{\S}L(P_l\trc)\left(2^{\frac{j}{2}}(N-N_\nu)\right)^pF_{j,-1}(u)\eta_j^\nu(\o)d\o}_{L^2(\MM)}\\
\nn&&\times\normm{\int_{\S}N'(P_m\trc')\left(2^{\frac{j}{2}}(N'-N_{\nu'})\right)^qF_{j,-1}(u')\eta_j^{\nu'}(\o')d\o'}_{L^2(\MM)}\\
\nn&\les& (1+p^2)2^{\frac{11j}{12}}\ep^2\gamma^\nu_j\gamma^{\nu'}_j,
\eea
where we used in the last inequality the estimates \eqref{nyc32} and \eqref{vino56}.

Finally, we have in view of \eqref{bizu10}:
\bee
\left|\sum_{m/ m\leq l}B^{1,2,3,1}_{j,\nu,\nu',l,m}\right|& \les & 2^{-2j}\sum_{p, q\geq 0}c_{pq}\normm{\frac{1}{(2^{\frac{j}{2}}|N_\nu-N_{\nu'}|)^{p+q}}}_{L^\infty(\MM)}\\
\nn&&\times\bigg[\normm{\frac{1}{|N_\nu-N_{\nu'}|^2}}_{L^\infty(\MM)}\sum_{m/ m\leq l}(\norm{h_{1,p,q,l,m}}_{L^1(\MM)}+\norm{h_{2,p,q,l,m}}_{L^1(\MM)})\\
\nn&& +\normm{\frac{1}{|N_\nu-N_{\nu'}|}}_{L^1(\MM)}\left(\normm{\sum_{m/ m\leq l}h_{3,p,q,l,m}}_{L^1(\MM)}+\sum_{m/ m\leq l}\norm{h_{4,p,q,l,m}}_{L^1(\MM)}\right)\\
\nn&&+\sum_{m/ m\leq l}\norm{h_{5,p,q,l,m}}_{L^1(\MM)}\bigg] d\MM.
\eee
Together with the estimates \eqref{nice26}, \eqref{bizu20}, \eqref{bizu24}, \eqref{bizu29}, \eqref{bizu33} and \eqref{bizu34}, and the fact that we are in the range $2^m\leq 2^l\leq 2^j|\nu-\nu'|$, we obtain:
\bee
\left|\sum_{m/ m\leq l}B^{1,2,3,1}_{j,\nu,\nu',l,m}\right|& \les & 2^{-2j}\sum_{p, q\geq 0}c_{pq}\frac{1}{(2^{\frac{j}{2}}|\nu-\nu'|)^{p+q}}\\
\nn&&\times\bigg[\frac{1}{|\nu-\nu'|^2}(1+q^2)j 2^{\frac{j}{2}}+\frac{1}{|\nu-\nu'|}((1+q^2)2^{\frac{j}{2}+\frac{l}{2}}+(1+p^2) j 2^{\frac{11j}{12}})\\
\nn&&+(1+p^2) j 2^{\frac{11j}{12}}\bigg]\ep^2\gamma^\nu_j\gamma^{\nu'}_j\\
\nn& \les & \bigg[\frac{j 2^{-\frac{j}{2}}}{(2^{\frac{j}{2}}|\nu-\nu'|)^2}+\frac{2^{-\frac{j}{2}+\frac{l}{2}}+ j 2^{-\frac{j}{12}}}{2^{\frac{j}{2}}(2^{\frac{j}{2}}|\nu-\nu'|)}+ j 2^{-\frac{j}{12}}2^{-j}\bigg]\ep^2\gamma^\nu_j\gamma^{\nu'}_j.
\eee
Summing in $l$, we finally obtain:
\be\lab{bizu35}
\left|\sum_{(l,m)/ 2^{\max(l,m)}\leq 2^j|\nu-\nu'|}B^{1,2,3,1}_{j,\nu,\nu',l,m}\right| \les  \bigg[\frac{j^2 2^{-\frac{j}{2}}}{(2^{\frac{j}{2}}|\nu-\nu'|)^2}+\frac{1}{2^{\frac{3j}{4}}(2^{\frac{j}{2}}|\nu-\nu'|)^{\frac{1}{2}}}+ 2^{-j}\bigg]\ep^2\gamma^\nu_j\gamma^{\nu'}_j.
\ee 

Next, we estimate $B^{1,2,3,2}_{j,\nu,\nu',l,m}$ in the range $2^m\leq 2^l\leq 2^j|\nu-\nu'|$. We obtain the analog of the estimate \eqref{bizu35}:
\be\lab{bizu36}
\left|\sum_{(l,m)/ 2^{\max(l,m)}\leq 2^j|\nu-\nu'|}B^{1,2,3,2}_{j,\nu,\nu',l,m}\right| \les  \bigg[\frac{j^2 2^{-\frac{j}{2}}}{(2^{\frac{j}{2}}|\nu-\nu'|)^2}+\frac{1}{2^{\frac{3j}{4}}(2^{\frac{j}{2}}|\nu-\nu'|)^{\frac{1}{2}}}+ 2^{-j}\bigg]\ep^2\gamma^\nu_j\gamma^{\nu'}_j.
\ee 
To this end, we proceed exactly as for  $B^{1,2,3,1}_{j,\nu,\nu',l,m}$, the only difference being that we integrate by parts tangentially using \eqref{fete1} instead of \eqref{fete} to obtain the analog of Lemma \ref{lemma:bizu}. This is left to the reader.

Finally, the decomposition \eqref{bizu7} of $B^{1,2,3}_{j,\nu,\nu',l,m}$ together with the estimates \eqref{bizu35} and \eqref{bizu36} imply:
\bee
&&\left|\sum_{(l,m)/ 2^{\max(l,m)}\leq 2^j|\nu-\nu'|}B^{1,2,3}_{j,\nu,\nu',l,m}\right| \\
\nn&\les&  \left|\sum_{(l,m)/ 2^{\max(l,m)}\leq 2^j|\nu-\nu'|}B^{1,2,3,1}_{j,\nu,\nu',l,m}\right|+\left|\sum_{(l,m)/ 2^{\max(l,m)}\leq 2^j|\nu-\nu'|}B^{1,2,3,2}_{j,\nu,\nu',l,m}\right|\\
\nn&\les&\bigg[\frac{j^2 2^{-\frac{j}{2}}}{(2^{\frac{j}{2}}|\nu-\nu'|)^2}+\frac{1}{2^{\frac{3j}{4}}(2^{\frac{j}{2}}|\nu-\nu'|)^{\frac{1}{2}}}+ 2^{-j}\bigg]\ep^2\gamma^\nu_j\gamma^{\nu'}_j.
\eee
Together with the estimate \eqref{bizu6} in the range of $(l,m)$ such that $2^m\leq 2^j|\nu-\nu'|< 2^l$, we obtain:
\bee
&&\left|\sum_{(l,m)/ 2^{\min(l,m)}\leq 2^j|\nu-\nu'|}B^{1,2,3}_{j,\nu,\nu',l,m}\right|\\ 
&\les&\bigg[\frac{j^2 2^{-\frac{j}{2}}}{(2^{\frac{j}{2}}|\nu-\nu'|)^2}+\frac{1}{2^{\frac{j}{2}}(2^{\frac{j}{2}}|\nu-\nu'|)}+\frac{1}{2^{\frac{3j}{4}}(2^{\frac{j}{2}}|\nu-\nu'|)^{\frac{1}{2}}}+ 2^{-j}\bigg]\ep^2\gamma^\nu_j\gamma^{\nu'}_j.
\eee
This concludes the proof of Proposition \ref{prop:labexfsmp9}.

\subsection{Proof of Proposition \ref{prop:labexfsmp1} (Control of $B^2_{j,\nu,\nu',l,m}$)}\lab{sec:labex2}

Recall the definition of $B^2_{j,\nu,\nu',l,m}$ \eqref{nice10}:
\bee
\nn B^2_{j,\nu,\nu',l,m}&=& -i2^{-j}\int_{\MM}\int_{\S\times\S} \frac{b^{-1}}{\gg(L,L')}\Bigg((\gn-1)P_l\trc N'(P_m\trc')+\bigg(\trc-\db\\
\nn&&-\db'-(1-\gn)\d' -2\z'_{N-\gn N'}-\frac{\chi'(N-\gn N', N-\gn N')}{\gl}\bigg)\\
&&\times P_l\trc P_m\trc'\Bigg)F_j(u)F_{j,-1}(u')\eta_j^\nu(\o)\eta_j^{\nu'}(\o')d\o d\o' d\MM.
\eee
We first consider the range of $(l,m)$ such that:
$$2^m\leq 2^j|\nu-\nu'|<2^l.$$
This yields:
\bee
&&\nn \sum_{m/2^m\leq 2^j|\nu-\nu'|}B^2_{j,\nu,\nu',l,m}\\
&=& -i2^{-j}\int_{\MM}\int_{\S\times\S} \frac{b^{-1}}{\gg(L,L')}\Bigg((\gn-1)P_l\trc N'(P_{\leq 2^j|\nu-\nu'|}\trc')+\bigg(\trc-\db\\
\nn&&-\db'-(1-\gn)\d' -2\z'_{N-\gn N'}-\frac{\chi'(N-\gn N', N-\gn N')}{\gl}\bigg)\\
&&\times P_l\trc P_{\leq 2^j|\nu-\nu'|}\trc'\Bigg)F_j(u)F_{j,-1}(u')\eta_j^\nu(\o)\eta_j^{\nu'}(\o')d\o d\o' d\MM.
\eee
Together with the identity \eqref{nice24}:
$$\gl=-1+\gn,$$
we obtain:
\bea
\lab{gini}&&\sum_{m/2^m\leq 2^j|\nu-\nu'|}B^2_{j,\nu,\nu',l,m}\\
\nn&=& -i2^{-j}\int_{\MM}\left(\int_{\S}b^{-1}P_l\trc F_j(u)\eta_j^\nu(\o)d\o\right)\left(N'(P_{\leq 2^j|\nu-\nu'|}\trc')F_{j,-1}(u')\eta_j^{\nu'}(\o')d\o'\right) d\MM\\
\nn&&-i2^{-j}\int_{\MM}\int_{\S\times\S} \frac{b^{-1}}{\gg(L,L')}\bigg(\trc-\db-\db'-(1-\gn)\d' -2\z'_{N-\gn N'}\\
\nn&&-\frac{\chi'(N-\gn N', N-\gn N')}{\gl}\bigg) P_l\trc P_{\leq 2^j|\nu-\nu'|}\trc'\\ 
\nn&&\times F_j(u)F_{j,-1}(u')\eta_j^\nu(\o)\eta_j^{\nu'}(\o')d\o d\o' d\MM.
\eea

Next, recall the identities \eqref{nice24} and \eqref{nice25}:
$$\gg(L,L')=-1+\gn\textrm{ and }1-\gn=\frac{\gg(N-N',N-N')}{2}.$$
We may thus expand 
$$\frac{1}{\gl}$$ 
in the same fashion than \eqref{nice27}, and in view of \eqref{gini}, we obtain, schematically:
\bea\lab{gini1}
&&\sum_{m/2^m\leq 2^j|\nu-\nu'|}B^2_{j,\nu,\nu',l,m}\\
\nn &=& 2^{-\frac{j}{2}}\sum_{p, q\geq 0}c_{pq}\int_{\MM}\frac{1}{(2^{\frac{j}{2}}|N_\nu-N_{\nu'}|)^{p+q+1}}\left[\frac{1}{|N_\nu-N_{\nu'}|}(h_{1,p,q,l}+h_{2,p,q,l})+h_{3,p,q,l}\right] d\MM\\
\nn&&-i2^{-j}\int_{\MM}\left(\int_{\S}b^{-1}P_l\trc F_j(u)\eta_j^\nu(\o)d\o\right)\\
\nn&&\times\left(\int_{\S}N'(P_{\leq 2^j|\nu-\nu'|}\trc')F_{j,-1}(u')\eta_j^{\nu'}(\o')d\o'\right) d\MM,
\eea
where the scalar functions $h_{1,p,q,l}, h_{2,p,q,l}, h_{3,p,q,l}$ on $\MM$ are given by:
\bea\lab{gini2}
h_{1,p,q,l}&=& \left(\int_{\S} (\trc-\db)P_l\trc\left(2^{\frac{j}{2}}(N-N_\nu)\right)^pF_{j,-1}(u)\eta_j^\nu(\o)d\o\right)\\
\nn&&\times\left(\int_{\S}P_{\leq 2^j|\nu-\nu'|}\trc'\left(2^{\frac{j}{2}}(N'-N_{\nu'})\right)^qF_{j,-1}(u')\eta_j^{\nu'}(\o')d\o'\right),
\eea
\bea\lab{gini3}
h_{2,p,q,l}&=& \left(\int_{\S} P_l\trc\left(2^{\frac{j}{2}}(N-N_\nu)\right)^pF_{j,-1}(u)\eta_j^\nu(\o)d\o\right)\\
\nn&&\times\left(\int_{\S}(-\db'-\chi')P_{\leq 2^j|\nu-\nu'|}\trc'\left(2^{\frac{j}{2}}(N'-N_{\nu'})\right)^qF_{j,-1}(u')\eta_j^{\nu'}(\o')d\o'\right),
\eea
and:
\bea\lab{gini4}
h_{3,p,q,l}&=& \left(\int_{\S} P_l\trc\left(2^{\frac{j}{2}}(N-N_\nu)\right)^pF_{j,-1}(u)\eta_j^\nu(\o)d\o\right)\\
\nn&&\times\left(\int_{\S}\z' P_{\leq 2^j|\nu-\nu'|}\trc'\left(2^{\frac{j}{2}}(N'-N_{\nu'})\right)^qF_{j,-1}(u')\eta_j^{\nu'}(\o')d\o'\right),
\eea
and where $c_{pq}$ are explicit real coefficients such that the series 
$$\sum_{p, q\geq 0}c_{pq}x^py^q$$
has radius of convergence 1. 

Next, we evaluate the $L^1(\MM)$ norm of $h_{1,p,q,l}, h_{2,p,q,l}, h_{3,p,q,l}$ starting with $h_{1,p,q,l}$. We have:
\bea\lab{gini5}
\norm{h_{1,p,q,l}}_{L^1(\MM)}&\les& \normm{\int_{\S} (\trc-\db)P_l\trc\left(2^{\frac{j}{2}}(N-N_\nu)\right)^pF_{j,-1}(u)\eta_j^\nu(\o)d\o}_{L^2(\MM)}\\
\nn&&\times\normm{\int_{\S}P_{\leq 2^j|\nu-\nu'|}\trc'\left(2^{\frac{j}{2}}(N'-N_{\nu'})\right)^qF_{j,-1}(u')\eta_j^{\nu'}(\o')d\o'}_{L^2(\MM)}\\
\nn&\les& (1+q^2)\ep\gamma^{\nu'}_j\normm{\int_{\S} (\trc-\db)P_l\trc\left(2^{\frac{j}{2}}(N-N_\nu)\right)^pF_{j,-1}(u)\eta_j^\nu(\o)d\o}_{L^2(\MM)},
\eea
where we used \eqref{mgen7} in the last inequality. The basic estimate in $L^2(\MM)$ \eqref{oscl2bis} yields:
\bea\lab{gini6}
&&\normm{\int_{\S} (\trc-\db)P_l\trc\left(2^{\frac{j}{2}}(N-N_\nu)\right)^pF_{j,-1}(u)\eta_j^\nu(\o)d\o}_{L^2(\MM)}\\
\nn&\les& \left(\sup_\o\normm{(\trc-\db)P_l\trc\left(2^{\frac{j}{2}}(N-N_\nu)\right)^p}_{\li{\infty}{2}}\right)2^{\frac{j}{2}}\gamma^\nu_j.
\eea
Now, for any tensor $G$, we have:
$$\norm{G P_m\trc'}_{\lprime{\infty}{2}}\les \norm{G}_{\xt{\infty}{2}}\norm{P_m\trc'}_{\xt{2}{\infty}},$$
which together with the estimate \eqref{lievremont1} for $P_m\trc'$ yields:
\be\lab{zaratustra}
\norm{G P_m\trc'}_{\lprime{\infty}{2}}\les 2^{-m}\ep \norm{G}_{\xt{\infty}{2}}.
\ee
Using the estimate \eqref{zaratustra} with $G=\trc-\db$, the estimate \eqref{esttrc} for $\trc$, the estimates \eqref{estn} \eqref{estk} for $\db$, the estimate \eqref{estNomega} for $\po N$ and the size of the patch, we have:
$$\normm{(\trc-\db)P_l\trc\left(2^{\frac{j}{2}}(N-N_\nu)\right)^p}_{\li{\infty}{2}}\les \ep 2^{-l}.$$
Together with \eqref{gini5} and \eqref{gini6}, we obtain:
\be\lab{gini7}
\norm{h_{1,p,q,l}}_{L^1(\MM)}\les (1+q^2)2^{-l+\frac{j}{2}}\ep^2\gamma^\nu_j\gamma^{\nu'}_j.
\ee

Next, we evaluate the $L^1(\MM)$ norm of $h_{2,p,q,l}$. In view of the definition \eqref{gini3}, we have:
\bea\lab{gini8}
&&\norm{h_{2,p,q,l}}_{L^1(\MM)}\\
\nn&\les& \normm{\int_{\S} P_l\trc\left(2^{\frac{j}{2}}(N-N_\nu)\right)^pF_{j,-1}(u)\eta_j^\nu(\o)d\o}_{L^2(\MM)}\\
\nn&&\times\normm{\int_{\S}(-\db'-\chi')P_{\leq 2^j|\nu-\nu'|}\trc'\left(2^{\frac{j}{2}}(N'-N_{\nu'})\right)^qF_{j,-1}(u')\eta_j^{\nu'}(\o')d\o'}_{L^2(\MM)}.
\eea
The basic estimate in $L^2(\MM)$ yields:
\bea\lab{gini9}
&&\normm{\int_{\S} P_l\trc\left(2^{\frac{j}{2}}(N-N_\nu)\right)^pF_{j,-1}(u)\eta_j^\nu(\o)d\o}_{L^2(\MM)}\\
\nn&\les& \left(\sup_\o\normm{P_l\trc\left(2^{\frac{j}{2}}(N-N_\nu)\right)^p}_{\li{\infty}{2}}\right)2^{\frac{j}{2}}\gamma^\nu_j\\
\nn&\les& 2^{-l+\frac{j}{2}}\ep\gamma^\nu_j,
\eea
where we used in the last inequality the finite band property for $P_l$, the estimate \eqref{esttrc} for $\trc$, the estimate \eqref{estNomega} for $\po N$ and the size of the patch. On the other hand, the decomposition \eqref{decchiom} for $\chi'$, the decomposition \eqref{beig} for $\db'$, and the estimate \eqref{estNomega} for $\po N$ yield the following decomposition for $\db'+\chi'$:
\be\lab{gini10}
\db'+\chi'=F^j_1+F^j_2
\ee
where the tensor $F^j_1$ only depends on $\nu'$ and satisfies:
\be\lab{gini11}
\norm{F^j_1}_{L^\infty_{u_{\nu'}}, x'_{\nu'}L^2_t}\les \ep,
\ee
and where the tensor $F^j_2$ satisfies:
\be\lab{gini12}
\norm{F^j_2}_{\lprime{\infty}{2}}\les \ep 2^{-\frac{j}{2}}.
\ee
In view of \eqref{gini10}, we obtain:
\bea\lab{gini13}
&&\normm{\int_{\S}(-\db'-\chi')P_{\leq 2^j|\nu-\nu'|}\trc'\left(2^{\frac{j}{2}}(N'-N_{\nu'})\right)^qF_{j,-1}(u')\eta_j^{\nu'}(\o')d\o'}_{L^2(\MM)}\\
\nn&\les& \norm{F^j_1}_{L^\infty_{u_{\nu'}}, x'_{\nu'}L^2_t}\normm{\int_{\S}P_{\leq 2^j|\nu-\nu'|}\trc'\left(2^{\frac{j}{2}}(N'-N_{\nu'})\right)^qF_{j,-1}(u')\eta_j^{\nu'}(\o')d\o'}_{L^2_{u_{\nu'}}, x'_{\nu'}L^\infty_t}\\
\nn&&+\normm{\int_{\S}F^j_2 P_{\leq 2^j|\nu-\nu'|}\trc'\left(2^{\frac{j}{2}}(N'-N_{\nu'})\right)^qF_{j,-1}(u')\eta_j^{\nu'}(\o')d\o'}_{L^2(\MM)}\\
\nn&\les& (1+q^{\frac{5}{2}})\ep\gamma^{\nu'}_j+\normm{\int_{\S}F^j_2 P_{\leq 2^j|\nu-\nu'|}\trc'\left(2^{\frac{j}{2}}(N'-N_{\nu'})\right)^qF_{j,-1}(u')\eta_j^{\nu'}(\o')d\o'}_{L^2(\MM)},
\eea
where we used in the last inequality the estimate \eqref{gini11} for $F^j_1$ and the estimate \eqref{mgen19}. Now, the basic estimate in $L^2(\MM)$ \eqref{oscl2bis} yields:
\bee
&&\normm{\int_{\S}F^j_2 P_{\leq 2^j|\nu-\nu'|}\trc'\left(2^{\frac{j}{2}}(N'-N_{\nu'})\right)^qF_{j,-1}(u')\eta_j^{\nu'}(\o')d\o'}_{L^2(\MM)}\\
&\les& \left(\sup_{\o'}\normm{F^j_2 P_{\leq 2^j|\nu-\nu'|}\trc'\left(2^{\frac{j}{2}}(N'-N_{\nu'})\right)^q}_{\lprime{\infty}{2}}\right)2^{\frac{j}{2}}\gamma^{\nu'}_j\\
&\les &\ep\gamma^{\nu'}_j,
\eee
where we used in the last inequality the estimate \eqref{gini12} for $F^j_2$, the boundedness of $P_{\leq 2^j|\nu-\nu'|}$ on $L^\infty(P_{t,u'})$, the estimate \eqref{esttrc} for $\trc$, the estimate \eqref{estNomega} for $\po N$ and the size of the patch. Together with \eqref{gini13}, this yields:
$$\normm{\int_{\S}(-\db'-\chi')P_{\leq 2^j|\nu-\nu'|}\trc'\left(2^{\frac{j}{2}}(N'-N_{\nu'})\right)^qF_{j,-1}(u')\eta_j^{\nu'}(\o')d\o'}_{L^2(\MM)}\les (1+q^{\frac{5}{2}})\ep\gamma^{\nu'}_j.$$
Together with \eqref{gini8} and \eqref{gini9}, we finally obtain:
\be\lab{gini14}
\norm{h_{2,p,q,l}}_{L^1(\MM)}\les (1+q^{\frac{5}{2}})2^{-l+\frac{j}{2}}\ep^2\gamma^\nu_j\gamma^{\nu'}_j.
\ee

Next, we evaluate the $L^1(\MM)$ norm of $h_{3,p,q,l}$. In view of the definition \eqref{gini4}, we have:
\bea\lab{gini15}
\norm{h_{3,p,q,l}}_{L^2(\MM)}&\les& \normm{\int_{\S} P_l\trc\left(2^{\frac{j}{2}}(N-N_\nu)\right)^pF_{j,-1}(u)\eta_j^\nu(\o)d\o}_{L^2(\MM)}\\
\nn&&\times\normm{\int_{\S}\z' P_{\leq 2^j|\nu-\nu'|}\trc'\left(2^{\frac{j}{2}}(N'-N_{\nu'})\right)^qF_{j,-1}(u')\eta_j^{\nu'}(\o')d\o'}_{L^2(\MM)}\\
\nn&\les& 2^{-l+\frac{j}{2}}\ep\gamma^\nu_j\normm{\int_{\S}\z' P_{\leq 2^j|\nu-\nu'|}\trc'\left(2^{\frac{j}{2}}(N'-N_{\nu'})\right)^qF_{j,-1}(u')\eta_j^{\nu'}(\o')d\o'}_{L^2(\MM)},
\eea
where we used \eqref{gini9} in the last inequality. In order to estimate the right-hand side of \eqref{gini15}, we use the decomposition \eqref{deczetaom} of $\z'$. We have:
\be\lab{gini16}
\z'=F^j_1+F^j_2
\ee
where the tensor $F^j_1$ only depends on $\nu'$ and satisfies:
\be\lab{gini17}
\norm{F^j_1}_{L^\infty_{u_{\nu'}}L^2_t,L^8_{x'_{\nu'}}}\les \ep,
\ee
and where the tensor $F^j_2$ satisfies:
\be\lab{gini18}
\norm{F^j_2}_{L^\infty_{u'}L^2(\H_{u'})}\les \ep 2^{-\frac{j}{4}}.
\ee
In view of \eqref{gini16}, we have:
\bea\lab{gini19}
&&\normm{\int_{\S}\z' P_{\leq 2^j|\nu-\nu'|}\trc'\left(2^{\frac{j}{2}}(N'-N_{\nu'})\right)^qF_{j,-1}(u')\eta_j^{\nu'}(\o')d\o'}_{L^2(\MM)}\\
\nn&\les& \norm{F^j_1}_{L^2_{u_{\nu'}}L^2_t,L^8_{x'_{\nu'}}}\normm{\int_{\S}P_{\leq 2^j|\nu-\nu'|}\trc'\left(2^{\frac{j}{2}}(N'-N_{\nu'})\right)^qF_{j,-1}(u')\eta_j^{\nu'}(\o')d\o'}_{L^{\frac{8}{3}}_{u_{\nu'}, x_{\nu'}'}L^\infty_t}\\
\nn&&+\normm{\int_{\S}F^j_2P_{\leq 2^j|\nu-\nu'|}\trc'\left(2^{\frac{j}{2}}(N'-N_{\nu'})\right)^qF_j(u')\eta^{\nu'}_j(\o') d\o'}_{L^2(\MM)}\\
\nn&\les& \ep\normm{\int_{\S}P_{\leq 2^j|\nu-\nu'|}\trc'\left(2^{\frac{j}{2}}(N'-N_{\nu'})\right)^qF_j(u')\eta^{\nu'}_j(\o') d\o'}_{L^{\frac{8}{3}}_{u_{\nu'}, x_{\nu'}'}L^\infty_t}\\
\nn&& +\normm{\int_{\S}F^j_2P_{\leq 2^j|\nu-\nu'|}\trc'\left(2^{\frac{j}{2}}(N'-N_{\nu'})\right)^qF_j(u')\eta^{\nu'}_j(\o') d\o'}_{L^2(\MM)},
\eea
where we used the estimate \eqref{gini17} for $F^j_1$ in the last inequality. We have the analog of \eqref{vino21}:
\be\lab{gini19bis}
\normm{\int_{\S} P_{\leq 2^j|\nu-\nu'|}\trc'\left(2^{\frac{j}{2}}(N-N_\nu)\right)^pF_{j,-1}(u)\eta_j^\nu(\o)d\o}_{L^\infty(\MM)}\les  \ep 2^j\gamma^\nu_j.
\ee
Now, interpolating between the the estimate in $L^2_{u_{\nu'},x_{\nu'}'}L^\infty_t$ \eqref{mgen19} and the $L^\infty$ estimate \eqref{gini19bis}, we obtain:
\be\lab{gini20}
\normm{\int_{\S}P_{\leq 2^j|\nu-\nu'|}\trc'\left(2^{\frac{j}{2}}(N'-N_{\nu'})\right)^qF_j(u')\eta^{\nu'}_j(\o') d\o'}_{L^4_{u_{\nu'}, x_{\nu'}'}L^\infty_t}\les  2^{\frac{j}{4}}\ep \gamma^{\nu'}_j.
\ee
For the second term in the right-hand side of \eqref{gini19}, we have:
\bee
&& \normm{\int_{\S}F^j_2P_{\leq 2^j|\nu-\nu'|}\trc'\left(2^{\frac{j}{2}}(N'-N_{\nu'})\right)^qF_j(u')\eta^{\nu'}_j(\o') d\o'}_{L^2(\MM)}\\
\nn&\les& \int_{\S}\norm{F^j_2 P_{\leq 2^j|\nu-\nu'|}\trc'\left(2^{\frac{j}{2}}(N'-N_{\nu'})\right)^qF_j(u')}_{L^2(\MM)}\eta^{\nu'}_j(\o') d\o'\\
\nn&\les& \int_{\S}\norm{F^j_2}_{L^\infty_{u'}L^2(\H_{u'}}\norm{F_j(u')}_{L^2_{u'}}\normm{P_{\leq 2^j|\nu-\nu'|}\trc'\left(2^{\frac{j}{2}}(N'-N_{\nu'})\right)^q}_{L^\infty(\MM)}\eta^{\nu'}_j(\o') d\o'.
\eee
Together with the estimate \eqref{cars16} for $F^j_2$, the boundedness of $P_{\leq 2^j|\nu-\nu'|}$ on $L^\infty(P_{t,u'})$, the estimate \eqref{esttrc} for $\trc$, the estimate \eqref{estNomega} for $\po N$ and the size of the patch, we obtain:
\bea\lab{gini21}
&& \normm{\int_{\S}F^j_2P_{\leq 2^j|\nu-\nu'|}\trc'\left(2^{\frac{j}{2}}(N'-N_{\nu'})\right)^qF_j(u')\eta^{\nu'}_j(\o') d\o'}_{L^2(\MM)}\\
\nn&\les& 2^{-\frac{j}{4}}\ep\int_{\S}\norm{F_j(u')}_{L^2_{u'}}\eta^{\nu'}_j(\o') d\o'\\
\nn&\les& \ep 2^{\frac{j}{4}}\gamma^{\nu'}_j,
\eea
where we used in the last inequality Plancherel in $\la'$, Cauchy Schwartz in $\o'$ and the size of the patch. Finally, \eqref{gini19}, \eqref{gini20} and \eqref{gini21} imply:
\be\lab{gini22}
\normm{\int_{\S}\z' P_{\leq 2^j|\nu-\nu'|}\trc'\left(2^{\frac{j}{2}}(N'-N_{\nu'})\right)^qF_{j,-1}(u')\eta_j^{\nu'}(\o')d\o'}_{L^2(\MM)}\les \ep 2^{\frac{j}{4}} \gamma^{\nu'}_j.
\ee
Together with \eqref{gini15}, this yields:
\be\lab{gini23}
\norm{h_{3,p,q,l}}_{L^2(\MM)}\les 2^{-l+\frac{j}{2}}2^{\frac{j}{4}}\ep^2\gamma^\nu_j\gamma^{\nu'}_j.
\ee

Next, we estimate the last term in the right-hand side of \eqref{gini1}:
$$\int_{\MM}\left(\int_{\S}b^{-1}P_l\trc F_j(u)\eta_j^\nu(\o)d\o\right)\left(N'(P_{\leq 2^j|\nu-\nu'|}\trc')F_{j,-1}(u')\eta_j^{\nu'}(\o')d\o'\right) d\MM.$$
We have:
\bea\lab{gini24}
\nn&&\left|\int_{\MM}\left(\int_{\S}b^{-1}P_l\trc F_j(u)\eta_j^\nu(\o)d\o\right)\left(N'(P_{\leq 2^j|\nu-\nu'|}\trc')F_{j,-1}(u')\eta_j^{\nu'}(\o')d\o'\right) d\MM\right|\\
\nn&\les& \normm{\int_{\S}b^{-1}P_l\trc F_j(u)\eta_j^\nu(\o)d\o}_{L^2(\MM)}\normm{\int_{\S}N'(P_{\leq 2^j|\nu-\nu'|}\trc')F_{j,-1}(u')\eta_j^{\nu'}(\o')d\o'}_{L^2(\MM)}\\
&\les& 2^{-l+j}\ep^2\gamma^\nu_j\gamma^{\nu'}_j,
\eea
where we used in the last inequality the estimate \eqref{gini9} and the estimate \eqref{vino56}.

Now, we have in view of the decomposition \eqref{gini1} of $B^2_{j,\nu,\nu',l,m}$:
\bee
&&\left|\sum_{m/2^m\leq 2^j|\nu-\nu'|}B^2_{j,\nu,\nu',l,m}\right|\\
\nn &\les& 2^{-\frac{j}{2}}\sum_{p, q\geq 0}c_{pq}\normm{\frac{1}{(2^{\frac{j}{2}}|N_\nu-N_{\nu'}|)^{p+q+1}}}_{L^\infty(\MM)}\bigg[\normm{\frac{1}{|N_\nu-N_{\nu'}|}}_{L^\infty(\MM)}(\norm{h_{1,p,q,l}}_{L^1(\MM)}\\
\nn&&+\norm{h_{2,p,q,l}}_{L^1(\MM)})+\norm{h_{3,p,q,l}}_{L^1(\MM)}\bigg] \\
\nn&& +2^{-j}\left|\int_{\MM}\left(\int_{\S}b^{-1}P_l\trc F_j(u)\eta_j^\nu(\o)d\o\right)\left(N'(P_{\leq 2^j|\nu-\nu'|}\trc')F_{j,-1}(u')\eta_j^{\nu'}(\o')d\o'\right) d\MM\right|.
\eee
Together with \eqref{nice26}, \eqref{gini7}, \eqref{gini14}, \eqref{gini23} and \eqref{gini24}, we obtain:
\bee
&&\left|\sum_{m/2^m\leq 2^j|\nu-\nu'|}B^2_{j,\nu,\nu',l,m}\right|\\
\nn &\les& 2^{-\frac{j}{2}}\sum_{p, q\geq 0}c_{pq}\frac{1}{(2^{\frac{j}{2}}|\nu-\nu'|)^{p+q+1}}\left[\frac{1}{|\nu-\nu'|}(1+q^{\frac{5}{2}})2^{-l+\frac{j}{2}}+2^{-l+\frac{j}{2}}2^{\frac{j}{4}}\right]\ep^2\gamma^\nu_j\gamma^{\nu'}_j \\
\nn&& +2^{-l}\ep^2\gamma^\nu_j\gamma^{\nu'}_j\\
\nn &\les& \left[\frac{2^{-l+\frac{j}{2}}}{(2^{\frac{j}{2}}|\nu-\nu'|)^2}+\frac{2^{-l}2^{\frac{j}{4}}}{(2^{\frac{j}{2}}|\nu-\nu'|)}+2^{-l}\right]\ep^2\gamma^\nu_j\gamma^{\nu'}_j.
\eee
Summing in $l$, we finally obtain in the range $2^m\leq 2^j|\nu-\nu'|<2^l$:
\bea\lab{gini25}
&&\left|\sum_{(l,m)/2^m\leq 2^j|\nu-\nu'|<2^l}B^2_{j,\nu,\nu',l,m}\right|\\
\nn &\les& \left[\frac{1}{(2^{\frac{j}{2}}|\nu-\nu'|)^3}+\frac{2^{-\frac{j}{4}}}{(2^{\frac{j}{2}}|\nu-\nu'|)^2}+\frac{1}{2^{\frac{j}{2}}(2^{\frac{j}{2}}|\nu-\nu'|)}\right]\ep^2\gamma^\nu_j\gamma^{\nu'}_j.
\eea

Next, we consider the range of $(l,m)$ such that:
$$2^m\leq 2^l\leq 2^j|\nu-\nu'|.$$
We integrate by parts in tangential directions using \eqref{fete1}. 
\begin{lemma}\lab{lemma:ldc}
Let $B^2_{j,\nu,\nu',l,m}$ be defined by \eqref{nice10}. Integrating by parts using \eqref{fete1} yields:
\bea\lab{ldc1}
\nn B^2_{j,\nu,\nu',l,m}&=& 2^{-\frac{3j}{2}}\sum_{p, q\geq 0}c_{pq}\int_{\MM}\frac{1}{(2^{\frac{j}{2}}|N_\nu-N_{\nu'}|)^{p+q+1}}\bigg[\frac{1}{|N_\nu-N_{\nu'}|^3}(h_{1,p,q,l,m}+h_{2,p,q,l,m})\\
\nn&&+\frac{1}{|N_\nu-N_{\nu'}|^2}(h_{3,p,q,l,m}+h_{4,p,q,l,m}+h_{5,p,q,l,m})+\frac{1}{|N_\nu-N_{\nu'}|}(h_{6,p,q,l,m}\\
\nn&&+h_{7,p,q,l,m}+h_{8,p,q,l,m})+h_{9,p,q,l,m}+h_{10,p,q,l,m}\bigg] d\MM\\
&&+\textrm{ terms interverting }(\nu,\nu')+B^{2,1}_{j,\nu,\nu',l,m}+B^{2,2}_{j,\nu,\nu',l,m},
\eea
where $c_{pq}$ are explicit real coefficients such that the series 
$$\sum_{p, q\geq 0}c_{pq}x^py^q$$
has radius of convergence 1, where the scalar functions $h_{1,p,q,l,m}$, $h_{2,p,q,l,m}$, $h_{3,p,q,l,m}$, $h_{4,p,q,l,m}$, $h_{5,p,q,l,m}$, $h_{6,p,q,l,m}$, $h_{7,p,q,l,m}$, $h_{8,p,q,l,m}$, $h_{9,p,q,l,m}$, $h_{10,p,q,l,m}$ on $\MM$ are given by:
\bea\lab{ldc2}
h_{1,p,q,l,m}&=& \left(\int_{\S}\chi(\chi+\db)P_l\trc \left(2^{\frac{j}{2}}(N-N_\nu)\right)^pF_{j,-1}(u)\eta_j^\nu(\o)d\o\right)\\
\nn&&\times\left(\int_{\S}P_m\trc'\left(2^{\frac{j}{2}}(N'-N_{\nu'})\right)^qF_{j,-1}(u')\eta_j^{\nu'}(\o')d\o'\right),
\eea
\bea\lab{ldc3}
h_{2,p,q,l,m}&=& \left(\int_{\S}(\chi+\db)P_l\trc\left(2^{\frac{j}{2}}(N-N_\nu)\right)^pF_{j,-1}(u)\eta_j^\nu(\o)d\o\right)\\
\nn&&\times\left(\int_{\S}(\chi'+\db')P_m\trc' \left(2^{\frac{j}{2}}(N'-N_{\nu'})\right)^qF_{j,-1}(u')\eta_j^{\nu'}(\o')d\o'\right),
\eea
\bea\lab{ldc4}
h_{3,p,q,l,m}&=& \left(\int_{\S}G_1\left(2^{\frac{j}{2}}(N-N_\nu)\right)^pF_{j,-1}(u)\eta_j^\nu(\o)d\o\right)\\
\nn&&\times\left(\int_{\S}P_m\trc'\left(2^{\frac{j}{2}}(N'-N_{\nu'})\right)^qF_{j,-1}(u')\eta_j^{\nu'}(\o')d\o'\right),
\eea
\bea\lab{ldc5}
h_{4,p,q,l,m}&=& \left(\int_{\S} \nabb(P_l\trc)\left(2^{\frac{j}{2}}(N-N_\nu)\right)^pF_{j,-1}(u)\eta_j^\nu(\o)d\o\right)\\
\nn&&\times\left(\int_{\S}(\chi'+\db')P_m\trc'\left(2^{\frac{j}{2}}(N'-N_{\nu'})\right)^qF_{j,-1}(u')\eta_j^{\nu'}(\o')d\o'\right),
\eea
\bea\lab{ldc6}
h_{5,p,q,l,m}&=& \left(\int_{\S} (\chi+\db)P_l\trc\left(2^{\frac{j}{2}}(N-N_\nu)\right)^pF_{j,-1}(u)\eta_j^\nu(\o)d\o\right)\\
\nn&&\times\left(\int_{\S}(\th'+\z')P_m\trc'\left(2^{\frac{j}{2}}(N'-N_{\nu'})\right)^qF_{j,-1}(u')\eta_j^{\nu'}(\o')d\o'\right),
\eea
\bea\lab{ldc7}
h_{6,p,q,l,m}&=& \left(\int_{\S}G_2\left(2^{\frac{j}{2}}(N-N_\nu)\right)^pF_{j,-1}(u)\eta_j^\nu(\o)d\o\right)\\
\nn&&\times\left(\int_{\S}P_m\trc'\left(2^{\frac{j}{2}}(N'-N_{\nu'})\right)^qF_{j,-1}(u')\eta_j^{\nu'}(\o')d\o'\right),
\eea
\bea\lab{ldc8}
h_{7,p,q,l,m}&=& \left(\int_{\S}(N(P_l\trc)+\nabb(P_l\trc))\left(2^{\frac{j}{2}}(N-N_\nu)\right)^pF_{j,-1}(u)\eta_j^\nu(\o)d\o\right)\\
\nn&&\times\left(\int_{\S}(\chi'+\db'+\z')P_m\trc'\left(2^{\frac{j}{2}}(N'-N_{\nu'})\right)^qF_{j,-1}(u')\eta_j^{\nu'}(\o')d\o'\right),
\eea
\bea\lab{ldc9}
h_{8,p,q,l,m}&=& \left(\int_{\S}(\th+b^{-1}\nabb(b))P_l\trc\left(2^{\frac{j}{2}}(N-N_\nu)\right)^pF_{j,-1}(u)\eta_j^\nu(\o)d\o\right)\\
\nn&&\times\left(\int_{\S}\z'P_m\trc'\left(2^{\frac{j}{2}}(N'-N_{\nu'})\right)^qF_{j,-1}(u')\eta_j^{\nu'}(\o')d\o'\right),
\eea
\bea\lab{ldc10}
h_{9,p,q,l,m}&=& \left(\int_{\S}P_l\trc\left(2^{\frac{j}{2}}(N-N_\nu)\right)^pF_{j,-1}(u)\eta_j^\nu(\o)d\o\right)\\
\nn&&\times\left(\int_{\S}\nabb'(N'(P_m\trc'))\left(2^{\frac{j}{2}}(N'-N_{\nu'})\right)^qF_{j,-1}(u')\eta_j^{\nu'}(\o')d\o'\right),
\eea
and:
\bea\lab{ldc11}
h_{10,p,q,l,m}&=& \left(\int_{\S}P_l\trc\left(2^{\frac{j}{2}}(N-N_\nu)\right)^pF_{j,-1}(u)\eta_j^\nu(\o)d\o\right)\\
\nn&&\times\left(\int_{\S}{b'}^{-1}\nabb(b') N'(P_m\trc')\left(2^{\frac{j}{2}}(N'-N_{\nu'})\right)^qF_{j,-1}(u')\eta_j^{\nu'}(\o')d\o'\right),
\eea
where the tensors $G_1$ and $G_2$ on $\MM$ are given by:
\be\lab{ldc12}
G_1=(\chi+\db)\nabb P_l\trc+(\nabb(\chi)+\nabb(\db)+(\chi+\db)(\th+b^{-1}\nabb(b)))P_l\trc,
\ee
and:
\be\lab{ldc13}
G_2=(\chi+\db)N(P_l\trc) +\z\nabb P_l\trc+(\dd_N(\chi)+N(\db)+\nabb(\z)\th) P_l\trc,
\ee
and where $B^{2,1}_{j,\nu,\nu',l,m}$ and $B^{2,2}_{j,\nu,\nu',l,m}$ are defined by:
\bea\lab{ldc14}
B^{2,1}_{j,\nu,\nu',l,m}&=&2^{-2j}\int_{\MM}\left(\int_{\S} N(P_l\trc)F_{j,-1}(u)\eta_j^\nu(\o)d\o\right)\\
\nn&&\times\left(\int_{\S}N'(P_m\trc')F_{j,-1}(u')\eta_j^{\nu'}(\o')d\o'\right)d\MM,
\eea
and:
\bea\lab{ldc15}
B^{2,2}_{j,\nu,\nu',l,m}&=&2^{-2j}\int_{\MM}\int_{\S\times\S}\frac{(N'-\gn N)(P_l\trc)N'(P_m\trc')}{1-\gn^2}\\
\nn&&\times F_{j,-1}(u)F_{j,-1}(u')\eta_j^\nu(\o)\eta_j^{\nu'}(\o')d\o d\o'd\MM.
\eea
\end{lemma}
 
The proof of Lemma \ref{lemma:ldc} is postponed to Appendix G. In the rest of this section, we use Lemma \ref{lemma:ldc} to control $B^2_{j,\nu,\nu',l,m}$ over the range of $(l,m)$ such that $2^m\leq 2^l\leq 2^j|\nu-\nu'|$.

\subsubsection{Control of the $L^1(\MM)$ norm of $h_{1,p,q,l,m}$}

We estimate the $L^1(\MM)$ norm of $h_{1,p,q,l,m}$, $h_{2,p,q,l,m}$, $h_{3,p,q,l,m}$, $h_{4,p,q,l,m}$, $h_{5,p,q,l,m}$, $h_{6,p,q,l,m}$, $h_{7,p,q,l,m}$, $h_{8,p,q,l,m}$, $h_{9,p,q,l,m}$, $h_{10,p,q,l,m}$ starting with $h_{1,p,q,l,m}$. 
Consider first the case $l>j/2$. Let $H$ be defined by:
\be\lab{ldc16}
H=\int_{\S}P_m\trc'\left(2^{\frac{j}{2}}(N'-N_{\nu'})\right)^qF_{j,-1}(u')\eta_j^{\nu'}(\o')d\o'.
\ee
Then, we have in view of the definition \eqref{ldc2} of $h_{1,p,q,l,m}$:
 \bea\lab{ldc17}
\norm{h_{1,p,q,l,m}}_{L^1(\MM)}&\les & \normm{\int_{\S}\chi H(\chi+\db)P_l\trc \left(2^{\frac{j}{2}}(N-N_\nu)\right)^pF_{j,-1}(u)\eta_j^\nu(\o)d\o}_{L^1(\MM)}\\
\nn&\les&\int_{\S}\norm{\chi H}_{L^2(\MM)}\normm{(\chi+\db)P_l\trc \left(2^{\frac{j}{2}}(N-N_\nu)\right)^pF_{j,-1}(u)}_{L^2(\MM)}\eta_j^\nu(\o)d\o\\
\nn&\les&\int_{\S}\norm{\chi}_{\xt{\infty}{2}}\norm{H}_{L^2_{u, x'}L^\infty_t}\norm{(\chi+\db)P_l\trc}_{\li{\infty}{2}} \\
\nn&&\times\normm{\left(2^{\frac{j}{2}}(N-N_\nu)\right)^p}_{L^\infty(\MM)}\norm{F_{j,-1}(u)}_{L^2_u}\eta_j^\nu(\o)d\o\\
\nn&\les& \ep 2^{-l}\int_{\S}\norm{H}_{L^2_{u, x'}L^\infty_t}\norm{F_{j,-1}(u)}_{L^2_u}\eta_j^\nu(\o)d\o,
\eea
where we used in the last inequality the estimate \eqref{zaratustra} with the choice $G=\chi+\db$, the estimates \eqref{esttrc} \eqref{esthch} for $\chi$, the estimates \eqref{estn} \eqref{estk} for $\db$, the estimate \eqref{estNomega} for $\po N$ and the size of the patch. Next we estimate the term in $H$ in the right-hand side of \eqref{ldc17}. Using the estimate \eqref{messi4:0} in the case $m>j/2$, we have:
\bea\lab{ldc18}
&&\normm{\int_{\S}P_m\trc'\left(2^{\frac{j}{2}}(N'-N_{\nu'})\right)^qF_{j,-1}(u')\eta_j^{\nu'}(\o')d\o'}_{L^2_{u, x'}L^\infty_t}\\
\nn&\les& \ep\big(2^{\frac{j}{2}}|\nu-\nu'|2^{-m+\frac{j}{2}}+(2^{\frac{j}{2}}|\nu-\nu'|)^{\frac{1}{2}}2^{-\frac{m}{2}+\frac{j}{4}}\big)\gamma^{\nu'}_j.
\eea
Also, using the decomposition:
$$P_{\leq j/2}\trc'=\trc'-\sum_{m/m>j/2}P_m\trc',$$
together with the estimate \eqref{messi4:0} and the estimate \eqref{koko1}, we obtain in the case $m=j/2$:
$$\normm{\int_{\S}P_{\leq j/2}\trc'\left(2^{\frac{j}{2}}(N'-N_{\nu'})\right)^qF_{j,-1}(u')\eta_j^{\nu'}(\o')d\o'}_{L^2_{u, x'}L^\infty_t}\les (1+q^{\frac{5}{2}})\ep\big(2^{\frac{j}{2}}|\nu-\nu'|+1\big)\gamma^{\nu'}_j.$$
Together with \eqref{ldc18}, we obtain for all $m\geq j/2$:
\bea\lab{ldc18bis}
&&\normm{\int_{\S}P_m\trc'\left(2^{\frac{j}{2}}(N'-N_{\nu'})\right)^qF_{j,-1}(u')\eta_j^{\nu'}(\o')d\o'}_{L^2_{u, x'}L^\infty_t}\\
\nn&\les& \ep\big(2^{\frac{j}{2}}|\nu-\nu'|2^{-m+\frac{j}{2}}+(2^{\frac{j}{2}}|\nu-\nu'|)^{\frac{1}{2}}2^{-\frac{m}{2}+\frac{j}{4}}\big)\gamma^{\nu'}_j.
\eea
In view of the definition \eqref{ldc16} of $H$, this yields:
$$\norm{H}_{L^2_{u_{\nu'}, x_{\nu'}'}L^\infty_t}\les  \ep\big(2^{\frac{j}{2}}|\nu-\nu'|2^{-m+\frac{j}{2}}+(2^{\frac{j}{2}}|\nu-\nu'|)^{\frac{1}{2}}2^{-\frac{m}{2}+\frac{j}{4}}\big)\gamma^{\nu'}_j.$$
Together with \eqref{ldc17}, we obtain in the case $l>j/2$:
\bea\lab{ldc19}
\norm{h_{1,p,q,l,m}}_{L^1(\MM)}&\les &  \ep^2\big(2^{\frac{j}{2}}|\nu-\nu'|2^{-m+\frac{j}{2}}+(2^{\frac{j}{2}}|\nu-\nu'|)^{\frac{1}{2}}2^{-\frac{m}{2}+\frac{j}{4}}\big)2^{-l}\gamma^{\nu'}_j\\
\nn&&\times\left(\int_{\S}\norm{F_{j,-1}(u)}_{L^2_u}\eta_j^\nu(\o)d\o\right)\\
\nn&\les & \ep^2\big(2^{\frac{j}{2}}|\nu-\nu'|2^{-m+\frac{j}{2}}+(2^{\frac{j}{2}}|\nu-\nu'|)^{\frac{1}{2}}2^{-\frac{m}{2}+\frac{j}{4}}\big)2^{\frac{j}{2}-l}\gamma^\nu_j\gamma^{\nu'}_j,
\eea
where we used in the last inequality Plancherel in $\la$ for $\norm{F_{j,-1}(u)}_{L^2_u}$, Cauchy Schwartz in 
 $\o$ and the size of the patch. 
 
Next, we consider the case $l=j/2$. Recall that in view of the decomposition for $\chi$ \eqref{decchiom}, we have:
\be\lab{ldc20}
\chi=F^j_1+F^j_2
\ee
where the tensor $F^j_1$ only depends on $\nu$ and satisfies:
\be\lab{ldc21}
\norm{F^j_1}_{L^\infty_{u_\nu, x'_{\nu}}L^2_t}\les \ep,
\ee
and where the tensor $F^j_2$ satisfies:
\be\lab{ldc22}
\norm{F^j_2}_{L^\infty_u\lh{2}}\les \ep 2^{-\frac{j}{2}}.
\ee
In view of the definition \eqref{ldc2} of $h_{1,p,q,l,m}$, the decomposition \eqref{ldc20} and the definition of $H$ \eqref{ldc16}, we have in the case $l=j/2$:
\bea\lab{ldc23}
&&\norm{h_{1,p,q,l,m}}_{L^1(\MM)}\\
\nn&\les& \norm{F^j_1}_{L^\infty_{u_\nu, x'_{\nu}}L^2_t}\normm{\int_{\S}(\chi+\db)P_{\leq j/2}\trc \left(2^{\frac{j}{2}}(N-N_\nu)\right)^pF_{j,-1}(u)\eta_j^\nu(\o)d\o}_{L^2(\MM)}\normm{H}_{L^2_{u_\nu, x'_{\nu}}L^\infty_t}\\
\nn&&+\normm{\int_{\S}F^j_2(\chi+\db)HP_{\leq j/2}\trc \left(2^{\frac{j}{2}}(N-N_\nu)\right)^pF_{j,-1}(u)\eta_j^\nu(\o)d\o}_{L^1(\MM)}\\
\nn&\les& \ep\big(2^{\frac{j}{2}}|\nu-\nu'|2^{-m+\frac{j}{2}}+(2^{\frac{j}{2}}|\nu-\nu'|)^{\frac{1}{2}}2^{-\frac{m}{2}+\frac{j}{4}}\big)\gamma^{\nu'}_j\\
\nn&&\times\normm{\int_{\S}(\chi+\db)P_{\leq j/2}\trc \left(2^{\frac{j}{2}}(N-N_\nu)\right)^pF_{j,-1}(u)\eta_j^\nu(\o)d\o}_{L^2(\MM)}\\
\nn&&+\int_{\S}\normm{F^j_2(\chi+\db)HP_{\leq j/2}\trc \left(2^{\frac{j}{2}}(N-N_\nu)\right)^pF_{j,-1}(u)}_{L^1(\MM)}\eta_j^\nu(\o)d\o,
\eea
where we used in the last inequality the estimate \eqref{ldc17} for $F^j_1$ and the estimate \eqref{ldc18bis} for $H$. Next, we estimate the two terms in the right-hand side of \eqref{ldc23} starting with the first one. In view of the decomposition \eqref{beig} for $\db$, $\db$ also has a decomposition of the form \eqref{ldc16} \eqref{ldc17} \eqref{ldc18}, and thus so has $\chi+\db$. Proceeding as in \eqref{mgen16}, \eqref{mgen19} and \eqref{mgen20}, we obtain the analog of \eqref{mgen21}:
\bea\lab{ldc24}
\normm{\int_{\S}(\chi+\db)P_{\leq j/2}\trc \left(2^{\frac{j}{2}}(N-N_\nu)\right)^pF_{j,-1}(u)\eta_j^\nu(\o)d\o}_{L^2(\MM)}\les \ep(1+q^{\frac{5}{2}})\gamma^\nu_j.
\eea
On the other hand, we have:
\bee
&&\int_{\S}\normm{F^j_2(\chi+\db)HP_{\leq j/2}\trc \left(2^{\frac{j}{2}}(N-N_\nu)\right)^pF_{j,-1}(u)}_{L^1(\MM)}\eta_j^\nu(\o)d\o\\
\nn&\les& \int_{\S}\norm{F^j_2}_{\li{\infty}{2}}\norm{\chi+\db}_{\xt{\infty}{2}}\norm{H}_{L^2_{u, x'}L^\infty_t}\normm{P_{\leq j/2}\trc \left(2^{\frac{j}{2}}(N-N_\nu)\right)^p}_{L^\infty(\MM)}\\
\nn&&\times\norm{F_{j,-1}(u)}_{L^2_u}\eta_j^\nu(\o)d\o\\
\nn&\les& \ep 2^{-\frac{j}{2}}\int_{\S}\norm{H}_{L^2_{u, x'}L^\infty_t}\norm{F_{j,-1}(u)}_{L^2_u}\eta_j^\nu(\o)d\o,
\eee
where we used in the last inequality the estimate \eqref{ldc22} for $F^j_2$, the estimates \eqref{esttrc} \eqref{esthch} for $\chi$, the estimates \eqref{estn} \eqref{estk} for $\db$, the boundedness of $P_{\leq j/2}$ on $L^\infty(\ptu)$, the estimate \eqref{estNomega} for $\po N$ and the size of the patch. Together with the estimate \eqref{ldc18bis} for $H$, Plancherel in $\la$ for $\norm{F_{j,-1}(u)}_{L^2_u}$, Cauchy Schwartz in $\o$ and the size of the patch, we obtain:
\bea\lab{ldc25}
&&\int_{\S}\normm{F^j_2(\chi+\db)HP_{\leq j/2}\trc \left(2^{\frac{j}{2}}(N-N_\nu)\right)^pF_{j,-1}(u)}_{L^1(\MM)}\eta_j^\nu(\o)d\o\\
\nn&\les& \ep^2\big(2^{\frac{j}{2}}|\nu-\nu'|2^{-m+\frac{j}{2}}+(2^{\frac{j}{2}}|\nu-\nu'|)^{\frac{1}{2}}2^{-\frac{m}{2}+\frac{j}{4}}\big)\gamma^\nu_j\gamma^{\nu'}_j,
\eea
Now, \eqref{ldc23}, \eqref{ldc24} and \eqref{ldc25} yield in the case $l=j/2$:
$$\norm{h_{1,p,q,l,m}}_{L^1(\MM)}\les (1+q^{\frac{5}{2}})\ep^2\big(2^{\frac{j}{2}}|\nu-\nu'|2^{-m+\frac{j}{2}}+(2^{\frac{j}{2}}|\nu-\nu'|)^{\frac{1}{2}}2^{-\frac{m}{2}+\frac{j}{4}}\big)\gamma^\nu_j\gamma^{\nu'}_j.$$
Together with \eqref{ldc19}, we finally obtain for all $l\geq j/2$:
\be\lab{ldc26}
\norm{h_{1,p,q,l,m}}_{L^1(\MM)}\les (1+q^{\frac{5}{2}})\ep^2\big(2^{\frac{j}{2}}|\nu-\nu'|2^{-m+\frac{j}{2}}+(2^{\frac{j}{2}}|\nu-\nu'|)^{\frac{1}{2}}2^{-\frac{m}{2}+\frac{j}{4}}\big)2^{\frac{j}{2}-l}\gamma^\nu_j\gamma^{\nu'}_j
\ee

\subsubsection{Control of the $L^1(\MM)$ norm of $h_{2,p,q,l,m}$}

Next, we estimate the $L^1(\MM)$ norm of $h_{2,p,q,l,m}$. In view of the definition \eqref{ldc3} of $h_{2,p,q,l,m}$, we have:
\bea\lab{ldc27}
\norm{h_{2,p,q,l,m}}_{L^1(\MM)}&\les& \normm{\int_{\S}(\chi+\db)P_l\trc\left(2^{\frac{j}{2}}(N-N_\nu)\right)^pF_{j,-1}(u)\eta_j^\nu(\o)d\o}_{L^2(\MM)}\\
\nn&&\times\normm{\int_{\S}(\chi'+\db')P_m\trc' \left(2^{\frac{j}{2}}(N'-N_{\nu'})\right)^qF_{j,-1}(u')\eta_j^{\nu'}(\o')d\o'}_{L^2(\MM)},
\eea
In the case $l>j/2$, we use the basic estimate in $L^2(\MM)$:
\bee
&&\normm{\int_{\S}(\chi+\db)P_l\trc \left(2^{\frac{j}{2}}(N-N_\nu)\right)^pF_{j,-1}(u)\eta_j^\nu(\o)d\o}_{L^2(\MM)}\\
&\les& \left(\sup_\o\normm{(\chi+\db)P_l\trc \left(2^{\frac{j}{2}}(N-N_\nu)\right)^p}_{\li{\infty}{2}}\right)2^{\frac{j}{2}}\gamma^\nu_j\\
&\les & \ep 2^{\frac{j}{2}-l}\gamma^\nu_j,
\eee
where we used in the last inequality the estimate \eqref{zaratustra} with the choice $G=\chi+\db$, the estimates \eqref{esttrc} \eqref{esthch} for $\chi$, the estimates \eqref{estn} \eqref{estk} for $\db$, the estimate \eqref{estNomega} for $\po N$ and the size of the patch. Together with \eqref{ldc21}, we obtain for all $l\geq j/2$:
\be\lab{ldc28}
\normm{\int_{\S}(\chi+\db)P_l\trc \left(2^{\frac{j}{2}}(N-N_\nu)\right)^pF_{j,-1}(u)\eta_j^\nu(\o)d\o}_{L^2(\MM)}
\les  \ep 2^{\frac{j}{2}-l}\gamma^\nu_j.
\ee
Finally, \eqref{ldc27}, \eqref{ldc28} and the analog of \eqref{ldc28} for the second term in the right-hand side of \eqref{ldc27} implies:
\be\lab{ldc29}
\norm{h_{2,p,q,l,m}}_{L^1(\MM)}\les  \ep^2 2^{j-l-m}\gamma^\nu_j\gamma^{\nu'}_j.
\ee

\subsubsection{Control of the $L^1(\MM)$ norm of $h_{3,p,q,l,m}$}

In view of the definition \eqref{ldc4} of $h_{3,p,q,l,m}$ and in view of the definition \eqref{ldc12} of $G_1$, we have:
\bee
\sum_{l/l\geq m}h_{3,p,q,l,m}&=& \left(\int_{\S}\widetilde{G}_1\left(2^{\frac{j}{2}}(N-N_\nu)\right)^pF_{j,-1}(u)\eta_j^\nu(\o)d\o\right)\\
\nn&&\times\left(\int_{\S}P_m\trc'\left(2^{\frac{j}{2}}(N'-N_{\nu'})\right)^qF_{j,-1}(u')\eta_j^{\nu'}(\o')d\o'\right),
\eee
where $\widetilde{G}_1$ is given by:
\be\lab{ldc31}
\widetilde{G}_1=(\chi+\db)\nabb P_{\geq m}\trc+(\nabb(\chi)+\nabb(\db)+(\chi+\db)(\th+b^{-1}\nabb(b)))P_{\geq m}\trc.
\ee
This yields the following estimate:
\bea\lab{ldc32}
\normm{\sum_{l/l\geq m}h_{3,p,q,l,m}}_{L^1(\MM)}&\les & \normm{\int_{\S}\widetilde{G}_1\left(2^{\frac{j}{2}}(N-N_\nu)\right)^pF_{j,-1}(u)\eta_j^\nu(\o)d\o}_{L^2(\MM)}\\
\nn&&\times\normm{\int_{\S}P_m\trc'\left(2^{\frac{j}{2}}(N'-N_{\nu'})\right)^qF_{j,-1}(u')\eta_j^{\nu'}(\o')d\o'}_{L^2(\MM)}.
\eea
Now, using \eqref{messi4:0} in the case $m>j/2$, and \eqref{vino20} in the case $m=j/2$, we obtain for all $m\geq j/2$:
\be\lab{ldc33}
\normm{\int_{\S}P_m\trc'\left(2^{\frac{j}{2}}(N'-N_{\nu'})\right)^qF_{j,-1}(u')\eta_j^{\nu'}(\o')d\o'}_{L^2_{u_{\nu'}, x_{\nu'}'}L^\infty_t}\les  (1+q^2)\ep 2^{\frac{j}{2}-m}\gamma^{\nu'}_j.
\ee
Together with \eqref{ldc32}, this yields:
\be\lab{ldc34}
\normm{\sum_{l/l\geq m}h_{3,p,q,l,m}}_{L^1(\MM)}\les  (1+q^2)\ep 2^{\frac{j}{2}-m}\gamma^{\nu'}_j\normm{\int_{\S}\widetilde{G}_1\left(2^{\frac{j}{2}}(N-N_\nu)\right)^pF_{j,-1}(u)\eta_j^\nu(\o)d\o}_{L^2(\MM)}.
\ee

Next, we estimate the right-hand side of \eqref{ldc34}. The basic estimate in $L^2(\MM)$ \eqref{oscl2bis} yields:
\bea\lab{ldc35}
&&\normm{\int_{\S}\widetilde{G}_1\left(2^{\frac{j}{2}}(N-N_\nu)\right)^pF_{j,-1}(u)\eta_j^\nu(\o)d\o}_{L^2(\MM)}\\
\nn&\les& \left(\sup_\o\normm{\widetilde{G}_1\left(2^{\frac{j}{2}}(N-N_\nu)\right)^p}_{\li{\infty}{2}}\right)2^{\frac{j}{2}}\gamma^\nu_j\\
\nn&\les& \left(\sup_\o\normm{\widetilde{G}_1}_{\li{\infty}{2}}\right)2^{\frac{j}{2}}\gamma^\nu_j,
\eea
where we used in the last inequality the estimate \eqref{estNomega} for $\po N$ and the size of the patch. In view of \eqref{ldc31}, we have:
\bea\lab{ldc36}
\normm{\widetilde{G}_1}_{\li{\infty}{2}}&\les& (\norm{\chi}_{\xt{\infty}{2}}+\norm{\db}_{\xt{\infty}{2}})(\norm{\nabb\trc}_{\xt{2}{\infty}}+\norm{\nabb P_{\leq m}\trc}_{\xt{2}{\infty}})\\
\nn&&+(\norm{\nabb(\chi)}_{\li{\infty}{2}}+\norm{\nabb(\db)}_{\li{\infty}{2}}\\
\nn&&+\norm{(\chi+\db)(\th+b^{-1}\nabb(b))}_{\li{\infty}{2}})\norm{P_{\geq m}\trc}_{L^\infty}\\
\nn&\les& \ep,
\eea
where we used in the last inequality the estimates \eqref{esttrc} \eqref{esthch} for $\chi$, the estimates \eqref{estk} \eqref{estn} for $\db$, the estimate \eqref{estb} for $b$, the estimates \eqref{estk} \eqref{esttrc} \eqref{esthch} for $\th$, the decomposition: 
$$\trc=P_{\leq m}\trc+P_{>m}\trc,$$
the estimate \eqref{lievremont2} for $\nabb P_{\leq m}\trc$, and the fact that $P_{\geq m}$ is bounded on $L^\infty(\ptu)$. Finally, \eqref{ldc34}, \eqref{ldc35} and \eqref{ldc36} imply:
\be\lab{ldc37}
\normm{\sum_{l/l\geq m}h_{3,p,q,l,m}}_{L^1(\MM)}\les  (1+q^2)\ep^2 2^{j-m}\gamma^\nu_j\gamma^{\nu'}_j.
\ee

\subsubsection{Control of the $L^1(\MM)$ norm of $h_{4,p,q,l,m}$}

In view of the definition \eqref{ldc5} of $h_{4,p,q,l,m}$, we have:
\bee
\sum_{l/l\geq m}h_{4,p,q,l,m}&=& \left(\int_{\S} \nabb(P_{\geq m}\trc)\left(2^{\frac{j}{2}}(N-N_\nu)\right)^pF_{j,-1}(u)\eta_j^\nu(\o)d\o\right)\\
\nn&&\times\left(\int_{\S}(\chi'+\db')P_m\trc'\left(2^{\frac{j}{2}}(N'-N_{\nu'})\right)^qF_{j,-1}(u')\eta_j^{\nu'}(\o')d\o'\right)
\eee
which yields:
\bea\lab{ldc38}
\normm{\sum_{l/l\geq m}h_{4,p,q,l,m}}_{L^1(\MM)}&\leq& \normm{\int_{\S} \nabb(P_{\geq m}\trc)\left(2^{\frac{j}{2}}(N-N_\nu)\right)^pF_{j,-1}(u)\eta_j^\nu(\o)d\o}_{L^2(\MM)}\\
\nn&&\times\normm{\int_{\S}(\chi'+\db')P_m\trc'\left(2^{\frac{j}{2}}(N'-N_{\nu'})\right)^qF_{j,-1}(u')\eta_j^{\nu'}(\o')d\o'}_{L^2(\MM)}\\
\nn&\les& \ep 2^{\frac{j}{2}-m}\gamma^{\nu'}_j\normm{\int_{\S} \nabb(P_{\geq m}\trc)\left(2^{\frac{j}{2}}(N-N_\nu)\right)^pF_{j,-1}(u)\eta_j^\nu(\o)d\o}_{L^2(\MM)},
\eea
where we used in the last inequality the estimate \eqref{ldc28}. 

Now, the basic estimate in $L^2(\MM)$ \eqref{oscl2bis} yields:
\bea\lab{ldc39}
&&\normm{\int_{\S} \nabb(P_{\geq m}\trc)\left(2^{\frac{j}{2}}(N-N_\nu)\right)^pF_{j,-1}(u)\eta_j^\nu(\o)d\o}_{L^2(\MM)}\\
\nn&\les& \left(\sup_\o\normm{\nabb(P_{\geq m}\trc)\left(2^{\frac{j}{2}}(N-N_\nu)\right)^p}_{\li{\infty}{2}}\right)2^{\frac{j}{2}}\gamma^\nu_j\\
\nn&\les& \ep 2^{\frac{j}{2}}\gamma^\nu_j,
\eea
where we used in the last inequality the finite band property for $P_{\geq m}$, the estimate \eqref{esttrc} for $\trc$, the estimate \eqref{estNomega} for $\po N$ and the size of the patch. Finally, \eqref{ldc38} and \eqref{ldc39} yield:
\be\lab{ldc40}
\normm{\sum_{l/l\geq m}h_{4,p,q,l,m}}_{L^1(\MM)}\les  \ep^2 2^{j-m}\gamma^\nu_j\gamma^{\nu'}_j.
\ee

\subsubsection{Control of the $L^1(\MM)$ norm of $h_{5,p,q,l,m}$}

In view of the definition \eqref{ldc6} of $h_{5,p,q,l,m}$, we have:
\bea\lab{ldc41}
\norm{h_{5,p,q,l,m}}_{L^1(\MM)}&\les& \normm{\int_{\S} (\chi+\db)P_l\trc\left(2^{\frac{j}{2}}(N-N_\nu)\right)^pF_{j,-1}(u)\eta_j^\nu(\o)d\o}_{L^2(\MM)}\\
\nn&&\times\normm{\int_{\S}(\th'+\z')P_m\trc'\left(2^{\frac{j}{2}}(N'-N_{\nu'})\right)^qF_{j,-1}(u')\eta_j^{\nu'}(\o')d\o'}_{L^2(\MM)}\\
\nn&\les & \ep 2^{\frac{j}{2}-l}\gamma^\nu_j\normm{\int_{\S}(\th'+\z')P_m\trc'\left(2^{\frac{j}{2}}(N'-N_{\nu'})\right)^qF_{j,-1}(u')\eta_j^{\nu'}(\o')d\o'}_{L^2(\MM)},
\eea
where we used in the last inequality the estimate \eqref{ldc28}.

Now, the basic estimate in $L^2(\MM)$ \eqref{oscl2bis} yields:
\bea\lab{ldc42}
&&\normm{\int_{\S}(\th'+\z')P_m\trc'\left(2^{\frac{j}{2}}(N'-N_{\nu'})\right)^qF_{j,-1}(u')\eta_j^{\nu'}(\o')d\o'}_{L^2(\MM)}\\
\nn&\les& \left(\sup_{\o'}\normm{(\th'+\z')P_m\trc'\left(2^{\frac{j}{2}}(N'-N_{\nu'})\right)^q}_{\lprime{\infty}{2}}\right)2^{\frac{j}{2}}\gamma^\nu_j\\
\nn&\les& \left(\sup_{\o'}\normm{(\th'+\z')P_m\trc'}_{\lprime{\infty}{2}}\right) 2^{\frac{j}{2}}\gamma^\nu_j,
\eea
where we used in the last inequality the estimate \eqref{estNomega} for $\po N$ and the size of the patch. We have:
\bea\lab{ldc43}
\normm{(\th'+\z')P_m\trc'}_{\lprime{\infty}{2}}&\les &(\norm{\th'}_{\tx{\infty}{4}}+\norm{\z'}_{\tx{\infty}{4}})\norm{P_m\trc'}_{\lprime{\infty}{4}}\\
\nn&\les & 2^{-\frac{m}{2}(1-\delta_{m,j/2})}\ep,
\eea
where we used in the last inequality the embedding \eqref{sobineq1}, the estimates \eqref{estk} \eqref{esttrc} \eqref{esthch} for $\th'$, the estimates \eqref{estzeta} for $\z'$, the Bernstein inequality and the finite band property for $P_m$, and the estimate \eqref{esttrc} for $\trc'$. Note that the factor $1-\delta_{j/2,m}$ comes from the fact that we use the finite band property for $P_m$ in the case $m>j/2$, but only the boundedness of $P_{\leq j/2}$ in the case $m=j/2$. Finally, \eqref{ldc41}, \eqref{ldc42} and \eqref{ldc43} yield:
\be\lab{ldc44}
\norm{h_{5,p,q,l,m}}_{L^1(\MM)}\les  \ep^2 2^{j-l-\frac{m}{2}(1-\delta_{m,j/2})}\gamma^\nu_j\gamma^{\nu'}_j.
\ee

\subsubsection{Control of the $L^1(\MM)$ norm of $h_{6,p,q,l,m}$}

In view of the definition \eqref{ldc7} of $h_{6,p,q,l,m}$, we have:
\bea\lab{ldc45}
\norm{h_{6,p,q,l,m}}_{L^1(\MM)}&\leq& \normm{\int_{\S}G_2\left(2^{\frac{j}{2}}(N-N_\nu)\right)^pF_{j,-1}(u)\eta_j^\nu(\o)d\o}_{L^2(\MM)}\\
\nn&&\times\normm{\int_{\S}P_m\trc'\left(2^{\frac{j}{2}}(N'-N_{\nu'})\right)^qF_{j,-1}(u')\eta_j^{\nu'}(\o')d\o'}_{L^2(\MM)}\\
\nn&\les& \ep 2^{\frac{j}{2}-m}\gamma^{\nu'}_j\normm{\int_{\S}G_2\left(2^{\frac{j}{2}}(N-N_\nu)\right)^pF_{j,-1}(u)\eta_j^\nu(\o)d\o}_{L^2(\MM)},
\eea
where we used in the last inequality the estimate \eqref{ldc28}.

Now, the basic estimate in $L^2(\MM)$ \eqref{oscl2bis} yields:
\bea\lab{ldc46}
&&\normm{\int_{\S} G_2\left(2^{\frac{j}{2}}(N-N_\nu)\right)^pF_{j,-1}(u)\eta_j^\nu(\o)d\o}_{L^2(\MM)}\\
\nn&\les& \left(\sup_\o\normm{G_2\left(2^{\frac{j}{2}}(N-N_\nu)\right)^p}_{\li{\infty}{2}}\right)2^{\frac{j}{2}}\gamma^\nu_j\\
\nn&\les&  \left(\sup_\o\normm{G_2}_{\li{\infty}{2}}\right) 2^{\frac{j}{2}}\gamma^\nu_j,
\eea
where we used in the last inequality the estimate \eqref{estNomega} for $\po N$ and the size of the patch. Next, we estimate $G_2$. In view of the definition \eqref{ldc13} of $G_2$, we have:
\bea\lab{ldc47}
\normm{G_2}_{\li{\infty}{2}}&\les& (\norm{\chi}_{\tx{\infty}{4}}+\norm{\db}_{\tx{\infty}{4}})\norm{N(P_l\trc)}_{\tx{2}{4}} \\
\nn&&+\norm{\z}_{\tx{\infty}{4}}\norm{\nabb P_l\trc}_{\tx{2}{4}}+(\norm{\dd_N(\chi)}_{\li{\infty}{2}}\\
\nn&&+\norm{N(\db)}_{\li{\infty}{2}}+\norm{\nabb(\z)\th}_{\li{\infty}{2}}) \norm{P_l\trc}_{L^\infty}\\
\nn&\les& \ep(\norm{N(P_l\trc)}_{\tx{2}{4}}+\norm{\nabb P_l\trc}_{\tx{2}{4}})+\ep,
\eea
where we used in the last inequality the embedding \eqref{sobineq1}, the estimates \eqref{esttrc} \eqref{esthch} for $\chi$, the estimates \eqref{estn} \eqref{estk} for $\db$, the estimate \eqref{estzeta} for $\z$, the estimates \eqref{estk} \eqref{esttrc} \eqref{esthch} for $\th$, the estimate \eqref{esttrc} for $\trc$ and the boundedness of $P_l$ on $L^\infty(\ptu)$.
Now, the estimate \eqref{estb} for $b$ and the Gagliardo-Nirenberg inequality \eqref{eq:GNirenberg},  yields:
\bee
&&\norm{N(P_l\trc)}_{\tx{2}{4}}+\norm{\nabb P_l\trc}_{\tx{2}{4}}\\
&\les& \norm{P_l(bN\trc)}_{\tx{2}{4}}+\norm{[bN,P_l]\trc}_{\tx{2}{4}}+\norm{\nabb^2 P_l\trc}_{\li{\infty}{2}}^{\frac{1}{2}}\norm{\nabb P_l\trc}_{\li{\infty}{2}}^{\frac{1}{2}}\\
&\les & 2^{\frac{l}{2}}\norm{bN\trc}_{\li{\infty}{2}}+\norm{\nabb[bN,P_l]\trc}_{\li{\infty}{2}}^{\frac{1}{2}}\norm{[bN,P_l]\trc}_{\li{\infty}{2}}^{\frac{1}{2}}+2^{\frac{l}{2}}\norm{\nabb\trc}_{\li{\infty}{2}},
\eee
where we used in the last inequality the Gagliardo-Nirenberg inequality \eqref{eq:GNirenberg}, the Bernstein inequality for $P_l$, the Bochner inequality \eqref{eq:Bochconseqbis}, and the finite band property for $P_l$. 
Together with the estimate \eqref{esttrc} for $\trc$ and the commutator estimate \eqref{commlp3bis} for $[bN,P_l]\trc$,  we obtain:
\be\lab{ldc48}
\norm{N(P_l\trc)}_{\tx{2}{4}}+\norm{\nabb P_l\trc}_{\tx{2}{4}}\les  2^{\frac{l}{2}}\ep.
\ee
Finally, \eqref{ldc45}, \eqref{ldc46}, \eqref{ldc47} and \eqref{ldc48} imply:
\be\lab{ldc49}
\norm{h_{6,p,q,l,m}}_{L^1(\MM)}\les \ep^2 2^{j+\frac{l}{2}-m}\gamma^\nu_j\gamma^{\nu'}_j.
\ee

\subsubsection{Control of the $L^1(\MM)$ norm of $h_{7,p,q,l,m}$}

In view of the definition \eqref{ldc8} of $h_{7,p,q,l,m}$, we have:
\bea\lab{ldc50}
\norm{h_{7,p,q,l,m}}_{L^1(\MM)}&\leq& \normm{\int_{\S}(N(P_l\trc)+\nabb(P_l\trc))\left(2^{\frac{j}{2}}(N-N_\nu)\right)^pF_{j,-1}(u)\eta_j^\nu(\o)d\o}_{L^2(\MM)}\\
\nn&&\times\normm{\int_{\S}(\chi'+\db'+\z')P_m\trc'\left(2^{\frac{j}{2}}(N'-N_{\nu'})\right)^qF_{j,-1}(u')\eta_j^{\nu'}(\o')d\o'}_{L^2(\MM)}.
\eea
The basic estimate in $L^2(\MM)$ \eqref{oscl2bis} yields:
\bea\lab{ldc52}
&&\normm{\int_{\S}(N(P_l\trc)+\nabb(P_l\trc))\left(2^{\frac{j}{2}}(N-N_\nu)\right)^pF_{j,-1}(u)\eta_j^\nu(\o)d\o}_{L^2(\MM)}\\
\nn&\les& \left(\sup_\o\normm{(N(P_l\trc)+\nabb(P_l\trc))\left(2^{\frac{j}{2}}(N-N_\nu)\right)^p}_{\li{\infty}{2}}\right) 2^{\frac{j}{2}}\gamma^\nu_j\\
\nn&\les& \left(\sup_\o(\norm{N(P_l\trc)}_{\li{\infty}{2}}+\norm{\nabb(P_l\trc)}_{\li{\infty}{2}})\normm{\left(2^{\frac{j}{2}}(N-N_\nu)\right)^p}_{L^\infty}\right) 2^{\frac{j}{2}}\gamma^\nu_j\\
\nn&\les& \ep 2^{\frac{j}{2}}\gamma^\nu_j,
\eea
where we used in the last inequality the estimate \eqref{paysdegalles} for $N(P_l\trc)$, the finite band property for $P_l$, the estimate \eqref{esttrc} for $\trc$, the estimate \eqref{estNomega} for $\po N$ and the size of the patch. 

Now, since $k$ does not depend on $\o$, and in view of the decomposition \eqref{dectrcom} \eqref{dechchom} for $\chi'$, and the decomposition \eqref{deczetaom} for $z'$, we have the following decomposition for $\chi'+\db'+\z'$:
$$\chi'+\db'+\z'=F^j_1+F^j_2$$
where the tensor $F^j_1$ only depends on $\nu'$ and satisfies:
$$\norm{F^j_1}_{L^\infty_{u_{\nu'}}L^2_t L^8_{x'_{\nu'}}}\les \ep,$$
and where the tensor $F^j_2$ satisfies:
$$\norm{F^j_2}_{L^\infty_{u'}L^2(\H_{u'}}\les \ep 2^{-\frac{j}{4}}.$$
Note that this decomposition has the same properties as the decomposition \eqref{vino10} \eqref{vino11} \eqref{vino12} for $\th+b^{-1}\nabb(b)$. Thus, arguing as in \eqref{vino19}-\eqref{vino24}, we obtain:
\be\lab{ldc51}
\normm{\int_{\S}(\chi'+\db'+\z')P_m\trc'\left(2^{\frac{j}{2}}(N'-N_{\nu'})\right)^qF_{j,-1}(u')\eta_j^{\nu'}(\o')d\o'}_{L^2(\MM)}\les 2^{\frac{j}{4}}\ep \gamma^{\nu'}_j.
\ee
Finally, \eqref{ldc50}, \eqref{ldc52} and \eqref{ldc51} imply:
\be\lab{ldc53}
\norm{h_{7,p,q,l,m}}_{L^1(\MM)}\leq \ep^2 2^{\frac{3j}{4}}\gamma^\nu_j\gamma^{\nu'}_j.
\ee

\subsubsection{Control of the $L^1(\MM)$ norm of $h_{8,p,q,l,m}$}

In view of the definition \eqref{ldc9} of $h_{8,p,q,l,m}$, we have:
\bea\lab{ldc54}
\norm{h_{8,p,q,l,m}}_{L^1(\MM)}&\leq& \normm{\int_{\S}(\th+b^{-1}\nabb(b))P_l\trc\left(2^{\frac{j}{2}}(N-N_\nu)\right)^pF_{j,-1}(u)\eta_j^\nu(\o)d\o}_{L^2(\MM)}\\
\nn&&\times\normm{\int_{\S}\z'P_m\trc'\left(2^{\frac{j}{2}}(N'-N_{\nu'})\right)^qF_{j,-1}(u')\eta_j^{\nu'}(\o')d\o'}_{L^2(\MM)}\\
\nn&\les& 2^{\frac{j}{4}}\ep \gamma^\nu_j\normm{\int_{\S}\z'P_m\trc'\left(2^{\frac{j}{2}}(N'-N_{\nu'})\right)^qF_{j,-1}(u')\eta_j^{\nu'}(\o')d\o'}_{L^2(\MM)},
\eea
where we used the estimate \eqref{vino24} in the last inequality. Also, we have the analog of \eqref{ldc51}:
\be\lab{ldc55}
\normm{\int_{\S}\z' P_m\trc'\left(2^{\frac{j}{2}}(N'-N_{\nu'})\right)^qF_{j,-1}(u')\eta_j^{\nu'}(\o')d\o'}_{L^2(\MM)}\les 2^{\frac{j}{4}}\ep \gamma^{\nu'}_j.
\ee
Finally, \eqref{ldc54} and \eqref{ldc55} imply:
\be\lab{ldc56}
\norm{h_{8,p,q,l,m}}_{L^1(\MM)}\les 2^{\frac{j}{2}}\ep^2 \gamma^\nu_j\gamma^{\nu'}_j.
\ee

\subsubsection{Control of the $L^1(\MM)$ norm of $h_{9,p,q,l,m}$}

In view of the definition \eqref{ldc10} of $h_{9,p,q,l,m}$, we have:
\bee
\norm{h_{9,p,q,l,m}}_{L^1(\MM)}&\leq& \normm{\int_{\S}P_l\trc\left(2^{\frac{j}{2}}(N-N_\nu)\right)^pF_{j,-1}(u)\eta_j^\nu(\o)d\o}_{L^2(\MM)}\\
\nn&&\times\normm{\int_{\S}\nabb'(N'(P_m\trc'))\left(2^{\frac{j}{2}}(N'-N_{\nu'})\right)^qF_{j,-1}(u')\eta_j^{\nu'}(\o')d\o'}_{L^2(\MM)}\\
\nn&\les& \left(\int_{\S}\normm{P_l\trc\left(2^{\frac{j}{2}}(N-N_\nu)\right)^pF_{j,-1}(u)}_{L^2(\MM)}\eta_j^\nu(\o)d\o\right)\\
\nn&&\times\left(\int_{\S}\normm{\nabb'(N'(P_m\trc'))\left(2^{\frac{j}{2}}(N'-N_{\nu'})\right)^qF_{j,-1}(u')}_{L^2(\MM)}\eta_j^{\nu'}(\o')d\o'\right).
\eee
Together with the estimate \eqref{estNomega} for $\po N$ and the size of the patch, we obtain:
\bee
\norm{h_{9,p,q,l,m}}_{L^1(\MM)}&\les& \left(\int_{\S}\normm{\norm{P_l\trc}_{\lh{2}}F_{j,-1}(u)}_{L^2_u}\eta_j^\nu(\o)d\o\right)\\
\nn&&\times\left(\int_{\S}\normm{\norm{\nabb'(N'(P_m\trc'))}_{L^2(\H_{u'})} F_{j,-1}(u')}_{L^2_{u'}}\eta_j^{\nu'}(\o')d\o'\right)\\
\nn&\les& 2^{-j}\normm{\norm{P_l\trc}_{\lh{2}}F_{j,-1}(u)\sqrt{\eta_j^\nu(\o)}}_{L^2_{u,\o}}\\
\nn&&\times\normm{\norm{\nabb'(N'(P_m\trc'))}_{L^2(\H_{u'})} F_{j,-1}(u')\sqrt{\eta_j^{\nu'}(\o')}}_{L^2_{u',\o'}},
\eee
where we used in the last inequality Cauchy Schwarz in $\o$ and $\o'$, and the size of the patch. This yields:
\bea\lab{ldc57}
&&\sum_{(l,m)/m\leq l}\norm{h_{9,p,q,l,m}}_{L^1(\MM)}\\
\nn&\les& 2^{-j}\sum_{(l,m)}2^{-|l-m|}\left(2^l\normm{\norm{P_l\trc}_{\lh{2}}F_{j,-1}(u)\sqrt{\eta_j^\nu(\o)}}_{L^2_{u,\o}}\right)\\
\nn&&\times\left(2^{-m}\normm{\norm{\nabb'(N'(P_m\trc'))}_{L^2(\H_{u'})} F_{j,-1}(u')\sqrt{\eta_j^{\nu'}(\o')}}_{L^2_{u',\o'}}\right)\\
\nn&\les& 2^{-j}\left(\sum_l 2^{2l}\normm{\norm{P_l\trc}_{\lh{2}}F_{j,-1}(u)\sqrt{\eta_j^\nu(\o)}}^2_{L^2_{u,\o}}\right)^{\frac{1}{2}}\\
\nn&&\times\left(\sum_m 2^{-2m}\normm{\norm{\nabb'(N'(P_m\trc'))}_{L^2(\H_{u'})} F_{j,-1}(u')\sqrt{\eta_j^{\nu'}(\o')}}^2_{L^2_{u',\o'}}\right)^{\frac{1}{2}}.
\eea
Now, we have:
\bea\lab{ldc58}
&& \sum_l\normm{2^l\norm{P_l\trc}_{L^2(\H_u)} F_{j,-1}(u)\sqrt{\eta_j^\nu(\o)}}_{L^2_{\o, u}}^2\\
\nn&=& \int_{\S}\left(\int_u\left(\sum_l2^{2l}\norm{P_l\trc}_{L^2(\H_u)}^2\right)|F_j(u)|^2du\right)\eta^\nu_j(\o)d\o\\
\nn&\les & \int_{\S}\norm{\nabb\trc}_{\li{\infty}{2}}^2\norm{F_j(u)}_{L^2_u}^2\eta^\nu_j(\o)d\o\\
\nn&\les & \ep^22^{2j}(\gamma^\nu_j)^2,
\eea
where we used the finite band property for $P_l$, the estimates \eqref{esttrc} for $\trc$ and  Plancherel in $\la$. Also, 
in view of the estimate \eqref{estb} for $b'$, we have:
\bee
&&\norm{\nabb'(N'(P_m\trc'))}_{L^2(\H_{u'})}\\
&\les& \norm{b'\nabb'(N'(P_m\trc'))}_{L^2(\H_{u'})}\\
&\les& \norm{\nabb'(b')}_{\tx{\infty}{4}} \norm{N'(P_m\trc')}_{\tx{2}{4}}+\norm{\nabb'(P_m(b'N'\trc'))}_{L^2(\H_{u'})}\\
&&+\norm{\nabb'([b'N',P_m]\trc')}_{\lprime{\infty}{2}}.
\eee
Together with the estimate \eqref{estb} for $b$, the estimate \eqref{ldc48} for $N'(P_m\trc')$, the finite band property for $P_m$, and the commutator estimate \eqref{commlp3bis}, we obtain:
\be\lab{ldc59}
\norm{\nabb'(N'(P_m\trc'))}_{L^2(\H_{u'})}\les 2^{\frac{m}{2}}\ep+2^m\norm{P_m(b'N'\trc')}_{L^2(\H_{u'})}.
\ee
This yields:
\bea\lab{ldc60}
&& \sum_m 2^{-2m}\normm{\norm{\nabb'(N'(P_m\trc'))}_{L^2(\H_{u'})} F_{j,-1}(u')\sqrt{\eta_j^{\nu'}(\o')}}_{L^2_{\o', u'}}^2\\
\nn&=& \int_{\S}\left(\int_{u'}\left(\sum_m\left(\norm{P_m(b'N'\trc')}_{L^2(\H_{u'})}+2^{-\frac{m}{2}}\ep\right)^2\right)|F_j(u')|^2du'\right)\eta^{\nu'}_j(\o')d\o'\\
\nn&\les & \int_{\S}\left(\norm{b'N'\trc}_{\lprime{\infty}{2}}^2+\ep^2\right)\norm{F_j(u')}_{L^2_{u'}}^2\eta^{\nu'}_j(\o')d\o'\\
\nn&\les & \ep^22^{2j}(\gamma^{\nu'}_j)^2,
\eea
where we used the finite band property for $P_m$, the estimates \eqref{esttrc} for $\trc'$ and  Plancherel in $\la'$.

Finally, \eqref{ldc57}, \eqref{ldc58} and \eqref{ldc60} yield:
\be\lab{ldc61}
\sum_{(l,m)/m\leq l}\norm{h_{9,p,q,l,m}}_{L^1(\MM)}\les 2^j\ep^2\gamma^\nu_j\gamma^{\nu'}_j.
\ee

\subsubsection{Control of the $L^1(\MM)$ norm of $h_{10,p,q,l,m}$}

In view of the definition \eqref{ldc11} of $h_{10,p,q,l,m}$, we have:
\bea\lab{ldc62}
&&\norm{h_{10,p,q,l,m}}_{L^1(\MM)}\\
\nn&\leq& \normm{\int_{\S}P_l\trc\left(2^{\frac{j}{2}}(N-N_\nu)\right)^pF_{j,-1}(u)\eta_j^\nu(\o)d\o}_{L^2(\MM)}\\
\nn&&\times\normm{\int_{\S}{b'}^{-1}\nabb(b') N'(P_m\trc')\left(2^{\frac{j}{2}}(N'-N_{\nu'})\right)^qF_{j,-1}(u')\eta_j^{\nu'}(\o')d\o'}_{L^2(\MM)}.
\eea
We estimate the first term in the right-hand side of \eqref{ldc62}. Using \eqref{gini9} for $l>j/2$ and \eqref{nyc41} for $l=j/2$, we obtain for all $l\geq j/2$:
\be\lab{ldc63}
\normm{\int_{\S} P_l\trc\left(2^{\frac{j}{2}}(N-N_\nu)\right)^pF_{j,-1}(u)\eta_j^\nu(\o)d\o}_{L^2(\MM)}\les (1+p^2)2^{-l+\frac{j}{2}}\ep\gamma^\nu_j.
\ee

Next, we estimate the second term in the right-hand side of \eqref{ldc62}. The basic estimate in $L^2(\MM)$ \eqref{oscl2bis} yields:
\bea\lab{ldc64}
&&\normm{\int_{\S}{b'}^{-1}\nabb(b') N'(P_m\trc')\left(2^{\frac{j}{2}}(N'-N_{\nu'})\right)^qF_{j,-1}(u')\eta_j^{\nu'}(\o')d\o'}_{L^2(\MM)}\\
\nn&\les& \left(\sup_{\o'}\normm{{b'}^{-1}\nabb(b') N'(P_m\trc')\left(2^{\frac{j}{2}}(N'-N_{\nu'})\right)^q}_{\lprime{\infty}{2}}\right) 2^{\frac{j}{2}}\gamma^{\nu'}_j\\
\nn&\les& \left(\sup_{\o'}\normm{{b'}^{-1}\nabb(b') N'(P_m\trc')}_{\lprime{\infty}{2}}\right) 2^{\frac{j}{2}}\gamma^{\nu'}_j
\eea
Now, we have:
\bee
\norm{{b'}^{-1}\nabb(b') N'(P_m\trc')}_{\lprime{\infty}{2}}&\les& \norm{{b'}^{-1}\nabb(b')}_{\tx{\infty}{4}}\norm{N'(P_m\trc')}_{\tx{2}{4}}\\
&\les &  \ep 2^{\frac{m}{2}},
\eee
where we used in the last inequality the estimate \eqref{estb} for $b$ and the estimate \eqref{ldc48} for $N'(P_m\trc')$. Together with \eqref{ldc64}, we obtain:
\be\lab{ldc65}
\normm{\int_{\S}{b'}^{-1}\nabb(b') N'(P_m\trc')\left(2^{\frac{j}{2}}(N'-N_{\nu'})\right)^qF_{j,-1}(u')\eta_j^{\nu'}(\o')d\o'}_{L^2(\MM)}\les \ep  2^{\frac{j}{2}+\frac{m}{2}}\gamma^{\nu'}_j.
\ee
Finally, \eqref{ldc62}, \eqref{ldc63} and \eqref{ldc65} imply:
\be\lab{ldc66}
\norm{h_{10,p,q,l,m}}_{L^1(\MM)}\les (1+p^2)2^{j-l+\frac{m}{2}}\ep^2\gamma^\nu_j\gamma^{\nu'}_j.
\ee

\subsubsection{Control of $B^{2,1}_{j,\nu,\nu',l,m}$}

Recall the definition \eqref{ldc14} of $B^{2,1}_{j,\nu,\nu',l,m}$:
\bee
B^{2,1}_{j,\nu,\nu',l,m}&=&2^{-2j}\int_{\MM}\left(\int_{\S} N(P_l\trc)F_{j,-1}(u)\eta_j^\nu(\o)d\o\right)\\
\nn&&\times\left(\int_{\S}N'(P_m\trc')F_{j,-1}(u')\eta_j^{\nu'}(\o')d\o'\right)d\MM,
\eee
Recall that we are considering the range of $(l,m)$:
$$2^m\leq 2^l\leq 2^j|\nu-\nu'|.$$
Summing in $(l,m)$, we have:
$$\sum_{(l,m)/2^m\leq 2^l\leq 2^j|\nu-\nu'|}N(P_l\trc)N'(P_m\trc')= N(P_{\leq 2^j|\nu-\nu'|}\trc)N'(P_{\leq 2^j|\nu-\nu'|}\trc').$$
Thus, using the symmetry in $(\o, \o')$ of the integrant in $B^{2,1}_{j,\nu,\nu',l,m}$, we obtain in view of the definition  $B^{2,1}_{j,\nu,\nu',l,m}$:
\bee
&&\sum_{(l,m)/2^{\max(l,m)}l\leq 2^j|\nu-\nu'|}(B^{2,1}_{j,\nu',\nu,l,m}+B^{2,1}_{j,\nu,\nu',l,m})\\
\nn&=& 2^{-2j}\int_{\MM}\left(\int_{\S}N(P_{\leq 2^j|\nu-\nu'|}\trc)F_{j,-1}(u)\eta_j^\nu(\o)d\o\right)\\
\nn&&\times\left(\int_{\S}N'(P_{\leq 2^j|\nu-\nu'|}\trc')F_{j,-1}(u')\eta_j^{\nu'}(\o')d\o'\right)d\MM,
\eee
which yields:
\bea\lab{ldc67oo}
&&\left|\sum_{(l,m)/2^{\max(l,m)}l\leq 2^j|\nu-\nu'|}(B^{2,1}_{j,\nu',\nu,l,m}+B^{2,1}_{j,\nu,\nu',l,m})\right|\\
\nn&\leq& 2^{-2j}\normm{\int_{\S}N(P_{\leq 2^j|\nu-\nu'|}\trc)F_{j,-1}(u)\eta_j^\nu(\o)d\o}_{L^2(\MM)}\\
\nn&&\times\normm{\int_{\S}N'(P_{\leq 2^j|\nu-\nu'|}\trc')F_{j,-1}(u')\eta_j^{\nu'}(\o')d\o'}_{L^2(\MM)}.
\eea
Together with the estimate \eqref{vino56} applied to both terms in the right-hand side of \eqref{ldc67oo}, we finally obtain:
\be\lab{ldc67}
\left|\sum_{(l,m)/2^{\max(l,m)}l\leq 2^j|\nu-\nu'|}(B^{2,1}_{j,\nu',\nu,l,m}+B^{2,1}_{j,\nu,\nu',l,m})\right|\les 2^{-j}\ep^2\gamma^\nu_j\gamma^{\nu'}_j.
\ee

\subsubsection{Control of $B^{2,2}_{j,\nu,\nu',l,m}$}

Recall the definition \eqref{ldc15} of $B^{2,2}_{j,\nu,\nu',l,m}$:
\bee
B^{2,2}_{j,\nu,\nu',l,m}&=&2^{-2j}\int_{\MM}\int_{\S\times\S}\frac{(N'-\gn N)(P_l\trc)N'(P_m\trc')}{1-\gn^2}\\
\nn&&\times F_{j,-1}(u)F_{j,-1}(u')\eta_j^\nu(\o)\eta_j^{\nu'}(\o')d\o d\o'd\MM.
\eee
We sum for $m\leq l$, and we obtain:
\bea\lab{duc0}
\sum_{m\leq l}B^{2,2}_{j,\nu,\nu',l,m}&=&2^{-2j}\int_{\MM}\int_{\S\times\S}\frac{(N'-\gn N)(P_l\trc)N'(P_{\leq l}\trc')}{1-\gn^2}\\
\nn&&\times F_{j,-1}(u)F_{j,-1}(u')\eta_j^\nu(\o)\eta_j^{\nu'}(\o')d\o d\o'd\MM.
\eea

We can not estimate $\sum_{m\leq l}B^{2,2}_{j,\nu,\nu',l,m}$ directly due to a lack of summability in $l$. Instead, we integrate by parts in tangential directions using \eqref{fete1}. 
\begin{lemma}\lab{lemma:duc}
Let $\sum_{m\leq l}B^{2,2}_{j,\nu,\nu',l,m}$ be defined by \eqref{duc0}. Integrating by parts using \eqref{fete1} yields:
\bea\lab{duc1}
\sum_{m\leq l}B^{2,2}_{j,\nu,\nu',l,m}&=& 2^{-\frac{5j}{2}}\sum_{p, q\geq 0}c_{pq}\int_{\MM}\frac{1}{(2^{\frac{j}{2}}|N_\nu-N_{\nu'}|)^{p+q+1}}\\
\nn&&\times\bigg[\frac{1}{|N_\nu-N_{\nu'}|^2}(h_{1,p,q,l,m}'+h_{2,p,q,l,m}')+\frac{1}{|N_\nu-N_{\nu'}|}(h_{3,p,q,l,m}'+h_{4,p,q,l,m}'\\
\nn&&+h_{5,p,q,l,m}'+h_{6,p,q,l,m}')+h_{7,p,q,l,m}'+h_{8,p,q,l,m}'\bigg] d\MM,
\eea
where $c_{pq}$ are explicit real coefficients such that the series 
$$\sum_{p, q\geq 0}c_{pq}x^py^q$$
has radius of convergence 1, where the scalar functions $h_{1,p,q,l,m}'$, $h_{2,p,q,l,m}'$, $h_{3,p,q,l,m}'$, $h_{4,p,q,l,m}'$, $h_{5,p,q,l,m}'$, $h_{6,p,q,l,m}'$, $h_{7,p,q,l,m}'$, $h_{8,p,q,l,m}'$ on $\MM$ are given by:
\bea\lab{duc2}
h_{1,p,q,l,m}'&=& \left(\int_{\S}\chi \nabla(P_l\trc) \left(2^{\frac{j}{2}}(N-N_\nu)\right)^pF_{j,-1}(u)\eta_j^\nu(\o)d\o\right)\\
\nn&&\times\left(\int_{\S}N'(P_{\leq l}\trc')\left(2^{\frac{j}{2}}(N'-N_{\nu'})\right)^qF_{j,-1}(u')\eta_j^{\nu'}(\o')d\o'\right),
\eea
\bea\lab{duc3}
h_{2,p,q,l,m}'&=& \left(\int_{\S}\nabla(P_l\trc)\left(2^{\frac{j}{2}}(N-N_\nu)\right)^pF_{j,-1}(u)\eta_j^\nu(\o)d\o\right)\\
\nn&&\times\left(\int_{\S}\chi' N'(P_{\leq l}\trc') \left(2^{\frac{j}{2}}(N'-N_{\nu'})\right)^qF_{j,-1}(u')\eta_j^{\nu'}(\o')d\o'\right),
\eea
\bea\lab{duc4}
h_{3,p,q,l,m}'&=& \left(\int_{\S}\nabb^2(P_l\trc)\left(2^{\frac{j}{2}}(N-N_\nu)\right)^pF_{j,-1}(u)\eta_j^\nu(\o)d\o\right)\\
\nn&&\times\left(\int_{\S}N'(P_{\leq l}\trc')\left(2^{\frac{j}{2}}(N'-N_{\nu'})\right)^qF_{j,-1}(u')\eta_j^{\nu'}(\o')d\o'\right),
\eea
\bea\lab{duc5}
h_{4,p,q,l,m}'&=& \left(\int_{\S} \nabb(P_l\trc)\left(2^{\frac{j}{2}}(N-N_\nu)\right)^pF_{j,-1}(u)\eta_j^\nu(\o)d\o\right)\\
\nn&&\times\left(\int_{\S}\nabb'(N'(P_{\leq l}\trc'))\left(2^{\frac{j}{2}}(N'-N_{\nu'})\right)^qF_{j,-1}(u')\eta_j^{\nu'}(\o')d\o'\right),
\eea
\bea\lab{duc6}
h_{5,p,q,l,m}'&=& \left(\int_{\S} (\th+b^{-1}\nabb(b))\nabla(P_l\trc)\left(2^{\frac{j}{2}}(N-N_\nu)\right)^pF_{j,-1}(u)\eta_j^\nu(\o)d\o\right)\\
\nn&&\times\left(\int_{\S}N'(P_{\leq l}\trc')\left(2^{\frac{j}{2}}(N'-N_{\nu'})\right)^qF_{j,-1}(u')\eta_j^{\nu'}(\o')d\o'\right),
\eea
\bea\lab{duc7}
h_{6,p,q,l,m}'&=& \left(\int_{\S}\nabla(P_l\trc)\left(2^{\frac{j}{2}}(N-N_\nu)\right)^pF_{j,-1}(u)\eta_j^\nu(\o)d\o\right)\\
\nn&&\times\left(\int_{\S}({b'}^{-1}\nabb'(b')+\th')N'(P_{\leq l}\trc')\left(2^{\frac{j}{2}}(N'-N_{\nu'})\right)^qF_{j,-1}(u')\eta_j^{\nu'}(\o')d\o'\right),
\eea
\bea\lab{duc8}
h_{7,p,q,l,m}'&=& \left(\int_{\S}\nabb(N(P_l\trc))\left(2^{\frac{j}{2}}(N-N_\nu)\right)^pF_{j,-1}(u)\eta_j^\nu(\o)d\o\right)\\
\nn&&\times\left(\int_{\S}N'(P_{\leq l}\trc')\left(2^{\frac{j}{2}}(N'-N_{\nu'})\right)^qF_{j,-1}(u')\eta_j^{\nu'}(\o')d\o'\right),
\eea
and:
\bea\lab{duc9}
h_{8,p,q,l,m}'&=& \left(\int_{\S}(\th+N(b))\nabla(P_l\trc)\left(2^{\frac{j}{2}}(N-N_\nu)\right)^pF_{j,-1}(u)\eta_j^\nu(\o)d\o\right)\\
\nn&&\times\left(\int_{\S}N'(P_m\trc')\left(2^{\frac{j}{2}}(N'-N_{\nu'})\right)^qF_{j,-1}(u')\eta_j^{\nu'}(\o')d\o'\right).
\eea
\end{lemma}
 
The proof of Lemma \ref{lemma:duc} is postponed to Appendix H. In the rest of this section, we use Lemma \ref{lemma:duc} to control $\sum_{m\leq l}B^{2,2}_{j,\nu,\nu',l,m}$ over the range of $(l,m)$ such that $2^m\leq 2^l\leq 2^j|\nu-\nu'|$.

Next, we evaluate the $L^1(\MM)$ norm of $h_{1,p,q,l,m}'$, $h_{2,p,q,l,m}'$, $h_{3,p,q,l,m}'$, $h_{4,p,q,l,m}'$, $h_{5,p,q,l,m}'$, $h_{6,p,q,l,m}'$, $h_{7,p,q,l,m}'$, $h_{8,p,q,l,m}'$ starting with $h_{1,p,q,l,m}'$. In view of the definition \eqref{duc2}, we have:
\bea\lab{duc10}
\norm{h_{1,p,q,l,m}'}_{L^1(\MM)}&\leq & \normm{\int_{\S}\chi \nabla(P_l\trc) \left(2^{\frac{j}{2}}(N-N_\nu)\right)^pF_{j,-1}(u)\eta_j^\nu(\o)d\o}_{L^2(\MM)}\\
\nn&&\times\normm{\int_{\S}N'(P_{\leq l}\trc')\left(2^{\frac{j}{2}}(N'-N_{\nu'})\right)^qF_{j,-1}(u')\eta_j^{\nu'}(\o')d\o'}_{L^2(\MM)}\\
\nn&\les& 2^{\frac{j}{2}}\ep\gamma^{\nu'}_j\normm{\int_{\S}\chi \nabla(P_l\trc) \left(2^{\frac{j}{2}}(N-N_\nu)\right)^pF_{j,-1}(u)\eta_j^\nu(\o)d\o}_{L^2(\MM)},
\eea
where we used in the last inequality the estimate \eqref{vino56}. Now, the basic estimate in $L^2(\MM)$ \eqref{oscl2bis} yields:
\bea\lab{duc11}
&&\normm{\int_{\S}\chi \nabla(P_l\trc) \left(2^{\frac{j}{2}}(N-N_\nu)\right)^pF_{j,-1}(u)\eta_j^\nu(\o)d\o}_{L^2(\MM)}\\
\nn&\les& \left(\sup_\o\normm{\chi \nabla(P_l\trc) \left(2^{\frac{j}{2}}(N-N_\nu)\right)^p}_{\li{\infty}{2}}\right)2^{\frac{j}{2}}\gamma^\nu_j\\
\nn&\les& \left(\sup_\o\norm{\chi}_{\tx{\infty}{4}}\norm{\nabla(P_l\trc)}_{\tx{2}{4}}\normm{\left(2^{\frac{j}{2}}(N-N_\nu)\right)^p}_{L^\infty}\right)2^{\frac{j}{2}}\gamma^\nu_j\\
\nn&\les& \ep 2^{\frac{l}{2}+\frac{j}{2}}\gamma^\nu_j,
\eea
where we used in the last inequality the estimates \eqref{esttrc} \eqref{esthch} for $\chi$, the estimate \eqref{ldc48} for $\nabla(P_l\trc)$, the estimate \eqref{estNomega} for $\po N$ and the size of the patch. In view of \eqref{duc10} and \eqref{duc11}, we obtain:
\be\lab{duc12}
\norm{h_{1,p,q,l,m}'}_{L^1(\MM)}\les \ep^2 2^{j+\frac{l}{2}}\gamma^\nu_j\gamma^{\nu'}_j.
\ee

Next, we evaluate the $L^1(\MM)$ norm of $h_{2,p,q,l,m}'$. In view of the definition \eqref{duc3} of $h_{2,p,q,l,m}'$, we have:
\bea\lab{duc13}
\norm{h_{2,p,q,l,m}'}_{L^1(\MM)}&\leq& \normm{\int_{\S}\nabla(P_l\trc)\left(2^{\frac{j}{2}}(N-N_\nu)\right)^pF_{j,-1}(u)\eta_j^\nu(\o)d\o}_{L^2(\MM)}\\
\nn&&\times\normm{\int_{\S}\chi' N'(P_{\leq l}\trc') \left(2^{\frac{j}{2}}(N'-N_{\nu'})\right)^qF_{j,-1}(u')\eta_j^{\nu'}(\o')d\o'}_{L^2(\MM)}\\
\nn&\les& \ep^2 2^{j+\frac{l}{2}}\gamma^\nu_j\gamma^{\nu'}_j,
\eea
where we used in the last inequality the analog of the estimate \eqref{vino56} for the first term and the analog of the estimate \eqref{duc11} for the second term. 

Next, we evaluate the $L^1(\MM)$ norm of $h_{3,p,q,l,m}'$. In view of the definition \eqref{duc4} of $h_{3,p,q,l,m}'$, we have:
\bea\lab{duc14}
\norm{h_{3,p,q,l,m}'}_{L^1(\MM)}&\leq& \normm{\int_{\S}\nabb^2(P_l\trc)\left(2^{\frac{j}{2}}(N-N_\nu)\right)^pF_{j,-1}(u)\eta_j^\nu(\o)d\o}_{L^2(\MM)}\\
\nn&&\times\normm{\int_{\S}N'(P_{\leq l}\trc')\left(2^{\frac{j}{2}}(N'-N_{\nu'})\right)^qF_{j,-1}(u')\eta_j^{\nu'}(\o')d\o'}_{L^2(\MM)}\\
\nn&\les& \ep 2^{\frac{j}{2}}\gamma^{\nu'}_j\normm{\int_{\S}\nabb^2(P_l\trc)\left(2^{\frac{j}{2}}(N-N_\nu)\right)^pF_{j,-1}(u)\eta_j^\nu(\o)d\o}_{L^2(\MM)},
\eea
where we used in the last inequality the estimate \eqref{vino56}. Now, the basic estimate in $L^2(\MM)$ yields:
\bea\lab{duc15oo}
&&\normm{\int_{\S}\nabb^2(P_l\trc)\left(2^{\frac{j}{2}}(N-N_\nu)\right)^pF_{j,-1}(u)\eta_j^\nu(\o)d\o}_{L^2(\MM)}\\
\nn&\les& \left(\sup_\o\normm{\nabb^2(P_l\trc)\left(2^{\frac{j}{2}}(N-N_\nu)\right)^p}_{\li{\infty}{2}}\right)2^{\frac{j}{2}}\gamma^\nu_j\\
\nn&\les& \ep 2^{l+\frac{j}{2}}\gamma^\nu_j,
\eea
where we used in the last inequality the estimate \eqref{vino78} for $\nabb^2(P_l\trc)$, the estimate \eqref{estNomega} for $\po N$ and the size of the patch. Finally, \eqref{duc14} and \eqref{duc15oo} yield:
\be\lab{duc15}
\norm{h_{3,p,q,l,m}'}_{L^1(\MM)}\les  \ep^2 2^{j+l}\gamma^\nu_j\gamma^{\nu'}_j.
\ee

Next, we evaluate the $L^1(\MM)$ norm of $h_{4,p,q,l,m}'$. In view of the definition \eqref{duc5} of $h_{4,p,q,l,m}'$, we have:
\bea\lab{duc16}
\norm{h_{4,p,q,l,m}'}_{L^1(\MM)}&\leq& \normm{\int_{\S} \nabb(P_l\trc)\left(2^{\frac{j}{2}}(N-N_\nu)\right)^pF_{j,-1}(u)\eta_j^\nu(\o)d\o}_{L^2(\MM)}\\
\nn&&\times\normm{\int_{\S}\nabb'(N'(P_{\leq l}\trc'))\left(2^{\frac{j}{2}}(N'-N_{\nu'})\right)^qF_{j,-1}(u')\eta_j^{\nu'}(\o')d\o'}_{L^2(\MM)}\\
\nn&\les& \ep 2^{\frac{j}{2}}\gamma^\nu_j\normm{\int_{\S}\nabb'(N'(P_{\leq l}\trc'))\left(2^{\frac{j}{2}}(N'-N_{\nu'})\right)^qF_{j,-1}(u')\eta_j^{\nu'}(\o')d\o'}_{L^2(\MM)},
\eea
where we used in the last inequality the analog of estimate \eqref{vino56}. Now, the basic estimate in $L^2(\MM)$ yields:
\bea\lab{duc17}
&&\normm{\int_{\S}\nabb'(N'(P_{\leq l}\trc'))\left(2^{\frac{j}{2}}(N'-N_{\nu'})\right)^qF_{j,-1}(u')\eta_j^{\nu'}(\o')d\o'}_{L^2(\MM)}\\
\nn&\les& \left(\sup_{\o'}\normm{\nabb'(N'(P_{\leq l}\trc'))\left(2^{\frac{j}{2}}(N'-N_{\nu'})\right)^q}_{\lprime{\infty}{2}}\right)2^{\frac{j}{2}}\gamma^{\nu'}_j\\
\nn&\les& \ep 2^{l+\frac{j}{2}}\gamma^{\nu'}_j,
\eea
where we used in the last inequality the estimate \eqref{vino81} for $\nabb'(N'(P_{\leq l}\trc'))$, the estimate \eqref{estNomega} for $\po N$ and the size of the patch. Finally, \eqref{duc16} and \eqref{duc17} yield:
\be\lab{duc18}
\norm{h_{4,p,q,l,m}'}_{L^1(\MM)}\les \ep^2 2^{j+l}\gamma^\nu_j\gamma^{\nu'}_j.
\ee

Next, we evaluate the $L^1(\MM)$ norm of $h_{5,p,q,l,m}'$. In view of the definition \eqref{duc6} of $h_{5,p,q,l,m}'$, we have:
\bea\lab{duc19}
&&\norm{h_{5,p,q,l,m}'}_{L^1(\MM)}\\
\nn&\leq& \normm{\int_{\S} (\th+b^{-1}\nabb(b))\nabla(P_l\trc)\left(2^{\frac{j}{2}}(N-N_\nu)\right)^pF_{j,-1}(u)\eta_j^\nu(\o)d\o}_{L^2(\MM)}\\
\nn&&\times\normm{\int_{\S}N'(P_{\leq l}\trc')\left(2^{\frac{j}{2}}(N'-N_{\nu'})\right)^qF_{j,-1}(u')\eta_j^{\nu'}(\o')d\o'}_{L^2(\MM)}\\
\nn&\les& \ep 2^{\frac{j}{2}}\gamma^\nu_j\normm{\int_{\S} (\th+b^{-1}\nabb(b))\nabla(P_l\trc)\left(2^{\frac{j}{2}}(N-N_\nu)\right)^pF_{j,-1}(u)\eta_j^\nu(\o)d\o}_{L^2(\MM)},
\eea
where we used in the last inequality the estimate \eqref{vino56}. Proceeding as for the estimate of \eqref{duc11} and using the estimates \eqref{estb} for $b$ and \eqref{estk} \eqref{esttrc} \eqref{esthch} for $\th$, we obtain the following estimate:
\be\lab{duc20}
 \normm{\int_{\S} (\th+b^{-1}\nabb(b))\nabla(P_l\trc)\left(2^{\frac{j}{2}}(N-N_\nu)\right)^pF_{j,-1}(u)\eta_j^\nu(\o)d\o}_{L^2(\MM)}\les \ep 2^{\frac{l}{2}+\frac{j}{2}}\gamma^\nu_j.
\ee
Finally, \eqref{duc19} and \eqref{duc20} imply:
\be\lab{duc21}
\norm{h_{5,p,q,l,m}'}_{L^1(\MM)}\les \ep^2 2^{j+\frac{l}{2}}\gamma^\nu_j\gamma^{\nu'}_j.
\ee

Next, we evaluate the $L^1(\MM)$ norm of $h_{6,p,q,l,m}'$. In view of the definition \eqref{duc7} of $h_{6,p,q,l,m}'$, we have:
\bea\lab{duc22}
&&\norm{h_{6,p,q,l,m}'}_{L^1(\MM)}\\
\nn&\leq& \normm{\int_{\S}\nabla(P_l\trc)\left(2^{\frac{j}{2}}(N-N_\nu)\right)^pF_{j,-1}(u)\eta_j^\nu(\o)d\o}_{L^2(\MM)}\\
\nn&&\times\normm{\int_{\S}({b'}^{-1}\nabb'(b')+\th')N'(P_{\leq l}\trc')\left(2^{\frac{j}{2}}(N'-N_{\nu'})\right)^qF_{j,-1}(u')\eta_j^{\nu'}(\o')d\o'}_{L^2(\MM)}\\
\nn&\les& \ep^2 2^{j+\frac{l}{2}}\gamma^\nu_j\gamma^{\nu'}_j,
\eea
where we used in the last inequality the analog of the estimate \eqref{vino56} for the first term, and the analog of the estimate \eqref{duc20} for the second term.

Next, we evaluate the $L^1(\MM)$ norm of $h_{7,p,q,l,m}'$. In view of the definition \eqref{duc8} of $h_{7,p,q,l,m}'$, we have:
\bea\lab{duc23}
\norm{h_{7,p,q,l,m}'}_{L^1(\MM)}&\leq& \normm{\int_{\S}\nabb(N(P_l\trc))\left(2^{\frac{j}{2}}(N-N_\nu)\right)^pF_{j,-1}(u)\eta_j^\nu(\o)d\o}_{L^2(\MM)}\\
\nn&&\times\normm{\int_{\S}N'(P_{\leq l}\trc')\left(2^{\frac{j}{2}}(N'-N_{\nu'})\right)^qF_{j,-1}(u')\eta_j^{\nu'}(\o')d\o'}_{L^2(\MM)}\\
\nn&\les& \ep^2 2^{j+l}\gamma^\nu_j\gamma^{\nu'}_j,
\eea
where we used in the last inequality the analog of the estimate \eqref{duc17} for the first term, and the estimate \eqref{vino56} for the second term.

Next, we evaluate the $L^1(\MM)$ norm of $h_{8,p,q,l,m}'$. In view of the definition \eqref{duc9} of $h_{8,p,q,l,m}'$, we have:
\bea\lab{duc24}
&&\norm{h_{8,p,q,l,m}'}_{L^1(\MM)}\\
\nn&\leq& \normm{\int_{\S}(\th+N(b))\nabla(P_l\trc)\left(2^{\frac{j}{2}}(N-N_\nu)\right)^pF_{j,-1}(u)\eta_j^\nu(\o)d\o}_{L^2(\MM)}\\
\nn&&\times\normm{\int_{\S}N'(P_m\trc')\left(2^{\frac{j}{2}}(N'-N_{\nu'})\right)^qF_{j,-1}(u')\eta_j^{\nu'}(\o')d\o'}_{L^2(\MM)}\\
\nn&\les& \ep 2^{\frac{j}{2}}\gamma^{\nu'}_j\normm{\int_{\S}(\th+N(b))\nabla(P_l\trc)\left(2^{\frac{j}{2}}(N-N_\nu)\right)^pF_{j,-1}(u)\eta_j^\nu(\o)d\o}_{L^2(\MM)},
\eea
where we used in the last inequality the analog of the estimate \eqref{vino56}. Proceeding as for the estimate of \eqref{duc11} and using the estimates \eqref{estb} for $b$ and \eqref{estk} \eqref{esttrc} \eqref{esthch} for $\th$, we obtain the following estimate:
\be\lab{duc25}
\normm{\int_{\S}(\th+N(b))\nabla(P_l\trc)\left(2^{\frac{j}{2}}(N-N_\nu)\right)^pF_{j,-1}(u)\eta_j^\nu(\o)d\o}_{L^2(\MM)}  \les \ep 2^{\frac{l}{2}+\frac{j}{2}}\gamma^\nu_j.
\ee
Finally, \eqref{duc24} and \eqref{duc25} imply:
\be\lab{duc26}
\norm{h_{8,p,q,l,m}'}_{L^1(\MM)}\les \ep^2 2^{j+\frac{l}{2}}\gamma^\nu_j\gamma^{\nu'}_j.
\ee

Now, we have in view of the decomposition \eqref{duc1} of $\sum_{m\leq l}B^{2,2}_{j,\nu,\nu',l,m}$:
\bee
&&\left|\sum_{m\leq l}B^{2,2}_{j,\nu,\nu',l,m}\right|\\
\nn&\les & 2^{-\frac{5j}{2}}\sum_{p, q\geq 0}c_{pq}\normm{\frac{1}{(2^{\frac{j}{2}}|N_\nu-N_{\nu'}|)^{p+q+1}}}_{L^\infty}\bigg[\normm{\frac{1}{|N_\nu-N_{\nu'}|^2}}_{L^\infty}(\norm{h_{1,p,q,l,m}'}_{L^1(\MM)}\\
\nn&&+\norm{h_{2,p,q,l,m}'}_{L^2(\MM)})+\normm{\frac{1}{|N_\nu-N_{\nu'}|}}_{L^\infty}(\norm{h_{3,p,q,l,m}'}_{L^1(\MM)}+\norm{h_{4,p,q,l,m}'}_{L^1(\MM)}\\
\nn&&+\norm{h_{5,p,q,l,m}'}_{L^1(\MM)}+\norm{h_{6,p,q,l,m}'}_{L^1(\MM)})+\norm{h_{7,p,q,l,m}'}_{L^1(\MM)}+\norm{h_{8,p,q,l,m}'}_{L^1(\MM)}\bigg].
\eee
Together with \eqref{nice26}, \eqref{duc12}, \eqref{duc13}, \eqref{duc15}, \eqref{duc18}, \eqref{duc21}, \eqref{duc22}, \eqref{duc23} and \eqref{duc26}, we obtain:
\bee
&&\left|\sum_{m\leq l}B^{2,2}_{j,\nu,\nu',l,m}\right|\\
\nn&\les & 2^{-\frac{5j}{2}}\sum_{p, q\geq 0}c_{pq}\frac{1}{(2^{\frac{j}{2}}|\nu-\nu'|)^{p+q+1}}\bigg[\frac{2^{j+\frac{l}{2}}}{|\nu-\nu'|^2}+\frac{2^{j+l}}{|\nu-\nu'|}+2^{j+l}\bigg] \ep^2 \gamma^\nu_j\gamma^{\nu'}_j\\
\nn&\les& \bigg[\frac{2^{-\frac{j}{2}+\frac{l}{2}}}{(2^{\frac{j}{2}}|\nu-\nu'|)^3}+\frac{2^{-j+l}}{(2^{\frac{j}{2}}|\nu-\nu'|)^2}+\frac{2^{-\frac{3j}{2}+l}}{2^{\frac{j}{2}}|\nu-\nu'|}\bigg] \ep^2 \gamma^\nu_j\gamma^{\nu'}_j.
\eee
Summing in $l$, we finally obtain in the range $2^m\leq 2^l\leq 2^j|\nu-\nu'|$.
\be\lab{duc27}
\left|\sum_{(l,m)/2^m\leq 2^l\leq 2^j|\nu-\nu'|}B^{2,2}_{j,\nu,\nu',l,m}\right|\les \bigg[\frac{2^{-\frac{j}{4}}}{(2^{\frac{j}{2}}|\nu-\nu'|)^{\frac{5}{2}}}+\frac{1}{2^{\frac{j}{2}}(2^{\frac{j}{2}}|\nu-\nu'|)}+2^{-j}\bigg] \ep^2 \gamma^\nu_j\gamma^{\nu'}_j.
\ee

\subsection{End of the proof of Proposition \ref{prop:labexfsmp1}}

In view of the decomposition \eqref{ldc1} of $B^2_{j,\nu,\nu',l,m}$, the estimate \eqref{nice26}, the estimates \eqref{ldc19} \eqref{ldc29} \eqref{ldc37} \eqref{ldc40} \eqref{ldc44} \eqref{ldc49} \eqref{ldc53} \eqref{ldc56} \eqref{ldc61} \eqref{ldc66} for $h_{1,p,q,l,m}$, $h_{2,p,q,l,m}$, $h_{3,p,q,l,m}$, $h_{4,p,q,l,m}$, $h_{5,p,q,l,m}$, $h_{6,p,q,l,m}$, $h_{7,p,q,l,m}$, $h_{8,p,q,l,m}$, $h_{9,p,q,l,m}$, $h_{10,p,q,l,m}$, the estimate \eqref{ldc67} for $B^{2,1}_{j,\nu,\nu',l,m}$ and the estimate \eqref{duc27} for $B^{2,2}_{j,\nu,\nu',l,m}$, we obtain:
\bee
&&\left|\sum_{(l,m)/2^m\leq 2^l\leq 2^j|\nu-\nu'|}(B^2_{j,\nu,\nu',l,m}+B^2_{j,\nu',\nu,l,m})\right|\\
\nn&\les& 2^{-\frac{3j}{2}}\sum_{p, q\geq 0}c_{pq}\frac{1}{(2^{\frac{j}{2}}|\nu-\nu'|)^{p+q+1}}\\
\nn&&\times\bigg[\frac{(1+q^2)2^{\frac{j}{2}}}{|\nu-\nu'|^2}+\frac{2^{\frac{3j}{4}}(2^{\frac{j}{2}}|\nu-\nu'|)^{\frac{1}{2}}+2^{\frac{3j}{4}}}{|\nu-\nu'|}+2^j+2^{\frac{3j}{4}}(2^{\frac{j}{2}}|\nu-\nu'|)^{\frac{1}{2}}\bigg]\ep^2\gamma^\nu_j\gamma^{\nu'}_j\\
\nn&&+2^{-j}\ep^2\gamma^\nu_j\gamma^{\nu'}_j+\bigg[\frac{2^{-\frac{j}{4}}}{(2^{\frac{j}{2}}|\nu-\nu'|)^{\frac{5}{2}}}+\frac{1}{2^{\frac{j}{2}}(2^{\frac{j}{2}}|\nu-\nu'|)}+2^{-j}\bigg] \ep^2 \gamma^\nu_j\gamma^{\nu'}_j\\
\nn&\les& \bigg[\frac{1}{(2^{\frac{j}{2}}|\nu-\nu'|)^3}+\frac{1}{(2^{\frac{j}{2}}|\nu-\nu'|)^{\frac{5}{2}}}+\frac{1}{2^{\frac{j}{4}}(2^{\frac{j}{2}}|\nu-\nu'|)^{\frac{3}{2}}}+\frac{2^{-\frac{j}{4}}}{(2^{\frac{j}{2}}|\nu-\nu'|)^2}+\frac{1}{2^{\frac{j}{2}}(2^{\frac{j}{2}}|\nu-\nu'|)}\\
\nn&&+\frac{1}{2^{\frac{3j}{4}}(2^{\frac{j}{2}}|\nu-\nu'|)^{\frac{1}{2}}}+2^{-j}\bigg]\ep^2\gamma^\nu_j\gamma^{\nu'}_j.
\eee
Together with the estimate \eqref{gini25} for $B^2_{j,\nu,\nu',l,m}$ in $2^m\leq 2^j|\nu-\nu'|<2^l$, we finally obtain the following control for $B^2_{j,\nu,\nu',l,m}$:
\bee
&&\left|\sum_{(l,m)/2^{\min(l,m)}\leq 2^j|\nu-\nu'|}(B^2_{j,\nu,\nu',l,m}+B^2_{j,\nu',\nu,l,m})\right|\\
\nn&\les& \bigg[\frac{1}{(2^{\frac{j}{2}}|\nu-\nu'|)^3}+\frac{1}{(2^{\frac{j}{2}}|\nu-\nu'|)^{\frac{5}{2}}}+\frac{1}{2^{\frac{j}{4}}(2^{\frac{j}{2}}|\nu-\nu'|)^{\frac{3}{2}}}+\frac{2^{-\frac{j}{4}}}{(2^{\frac{j}{2}}|\nu-\nu'|)^2}+\frac{1}{2^{\frac{j}{2}}(2^{\frac{j}{2}}|\nu-\nu'|)}\\
\nn&&+\frac{1}{2^{\frac{3j}{4}}(2^{\frac{j}{2}}|\nu-\nu'|)^{\frac{1}{2}}}+2^{-j}\bigg]\ep^2\gamma^\nu_j\gamma^{\nu'}_j.
\eee
This concludes the proof of Proposition \ref{prop:labexfsmp1}.

%%%%%%%%%%%%%%%%%%%%%%%%%%%%%%%%%

\section{Proof of Proposition \ref{prop:tsonga2}}\lab{sec:tsonga2}

Since $2^{\min(l,m)}\leq 2^j|\nu-\nu'|<2^{\max(l,m)}$, we may assume that $l>m$ and thus:
\be\lab{uso1bis}
2^m\leq 2^j|\nu-\nu'|<2^l.
\ee
In order to prove Proposition \ref{prop:tsonga2}, recall that we need to show:
\bee
&&\left|\sum_{(l,m)/2^{\min(l,m)}\leq 2^j|\nu-\nu'|<2^{\max(l,m)}}A_{j,\nu,\nu',l,m}\right|\\
\nn&\les& \sum_{(l,m)/2^{\min(l,m)}\leq 2^j|\nu-\nu'|<2^{\max(l,m)}}\frac{2^{-2j}2^{2\min(l,m)}}{(2^{\frac{j}{2}}|\nu-\nu'|)^2}\norm{\mu_{j,\nu,l}}_{L^2(\R\times\S)}\norm{\mu_{j,\nu',m}}_{L^2(\R\times\S)}\\
\nn&&+\left[\frac{1}{(2^{\frac{j}{2}}|\nu-\nu'|)^3}+\frac{2^{-\frac{j}{4}}}{(2^{\frac{j}{2}}|\nu-\nu'|)^2}+\frac{1}{(2^{\frac{j}{2}}|\nu-\nu'|)2^{\frac{j}{2}}}\right]\ep^2\gamma^\nu_j\gamma^{\nu'}_j,
\eee
where the sequence of functions $(\mu_{j,\nu,l})_{l\geq 0}$ on $\R\times\S$ satisfies:
$$\sum_{\nu}\sum_{l\geq 0}2^{2l}\norm{\mu_{j,\nu,l}}^2_{L^2(\R\times\S)}\les \ep^2 2^{2j}\norm{f}^2_{L^2(\R^3)},$$
and where $A_{j,\nu,\nu',l,m}$ is given by \eqref{nice8}:
\bee
A_{j,\nu,\nu',l,m}&=& -i2^{-j}\int_{\MM}\int_{\S\times\S} \frac{P_l\trc (N-\gn N')(P_m\trc')}{\gg(L,L')}\\
\nn&&\times F_j(u)F_{j,-1}(u')\eta_j^\nu(\o)\eta_j^{\nu'}(\o')d\o d\o' d\MM.
\eee
We may sum over the region \eqref{uso1bis}, and we obtain:
\bea\lab{zoo}
&&\sum_{(l,m)/2^{\min(l,m)}\leq 2^j|\nu-\nu'|<2^{\max(l,m)}}A_{j,\nu,\nu',l,m}\\
\nn&=& -i2^{-j}\int_{\MM}\int_{\S\times\S} \frac{P_{>2^j|\nu-\nu'|}\trc (N-\gn N')(P_{\leq 2^j|\nu-\nu'|}\trc')}{\gg(L,L')}\\
\nn&&\times F_j(u)F_{j,-1}(u')\eta_j^\nu(\o)\eta_j^{\nu'}(\o')d\o d\o' d\MM.
\eea

We integrate by parts using \eqref{fetebis}. 
\begin{lemma}\lab{lemma:zoo}
Let $\sum_{(l,m)/2^{\min(l,m)}\leq 2^j|\nu-\nu'|<2^{\max(l,m)}}A_{j,\nu,\nu',l,m}$ be defined by \eqref{zoo}. Integrating by parts using \eqref{fetebis} yields:
\bea\lab{zoo1}
&&\sum_{(l,m)/2^{\min(l,m)}\leq 2^j|\nu-\nu'|<2^{\max(l,m)}}A_{j,\nu,\nu',l,m}\\
\nn&=& 2^{-\frac{3j}{2}}\sum_{p, q\geq 0}c_{pq}\int_{\MM}\frac{1}{(2^{\frac{j}{2}}|N_\nu-N_{\nu'}|)^{p+q+1}}\bigg[\frac{1}{|N_\nu-N_{\nu'}|^2}(h_{1,p,q}+h_{2,p,q}+h_{3,p,q}+h_{4,p,q})\\
\nn&&+\frac{1}{|N_\nu-N_{\nu'}|}(h_{5,p,q}+h_{6,p,q}+h_{7,p,q})+h_{8,p,q}\bigg] d\MM,
\eea
where $c_{pq}$ are explicit real coefficients such that the series 
$$\sum_{p, q\geq 0}c_{pq}x^py^q$$
has radius of convergence 1, where the scalar functions $h_{1,p,q}$, $h_{2,p,q}$, $h_{3,p,q}$, $h_{4,p,q}$, $h_{5,p,q}$, $h_{6,p,q}$, $h_{7,p,q}$, $h_{8,p,q}$ on $\MM$ are given by:
\bea\lab{zoo2}
h_{1,p,q}&=& \left(\int_{\S} L(P_{>2^j|\nu-\nu'|}\trc)\left(2^{\frac{j}{2}}(N-N_\nu)\right)^pF_{j,-1}(u)\eta_j^\nu(\o)d\o\right)\\
\nn&&\times\left(\int_{\S}\nabb'(P_{\leq 2^j|\nu-\nu'|}\trc')\left(2^{\frac{j}{2}}(N'-N_{\nu'})\right)^qF_{j,-1}(u')\eta_j^{\nu'}(\o')d\o'\right),
\eea
\bea\lab{zoo3}
h_{2,p,q}&=& \left(\int_{\S}P_{>2^j|\nu-\nu'|}\trc\left(2^{\frac{j}{2}}(N-N_\nu)\right)^pF_{j,-1}(u)\eta_j^\nu(\o)d\o\right)\\
\nn&&\times\left(\int_{\S}\nabb'(L'(P_{\leq 2^j|\nu-\nu'|}\trc'))\left(2^{\frac{j}{2}}(N'-N_{\nu'})\right)^qF_{j,-1}(u')\eta_j^{\nu'}(\o')d\o'\right),
\eea
\bea\lab{zoo4}
h_{3,p,q}&=& \left(\int_{\S}H_1 P_{>2^j|\nu-\nu'|}\trc\left(2^{\frac{j}{2}}(N-N_\nu)\right)^pF_{j,-1}(u)\eta_j^\nu(\o)d\o\right)\\
\nn&&\times\left(\int_{\S}\nabb'(P_{\leq 2^j|\nu-\nu'|}\trc')\left(2^{\frac{j}{2}}(N'-N_{\nu'})\right)^qF_{j,-1}(u')\eta_j^{\nu'}(\o')d\o'\right),
\eea
\bea\lab{zoo5}
h_{4,p,q}&=& \left(\int_{\S}P_{>2^j|\nu-\nu'|}\trc\left(2^{\frac{j}{2}}(N-N_\nu)\right)^pF_{j,-1}(u)\eta_j^\nu(\o)d\o\right)\\
\nn&&\times\left(\int_{\S}H_2\nabb'(P_{\leq 2^j|\nu-\nu'|}\trc')\left(2^{\frac{j}{2}}(N'-N_{\nu'})\right)^qF_{j,-1}(u')\eta_j^{\nu'}(\o')d\o'\right),
\eea
\bea\lab{zoo6}
h_{5,p,q}&=& \left(\int_{\S}P_{>2^j|\nu-\nu'|}\trc\left(2^{\frac{j}{2}}(N-N_\nu)\right)^pF_{j,-1}(u)\eta_j^\nu(\o)d\o\right)\\
\nn&&\times\left(\int_{\S}{\nabb'}^2(P_{\leq 2^j|\nu-\nu'|}\trc')\left(2^{\frac{j}{2}}(N'-N_{\nu'})\right)^qF_{j,-1}(u')\eta_j^{\nu'}(\o')d\o'\right),
\eea
\bea\lab{zoo7}
h_{6,p,q}&=& \left(\int_{\S}H_3P_{>2^j|\nu-\nu'|}\trc\left(2^{\frac{j}{2}}(N-N_\nu)\right)^pF_{j,-1}(u)\eta_j^\nu(\o)d\o\right)\\
\nn&&\times\left(\int_{\S}\nabla(P_{\leq 2^j|\nu-\nu'|}\trc')\left(2^{\frac{j}{2}}(N'-N_{\nu'})\right)^qF_{j,-1}(u')\eta_j^{\nu'}(\o')d\o'\right),
\eea
\bea\lab{zoo8}
h_{7,p,q}&=& \left(\int_{\S}P_{>2^j|\nu-\nu'|}\trc\left(2^{\frac{j}{2}}(N-N_\nu)\right)^pF_{j,-1}(u)\eta_j^\nu(\o)d\o\right)\\
\nn&&\times\left(\int_{\S}H_4\nabla(P_{\leq 2^j|\nu-\nu'|}\trc')\left(2^{\frac{j}{2}}(N'-N_{\nu'})\right)^qF_{j,-1}(u')\eta_j^{\nu'}(\o')d\o'\right),
\eea
and:
\bea\lab{zoo9}
h_{8,p,q}&=& \left(\int_{\S}P_{>2^j|\nu-\nu'|}\trc\left(2^{\frac{j}{2}}(N-N_\nu)\right)^pF_{j,-1}(u)\eta_j^\nu(\o)d\o\right)\\
\nn&&\times\left(\int_{\S}\nabb'(N'(P_{\leq 2^j|\nu-\nu'|}\trc'))\left(2^{\frac{j}{2}}(N'-N_{\nu'})\right)^qF_{j,-1}(u')\eta_j^{\nu'}(\o')d\o'\right),
\eea
where the tensor $H_1$ on $\MM$ involved in the definition of $h_{3,p,q}$ is given by:
\be\lab{zoo10}
H_1=\chi+\kep+\d+n^{-1}\nabla n+L(b),
\ee
where the tensor $H_2$ on $\MM$ involved in the definition of $h_{4,p,q}$ is given by:
\be\lab{zoo11}
H_2=\chi'+\kep'+\d'+n^{-1}\nabla n+L'(b').
\ee
where the tensor $H_3$ on $\MM$ involved in the definition of $h_{6,p,q}$ is given by:
\be\lab{zoo12}
H_3=k+n^{-1}\nabla n+\th+b^{-1}\nabb(b)+\chi+\z,
\ee
and where the tensor $H_4$ on $\MM$ involved in the definition of $h_{7,p,q}$ is given by:
\be\lab{zoo13}
H_4=k+n^{-1}\nabla n+\th'+{b'}^{-1}\nabb'(b')+\z'+\nabla_{N'}(b').
\ee
\end{lemma}
The proof of lemma \ref{lemma:zoo} is postponed to Appendix I. In the rest of this section, we use Lemma \ref{lemma:zoo} to obtain the control of $\sum_{(l,m)/2^{\min(l,m)}\leq 2^j|\nu-\nu'|<2^{\max(l,m)}}A_{j,\nu,\nu',l,m}$.

We evaluate the $L^1(\MM)$ norm of $h_{1,p,q}$, $h_{2,p,q}$, $h_{3,p,q}$, $h_{4,p,q}$, $h_{5,p,q}$, $h_{6,p,q}$, $h_{7,p,q}$, $h_{8,p,q}$ starting with $h_{1,p,q}$. In view of the definition \eqref{zoo2} of $h_{1,p,q}$, we have:
$$h_{1,p,q}= \sum_{l>2^j|\nu-\nu'|}\int_{\S} G_1 L(P_{>2^j|\nu-\nu'|}\trc)\left(2^{\frac{j}{2}}(N-N_\nu)\right)^pF_{j,-1}(u)\eta_j^\nu(\o)d\o,$$
where the tensor $G_1$ on $\MM$ is given by:
\be\lab{zoo14}
G_1=\int_{\S}\nabb'(P_{\leq 2^j|\nu-\nu'|}\trc')\left(2^{\frac{j}{2}}(N'-N_{\nu'})\right)^qF_{j,-1}(u')\eta_j^{\nu'}(\o')d\o'.
\ee
In view of the estimate \eqref{nadal}, this yields:
\bea\lab{zoo15}
\norm{h_{1,p,q}}_{L^1(\MM)}&\les&  \left(\sup_{\o\in\textrm{supp}(\eta_j^{\nu})}\norm{G_1}_{L^2_{u, x'}L^\infty_t}\right)\left(\sum_{l>2^j|\nu-\nu'|}2^{\frac{j}{2}-l}\right)\ep\gamma_j^\nu\\
\nn&\les& \left(\sup_{\o\in\textrm{supp}(\eta_j^{\nu})}\norm{G_1}_{L^2_{u, x'}L^\infty_t}\right)\frac{\ep\gamma_j^\nu}{2^{\frac{j}{2}}|\nu-\nu'|}.
\eea
Now, in view of the definition \eqref{zoo14} of $G_1$ and the estimate \eqref{messi11}, we have:
\bea\lab{zoo16}
\sup_{\o\in\textrm{supp}(\eta_j^{\nu})}\norm{G_1}_{L^2_{u, x'}L^\infty_t}&\les&  \left(\sup_{\o'}\normm{\left(2^{\frac{j}{2}}(N'-N_{\nu'})\right)^q}_{L^\infty}\right)\ep 2^j|\nu-\nu'|\gamma^{\nu'}_j\\
\nn&\les& \ep 2^j|\nu-\nu'|\gamma^{\nu'}_j,
\eea
where we used in the last inequality the estimate \eqref{estNomega} for $\po N$ and the size of the patch. Finally, \eqref{zoo15} and \eqref{zoo16} imply:
\be\lab{zoo17}
\norm{h_{1,p,q}}_{L^1(\MM)}\les \ep^2 2^{\frac{j}{2}}\gamma_j^\nu\gamma^{\nu'}_j.
\ee

Next, we evaluate the $L^1(\MM)$ norm of $h_{2,p,q}$. In view of the definition \eqref{zoo3} of $h_{2,p,q}$, we have:
$$h_{2,p,q}= \sum_{l>2^j|\nu-\nu'|}\int_{\S}G_2\nabb'(L'(P_{\leq 2^j|\nu-\nu'|}\trc'))\left(2^{\frac{j}{2}}(N'-N_{\nu'})\right)^qF_{j,-1}(u')\eta_j^{\nu'}(\o')d\o',$$
where the tensor $G_2$ on $\MM$ is given by:
\be\lab{zoo18}
G_2=\int_{\S}P_{>2^j|\nu-\nu'|}\trc\left(2^{\frac{j}{2}}(N-N_\nu)\right)^pF_{j,-1}(u)\eta_j^\nu(\o)d\o.
\ee
In view of the estimate \eqref{nadal:1}, this yields:
\bea\lab{zoo19}
\norm{h_{2,p,q}}_{L^1(\MM)}&\les&  \left(\sup_{\o'\in\textrm{supp}(\eta_j^{\nu'})}\norm{G_2}_{L^2_{u, x'}L^\infty_t}\right)2^{\frac{j}{2}}\ep\gamma_j^{\nu'}.
\eea
Now, in view of the definition \eqref{zoo18} of $G_2$ and the estimate \eqref{messi4:0}, we have:
\bea\lab{zoo20}
\sup_{\o'\in\textrm{supp}(\eta_j^{\nu'})}\norm{G_2}_{L^2_{u, x'}L^\infty_t}&\les&  \left(\sup_\o\normm{\left(2^{\frac{j}{2}}(N-N_{\nu})\right)^p}_{L^\infty}\right)\ep\gamma^\nu_j\\
\nn&&\times\sum_{2^l>2^j|\nu-\nu'|}\big(2^{\frac{j}{2}}|\nu-\nu'|2^{-l+\frac{j}{2}}+(2^{\frac{j}{2}}|\nu-\nu'|)^{\frac{1}{2}}2^{-\frac{l}{2}+\frac{j}{4}}\big)\\
\nn&\les& \ep\gamma^\nu_j,
\eea
where we used in the last inequality the estimate \eqref{estNomega} for $\po N$ and the size of the patch. Finally, \eqref{zoo19} and \eqref{zoo20} imply:
\be\lab{zoo21}
\norm{h_{2,p,q}}_{L^1(\MM)}\les \ep^2 2^{\frac{j}{2}}\gamma_j^\nu\gamma^{\nu'}_j.
\ee

Next, we evaluate the $L^1(\MM)$ norm of $h_{3,p,q}$. In view of the definition \eqref{zoo4} of $h_{3,p,q}$, we have:
\bea\lab{zoo22}
\norm{h_{3,p,q}}_{L^1(\MM)}&\les& \normm{\int_{\S}H_1 P_{>2^j|\nu-\nu'|}\trc\left(2^{\frac{j}{2}}(N-N_\nu)\right)^pF_{j,-1}(u)\eta_j^\nu(\o)d\o}_{L^2(\MM)}\\
\nn&&\times\normm{\int_{\S}\nabb'(P_{\leq 2^j|\nu-\nu'|}\trc')\left(2^{\frac{j}{2}}(N'-N_{\nu'})\right)^qF_{j,-1}(u')\eta_j^{\nu'}(\o')d\o'}_{L^2(\MM)}\\
\nn&\les& \normm{\int_{\S}H_1 P_{>2^j|\nu-\nu'|}\trc\left(2^{\frac{j}{2}}(N-N_\nu)\right)^pF_{j,-1}(u)\eta_j^\nu(\o)d\o}_{L^2(\MM)}2^{\frac{j}{2}}\ep\gamma^{\nu'}_j,
\eea
where we used in the last inequality the analog of the estimate \eqref{vino56}. The basic estimate in $L^2(\MM)$ \eqref{oscl2bis} yields:
\bea\lab{zoo23}
&&\normm{\int_{\S}H_1 P_l\trc\left(2^{\frac{j}{2}}(N-N_\nu)\right)^pF_{j,-1}(u)\eta_j^\nu(\o)d\o}_{L^2(\MM)}\\
\nn&\les&\left(\sup_\o\normm{H_1 P_l\trc\left(2^{\frac{j}{2}}(N-N_\nu)\right)^p}_{\li{\infty}{2}}\right)2^{\frac{j}{2}}\gamma^\nu_j\\
\nn&\les&\left(\sup_\o\norm{H_1}_{\xt{\infty}{2}}\norm{P_l\trc}_{\xt{2}{\infty}}\normm{\left(2^{\frac{j}{2}}(N-N_\nu)\right)^p}_{L^\infty}\right)2^{\frac{j}{2}}\gamma^\nu_j\\
\nn&\les& \left(\sup_\o\norm{H_1}_{\xt{\infty}{2}}\right)2^{\frac{j}{2}-l}\ep\gamma^\nu_j,
\eea
where we used in the last inequality the estimate \eqref{lievremont1} for $P_l\trc$, the estimate \eqref{estNomega} for $\po N$ and the size of the patch. Now, the definition of $H_1$ \eqref{zoo10}, the estimates \eqref{esttrc} \eqref{esthch} for $\chi$, the estimate \eqref{estk} for $\kep$ and $\d$, the estimate \eqref{estn} for $n$ and the estimate \eqref{estb} for $b$ imply:
\bea\lab{zoo23bis}
\norm{H_1}_{\xt{\infty}{2}}&\les&\norm{\chi}_{\xt{\infty}{2}}+\norm{\kep}_{\xt{\infty}{2}}+\norm{\d}_{\xt{\infty}{2}}+\norm{n^{-1}\nabla n}_{\xt{\infty}{2}}+\norm{L(b)}_{\xt{\infty}{2}}\\
\nn&\les& \ep,
\eea
which together with \eqref{zoo23} yields:
\be\lab{zoo24}
\normm{\int_{\S}H_1 P_l\trc\left(2^{\frac{j}{2}}(N-N_\nu)\right)^pF_{j,-1}(u)\eta_j^\nu(\o)d\o}_{L^2(\MM)}\les 2^{\frac{j}{2}-l}\ep\gamma^\nu_j.
\ee
Finally, \eqref{zoo22} and \eqref{zoo24} imply:
\bea\lab{zoo25}
\norm{h_{3,p,q}}_{L^1(\MM)}&\les&  2^j\left(\sum_{2^l>2^j|\nu-\nu'|}2^{-l}\right)\ep^2\gamma^\nu_j\gamma^{\nu'}_j\\
\nn&\les&  \frac{2^{\frac{j}{2}}}{2^{\frac{j}{2}}|\nu-\nu'|}\ep^2\gamma^\nu_j\gamma^{\nu'}_j.
\eea

Next, we evaluate the $L^1(\MM)$ norm of $h_{4,p,q}$. In view of the definition \eqref{zoo5} of $h_{4,p,q}$, we have:
\bea\lab{zoo26}
&&\norm{h_{4,p,q}}_{L^2(\MM)}\\
\nn&\les& \normm{\int_{\S}P_{>2^j|\nu-\nu'|}\trc\left(2^{\frac{j}{2}}(N-N_\nu)\right)^pF_{j,-1}(u)\eta_j^\nu(\o)d\o}_{L^2(\MM)}\\
\nn&&\times\normm{\int_{\S}H_2\nabb'(P_{\leq 2^j|\nu-\nu'|}\trc')\left(2^{\frac{j}{2}}(N'-N_{\nu'})\right)^qF_{j,-1}(u')\eta_j^{\nu'}(\o')d\o'}_{L^2(\MM)}\\
\nn&\les& \frac{\ep\gamma^\nu_j}{2^{\frac{j}{2}}|\nu-\nu'|}\normm{\int_{\S}H_2\nabb'(P_{\leq 2^j|\nu-\nu'|}\trc')\left(2^{\frac{j}{2}}(N'-N_{\nu'})\right)^qF_{j,-1}(u')\eta_j^{\nu'}(\o')d\o'}_{L^2(\MM)},
\eea
where we used in the last inequality the estimate \eqref{nycc133}. The basic estimate in $L^2(\MM)$ \eqref{oscl2bis} yields:
\bea\lab{zoo27}
&&\normm{\int_{\S}H_2\nabb'(P_{\leq 2^j|\nu-\nu'|}\trc')\left(2^{\frac{j}{2}}(N'-N_{\nu'})\right)^qF_{j,-1}(u')\eta_j^{\nu'}(\o')d\o'}_{L^2(\MM)}\\
\nn&\les&\left(\sup_{\o'}\normm{H_2\nabb'(P_{\leq 2^j|\nu-\nu'|}\trc')\left(2^{\frac{j}{2}}(N'-N_{\nu'})\right)^q}_{\li{\infty}{2}}\right)2^{\frac{j}{2}}\gamma^\nu_j\\
\nn&\les&\left(\sup_\o\norm{H_2}_{\xt{\infty}{2}}\norm{\nabb'(P_{\leq 2^j|\nu-\nu'|}\trc')}_{\xt{2}{\infty}}\normm{\left(2^{\frac{j}{2}}(N'-N_{\nu'})\right)^q}_{L^\infty}\right)2^{\frac{j}{2}}\gamma^{\nu'}_j\\
\nn&\les& \left(\sup_{\o'}\norm{H_2}_{\xt{\infty}{2}}\right)2^{\frac{j}{2}}\ep\gamma^{\nu'}_j,
\eea
where we used in the last inequality the estimate \eqref{lievremont2} for $\nabb'(P_{\leq 2^j|\nu-\nu'|}\trc')$, the estimate \eqref{estNomega} for $\po N$ and the size of the patch. Now, in view of the definition of $H_2$ \eqref{zoo11}, and proceeding as for the proof of \eqref{zoo23bis}, we have:
$$\norm{H_2}_{\xt{\infty}{2}}\les \ep,$$
which together with \eqref{zoo27} yields:
\be\lab{zoo28}
\normm{\int_{\S}H_2\nabb'(P_{\leq 2^j|\nu-\nu'|}\trc')\left(2^{\frac{j}{2}}(N'-N_{\nu'})\right)^qF_{j,-1}(u')\eta_j^{\nu'}(\o')d\o'}_{L^2(\MM)}\les 2^{\frac{j}{2}}\ep\gamma^{\nu'}_j.
\ee
Finally, \eqref{zoo26} and \eqref{zoo28} imply:
\be\lab{zoo29}
\norm{h_{4,p,q}}_{L^1(\MM)}\les  \frac{2^{\frac{j}{2}}}{2^{\frac{j}{2}}|\nu-\nu'|}\ep^2\gamma^\nu_j\gamma^{\nu'}_j.
\ee

Next, we evaluate the $L^1(\MM)$ norm of $h_{5,p,q}$. In view of the definition \eqref{zoo6} of $h_{5,p,q}$, we have:
\bee
\norm{h_{5,p,q}}_{L^1(\MM)}&\les& \normm{\int_{\S}P_{>2^j|\nu-\nu'|}\trc\left(2^{\frac{j}{2}}(N-N_\nu)\right)^pF_{j,-1}(u)\eta_j^\nu(\o)d\o}_{L^2(\MM)}\\
\nn&&\times\normm{\int_{\S}{\nabb'}^2(P_{\leq 2^j|\nu-\nu'|}\trc')\left(2^{\frac{j}{2}}(N'-N_{\nu'})\right)^qF_{j,-1}(u')\eta_j^{\nu'}(\o')d\o'}_{L^2(\MM)}\\
\nn&\les& \sum_{2^m\leq 2^j|\nu-\nu'|<2^l}\left(\int_{\S}\normm{P_l\trc\left(2^{\frac{j}{2}}(N-N_\nu)\right)^pF_{j,-1}(u)}_{L^2(\MM)}\eta_j^\nu(\o)d\o\right)\\
\nn&&\times\left(\int_{\S}\normm{{\nabb'}^2(P_m\trc')\left(2^{\frac{j}{2}}(N'-N_{\nu'})\right)^qF_{j,-1}(u')}_{L^2(\MM)}\eta_j^{\nu'}(\o')d\o'\right)\\
\nn&\les& \sum_{2^m\leq 2^j|\nu-\nu'|<2^l}\left(\int_{\S}\normm{\norm{P_l\trc}_{L^2(\H_u)}F_{j,-1}(u)}_{L^2_u}\eta_j^\nu(\o)d\o\right)\\
\nn&&\times\left(\int_{\S}\normm{\norm{{\nabb'}^2(P_m\trc')}_{L^2(\H_{u'})} F_{j,-1}(u')}_{L^2_{u'}}\eta_j^{\nu'}(\o')d\o'\right)
\eee
where we used in the last inequality the estimate \eqref{estNomega} for $\po N$ and the size of the patch. Taking Cauchy Schwartz in $\o$ and $\o'$, using the size of the patches, and using the Bochner inequality \eqref{eq:Bochconseqbis}, we obtain:
\bea\lab{zoo30}
\norm{h_{5,p,q}}_{L^1(\MM)}&\les& \sum_{2^m\leq 2^j|\nu-\nu'|<2^l}2^{2m-j}\normm{\norm{P_l\trc}_{L^2(\H_u)}F_j(u)\sqrt{\eta^\nu_j(\o)}}_{L^2_{\o,u}}\\
\nn&&\times\normm{\norm{P_m'\trc}_{L^2(\H_u)}F_j(u)\sqrt{\eta^{\nu'}_j(\o)}}_{L^2_{\o,u}}.
\eea
In view of \eqref{zoo30} and the estimate \eqref{tsonga2}, we finally obtain:
\be\lab{zoo31}
\norm{h_{5,p,q}}_{L^1(\MM)}\les \sum_{2^m\leq 2^j|\nu-\nu'|<2^l}2^{2m-j}\norm{\mu_{j,\nu,l}}_{L^2(\R\times\S)}\norm{\mu_{j,\nu',m}}_{L^2(\R\times\S)},
\ee
where the sequence of functions $(\mu_{j,\nu,l})_{l> j/2}$ on $\R\times\S$ satisfies:
$$\sum_{\nu}\sum_{l> j/2}2^{2l}\norm{\mu_{j,\nu,l}}^2_{L^2(\R\times\S)}\les \ep^2 2^{2j}\norm{f}^2_{L^2(\R^3)}.$$

Next, we evaluate the $L^1(\MM)$ norm of $h_{6,p,q}$. In view of the definition \eqref{zoo7} of $h_{6,p,q}$, we have:
\bea\lab{zoo32}
\norm{h_{6,p,q}}_{L^1(\MM)}&\leq& \normm{\int_{\S}H_3P_{>2^j|\nu-\nu'|}\trc\left(2^{\frac{j}{2}}(N-N_\nu)\right)^pF_{j,-1}(u)\eta_j^\nu(\o)d\o}_{L^2(\MM)}\\
\nn&&\times\normm{\int_{\S}\nabla(P_{\leq 2^j|\nu-\nu'|}\trc')\left(2^{\frac{j}{2}}(N'-N_{\nu'})\right)^qF_{j,-1}(u')\eta_j^{\nu'}(\o')d\o'}_{L^2(\MM)}\\
\nn&\les& \normm{\int_{\S}H_3P_{>2^j|\nu-\nu'|}\trc\left(2^{\frac{j}{2}}(N-N_\nu)\right)^pF_{j,-1}(u)\eta_j^\nu(\o)d\o}_{L^2(\MM)}2^{\frac{j}{2}}\ep\gamma^{\nu'}_j,
\eea
where we used in the last inequality the analog of the estimate \eqref{vino56}. The basic estimate in $L^2(\MM)$ \eqref{oscl2bis} yields:
\bea\lab{zoo33}
&&\normm{\int_{\S}H_3P_l\trc\left(2^{\frac{j}{2}}(N-N_\nu)\right)^pF_{j,-1}(u)\eta_j^\nu(\o)d\o}_{L^2(\MM)}\\
\nn&\les&  \left(\sup_{\o}\normm{H_3P_l\trc\left(2^{\frac{j}{2}}(N-N_\nu)\right)^p}_{\lprime{\infty}{2}}\right)2^{\frac{j}{2}}\gamma^{\nu}_j\\
\nn&\les& \left(\sup_\o\norm{H_3}_{\tx{\infty}{4}}\norm{P_l\trc}_{\tx{2}{4}}\normm{\left(2^{\frac{j}{2}}(N-N_\nu)\right)^p}_{L^\infty}\right)2^{\frac{j}{2}}\gamma^{\nu}_j\\
\nn&\les& \left(\sup_\o\norm{H_3}_{\tx{\infty}{4}}\right)\ep 2^{\frac{j}{2}-\frac{l}{2}}\gamma^{\nu}_j,
\eea
where we used in the last inequality the estimate \eqref{vino40} for $P_l\trc$ and the estimate \eqref{estNomega} for $\po N$ and the size of the patch. Now, in view of the definition \eqref{zoo12} for $H_3$, we have:
\bea\lab{zoo34}
\norm{H_3}_{\tx{\infty}{4}}&\les& \norm{k}_{\tx{\infty}{4}}+\norm{n^{-1}\nabla n}_{\tx{\infty}{4}}+\norm{\th}_{\tx{\infty}{4}}\\
\nn&&+\norm{b^{-1}\nabb(b)}_{\tx{\infty}{4}}+\norm{\chi}_{\tx{\infty}{4}}+\norm{\z}_{\tx{\infty}{4}}\\
\nn&\les& \ep,
\eea
where we used in the last inequality the embedding \eqref{sobineq1}, the estimate \eqref{estk} for $k$, the estimate \eqref{estn} for $n$, the estimates \eqref{estk} \eqref{esttrc} \eqref{esthch} for $\th$, the estimate \eqref{estb} for $b$, the estimate \eqref{esttrc} \eqref{esthch} for $\chi$ and the estimate \eqref{estzeta} for $\z$. \eqref{zoo33} and \eqref{zoo34} yield:
\be\lab{zoo35}
\normm{\int_{\S}H_3P_l\trc\left(2^{\frac{j}{2}}(N-N_\nu)\right)^pF_{j,-1}(u)\eta_j^\nu(\o)d\o}_{L^2(\MM)}\les \ep 2^{\frac{j}{2}-\frac{l}{2}}\gamma^{\nu}_j.
\ee
Finally, \eqref{zoo32} and \eqref{zoo35} imply:
\bea\lab{zoo36}
\norm{h_{6,p,q}}_{L^1(\MM)}&\les& \ep^2\left(\sum_{2^l>2^j|\nu-\nu'|}2^{j-\frac{l}{2}}\right)\gamma^{\nu}_j\gamma^{\nu'}_j\\
\nn&\les& \frac{\ep^2 2^{\frac{3j}{4}}\gamma^{\nu}_j\gamma^{\nu'}_j}{(2^{\frac{j}{2}}|\nu-\nu'|)^{\frac{1}{2}}}.
\eea

Next, we evaluate the $L^1(\MM)$ norm of $h_{7,p,q}$. In view of the definition \eqref{zoo8} of $h_{7,p,q}$, we have:
\bea\lab{zoo37}
&&\norm{h_{7,p,q}}_{L^1(\MM)}\\
\nn&\leq& \normm{\int_{\S}P_{>2^j|\nu-\nu'|}\trc\left(2^{\frac{j}{2}}(N-N_\nu)\right)^pF_{j,-1}(u)\eta_j^\nu(\o)d\o}_{L^2(\MM)}\\
\nn&&\times\normm{\int_{\S}H_4\nabla(P_{\leq 2^j|\nu-\nu'|}\trc')\left(2^{\frac{j}{2}}(N'-N_{\nu'})\right)^qF_{j,-1}(u')\eta_j^{\nu'}(\o')d\o'}_{L^2(\MM)}\\
\nn&\les& \frac{\ep\gamma^\nu_j}{2^{\frac{j}{2}}|\nu-\nu'|}\normm{\int_{\S}H_4\nabla(P_{\leq 2^j|\nu-\nu'|}\trc')\left(2^{\frac{j}{2}}(N'-N_{\nu'})\right)^qF_{j,-1}(u')\eta_j^{\nu'}(\o')d\o'}_{L^2(\MM)},
\eea
where we used in the last inequality the estimate \eqref{nycc133}. Proceeding as for the estimate of \eqref{duc11}, 
we have:
\bea\lab{zoo38}
&&\normm{\int_{\S}H_4\nabla(P_l\trc')\left(2^{\frac{j}{2}}(N'-N_{\nu'})\right)^qF_{j,-1}(u')\eta_j^{\nu'}(\o')d\o'}_{L^2(\MM)}\\
\nn&\les& \left(\sup_{\o'}\norm{H_4}_{\tx{\infty}{4}}\right)2^{\frac{l}{2}+\frac{j}{2}}\ep \gamma^{\nu'}_j.
\eea
Now, in view of the definition \eqref{zoo13} of $H_4$, we have:
\bea\lab{zoo39}
\norm{H_4}_{\tx{\infty}{4}}&\les& \norm{k}_{\tx{\infty}{4}}+\norm{n^{-1}\nabla n}_{\tx{\infty}{4}}+\norm{\th'}_{\tx{\infty}{4}}\\
\nn&&+\norm{{b'}^{-1}\nabb'(b')}_{\tx{\infty}{4}}+\norm{\z'}_{\tx{\infty}{4}}+\norm{\nabla_{N'}(b')}_{\tx{\infty}{4}}\\
\nn&\les&\ep,
\eea
where we used in the last inequality the embedding \eqref{sobineq1}, the estimate \eqref{estk} for $k$, the estimate \eqref{estn} for $n$, the estimates \eqref{estk} \eqref{esttrc} \eqref{esthch} for $\th'$, the estimate \eqref{estb} for $b'$, and the estimate \eqref{estzeta} for $\z'$. \eqref{zoo38} and \eqref{zoo39} yield:
\be\lab{zoo40}
\normm{\int_{\S}H_4\nabla(P_l\trc')\left(2^{\frac{j}{2}}(N'-N_{\nu'})\right)^qF_{j,-1}(u')\eta_j^{\nu'}(\o')d\o'}_{L^2(\MM)}\les 2^{\frac{l}{2}+\frac{j}{2}}\ep \gamma^{\nu'}_j.
\ee
Finally, \eqref{zoo37} and \eqref{zoo40} imply:
\bea\lab{zoo41}
\norm{h_{7,p,q}}_{L^1(\MM)}&\les& \frac{1}{2^{\frac{j}{2}}|\nu-\nu'|}\left(\sum_{2^l>2^j|\nu-\nu'|}2^{\frac{l}{2}}\right)2^{\frac{j}{2}}\ep^2\gamma^\nu_j\gamma^{\nu'}_j\\
\nn&\les& \frac{2^{\frac{3j}{4}}\ep^2\gamma^\nu_j\gamma^{\nu'}_j}{(2^{\frac{j}{2}}|\nu-\nu'|)^{\frac{1}{2}}}.
\eea

Next, we evaluate the $L^1(\MM)$ norm of $h_{8,p,q}$. In view of the definition \eqref{zoo9} of $h_{8,p,q}$, we have:
\bea\lab{zoo42}
\norm{h_{8,p,q}}_{L^1(\MM)}&\leq& \normm{\int_{\S}P_{>2^j|\nu-\nu'|}\trc\left(2^{\frac{j}{2}}(N-N_\nu)\right)^pF_{j,-1}(u)\eta_j^\nu(\o)d\o}_{L^2(\MM)}\\
\nn&&\times\normm{\int_{\S}\nabb'(N'(P_{\leq 2^j|\nu-\nu'|}\trc'))\left(2^{\frac{j}{2}}(N'-N_{\nu'})\right)^qF_{j,-1}(u')\eta_j^{\nu'}(\o')d\o'}_{L^2(\MM)}\\
\nn&\les& 2^j\ep^2\gamma^\nu_j\gamma^{\nu'}_j,
\eea
where we used in the last inequality the estimate \eqref{nycc133} for the first term, and the analog of the estimate \eqref{duc17} for the second term. 

Now, we have in view of the decomposition \eqref{zoo1} of $\sum_{(l,m)/2^{\min(l,m)}\leq 2^j|\nu-\nu'|<2^{\max(l,m)}}A_{j,\nu,\nu',l,m}$:
\bee
&&\left|\sum_{(l,m)/2^{\min(l,m)}\leq 2^j|\nu-\nu'|<2^{\max(l,m)}}A_{j,\nu,\nu',l,m}\right|\\
\nn&\leq& 2^{-\frac{3j}{2}}\sum_{p, q\geq 0}c_{pq}\normm{\frac{1}{(2^{\frac{j}{2}}|N_\nu-N_{\nu'}|)^{p+q+1}}}_{L^\infty(\MM)}\\
\nn&&\times\Bigg[\normm{\frac{1}{|N_\nu-N_{\nu'}|^2}}_{L^\infty(\MM)}(\norm{h_{1,p,q}}_{L^1(\MM)}+\norm{h_{2,p,q}}_{L^1(\MM)}+\norm{h_{3,p,q}}_{L^1(\MM)}+\norm{h_{4,p,q}}_{L^1(\MM)})\\
\nn&&+\normm{\frac{1}{|N_\nu-N_{\nu'}|}}_{L^\infty(\MM)}(\norm{h_{5,p,q}}_{L^1(\MM)}+\norm{h_{6,p,q}}_{L^1(\MM)}+\norm{h_{7,p,q}}_{L^1(\MM)})+\norm{h_{8,p,q}}_{L^1(\MM)}\Bigg],
\eee
Together with \eqref{nice26}, \eqref{zoo17}, \eqref{zoo21}, \eqref{zoo25}, \eqref{zoo29}, \eqref{zoo31}, \eqref{zoo36}, \eqref{zoo41},  and \eqref{zoo42}, we obtain:
\bee
&&\left|\sum_{(l,m)/2^{\min(l,m)}\leq 2^j|\nu-\nu'|<2^{\max(l,m)}}A_{j,\nu,\nu',l,m}\right|\\
\nn&\les& 2^{-\frac{3j}{2}}\sum_{p, q\geq 0}c_{pq}\frac{1}{(2^{\frac{j}{2}}|\nu-\nu'|)^{p+q+1}}\Bigg[\frac{2^{\frac{j}{2}}}{|\nu-\nu'|^2}\\
\nn&&+\frac{1}{|\nu-\nu'|}\left(\sum_{2^m\leq 2^j|\nu-\nu'|<2^l}2^{2m-j}\norm{\mu_{j,\nu,l}}_{L^2(\R\times\S)}\norm{\mu_{j,\nu',m}}_{L^2(\R\times\S)}+2^{\frac{3j}{4}}\right)+2^j\Bigg]\ep^2\gamma^\nu_j\gamma^{\nu'}_j\\
\nn&\les& \sum_{2^m\leq 2^j|\nu-\nu'|<2^l}\frac{2^{2m-2j}}{(2^{\frac{j}{2}}|\nu-\nu'|)^2}\norm{\mu_{j,\nu,l}}_{L^2(\R\times\S)}\norm{\mu_{j,\nu',m}}_{L^2(\R\times\S)}\\
\nn&&+\Bigg[\frac{1}{(2^{\frac{j}{2}}|\nu-\nu'|)^3}+\frac{2^{-\frac{j}{4}}}{(2^{\frac{j}{2}}|\nu-\nu'|)^2}+\frac{1}{2^{\frac{j}{2}}(2^{\frac{j}{2}}|\nu-\nu'|)}\Bigg]\ep^2\gamma^\nu_j\gamma^{\nu'}_j,
\eee
where the sequence of functions $(\mu_{j,\nu,l})_{l> j/2}$ on $\R\times\S$ satisfies:
$$\sum_{\nu}\sum_{l> j/2}2^{2l}\norm{\mu_{j,\nu,l}}^2_{L^2(\R\times\S)}\les \ep^2 2^{2j}\norm{f}^2_{L^2(\R^3)}.$$
This concludes the proof of Proposition \ref{prop:tsonga2}.

%%%%%%%%%%%%%%%%%%%%%%%%%%%%%%%

\section{Proof of Proposition \ref{prop:tsonga3}}\lab{sec:tsonga3}

Since $2^{\max(l,m)}\leq 2^j|\nu-\nu'|$, we may assume that $l\geq m$ and thus:
\be\lab{uso1ter}
2^m\leq 2^l\leq 2^j|\nu-\nu'|.
\ee
In order to prove Proposition \ref{prop:tsonga3}, recall that we need to show:
\bee
&&\left|\sum_{(l,m)/2^{\max(l,m)}\leq 2^j|\nu-\nu'|}A_{j,\nu,\nu',l,m}\right|\\
\nn&\les& \sum_{(l,m)/2^{\max(l,m)}\leq 2^j|\nu-\nu'|}\frac{2^{-\frac{5j}{2}}2^{l+m+\min(l,m)}}{(2^{\frac{j}{2}}|\nu-\nu'|)^3}\norm{\mu_{j,\nu,l}}_{L^2(\R\times\S)}\norm{\mu_{j,\nu',m}}_{L^2(\R\times\S)}\\
\nn&&+\bigg[\frac{1}{(2^{\frac{j}{2}}|\nu-\nu'|)^3}+\frac{1}{(2^{\frac{j}{2}}|\nu-\nu'|)^{\frac{5}{2}}}+ \frac{2^{-(\frac{1}{6})_-j}}{(2^{\frac{j}{2}}|\nu-\nu'|)^2}+\frac{1}{2^{\frac{j}{2}}(2^{\frac{j}{2}}|\nu-\nu'|)}+2^{-j}\bigg]\ep^2\gamma^\nu_j\gamma^{\nu'}_j,
\eee
where the sequence of functions $(\mu_{j,\nu,l})_{l>j/2}$ on $\R\times\S$ satisfies:
$$\sum_{\nu}\sum_{l>j/2}2^{2l}\norm{\mu_{j,\nu,l}}^2_{L^2(\R\times\S)}\les \ep^2 2^{2j}\norm{f}^2_{L^2(\R^3)},$$
and where $A_{j,\nu,\nu',l,m}$ is given by \eqref{nice8}:
\bea\lab{zol}
A_{j,\nu,\nu',l,m}&=& -i2^{-j}\int_{\MM}\int_{\S\times\S} \frac{P_l\trc (N-\gn N')(P_m\trc')}{\gg(L,L')}\\
\nn&&\times F_j(u)F_{j,-1}(u')\eta_j^\nu(\o)\eta_j^{\nu'}(\o')d\o d\o' d\MM.
\eea

We integrate by parts using \eqref{fete1}. 
\begin{lemma}\lab{lemma:zol}
Let $A_{j,\nu,\nu',l,m}$ be defined by \eqref{zol}. Integrating by parts using \eqref{fete1} yields:
\bea
\nn A_{j,\nu,\nu',l,m}&=& A^1_{j,\nu,\nu',l,m}+A^2_{j,\nu,\nu',l,m}+A^3_{j,\nu,\nu',l,m}+2^{-j}\sum_{p, q\geq 0}c_{pq}\int_{\MM}\frac{1}{(2^{\frac{j}{2}}|N_\nu-N_{\nu'}|)^{p+q+2}}\\
\lab{zol1}&&\times\left[\frac{1}{|N_\nu-N_{\nu'}|}(h_{1,p,q,l,m}+h_{2,p,q,l,m})+h_{3,p,q,l,m}+h_{4,p,q,l,m}\right] d\MM,
\eea
where $c_{pq}$ are explicit real coefficients such that the series 
$$\sum_{p, q\geq 0}c_{pq}x^py^q$$
has radius of convergence 1, where the scalar functions $A^1_{j,\nu,\nu',l,m}$, $A^2_{j,\nu,\nu',l,m}$, $A^3_{j,\nu,\nu',l,m}$ on $\MM$ are given by:
\bea
\nn A^1_{j,\nu,\nu',l,m}&=& 2^{-2j}\int_{\MM}\int_{\S\times\S}\frac{P_l\trc {\nabb'}^2P_m(\trc')(N-\gn N',N-\gn N')}{\gl(1-\gn^2)}\\
\lab{zol2}&&\times F_j(u)F_{j,-1}(u')\eta_j^\nu(\o)\eta_j^{\nu'}(\o')d\o d\o' d\MM,
\eea
\bea
\nn A^2_{j,\nu,\nu',l,m}&=& 2^{-2j}\int_{\MM}\int_{\S\times\S}\frac{(N'-\gn N)(P_l\trc)(N-\gn N')(P_m\trc')}{\gl(1-\gn^2)}\\
\lab{zol3}&&\times F_j(u)F_{j,-1}(u')\eta_j^\nu(\o)\eta_j^{\nu'}(\o')d\o d\o' d\MM,
\eea
and:
\bea\lab{zol4}
A^3_{j,\nu,\nu',l,m}&=& 2^{-2j}\int_{\MM}\int_{\S\times\S}\frac{N(P_l\trc)(N-\gn N')(P_m\trc')}{\gl}\\
\nn&&\times F_j(u)F_{j,-1}(u')\eta_j^\nu(\o)\eta_j^{\nu'}(\o')d\o d\o' d\MM,
\eea
and where the scalar functions $h_{1,p,q,l,m}$, $h_{2,p,q,l,m}$, $h_{3,p,q,l,m}$, $h_{4,p,q,l,m}$ on $\MM$ are given by:
\bea\lab{zol5}
h_{1,p,q,l,m}&=& \left(\int_{\S} \chi P_l\trc\left(2^{\frac{j}{2}}(N-N_\nu)\right)^pF_{j,-1}(u)\eta_j^\nu(\o)d\o\right)\\
\nn&&\times\left(\int_{\S}\nabb'(P_m\trc')\left(2^{\frac{j}{2}}(N'-N_{\nu'})\right)^qF_{j,-1}(u')\eta_j^{\nu'}(\o')d\o'\right),
\eea
\bea\lab{zol6}
h_{2,p,q,l,m}&=& \left(\int_{\S}P_l\trc\left(2^{\frac{j}{2}}(N-N_\nu)\right)^pF_{j,-1}(u)\eta_j^\nu(\o)d\o\right)\\
\nn&&\times\left(\int_{\S}\chi'\nabb'(P_m\trc')\left(2^{\frac{j}{2}}(N'-N_{\nu'})\right)^qF_{j,-1}(u')\eta_j^{\nu'}(\o')d\o'\right),
\eea
\bea\lab{zol7}
h_{3,p,q,l,m}&=& \left(\int_{\S}(\th+b^{-1}\nabla(b)) P_l\trc\left(2^{\frac{j}{2}}(N-N_\nu)\right)^pF_{j,-1}(u)\eta_j^\nu(\o)d\o\right)\\
\nn&&\times\left(\int_{\S}\nabb'(P_m\trc')\left(2^{\frac{j}{2}}(N'-N_{\nu'})\right)^qF_{j,-1}(u')\eta_j^{\nu'}(\o')d\o'\right),
\eea
and:
\bea\lab{zol8}
h_{4,p,q,l,m}&=& \left(\int_{\S}P_l\trc\left(2^{\frac{j}{2}}(N-N_\nu)\right)^pF_{j,-1}(u)\eta_j^\nu(\o)d\o\right)\\
\nn&&\times\left(\int_{\S}(\th'+{b'}^{-1}\nabb'(b'))\nabb'(P_m\trc')\left(2^{\frac{j}{2}}(N'-N_{\nu'})\right)^qF_{j,-1}(u')\eta_j^{\nu'}(\o')d\o'\right).
\eea
\end{lemma}
The proof of lemma \ref{lemma:zol} is postponed to Appendix J. In the rest of this section, we use Lemma \ref{lemma:zol} to obtain the control of $A_{j,\nu,\nu',l,m}$.

We evaluate the $L^1(\MM)$ norm of $h_{1,p,q,l,m}$, $h_{2,p,q,l,m}$, $h_{3,p,q,l,m}$, $h_{4,p,q,l,m}$ starting with $h_{1,p,q,l,m}$. In view of the definition \eqref{zol5} of $h_{1,p,q,l,m}$, we have:
\bee
\sum_{m\leq l}h_{1,p,q,l,m}&=& \left(\int_{\S} \chi P_l\trc\left(2^{\frac{j}{2}}(N-N_\nu)\right)^pF_{j,-1}(u)\eta_j^\nu(\o)d\o\right)\\
\nn&&\times\left(\int_{\S}\nabb'(P_{\leq l}\trc')\left(2^{\frac{j}{2}}(N'-N_{\nu'})\right)^qF_{j,-1}(u')\eta_j^{\nu'}(\o')d\o'\right).
\eee
This yields:
\bea\lab{zol9}
\normm{\sum_{m\leq l}h_{1,p,q,l,m}}_{L^1(\MM)}&\les & \normm{\int_{\S} \chi P_l\trc\left(2^{\frac{j}{2}}(N-N_\nu)\right)^pF_{j,-1}(u)\eta_j^\nu(\o)d\o}_{L^2(\MM)}\\
\nn&&\times\normm{\int_{\S}\nabb'(P_{\leq l}\trc')\left(2^{\frac{j}{2}}(N'-N_{\nu'})\right)^qF_{j,-1}(u')\eta_j^{\nu'}(\o')d\o'}_{L^2(\MM)}\\
\nn&\les & \normm{\int_{\S} \chi P_l\trc\left(2^{\frac{j}{2}}(N-N_\nu)\right)^pF_{j,-1}(u)\eta_j^\nu(\o)d\o}_{L^2(\MM)}\ep 2^{\frac{j}{2}}\gamma^{\nu'}_j,
\eea
where we used in the last inequality the analog of the estimate \eqref{vino56}. Now, the analog of \eqref{ldc28} implies:
\be\lab{zol10}
\normm{\int_{\S} \chi P_l\trc\left(2^{\frac{j}{2}}(N-N_\nu)\right)^pF_{j,-1}(u)\eta_j^\nu(\o)d\o}_{L^2(\MM)}\les \ep 2^{\frac{j}{2}-l}\gamma^\nu_j.
\ee
Finally, \eqref{zol9} and \eqref{zol10} imply:
\be\lab{zol11}
\normm{\sum_{m\leq l}h_{1,p,q,l,m}}_{L^1(\MM)}\les  \ep^2 2^{j-l}\gamma^\nu_j\gamma^{\nu'}_j.
\ee

Next, we evaluate the $L^1(\MM)$ norm of $h_{2,p,q,l,m}$. In view of the definition \eqref{zol6} of $h_{2,p,q,l,m}$, we have:
\bee
\sum_{m\leq l}h_{2,p,q,l,m}&=& \left(\int_{\S}P_l\trc\left(2^{\frac{j}{2}}(N-N_\nu)\right)^pF_{j,-1}(u)\eta_j^\nu(\o)d\o\right)\\
\nn&&\times\left(\int_{\S}\chi'\nabb'(P_{\leq l}\trc')\left(2^{\frac{j}{2}}(N'-N_{\nu'})\right)^qF_{j,-1}(u')\eta_j^{\nu'}(\o')d\o'\right).
\eee
This yields:
\bea\lab{zol12}
\normm{\sum_{m\leq l}h_{2,p,q,l,m}}_{L^1(\MM)}&\les& \normm{\int_{\S}P_l\trc\left(2^{\frac{j}{2}}(N-N_\nu)\right)^pF_{j,-1}(u)\eta_j^\nu(\o)d\o}_{L^2(\MM)}\\
\nn&&\times\normm{\int_{\S}\chi'\nabb'(P_{\leq l}\trc')\left(2^{\frac{j}{2}}(N'-N_{\nu'})\right)^qF_{j,-1}(u')\eta_j^{\nu'}(\o')d\o'}_{L^2(\MM)}.
\eea
Using \eqref{nyc40} for $l>j/2$ and \eqref{nyc41} for $l=j/2$, we obtain for all $l\geq j/2$:
\be\lab{zol13}
\normm{\int_{\S}P_l\trc\left(2^{\frac{j}{2}}(N-N_\nu)\right)^pF_{j,-1}(u)\eta_j^\nu(\o)d\o}_{L^2(\MM)}\les 2^{\frac{j}{2}-l}\ep\gamma^\nu_j.
\ee
Also, the basic estimate in $L^2(\MM)$ \eqref{oscl2bis} yields:
\bea\lab{zol14}
&&\normm{\int_{\S}\chi'\nabb'(P_{\leq l}\trc')\left(2^{\frac{j}{2}}(N'-N_{\nu'})\right)^qF_{j,-1}(u')\eta_j^{\nu'}(\o')d\o'}_{L^2(\MM)}\\
\nn&\les& \left(\sup_{\o'}\normm{\chi'\nabb'(P_{\leq l}\trc')\left(2^{\frac{j}{2}}(N'-N_{\nu'})\right)^q}_{\lprime{\infty}{2}}\right)2^{\frac{j}{2}}\gamma^{\nu'}_j\\
\nn&\les& \left(\sup_{\o'}\norm{\chi'}_{\xt{\infty}{2}}\norm{\nabb'(P_{\leq l}\trc')}_{\xt{2}{\infty}}\normm{\left(2^{\frac{j}{2}}(N'-N_{\nu'})\right)^q}_{L^\infty}\right)2^{\frac{j}{2}}\gamma^{\nu'}_j\\
\nn&\les& \ep 2^{\frac{j}{2}}\gamma^{\nu'}_j,
\eea
where we used in the last inequality the estimates \eqref{esttrc} \eqref{esthch} for $\chi$, the estimate \eqref{lievremont2} for $\nabb'(P_{\leq l}\trc')$, the estimate \eqref{estNomega} for $\po N$ and the size of the patch. 
Finally, \eqref{zol12}, \eqref{zol13} and \eqref{zol14} imply:
\be\lab{zol15}
\normm{\sum_{m\leq l}h_{2,p,q,l,m}}_{L^1(\MM)}\les \ep^2 2^{j-l}\gamma^\nu_j\gamma^{\nu'}_j.
\ee

Next, we evaluate the $L^1(\MM)$ norm of $h_{3,p,q,l,m}$. In view of the definition \eqref{zol7} of $h_{3,p,q,l,m}$, we have:
\bea\lab{zol16}
\norm{h_{3,p,q,l,m}}_{L^1(\MM)}&\les& \normm{\int_{\S}(\th+b^{-1}\nabla(b)) P_l\trc\left(2^{\frac{j}{2}}(N-N_\nu)\right)^pF_{j,-1}(u)\eta_j^\nu(\o)d\o}_{L^2(\MM)}\\
\nn&&\times\normm{\int_{\S}\nabb'(P_m\trc')\left(2^{\frac{j}{2}}(N'-N_{\nu'})\right)^qF_{j,-1}(u')\eta_j^{\nu'}(\o')d\o'}_{L^2(\MM)}\\
\nn&\les& \ep^2 2^{(\frac{3}{4})_+j}\gamma^\nu_j\gamma^{\nu'}_j,
\eea
where we used in the last inequality the estimate \eqref{vino24} for the first term, and the analog of the estimate \eqref{vino56} for the second term. 

Next, we evaluate the $L^1(\MM)$ norm of $h_{4,p,q,l,m}$. In view of the definition \eqref{zol8} of $h_{4,p,q,l,m}$, we have:
\bea\lab{zol17}
&&\norm{h_{4,p,q,l,m}}_{L^1(\MM)}\\
\nn&\les& \normm{\int_{\S}P_l\trc\left(2^{\frac{j}{2}}(N-N_\nu)\right)^pF_{j,-1}(u)\eta_j^\nu(\o)d\o}_{L^2(\MM)}\\
\nn&&\times\normm{\int_{\S}(\th'+{b'}^{-1}\nabb'(b'))\nabb'(P_m\trc')\left(2^{\frac{j}{2}}(N'-N_{\nu'})\right)^qF_{j,-1}(u')\eta_j^{\nu'}(\o')d\o'}_{L^2(\MM)}\\
\nn&\les& \ep^2 2^{j+\frac{m}{2}-l}\gamma^\nu_j\gamma^{\nu'}_j,
\eea
where we used in the last inequality the estimate \eqref{zol13} for the first term, and the analog of the estimate \eqref{duc25} for the second term. 

Next, we estimate $A^3_{j,\nu,\nu',l,m}$. In view of the definition \eqref{zol4} of $A^3_{j,\nu,\nu',l,m}$ and the definition \eqref{ldc15} of $B^{2,2}_{j,\nu,\nu',l,m}$, and in view of the fact that $\gl=-1+\gn$, we see that $A^3_{j,\nu,\nu',l,m}$ is essentially obtained from $B^{2,2}_{j,\nu,\nu',l,m}$ by exchanging the role of $\nu$ and $\nu'$. Proceeding for $A^3_{j,\nu,\nu',l,m}$ as we did for $B^{2,2}_{j,\nu,\nu',l,m}$, using the integration by parts \eqref{fete}   instead of \eqref{fete1}, we obtain the analog of the estimate \eqref{duc27}:
\be\lab{zol18}
\left|\sum_{(l,m)/2^m\leq 2^l\leq 2^j|\nu-\nu'|}A^3_{j,\nu,\nu',l,m}\right|\les \bigg[\frac{2^{-\frac{j}{4}}}{(2^{\frac{j}{2}}|\nu-\nu'|)^{\frac{5}{2}}}+\frac{1}{2^{\frac{j}{2}}(2^{\frac{j}{2}}|\nu-\nu'|)}+2^{-j}\bigg] \ep^2 \gamma^\nu_j\gamma^{\nu'}_j.
\ee 

Next, we consider $A^1_{j,\nu,\nu',l,m}$. We have the following proposition. 
\begin{proposition}\lab{prop:buz}
Let $A^1_{j,\nu,\nu',l,m}$ be given by \eqref{zol2}. Then, $A^1_{j,\nu,\nu',l,m}$ satisfies the following estimate:
\bea\lab{buz}
&&\left|\sum_{(l,m)/2^{\max(l,m)}\leq 2^j|\nu-\nu'|}A^1_{j,\nu,\nu',l,m}\right|\\
\nn&\les& \sum_{(l,m)/2^{\max(l,m)}\leq 2^j|\nu-\nu'|}\frac{2^{-\frac{5j}{2}}2^{l+m+\min(l,m)}}{(2^{\frac{j}{2}}|\nu-\nu'|)^3}\norm{\mu_{j,\nu,l}}_{L^2(\R\times\S)}\norm{\mu_{j,\nu',m}}_{L^2(\R\times\S)}\\
\nn&&+\Bigg[\frac{1}{(2^{\frac{j}{2}}|\nu-\nu'|)^3}+\frac{2^{-(\frac{1}{6})_-j}}{(2^{\frac{j}{2}}|\nu-\nu'|)^2}+\frac{1}{2^{\frac{j}{2}}(2^{\frac{j}{2}}|\nu-\nu'|)}\Bigg]\ep^2\gamma^\nu_j\gamma^{\nu'}_j,
\eea
where the sequence of functions $(\mu_{j,\nu,l})_{l>j/2}$ on $\R\times\S$ satisfies:
$$\sum_{\nu}\sum_{l>j/2}2^{2l}\norm{\mu_{j,\nu,l}}^2_{L^2(\R\times\S)}\les \ep^2 2^{2j}\norm{f}^2_{L^2(\R^3)}.$$
\end{proposition}
The proof of Proposition \ref{prop:buz} is postponed to section \ref{sec:buz}. 

Next, we consider $A^2_{j,\nu,\nu',l,m}$. We have the following proposition. 
\begin{proposition}\lab{prop:biz}
Let $A^2_{j,\nu,\nu',l,m}$ be given by \eqref{zol3}. Then, $A^2_{j,\nu,\nu',l,m}$ satisfies the following estimate:
\bea\lab{biz}
&&\left|\sum_{(l,m)/2^{\max(l,m)}\leq 2^j|\nu-\nu'|}A^2_{j,\nu,\nu',l,m}\right|\\
\nn&\les& \sum_{(l,m)/2^{\max(l,m)}\leq 2^j|\nu-\nu'|}\frac{2^{-\frac{5j}{2}}2^{l+m+\min(l,m)}}{(2^{\frac{j}{2}}|\nu-\nu'|)^3}\norm{\mu_{j,\nu,l}}_{L^2(\R\times\S)}\norm{\mu_{j,\nu',m}}_{L^2(\R\times\S)}\\
\nn&&+\Bigg[\frac{1}{(2^{\frac{j}{2}}|\nu-\nu'|)^3}+\frac{1}{(2^{\frac{j}{2}}|\nu-\nu'|)^{\frac{5}{2}}}+\frac{1}{2^{\frac{j}{2}}(2^{\frac{j}{2}}|\nu-\nu'|)}\Bigg]\ep^2\gamma^\nu_j\gamma^{\nu'}_j,
\eea
where the sequence of functions $(\mu_{j,\nu,l})_{l>j/2}$ on $\R\times\S$ satisfies:
$$\sum_{\nu}\sum_{l>j/2}2^{2l}\norm{\mu_{j,\nu,l}}^2_{L^2(\R\times\S)}\les \ep^2 2^{2j}\norm{f}^2_{L^2(\R^3)}.$$
\end{proposition}
The proof of Proposition \ref{prop:biz} is postponed to section \ref{sec:biz}. 

Finally, we estimate $A_{j,\nu,\nu',l,m}$. In view of the decomposition \eqref{zol1} of $A_{j,\nu,\nu',l,m}$, the estimate \eqref{nice26}, the estimates \eqref{zol11} \eqref{zol15} \eqref{zol16} \eqref{zol17} for $h_{1,p,q,l,m}$, $h_{2,p,q,l,m}$, $h_{3,p,q,l,m}$, $h_{4,p,q,l,m}$, and the estimates \eqref{zol18} \eqref{buz} \eqref{biz} for $A^1_{j,\nu,\nu',l,m}$, $A^2_{j,\nu,\nu',l,m}$ and $A^3_{j,\nu,\nu',l,m}$, we obtain:
\bee
&&\left|\sum_{(l,m)/2^{\max(l,m)}\leq 2^j|\nu-\nu'|}A_{j,\nu,\nu',l,m}\right|\\
\nn&\les& \sum_{(l,m)/2^{\max(l,m)}\leq 2^j|\nu-\nu'|}\frac{2^{-\frac{5j}{2}}2^{l+m+\min(l,m)}}{(2^{\frac{j}{2}}|\nu-\nu'|)^3}\norm{\mu_{j,\nu,l}}_{L^2(\R\times\S)}\norm{\mu_{j,\nu',m}}_{L^2(\R\times\S)}\\
\nn&&+\bigg[\frac{1}{(2^{\frac{j}{2}}|\nu-\nu'|)^3}+\frac{1}{(2^{\frac{j}{2}}|\nu-\nu'|)^{\frac{5}{2}}}+\frac{2^{-(\frac{1}{6})_-j}}{(2^{\frac{j}{2}}|\nu-\nu'|)^2}+\frac{1}{2^{\frac{j}{2}}(2^{\frac{j}{2}}|\nu-\nu'|)}+2^{-j}\bigg] \ep^2 \gamma^\nu_j\gamma^{\nu'}_j\\
\nn&&+2^{-j}\sum_{p, q\geq 0}c_{pq}\frac{1}{(2^{\frac{j}{2}}|\nu-\nu'|)^{p+q+2}}\left[\frac{2^{\frac{j}{2}}}{|\nu-\nu'|} + 2^{(\frac{3}{4})_+j}\right]\ep^2\gamma^\nu_j\gamma^{\nu'}_j\\
\nn&\les& \sum_{(l,m)/2^{\max(l,m)}\leq 2^j|\nu-\nu'|}\frac{2^{-\frac{5j}{2}}2^{l+m+\min(l,m)}}{(2^{\frac{j}{2}}|\nu-\nu'|)^3}\norm{\mu_{j,\nu,l}}_{L^2(\R\times\S)}\norm{\mu_{j,\nu',m}}_{L^2(\R\times\S)}\\
\nn&&+\bigg[\frac{1}{(2^{\frac{j}{2}}|\nu-\nu'|)^3}+\frac{1}{(2^{\frac{j}{2}}|\nu-\nu'|)^{\frac{5}{2}}}+ \frac{2^{-(\frac{1}{6})_-j}}{(2^{\frac{j}{2}}|\nu-\nu'|)^2}+\frac{1}{2^{\frac{j}{2}}(2^{\frac{j}{2}}|\nu-\nu'|)}+2^{-j}\bigg] \ep^2 \gamma^\nu_j\gamma^{\nu'}_j,
\eee
where the sequence of functions $(\mu_{j,\nu,l})_{l>j/2}$ on $\R\times\S$ satisfies:
$$\sum_{\nu}\sum_{l>j/2}2^{2l}\norm{\mu_{j,\nu,l}}^2_{L^2(\R\times\S)}\les \ep^2 2^{2j}\norm{f}^2_{L^2(\R^3)}.$$
This concludes the proof of Proposition \ref{prop:tsonga3}.

%%%%%%%%%%%%%%%%%%%%%%%%%%%%%%%%%%%%%%%%%%%%%%

\subsection{Proof of Proposition \ref{prop:buz} (Control of $A^1_{j,\nu,\nu',l,m}$)}\lab{sec:buz}

In order to prove Proposition \ref{prop:buz}, recall that we need to show:
\bee
&&\left|\sum_{(l,m)/2^{\max(l,m)}\leq 2^j|\nu-\nu'|}A^1_{j,\nu,\nu',l,m}\right|\\
\nn&\les& \sum_{(l,m)/2^{\max(l,m)}\leq 2^j|\nu-\nu'|}\frac{2^{-\frac{5j}{2}}2^{l+m+\min(l,m)}}{(2^{\frac{j}{2}}|\nu-\nu'|)^3}\norm{\mu_{j,\nu,l}}_{L^2(\R\times\S)}\norm{\mu_{j,\nu',m}}_{L^2(\R\times\S)}\\
\nn&&+\Bigg[\frac{1}{(2^{\frac{j}{2}}|\nu-\nu'|)^3}+\frac{2^{-(\frac{1}{6})_-j}}{(2^{\frac{j}{2}}|\nu-\nu'|)^2}+\frac{1}{2^{\frac{j}{2}}(2^{\frac{j}{2}}|\nu-\nu'|)}\Bigg]\ep^2\gamma^\nu_j\gamma^{\nu'}_j,
\eee
where the sequence of functions $(\mu_{j,\nu,l})_{l>j/2}$ on $\R\times\S$ satisfies:
$$\sum_{\nu}\sum_{l>j/2}2^{2l}\norm{\mu_{j,\nu,l}}^2_{L^2(\R\times\S)}\les \ep^2 2^{2j}\norm{f}^2_{L^2(\R^3)},$$
and where $A^1_{j,\nu,\nu',l,m}$ is given by \eqref{zol2}:
\bea
\nn A^1_{j,\nu,\nu',l,m}&=& 2^{-2j}\int_{\MM}\int_{\S\times\S}\frac{P_l\trc {\nabb'}^2P_m(\trc')(N-\gn N',N-\gn N')}{\gl(1-\gn^2)}\\
\lab{buz0}&&\times F_j(u)F_{j,-1}(u')\eta_j^\nu(\o)\eta_j^{\nu'}(\o')d\o d\o' d\MM,
\eea

We integrate by parts using \eqref{fete1}. 
\begin{lemma}\lab{lemma:buz}
Let $A^1_{j,\nu,\nu',l,m}$ be defined by \eqref{buz0}. Integrating by parts using \eqref{fete1} yields:
\bea\lab{buz1}
&& A^1_{j,\nu,\nu',l,m}\\
\nn&=& 2^{-2j}\sum_{p, q\geq 0}c_{pq}\int_{\MM}\frac{1}{(2^{\frac{j}{2}}|N_\nu-N_{\nu'}|)^{p+q+2}}\Bigg[\frac{1}{|N_\nu-N_{\nu'}|^2}(h_{1,p,q,l,m}+h_{2,p,q,l,m})\\
\nn&& +\frac{1}{|N_\nu-N_{\nu'}|}(h_{3,p,q,l,m}+h_{4,p,q,l,m}+h_{5,p,q,l,m}+h_{6,p,q,l,m})+h_{7,p,q,l,m}\Bigg] d\MM,
\eea
where $c_{pq}$ are explicit real coefficients such that the series 
$$\sum_{p, q\geq 0}c_{pq}x^py^q$$
has radius of convergence 1, and where the scalar functions $h_{1,p,q,l,m}$, $h_{2,p,q,l,m}$, $h_{3,p,q,l,m}$, $h_{4,p,q,l,m}$, $h_{5,p,q,l,m}$, $h_{6,p,q,l,m}$, $h_{7,p,q,l,m}$ on $\MM$ are given by:
\bea\lab{buz2}
h_{1,p,q,l,m}&=& \left(\int_{\S} \chi P_l\trc\left(2^{\frac{j}{2}}(N-N_\nu)\right)^pF_{j,-1}(u)\eta_j^\nu(\o)d\o\right)\\
\nn&&\times\left(\int_{\S}{\nabb'}^2(P_m\trc')\left(2^{\frac{j}{2}}(N'-N_{\nu'})\right)^qF_{j,-1}(u')\eta_j^{\nu'}(\o')d\o'\right),
\eea
\bea\lab{buz3}
h_{2,p,q,l,m}&=& \left(\int_{\S}P_l\trc\left(2^{\frac{j}{2}}(N-N_\nu)\right)^pF_{j,-1}(u)\eta_j^\nu(\o)d\o\right)\\
\nn&&\times\left(\int_{\S}\chi'{\nabb'}^2(P_m\trc')\left(2^{\frac{j}{2}}(N'-N_{\nu'})\right)^qF_{j,-1}(u')\eta_j^{\nu'}(\o')d\o'\right),
\eea
\bea\lab{buz4}
h_{3,p,q,l,m}&=& \left(\int_{\S} \nabb(P_l\trc)\left(2^{\frac{j}{2}}(N-N_\nu)\right)^pF_{j,-1}(u)\eta_j^\nu(\o)d\o\right)\\
\nn&&\times\left(\int_{\S}{\nabb'}^2(P_m\trc')\left(2^{\frac{j}{2}}(N'-N_{\nu'})\right)^qF_{j,-1}(u')\eta_j^{\nu'}(\o')d\o'\right),
\eea
\bea\lab{buz5}
h_{4,p,q,l,m}&=& \left(\int_{\S}P_l\trc\left(2^{\frac{j}{2}}(N-N_\nu)\right)^pF_{j,-1}(u)\eta_j^\nu(\o)d\o\right)\\
\nn&&\times\left(\int_{\S}{\nabb'}^3(P_m\trc')\left(2^{\frac{j}{2}}(N'-N_{\nu'})\right)^qF_{j,-1}(u')\eta_j^{\nu'}(\o')d\o'\right),
\eea
\bea\lab{buz6}
h_{5,p,q,l,m}&=& \left(\int_{\S}(\th+b^{-1}\nabla(b)) P_l\trc\left(2^{\frac{j}{2}}(N-N_\nu)\right)^pF_{j,-1}(u)\eta_j^\nu(\o)d\o\right)\\
\nn&&\times\left(\int_{\S}{\nabb'}^2(P_m\trc')\left(2^{\frac{j}{2}}(N'-N_{\nu'})\right)^qF_{j,-1}(u')\eta_j^{\nu'}(\o')d\o'\right),
\eea
\bea\lab{buz7}
h_{6,p,q,l,m}&=& \left(\int_{\S}P_l\trc\left(2^{\frac{j}{2}}(N-N_\nu)\right)^pF_{j,-1}(u)\eta_j^\nu(\o)d\o\right)\\
\nn&&\times\left(\int_{\S}(\th'+{b'}^{-1}\nabb'(b')){\nabb'}^2(P_m\trc')\left(2^{\frac{j}{2}}(N'-N_{\nu'})\right)^qF_{j,-1}(u')\eta_j^{\nu'}(\o')d\o'\right),
\eea
and:
\bea\lab{buz8}
h_{7,p,q,l,m}&=& \left(\int_{\S}N(P_l\trc)\left(2^{\frac{j}{2}}(N-N_\nu)\right)^pF_{j,-1}(u)\eta_j^\nu(\o)d\o\right)\\
\nn&&\times\left(\int_{\S}{\nabb'}^2(P_m\trc')\left(2^{\frac{j}{2}}(N'-N_{\nu'})\right)^qF_{j,-1}(u')\eta_j^{\nu'}(\o')d\o'\right).
\eea
\end{lemma}
The proof of Lemma \ref{lemma:buz} is postponed to Appendix K. In the rest of this section, we use Lemma \ref{lemma:buz} to obtain the control of $A^1_{j,\nu,\nu',l,m}$.

We evaluate the $L^1(\MM)$ norm of $h_{1,p,q,l,m}$, $h_{2,p,q,l,m}$, $h_{3,p,q,l,m}$, $h_{4,p,q,l,m}$, $h_{5,p,q,l,m}$, $h_{6,p,q,l,m}$, $h_{7,p,q,l,m}$ starting with $h_{1,p,q,l,m}$. In view of the definition \eqref{buz2} of $h_{1,p,q,l,m}$, we have:
\bea\lab{buz9}
\norm{h_{1,p,q,l,m}}_{L^1(\MM)}&\les& \normm{\int_{\S} \chi P_l\trc\left(2^{\frac{j}{2}}(N-N_\nu)\right)^pF_{j,-1}(u)\eta_j^\nu(\o)d\o}_{L^2(\MM)}\\
\nn&&\times\normm{\int_{\S}{\nabb'}^2(P_m\trc')\left(2^{\frac{j}{2}}(N'-N_{\nu'})\right)^qF_{j,-1}(u')\eta_j^{\nu'}(\o')d\o'}_{L^2(\MM)}\\
\nn&\les& \ep 2^{\frac{j}{2}-l}\gamma^\nu_j\normm{\int_{\S}{\nabb'}^2(P_m\trc')\left(2^{\frac{j}{2}}(N'-N_{\nu'})\right)^qF_{j,-1}(u')\eta_j^{\nu'}(\o')d\o'}_{L^2(\MM)},
\eea
where we used in the last inequality the estimate \eqref{zol10}. The basic estimate in $L^2(\MM)$ \eqref{oscl2bis} yields:
\bea\lab{buz10}
&&\normm{\int_{\S}{\nabb'}^2(P_m\trc')\left(2^{\frac{j}{2}}(N'-N_{\nu'})\right)^qF_{j,-1}(u')\eta_j^{\nu'}(\o')d\o'}_{L^2(\MM)}\\
\nn&\les& \left(\sup_{\o'}\normm{{\nabb'}^2(P_m\trc')\left(2^{\frac{j}{2}}(N'-N_{\nu'})\right)^q}_{\lprime{\infty}{2}}\right)2^{\frac{j}{2}}\gamma^{\nu'}_j\\
\nn&\les & 2^{\frac{j}{2}+m}\ep\gamma^{\nu'}_j,
\eea
where we used in the last inequality the Bochner inequality \eqref{eq:Bochconseqbis}, the finite band property for $P_m$, the estimate \eqref{esttrc} for $\trc$, the estimate \eqref{estNomega} for $\po N$, and the size of the patch. 
Finally, \eqref{buz9} and \eqref{buz10} imply:
\be\lab{buz11}
\norm{h_{1,p,q,l,m}}_{L^1(\MM)}\les \ep^2 2^{j+m-l}\gamma^\nu_j\gamma^{\nu'}_j.
\ee

Next, we evaluate the $L^1(\MM)$ of $h_{2,p,q,l,m}$. In view of the definition \eqref{buz3} of $h_{2,p,q,l,m}$, we have:
$$h_{2,p,q,l,m}=\int_{\S}H\chi'{\nabb'}^2(P_m\trc')\left(2^{\frac{j}{2}}(N'-N_{\nu'})\right)^qF_{j,-1}(u')\eta_j^{\nu'}(\o')d\o',$$
where $H$ is given by:
\be\lab{buz12}
H=\int_{\S}P_l\trc\left(2^{\frac{j}{2}}(N-N_\nu)\right)^pF_{j,-1}(u)\eta_j^\nu(\o)d\o.
\ee
This yields:
\bea\lab{buz13}
&&\norm{h_{2,p,q,l,m}}_{L^1(\MM)}\\
\nn&\les& \int_{\S}\normm{H\chi'{\nabb'}^2(P_m\trc')\left(2^{\frac{j}{2}}(N'-N_{\nu'})\right)^qF_{j,-1}(u')}_{L^1(\MM)}\eta_j^{\nu'}(\o')d\o'\\
\nn&\les& \int_{\S}\normm{H\chi_1'{\nabb'}^2(P_m\trc')\left(2^{\frac{j}{2}}(N'-N_{\nu'})\right)^qF_{j,-1}(u')}_{L^1(\MM)}\eta_j^{\nu'}(\o')d\o'\\
\nn&&+\int_{\S}\normm{H\chi_2'{\nabb'}^2(P_m\trc')\left(2^{\frac{j}{2}}(N'-N_{\nu'})\right)^qF_{j,-1}(u')}_{L^1(\MM)}\eta_j^{\nu'}(\o')d\o',
\eea
where we used in the last inequality the decomposition \eqref{dechch} of $\hch'$, and where we neglected the $\trc'$ 
contribution to $\chi'$ since it satisfies better estimates than $\chi_1'$. We start with the first term in the right-hand side of \eqref{buz13}. We have:
\bea\lab{buz14}
&&\int_{\S}\normm{H\chi_1'{\nabb'}^2(P_m\trc')\left(2^{\frac{j}{2}}(N'-N_{\nu'})\right)^qF_{j,-1}(u')}_{L^1(\MM)}\eta_j^{\nu'}(\o')d\o'\\
\nn&\les& \norm{H}_{L^2(\MM)}\bigg(\int_{\S}\norm{\chi_1'}_{L^\infty_{u'}L^2_tL^\infty_{{x'}'}}\norm{{\nabb'}^2(P_m\trc')}_{\tx{\infty}{2}}\\
\nn&&\times\normm{\left(2^{\frac{j}{2}}(N'-N_{\nu'})\right)^qF_{j,-1}(u')}_{L^\infty}\eta_j^{\nu'}(\o')d\o'\bigg)\\
\nn &\les& \norm{H}_{L^2(\MM)}\ep 2^{\frac{j}{2}+m}\gamma^{\nu'}_j,
\eea
where we used in the last inequality the estimate \eqref{dechch2} for $\chi_1'$, the Bochner inequality \eqref{eq:Bochconseqbis}, the finite band property for $P_m$, the estimate \eqref{esttrc} for $\trc$, the estimate \eqref{estNomega} for $\po N$, the size of the patch, Plancherel in $\la'$ for $F_{j,-1}(u')$, and Cauchy Schwarz in $\o'$. 
Next, we estimate the second term in the right-hand side of \eqref{buz13}. We have:
\bea\lab{buz15}
&&\int_{\S}\normm{H\chi_2'{\nabb'}^2(P_m\trc')\left(2^{\frac{j}{2}}(N'-N_{\nu'})\right)^qF_{j,-1}(u')}_{L^1(\MM)}\eta_j^{\nu'}(\o')d\o'\\
\nn&\les& \int_{\S}\normm{H(\chi_2'-{\chi_2}_{\nu}){\nabb'}^2(P_m\trc')\left(2^{\frac{j}{2}}(N'-N_{\nu'})\right)^qF_{j,-1}(u')}_{L^1(\MM)}\eta_j^{\nu'}(\o')d\o'\\
\nn&& +\int_{\S}\normm{H{\chi_2}_{\nu}{\nabb'}^2(P_m\trc')\left(2^{\frac{j}{2}}(N'-N_{\nu'})\right)^qF_{j,-1}(u')}_{L^1(\MM)}\eta_j^{\nu'}(\o')d\o'\\
\nn&\les& \left(\left(\sup_{\o'}\norm{\chi_2'-{\chi_2}_{\nu}}_{L^{6_-}(\MM)}\right)\norm{H}_{L^{3_+}(\MM)}+\norm{{\chi_2}_{\nu}H}_{L^2(\MM)}\right)\\
\nn&&\times\left(\int_{\S}\normm{{\nabb'}^2(P_m\trc')\left(2^{\frac{j}{2}}(N'-N_{\nu'})\right)^qF_{j,-1}(u')}_{L^2(\MM)}\eta_j^{\nu'}(\o')d\o'\right)\\
\nn&\les& (|\nu-\nu'|\norm{H}_{L^{3_+}(\MM)}+\norm{{\chi_2}_{\nu}H}_{L^2(\MM)})\ep 2^{\frac{j}{2}+m}\gamma^{\nu'}_j,
\eea
where we used in the last inequality the estimate \eqref{dechch2} for $\po\chi_2'$, the Bochner inequality \eqref{eq:Bochconseqbis}, the finite band property for $P_m$, the estimate \eqref{esttrc} for $\trc$, the estimate \eqref{estNomega} for $\po N$, the size of the patch, Plancherel in $\la'$ for $F_{j,-1}(u')$, and Cauchy Schwarz in $\o'$. 
Next, we estimate $H$. In view of the definition \eqref{buz12} of $H$ and the estimate \eqref{zol13}, we have:
\be\lab{buz16}
\norm{H}_{L^2(\MM)}\les \ep 2^{\frac{j}{2}-l}\gamma^\nu_j.
\ee
Also, in view of the estimate \eqref{vino21}, we have:
$$\norm{H}_{L^\infty(\MM)}\les \ep 2^j\gamma^\nu_j,$$
which by interpolation with \eqref{buz16} implies:
\be\lab{buz17}
\norm{H}_{L^{3_+}(\MM)}\les \ep 2^{(\frac{1}{3})_+j}\gamma^\nu_j.
\ee
Finally, we have:
\bea\lab{buz18}
\norm{{\chi_2}_{\nu}H}_{L^2(\MM)}&\les& \norm{{\chi_2}_\nu}_{L^\infty_{{x'}_\nu}L^2_t}\norm{H}_{L^\infty_{u_\nu, {x'}_\nu}L^2_t}\\
\nn&\les& (1+p^2)\ep 2^{\frac{j}{4}-\frac{l}{2}}\gamma^\nu_j,
\eea
where we used in the last inequality the estimate \eqref{dechch1} for ${\chi_2}_\nu$, the estimate \eqref{messi4:0} for $H$ in the case $l>j/2$, and the estimate \eqref{vino20} for $H$ in the case $l=j/2$. Finally, \eqref{buz13}, \eqref{buz14}, \eqref{buz15}, \eqref{buz16}, \eqref{buz17} and \eqref{buz18} imply:
 \be\lab{buz19}
\norm{h_{2,p,q,l,m}}_{L^1(\MM)}\les (1+p^2)\ep^2(2^{\frac{j}{4}-\frac{l}{2}}+|\nu-\nu'|2^{(\frac{1}{3})_+j})2^{\frac{j}{2}+m}\gamma^\nu_j\gamma^{\nu'}_j.
\ee

Next, we evaluate the $L^1(\MM)$ of $h_{3,p,q,l,m}$. We first consider the case where $m=j/2$. In view of the definition \eqref{buz4} of $h_{3,p,q,l,m}$ with $m=j/2$, we have:
\bee
\norm{h_{3,p,q,l,j/2}}_{L^1(\MM)}&\les& \normm{\int_{\S} \nabb(P_l\trc)\left(2^{\frac{j}{2}}(N-N_\nu)\right)^pF_{j,-1}(u)\eta_j^\nu(\o)d\o}_{L^2(\MM)}\\
\nn&&\times\normm{\int_{\S}{\nabb'}^2(P_{\leq j/2}\trc')\left(2^{\frac{j}{2}}(N'-N_{\nu'})\right)^qF_{j,-1}(u')\eta_j^{\nu'}(\o')d\o'}_{L^2(\MM)}\\
\nn&\les& 2^{\frac{3j}{2}}\ep^2\gamma^\nu_j\gamma^{\nu'}_j,
\eee
where we used in the last inequality the analog of  the estimate \eqref{vino56} for the first term, and the analog of \eqref{buz10} for the second term. Since $l\geq j/2$, this yields in the case $m=j/2$:
\be\lab{buz20}
\norm{h_{3,p,q,l,j/2}}_{L^1(\MM)}\les 2^{\j+\frac{l}{2}}\ep^2\gamma^\nu_j\gamma^{\nu'}_j.
\ee
Next, we consider the case where $m>j/2$. Since $l\geq m$, we also have $l>j/2$. In view of the definition \eqref{buz4} of $h_{3,p,q,l,m}$, we have:
\bee
\norm{h_{3,p,q,l,m}}_{L^1(\MM)}&\les& \left(\int_{\S} \normm{\nabb(P_l\trc)\left(2^{\frac{j}{2}}(N-N_\nu)\right)^pF_{j,-1}(u)}_{L^2(\MM)}\eta_j^\nu(\o)d\o\right)\\
\nn&&\times\left(\int_{\S}\normm{{\nabb'}^2(P_m\trc')\left(2^{\frac{j}{2}}(N'-N_{\nu'})\right)^qF_{j,-1}(u')}_{L^2(\MM)}\eta_j^{\nu'}(\o')d\o'\right)\\
\nn&\les& \left(\int_{\S} \normm{\norm{\nabb(P_l\trc)}_{L^2(\H_u)}F_{j,-1}(u)}_{L^2_u}\eta_j^\nu(\o)d\o\right)\\
\nn&&\times\left(\int_{\S}\normm{\norm{{\nabb'}^2(P_m\trc')}_{L^2(\H_{u'}}F_{j,-1}(u')}_{L^2_{u'}}\eta_j^{\nu'}(\o')d\o'\right),
\eee
where we used in the last inequality the estimate \eqref{estNomega} for $\po N$ and the size of the patch. Taking Cauchy Schwartz in $\o$ and $\o'$, using the size of the patches, and using the Bochner inequality \eqref{eq:Bochconseqbis} and the finite band property for $P_l$ and $P_m$, we obtain:
\bea\lab{buz21}
\norm{h_{3,p,q,l,m}}_{L^1(\MM)}&\les& 2^{2m+l-j}\normm{\norm{P_l\trc}_{L^2(\H_u)}F_j(u)\sqrt{\eta^\nu_j(\o)}}_{L^2_{\o,u}}\\
\nn&&\times\normm{\norm{P_m'\trc}_{L^2(\H_u)}F_j(u)\sqrt{\eta^{\nu'}_j(\o)}}_{L^2_{\o,u}}.
\eea
In view of \eqref{buz21} and the estimate \eqref{tsonga2}, we finally obtain in the case $m>j/2$:
\be\lab{buz22}
\norm{h_{3,p,q,l,m}}_{L^1(\MM)}\les 2^{m+l+\min(l,m)-j}\norm{\mu_{j,\nu,l}}_{L^2(\R\times\S)}\norm{\mu_{j,\nu',m}}_{L^2(\R\times\S)},
\ee
where the sequence of functions $(\mu_{j,\nu,l})_{l> j/2}$ on $\R\times\S$ satisfies:
$$\sum_{\nu}\sum_{l> j/2}2^{2l}\norm{\mu_{j,\nu,l}}^2_{L^2(\R\times\S)}\les \ep^2 2^{2j}\norm{f}^2_{L^2(\R^3)}.$$

Next, we evaluate the $L^1(\MM)$ of $h_{4,p,q,l,m}$. We first consider the case $m=j/2$. In view of the definition \eqref{buz5} of $h_{4,p,q,l,m}$ with $m=j/2$, we have:
$$h_{4,p,q,l,j/2} = \int_{\S}H{\nabb'}^3(P_{\leq j/2}\trc')\left(2^{\frac{j}{2}}(N'-N_{\nu'})\right)^qF_{j,-1}(u')\eta_j^{\nu'}(\o')d\o'.$$
where $H$ is given by:
\be\lab{buz23}
H=\int_{\S}P_l\trc\left(2^{\frac{j}{2}}(N-N_\nu)\right)^pF_{j,-1}(u)\eta_j^\nu(\o)d\o.
\ee
This yields:
\bea\lab{buz24}
&&\norm{h_{4,p,q,l,j/2}}_{L^1(\MM)}\\
\nn&\les& \int_{\S}\normm{H{\nabb'}^3(P_{\leq j/2}\trc')\left(2^{\frac{j}{2}}(N'-N_{\nu'})\right)^qF_{j,-1}(u')}_{L^1(\MM)}\eta_j^{\nu'}(\o')d\o'\\
\nn&\les& \int_{\S}\normm{\norm{H}_{L^2(P_{t,u'})}\norm{{\nabb'}^3(P_{\leq j/2}\trc')}_{L^2(P_{t,u'})}F_{j,-1}(u')}_{L^1_{u',t}}\eta_j^{\nu'}(\o')d\o',
\eea
where we used in the last inequality the estimate \eqref{estNomega} for $\po N$ and the size of the patch. 
Now, in view of the estimate \eqref{eq:Bochconseqter}, we have:
\bea\lab{buz25}
&&\norm{{\nabb'}^3(P_{\leq j/2}\trc')}_{L^2(P_{t,u'})}\\
\nn&\les& \norm{\nabb\lap(P_{\leq j/2}\trc')}_{L^2(P_{t,u'})}+\mu(t)\norm{\lap (P_{\leq j/2}\trc')}_{L^2(P_{t,u'})}+\mu^2(t)\norm{\nabb (P_{\leq j/2}\trc')}_{L^2(P_{t,u'})}\\
\nn&\les& (2^j+\mu(t)2^{\frac{j}{2}}+\mu(t)^2)\ep,,
\eea
where we used in the last inequality the finite band property for $P_{\leq j/2}$ and the estimate \eqref{esttrc} for $\trc$, and where $\mu$ in a function in $L^2(\mathbb{R})$ satisfying:
\be\lab{buz26}
\norm{\mu}_{L^2(\mathbb{R})}\les 1.
\ee
\eqref{buz24}, \eqref{buz25} and \eqref{buz26} yield:
\bea\lab{buz27}
\norm{h_{4,p,q,l,j/2}}_{L^1(\MM)} &\les& 2^j\norm{H}_{L^2(\MM)}\left(\int_{\S}\norm{F_{j,-1}(u')}_{L^2_{u'}}\eta_j^{\nu'}(\o')d\o'\right)\\
\nn&& +\left(\sup_{\o'}\norm{H}_{L^2_{u'}L^\infty_t L^2_{x'}}\right)\left(\int_{\S}\norm{F_{j,-1}(u')}_{L^2_{u'}}\eta_j^{\nu'}(\o')d\o'\right)\\
\nn&\les& 2^{2j-l}\ep^2\gamma^\nu_j\gamma^{\nu'}_j+\left(\sup_{\o'}\norm{H}_{L^2_{u'}L^\infty_t L^2_{x'}}\right)2^{\frac{j}{2}}\ep\gamma^{\nu'}_j,
\eea
where we used in the last inequality the estimate \eqref{zol13} for $H$ which holds in view of the definition \eqref{buz23} of $H$, Plancherel in $\la'$ for $\norm{F_{j,-1}(u')}_{L^2_{u'}}$, Cauchy Schwarz in $\o'$ and the size of the patch. Using the estimate \eqref{messi4:0} in the case $l>j/2$, and the estimates \eqref{messi4:0} \eqref{koko1} in the case $l=j/2$ together with the decomposition
$$P_{\leq j/2}\trc=\trc-\sum_{l>\frac{j}{2}}P_l\trc,$$
we obtain:
\be\lab{buz27bis}
\sup_{\o'}\norm{H}_{L^2_{u', x'}L^\infty_t}\les (1+p^{\frac{5}{2}})\ep\big(2^{\frac{j}{2}}|\nu-\nu'|2^{-l+\frac{j}{2}}+(2^{\frac{j}{2}}|\nu-\nu'|)^{\frac{1}{2}}2^{-\frac{l}{2}+\frac{j}{4}}\big)\gamma^\nu_j.
\ee
Together with \eqref{buz27}, this yields in the case $m=j/2$:
\be\lab{buz28}
\norm{h_{4,p,q,l,j/2}}_{L^1(\MM)} \les 2^{2j-l}\ep^2\gamma^\nu_j\gamma^{\nu'}_j+(1+p^{\frac{5}{2}})\big(2^{\frac{j}{2}}|\nu-\nu'|2^{-l+j}+(2^{\frac{j}{2}}|\nu-\nu'|)^{\frac{1}{2}}2^{-\frac{l}{2}+\frac{3j}{4}}\big)\ep^2\gamma^\nu_j\gamma^{\nu'}_j.
\ee
Next, we consider the case where $m>j/2$. Since $l\geq m$, we also have $l>j/2$. In view of the definition \eqref{buz5} of $h_{4,p,q,l,m}$, we have:
$$h_{4,p,q,l,m} = \int_{\S}H{\nabb'}^3(P_m\trc')\left(2^{\frac{j}{2}}(N'-N_{\nu'})\right)^qF_{j,-1}(u')\eta_j^{\nu'}(\o')d\o'.$$
where $H$ is given by \eqref{buz23}. This yields:
This yields:
\bea\lab{buz29}
&&\norm{h_{4,p,q,l,m}}_{L^1(\MM)} \\
\nn&\les& \int_{\S}\normm{H{\nabb'}^3(P_m\trc')\left(2^{\frac{j}{2}}(N'-N_{\nu'})\right)^qF_{j,-1}(u')}_{L^1(\MM)}\eta_j^{\nu'}(\o')d\o'\\
\nn&\les& \int_{\S}\normm{\norm{H}_{L^2(P_{t,u'})}\norm{{\nabb'}^3(P_m\trc')}_{L^2(P_{t,u'})}F_{j,-1}(u')}_{L^1_{u',t}}\eta_j^{\nu'}(\o')d\o',
\eea
where we used in the last inequality the estimate \eqref{estNomega} for $\po N$ and the size of the patch. 
Now, in view of the estimate \eqref{eq:Bochconseqter}, we have:
\bea\lab{buz30}
&&\norm{{\nabb'}^3(P_m\trc')}_{L^2(P_{t,u'})}\\
\nn&\les& \norm{\nabb\lap(P_m\trc')}_{L^2(P_{t,u'})}+\mu(t)\norm{\lap (P_m\trc')}_{L^2(P_{t,u'})}+\mu^2(t)\norm{\nabb (P_m\trc')}_{L^2(P_{t,u'})}\\
\nn&\les& 2^{3m}\norm{P_m\trc'}_{L^2(P_{t,u'})}+2^m\mu(t)^2\ep,
\eea
where we used in the last inequality the finite band property for $P_m$ and the estimate \eqref{esttrc} for $\trc$, and where $\mu$ in a function in $L^2(\mathbb{R})$ satisfying \eqref{buz26}. \eqref{buz29}, \eqref{buz30} and \eqref{buz26} yield:
\bea\lab{buz31}
&&\norm{h_{4,p,q,l,m}}_{L^1(\MM)}\\ 
\nn&\les& 2^{3m}\norm{H}_{L^2(\MM)}\left(\int_{\S}\normm{\norm{P_m\trc'}_{L^2(\H_{u'})}F_{j,-1}(u')}_{L^2_{u'}}\eta_j^{\nu'}(\o')d\o'\right)\\
\nn&& +2^m\ep\left(\sup_{\o'}\norm{H}_{L^2_{u'}L^\infty_t L^2_{x'}}\right)\left(\int_{\S}\norm{F_{j,-1}(u')}_{L^2_{u'}}\eta_j^{\nu'}(\o')d\o'\right)\\
\nn&\les& 2^{3m}\norm{H}_{L^2(\MM)}\left(\int_{\S}\normm{\norm{P_m\trc'}_{L^2(\H_{u'})}F_{j,-1}(u')}_{L^2_{u'}}\eta_j^{\nu'}(\o')d\o'\right)\\
\nn&&+ (1+p^{\frac{5}{2}})\big(2^{\frac{j}{2}}|\nu-\nu'|2^{-l+j+m}+(2^{\frac{j}{2}}|\nu-\nu'|)^{\frac{1}{2}}2^{-\frac{l}{2}+\frac{3j}{4}+m}\big)\ep^2\gamma^\nu_j\gamma^{\nu'}_j,
\eea
where we used in the last inequality the estimate \eqref{buz27bis} for $H$, Plancherel in $\la'$ for $\norm{F_{j,-1}(u')}_{L^2_{u'}}$, Cauchy Schwarz in $\o'$ and the size of the patch. In view of the definition \eqref{buz23} of $H$, and using the estimate \eqref{estNomega} for $\po N$ and the size of the patch, we have:
$$\norm{H}_{L^2(\MM)}\les \int_{\S}\normm{\norm{P_l\trc}_{L^2(\H_u)}F_{j,-1}(u)}_{L^2_u}\eta_j^\nu(\o)d\o.$$
Together with \eqref{buz31}, this yields:
\bee
\norm{h_{4,p,q,l,m}}_{L^1(\MM)} &\les& 2^{3m}\left(\int_{\S}\normm{\norm{P_l\trc}_{L^2(\H_u)}F_{j,-1}(u)}_{L^2_u}\eta_j^\nu(\o)d\o\right)\\
\nn&&\times\left(\int_{\S}\normm{\norm{P_m\trc'}_{L^2(\H_{u'})}F_{j,-1}(u')}_{L^2_{u'}}\eta_j^{\nu'}(\o')d\o'\right)\\
\nn&&+ (1+p^{\frac{5}{2}})\big(2^{\frac{j}{2}}|\nu-\nu'|2^{-l+j+m}+(2^{\frac{j}{2}}|\nu-\nu'|)^{\frac{1}{2}}2^{-\frac{l}{2}+\frac{3j}{4}+m}\big)\ep^2\gamma^\nu_j\gamma^{\nu'}_j.
\eee
Taking Cauchy Schwartz in $\o$ and $\o'$, using the size of the patches, and using the Bochner inequality \eqref{eq:Bochconseqbis} and the finite band property for $P_l$ and $P_m$, we obtain:
\bea\lab{buz32}
\norm{h_{4,p,q,l,m}}_{L^1(\MM)}&\les& 2^{3m-j}\normm{\norm{P_l\trc}_{L^2(\H_u)}F_j(u)\sqrt{\eta^\nu_j(\o)}}_{L^2_{\o,u}}\\
\nn&&\times\normm{\norm{P_m'\trc}_{L^2(\H_u)}F_j(u)\sqrt{\eta^{\nu'}_j(\o)}}_{L^2_{\o,u}}\\
\nn&& + (1+p^{\frac{5}{2}})\big(2^{\frac{j}{2}}|\nu-\nu'|2^{-l+j+m}+(2^{\frac{j}{2}}|\nu-\nu'|)^{\frac{1}{2}}2^{-\frac{l}{2}+\frac{3j}{4}+m}\big)\ep^2\gamma^\nu_j\gamma^{\nu'}_j.
\eea
In view of \eqref{buz32} and the estimate \eqref{tsonga2}, we finally obtain in the case $m>j/2$:
\bea\lab{buz33}
\norm{h_{4,p,q,l,m}}_{L^1(\MM)}&\les& 2^{m+l+\min(l,m)-j}\norm{\mu_{j,\nu,l}}_{L^2(\R\times\S)}\norm{\mu_{j,\nu',m}}_{L^2(\R\times\S)}\\
\nn&&+ (1+p^{\frac{5}{2}})\big(2^{\frac{j}{2}}|\nu-\nu'|2^{-l+j+m}+(2^{\frac{j}{2}}|\nu-\nu'|)^{\frac{1}{2}}2^{-\frac{l}{2}+\frac{3j}{4}+m}\big)\ep^2\gamma^\nu_j\gamma^{\nu'}_j,
\eea
where the sequence of functions $(\mu_{j,\nu,l})_{l> j/2}$ on $\R\times\S$ satisfies:
$$\sum_{\nu}\sum_{l> j/2}2^{2l}\norm{\mu_{j,\nu,l}}^2_{L^2(\R\times\S)}\les \ep^2 2^{2j}\norm{f}^2_{L^2(\R^3)}.$$

Next, we evaluate the $L^1(\MM)$ of $h_{5,p,q,l,m}$. In view of the definition \eqref{buz6} of $h_{5,p,q,l,m}$, we have:
\bea\lab{buz34}
\norm{h_{5,p,q,l,m}}_{L^1(\MM)}&\les& \normm{\int_{\S}(\th+b^{-1}\nabla(b)) P_l\trc\left(2^{\frac{j}{2}}(N-N_\nu)\right)^pF_{j,-1}(u)\eta_j^\nu(\o)d\o}_{L^2(\MM)}\\
\nn&&\times\normm{\int_{\S}{\nabb'}^2(P_m\trc')\left(2^{\frac{j}{2}}(N'-N_{\nu'})\right)^qF_{j,-1}(u')\eta_j^{\nu'}(\o')d\o'}_{L^2(\MM)}\\
\nn&\les& 2^{(\frac{3}{4})_+j+m}\ep^2\gamma^\nu_j\gamma^{\nu'}_j,
\eea
where we used in the last inequality the estimate \eqref{vino24} for the first term and the estimate \eqref{buz10} for the second term.

Next, we evaluate the $L^1(\MM)$ of $h_{6,p,q,l,m}$. In view of the definition \eqref{buz7} of $h_{6,p,q,l,m}$, we have:
\bea\lab{buz35}
&&\norm{h_{6,p,q,l,m}}_{L^1(\MM)}\\
\nn&\les& \normm{\int_{\S}P_l\trc\left(2^{\frac{j}{2}}(N-N_\nu)\right)^pF_{j,-1}(u)\eta_j^\nu(\o)d\o}_{L^2(\MM)}\\
\nn&&\times\normm{\int_{\S}(\th'+{b'}^{-1}\nabb'(b')){\nabb'}^2(P_m\trc')\left(2^{\frac{j}{2}}(N'-N_{\nu'})\right)^qF_{j,-1}(u')\eta_j^{\nu'}(\o')d\o'}_{L^2(\MM)}\\
\nn&\les& 2^{\frac{j}{2}-l}\ep\gamma^\nu_j\normm{\int_{\S}(\th'+{b'}^{-1}\nabb'(b')){\nabb'}^2(P_m\trc')\left(2^{\frac{j}{2}}(N'-N_{\nu'})\right)^qF_{j,-1}(u')\eta_j^{\nu'}(\o')d\o'}_{L^2(\MM)},
\eea
where we used in the last inequality the estimate \eqref{zol13}. The basic estimate in $L^2(\MM)$ \eqref{oscl2bis} implies:
\bea\lab{buz36}
&&\normm{\int_{\S}(\th'+{b'}^{-1}\nabb'(b')){\nabb'}^2(P_m\trc')\left(2^{\frac{j}{2}}(N'-N_{\nu'})\right)^qF_{j,-1}(u')\eta_j^{\nu'}(\o')d\o'}_{L^2(\MM)}\\
\nn&\les& \left(\sup_{\o'}\normm{(\th'+{b'}^{-1}\nabb'(b')){\nabb'}^2(P_m\trc')\left(2^{\frac{j}{2}}(N'-N_{\nu'})\right)^q}_{\lprime{\infty}{2}}\right)2^{\frac{j}{2}}\gamma^{\nu'}_j\\
\nn&\les& \left(\sup_{\o'}\norm{\th'+{b'}^{-1}\nabb'(b')}_{\tx{\infty}{4}}\norm{{\nabb'}^2(P_m\trc')}_{\tx{2}{4}}\normm{\left(2^{\frac{j}{2}}(N'-N_{\nu'})\right)^q}_{L^\infty}\right)2^{\frac{j}{2}}\gamma^{\nu'}_j\\
\nn&\les& \left(\sup_{\o'}\norm{{\nabb'}^2(P_m\trc')}_{\tx{2}{4}}\right)\ep 2^{\frac{j}{2}}\gamma^{\nu'}_j,
\eea
where we used in the last inequality the estimates \eqref{estk} \eqref{esttrc} \eqref{esthch} for $\chi'$, the estimate \eqref{estb} for $b'$, the estimate \eqref{estNomega} for $\po N$ and the size of the patch. Now, the Gagliardo-Nirenberg inequality \eqref{eq:GNirenberg} yields:
\bee
\norm{{\nabb'}^2(P_m\trc')}_{L^4(P_{t,u'})}&\les& \norm{{\nabb'}^2(P_m\trc')}^{\frac{1}{2}}_{L^2(P_{t,u'})}\norm{{\nabb'}^3(P_m\trc')}^{\frac{1}{2}}_{L^2(P_{t,u'})}\\
\nn&\les& 2^{\frac{3m}{2}}\ep(1+\mu(t)),
\eee
where we used in the last inequality the Bochner inequality \eqref{eq:Bochconseqbis}, the finite band property for $P_m$, the estimates \eqref{buz25} and \eqref{buz30} for ${\nabb'}^3(P_m\trc')$, and the estimate \eqref{esttrc} for $\trc$. Together with the estimate \eqref{buz26} for $\mu$, we obtain:
$$\norm{{\nabb'}^2(P_m\trc')}_{\tx{2}{4}}\les 2^{\frac{3m}{2}}\ep,$$
which together with \eqref{buz36} yields:
\bea\lab{buz37}
&&\normm{\int_{\S}(\th'+{b'}^{-1}\nabb'(b')){\nabb'}^2(P_m\trc')\left(2^{\frac{j}{2}}(N'-N_{\nu'})\right)^qF_{j,-1}(u')\eta_j^{\nu'}(\o')d\o'}_{L^2(\MM)}\\
\nn&\les&  2^{\frac{3m}{2}+\frac{j}{2}}\ep^2\gamma^{\nu'}_j.
\eea
Finally, \eqref{buz35} and \eqref{buz37} imply:
\be\lab{buz38}
\norm{h_{6,p,q,l,m}}_{L^1(\MM)}\les 2^{j-l+\frac{3m}{2}}\ep^2\gamma^\nu_j\gamma^{\nu'}_j.
\ee

Next, we evaluate the $L^1(\MM)$ of $h_{7,p,q,l,m}$. In view of the definition \eqref{buz8} of $h_{7,p,q,l,m}$, we have:
\bea\lab{buz39}
\norm{h_{7,p,q,l,m}}_{L^1(\MM)}&\les& \normm{\int_{\S}N(P_l\trc)\left(2^{\frac{j}{2}}(N-N_\nu)\right)^pF_{j,-1}(u)\eta_j^\nu(\o)d\o}_{L^2(\MM)}\\
\nn&&\times\normm{\int_{\S}{\nabb'}^2(P_m\trc')\left(2^{\frac{j}{2}}(N'-N_{\nu'})\right)^qF_{j,-1}(u')\eta_j^{\nu'}(\o')d\o'}_{L^2(\MM)}\\
\nn&\les& 2^{j+m}\ep^2\gamma^\nu_j\gamma^{\nu'}_j,
\eea
where we used in the last inequality the analog of the estimate \eqref{vino56} for the first term, and the estimate \eqref{buz10} for the second term.

Finally, we estimate $A^1_{j,\nu,\nu',l,m}$. In view of the decomposition \eqref{buz1} of $A^1_{j,\nu,\nu',l,m}$, the estimate \eqref{nice26}, and the estimates \eqref{buz11} \eqref{buz19} \eqref{buz20} \eqref{buz22} \eqref{buz28} \eqref{buz33} \eqref{buz34} \eqref{buz38} \eqref{buz39} for $h_{1,p,q,l,m}$, $h_{2,p,q,l,m}$, $h_{3,p,q,l,m}$, $h_{4,p,q,l,m}$, $h_{5,p,q,l,m}$, $h_{6,p,q,l,m}$, $h_{7,p,q,l,m}$, we obtain:
\bee
&&\left|\sum_{(l,m)/2^{\max(l,m)}\leq 2^j|\nu-\nu'|}A^1_{j,\nu,\nu',l,m}\right|\\
\nn&\les& \sum_{(l,m)/2^{\max(l,m)}\leq 2^j|\nu-\nu'|}\sum_{p, q\geq 0}c_{pq}\frac{2^{-\frac{5j}{2}}2^{l+m+\min(l,m)}}{(2^{\frac{j}{2}}|\nu-\nu'|)^{p+q+3}}\norm{\mu_{j,\nu,l}}_{L^2(\R\times\S)}\norm{\mu_{j,\nu',m}}_{L^2(\R\times\S)}\\
\nn&&+\sum_{p, q\geq 0}c_{pq}\frac{1}{(2^{\frac{j}{2}}|\nu-\nu'|)^{p+q+2}}\Bigg[\frac{1+p^2}{2^{\frac{j}{2}}|\nu-\nu'|}+(1+p^{\frac{5}{2}})2^{-(\frac{1}{6})_-j}+2^{-\frac{j}{2}}(2^{\frac{j}{2}}|\nu-\nu'|)\Bigg]\ep^2\gamma^\nu_j\gamma^{\nu'}_j\\
\nn&\les&   \sum_{(l,m)/2^{\max(l,m)}\leq 2^j|\nu-\nu'|}\frac{2^{-\frac{5j}{2}}2^{l+m+\min(l,m)}}{(2^{\frac{j}{2}}|\nu-\nu'|)^3}\norm{\mu_{j,\nu,l}}_{L^2(\R\times\S)}\norm{\mu_{j,\nu',m}}_{L^2(\R\times\S)}\\
\nn&&+\Bigg[\frac{1}{(2^{\frac{j}{2}}|\nu-\nu'|)^3}+\frac{2^{-(\frac{1}{6})_-j}}{(2^{\frac{j}{2}}|\nu-\nu'|)^2}+\frac{1}{2^{\frac{j}{2}}(2^{\frac{j}{2}}|\nu-\nu'|)}\Bigg]\ep^2\gamma^\nu_j\gamma^{\nu'}_j,
\eee
where the sequence of functions $(\mu_{j,\nu,l})_{l>j/2}$ on $\R\times\S$ satisfies:
$$\sum_{\nu}\sum_{l>j/2}2^{2l}\norm{\mu_{j,\nu,l}}^2_{L^2(\R\times\S)}\les \ep^2 2^{2j}\norm{f}^2_{L^2(\R^3)}.$$
This concludes the proof of Proposition \ref{prop:buz}.

%%%%%%%%%%%%%%%%%%%%%%%%%%%%%%%%%%%%%%%%%%%%%%

\subsection{Proof of Proposition \ref{prop:biz} (Control of $A^2_{j,\nu,\nu',l,m}$)}\lab{sec:biz}

In order to prove Proposition \ref{prop:biz}, recall that we need to show:
\bee
&&\left|\sum_{(l,m)/2^{\max(l,m)}\leq 2^j|\nu-\nu'|}A^2_{j,\nu,\nu',l,m}\right|\\
\nn&\les& \sum_{(l,m)/2^{\max(l,m)}\leq 2^j|\nu-\nu'|}\frac{2^{-\frac{5j}{2}}2^{l+m+\min(l,m)}}{(2^{\frac{j}{2}}|\nu-\nu'|)^3}\norm{\mu_{j,\nu,l}}_{L^2(\R\times\S)}\norm{\mu_{j,\nu',m}}_{L^2(\R\times\S)}\\
\nn&&+\Bigg[\frac{1}{(2^{\frac{j}{2}}|\nu-\nu'|)^3}+\frac{1}{(2^{\frac{j}{2}}|\nu-\nu'|)^{\frac{5}{2}}}+\frac{1}{2^{\frac{j}{2}}(2^{\frac{j}{2}}|\nu-\nu'|)}\Bigg]]\ep^2\gamma^\nu_j\gamma^{\nu'}_j,
\eee
where the sequence of functions $(\mu_{j,\nu,l})_{l>j/2}$ on $\R\times\S$ satisfies:
$$\sum_{\nu}\sum_{l>j/2}2^{2l}\norm{\mu_{j,\nu,l}}^2_{L^2(\R\times\S)}\les \ep^2 2^{2j}\norm{f}^2_{L^2(\R^3)},$$
and where $A^2_{j,\nu,\nu',l,m}$ is given by \eqref{zol3}:
\bea
\nn A^2_{j,\nu,\nu',l,m}&=& 2^{-2j}\int_{\MM}\int_{\S\times\S}\frac{(N'-\gn N)(P_l\trc) (N-\gn N')(P_m(\trc'))}{\gl(1-\gn^2)}\\
\lab{biz0}&&\times F_j(u)F_{j,-1}(u')\eta_j^\nu(\o)\eta_j^{\nu'}(\o')d\o d\o' d\MM,
\eea

We integrate by parts using \eqref{fetebis}. 
\begin{lemma}\lab{lemma:biz}
Let $A^2_{j,\nu,\nu',l,m}$ be defined by \eqref{biz0}. Integrating by parts using \eqref{fetebis} yields:
\bea\lab{biz1}
&& A^2_{j,\nu,\nu',l,m}\\
\nn&=& 2^{-2j}\sum_{p, q\geq 0}c_{pq}\int_{\MM}\frac{1}{(2^{\frac{j}{2}}|N_\nu-N_{\nu'}|)^{p+q+2}}\Bigg[\frac{1}{|N_\nu-N_{\nu'}|^2}(h_{1,p,q,l,m}'+h_{2,p,q,l,m}'\\
\nn&& +h_{3,p,q,l,m}'+h_{4,p,q,l,m}')+\frac{1}{|N_\nu-N_{\nu'}|}(h_{5,p,q,l,m}'+h_{6,p,q,l,m}'+h_{7,p,q,l,m}')+h_{8,p,q,l,m}'\Bigg] d\MM,
\eea
where $c_{pq}$ are explicit real coefficients such that the series 
$$\sum_{p, q\geq 0}c_{pq}x^py^q$$
has radius of convergence 1, where the scalar functions $h_{1,p,q,l,m}$, $h_{2,p,q,l,m}'$, $h_{3,p,q,l,m}'$, $h_{4,p,q,l,m}'$, $h_{5,p,q,l,m}'$, $h_{6,p,q,l,m}'$, $h_{7,p,q,l,m}'$ on $\MM$ are given by:
\bea\lab{biz2}
h_{1,p,q,l,m}'&=& \left(\int_{\S} \nabb(L(P_l\trc))\left(2^{\frac{j}{2}}(N-N_\nu)\right)^pF_{j,-1}(u)\eta_j^\nu(\o)d\o\right)\\
\nn&&\times\left(\int_{\S}\nabb'(P_m\trc')\left(2^{\frac{j}{2}}(N'-N_{\nu'})\right)^qF_{j,-1}(u')\eta_j^{\nu'}(\o')d\o'\right),
\eea
\bea\lab{biz3}
h_{2,p,q,l,m}'&=& \left(\int_{\S}\nabb(P_l\trc)\left(2^{\frac{j}{2}}(N-N_\nu)\right)^pF_{j,-1}(u)\eta_j^\nu(\o)d\o\right)\\
\nn&&\times\left(\int_{\S}\nabb'(L'(P_m\trc'))\left(2^{\frac{j}{2}}(N'-N_{\nu'})\right)^qF_{j,-1}(u')\eta_j^{\nu'}(\o')d\o'\right),
\eea
\bea\lab{biz4}
h_{3,p,q,l,m}'&=& \left(\int_{\S}H_1 \nabb(P_l\trc)\left(2^{\frac{j}{2}}(N-N_\nu)\right)^pF_{j,-1}(u)\eta_j^\nu(\o)d\o\right)\\
\nn&&\times\left(\int_{\S}\nabb'(P_m\trc')\left(2^{\frac{j}{2}}(N'-N_{\nu'})\right)^qF_{j,-1}(u')\eta_j^{\nu'}(\o')d\o'\right),
\eea
\bea\lab{biz5}
h_{4,p,q,l,m}'&=& \left(\int_{\S}\nabb(P_l\trc)\left(2^{\frac{j}{2}}(N-N_\nu)\right)^pF_{j,-1}(u)\eta_j^\nu(\o)d\o\right)\\
\nn&&\times\left(\int_{\S}H_2\nabb'(P_m\trc')\left(2^{\frac{j}{2}}(N'-N_{\nu'})\right)^qF_{j,-1}(u')\eta_j^{\nu'}(\o')d\o'\right),
\eea
\bea\lab{biz6}
h_{5,p,q,l,m}'&=& \left(\int_{\S}\nabb(P_l\trc)\left(2^{\frac{j}{2}}(N-N_\nu)\right)^pF_{j,-1}(u)\eta_j^\nu(\o)d\o\right)\\
\nn&&\times\left(\int_{\S}{\nabb'}^2(P_m\trc')\left(2^{\frac{j}{2}}(N'-N_{\nu'})\right)^qF_{j,-1}(u')\eta_j^{\nu'}(\o')d\o'\right),
\eea
\bea\lab{biz7}
h_{6,p,q,l,m}'&=& \left(\int_{\S}H_3\nabla(P_l\trc)\left(2^{\frac{j}{2}}(N-N_\nu)\right)^pF_{j,-1}(u)\eta_j^\nu(\o)d\o\right)\\
\nn&&\times\left(\int_{\S}\nabla(P_m\trc')\left(2^{\frac{j}{2}}(N'-N_{\nu'})\right)^qF_{j,-1}(u')\eta_j^{\nu'}(\o')d\o'\right),
\eea
\bea\lab{biz8}
h_{7,p,q,l,m}'&=& \left(\int_{\S}\nabla(P_l\trc)\left(2^{\frac{j}{2}}(N-N_\nu)\right)^pF_{j,-1}(u)\eta_j^\nu(\o)d\o\right)\\
\nn&&\times\left(\int_{\S}H_4\nabla(P_m\trc')\left(2^{\frac{j}{2}}(N'-N_{\nu'})\right)^qF_{j,-1}(u')\eta_j^{\nu'}(\o')d\o'\right),
\eea
and:
\bea\lab{biz9}
h_{8,p,q,l,m}'&=& \left(\int_{\S}\nabb(P_l\trc)\left(2^{\frac{j}{2}}(N-N_\nu)\right)^pF_{j,-1}(u)\eta_j^\nu(\o)d\o\right)\\
\nn&&\times\left(\int_{\S}\nabb'(N'(P_m\trc'))\left(2^{\frac{j}{2}}(N'-N_{\nu'})\right)^qF_{j,-1}(u')\eta_j^{\nu'}(\o')d\o'\right),
\eea
and where the tensors $H_1$, $H_2$, $H_3$ and $H_4$ are given by:
\be\lab{biz10}
H_1=\chi+\kep+\d+n^{-1}\nabla n+L(b),
\ee
\be\lab{biz11}
H_2=\chi'+\kep'+\d'+n^{-1}\nabla n+L'(b'),
\ee
\be\lab{biz12}
H_3=k+n^{-1}\nabla n+\th+b^{-1}\nabb(b)+\chi+\z,
\ee
and:
\be\lab{biz13}
H_4=k+n^{-1}\nabla n+\th'+{b'}^{-1}\nabla(b')+\z'.
\ee
\end{lemma}
The proof of Lemma \ref{lemma:biz} is postponed to Appendix L. In the rest of this section, we use Lemma \ref{lemma:biz} to obtain the control of $A^2_{j,\nu,\nu',l,m}$.

We evaluate the $L^1(\MM)$ norm of $h_{1,p,q,l,m}'$, $h_{2,p,q,l,m}'$, $h_{3,p,q,l,m}'$, $h_{4,p,q,l,m}'$, $h_{5,p,q,l,m}'$, $h_{6,p,q,l,m}'$, $h_{7,p,q,l,m}'$ starting with $h_{1,p,q,l,m}'$. In view of the definition \eqref{biz2} of $h_{1,p,q,l,m}'$, we have:
$$h_{1,p,q,l,m}'= \int_{\S} G\nabb(L(P_l\trc))\left(2^{\frac{j}{2}}(N-N_\nu)\right)^pF_{j,-1}(u)\eta_j^\nu(\o)d\o,$$
where $G$ is given by:
\be\lab{biz14}
G=\int_{\S}\nabb'(P_m\trc')\left(2^{\frac{j}{2}}(N'-N_{\nu'})\right)^qF_{j,-1}(u')\eta_j^{\nu'}(\o')d\o'.
\ee
Using the analog of \eqref{nadal:1}, this yields:
\bea\lab{biz15}
\norm{h_{1,p,q,l,m}'}_{L^1(\MM)}&\les& \left(\sup_{\o'\in\textrm{supp}(\eta^{\nu'}_j}\normm{G\left(2^{\frac{j}{2}}(N-N_\nu)\right)^p}_{L^2_{u, x'}L^\infty_t}\right)2^{\frac{j}{2}}\ep\gamma^\nu_j\\
\nn&\les& \left(\sup_{\o\in\textrm{supp}(\eta^\nu_j)}\norm{G}_{L^2_{u, x'}L^\infty_t}\right)2^{\frac{j}{2}}\ep\gamma^\nu_j,
\eea
where we used in the last inequality the estimate \eqref{estNomega} for $\po N$ and the size of the patch. Now, 
in view of the definition \eqref{biz14} of $G$, the analog of the estimate \eqref{messi11} yields:
\bee
\norm{G}_{L^2_{u, x'}L^\infty_t}&\les& \left(\sup_{\o'}\normm{\left(2^{\frac{j}{2}}(N'-N_{\nu'})\right)^q}_{L^\infty}\right)\ep\big(2^{\frac{j}{2}}|\nu-\nu'|2^{\frac{j}{2}}+(2^{\frac{j}{2}}|\nu-\nu'|)^{\frac{1}{2}}2^{\frac{m}{2}+\frac{j}{4}}\big)\gamma^{\nu'}_j\\
&\les & \ep 2^{\frac{j}{2}}(2^{\frac{j}{2}}|\nu-\nu'|)\gamma^{\nu'}_j,
\eee
where we used in the last inequality the fact that $2^m\leq 2^j|\nu-\nu'|$, the estimate \eqref{estNomega} for $\po N$ and the size of the patch. Together with \eqref{biz15}, we obtain:
$$\norm{h_{1,p,q,l,m}'}_{L^1(\MM)}\les 2^j(2^{\frac{j}{2}}|\nu-\nu'|)\ep^2\gamma^\nu_j\gamma^{\nu'}_j.$$
Now, since $m\geq j/2$, we finally obtain:
\be\lab{biz16}
\norm{h_{1,p,q,l,m}'}_{L^1(\MM)}\les 2^{\frac{3j}{4}}2^{\frac{m}{2}}(2^{\frac{j}{2}}|\nu-\nu'|)\ep^2\gamma^\nu_j\gamma^{\nu'}_j.
\ee

Next, we evaluate the $L^1(\MM)$ norm of $h_{2,p,q,l,m}'$. Comparing the definition \eqref{biz3} of $h_{2,p,q,l,m}'$ and \eqref{biz2} of $h_{1,p,q,l,m}'$, we notice that these terms are similar. We proceed as for $h_{1,p,q,l,m}'$, and we obtain the analog of \eqref{biz16}:
\be\lab{biz17}
\norm{h_{2,p,q,l,m}'}_{L^1(\MM)}\les 2^{\frac{3j}{4}}2^{\frac{m}{2}}(2^{\frac{j}{2}}|\nu-\nu'|)\ep^2\gamma^\nu_j\gamma^{\nu'}_j.
\ee

Next, we evaluate the $L^1(\MM)$ norm of $h_{3,p,q,l,m}'$. In view of the definition \eqref{biz4} of $h_{3,p,q,l,m}'$, we have:
\bea\lab{biz18}
\norm{h_{3,p,q,l,m}'}_{L^1(\MM)}&\les& \normm{\int_{\S}H_1 \nabb(P_l\trc)\left(2^{\frac{j}{2}}(N-N_\nu)\right)^pF_{j,-1}(u)\eta_j^\nu(\o)d\o}_{L^2(\MM)}\\
\nn&&\times\normm{\int_{\S}\nabb'(P_m\trc')\left(2^{\frac{j}{2}}(N'-N_{\nu'})\right)^qF_{j,-1}(u')\eta_j^{\nu'}(\o')d\o'}_{L^2(\MM)}\\
\nn&\les& \normm{\int_{\S}H_1 \nabb(P_l\trc)\left(2^{\frac{j}{2}}(N-N_\nu)\right)^pF_{j,-1}(u)\eta_j^\nu(\o)d\o}_{L^2(\MM)}2^{\frac{j}{2}}\ep\gamma^{\nu'}_j,
\eea
where we used in the last inequality the estimate the analog of the estimate \eqref{vino56}. Now, using the basic estimate \eqref{oscl2bis} in $L^2(\MM)$, we have:
\bea\lab{biz19}
&&\normm{\int_{\S}H_1 \nabb(P_l\trc)\left(2^{\frac{j}{2}}(N-N_\nu)\right)^pF_{j,-1}(u)\eta_j^\nu(\o)d\o}_{L^2(\MM)}\\
\nn&\les& \left(\sup_\o\normm{H_1 \nabb(P_l\trc)\left(2^{\frac{j}{2}}(N-N_\nu)\right)^p}_{\li{\infty}{2}}\right)2^{\frac{j}{2}}\ep\gamma^\nu_j\\
\nn&\les& \left(\sup_\o\norm{H_1}_{\xt{\infty}{2}}\norm{\nabb(P_l\trc)}_{\xt{2}{1}}\normm{\left(2^{\frac{j}{2}}(N-N_\nu)\right)^p}_{L^\infty}\right)2^{\frac{j}{2}}\ep\gamma^\nu_j\\
\nn&\les& \left(\sup_\o\norm{H_1}_{\xt{\infty}{2}}\right)\ep 2^{\frac{j}{2}}\gamma^\nu_j,
\eea
where we used in the last inequality the estimate \eqref{lievremont2} for $\nabb'(P_{\leq l}\trc')$, the estimate \eqref{estNomega} for $\po N$ and the size of the patch. In view of the definition \eqref{biz10} of $H_1$, we have:
\bee
\norm{H_1}_{\xt{\infty}{2}}&\les& \norm{\chi}_{\xt{\infty}{2}}+\norm{\kep}_{\xt{\infty}{2}}+\norm{\d}_{\xt{\infty}{2}}+\norm{n^{-1}\nabla n}_{\xt{\infty}{2}}+\norm{L(b)}_{\xt{\infty}{2}}\\
&\les& \ep,
\eee
where we used in the last inequality the estimates \eqref{esttrc} \eqref{esthch} for $\chi$, the estimates \eqref{estn} \eqref{estk} for $\kep$, the estimate \eqref{estk} for $\d$, the estimate \eqref{estn} for $n$ and the estimate \eqref{estb} for $b$. Together with \eqref{biz19}, this yields:
\be\lab{biz20}
\normm{\int_{\S}H_1 \nabb(P_l\trc)\left(2^{\frac{j}{2}}(N-N_\nu)\right)^pF_{j,-1}(u)\eta_j^\nu(\o)d\o}_{L^2(\MM)}\les\ep 2^{\frac{j}{2}}\gamma^\nu_j.
\ee
\eqref{biz18} and \eqref{biz20} imply:
$$\norm{h_{3,p,q,l,m}'}_{L^1(\MM)}\les 2^j\ep^2\gamma^\nu_j\gamma^{\nu'}_j.$$
Now, since $m\geq j/2$, we finally obtain:
\be\lab{biz21}
\norm{h_{3,p,q,l,m}'}_{L^1(\MM)}\les 2^{\frac{j}{2}}2^m\ep^2\gamma^\nu_j\gamma^{\nu'}_j.
\ee

Next, we evaluate the $L^1(\MM)$ norm of $h_{4,p,q,l,m}'$. Comparing the definition \eqref{biz5} of $h_{4,p,q,l,m}'$ and \eqref{biz4} of $h_{3,p,q,l,m}'$, we notice that these terms are similar. We proceed as for $h_{3,p,q,l,m}'$, and we obtain the analog of \eqref{biz21}:
\be\lab{biz22}
\norm{h_{4,p,q,l,m}'}_{L^1(\MM)}\les 2^{\frac{j}{2}}2^m\ep^2\gamma^\nu_j\gamma^{\nu'}_j.
\ee

Next, we evaluate the $L^1(\MM)$ norm of $h_{5,p,q,l,m}'$. Comparing the definition \eqref{biz6} of $h_{5,p,q,l,m}'$ and \eqref{buz4} of $h_{3,p,q,l,m}$, we notice that these terms are the same. Thus, in view of the estimates \eqref{buz20} and \eqref{buz22}, we have in the case $m=j/2$:
\be\lab{biz23}
\norm{h_{5,p,q,l,j/2}'}_{L^1(\MM)}\les 2^{\j+\frac{l}{2}}\ep^2\gamma^\nu_j\gamma^{\nu'}_j,
\ee
and in the case $m>j/2$:
\be\lab{biz24}
\norm{h_{5,p,q,l,m}'}_{L^1(\MM)}\les 2^{m+l+\min(l,m)-j}\norm{\mu_{j,\nu,l}}_{L^2(\R\times\S)}\norm{\mu_{j,\nu',m}}_{L^2(\R\times\S)},
\ee
where the sequence of functions $(\mu_{j,\nu,l})_{l> j/2}$ on $\R\times\S$ satisfies:
$$\sum_{\nu}\sum_{l> j/2}2^{2l}\norm{\mu_{j,\nu,l}}^2_{L^2(\R\times\S)}\les \ep^2 2^{2j}\norm{f}^2_{L^2(\R^3)}.$$

Next, we evaluate the $L^1(\MM)$ norm of $h_{6,p,q,l,m}'$. In view of the definition \eqref{biz7} of $h_{6,p,q,l,m}'$, we have:
\bea\lab{biz25}
\norm{h_{6,p,q,l,m}'}_{L^1(\MM)}&\les& \normm{\int_{\S}H_3\nabla(P_l\trc)\left(2^{\frac{j}{2}}(N-N_\nu)\right)^pF_{j,-1}(u)\eta_j^\nu(\o)d\o}_{L^2(\MM)}\\
\nn&&\times\normm{\int_{\S}\nabla(P_m\trc')\left(2^{\frac{j}{2}}(N'-N_{\nu'})\right)^qF_{j,-1}(u')\eta_j^{\nu'}(\o')d\o'}_{L^2(\MM)}\\
\nn&\les& \normm{\int_{\S}H_3\nabla(P_l\trc)\left(2^{\frac{j}{2}}(N-N_\nu)\right)^pF_{j,-1}(u)\eta_j^\nu(\o)d\o}_{L^2(\MM)}2^{\frac{j}{2}}\ep\gamma^{\nu'}_j,
\eea
where we used in the last inequality the analog of \eqref{vino56}. Now, the basic estimate in $L^2(\MM)$ \eqref{oscl2bis} yields:
\bea\lab{biz26}
&&\normm{\int_{\S}H_3 \nabla(P_l\trc) \left(2^{\frac{j}{2}}(N-N_\nu)\right)^pF_{j,-1}(u)\eta_j^\nu(\o)d\o}_{L^2(\MM)}\\
\nn&\les& \left(\sup_\o\normm{H_3 \nabla(P_l\trc) \left(2^{\frac{j}{2}}(N-N_\nu)\right)^p}_{\li{\infty}{2}}\right)2^{\frac{j}{2}}\gamma^\nu_j\\
\nn&\les& \left(\sup_\o\norm{H_3}_{\tx{\infty}{4}}\norm{\nabla(P_l\trc)}_{\tx{2}{4}}\normm{\left(2^{\frac{j}{2}}(N-N_\nu)\right)^p}_{L^\infty}\right)2^{\frac{j}{2}}\gamma^\nu_j\\
\nn&\les&  \left(\sup_\o\norm{H_3}_{\tx{\infty}{4}}\right)\ep 2^{\frac{l}{2}+\frac{j}{2}}\gamma^\nu_j,
\eea
where we used in the last inequality the estimate \eqref{ldc48} for $\nabla(P_l\trc)$, the estimate \eqref{estNomega} for $\po N$ and the size of the patch.  In view of the definition \eqref{biz12} of $H_3$, we have:
\bee
&&\norm{H_3}_{\tx{\infty}{4}}\\
&\les& \norm{k}_{\tx{\infty}{4}}+\norm{n^{-1}\nabla n}_{\tx{\infty}{4}}+\norm{\th}_{\tx{\infty}{4}}+\norm{b^{-1}\nabb(b)}_{\tx{\infty}{4}}+\norm{\chi}_{\tx{\infty}{4}}+\norm{\z}_{\tx{\infty}{4}}\\
&\les& \ep,
\eee
where we used in the last inequality the embedding \eqref{sobineq1}, the estimate \eqref{estk} for $k$, the estimate \eqref{estn} for $n$, the estimates \eqref{estk} \eqref{esttrc} \eqref{esthch} for $\th$, the estimate \eqref{estb} for $b$, the estimates \eqref{esttrc} \eqref{esthch} for $\chi$, and the estimate \eqref{estzeta} for $\z$. Together with \eqref{biz26}, this yields:
\be\lab{biz27}
\normm{\int_{\S}H_3 \nabla(P_l\trc) \left(2^{\frac{j}{2}}(N-N_\nu)\right)^pF_{j,-1}(u)\eta_j^\nu(\o)d\o}_{L^2(\MM)}\les\ep 2^{\frac{l}{2}+\frac{j}{2}}\gamma^\nu_j.
\ee
Finally, \eqref{biz25} and \eqref{biz27} imply:
\be\lab{biz28}
\norm{h_{6,p,q,l,m}'}_{L^1(\MM)}\les 2^{j+\frac{l}{2}}\ep^2\gamma^\nu_j\gamma^{\nu'}_j.
\ee

Next, we evaluate the $L^1(\MM)$ norm of $h_{7,p,q,l,m}'$. Comparing the definition \eqref{biz8} of $h_{7,p,q,l,m}'$ and \eqref{biz7} of $h_{6,p,q,l,m}'$, we notice that these terms are similar. We proceed as for $h_{6,p,q,l,m}'$, and we obtain the analog of \eqref{biz28}:
\be\lab{biz29}
\norm{h_{7,p,q,l,m}'}_{L^1(\MM)}\les 2^{j+\frac{m}{2}}\ep^2\gamma^\nu_j\gamma^{\nu'}_j.
\ee

Next, we evaluate the $L^1(\MM)$ norm of $h_{8,p,q,l,m}'$. In view of the definition \eqref{biz9} of $h_{8,p,q,l,m}'$, we have:
\bea\lab{biz30}
\norm{h_{8,p,q,l,m}'}_{L^1(\MM)}&\les& \normm{\int_{\S}\nabb(P_l\trc)\left(2^{\frac{j}{2}}(N-N_\nu)\right)^pF_{j,-1}(u)\eta_j^\nu(\o)d\o}_{L^2(\MM)}\\
\nn&&\times\normm{\int_{\S}\nabb'(N'(P_m\trc'))\left(2^{\frac{j}{2}}(N'-N_{\nu'})\right)^qF_{j,-1}(u')\eta_j^{\nu'}(\o')d\o'}_{L^2(\MM)}\\
\nn&\les& 2^{j+m}\ep^2\gamma^\nu_j\gamma^{\nu'}_j,
\eea
where we used in the last inequality the analog of the estimate \eqref{vino56} for the first term, and the analog of the estimate \eqref{duc17} for the second term.

Finally, we estimate $A^2_{j,\nu,\nu',l,m}$. In view of the decomposition \eqref{biz1} of $A^2_{j,\nu,\nu',l,m}$, the estimate \eqref{nice26}, and the estimates \eqref{biz16} \eqref{biz17} \eqref{biz21} \eqref{biz22} \eqref{biz23} \eqref{biz24} \eqref{biz28} \eqref{biz29} \eqref{biz30} for $h_{1,p,q,l,m}'$, $h_{2,p,q,l,m}'$, $h_{3,p,q,l,m}'$, $h_{4,p,q,l,m}'$, $h_{5,p,q,l,m}'$, $h_{6,p,q,l,m}'$, $h_{7,p,q,l,m}'$, $h_{8,p,q,l,m}'$, we obtain:
\bee
&&\left|\sum_{(l,m)/2^{\max(l,m)}\leq 2^j|\nu-\nu'|}A^2_{j,\nu,\nu',l,m}\right|\\
\nn&\les& \sum_{(l,m)/2^{\max(l,m)}\leq 2^j|\nu-\nu'|}\sum_{p, q\geq 0}c_{pq}\frac{2^{-\frac{5j}{2}}2^{l+m+\min(l,m)}}{(2^{\frac{j}{2}}|\nu-\nu'|)^{p+q+3}}\norm{\mu_{j,\nu,l}}_{L^2(\R\times\S)}\norm{\mu_{j,\nu',m}}_{L^2(\R\times\S)}\\
\nn&&+\sum_{p, q\geq 0}c_{pq}\frac{1}{(2^{\frac{j}{2}}|\nu-\nu'|)^{p+q+2}}\Bigg[\frac{1}{2^{\frac{j}{2}}|\nu-\nu'|}+\frac{1}{(2^{\frac{j}{2}}|\nu-\nu'|)^{\frac{1}{2}}}+2^{-\frac{j}{2}}(2^{\frac{j}{2}}|\nu-\nu'|)\Bigg]\ep^2\gamma^\nu_j\gamma^{\nu'}_j\\
\nn&\les&   \sum_{(l,m)/2^{\max(l,m)}\leq 2^j|\nu-\nu'|}\frac{2^{-\frac{5j}{2}}2^{l+m+\min(l,m)}}{(2^{\frac{j}{2}}|\nu-\nu'|)^3}\norm{\mu_{j,\nu,l}}_{L^2(\R\times\S)}\norm{\mu_{j,\nu',m}}_{L^2(\R\times\S)}\\
\nn&&+\Bigg[\frac{1}{(2^{\frac{j}{2}}|\nu-\nu'|)^3}+\frac{1}{(2^{\frac{j}{2}}|\nu-\nu'|)^{\frac{5}{2}}}+\frac{1}{2^{\frac{j}{2}}(2^{\frac{j}{2}}|\nu-\nu'|)}\Bigg]\ep^2\gamma^\nu_j\gamma^{\nu'}_j,
\eee
where the sequence of functions $(\mu_{j,\nu,l})_{l>j/2}$ on $\R\times\S$ satisfies:
$$\sum_{\nu}\sum_{l>j/2}2^{2l}\norm{\mu_{j,\nu,l}}^2_{L^2(\R\times\S)}\les \ep^2 2^{2j}\norm{f}^2_{L^2(\R^3)}.$$
This concludes the proof of Proposition \ref{prop:biz}.

%%%%%%%%%%%%%%%%%%%%%%%%%%%%%%

\appendix

\section{Proof of Lemma \ref{lemma:app1}}

Recall from \eqref{nice42} that $B^{1,1,1,1}_{j,\nu,\nu'}$ is given by:
\bee
B^{1,1,1,1}_{j,\nu,\nu'}&=& -i2^{-j-1}\int_{\MM}\int_{\S\times\S} \frac{b^{-1}}{\gg(L,L')}L(\trc)\trc'\\
\nn&&\times F_j(u)F_{j,-1}(u')\eta_j^\nu(\o)\eta_j^{\nu'}(\o')d\o d\o' d\MM,
\eee
We integrate by parts in $B^{1,1,1,1}_{j,\nu,\nu',l,m}$ using \eqref{fete1} with 
\be\lab{boca0}
h= \frac{L(\trc)b'\trc'}{\gg(L,L')}.
\ee
We obtain:
\bea\lab{boca}
&& B^{1,1,1,1}_{j,\nu,\nu'}\\
\nn&=&   -2^{-2j-1}\int_{\MM}\int_{\S\times\S}\frac{{b'}^{-1}}{1-\gn^2}\Bigg((N-\gn N')(h)+\bigg(\trt-\gn \trt'\\
\nn&&-\th(N'-\gn N, N'-\gn N)-\gn b^{-1}(N'-\gn N)(b)\\
\nn&&+\frac{2\gn}{1-\gn^2}\Big(\th'(N-\gn N',N-\gn N')\\
\nn&&-\gn \th(N'-\gn N, N'-\gn N)\Big)\bigg)h\Bigg)\\
\nn&& \times F_{j,-1}(u)F_{j,-1}(u')\eta_j^\nu(\o)\eta_j^{\nu'}(\o')d\o d\o'd\MM.
\eea

Next, we compute the term $(N-\gn N')(h)$. We have:
\bea\lab{boca1}
&&(N-\gn N')(h)\\
\nn&=&\frac{\nab_{N-\gn N'}(L(\trc))b'\trc'}{\gg(L,L')}+\frac{L(\trc)(N-\gn N')(b'\trc')}{\gg(L,L')}\\
\nn&& -\frac{(N-\gn N')(\gl)L(\trc)b'\trc'}{\gg(L,L')^2}.
\eea
Decomposing $N-\gn N'$ on $N$ and $N'-\gn N$, we have:
\be\lab{boca3}
N-\gn N'=(1-\gn^2)N-\gn(N'-\gn N)
\ee
which yields schematically for the first term in the right-hand side of \eqref{boca1}:
\be\lab{boca4}
\frac{\nab_{N-\gn N'}(L(\trc))b'\trc'}{\gg(L,L')}=\frac{N(L(\trc))b'\trc'(N-N')^2}{\gg(L,L')}+\frac{\nabb(L(\trc))b'\trc'(N'-N)}{\gg(L,L')},
\ee
where we used the fact that:
$$1-\gn^2=(1+\gn)(1-\gn)=(1+\gn)\frac{\gg(N-N',N-N')}{2}\sim (N-N')^2.$$
Finally, in order to estimate the third term in the right-hand side of \eqref{boca1}, we need to compute 
$(N-\gn N')(\gl)$. Since $\gl=-1+\gn$, we need to compute $(N-\gn N')(\gn)$. Using the structure equation 
\eqref{frame} for $N$ and the decomposition \eqref{boca3}, we obtain:
\bea\lab{boca5}
\nabla_{N-\gn N'}(\gn)&=&g(\nabla_{N-\gn N'}N,N')+g(N,\nabla_{N-\gn N'}N')\\
\nn &=&-(1-\gn^2)g(b^{-1}\nabla b,N'-\gn N)\\
\nn && -\gn\th(N'-\gn N,N'-\gn N)\\
\nn &&+\th'(N-\gn N',N-\gn N').
\eea
Now, we have:
\begin{equation}\label{boca6}
(N-\gn N')+(N'-\gn N)=(N+N')(1-\gn)\sim (N-N')^2.
\end{equation}
Also, since $\th=\chi+k$ by definition, and since $k$ does not depend on $\o$, we have:
\be\label{boca7}
\th-\th'=\chi-\chi'.
\ee
In view of \eqref{boca5}, \eqref{boca6} and \eqref{boca7}, we obtain schematically for the third term in the right-hand side of \eqref{boca1}:
\bea\lab{boca8}
&&-\frac{(N-\gn N')(\gl)L(\trc)b'\trc'}{\gg(L,L')^2}\\
\nn&=& \frac{L(\trc)(\th+b^{-1}\nabb(b))b'\trc'(N-N')^3}{\gl^2}+\frac{L(\trc)\th'b'\trc'(N-N')^3}{\gl^2}\\
\nn&&+\frac{(\chi-\chi')L(\trc)b'\trc'(N-N')^2}{\gl^2}.
\eea
Finally, \eqref{boca1}, \eqref{boca4} and \eqref{boca8} imply, schematically:
\bea\lab{boca9}
&&(N-\gn N')(h)\\
\nn&=& \frac{N(L(\trc))b'\trc'(N-N')^2}{\gg(L,L')}+\frac{\nabb(L(\trc))b'\trc'(N-N')}{\gg(L,L')}\\
\nn&& +\frac{L(\trc)\nabb'(b'\trc')(N-N')}{\gg(L,L')}+\frac{L(\trc)(\th+b^{-1}\nabb(b))b'\trc'(N-N')^3}{\gl^2}\\
\nn&&+\frac{L(\trc)\th'b'\trc'(N-N')^3}{\gl^2}+\frac{(\chi-\chi')L(\trc)b'\trc'(N-N')^2}{\gl^2}.
\eea

We consider the term multiplied by $h$ in the right-hand side of \eqref{boca}. Using \eqref{boca7}, we have schematically:
\bea\lab{boca10}
&&\trt-\gn \trt'-\th(N'-\gn N, N'-\gn N)\\
\nn&&-\gn b^{-1}(N'-\gn N)(b)\\
\nn&&+\frac{2\gn}{1-\gn^2}\Big(\th'(N-\gn N',N-\gn N')\\
\nn&&-\gn \th(N'-\gn N, N'-\gn N)\Big)\\
\nn&=& \chi-\chi'+\th(N-N')+\th'(N-N')+b^{-1}\nabb(b)(N-N').
\eea
Thus, in view of \eqref{boca0}, \eqref{boca}, \eqref{boca9} and \eqref{boca10} we obtain:
\be\lab{boca11}
B^{1,1,1,1}_{j,\nu,\nu'}= 2^{-2j}\int_{\MM}\int_{\S\times\S}H F_{j,-1}(u)F_{j,-1}(u')\eta_j^\nu(\o)\eta_j^{\nu'}(\o')d\o d\o'd\MM,
\ee
with the tensor $H$ on $\MM$ given, schematically, by:
\bea\lab{boca12}
\nn H&=&\frac{1}{1-\gn^2}\Bigg(\frac{N(L(\trc))\trc'(N-N')^2}{\gg(L,L')}+\frac{\nabb(L(\trc))\trc'(N-N')}{\gg(L,L')}\\
\nn&& +\frac{L(\trc){b'}^{-1}\nabb'(b'\trc')(N-N')}{\gg(L,L')}+\frac{L(\trc)(\th+b^{-1}\nabb(b))\trc'(N-N')^3}{\gl^2}\\
&&+\frac{L(\trc)\th'\trc'(N-N')^3}{\gl^2}+\frac{(\chi-\chi')L(\trc)\trc'(N-N')^2}{\gl^2}\\
\nn&&+\bigg(\chi-\chi'+\th(N-N')+\th'(N-N')+b^{-1}\nabb(b)(N-N')\bigg)\frac{L(\trc)\trc'}{\gg(L,L')}\Bigg).
\eea

Recall the identities \eqref{nice24} and \eqref{nice25}:
$$\gg(L,L')=-1+\gn\textrm{ and }1-\gn=\frac{\gg(N-N',N-N')}{2}.$$
We may thus expand:
$$\frac{1}{(1-\gn^2)\gl},\, \frac{1}{\gl}\textrm{ and }\frac{1}{(1-\gn^2)\gl^2}$$
in the same fashion than \eqref{nice27}, and we obtain, schematically:
\bea\lab{boca13}
H&=&\frac{1}{|N_\nu-N_{\nu'}|^2}\left(\sum_{p, q\geq 0}c_{pq}\left(\frac{N-N_\nu}{|N_\nu-N_{\nu'}|}\right)^p\left(\frac{N'-N_{\nu'}}{|N_\nu-N_{\nu'}|}\right)^q\right)\\
\nn&&\times\bigg(H_1+\frac{1}{|N_\nu-N_{\nu'}|}H_2+\frac{1}{|N_{\nu}-N_{\nu'}|^2}H_3\bigg),
\eea
where the tensors $H_1, H_2$ and $H_3$ on $\MM$ are given by:
\be\lab{boca14}
H_1=N(L(\trc))\trc',
\ee
\bea\lab{boca15}
H_2&=& \nabb(L(\trc))\trc'+L(\trc){b'}^{-1}\nabb'(b'\trc')+L(\trc)(\th)\trc'\\
\nn&&+\Big(\th+\th'+b^{-1}\nabb(b)\Big)L(\trc)\trc',
\eea
and:
\be\lab{boca16}
H_3=(\chi-\chi')L(\trc)\trc',
\ee
and where $c_{pq}$ are explicit real coefficients such that the series 
$$\sum_{p, q\geq 0}c_{pq}x^py^q$$
has radius of convergence 1. In view of \eqref{boca11}, \eqref{boca13}, \eqref{boca14}, \eqref{boca15} and \eqref{boca16}, we obtain the decomposition \eqref{app1} \eqref{app2} \eqref{app3} \eqref{app4} \eqref{app5} \eqref{app7} \eqref{app8} of $B^{1,1,1,1}_{j,\nu,\nu'}$. This concludes the proof of Lemma \ref{lemma:app1}.

\section{Proof of Lemma \ref{lemma:app1:2}}

Recall from \eqref{nice43} that $B^{1,1,1,2}_{j,\nu,\nu'}$ is given by:
\bee
B^{1,1,1,2}_{j,\nu,\nu'}&=& -i2^{-j-1}\int_{\MM}\int_{\S\times\S} \frac{b^{-1}}{\gg(L,L')}\trc L'(\trc')\\
\nn&&\times F_j(u)F_{j,-1}(u')\eta_j^\nu(\o)\eta_j^{\nu'}(\o')d\o d\o' d\MM,
\eee
We integrate by parts in $B^{1,1,1,2}_{j,\nu,\nu',l,m}$ using \eqref{fete1} with 
\be\lab{boca0:2}
h= \frac{\trc b'L'(\trc')}{\gg(L,L')}.
\ee
We obtain:
\bea\lab{boca:2}
&& B^{1,1,1,2}_{j,\nu,\nu'}\\
\nn&=&   -2^{-2j-1}\int_{\MM}\int_{\S\times\S}\frac{{b'}^{-1}}{1-\gn^2}\Bigg((N-\gn N')(h)+\bigg(\trt-\gn \trt'\\
\nn&&-\th(N'-\gn N, N'-\gn N)-\gn b^{-1}(N'-\gn N)(b)\\
\nn&&+\frac{2\gn}{1-\gn^2}\Big(\th'(N-\gn N',N-\gn N')\\
\nn&&-\gn \th(N'-\gn N, N'-\gn N)\Big)\bigg)h\Bigg)\\
\nn&& \times F_{j,-1}(u)F_{j,-1}(u')\eta_j^\nu(\o)\eta_j^{\nu'}(\o')d\o d\o'd\MM.
\eea

Next, we compute the term $(N-\gn N')(h)$. Proceeding as in \eqref{boca1}, \eqref{boca4}, \eqref{boca8} and \eqref{boca9}, we obtain schematically:
\bea\lab{boca9:2}
&&(N-\gn N')(h)\\
\nn&=& \frac{N(\trc)b'L'(\trc')(N-N')^2}{\gg(L,L')}+\frac{\nabb(\trc)b'L'(\trc')(N-N')}{\gg(L,L')}\\
\nn&& +\frac{\trc\nabb'(b'L'(\trc'))(N-N')}{\gg(L,L')}+\frac{\trc(\th+b^{-1}\nabb(b))b'L'(\trc')(N-N')^3}{\gl^2}\\
\nn&&+\frac{\trc\th'b'L'(\trc')(N-N')^3}{\gl^2}+\frac{(\chi-\chi')\trc b'L'(\trc')(N-N')^2}{\gl^2}.
\eea
Thus, in view of \eqref{boca0:2}, \eqref{boca:2}, \eqref{boca9:2} and \eqref{boca10} we obtain:
\be\lab{boca11:2}
B^{1,1,1,1}_{j,\nu,\nu'}= 2^{-2j}\int_{\MM}\int_{\S\times\S}H F_{j,-1}(u)F_{j,-1}(u')\eta_j^\nu(\o)\eta_j^{\nu'}(\o')d\o d\o'd\MM,
\ee
with the tensor $H$ on $\MM$ given, schematically, by:
\bee
\nn H&=&\frac{1}{1-\gn^2}\Bigg(\frac{N(\trc)L'(\trc')(N-N')^2}{\gg(L,L')}+\frac{\nabb(\trc)L'(\trc')(N-N')}{\gg(L,L')}\\
\nn&& +\frac{\trc {b'}^{-1}\nabb'(b'L'(\trc'))(N-N')}{\gg(L,L')}+\frac{\trc (\th+b^{-1}\nabb(b))L'(\trc')(N-N')^3}{\gl^2}\\
\nn&&+\frac{\trc\th'L'(\trc')(N-N')^3}{\gl^2}+\frac{(\chi-\chi')\trc L'(\trc')(N-N')^2}{\gl^2}\\
\nn&&+\bigg(\chi-\chi'+\th(N-N')+\th'(N-N')+b^{-1}\nabb(b)(N-N')\bigg)\frac{\trc L'(\trc')}{\gg(L,L')}\Bigg).
\eee
Proceeding in the same fashion than \eqref{boca13}, we obtain, schematically:
\bea\lab{boca13:2}
H&=&\frac{1}{|N_\nu-N_{\nu'}|^2}\left(\sum_{p, q\geq 0}c_{pq}\left(\frac{N-N_\nu}{|N_\nu-N_{\nu'}|}\right)^p\left(\frac{N'-N_{\nu'}}{|N_\nu-N_{\nu'}|}\right)^q\right)\\
\nn&&\times\bigg(H_1+\frac{1}{|N_\nu-N_{\nu'}|}H_2+\frac{1}{|N_{\nu}-N_{\nu'}|^2}H_3\bigg),
\eea
where the tensors $H_1, H_2$ and $H_3$ on $\MM$ are given by:
\be\lab{boca14:2}
H_1=N(\trc)L'(\trc'),
\ee
\bea\lab{boca15:2}
H_2&=& \nabb(\trc)L'(\trc')+\trc{b'}^{-1}\nabb'(b'L'(\trc'))+\trc(\th)L'(\trc')\\
\nn&&+\Big(\th+\th'+b^{-1}\nabb(b)\Big)\trc L'(\trc'),
\eea
and:
\be\lab{boca16:2}
H_3=(\chi-\chi')\trc L'(\trc'),
\ee
and where $c_{pq}$ are explicit real coefficients such that the series 
$$\sum_{p, q\geq 0}c_{pq}x^py^q$$
has radius of convergence 1. In view of \eqref{boca11:2}, \eqref{boca13:2}, \eqref{boca14:2}, \eqref{boca15:2} and \eqref{boca16:2}, we obtain the decomposition \eqref{app1:2} \eqref{app2:2} \eqref{app3:2} \eqref{app4:2} \eqref{app5:2} \eqref{app7:2} \eqref{app8:2} of $B^{1,1,1,2}_{j,\nu,\nu'}$. This concludes the proof of Lemma \ref{lemma:app1:2}.

\section{Proof of Lemma \ref{lemma:vino}}

Recall from \eqref{nyc3} that $B^{1,2,,2}_{j,\nu,\nu',l,m}$ is given by:
\bee
B^{1,2,2}_{j,\nu,\nu',l,m} &=&  -i2^{-j-1}\int_{\MM}\int_{\S\times\S} \frac{1}{\gg(L,L')}\bigg(L(P_l\trc)P_m\trc'+P_l\trc L'(P_m\trc')\bigg)\\
\nn&&\times (b^{-1}-{b'}^{-1}\gn) F_{j,-1}(u)F_j(u')\eta_j^\nu(\o)\eta_j^{\nu'}(\o')d\o d\o' d\MM.
\eee
We integrate by parts in $B^{1,2,2}_{j,\nu,\nu',l,m}$ using \eqref{fete} with 
\be\lab{boca0:3}
h=\frac{\bigg(L(P_l\trc)P_m\trc'+P_l\trc L'(P_m\trc')\bigg)(b'-b)}{\gg(L,L')}.
\ee
We obtain:
\bea\lab{boca:3}
&& B^{1,2,2}_{j,\nu,\nu',l,m}\\
\nn&=&   -2^{-2j-1}\int_{\MM}\int_{\S\times\S}\frac{b^{-1}}{1-\gn^2}\Bigg((N'-\gn N)(h)+\bigg(\trt'-\gn \trt\\
\nn&&-\th'(N-\gn N', N-\gn N')-\gn {b'}^{-1}(N-\gn N')(b')\Big)\\
\nn&&+\frac{2\gn}{1-\gn^2}\Big(\th(N'-\gn N,N'-\gn N)\\
\nn&&-\gn \th'(N-\gn N', N-\gn N')\Big)\bigg)h\Bigg)\\
\nn&& \times F_{j,-1}(u)F_{j,-1}(u')\eta_j^\nu(\o)\eta_j^{\nu'}(\o')d\o d\o'd\MM.
\eea

Next, we compute the term $(N'-\gn N)(h)$. We have:
\bea\lab{boca1:3}
&&(N'-\gn N)(h)\\
\nn&=&\frac{\nab_{N'-\gn N}(L(P_l\trc))P_m\trc'(b'-b)}{\gg(L,L')}+\frac{L(P_l\trc)(N'-\gn N)(P_m\trc')(b'-b)}{\gg(L,L')}\\
\nn&&+\frac{(N'-\gn N)(P_l\trc)L'(P_m\trc')(b'-b)}{\gg(L,L')}+\frac{P_l\trc\nab_{N'-\gn N}(L'(P_m\trc'))(b'-b)}{\gg(L,L')}\\
\nn&& -\frac{(N'-\gn N)(\gl)\Big(L(P_l\trc)P_m\trc'+P_l\trc L'(P_m\trc')\Big)(b'-b)}{\gg(L,L')^2}.
\eea
Decomposing $N'-\gn N$ on $N'$ and $N-\gn N'$, we have:
\be\lab{boca3:3}
N'-\gn N=(1-\gn^2)N'-\gn(N-\gn N')
\ee
which yields schematically for the second and the fourth term in the right-hand side of \eqref{boca1:3}:
\bea\lab{boca4:3}
&&\frac{L(P_l\trc)(N'-\gn N)(P_m\trc')(b'-b)}{\gg(L,L')}+\frac{P_l\trc\nab_{N'-\gn N}(L'(P_m\trc'))(b'-b)}{\gg(L,L')}\\
\nn&=& -\frac{L(P_l\trc)(N-\gn N')(P_m\trc')(b'-b)}{\gg(L,L')}+\frac{L(P_l\trc) N'(P_m\trc')(b'-b)(N-N')^2}{\gg(L,L')}\\
\nn&&-\frac{P_l\trc\nab_{N-\gn N'}(L'(P_m\trc'))(b'-b)}{\gg(L,L')}+\frac{P_l\trc\nab_{N'}(L'(P_m\trc'))(b'-b)(N-N')^2}{\gg(L,L')},
\eea
where we used the fact that:
$$1-\gn^2=(1+\gn)(1-\gn)=(1+\gn)\frac{\gg(N-N',N-N')}{2}\sim (N-N')^2.$$
Finally, in order to estimate the last term in the right-hand side of \eqref{boca1}, we need to compute 
$(N'-\gn N)(\gl)$. We have the analog of \eqref{boca5}:
\bea\lab{boca5:3}
\nabla_{N'-\gn N}(\gn) &=&-(1-\gn^2)g({b'}^{-1}\nabla b',N-\gn N')\\
\nn && -\gn\th'(N-\gn N',N-\gn N')\\
\nn &&+\th(N'-\gn N,N'-\gn N).
\eea
In view of \eqref{boca5:3}, \eqref{boca6} and \eqref{boca7}, we obtain schematically for the last term in the right-hand side of \eqref{boca1:3}:
\bea\lab{boca8:3}
\nn&&-\frac{(N'-\gn N)(\gl)\Big(L(P_l\trc)P_m\trc'+P_l\trc L'(P_m\trc')\Big)(b'-b)}{\gg(L,L')^2}\\
\nn&=& \frac{\Big(L(P_l\trc)P_m\trc'+P_l\trc L'(P_m\trc')\Big)(b'-b)(\th'+{b'}^{-1}\nabb'(b')+\th)(N-N')^3}{\gl^2}\\
&&+\frac{(\chi'-\chi)\Big(L(P_l\trc)P_m\trc'+P_l\trc L'(P_m\trc')\Big)(b'-b)(N-N')^2}{\gl^2}.
\eea
Finally, \eqref{boca1:3}, \eqref{boca4:3} and \eqref{boca8:3} imply, schematically:
\bea\lab{boca9:3}
&&(N'-\gn N)(h)\\
\nn&=& \frac{\nab_{N'-\gn N}(L(P_l\trc))P_m\trc'(b'-b)}{\gg(L,L')}+\frac{(N'-\gn N)(P_l\trc)L'(P_m\trc')(b'-b)}{\gg(L,L')}\\
\nn&&-\frac{L(P_l\trc)(N-\gn N')(P_m\trc')(b'-b)}{\gg(L,L')}+\frac{L(P_l\trc) N'(P_m\trc')(b'-b)(N-N')^2}{\gg(L,L')}\\
\nn&&-\frac{P_l\trc\nab_{N-\gn N'}(L'(P_m\trc'))(b'-b)}{\gg(L,L')}+\frac{P_l\trc\nab_{N'}(L'(P_m\trc'))(b'-b)(N-N')^2}{\gg(L,L')}\\
\nn&&+\frac{\Big(L(P_l\trc)P_m\trc'+P_l\trc L'(P_m\trc')\Big)(b'-b)(\th'+{b'}^{-1}\nabb'(b')+\th)(N-N')^3}{\gl^2}\\
\nn&&+\frac{(\chi'-\chi)\Big(L(P_l\trc)P_m\trc'+P_l\trc L'(P_m\trc')\Big)(b'-b)(N-N')^2}{\gl^2}.
\eea

We consider the term multiplied by $h$ in the right-hand side of \eqref{boca:3}. Using \eqref{boca7}, we have schematically:
\bea\lab{boca10:3}
&&\trt-\gn \trt'-\th(N'-\gn N, N'-\gn N)\\
\nn&&-\gn b^{-1}(N'-\gn N)(b)\\
\nn&&+\frac{2\gn}{1-\gn^2}\Big(\th'(N-\gn N',N-\gn N')\\
\nn&&-\gn \th(N'-\gn N, N'-\gn N)\Big)\\
\nn&=& \chi-\chi'+\th(N-N')+\th'(N-N')+b^{-1}\nabb(b)(N-N').
\eea

Thus, in view of \eqref{boca0:3}, \eqref{boca:3}, \eqref{boca9:3} and \eqref{boca10:3} we obtain:
\be\lab{boca11:3}
B^{1,2,2}_{j,\nu,\nu',l,m}= 2^{-2j}\int_{\MM}\int_{\S\times\S}H F_{j,-1}(u)F_{j,-1}(u')\eta_j^\nu(\o)\eta_j^{\nu'}(\o')d\o d\o'd\MM,
\ee
with the tensor $H$ on $\MM$ given, schematically, by:
\bea
\nn H&=&\frac{1}{1-\gn^2}\\
\nn&&\times\Bigg[\frac{\nab_{N'-\gn N}(L(P_l\trc))P_m\trc'(b'-b)}{\gg(L,L')}+\frac{(N'-\gn N)(P_l\trc)L'(P_m\trc')(b'-b)}{\gg(L,L')}\\
\nn&&-\frac{L(P_l\trc)(N-\gn N')(P_m\trc')(b'-b)}{\gg(L,L')}+\frac{L(P_l\trc) N'(P_m\trc')(b'-b)(N-N')^2}{\gg(L,L')}\\
\nn&&-\frac{P_l\trc\nab_{N-\gn N'}(L'(P_m\trc'))(b'-b)}{\gg(L,L')}+\frac{P_l\trc\nab_{N'}(L'(P_m\trc'))(b'-b)(N-N')^2}{\gg(L,L')}\\
\nn&&+\frac{\Big(L(P_l\trc)P_m\trc'+P_l\trc L'(P_m\trc')\Big)(b'-b)(\th'+{b'}^{-1}\nabb'(b')+\th)(N-N')^3}{\gl^2}\\
\nn&&+\frac{(\chi'-\chi)\Big(L(P_l\trc)P_m\trc'+P_l\trc L'(P_m\trc')\Big)(b'-b)(N-N')^2}{\gl^2}\\
\nn&&+\bigg(\chi-\chi'+\th(N-N')+\th'(N-N')+b^{-1}\nabb(b)(N-N')\bigg)\\
\nn&&\times\frac{\bigg(L(P_l\trc)P_m\trc'+P_l\trc L'(P_m\trc')\bigg)(b'-b)}{\gg(L,L')}\Bigg].
\eea
Recall the identities \eqref{nice24} and \eqref{nice25}:
$$\gg(L,L')=-1+\gn\textrm{ and }1-\gn=\frac{\gg(N-N',N-N')}{2}.$$
We may thus expand:
$$\frac{1}{(1-\gn^2)\gl},\, \frac{1}{\gl}\textrm{ and }\frac{1}{(1-\gn^2)\gl^2}$$
in the same fashion than \eqref{nice27}, and we obtain, schematically:
\bea\lab{boca13:3}
H&=&\frac{1}{|N_\nu-N_{\nu'}|^2}\left(\sum_{p, q\geq 0}c_{pq}\left(\frac{N-N_\nu}{|N_\nu-N_{\nu'}|}\right)^p\left(\frac{N'-N_{\nu'}}{|N_\nu-N_{\nu'}|}\right)^q\right)\\
\nn&&\times\bigg(H_1+\frac{1}{|N_\nu-N_{\nu'}|}H_2\bigg)\\
\nn&&+\frac{(N'-N)(b'-b)}{\gl^2}\Big(\nabb(L(P_l\trc))P_m\trc'+\nabb(P_l\trc)L'(P_m\trc')\\
\nn&&+L(P_l\trc)\nabb'(P_m\trc')+P_l\trc\nabb'(L'(P_m\trc'))\Big)+\frac{P_l\trc\nab_{N'}(L'(P_m\trc'))(b'-b)}{\gg(L,L')}\\
\nn&& +\frac{(\chi'-\chi)(b'-b)}{\gl^2}\Big(L(P_l\trc)P_m\trc'+P_l\trc L'(P_m\trc')\Big),
\eea
where the tensors $H_1, H_2$ on $\MM$ are given by:
\be\lab{boca14:3}
H_1=L(P_l\trc) N'(P_m\trc')(b'-b),
\ee
and:
\be\lab{boca16:3}
H_2=\bigg(\th+\th'+b^{-1}\nabb(b)+{b'}^{-1}\nabb'(b')\bigg)\bigg(L(P_l\trc)P_m\trc'+P_l\trc L'(P_m\trc')\bigg)(b'-b),
\ee
and where $c_{pq}$ are explicit real coefficients such that the series 
$$\sum_{p, q\geq 0}c_{pq}x^py^q$$
has radius of convergence 1. In view of \eqref{boca11:3}, \eqref{boca13:3}, \eqref{boca14:3}, and \eqref{boca16:3}, we obtain the decomposition \eqref{vino1} \eqref{vino2} \eqref{vino4} \eqref{vino5bis} \eqref{vino5ter} \eqref{vino5quatre} of $B^{1,2,2}_{j,\nu,\nu',l,m}$. This concludes the proof of Lemma \ref{lemma:vino}.

\section{Proof of Lemma \ref{lemma:vinoroja}}

Recall from \eqref{vinoroja3} that $B^{1,2,2,2,3}_{j,\nu,\nu',l,m}$ is given by:
\bee
B^{1,2,2,2,3}_{j,\nu,\nu',l,m}= -i2^{-j}\int_{\MM}\int_{\S\times \S} b^{-1} P_l\trc N'(P_m\trc')(b'-b)F_j(u)\eta_j^\nu(\o)F_{j,-1}(u')\eta_j^{\nu'}(\o') d\o d\o'd\MM.
\eee
We integrate by parts in $B^{1,2,2,2,3}_{j,\nu,\nu',l,m}$ using \eqref{fete1} with 
\be\lab{boca0:4}
h= P_l\trc b'N'(P_m\trc')(b'-b).
\ee
We obtain:
\bea\lab{boca:4}
&& B^{1,2,2,2,3}_{j,\nu,\nu',l,m}\\
\nn&=&   -2^{-2j}\int_{\MM}\int_{\S\times\S}\frac{{b'}^{-1}}{1-\gn^2}\Bigg((N-\gn N')(h)+\bigg(\trt-\gn \trt'\\
\nn&&-\th(N'-\gn N, N'-\gn N)-\gn b^{-1}(N'-\gn N)(b)\\
\nn&&+\frac{2\gn}{1-\gn^2}\Big(\th'(N-\gn N',N-\gn N')\\
\nn&&-\gn \th(N'-\gn N, N'-\gn N)\Big)\bigg)h\Bigg)\\
\nn&& \times F_{j,-1}(u)F_{j,-1}(u')\eta_j^\nu(\o)\eta_j^{\nu'}(\o')d\o d\o'd\MM.
\eea

Next, we compute the term $(N-\gn N')(h)$. Proceeding as in \eqref{boca1}, \eqref{boca4} and \eqref{boca9}, we obtain schematically:
\bea\lab{boca9:4}
&&(N-\gn N')(h)\\
\nn&=& b'(1-\gn^2)N(P_l\trc) N'(P_m\trc')(b'-b)\\
\nn&&+b'(N'-\gn N)(P_l\trc) N'(P_m\trc')(b'-b)\\
\nn&& +b'P_l\trc \nabb'(N'(P_m\trc'))(b'-b)(N-N')+b'b^{-1}\nabb(b) P_l\trc' N'(P_m\trc')(N-N')\\
\nn&& +\nabb(b') P_l\trc' N'(P_m\trc')(N-N')+b'N(b) P_l\trc' N'(P_m\trc')(N-N')^2.
\eea
Thus, in view of \eqref{boca0:4}, \eqref{boca:4}, \eqref{boca9:4} and \eqref{boca10} we obtain:
\be\lab{boca11:4}
B^{1,2,2,2,3}_{j,\nu,\nu',l,m}= 2^{-2j}\int_{\MM}\int_{\S\times\S}H F_{j,-1}(u)F_{j,-1}(u')\eta_j^\nu(\o)\eta_j^{\nu'}(\o')d\o d\o'd\MM,
\ee
with the tensor $H$ on $\MM$ given, schematically, by:
\bee
\nn H&=&\frac{1}{1-\gn^2}\Bigg((1-\gn^2)N(P_l\trc) N'(P_m\trc')(b'-b)\\
\nn&& +(N'-\gn N)(P_l\trc) N'(P_m\trc')(b'-b)+P_l\trc \nabb'(N'(P_m\trc'))(b'-b)(N-N')\\
\nn&& +b^{-1}\nabb(b) P_l\trc' N'(P_m\trc')(N-N')+{b'}^{-1}\nabb(b') P_l\trc' N'(P_m\trc')(N-N')\\
\nn&& +N(b) P_l\trc' N'(P_m\trc')(N-N')^2+\bigg(\chi-\chi'+\th(N-N')+\th'(N-N')\\
\nn&& +b^{-1}\nabb(b)(N-N')\bigg)P_l\trc N'(P_m\trc')(b'-b)\Bigg).
\eee
Proceeding in the same fashion than \eqref{boca13}, we obtain, schematically:
\bea\lab{boca13:4}
H&=&\frac{1}{|N_\nu-N_{\nu'}|^2}\left(\sum_{p, q\geq 0}c_{pq}\left(\frac{N-N_\nu}{|N_\nu-N_{\nu'}|}\right)^p\left(\frac{N'-N_{\nu'}}{|N_\nu-N_{\nu'}|}\right)^q\right)\\
\nn&&\times\bigg(\frac{1}{|N_{\nu}-N_{\nu'}|^2}H_1+\frac{1}{|N_\nu-N_{\nu'}|}H_2+H_3\bigg)\\
\nn&&+N(P_l\trc) N'(P_m\trc')(b'-b)+\frac{(N'-\gn N)(P_l\trc) N'(P_m\trc')(b'-b)}{\gl},
\eea
where the tensors $H_1, H_2$ and $H_3$ on $\MM$ are given by:
\be\lab{boca14:4}
H_1=(\chi+\th+\chi'+L'(b')+\th')P_l\trc N'(P_m\trc')(b'-b),
\ee
\bea\lab{boca15:4}
H_2&=& P_l\trc\nabb' N'(P_m\trc')+(\nabb'(b')+\nabb(b))P_l\trc N'(P_m\trc'),
\eea
and:
\be\lab{boca16:4}
H_3=N(b)P_l\trc N'(P_m\trc'),
\ee
and where $c_{pq}$ are explicit real coefficients such that the series 
$$\sum_{p, q\geq 0}c_{pq}x^py^q$$
has radius of convergence 1. In view of \eqref{boca11:4}, \eqref{boca13:4}, \eqref{boca14:4}, \eqref{boca15:4} and \eqref{boca16:4}, we obtain the decomposition \eqref{vinoroja29} \eqref{vinoroja30} \eqref{vinoroja31} \eqref{vinoroja33bis} \eqref{vinoroja34} \eqref{vinoroja35} \eqref{vinoroja36} of $B^{1,2,2,2,3}_{j,\nu,\nu',l,m}$. This concludes the proof of Lemma \ref{lemma:vinoroja}.

\section{Proof of Lemma \ref{lemma:vinoverde}}

Recall from \eqref{vinoroja1} that $B^{1,2,2,2,1}_{j,\nu,\nu',l,m}$ is given by:
\bee
B^{1,2,2,2,1}_{j,\nu,\nu',l,m}&=& -2^{-2j}\int_{\MM}\int_{\S\times \S} \frac{(N'-\gn N)(P_l\trc) N'(P_m\trc')(b'-b)}{\gg(L,L')}\\
\nn&&\times F_{j,-1}(u)\eta_j^\nu(\o)F_{j,-1}(u')\eta_j^{\nu'}(\o') d\o d\o'd\MM,
\eee
We integrate by parts in $B^{1,1,1,2}_{j,\nu,\nu',l,m}$ using \eqref{fete1} with 
\be\lab{boca0:5}
h=\frac{b(N'-\gn N)(P_l\trc) b'N'(P_m\trc')(b'-b)}{\gg(L,L')}. 
\ee
We obtain:
\bea\lab{boca:5}
&& B^{1,1,1,2}_{j,\nu,\nu'}\\
\nn&=&   i2^{-3j}\int_{\MM}\int_{\S\times\S}\frac{{b'}^{-1}}{1-\gn^2}\Bigg((N-\gn N')(h)+\bigg(\trt-\gn \trt'\\
\nn&&-\th(N'-\gn N, N'-\gn N)-\gn b^{-1}(N'-\gn N)(b)\\
\nn&&+\frac{2\gn}{1-\gn^2}\Big(\th'(N-\gn N',N-\gn N')\\
\nn&&-\gn \th(N'-\gn N, N'-\gn N)\Big)\bigg)h\Bigg)\\
\nn&& \times F_{j,-2}(u)F_{j,-1}(u')\eta_j^\nu(\o)\eta_j^{\nu'}(\o')d\o d\o'd\MM.
\eea

Next, we compute the term $(N-\gn N')(h)$. Proceeding as in \eqref{boca1}, \eqref{boca4}, \eqref{boca8} and \eqref{boca9}, we obtain schematically:
\bea\lab{boca9:5}
&&(N-\gn N')(h)\\
\nn&=& \frac{N((N'-\gn N)(P_l\trc))N'(P_m\trc')(N-N')}{\gg(L,L')}\\
\nn&&+\frac{(N'-\gn N)((N'-\gn N)(P_l\trc))N'(P_m\trc')}{\gg(L,L')}\\
\nn&& +\frac{\nabb(P_l\trc))\nabb(N'(P_m\trc'))(N-N')^2}{\gg(L,L')}\\
\nn&&+\bigg(\chi+\chi'+\th+\th'+b^{-1}\nabb(b)+{b'}^{-1}\nabb(b')\bigg)\frac{\nabb(P_l\trc)N'(P_m\trc')(N-N')}{\gg(L,L')}.
\eea
Next, we evaluate the first two terms of \eqref{boca9:5} starting with the first one. We have, schematically:
\be\lab{plaque}
N((N'-\gn N)(P_l\trc))=\nabb(N'(P_l\trc))(N-N')+\nab_{\nab_N(N'-\gn N)}(P_l\trc).
\ee
Using the structure equations for $N$ \eqref{frame} together with the decomposition of $N$ given by \eqref{fete4}, we obtain, schematically:
\be\lab{plaque1}
\nab_N(N'-\gn N)=\th'(N-N')+b^{-1}\nabb(b)+{b'}^{-1}\nabb(b').
\ee
Together with \eqref{plaque}, this yields, schematically:
\bea\lab{plaque2}
N((N'-\gn N)(P_l\trc))&=&\nabb(N'(P_l\trc))(N-N')\\
\nn&&+(\th'(N-N')+b^{-1}\nabb(b)+{b'}^{-1}\nabb(b'))\dd(P_l\trc).
\eea
Next, we evaluate the second term in the right-hand side of \eqref{boca9:5}. We have, schematically:
\bea\lab{plaque3}
(N'-\gn N)((N'-\gn N)(P_l\trc))&=&\nabb^2P_l\trc(N-N')^2\\
\nn&& +\nab_{\nab_{N'-\gn N}(N'-\gn N)}(P_l\trc).
\eea
Using the structure equations for $N$ \eqref{frame} together with \eqref{boca5} and \eqref{boca6}, we obtain, schematically:
\bea\label{plaque4}
&&\nabla_{N'-\gn N}(N'-\gn N)\\
\nn&=&\nabla_{N'-\gn N}N'-\gn \nabla_{N'-\gn N}N-\nabla_{N'-\gn N}(\gn)N\\
\nn & =& (1-\gn^2)\nab_{N'}N'-\gn\nabla_{N-\gn N'}N'-\gn\th(N'-\gn N,e_A)e_A\\
\nn&&-((1-\gn^2)g(-\nabb'\log(a'),N) -\gn\th'(N-\gn N',N-\gn N')\\
\nn&&+\th(N'-\gn N,N'-\gn N))N\\
\nn&=& {b'}^{-1}\nabb(b')(N-N')^2+(\th+\th')(N-N').
\eea
Together with \eqref{plaque3}, this yields, schematically:
\bea\lab{plaque5}
&&(N'-\gn N)((N'-\gn N)(P_l\trc))\\
\nn&=&\nabb^2P_l\trc(N-N')^2+({b'}^{-1}\nabb(b')(N-N')^2+(\th+\th')(N-N'))\nabb P_l\trc.
\eea

In view of \eqref{boca0:5}, \eqref{boca:5}, \eqref{boca9:5}, \eqref{plaque2}, \eqref{plaque3}, and \eqref{boca10} we obtain:
\be\lab{boca11:5}
B^{1,2,2,2,1}_{j,\nu,\nu',l,m}= 2^{-3j}\int_{\MM}\int_{\S\times\S}H F_{j,-2}(u)F_{j,-1}(u')\eta_j^\nu(\o)\eta_j^{\nu'}(\o')d\o d\o'd\MM,
\ee
with the tensor $H$ on $\MM$ given, schematically, by:
\bee
\nn H&=&\frac{1}{1-\gn^2}\Bigg(\frac{\nabb^2(P_l\trc)N'(P_m\trc')(N-N')^2}{\gg(L,L')}+\frac{\nabb(N(P_l\trc))N'(P_m\trc')(N-N')^3}{\gg(L,L')}\\
\nn&& +\frac{\nabb'(P_l\trc')\nabb'(N'(P_m\trc'))(N-N')^2}{\gg(L,L')}+\bigg(\chi+\chi'+\th+\th'+b^{-1}\nabb(b)+{b'}^{-1}\nabb(b')\bigg)\\
\nn&&\times\frac{\nabb(P_l\trc)N'(P_m\trc')}{\gg(L,L')}\left((N-N')+\frac{(N-N')^3}{\gg(L,L')}\right)\Bigg).
\eee
Proceeding in the same fashion than \eqref{boca13}, we obtain, schematically:
\bea\lab{boca13:5}
H&=&\frac{1}{|N_\nu-N_{\nu'}|}\left(\sum_{p, q\geq 0}c_{pq}\left(\frac{N-N_\nu}{|N_\nu-N_{\nu'}|}\right)^p\left(\frac{N'-N_{\nu'}}{|N_\nu-N_{\nu'}|}\right)^q\right)\\
\nn&&\times\bigg(H_1+\frac{1}{|N_\nu-N_{\nu'}|}H_2+\frac{1}{|N_\nu-N_{\nu'}|^2}H_3\bigg),
\eea
where the tensors $H_1, H_2$ and $H_3$ on $\MM$ are given by:
\be\lab{boca14:5}
H_1=\nabb(N(P_l\trc))N'(P_m\trc'),
\ee
\be\lab{boca15:5}
H_2=\nabb^2(P_l\trc)N'(P_m\trc')+\nabb'(P_l\trc')\nabb'(N'(P_m\trc'),
\ee
and:
\be\lab{boca16:5}
H_3= \bigg(\chi+\chi'+\th+\th'+b^{-1}\nabb(b)+{b'}^{-1}\nabb(b')\bigg)\nabb(P_l\trc)N'(P_m\trc'),
\ee
and where $c_{pq}$ are explicit real coefficients such that the series 
$$\sum_{p, q\geq 0}c_{pq}x^py^q$$
has radius of convergence 1. In view of \eqref{boca11:5}, \eqref{boca13:5}, \eqref{boca14:5}, \eqref{boca15:5} and \eqref{boca16:5}, we obtain the decomposition \eqref{vinoverde29}, \eqref{vinoverde30}, \eqref{vinoverde31}, \eqref{vinoverde32}, \eqref{vinoverde33}, \eqref{vinoverde33bis}, \eqref{vinoverde34}, \eqref{vinoverde35} of $B^{1,2,2,2,1}_{j,\nu,\nu',l,m}$. This concludes the proof of Lemma \ref{lemma:vinoverde}.

\section{Proof of Lemma \ref{lemma:bizu}}

Recall from \eqref{bizu8} that $B^{1,2,3,1}_{j,\nu,\nu',l,m}$ is given by:
$$B^{1,2,3,1}_{j,\nu,\nu',l,m} =  i2^{-j-1}\int_{\MM}\int_{\S\times\S}  {b'}^{-1}L(P_l\trc)P_m\trc' F_{j,-1}(u)F_j(u')\eta_j^\nu(\o)\eta_j^{\nu'}(\o')d\o d\o' d\MM.$$
We integrate by parts in $B^{1,2,3,1}_{j,\nu,\nu',l,m}$ using \eqref{fete} with 
\be\lab{boca0:6}
h=b L(P_l\trc)P_m\trc'.
\ee
We obtain:
\bea\lab{boca:6}
&& B^{1,2,3,1}_{j,\nu,\nu',l,m}\\
\nn&=&   -2^{-2j-1}\int_{\MM}\int_{\S\times\S}\frac{b^{-1}}{1-\gn^2}\Bigg((N'-\gn N)(h)+\bigg(\trt'-\gn \trt\\
\nn&&-\th'(N-\gn N', N-\gn N')-\gn {b'}^{-1}(N-\gn N')(b')\Big)\\
\nn&&+\frac{2\gn}{1-\gn^2}\Big(\th(N'-\gn N,N'-\gn N)\\
\nn&&-\gn \th'(N-\gn N', N-\gn N')\Big)\bigg)h\Bigg)\\
\nn&& \times F_{j,-1}(u)F_{j,-1}(u')\eta_j^\nu(\o)\eta_j^{\nu'}(\o')d\o d\o'd\MM.
\eea

Next, we compute the term $(N'-\gn N)(h)$. Proceeding as in \eqref{boca1:3}, \eqref{boca4:3}, \eqref{boca8:3} and \eqref{boca9:3}, we obtain schematically:
\bea\lab{boca9:6}
&&(N'-\gn N)(h)\\
\nn&=& b \nabb(L(P_l\trc))P_m\trc (N-N')+b L(P_l\trc)\nabb'(P_m\trc)(N-N')\\
\nn&&+b L(P_l\trc)N'(P_m\trc)(N-N')^2+(\chi-\chi')L(P_l\trc)P_m\trc\\
\nn&&+(\th+\th'+b^{-1}\nabb(b)+{b'}^{-1}\nabb(b'))L(P_l\trc)P_m\trc(N-N').
\eea
Thus, in view of \eqref{boca0:6}, \eqref{boca:6}, \eqref{boca9:6} and \eqref{boca10:3} we obtain:
\be\lab{boca11:6}
B^{1,2,3,1}_{j,\nu,\nu',l,m}= 2^{-2j}\int_{\MM}\int_{\S\times\S}H F_{j,-1}(u)F_{j,-1}(u')\eta_j^\nu(\o)\eta_j^{\nu'}(\o')d\o d\o'd\MM,
\ee
with the tensor $H$ on $\MM$ given, schematically, by:
\bea
\nn H&=&\frac{1}{1-\gn^2}\\
\nn&&\times\Bigg[\nabb(L(P_l\trc))P_m\trc (N-N')+L(P_l\trc)\nabb'(P_m\trc)(N-N')\\
\nn&&+L(P_l\trc)N'(P_m\trc)(N-N')^2+(\chi-\chi')L(P_l\trc)P_m\trc\\
\nn&&+(\th+\th'+b^{-1}\nabb(b)+{b'}^{-1}\nabb(b'))L(P_l\trc)P_m\trc(N-N')\Bigg].
\eea
Proceeding in the same fashion than \eqref{boca13:3}, we obtain, schematically:
\bea\lab{boca13:6}
H&=&\frac{1}{|N_\nu-N_{\nu'}|^2}\left(\sum_{p, q\geq 0}c_{pq}\left(\frac{N-N_\nu}{|N_\nu-N_{\nu'}|}\right)^p\left(\frac{N'-N_{\nu'}}{|N_\nu-N_{\nu'}|}\right)^q\right)\\
\nn&&\times\bigg(\frac{1}{|N_\nu-N_{\nu'}|^2}H_1+\frac{1}{|N_\nu-N_{\nu'}|}H_2+H_3\bigg),
\eea
where the tensors $H_1, H_2, H_3$ on $\MM$ are given by:
\be\lab{boca14:6}
H_1=(\chi-\chi')L(P_l\trc)P_m\trc,
\ee
\bea\lab{boca15:6}
H_2&=&\nabb(L(P_l\trc))P_m\trc+L(P_l\trc)\nabb'(P_m\trc)\\
\nn&&+(\th+\th'+b^{-1}\nabb(b)+{b'}^{-1}\nabb(b'))L(P_l\trc)P_m\trc,
\eea
and:
\be\lab{boca16:6}
H_3=L(P_l\trc)N'(P_m\trc),
\ee
and where $c_{pq}$ are explicit real coefficients such that the series 
$$\sum_{p, q\geq 0}c_{pq}x^py^q$$
has radius of convergence 1. In view of \eqref{boca11:6}, \eqref{boca13:6}, \eqref{boca14:6}, \eqref{boca15:6} and \eqref{boca16:6}, we obtain the decomposition \eqref{bizu10} \eqref{bizu11} \eqref{bizu12} \eqref{bizu13} \eqref{bizu14} \eqref{bizu15} \eqref{bizu16} \eqref{bizu17} of $B^{1,2,3,1}_{j,\nu,\nu',l,m}$. This concludes the proof of Lemma \ref{lemma:bizu}.

\section{Proof of Lemma \ref{lemma:ldc}}

Recall from \eqref{nice10} that $B^2_{j,\nu,\nu',l,m}$ is given by:
\bee
\nn B^2_{j,\nu,\nu',l,m}&=& -i2^{-j}\int_{\MM}\int_{\S\times\S} \frac{b^{-1}}{\gg(L,L')}\Bigg((\gn-1)P_l\trc N'(P_m\trc')+\bigg(\trc-\db-\db'\\
&& -(1-\gn)\d'-2\z'_{N-\gn N'}-\frac{\chi'(N-\gn N', N-\gn N')}{\gl}\bigg)\\
&&\times P_l\trc P_m\trc'\Bigg) F_j(u)F_{j,-1}(u')\eta_j^\nu(\o)\eta_j^{\nu'}(\o')d\o d\o' d\MM.
\eee
Together with the identity \eqref{nice24}:
$$\gl=-1+\gn,$$
we obtain:
\bee
B^2_{j,\nu,\nu',l,m}&=& -i2^{-j}\int_{\MM}\int_{\S\times\S} \Bigg[\frac{b^{-1}}{\gg(L,L')}\bigg(\trc-\db-\db'-(1-\gn)\d' -2\z'_{N-\gn N'}\\
\nn&&-\frac{\chi'(N-\gn N', N-\gn N')}{\gl}\bigg) P_l\trc P_m\trc'\\
&&+b^{-1}P_l\trc N'(P_m\trc')\Bigg] F_j(u)F_{j,-1}(u')\eta_j^\nu(\o)\eta_j^{\nu'}(\o')d\o d\o' d\MM.
\eee
We integrate by parts in $B^{1,1,1,2}_{j,\nu,\nu',l,m}$ using \eqref{fete1} with 
\bea\lab{boca0:7}
h&=& \frac{b'}{\gg(L,L')}\bigg(\trc-\db-\db'-(1-\gn)\d'-2\z'_{N-\gn N'}\\
\nn&&-\frac{\chi'(N-\gn N', N-\gn N')}{\gl}\bigg)P_l\trc P_m\trc'+b'P_l\trc N'(P_m\trc').
\eea
We obtain:
\bea\lab{boca:7}
&& B^2_{j,\nu,\nu',l,m}\\
\nn&=&   -2^{-2j}\int_{\MM}\int_{\S\times\S}\frac{{b'}^{-1}}{1-\gn^2}\Bigg((N-\gn N')(h)+\bigg(\trt-\gn \trt'\\
\nn&&-\th(N'-\gn N, N'-\gn N)-\gn b^{-1}(N'-\gn N)(b)\\
\nn&&+\frac{2\gn}{1-\gn^2}\Big(\th'(N-\gn N',N-\gn N')\\
\nn&&-\gn \th(N'-\gn N, N'-\gn N)\Big)\bigg)h\Bigg)\\
\nn&& \times F_{j,-1}(u)F_{j,-1}(u')\eta_j^\nu(\o)\eta_j^{\nu'}(\o')d\o d\o'd\MM.
\eea

Next, we compute the term $(N-\gn N')(h)$. Using the structure equation for $N$ \eqref{frame}, we have, schematically:
\bee
&&\nabla_{N-\gn N'}(N-\gn N')\\
&=&\nabla_{N-\gn N'}N-\gn \nabla_{N-\gn N'}N'-\nabla_{N-\gn N'}(\gn)N'\\
&=& (1-\gn^2)\nabn N-\gn\nabla_{N'-\gn N}N-\gn\th'(N-\gn N',e_{A'})e_{A'}\\
&&-((1-\gn^2)g(-\nabb\log(b),N')-\gn\th(N'-\gn N,N'-\gn N)\\
&&+\th'(N-\gn N',N-\gn N'))N'\\
& =& -(1-\gn^2)\nabb\log(b)-\gn\th(N'-\gn N, e_A)e_A\\
&&-\gn\th'(N-\gn N',e_{A'})e_{A'}+(1-\gn^2)\nabla_{N'-\gn N}\log(b) N' \\
&&+\gn\th(N'-\gn N,N'-\gn N)N'\\
&&-\th'(N-\gn N',N-\gn N')N',
\eee
which we rewrite schematically as:
\be\lab{jairdvavecvous}
\nabla_{N-\gn N'}(N-\gn N')=(\th-\th')(N-N')+(\th+\th'+b^{-1}\nabb(b))(N-N')^2.
\ee
Proceeding as in \eqref{boca1}, \eqref{boca4}, \eqref{boca8} and \eqref{boca9}, and using \eqref{jairdvavecvous}, we obtain schematically:
\bea\lab{boca9:7}
&&(N-\gn N')(h)\\
\nn&=& \Big(\chi+\db+\chi'+\db'+\z'(N-N')\Big)\frac{b'}{\gl}\Big(\nabb P_l\trc P_m\trc'(N-N')\\
\nn&&+P_l\trc\nabb' P_m\trc'(N-N')+N(P_l\trc)P_m\trc'(N-N')^2\Big)+\Big(\nabb(\chi)(N-N')\\
\nn&&+\nabb(\db)(N-N')+\dd_N(\chi)(N-N')^2+N(\db)(N-N')^2+\nabb'(\chi')(N-N')\\
\nn&&+\nabb'(\db')(N-N')+\nabb'(\z')(N-N')^2+{b'}^{-1}\nabb(b')(N-N')+(\chi+\db+\chi'+\db'\\
\nn&&+\z'(N-N'))(\th-\th'+(\th+\th'+b^{-1}\nabb(b))(N-N'))\Big)\frac{b'}{\gl}P_l\trc P_m\trc'\\
\nn&&+(1-\gn^2)b'N(P_l\trc)N'(P_m\trc')+b'(N'-\gn N)(P_l\trc)N(P_m\trc')\\
\nn&&+(N-N')b'P_t\trc\nabb'(N'(P_m\trc'))+(N-N'){b'}^{-1}\nabb'(b') P_l\trc N'(P_m\trc').
\eea
Thus, in view of \eqref{boca0:7}, \eqref{boca:7}, \eqref{boca9:7}, \eqref{boca7} and \eqref{boca10} we obtain:
\be\lab{boca11:7}
B^2_{j,\nu,\nu',l,m}= 2^{-2j}\int_{\MM}\int_{\S\times\S}H F_{j,-1}(u)F_{j,-1}(u')\eta_j^\nu(\o)\eta_j^{\nu'}(\o')d\o d\o'd\MM,
\ee
with the tensor $H$ on $\MM$ given, schematically, by:
\bee
\nn H&=&\frac{1}{1-\gn^2}\Bigg(\Big(\chi+\db+\chi'+\db'+\z'(N-N')\Big)\frac{1}{\gl}\Big(\nabb P_l\trc P_m\trc'(N-N')\\
\nn&&+P_l\trc\nabb' P_m\trc'(N-N')+N(P_l\trc)P_m\trc'(N-N')^2\Big)+\Big(\nabb(\chi)(N-N')\\
\nn&&+\nabb(\db)(N-N')+\dd_N(\chi)(N-N')^2+N(\db)(N-N')^2+\nabb'(\chi')(N-N')\\
\nn&&+\nabb'(\db')(N-N')+\nabb'(\z')(N-N')^2+{b'}^{-1}\nabb(b')(N-N')+(\chi+\db+\chi'+\db'\\
\nn&&+\z'(N-N'))(\chi-\chi'+(\th+\th'+b^{-1}\nabb(b))(N-N'))\Big)\frac{1}{\gl}P_l\trc P_m\trc'\\
\nn&&+(1-\gn^2)b'N(P_l\trc)N'(P_m\trc')+(N'-\gn N)(P_l\trc)N(P_m\trc')\\
\nn&&+(N-N')P_l\trc\nabb'(N'(P_m\trc'))+(N-N'){b'}^{-1}\nabb'(b') P_l\trc N'(P_m\trc')\Bigg).
\eee
Proceeding in the same fashion than \eqref{boca13}, we obtain, schematically:
\bea\lab{boca13:7}
H&=&\frac{1}{|N_\nu-N_{\nu'}|}\left(\sum_{p, q\geq 0}c_{pq}\left(\frac{N-N_\nu}{|N_\nu-N_{\nu'}|}\right)^p\left(\frac{N'-N_{\nu'}}{|N_\nu-N_{\nu'}|}\right)^q\right)\\
\nn&&\times\bigg(\frac{1}{|N_{\nu}-N_{\nu'}|^3}H_1+\frac{1}{|N_\nu-N_{\nu'}|^2}H_2+\frac{1}{|N_{\nu}-N_{\nu'}|}H_3+H_4\bigg)\\
\nn&&+N(P_l\trc)N'(P_m\trc')+\frac{(N'-\gn N)(P_l\trc)N(P_m\trc')}{1-\gn^2},
\eea
where the tensors $H_1, H_2, H_3$ and $H_4$ on $\MM$ are given by:
\be\lab{boca14:7}
H_1=(\chi-\chi')(\chi+\db+\chi'+\db')P_l\trc P_m\trc',
\ee
\bea\lab{boca15:7}
H_2&=& (\chi+\db+\chi'+\db')(\nabb P_l\trc P_m\trc'+P_l\trc\nabb' P_m\trc')\\
\nn&&+\Big(\nabb(\chi)+\nabb(\db)+\nabb'(\chi')+\nabb'(\db')+{b'}^{-1}\nabb(b')\\
\nn&&+(\chi+\db+\chi'+\db')(\th+\th'+b^{-1}\nabb(b))+(\chi-\chi')\z'\Big)P_l\trc P_m\trc',
\eea
\bea\lab{boca16:7}
H_3&=&(\chi+\db+\chi'+\db')N(P_l\trc) P_m\trc'+\z'(\nabb P_l\trc P_m\trc'+P_l\trc\nabb' P_m\trc')\\
\nn&&+\Big(\dd_N(\chi)+N(\db)+\nabb'(\z')+\z'(\th+\th'+b^{-1}\nabb(b))\Big) P_l\trc P_m\trc',
\eea
and:
\be\lab{boca17:7}
H_4=P_l\trc\nabb'(N'(P_m\trc'))+{b'}^{-1}\nabb'(b') P_l\trc N'(P_m\trc'),
\ee
and where $c_{pq}$ are explicit real coefficients such that the series 
$$\sum_{p, q\geq 0}c_{pq}x^py^q$$
has radius of convergence 1. In view of \eqref{boca11:7}, \eqref{boca13:7}, \eqref{boca14:7}, \eqref{boca15:7}, \eqref{boca16:7} and \eqref{boca17:7}, we obtain the decomposition \eqref{ldc1} \eqref{ldc2} \eqref{ldc3} \eqref{ldc4} \eqref{ldc5} \eqref{ldc6} \eqref{ldc7} \eqref{ldc8} \eqref{ldc9} \eqref{ldc10} \eqref{ldc11} \eqref{ldc12} \eqref{ldc13} \eqref{ldc14} \eqref{ldc15} of $B^2_{j,\nu,\nu',l,m}$. This concludes the proof of Lemma \ref{lemma:ldc}.

\section{Proof of Lemma \ref{lemma:duc}}

Recall from \eqref{duc0} that $\sum_{m\leq l}B^{2,2}_{j,\nu,\nu',l,m}$ is given by:
\bee
\sum_{m\leq l}B^{2,2}_{j,\nu,\nu',l,m}&=&2^{-2j}\int_{\MM}\int_{\S\times\S}\frac{(N'-\gn N)(P_l\trc)N'(P_{\leq l}\trc')}{1-\gn^2}\\
\nn&&\times F_{j,-1}(u)F_{j,-1}(u')\eta_j^\nu(\o)\eta_j^{\nu'}(\o')d\o d\o'd\MM.
\eee
We integrate by parts using \eqref{fete1} with 
\bea\lab{boca0:8}
h&=& \frac{bb'(N'-\gn N)(P_l\trc)N(P_{\leq l}\trc')}{1-\gn^2}.
\eea
We obtain:
\bea\lab{boca:8}
&& \sum_{m\leq l}B^{2,2}_{j,\nu,\nu',l,m}\\
\nn&=&   -i2^{-3j}\int_{\MM}\int_{\S\times\S}\frac{{b'}^{-1}}{1-\gn^2}\Bigg((N-\gn N')(h)+\bigg(\trt-\gn \trt'\\
\nn&&-\th(N'-\gn N, N'-\gn N)-\gn b^{-1}(N'-\gn N)(b)\\
\nn&&+\frac{2\gn}{1-\gn^2}\Big(\th'(N-\gn N',N-\gn N')\\
\nn&&-\gn \th(N'-\gn N, N'-\gn N)\Big)\bigg)h\Bigg)\\
\nn&& \times F_{j,-1}(u)F_{j,-1}(u')\eta_j^\nu(\o)\eta_j^{\nu'}(\o')d\o d\o'd\MM.
\eea

Next, we compute the term $(N-\gn N')(h)$. Using the structure equation for $N$ \eqref{frame}, we have, schematically:
\be\lab{jairdvavecvous1}
\nabla_N(N'-\gn N)=b^{-1}\nabb(b)+{b'}^{-1}\nabb'(b')+(N-N')\th'.
\ee
Proceeding as in \eqref{boca1}, \eqref{boca4}, \eqref{boca8} and \eqref{boca9}, and using \eqref{jairdvavecvous}and \eqref{jairdvavecvous1}, we obtain schematically:
\bea\lab{boca9:8}
&&(N-\gn N')(h)\\
\nn&=& \frac{bb'}{1-\gn^2}\Bigg((N-N')^2\nabb^2(P_l\trc) N'(P_{\leq l}\trc')+(N-N')^3\nabb(N(P_l\trc)) N'(P_{\leq l}\trc')\\
\nn &&+(N-N')^2\nabb(P_l\trc)\nabb' N'(P_{\leq l}\trc')+\Big((N-N')(\th-\th')\\
\nn&&+(N-N')^2(\th+\th'+b^{-1}\nabb(b)+{b'}^{-1}\nabb'(b'))+(N-N')^3(\th+\th'+N(b))\Big)\nabla(P_l\trc) N'(P_{\leq l}\trc')\Bigg)\\
\nn&&+\frac{bb'}{(1-\gn^2)^2}\Bigg((N-N')^3(\th-\th')+(N-N')^4b^{-1}\nabb(b)\Bigg)\nabb(P_l\trc) N'(P_{\leq l}\trc').
\eea
Thus, in view of \eqref{boca0:8}, \eqref{boca:8}, \eqref{boca9:8}, \eqref{boca7} and \eqref{boca10} we obtain:
\be\lab{boca11:8}
\sum_{m\leq l}B^{2,2}_{j,\nu,\nu',l,m}= 2^{-3j}\int_{\MM}\int_{\S\times\S}H F_{j,-1}(u)F_{j,-1}(u')\eta_j^\nu(\o)\eta_j^{\nu'}(\o')d\o d\o'd\MM,
\ee
with the tensor $H$ on $\MM$ given, schematically, by:
\bee
\nn H&=& \frac{1}{(1-\gn^2)^2}\Bigg((N-N')^2\nabb^2(P_l\trc) N'(P_{\leq l}\trc')+(N-N')^3\nabb(N(P_l\trc)) N'(P_{\leq l}\trc')\\
\nn &&+(N-N')^2\nabb(P_l\trc)\nabb' N'(P_{\leq l}\trc')+\Big((N-N')(\chi-\chi')\\
\nn&&+(N-N')^2(\th+\th'+b^{-1}\nabb(b)+{b'}^{-1}\nabb'(b'))+(N-N')^3(\th+\th'+N(b))\Big)\nabla(P_l\trc) N'(P_{\leq l}\trc')\Bigg)\\
\nn&&+\frac{1}{(1-\gn^2)^3}\Bigg((N-N')^3(\chi-\chi')+(N-N')^4b^{-1}\nabb(b)\Bigg)\nabb(P_l\trc) N'(P_{\leq l}\trc').
\eee
Proceeding in the same fashion than \eqref{boca13}, we obtain, schematically:
\bea\lab{boca13:8}
H&=&\frac{1}{|N_\nu-N_{\nu'}|}\left(\sum_{p, q\geq 0}c_{pq}\left(\frac{N-N_\nu}{|N_\nu-N_{\nu'}|}\right)^p\left(\frac{N'-N_{\nu'}}{|N_\nu-N_{\nu'}|}\right)^q\right)\\
\nn&&\times\bigg(\frac{1}{|N_\nu-N_{\nu'}|^2}H_1+\frac{1}{|N_{\nu}-N_{\nu'}|}H_2+H_3\bigg)\\
\nn&&+N(P_l\trc)N'(P_m\trc')+\frac{(N'-\gn N)(P_l\trc)N(P_m\trc')}{1-\gn^2},
\eea
where the tensors $H_1, H_2$ and $H_3$ on $\MM$ are given by:
\be\lab{boca14:8}
H_1=(\chi-\chi')\nabla(P_l\trc) N'(P_{\leq l}\trc'),
\ee
\bea\lab{boca15:8}
H_2&=& \nabb^2(P_l\trc) N'(P_{\leq l}\trc')+\nabb(P_l\trc)\nabb' N'(P_{\leq l}\trc')\\
\nn&& +(\th+\th'+b^{-1}\nabb(b)+{b'}^{-1}\nabb'(b'))\nabla(P_l\trc) N'(P_{\leq l}\trc'),
\eea
and:
\be\lab{boca16:8}
H_3= \nabb(N(P_l\trc)) N'(P_{\leq l}\trc')+(\th+\th'+N(b))\nabla(P_l\trc) N'(P_{\leq l}\trc'),
\ee
and where $c_{pq}$ are explicit real coefficients such that the series 
$$\sum_{p, q\geq 0}c_{pq}x^py^q$$
has radius of convergence 1. In view of \eqref{boca11:8}, \eqref{boca13:8}, \eqref{boca14:8}, \eqref{boca15:8} and  \eqref{boca16:8}, we obtain the decomposition \eqref{duc1} \eqref{duc2} \eqref{duc3} \eqref{duc4} \eqref{duc5} \eqref{duc6} \eqref{duc7} \eqref{duc8} \eqref{duc9} of $\sum_{m\leq l}B^{2,2}_{j,\nu,\nu',l,m}$. This concludes the proof of Lemma \ref{lemma:duc}.

%%%%%%%%%%%%%%%%%%%%%%%%%%%%%%%%%

\section{Proof of Lemma \ref{lemma:zoo}}

Recall from \eqref{zoo} that $\sum_{(l,m)/2^{\min(l,m)}\leq 2^j|\nu-\nu'|<2^{\max(l,m)}}A_{j,\nu,\nu',l,m}$ is given by:
\bee
&&\sum_{(l,m)/2^{\min(l,m)}\leq 2^j|\nu-\nu'|<2^{\max(l,m)}}A_{j,\nu,\nu',l,m}\\
\nn&=& -i2^{-j}\int_{\MM}\int_{\S\times\S} \frac{P_{>2^j|\nu-\nu'|}\trc (N-\gn N')(P_{\leq 2^j|\nu-\nu'|}\trc')}{\gg(L,L')}\\
\nn&&\times F_j(u)F_{j,-1}(u')\eta_j^\nu(\o)\eta_j^{\nu'}(\o')d\o d\o' d\MM.
\eee
We integrate by parts in $\sum_{(l,m)/2^{\min(l,m)}\leq 2^j|\nu-\nu'|<2^{\max(l,m)}}A_{j,\nu,\nu',l,m}$ using \eqref{fetebis} with 
\be\lab{boca0:9}
h= \frac{bb' P_{>2^j|\nu-\nu'|}\trc (N-\gn N')(P_{\leq 2^j|\nu-\nu'|}\trc')}{\gg(L,L')}.
\ee
We obtain:
\bea\lab{boca:9}
&&\sum_{(l,m)/2^{\min(l,m)}\leq 2^j|\nu-\nu'|<2^{\max(l,m)}}A_{j,\nu,\nu',l,m}\\
\nn&=&   -2^{-2j}\int_{\MM}\int_{\S\times\S}\frac{b^{-1}}{\gl}\bigg(L(h)+\trc h-\db h-\db'h-(1-\gn)\d'h\\
\nn&& -2\z'_{N-\gn N'}h-\frac{\chi'(N-\gn N', N-\gn N')}{\gl}h\bigg)\\
\nn&& F_j(u)F_{j,-1}(u')\eta_j^\nu(\o)\eta_j^{\nu'}(\o')d\o d\o'd\MM.
\eea

Next, we compute the term $L(h)$. We have:
\bea\lab{boca1:9}
L(h)&=& \frac{bb' L(P_{>2^j|\nu-\nu'|}\trc) (N-\gn N')(P_{\leq 2^j|\nu-\nu'|}\trc')}{\gg(L,L')}\\
\nn&&+\frac{bb' P_{>2^j|\nu-\nu'|}\trc L((N-\gn N')(P_{\leq 2^j|\nu-\nu'|}\trc'))}{\gg(L,L')}\\
\nn&&+\frac{(L(b)+L(b'))P_{>2^j|\nu-\nu'|}\trc (N-\gn N')(P_{\leq 2^j|\nu-\nu'|}\trc')}{\gg(L,L')}\\
\nn&&-\frac{L(\gl)bb' P_{>2^j|\nu-\nu'|}\trc (N-\gn N')(P_{\leq 2^j|\nu-\nu'|}\trc')}{\gg(L,L')^2}.
\eea
Decomposing $L$ on $L', N'$ and $N-\gn N'$, we have:
\be\lab{boca3:9}
L=L'+(N-\gn N')+(\gn-1)N',
\ee
which yields schematically for the derivative in the second term in the right-hand side of \eqref{boca1:9}:
\bea\lab{boca4:9}
&&L((N-\gn N')(P_{\leq 2^j|\nu-\nu'|}\trc'))\\
\nn&=& L'((N-\gn N')(P_{\leq 2^j|\nu-\nu'|}\trc'))+(N-\gn N')((N-\gn N')(P_{\leq 2^j|\nu-\nu'|}\trc'))\\
\nn&&+(\gn-1)N'((N-\gn N')(P_{\leq 2^j|\nu-\nu'|}\trc'))\\
\nn&=& (N-N')\nabb(L'(P_{\leq 2^j|\nu-\nu'|}\trc'))+[L',N-\gn N'](P_{\leq 2^j|\nu-\nu'|}\trc')\\
\nn&&+(N-N')^2{\nabb'}^2(P_{\leq 2^j|\nu-\nu'|}\trc'))+\nabla_{\nabla_{N-\gn N'}(N-\gn N')}(P_{\leq 2^j|\nu-\nu'|}\trc')\\
\nn&&+(N-N')^3\nabb'(N'(P_{\leq 2^j|\nu-\nu'|}\trc'))+(N-N')^2 [N',N-\gn N'](P_{\leq 2^j|\nu-\nu'|}\trc'),
\eea
where we used in the last inequality the fact that, schematically, $1-\gn= (N-N')^2$. Next, we compute the two  commutators in the right-hand side of \eqref{boca4:9}. Using the structure equation for $N$ \eqref{frame}, we have, schematically:
\be\lab{boca5:9}
 [N',N-\gn N'] =b^{-1}\nabb(b)+{b'}^{-1}\nabb'(b')+(N-N')(\th+\th').
\ee
Also, using the fact that $L=T+N$, $L'=T+N'$ and $\gl=-1+\gn$, we have:
\bee
[L',N-\gn N']&=&[L', L-\gn L'+(\gn-1)T]\\
&=& [L',L]-L'(\gl)N'+(\gn-1)[L',T]
\eee
which together with the Ricci equations \eqref{ricciform} implies, schematically:
\bee
[L',N-\gn N']&=&-\db L+\db'L'+(N-N')(\chi+\chi'+\kep+\kep')\\
&&+(N-N')^2(\zb+\z+\d+n^{-1}\nabla n+\db'+\chi).
\eee
Using the analog of \eqref{loeb25} \eqref{loeb26} for $-\db L+\db'L'$, we finally obtain, schematically:
\bea\lab{boca6:9}
[L',N-\gn N']&=& (N-N')(\chi+\chi'+\kep+\kep'+\d+\d'+n^{-1}\nabla n)\\
\nn&&+(N-N')^2(k+n^{-1}\nabla n+\chi+\z).
\eea
Now, in view of \eqref{jairdvavecvous}, \eqref{boca4:9}, \eqref{boca5:9} and \eqref{boca6:9}, we obtain, schematically:
\bea\lab{boca7:9}
&&L((N-\gn N')(P_{\leq 2^j|\nu-\nu'|}\trc'))\\
\nn&=& (N-N')\nabb(L'(P_{\leq 2^j|\nu-\nu'|}\trc'))+(N-N')^2{\nabb'}^2(P_{\leq 2^j|\nu-\nu'|}\trc'))\\
\nn&&+(N-N')^3\nabb'(N'(P_{\leq 2^j|\nu-\nu'|}\trc'))\\
\nn&& +(N-N')(\chi+\chi'+\kep+\kep'+\d+\d'+n^{-1}\nabla n)\nabb(P_{\leq 2^j|\nu-\nu'|}\trc')\\
\nn&&+(N-N')^2(k+n^{-1}\nabla n+\th+\th'+b^{-1}\nabb(b)+{b'}^{-1}\nabb'(b')+\chi+\z)\nabla(P_{\leq 2^j|\nu-\nu'|}\trc').
\eea
Also, in view of \eqref{fete6bis}, we have, schematically:
\be\lab{boca8:9}
L(\gg(L,L'))= (N-N')^2(\d+\d'+n^{-1}\nabla n+\chi')+(N-N')^3\z'.
\ee
Finally, \eqref{boca1:9}, \eqref{boca7:9} and \eqref{boca8:9} yield, schematically:
\bea\lab{boca9:9}
&& L(h)\\
\nn&=& \frac{1}{\gl}\Bigg[bb' (N-N')L(P_{>2^j|\nu-\nu'|}\trc)\nabb'(P_{\leq 2^j|\nu-\nu'|}\trc')\\
\nn&&+bb' P_{>2^j|\nu-\nu'|}\trc\bigg((N-N')\nabb(L'(P_{\leq 2^j|\nu-\nu'|}\trc'))+(N-N')^2{\nabb'}^2(P_{\leq 2^j|\nu-\nu'|}\trc')\\
\nn&&+(N-N')^3\nabb'(N'(P_{\leq 2^j|\nu-\nu'|}\trc'))\\
\nn&& +(N-N')(\chi+\chi'+\kep+\kep'+\d+\d'+n^{-1}\nabla n+L(b)+L'(b'))\nabb(P_{\leq 2^j|\nu-\nu'|}\trc')\\
\nn&&+(N-N')^2(k+n^{-1}\nabla n+\th+\th'+b^{-1}\nabb(b)+{b'}^{-1}\nabb'(b')+\chi+\z\\
\nn&&+\nabla_{N'}(b'))\nabla(P_{\leq 2^j|\nu-\nu'|}\trc')\bigg)\Bigg]\\
\nn&&-\frac{\left((N-N')^3(\d+\d'+n^{-1}\nabla n+\chi')+(N-N')^4\z'\right)bb' P_{>2^j|\nu-\nu'|}\trc 
\nabb'(P_{\leq 2^j|\nu-\nu'|}\trc')}{\gg(L,L')^2}.
\eea

We consider the term multiplied by $h$ in the right-hand side of \eqref{boca:9}. We have schematically:
\bea\lab{boca10:9}
&&\trc -\db -\db'-(1-\gn)\d'-2\z'_{N-\gn N'}\\
\nn&&-\frac{\chi'(N-\gn N', N-\gn N')}{\gl}\\
\nn&=& \chi+\db+\db'+(N-N')\z'+\frac{(N-N')^2\chi'}{\gl}.
\eea
Thus, in view of \eqref{boca0:9}, \eqref{boca:9}, \eqref{boca9:9} and \eqref{boca10:9} we obtain:
\bea\lab{boca11:9}
&&\sum_{(l,m)/2^{\min(l,m)}\leq 2^j|\nu-\nu'|<2^{\max(l,m)}}A_{j,\nu,\nu',l,m}\\
\nn&=& 2^{-2j}\int_{\MM}\int_{\S\times\S}H F_{j,-1}(u)F_{j,-1}(u')\eta_j^\nu(\o)\eta_j^{\nu'}(\o')d\o d\o'd\MM,
\eea
with the tensor $H$ on $\MM$ given, schematically, by:
\bea\lab{boca12:9}
\nn H&=& \frac{1}{\gl^2}\Bigg[(N-N')L(P_{>2^j|\nu-\nu'|}\trc)\nabb'(P_{\leq 2^j|\nu-\nu'|}\trc')\\
\nn&&+P_{>2^j|\nu-\nu'|}\trc\bigg((N-N')\nabb(L'(P_{\leq 2^j|\nu-\nu'|}\trc'))+(N-N')^2{\nabb'}^2(P_{\leq 2^j|\nu-\nu'|}\trc')\\
\nn&&+(N-N')^3\nabb'(N'(P_{\leq 2^j|\nu-\nu'|}\trc'))\\
\nn&& +(N-N')(\chi+\chi'+\kep+\kep'+\d+\d'+n^{-1}\nabla n+L(b)+L'(b'))\nabb(P_{\leq 2^j|\nu-\nu'|}\trc')\\
\nn&&+(N-N')^2(k+n^{-1}\nabla n+\th+\th'+b^{-1}\nabb(b)+{b'}^{-1}\nabb'(b')+\chi+\z+\z'\\
\nn&&+\nabla_{N'}(b'))\nabla(P_{\leq 2^j|\nu-\nu'|}\trc')\bigg)\Bigg]\\
\nn&&-\frac{\left((N-N')^3(\d+\d'+n^{-1}\nabla n+\chi')+(N-N')^4\z'\right)P_{>2^j|\nu-\nu'|}\trc 
\nabb'(P_{\leq 2^j|\nu-\nu'|}\trc')}{\gl^3}
\eea
Recall the identities \eqref{nice24} and \eqref{nice25}:
$$\gg(L,L')=-1+\gn\textrm{ and }1-\gn=\frac{\gg(N-N',N-N')}{2}.$$
We may thus expand:
$$\frac{1}{\gl^2}\textrm{ and }\frac{1}{\gl^3}$$
in the same fashion than \eqref{nice27}, and we obtain, schematically:
\bea\lab{boca13:9}
H&=&\frac{1}{|N_\nu-N_{\nu'}|}\left(\sum_{p, q\geq 0}c_{pq}\left(\frac{N-N_\nu}{|N_\nu-N_{\nu'}|}\right)^p\left(\frac{N'-N_{\nu'}}{|N_\nu-N_{\nu'}|}\right)^q\right)\\
\nn&&\times\bigg(H_1+\frac{1}{|N_\nu-N_{\nu'}|}H_2+\frac{1}{|N_{\nu}-N_{\nu'}|^2}H_3\bigg),
\eea
where the tensors $H_1, H_2$ and $H_3$ on $\MM$ are given by:
\be\lab{boca14:9}
H_1= P_{>2^j|\nu-\nu'|}\trc\nabb'(N'(P_{\leq 2^j|\nu-\nu'|}\trc')),
\ee
\bea\lab{boca15:9}
H_2&=& P_{>2^j|\nu-\nu'|}\trc\bigg({\nabb'}^2(P_{\leq 2^j|\nu-\nu'|}\trc')+(k+n^{-1}\nabla n+\th+\th'\\
\nn&&+b^{-1}\nabb(b)+{b'}^{-1}\nabb'(b')+\chi+\z+\z'+\nabla_{N'}(b'))\nabla(P_{\leq 2^j|\nu-\nu'|}\trc')\bigg),
\eea
and:
\bea\lab{boca16:9}
H_3&=&L(P_{>2^j|\nu-\nu'|}\trc)\nabb'(P_{\leq 2^j|\nu-\nu'|}\trc')+P_{>2^j|\nu-\nu'|}\trc\bigg(\nabb(L'(P_{\leq 2^j|\nu-\nu'|}\trc'))\\
\nn&&+(\chi+\chi'+\kep+\kep'+\d+\d'+n^{-1}\nabla n+L(b)+L'(b'))\nabb(P_{\leq 2^j|\nu-\nu'|}\trc')\bigg),
\eea
and where $c_{pq}$ are explicit real coefficients such that the series 
$$\sum_{p, q\geq 0}c_{pq}x^py^q$$
has radius of convergence 1. In view of \eqref{boca11:9}, \eqref{boca13:9}, \eqref{boca14:9}, \eqref{boca15:9} and \eqref{boca16:9}, we obtain the decomposition \eqref{zoo1} \eqref{zoo2} \eqref{zoo3} \eqref{zoo4} \eqref{zoo5} \eqref{zoo6} \eqref{zoo7} \eqref{zoo8} \eqref{zoo9} \eqref{zoo10} \eqref{zoo11} \eqref{zoo12} \eqref{zoo13} of $\sum_{(l,m)/2^{\min(l,m)}\leq 2^j|\nu-\nu'|<2^{\max(l,m)}}A_{j,\nu,\nu',l,m}$. This concludes the proof of Lemma \ref{lemma:zoo}.

%%%%%%%%%%%%%%%%%%%%%%%%%%%%%%%%%%

\section{Proof of Lemma \ref{lemma:zol}}

Recall from \eqref{zol} that $A_{j,\nu,\nu',l,m}$ is given by:
\bee
A_{j,\nu,\nu',l,m}&=& -i2^{-j}\int_{\MM}\int_{\S\times\S} \frac{P_l\trc (N-\gn N')(P_m\trc')}{\gg(L,L')}\\
\nn&&\times F_j(u)F_{j,-1}(u')\eta_j^\nu(\o)\eta_j^{\nu'}(\o')d\o d\o' d\MM.
\eee
We integrate by parts using \eqref{fete1} with 
\bea\lab{boca0:10}
h&=& \frac{bb'P_l\trc (N-\gn N')(P_m\trc')}{\gl}.
\eea
We obtain:
\bea\lab{boca:10}
&& A_{j,\nu,\nu',l,m}\\
\nn&=&   -2^{-2j}\int_{\MM}\int_{\S\times\S}\frac{{b'}^{-1}}{1-\gn^2}\Bigg((N-\gn N')(h)+\bigg(\trt-\gn \trt'\\
\nn&&-\th(N'-\gn N, N'-\gn N)-\gn b^{-1}(N'-\gn N)(b)\\
\nn&&+\frac{2\gn}{1-\gn^2}\Big(\th'(N-\gn N',N-\gn N')\\
\nn&&-\gn \th(N'-\gn N, N'-\gn N)\Big)\bigg)h\Bigg)\\
\nn&& \times F_{j,-1}(u)F_{j,-1}(u')\eta_j^\nu(\o)\eta_j^{\nu'}(\o')d\o d\o'd\MM.
\eea

Next, we compute the term $(N-\gn N')(h)$. Proceeding as in \eqref{boca1}, \eqref{boca4}, \eqref{boca8} and \eqref{boca9}, and using \eqref{jairdvavecvous}, we obtain schematically:
\bea\lab{boca9:10}
&&(N-\gn N')(h)\\
\nn&=& \frac{bb'}{\gl}\Bigg((N'-\gn N)(P_l\trc)(N-\gn N')(P_m\trc')\\
\nn&& +P_l\trc {\nabb'}^2P_m(\trc')(N-\gn N',N-\gn N')\\
\nn&& +(1-\gn^2)N(P_l\trc)(N-\gn N')(P_m\trc')+(N-N')(\th-\th')P_l\trc \nabb'(P_m\trc')\\
\nn&&+(N-N')^2(\th+\th'+b^{-1}\nabla(b)+{b'}^{-1}\nabb'(b'))P_l\trc \nabb'(P_m\trc')\Bigg)\\
\nn&&+\frac{bb'}{\gl^2}\Bigg((N-N')^3(\th-\th')+(N-N')^4(\th+\th'+b^{-1}\nabb(b))\Bigg) P_l\trc \nabb'(P_m\trc').
\eea
Thus, in view of \eqref{boca0:10}, \eqref{boca:10}, \eqref{boca9:10}, \eqref{boca7} and \eqref{boca10} we obtain:
\be\lab{boca11:10}
A_{j,\nu,\nu',l,m}= 2^{-2j}\int_{\MM}\int_{\S\times\S}H F_{j,-1}(u)F_{j,-1}(u')\eta_j^\nu(\o)\eta_j^{\nu'}(\o')d\o d\o'd\MM,
\ee
with the tensor $H$ on $\MM$ given, schematically, by:
\bee
&&H\\
&=&  \frac{1}{\gl(1-\gn^2)}\Bigg((N'-\gn N)(P_l\trc)(N-\gn N')(P_m\trc')\\
\nn&& +P_l\trc {\nabb'}^2P_m(\trc')(N-\gn N',N-\gn N')\\
\nn&& +(1-\gn^2)N(P_l\trc)(N-\gn N')(P_m\trc')+(N-N')(\chi-\chi')P_l\trc \nabb'(P_m\trc')\\
\nn&&+(N-N')^2(\th+\th'+b^{-1}\nabla(b)+{b'}^{-1}\nabb'(b'))P_l\trc \nabb'(P_m\trc')\Bigg)\\
\nn&&+\frac{1}{\gl^2(1-\gn^2)}\Bigg((N-N')^3(\chi-\chi')+(N-N')^4(\th+\th'+b^{-1}\nabb(b))\Bigg)\\
\nn&&\times P_l\trc \nabb'(P_m\trc').
\eee
Proceeding in the same fashion than \eqref{boca13}, we obtain, schematically:
\bea
\nn H&=&\frac{1}{|N_\nu-N_{\nu'}|^2}\left(\sum_{p, q\geq 0}c_{pq}\left(\frac{N-N_\nu}{|N_\nu-N_{\nu'}|}\right)^p\left(\frac{N'-N_{\nu'}}{|N_\nu-N_{\nu'}|}\right)^q\right)\left(\frac{1}{|N_\nu-N_{\nu'}|}H_1+H_2\right)\\
\nn&& +\frac{(N'-\gn N)(P_l\trc)(N-\gn N')(P_m\trc')}{\gl(1-\gn^2)}\\
\nn&& +\frac{P_l\trc {\nabb'}^2P_m(\trc')(N-\gn N',N-\gn N')}{\gl(1-\gn^2)}\\
\lab{boca13:10}&& +\frac{N(P_l\trc)(N-\gn N')(P_m\trc')}{\gl},
\eea
where the tensors $H_1$ and $H_2$ on $\MM$ are given by:
\be\lab{boca14:10}
H_1=(\chi-\chi')P_l\trc \nabb'(P_m\trc'),
\ee
and:
\bea\lab{boca15:10}
H_2&=& (\th+\th'+b^{-1}\nabla(b)+{b'}^{-1}\nabb'(b'))P_l\trc \nabb'(P_m\trc'),
\eea
and where $c_{pq}$ are explicit real coefficients such that the series 
$$\sum_{p, q\geq 0}c_{pq}x^py^q$$
has radius of convergence 1. In view of \eqref{boca11:10}, \eqref{boca13:10}, \eqref{boca14:10} and \eqref{boca15:10}, we obtain the decomposition \eqref{zol1} \eqref{zol2} \eqref{zol3} \eqref{zol4} \eqref{zol5} \eqref{zol6} \eqref{zol7} \eqref{zol8} of $A_{j,\nu,\nu',l,m}$. This concludes the proof of Lemma \ref{lemma:zol}.

%%%%%%%%%%%%%%%%%%%%%%%%%%%%%%%%%%%

\section{Proof of Lemma \ref{lemma:buz}}

Recall from \eqref{zol} that $A^1_{j,\nu,\nu',l,m}$ is given by:
\bee
\nn A^1_{j,\nu,\nu',l,m}&=& 2^{-2j}\int_{\MM}\int_{\S\times\S}\frac{P_l\trc {\nabb'}^2P_m(\trc')(N-\gn N',N-\gn N')}{\gl(1-\gn^2)}\\
&&\times F_j(u)F_{j,-1}(u')\eta_j^\nu(\o)\eta_j^{\nu'}(\o')d\o d\o' d\MM.
\eee
We integrate by parts using \eqref{fete1} with 
\bea\lab{boca0:11}
h&=& \frac{bb'P_l\trc {\nabb'}^2P_m(\trc')(N-\gn N',N-\gn N')}{\gl(1-\gn^2)}.
\eea
We obtain:
\bea\lab{boca:11}
&& A^1_{j,\nu,\nu',l,m}\\
\nn&=&   -i2^{-3j}\int_{\MM}\int_{\S\times\S}\frac{{b'}^{-1}}{1-\gn^2}\Bigg((N-\gn N')(h)+\bigg(\trt-\gn \trt'\\
\nn&&-\th(N'-\gn N, N'-\gn N)-\gn b^{-1}(N'-\gn N)(b)\\
\nn&&+\frac{2\gn}{1-\gn^2}\Big(\th'(N-\gn N',N-\gn N')\\
\nn&&-\gn \th(N'-\gn N, N'-\gn N)\Big)\bigg)h\Bigg)\\
\nn&& \times F_{j,-1}(u)F_{j,-1}(u')\eta_j^\nu(\o)\eta_j^{\nu'}(\o')d\o d\o'd\MM.
\eea

Next, we compute the term $(N-\gn N')(h)$. Proceeding as in \eqref{boca1}, \eqref{boca4}, \eqref{boca8} and \eqref{boca9}, and using \eqref{jairdvavecvous}, we obtain schematically:
\bea\lab{boca9:11}
&&(N-\gn N')(h)\\
\nn&=& \frac{bb'}{\gl(1-\gn^2)}\Bigg((N-N')^3\nabb(P_l\trc){\nabb'}^2(P_m\trc')\\
\nn&&+(N-N')^4N(P_l\trc){\nabb'}^2(P_m\trc')+(N-N')^3P_l\trc {\nabb'}^3(P_m\trc')\\
\nn&&+\Big((N-N')^2(\th-\th')+(N-N')^3(\th+\th'+b^{-1}\nabla(b)+{b'}^{-1}\nabb'(b'))\Big)P_l\trc {\nabb'}^2(P_m\trc')\Bigg)\\
\nn&&+\frac{bb'}{\gl(1-\gn^2)}\left(\frac{1}{\gl}+\frac{1}{1-\gn^2}\right)\Bigg((N-N')^4(\th-\th')\\
\nn&&+(N-N')^5(\th+\th'+b^{-1}\nabb(b))\Bigg) P_l\trc {\nabb'}^2(P_m\trc').
\eea
Thus, in view of \eqref{boca0:11}, \eqref{boca:11}, \eqref{boca9:11}, \eqref{boca7} and \eqref{boca10} we obtain:
\be\lab{boca11:11}
A^1_{j,\nu,\nu',l,m}= 2^{-3j}\int_{\MM}\int_{\S\times\S}H F_{j,-1}(u)F_{j,-1}(u')\eta_j^\nu(\o)\eta_j^{\nu'}(\o')d\o d\o'd\MM,
\ee
with the tensor $H$ on $\MM$ given, schematically, by:
\bee
\nn H&=&  \frac{1}{\gl(1-\gn^2)^2}\Bigg((N-N')^3\nabb(P_l\trc){\nabb'}^2(P_m\trc')\\
\nn&&+(N-N')^4N(P_l\trc){\nabb'}^2(P_m\trc')+(N-N')^3P_l\trc {\nabb'}^3(P_m\trc')\\
\nn&&+\Big((N-N')^2(\chi-\chi')+(N-N')^3(\th+\th'+b^{-1}\nabla(b)+{b'}^{-1}\nabb'(b'))\Big)P_l\trc {\nabb'}^2(P_m\trc')\Bigg)\\
\nn&&+\frac{1}{\gl(1-\gn^2)^2}\left(\frac{1}{\gl}+\frac{1}{1-\gn^2}\right)\Bigg((N-N')^4(\chi-\chi')\\
\nn&&+(N-N')^5(\th+\th'+b^{-1}\nabb(b))\Bigg) P_l\trc {\nabb'}^2(P_m\trc').
\eee
Proceeding in the same fashion than \eqref{boca13}, we obtain, schematically:
\bea\lab{boca13:11}
 H&=&\frac{1}{|N_\nu-N_{\nu'}|^2}\left(\sum_{p, q\geq 0}c_{pq}\left(\frac{N-N_\nu}{|N_\nu-N_{\nu'}|}\right)^p\left(\frac{N'-N_{\nu'}}{|N_\nu-N_{\nu'}|}\right)^q\right)\\
\nn&&\times\left(\frac{1}{|N_\nu-N_{\nu'}|^2}H_1+\frac{1}{|N_\nu-N_{\nu'}|}H_2+H_3\right),
\eea
where the tensors $H_1$, $H_2$ and $H_3$ on $\MM$ are given by:
\be\lab{boca14:11}
H_1=(\chi-\chi')P_l\trc {\nabb'}^2(P_m\trc'),
\ee
\bea\lab{boca15:11}
H_2&=& \nabb(P_l\trc){\nabb'}^2(P_m\trc')+P_l\trc {\nabb'}^3(P_m\trc')\\
\nn&&+(\th+\th'+b^{-1}\nabla(b)+{b'}^{-1}\nabb'(b'))P_l\trc {\nabb'}^2(P_m\trc'),
\eea
and:
\be\lab{boca16:11}
H_3= N(P_l\trc){\nabb'}^2(P_m\trc'),
\ee
and where $c_{pq}$ are explicit real coefficients such that the series 
$$\sum_{p, q\geq 0}c_{pq}x^py^q$$
has radius of convergence 1. In view of \eqref{boca11:11}, \eqref{boca13:11}, \eqref{boca14:11}, \eqref{boca15:11} and \eqref{boca16:11}, we obtain the decomposition \eqref{buz1} \eqref{buz2} \eqref{buz3} \eqref{buz4} \eqref{buz5} \eqref{buz6} \eqref{buz7} \eqref{buz8} of $A^1_{j,\nu,\nu',l,m}$. This concludes the proof of Lemma \ref{lemma:buz}.

%%%%%%%%%%%%%%%%%%%%%%%%%%%%%%%%%%

\section{Proof of Lemma \ref{lemma:biz}}

Recall from \eqref{biz0} that $A^2_{j,\nu,\nu',l,m}$ is given by:
\bee
\nn A^2_{j,\nu,\nu',l,m}&=& 2^{-2j}\int_{\MM}\int_{\S\times\S}\frac{(N'-\gn N)(P_l\trc) (N-\gn N')(P_m(\trc'))}{\gl(1-\gn^2)}\\
\nn&&\times F_j(u)F_{j,-1}(u')\eta_j^\nu(\o)\eta_j^{\nu'}(\o')d\o d\o' d\MM,
\eee
We integrate by parts in $A^2_{j,\nu,\nu',l,m}$ using \eqref{fetebis} with 
\be\lab{boca0:12}
h= \frac{bb'(N'-\gn N)(P_l\trc) (N-\gn N')(P_m(\trc'))}{\gl(1-\gn^2)}.
\ee
We obtain:
\bea\lab{boca:12}
\nn A^2_{j,\nu,\nu',l,m}&=&   -i2^{-3j}\int_{\MM}\int_{\S\times\S}\frac{b^{-1}}{\gl}\bigg(L(h)+\trc h-\db h-\db'h-(1-\gn)\d'h\\
\nn&& -2\z'_{N-\gn N'}h-\frac{\chi'(N-\gn N', N-\gn N')}{\gl}h\bigg)\\
&& F_j(u)F_{j,-1}(u')\eta_j^\nu(\o)\eta_j^{\nu'}(\o')d\o d\o'd\MM.
\eea

Next, we compute the term $L(h)$. We have:
\bea\lab{boca1:12}
L(h)&=& \frac{bb' L((N'-\gn N)(P_l\trc)) (N-\gn N')(P_m\trc')}{\gg(L,L')(1-\gn^2)}\\
\nn&&+\frac{bb' (N'-\gn N)(P_l\trc) L((N-\gn N')(P_m\trc'))}{\gg(L,L')(1-\gn^2)}\\
\nn&&+\frac{(L(b)+L(b'))(N'-\gn N)(P_l\trc) (N-\gn N')(P_m\trc')}{\gg(L,L')(1-\gn^2)}\\
\nn&&-\left(\frac{1}{\gl}+\frac{1}{1-\gn^2}\right)\\
\nn&&\times\frac{L(\gl)bb' (N'-\gn N)(P_l\trc) (N-\gn N')(P_m\trc')}{\gg(L,L')(1-\gn^2)}.
\eea
In view of \eqref{boca6:9}, we have, schematically:
\bea\lab{boca6:12}
&&L((N'-\gn N)(P_l\trc))\\
\nn&=& (N-N')\nabb(L(P_l\trc))+(N-N')(\chi+\chi'+\kep+\kep'+\d+\d'+n^{-1}\nabla n)\nabb(P_l\trc)\\
\nn&&+(N-N')^2(k+n^{-1}\nabla n+\chi+\z)\nabla(P_l\trc).
\eea
Next, recall from \eqref{boca7:9} that we have, schematically:
\bea\lab{boca7:12}
&&L((N-\gn N')(P_m\trc'))\\
\nn&=& (N-N')\nabb(L'(P_m\trc'))+(N-N')^2{\nabb'}^2(P_m\trc'))\\
\nn&&+(N-N')^3\nabb'(N'(P_m\trc'))\\
\nn&& +(N-N')(\chi+\chi'+\kep+\kep'+\d+\d'+n^{-1}\nabla n)\nabb(P_m\trc')\\
\nn&&+(N-N')^2(k+n^{-1}\nabla n+\th+\th'+b^{-1}\nabb(b)+{b'}^{-1}\nabb'(b')+\chi+\z)\nabla(P_m\trc').
\eea
Also, in view of \eqref{fete6bis}, we have, schematically:
\be\lab{boca8:12}
L(\gg(L,L'))= (N-N')^2(\d+\d'+n^{-1}\nabla n+\chi')+(N-N')^3\z'.
\ee
Finally, \eqref{boca1:12}, \eqref{boca6:12}, \eqref{boca7:12} and \eqref{boca8:12} yield, schematically:
\bea\lab{boca9:12}
L(h)&=& \frac{1}{\gl(1-\gn^2)}\Bigg[bb' (N-N')^2\nabb(L(P_l\trc))\nabb'(P_m\trc')\\
\nn&&+bb' \nabb(P_l\trc)\bigg((N-N')^2\nabb(L'(P_m\trc'))+(N-N')^3{\nabb'}^2(P_m\trc')\\
\nn&&+(N-N')^4\nabb'(N'(P_m\trc'))\\
\nn&& +(N-N')^2(\chi+\chi'+\kep+\kep'+\d+\d'+n^{-1}\nabla n+L(b)+L'(b'))\nabb(P_m\trc')\\
\nn&&+(N-N')^3(k+n^{-1}\nabla n+\th+\th'+b^{-1}\nabb(b)+{b'}^{-1}\nabb'(b')+\chi+\z\\
\nn&&+\nabla_{N'}(b'))\nabla(P_m\trc')\bigg)\Bigg]+bb' \nabla(P_l\trc)\nabb'(P_m\trc')(N-N')^3(k+n^{-1}\nabla n+\chi+\z)\\
\nn&&+\left(\frac{1}{\gl}+\frac{1}{1-\gn^2}\right)\\
\nn&&\times\frac{\left((N-N')^4(\d+\d'+n^{-1}\nabla n+\chi')+(N-N')^5\z'\right)bb' \nabb(P_l\trc)\nabb'(P_m\trc')}{\gg(L,L')(1-\gn^2)}.
\eea
Thus, in view of \eqref{boca0:12}, \eqref{boca:12}, \eqref{boca9:12} and \eqref{boca10:9} we obtain:
\be\lab{boca11:12}
A^2_{j,\nu,\nu',l,m}= 2^{-3j}\int_{\MM}\int_{\S\times\S}H F_{j,-1}(u)F_{j,-1}(u')\eta_j^\nu(\o)\eta_j^{\nu'}(\o')d\o d\o'd\MM,
\ee
with the tensor $H$ on $\MM$ given, schematically, by:
\bea\lab{boca12:12}
\nn H&=& \frac{1}{\gl^2(1-\gn^2)}\Bigg[(N-N')^2\nabb(L(P_l\trc))\nabb'(P_m\trc')\\
\nn&&+\nabb(P_l\trc)\bigg((N-N')^2\nabb(L'(P_m\trc'))+(N-N')^3{\nabb'}^2(P_m\trc')\\
\nn&&+(N-N')^4\nabb'(N'(P_m\trc'))\\
\nn&& +(N-N')^2(\chi+\chi'+\kep+\kep'+\d+\d'+n^{-1}\nabla n+L(b)+L'(b'))\nabb(P_m\trc')\\
\nn&&+(N-N')^3(k+n^{-1}\nabla n+\th+\th'+b^{-1}\nabb(b)+{b'}^{-1}\nabb'(b')+\chi+\z+\z'\\
\nn&&+\nabla_{N'}(b'))\nabla(P_m\trc')\bigg)\Bigg]+\nabla(P_l\trc)\nabb'(P_m\trc')(N-N')^3(k+n^{-1}\nabla n+\chi+\z)\\
\nn&&+\left(\frac{1}{\gl}+\frac{1}{1-\gn^2}\right)\\
\nn&&\times\frac{\left((N-N')^4(\d+\d'+n^{-1}\nabla n+\chi')+(N-N')^5\z'\right)\nabb(P_l\trc)\nabb'(P_m\trc')}{\gg(L,L')^2(1-\gn^2)}.
\eea
Recall the identities \eqref{nice24} and \eqref{nice25}:
$$\gg(L,L')=-1+\gn\textrm{ and }1-\gn=\frac{\gg(N-N',N-N')}{2}.$$
We may thus expand:
$$\frac{1}{\gl^2(1-\gn^2)},\,\frac{1}{\gl^3(1-\gn^2)} \textrm{ and }\frac{1}{\gl^2(1-\gn^2)^2}$$
in the same fashion than \eqref{nice27}, and we obtain, schematically:
\bea\lab{boca13:12}
H&=&\frac{1}{|N_\nu-N_{\nu'}|^2}\left(\sum_{p, q\geq 0}c_{pq}\left(\frac{N-N_\nu}{|N_\nu-N_{\nu'}|}\right)^p\left(\frac{N'-N_{\nu'}}{|N_\nu-N_{\nu'}|}\right)^q\right)\\
\nn&& \times\bigg(H_1+\frac{1}{|N_\nu-N_{\nu'}|}H_2+\frac{1}{|N_{\nu}-N_{\nu'}|^2}H_3\bigg),
\eea
where the tensors $H_1, H_2$ and $H_3$ on $\MM$ are given by:
\be\lab{boca14:12}
H_1= \nabb(P_l\trc)\nabb'(N'(P_m\trc')),
\ee
\bea\lab{boca15:12}
H_2&=& \nabb(P_l\trc){\nabb'}^2(P_m\trc')+\bigg(k+n^{-1}\nabla n+\th+\th'\\
\nn&&+b^{-1}\nabb(b)+{b'}^{-1}\nabb'(b')+\chi+\z+\z'+\nabla_{N'}(b')\bigg)\nabla(P_l\trc)\nabla(P_m\trc'),
\eea
and:
\bea\lab{boca16:12}
H_3&=&\nabb(L(P_l\trc))\nabb'(P_m\trc')+\nabb(P_l\trc)\bigg(\nabb'(L'(P_m\trc'))\\
\nn&&+(\chi+\chi'+\kep+\kep'+\d+\d'+n^{-1}\nabla n+L(b)+L'(b'))\nabb'(P_m\trc')\bigg),
\eea
and where $c_{pq}$ are explicit real coefficients such that the series 
$$\sum_{p, q\geq 0}c_{pq}x^py^q$$
has radius of convergence 1. In view of \eqref{boca11:12}, \eqref{boca13:12}, \eqref{boca14:12}, \eqref{boca15:12} and \eqref{boca16:12}, we obtain the decomposition \eqref{biz1} \eqref{biz2} \eqref{biz3} \eqref{biz4} \eqref{biz5} \eqref{biz6} \eqref{biz7} \eqref{biz8} \eqref{biz9} \eqref{biz10} \eqref{biz11} \eqref{biz12} \eqref{biz13} of $A^2_{j,\nu,\nu',l,m}$. This concludes the proof of Lemma \ref{lemma:biz}.

%%%%%%%%%%%%%%%%%%%%%%%%%%%%%%%%%%

\end{document}